\documentclass[reqno,letterpaper,12pt]{amsart}					
\usepackage[english]{babel}			
\usepackage{lmodern,geometry,amsmath,amsfonts,amssymb,amsthm,mathtools,graphicx,csquotes,accents,amssymb,url,enumitem, verbatim}
\usepackage[T1]{fontenc}
\usepackage[usenames,dvipsnames]{xcolor}
\usepackage[colorlinks=true,linkcolor=Red,citecolor=Green]{hyperref}
\usepackage{hyperref}
\hypersetup{pdfstartview=XYZ}
\usepackage{tikz}
\usetikzlibrary{decorations.markings,arrows.meta,3d}
\usepackage{tikz-3dplot}

\def\centerarc[#1](#2)(#3:#4:#5)
    { \draw[#1] ($(#2)+({#5*cos(#3)},{#5*sin(#3)})$) arc (#3:#4:#5); }

\DeclareMathOperator\supp{supp}

\usepackage{tikz}

\newcommand*{\cob}{b}
\newcommand*{\bdf}{\bff}

\newcommand*{\ibdf}{\rho_{\tindex,0}}

\newcommand*{\ind}{\mathcal{A}}

\def\ang#1{{\langle #1 \rangle }}
\newcommand*{\sfr}{\mathsf{r}}
\newcommand*{\sfl}{\mathsf{l}}

\newcommand*{\bbdf}{\hat{\rho}_{\alpha}}
\newcommand*{\tc}{\mathrm{3co}}
\newcommand*{\te}{\mathrm{3co,sf}}
\newcommand*{\sco}{ { o }_{\mathcal{C}^{\tindex}} }

\newcommand*{\tscf}{ {^{\mathrm{3sc}}\pi}_{ \mathrm{ff}_{\tindex}} }

\newcommand*{\itccf}{ {^{\mathrm{3co}}\iota}_{\mathrm{cf}_{\tindex}} }
\newcommand*{\itcodff}{ {^{\mathrm{3co}}\iota}_{\mathrm{dff}_{\tindex}} }
\newcommand*{\idtscdff}{ {^{\mathrm{d3sc}}\iota}_{\mathrm{dff}_{\tindex}} }

\newcommand*{\zf}{\mathrm{3cocf}_{\tindex}}

\newcommand*{\lz}{\zeta_{\tindex} }
\newcommand*{\tsclz}{\zeta_{\tindex} }

\newcommand*{\gtau}{\tau_{\tindex}^{\mathrm{sf}}}
\newcommand*{\bfa}{\mathrm{bf}}
\newcommand*{\sfa}{\mathrm{sf}}
\newcommand*{\econ}{[\overline{\mathbb{R}^{n}};\{0\}]}

\newcommand*{\WFsL}{\mathrm{WF}'_{\sigma, L^{\infty}}}
\newcommand*{\WFs}{\mathrm{WF}_{\sigma}'}
\newcommand*{\WFdff}{\mathrm{WF}_{\dff}'}
\newcommand*{\WFcf}{\mathrm{WF}_{\cf}'}

\newcommand*{\Ellss}{\mathrm{Ell}_{\sigma}} 
\newcommand*{\Elldff}{\mathrm{Ell}_{\dff}}
\newcommand*{\Ellcf}{\mathrm{Ell}_{\cf}}

\newcommand*{\tcosclt}{\tau_{\tindex}^{\mathrm{3co}}}
\newcommand*{\tcoscut}{\tau^{\tindex}_{\mathrm{co,sc}}}
\newcommand*{\tcosclm}{\mu_{\tindex}^{\mathrm{3co}}}
\newcommand*{\tcoscum}{\mu^{\tindex}_{\mathrm{co,sc}}}
\newcommand*{\tscresut}{\tau^{\tindex}_{\mathrm{sf}}}
\newcommand*{\tscresum}{\mu^{\tindex}_{\mathrm{sf}}}
\newcommand*{\tscreslt}{\tau_{\tindex}^{\mathrm{sf}}}
\newcommand*{\tscreslm}{\mu_{\tindex}}

\newcommand*{\tscblz}{\zeta_{\tindex} }
\newcommand*{\tscuz}{\zeta_{\tindex} }
\newcommand*{\tcoblz}{\zeta_{\tindex}^{\mathrm{3co}} }

\newcommand*{\tcolz}{\zeta_{\tindex}^{\mathrm{3co}} }
\newcommand*{\tscttlz}{\zeta_{\tindex} }
\newcommand*{\tcottlz}{\zeta_{\tindex}^{\mathrm{3co}}}
\newcommand*{\tscblt}{\tau_{\tindex}}
\newcommand*{\tscbut}{\tau_{\mathrm{b}}^{\tindex}}
\newcommand*{\tscblm}{\mu_{\tindex}}
\newcommand*{\tscbum}{\mu_{\mathrm{b}}^{\tindex}}
\newcommand*{\tcoblt}{\tau^{\mathrm{3co}}_{\tindex}}
\newcommand*{\tcobut}{\tau_{\mathrm{co,b}}^{\tindex}}
\newcommand*{\tcoblm}{\mu^{\mathrm{3co}}_{\tindex}}
\newcommand*{\tcobum}{\mu_{\mathrm{co,b}}^{\tindex}}
\newcommand*{\tscodlz}{\zeta_{\tindex} }
\newcommand*{\tcoodlz}{\zeta_{\tindex}^{\mathrm{3co}}}
\newcommand*{\tcosclz}{\zeta_{\tindex}^{\mathrm{3co}}}
\newcommand*{\tcoscuz}{\zeta^{\tindex}_{\mathrm{co,sc}}}

\newcommand*{\dtsccb}{\mathrm{d3sc,3co,res}}

\newcommand*{\tB}{G}
\newcommand*{\ttB}{\tilde{G}}
\newcommand*{\hutau}{\hat{\tau}^{\tindex}}
\newcommand*{\humu}{\hat{\mu}^{\tindex}}

\newcommand*{\humuresj}{ \hat{\mu}^{\tindex}_{\mathrm{sf},j} }

\newcommand*{\Xo}{\overline{\mathbb{R}^{n}}}
\newcommand*{\voro}{\tilde{\vor}}
\newcommand*{\tscX}{[ \overline{\mathbb{R}^{n}} ; \mathcal{C}_{\tindex} ]}

\newcommand*{\Cinfty}{\mathcal{C}^{\tindex}_{\infty}}
\newcommand*{\Co}{\mathcal{C}^{\tindex}_0}
\newcommand*{\tvor}{\tilde{\vor}}
\newcommand*{\vom}{\mathsf{m}}

\newcommand*{\tbff}{\tilde{\rho}_{\ff}}
\newcommand*{\bcv}{ \Sigma_{\sigma} }
\newcommand*{\brpm}{ \mathcal{R}_{\mathrm{2sc,dmf}, \pm} }
\newcommand*{\brp}{ \mathcal{R}_{\mathrm{2sc,dmf}, +} }
\newcommand*{\brm}{ \mathcal{R}_{\mathrm{2sc,dmf}, -} }
\newcommand*{\hutaub}{ \hat{\tau}^{\tindex}_{\mathrm{b}} }
\newcommand*{\humub}{ \hat{\mu}^{\tindex}_{\mathrm{b}} }
\newcommand*{\rhobf}{ \rho_{ \mathrm{b}, \infty } }
\newcommand*{\rhoresf}{ \rho_{\mathrm{sf},\infty} }

\newcommand*{\hutaures}{ \hat{\tau}^{\tindex}_{\mathrm{sf}} }
\newcommand*{\intnb}{ \fint }

\newcommand*{\humures}{\hat{\mu}^{\tindex}_{\mathrm{sf}}}

\newcommand*{\Psfo}{\Psi_{\mathrm{d3sc,3co,res}}}
\newcommand*{\hvor}{\hat{\vor}}
\newcommand*{\sHno}{ \sH_{\intn} }
\newcommand*{\sHto}{ \sH_{\intt} }
\newcommand*{\rf}{\mathrm{rf}_{\tindex}}
\newcommand*{\twosc}{ \mathrm{2sc} }
\newcommand*{\so}{ H_{\mathrm{d3sc,3co,res}} }
\newcommand*{\efN}{ \tilde{N}_{\cf} (\Delta_{z^{\tindex}})}
\newcommand*{\intn}{ |\utau|^{2} + |\umu|_{h^{\tindex}}^{2} }
\newcommand*{\intt}{ |\ltau|^{2} + |\lmu|_{h_{\tindex}}^{2} }
\newcommand*{\intnres}{ | \utaures |^{2} + |\umures|^{2}_{h^{\tindex}} }

\newcommand*{\tindex}{\alpha}
\newcommand*{\fint}{ |\ubtau|^{2} + | \ubmu |_{h^{\tindex}}^2 }
\newcommand*{\Xd}{X}
\newcommand*{\psf}{ \overline{^{\mathrm{d3sc,3co,res}}T^{\ast}} }

\newcommand*{\dmf}{\mathrm{dmf}}
\newcommand*{\dtsccf}{\mathrm{d3sccf}_{\tindex}}
\newcommand{\dff}{\mathrm{dff}_{\tindex}}
\newcommand*{\tcocf}{\mathrm{3cocf}_{\tindex}}
\newcommand*{\mf}{\mathrm{mf}}

\newcommand*{\Pf}{\Psi_{\mathrm{d3sc,3co,res},\delta}}

\newcommand*{\vor}{\mathsf{r}}
\newcommand*{\vol}{\mathsf{l}}
\newcommand*{\vob}{\mathsf{b}}
\newcommand*{\vos}{\mathsf{s}}
\newcommand*{\vov}{\mathsf{v}}
\newcommand*{\ubtau}{\tau^{\tindex}_{\mathrm{b}}}
\newcommand*{\ubmu}{\mu^{\tindex}_{\mathrm{b}}}
\newcommand*{\utaub}{\tau^{\tindex}_{\mathrm{b}}}
\newcommand*{\umub}{\mu^{\tindex}_{\mathrm{b}}}

\newcommand*{\btcocf}{\rho_{\mathrm{3cocf}_{\tindex}}}
\newcommand*{\bff}{\rho_{\alpha}}
\newcommand*{\hbff}{\hat{\rho}_{\tindex}}
\newcommand*{\cf}{\mathrm{cf}_{\tindex}}

\newcommand*{\utaucob}{ \tau^{\tindex}_{\mathrm{co,b}} }
\newcommand*{\umucob}{ \mu^{\tindex}_{\mathrm{co,b}} }

\newcommand*{\utaucosc}{ \tau^{\tindex}_{\mathrm{co,sc}} }
\newcommand*{\umucosc}{ \mu^{\tindex}_{\mathrm{co,sc}} }
\newcommand*{\ltau}{\tau_{\tindex}}
\newcommand*{\lmu}{\mu_{\tindex}}
\newcommand*{\utau}{\tau^{\tindex}}
\newcommand*{\umu}{\mu^{\tindex}}

\newcommand*{\utaures}{ \tau^{\tindex}_{\mathrm{sf}} }
\newcommand*{\umures}{ \mu^{\tindex}_{\mathrm{sf}} }
\newcommand*{\ltaucf}{\ltau}
\newcommand*{\lmucf}{\lmu}
\newcommand*{\inttcf}{ | \ltaucf |^{2} + | \lmucf |^{2}_{h_{\tindex}} }
\newcommand*{\Psf}{\Pf}

\newcommand*{\sH}{\mathsf{H}}
\newcommand*{\sclH}{ \mathsf{H}_{ | \tau_{\tindex} |^{2} + | \mu_{\tindex} |_{h_{\tindex}}^{2} } }
\newcommand*{\scuH}{ \mathsf{H}_{ | \tau^{\tindex} |^{2} + | \mu^{\tindex} |_{h^{\tindex}}^2 } }

\newcommand*{\ff}{\mathrm{ff}_{\tindex}}

\newcommand*{\ep}{ -- }

\numberwithin{equation}{section}

\theoremstyle{plain}
\newtheorem{theorem}{Theorem}[section]
\newtheorem{lemma}{Lemma}[section]
\newtheorem{proposition}{Proposition}[section]
\newtheorem{corollary}{Corollary}[section]

\theoremstyle{plain}
\newtheorem{main theorem}{Theorem}

\theoremstyle{definition}
\newtheorem{definition}{Definition}[section]
\newtheorem{remark}{Remark}[section]

\begin{document}
\title[Second microlocalized Fredholm theory for the three-body problem]{Second microlocalization and Fredholm theory for the three-body problem}
\author{Yilin Ma}
\address{Mathematical Sciences Institute, Australian National University, Acton, ACT 2601, Australia}
\email{yilin.ma@anu.edu.au}
\address{Department of Mathematics, University College London, 25 Gordon Street, London, WC1H 0AY, UK}
\email{yilin.ma@ucl.ac.uk}

\begin{abstract}
This paper studies quantum three-body scattering within a modern microlocal framework. We show that the three-body Helmholtz operator at positive energy gives rise to a pair of Fredholm maps between suitable anisotropic Hilbert spaces. Notably, we consider decay at various faces of spatial infinity separately, made precise via a compactification. Despite the problem's extensive history, new phenomena arise under this perspective, particularly regarding diffraction. Treating these phenomena requires the method of `second microlocalization' introduced by Vasy in \cite{AndrasSM} for the uniform Fredholm analysis of two-body Helmholtz operators at low energy, which does not directly extend to the three-body setting. This paper clarifies this structure.

We construct the conormal three-cone algebra, which serves as a `converse perspective' to the second microlocalization in question. This algebra exhibits a scattering structure at one spatial infinity face and a specific fibered structure at another, connected by a fibered cone. We show that by introducing suitable microlocal blow-ups at fiber infinity, the conormal three-cone algebra can be modified at the symbolic level to construct the desired second microlocalized algebra, which can be further promoted to a calculus. Incorporating variable orders, we use this calculus to prove microlocal propagation estimates with respect to a new flow in phase space. This flow has several radial sets (i.e., equilibria) which behave like saddles, so radial point estimates are required. By combining these estimates with elliptic regularity, and applying the result of Vasy in \cite{AndrasSM} as a black box, we show that the refined Fredholm maps can indeed be constructed.
\end{abstract}
\maketitle

\tableofcontents
\section{Introduction}
\subsection{The three-body problem}
\label{first subsection in the introduction section}
On the Euclidean space $\mathbb{R}^{n}$, consider a family of linear subspaces 
\begin{equation*}
 \{ X_{\tindex}^{\circ} \subset \mathbb{R}^{n} :  \alpha \in \ind \},
\end{equation*}
where $\ind$ is a finite index set, such that
\begin{equation} \label{three-body condition on the subspaces}
X_{\tindex}^{\circ} \cap X_{\tindex'}^{\circ} = \{ 0 \}, \quad \tindex \neq \tindex'.
\end{equation}
For each $\alpha \in \ind$, we will use the convention that
\begin{equation*}
(X^{\tindex})^{\circ} \coloneq (X_{\tindex}^{\circ})^{\perp}, \quad n_{\tindex}\coloneq \text{dim}( X_{\tindex}^{\circ} ), \ n^{\tindex} \coloneq n - n_{\tindex}.
\end{equation*}
Then it is clear that there exists a family of linear decompositions
\begin{equation} \label{distingushed decomposition}
\mathbb{R}^{n} = X_{\tindex}^{\circ} \oplus ( X^{\tindex} )^{\circ},
\end{equation}
with the corresponding decomposition of variables being $ z = (z_{\alpha}, z^{\alpha})$, where 
\begin{equation*}
z_{\alpha} = ( z_{\tindex,1}, ... , z_{\tindex,n_{\tindex}} ), \quad  z^{\alpha} = ( z^{\alpha}_{1}, ... , z^{\alpha}_{n^{\tindex}} ).
\end{equation*}
 \par 
It is extremely useful to introduce the radial compactification
\begin{equation*}
\overline{\mathbb{R}^{n}} \coloneq \mathbb{R}^{n} \sqcup \mathbb{S}^{n-1} \cong \mathbb{B}^{n},
\end{equation*}
where $\mathbb{B}^{n}$ is the $n$-dimensional, closed unit ball. See for instance \cite[Definition 6.48]{PeterNotes} or \cite[\S 2.1]{AndrewNSC} for the precise construction of $\overline{\mathbb{R}^{n}}$. Thus $\overline{\mathbb{R}^{n}}$ is a smooth, compact manifold with boundary $\mathbb{S}^{n-1}$ (the `sphere at infinity' representing the set of directions). Each factor in (\ref{distingushed decomposition}) can also be compactified in this way, and their compactifications will be denoted by 
\begin{equation*}
X_{\tindex} \coloneq \overline{X_{\tindex}^{\circ}}, \quad X^{\tindex}  \coloneq \overline{(X^{\tindex})^{\circ}}.
\end{equation*}
It follows that $X_{\tindex}$ and $X^{\tindex}$ are also smooth, compact manifolds with boundaries, with their boundaries being denoted by $
\mathcal{C}_{\alpha} \coloneq \partial X_{\tindex}$, $\mathcal{C}^{\alpha} \coloneq \partial X^{\tindex}$. Note that $\mathcal{C}_{\alpha}$, $\mathcal{C}^{\alpha}$ can be identified as embedded submanifolds of $\overline{\mathbb{R}^{n}}$. \par

A natural operator one studies under this setting is the Helmholtz operator 
\begin{equation} \label{Helmhotz operator}
P \coloneq \Delta + V - \lambda^2, \quad \lambda > 0,
\end{equation}
where $\Delta$ is the positive Laplacian (i.e., it is negative of the usual Laplacian, which is positive as an operator). Here, the potential $V$ is of the form
\begin{equation} \label{three-body potential full}
V \coloneq \sum_{ \tindex \in \ind } \pi_{\tindex}^{\ast}( V^{\tindex} ),
\end{equation}
where for each $\tindex \in \ind$, we let $\pi_{\tindex} :  X_{\tindex}^{\circ} \oplus (X^{\tindex})^{\circ} \rightarrow (X^{\tindex})^{\circ}$ be the natural projection, and
\begin{equation*}
V^{\alpha} \in  S^{-2-\delta} ( X^{\tindex} )
\end{equation*}
for some arbitrarily small $\delta > 0$, in the sense that $V^{\tindex} \in \mathcal{C}^{\infty}( (X^{\tindex})^{\circ} )$ and
\begin{equation*}
| \partial_{z^{\tindex}}^{\beta^{\tindex}} V^{\tindex} | \leq C_{\beta^{\tindex}} \langle z^{\tindex} \rangle^{-2 - \delta - |\beta^{\tindex}|}
\end{equation*}
for all $\beta^{\tindex} \in \mathbb{N}^{n^{\tindex}}_0$. We shall henceforth always write $V^{\tindex} = \pi_{\tindex}^{\ast}(V^{\tindex})$ for brevity. \par

If $n = 6$ and $n_{\tindex} = n^{\tindex} = 3$ for each $\tindex \in \ind$, then this includes the setup for the genuine quantum mechanical three-body problem with interesting potentials (e.g., Yukawa) and the center of mass removed. Thus in particular, we will often refer to $z_{\tindex}$ as the `free variables' and $z^{\tindex}$ the `interaction variables' for every $\tindex \in \ind$. \emph{It is a general principle in this article that we will use subscript to denote objects related to $X_{\tindex}$ and superscript to denote objects related to $X^{\tindex}$}. \par

The study of scattering theory for the three-body problem has a long history; among its most significant developments are the seminal works of \cite{EnssScattering, SSACompleteness}. Here we refrain from giving an exhaustive list of references.

In this paper, we will prove the following theorem:
\begin{main theorem} \label{main theorem 1}
Let $P$ be defined by (\ref{Helmhotz operator}) with $n_{\tindex} \geq 2$ and $n^{\tindex} \geq 3$. Moreover, suppose that $\Delta_{z^{\tindex}} + V^{\tindex}$ has no bound state nor half-bound state for every $\tindex \in \ind$. Then we can construct a pair of Hilbert spaces $\mathcal{X}_{\pm}$ and $\mathcal{Y}_{\pm}$ explicitly, such that the maps
\begin{equation} \label{introduction; main maps}
P: \mathcal{X}_{\pm} \rightarrow \mathcal{Y}_{\pm}
\end{equation}
are Fredholm. 
\end{main theorem}
We will delay the precise definitions of $\mathcal{X}_{\pm}$ and $\mathcal{Y}_{\pm}$ to \S \ref{subsection; statements of the main theorem} below. \par

Here, we remark that the condition of no bound state for every `subsystem' $\Delta_{z^{\tindex}} + V^{\tindex}$ is imposed for simplicity (and only really relevant for our analysis in \S \ref{decay at the front face section}), while the condition of no half-bound state for $\Delta_{z^{\tindex}} + V^{\tindex}$ is indeed generic.  \par

Theorem \ref{main theorem 1} should also hold in the case where $n_{\tindex} = 1$. However, in this case, the mathematical techniques required differ slightly from that of the $n_{\tindex} \geq 2$ case. Meanwhile, the assumption that $n^{\tindex} \geq 3$ is necessary to apply the results of \cite{AndrasSM}.

In fact, we expect the maps (\ref{introduction; main maps}) to be \emph{invertible}, which we will not prove in this paper. However, once invertibility is shown, we would have a precise description for the solution to $Pu = f \in \mathcal{Y}_{\pm}$, as the spaces $\mathcal{X}_{\pm}$ can be constructed explicitly. Our construction will allow for a large class of $f$. In particular, it holds that $\mathcal{S}(\mathbb{R}^{n}) \subset \mathcal{Y}_{\pm}$. In \S \ref{subsection; Summary of main innovations and future directions}, we will also briefly explain why proving invertibility requires non-trivial effort.

These restrictions are imposed primarily to contain the length of this paper.

It is clear that the constructions of $\mathcal{X}_{\pm}$ and $\mathcal{Y}_{\pm}$ are not unique. Indeed, although they are not phrased as so, the propagation results of \cite{AndrasThesis} are already enough to establish a Fredholm theory. This can be proved systematically if one makes a mild resolution to the three-body calculus (which is similar to the resolution we make in \S \ref{a further resolution at fiber infinity} below) at the fiber infinity of the three-body cotangent bundle, which will allow us to incorporate appropriate variable orders at spatial infinity. Nevertheless, one could also proceed more directly. See \cite{JesseWave, JesseWave2} for how this is done in the context of the Klein-Gordon equation.

However, the tools we develop in this paper will lead to a different, and arguably \emph{more precise} Fredholm theory, in that the construction of $\mathcal{X}_{\pm}$ also accounts for \emph{geometric diffraction} of the bicharacteristic flow. These tools naturally generalize the method of `second microlocalization' developed by Vasy in \cite{AndrasSM}, where the zero-energy problem is considered for the Helmholtz operator in the two-body setting, i.e., when $V \in S^{-2 - \delta}( \overline{\mathbb{R}^{n}} )$ instead. \par

Since this paper is lengthy, we will dedicate a considerable amount of effort in this introduction to motivate our construction. In particular, we will discuss briefly: 
\begin{itemize}
\item the meaning of second microlocalization within the context of the three-body problem, and why it arises naturally; as well as
\item how this second microlocal structure leads to a natural Fredholm framework.
\end{itemize}
Our aim is to provide the readers with a clear roadmap before delving into the details.

\subsection{Review of the two-body problem}
\label{a quick review of the two-body problem subsection}
Suppose for the moment that we have a potential $V \in S^{-\delta}( \overline{\mathbb{R}^{n}} )$ for some small $\delta > 0$ (the study of $P$ with such potentials will be referred to as the `two-body' problem). Then much is known for the solution theory to (\ref{Helmhotz operator}). In particular, there exists anisotropic Hilbert spaces $\mathcal{X}_{\mathrm{sc}}^{m_{\pm}, \mathsf{r}_{\pm}}$ and $\mathcal{Y}^{m_{\pm}, \mathsf{r}_{\pm}}_{\mathrm{sc}}$ such that 
\begin{equation} \label{Fredholm statement two-body case}
P : \mathcal{X}^{m_{\pm}, \mathsf{r}_{\pm} }_{\mathrm{sc}} \rightarrow \mathcal{Y}^{m_{\pm}, \mathsf{r}_{\pm}}_{\mathrm{sc}}
\end{equation}
are Fredholm, and in fact invertible maps. \par

The proof of the above statement relies crucially on the calculus of \emph{scattering operators} $\Psi^{m, \mathsf{r}}_{\mathrm{sc}, \delta}( \overline{\mathbb{R}^{n}} )$ with variable orders (the presence of a small $\delta >0$ is typical in the definition of variable order operators) $\vor \in \mathcal{C}^{\infty}( \mathbb{S}^{n-1} \times \overline{\mathbb{R}^{n}} )$. Here $m \in \mathbb{R}$ is the index of microlocal differential regularity (in the usual sense), while roughly speaking, $\vor$ measures decay microlocally at the spatial infinity $\mathbb{S}^{n-1}$. (However, if $\vor$ is replaced by some $r \in \mathbb{R}$, then $r$ would measure genuine spatial decay, which can nevertheless be microlocalized.) Scattering operators are microlocalized on the phase space $\overline{^{\mathrm{sc}}T^{\ast}} \overline{\mathbb{R}^{n}} \cong \overline{\mathbb{R}^{n}} \times \overline{\mathbb{R}^{n}} $. Thus, we can also understand $\vor$ as smooth functions on the spatial infinity (i.e., the boundary of the first factor) of $\overline{\mathbb{R}^{n}} \times \overline{\mathbb{R}^{n}}$. We remark that it is natural to work with scattering operators since $P \in \Psi^{2,0}_{\mathrm{sc}}( \overline{\mathbb{R}^{n}} )$. \par

Now, Sobolev spaces can also be defined with respect to variable order scattering operators. For $m \in \mathbb{R}$ and $\vor \in \mathcal{C}^{\infty}( \mathbb{S}^{n-1} \times \overline{\mathbb{R}^{n}} )$, such spaces will be denoted by $H^{m,\vor}( {\mathbb{R}^{n}} )$. The fact that (\ref{Fredholm statement two-body case}) are Fredholm can then be characterized quantitatively by a pair of estimates
\begin{equation} 
\label{2 body Fredholm estimate}
\| u \|_{H^{m_{\pm}, \mathsf{r}_{\pm}} } \leq C( \| Pu \|_{ H^{ m_{\pm}  -2 , \mathsf{r}_{\pm} + 1 } } + \| u \|_{H^{M , N }} )
\end{equation}
for all $u \in \mathcal{X}_{\mathrm{sc}}^{m_{\pm}, \vor_{\pm}}
$. Here $M < m_{\pm}$, $N < \vor_{\pm}$, the constant $C > 0$ is independent of $\lambda$, and
\begin{equation*}
\mathcal{X}_{\mathrm{sc}}^{m_{\pm}, \mathsf{r}_{\pm} } \coloneq \{ u \in H^{ m_{\pm} , \mathsf{r}_{\pm}}(  {\mathbb{R}^{n}} ) : Pu \in H^{ m_{\pm} - 2, \mathsf{r}_{\pm}+1 } (  {\mathbb{R}^{n}} ) \}, \quad \mathcal{Y}_{\mathrm{sc}}^{ m_{\pm} , \mathsf{r}_{\pm}} \coloneq H^{ m_{\pm} - 2, \mathsf{r}_{\pm}+1}( {\mathbb{R}^{n}} ).
\end{equation*}
See \cite{AndrewNSC, AndrasBook} for such a characterization of the Fredholm theory. Here we just remark it is crucial that 
\begin{equation*}
H^{m_{\pm}, \vor_{\pm}} (  {\mathbb{R}^{n}} ) \subset H^{M , N } ( {\mathbb{R}^{n}} )
\end{equation*}
are compact inclusions. From the quantitative perspective, (\ref{2 body Fredholm estimate}) can be seen as an alternative to elliptic regularity estimate for non-elliptic operators whose characteristic sets have well-behaved bicharacteristic flow, the proof of which will follow from a combination of elliptic regularity, propagation of regularity and radial point estimates.  \par

More concretely, let $\zeta$ be the variable dual to $z$. Then it can be checked that
\begin{equation} \label{Elliptic set of P in the scattering sense}
\Sigma_{\mathrm{sc}}( \lambda ) \coloneq \{ (y,\zeta) \in \mathbb{S}^{n-1} \times \mathbb{R}^{n} :  |\zeta| = \lambda \}
\end{equation}
is the characteristic set for $P$ in the scattering sense. Indeed, clearly $P$ is elliptic at the fiber infinity $\overline{\mathbb{R}^{n}} \times \mathbb{S}^{n-1}$, while at the spatial infinity $\mathbb{S}^{n-1} \times \mathbb{R}^{n}$, the principal symbol $p = |\zeta|^{2} - \lambda^{2}$ is only non-vanishing on (\ref{Elliptic set of P in the scattering sense}). For $x \coloneq |z|^{-1}$, the rescaled Hamiltonian vector field ${\mathsf{H}_{p}} \coloneq \frac{1}{x} H_{p}|_{x = 0}$ defines a complete flow on $\Sigma_{\mathrm{sc}}(\lambda)$, with bicharacteristics (i.e., integral curves of $\mathsf{H}_{p}$) converging to the sink$(-)$/source$(+)$ $\mathcal{R}_{\mathrm{sc},\pm}(\lambda)$ asymptotically in forward$(-)$/backward$(+)$ times. The sets $\mathcal{R}_{\mathrm{sc},\pm}(\lambda)$ are referred to as the outgoing$(-)$/incoming$(+)$ \emph{radial sets}. \par

\begin{figure}
\includegraphics{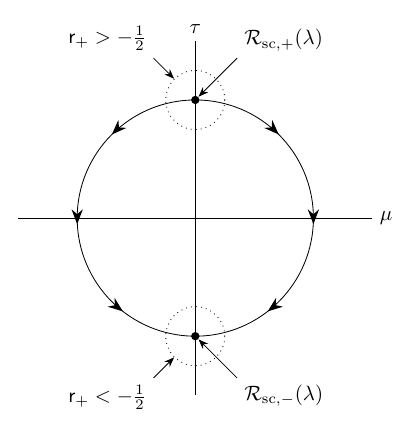}
\caption{The projection of $\Sigma_{\mathrm{sc}}(\lambda)$ to momentum space and the behavior of the bicharacteristic flow there. Forward bicharacteristics travel from $\mathcal{R}_{\mathrm{sc}, +}(\lambda)$ to $\mathcal{R}_{\mathrm{sc}, -}(\lambda)$ along a circle of radius $\lambda$. The variable orders $\mathsf{r}_{\pm}$ satisfy threshold conditions at $\mathcal{R}_{\mathrm{sc},\pm}(\lambda)$.} 
\end{figure}

The radial sets $\mathcal{R}_{\mathrm{sc}, \pm}(\lambda)$ can be most conveniently characterized with projective coordinates: Let $y = z/|z|$ and $( x, y, \tau, \mu )$ be phase space variables defined with respect to the degenerate form $\tau dx/x^2 + \mu \cdot dy/x$. Then we have 
\begin{equation*}
\mathcal{R}_{\mathrm{sc}, \pm}(\lambda) = \{ (y, \zeta) \in \mathbb{S}^{n-1} \times \mathbb{R}^{n} : \tau = \pm \lambda, \mu = 0 \}.
\end{equation*} 
Now, real principal type propagation of regularity estimate applies away from $\mathcal{R}_{\mathrm{sc}, \pm}( \lambda )$. On the other hand, the different, \emph{radial type} propagation estimates must be used at $\mathcal{R}_{\mathrm{sc}, \pm}(\lambda)$, which will have the threshold requirements of:
\begin{gather*}
\text{$\mathsf{r}_{+} > -1/2$ at $\mathcal{R}_{\mathrm{sc}, +}(\lambda)$ and $\mathsf{r}_{+} < - 1/2$ at $\mathcal{R}_{\mathrm{sc},-}(\lambda)$ for the forward time estimates;} \\
\text{$\mathsf{r}_{-} < -1/2$ at $\mathcal{R}_{\mathrm{sc},+}(\lambda)$ and $\mathsf{r}_{-} > -1/2$ at $\mathcal{R}_{\mathrm{sc}, -}(\lambda)$ for the backward time estimates.}
\end{gather*}
In particular, this forces $\mathsf{r}_{\pm}$ to be functions if we wish to obtain \emph{global} estimates (i.e., if we require both the above and below threshold estimates at $\mathcal{R}_{\mathrm{sc}, \pm}(\lambda)$ to be valid for \emph{some fixed choices of} $\mathsf{r}_{\pm}$). Moreover, $\mathsf{r}_{\pm}$ need to be monotonically decreasing$(+)$/increasing$(-)$ along the bicharacteristic flow. 
\par

We remark that propagation of regularity, classically also known as propagation of singularities, has been intensely studied in the late 1900s. See \cite{HormanderVolume3} and the reference therein for more details. A modern treatment can also be found in \cite{AndrasBook, PeterNotes}. On the other hand, radial point estimates were pioneered by Melrose in \cite{Melrose94}, though a natural precursor is the Mourre estimate (in fact, Melrose's radial point estimates could also be viewed as `microlocal Mourre estimates'). See \cite{Mourre1,Mourre2} for the classical results of Mourre, and also the paper \cite{SimonNbody} of Perry--Sigal--Simon, in which they extend Mourre's results to the many-body setting. Nevertheless, microlocal estimates which contain both propagation of regularity and radial point estimates were not established until much later. \par

The scattering calculus will be reviewed in \S \ref{subsection the scattering calculus}, though we also refer the readers to \cite{AndrasBook} for a comprehensive treatment.
\par

\subsection{The necessity of three-body operators} In the three-body case, the potential $V$ defined in (\ref{three-body potential full}) is not smooth on $\overline{\mathbb{R}^{n}}$ in general. Thus, the scattering calculus is no longer suitable for the study of $P$, since $P$ is no longer a scattering operator. To overcome this obstacle, an important observation made by Vasy in \cite{AndrasThesis} is to consider the blow-up 
\begin{equation}
\beta_{\text{3sc}}: [ \overline{\mathbb{R}^{n}}; \mathcal{C} ] \rightarrow \overline{\mathbb{R}^{n}},
\end{equation}
where we are writing
\begin{equation} \label{three-body blowup}
\mathcal{C}  \coloneq \{ \mathcal{C}_{\tindex} :  \tindex \in \ind \}.
\end{equation}
By (\ref{three-body condition on the subspaces}), we know that $\mathcal{C}_{ \tindex } \cap \mathcal{C}_{ \tindex' } = \emptyset$ for all $\tindex \neq \tindex'$. Thus $\mathcal{C}$ is a disjoint union of $\mathcal{C}_{\tindex}$ over $\tindex \in \ind$ (in particular, it is often enough to assume that $\mathcal{C} = \mathcal{C}_{\tindex}$). Let $\mf$ (the `main face'), $\ff$ (the `front face in the direction of $\mathcal{C}_{\tindex}$') be the lifts of $\mathbb{S}^{n-1}$ and $\mathcal{C}_{\tindex}$ to $[ \overline{\mathbb{R}^{n}} ; \mathcal{C} ]$ respectively, and let $\mathrm{ff}$ be the disjoint union of $\ff$ over $\tindex \in \ind$. Then $[ \overline{\mathbb{R}^{n}} ; \mathcal{C} ]$ is a smooth, compact manifold with corners and boundary faces $\mf$, $\mathrm{ff}$. Moreover, it is not hard to check that
\begin{equation*}
V \in S^{-2-\delta, 0}( [ \overline{\mathbb{R}^{n}} ; \mathcal{C} ] ),
\end{equation*}
where $-2 -\delta$ is the order of decay at $\mf$ and $0$ is the order of decay at $\mathrm{ff}$. See the beginning of \S \ref{section microlocal preliminaries} below for clarification on this notation. \par

Vasy constructed the spaces of conormal \emph{three-body} operators $\Psi^{m,r}_{\mathrm{3scc}}( [ \overline{\mathbb{R}^{n}} ; \mathcal{C} ] )$ such that $P \in \Psi^{2,0}_{\mathrm{3scc}}( [ \overline{\mathbb{R}^{n}} ; \mathcal{C} ] )$. Here $m$ is again the index of microlocal differential regularity, while $r$ measures decay at the \emph{full boundary} of $[ \overline{\mathbb{R}^{n}} ; \mathcal{C} ]$ (i.e., simutantously at both $\mf$ and $\mathrm{ff}$). Vasy proved that regularity of solutions to $Pu = f \in \mathcal{S}( \mathbb{R}^{n} )$ propagates microlocally along bicharacteristics which are `broken' at $\mathcal{C} \times \mathbb{R}^{n}$ in some sense. In particular, there is a natural notion of diffraction when the bicharacteristics of $P$ meet $\mathcal{C}$ at an angle. \par

However, since $[\overline{\mathbb{R}^{n}} ; \mathcal{C}]$ has several boundary hypersurfaces, the three-body operators can also be defined such that decay at $\mf$ and $\mathrm{ff}$ are measured separately, i.e., we can consider spaces of operators $\Psi^{m,r,l}_{\mathrm{3scc}}( [ \overline{\mathbb{R}} ; \mathcal{C} ] )$ such that $r,l \in \mathbb{R}$ measure decay at $\mf$ and $\mathrm{ff}$ respectively. The corresponding Sobolev spaces can also be weighted at $\mf$ and $\ff$ separately. One of our aims is to prove an analogue of (\ref{Fredholm statement two-body case}) with \emph{decoupled} indices in this sense.

\subsection{Motivating the second microlocalization}
\label{motivation subsection}
To illustrate the necessity of the second microlocalization we introduce in this paper, let us briefly explain the changes we encounter when attempting to show that regularity for solutions to $Pu=f$ propagates along bicharacteristics in the three-body setting. \par

To this end, let us discuss the bicharacteristic flow of $P$ in more details. At first we shall follow \cite{AndrasThesis}. For brevity, in the below we will often omit writing $\lambda$ in $\Sigma_{\mathrm{sc}}(\lambda)$, $\mathcal{R}_{\mathrm{sc},\pm}(\lambda)$. To begin, let us note that
\begin{equation*}
\Sigma_{ \mathrm{sc} } \cap ( \mathcal{C}_{\tindex} \times \mathbb{R}^{n} ) = \{ ( y_{\tindex},  \zeta )  \in \mathcal{C}_{\tindex} \times \mathbb{R}^{n} : |\zeta_{\tindex}|^2 + |\zeta^{\tindex}|^{2} = \lambda^{2} \},
\end{equation*}
where $\zeta_{\tindex}$, $\zeta^{\tindex}$ are dual to $z_{\tindex}$, $z^{\tindex}$ respectively. Consider the natural projection
\begin{equation*}
{^{\mathrm{sc}}\pi}_{ \mathcal{C}_{\tindex} }^{\perp}: \mathcal{C}_{\tindex} \times \mathbb{R}^{n}_{\zeta} \rightarrow \mathcal{C}_{\tindex} \times \mathbb{R}^{n_{\tindex}}_{\zeta_{\tindex}}.
\end{equation*}
Then we have
\begin{equation*}
{^{\mathrm{sc}}\pi}_{ \mathcal{C}_{\tindex}}^{\perp} \left( \Sigma_{ \mathrm{sc} } \cap ( \mathcal{C}_{\tindex} \times \mathbb{R}^{n} ) \right) = \{ ( y_{\tindex}, \zeta_{\tindex}) \in \mathcal{C}_{\tindex} \times \mathbb{R}^{n_{\tindex}} : |\zeta_{\tindex}| \leq \lambda \}.
\end{equation*}
This set splits into the disjoint union of
\begin{equation*}
\begin{gathered}
\Sigma_{\mathrm{n}} \coloneq \{ ( y_{\tindex}, \zeta_{\tindex} ) \in \mathcal{C}_{\tindex} \times \mathbb{R}^{n_{\tindex}} : |\zeta_{\tindex}| < \lambda  \} \ \text{and} \\
  \Sigma_{\mathrm{t}}\coloneq \{ ( y_{\tindex}, \zeta_{\tindex} ) \in \mathcal{C}_{\tindex} \times \mathbb{R}^{n_{\tindex}} : |\zeta_{\tindex}| = \lambda \},
\end{gathered}
\end{equation*}
where $\Sigma_{\mathrm{n}}$ is the `normal' (in fact, `transversal' would be a more accurate description, though then there is a clash of notations) part, and $\Sigma_{\mathrm{t}}$ is the `tangential' part. One can then show that a \emph{free} bicharacteristic segment $\gamma \subset \Sigma_{\mathrm{sc}}$ (i.e., a segment of some integral curve of $\mathsf{H}_{p}$) that is away from the radial sets $\mathcal{R}_{\mathrm{sc}, \pm}$ falls under three categories near $\mathcal{C}_{\tindex} \times \mathbb{R}^{n}$: 
\begin{enumerate}
\item Free propagation: If $ \gamma \cap ( \mathcal{C} \times \mathbb{R}^{n} ) = \emptyset$, then regularity propagates along $\gamma$ as in the two-body, or free setting. 
\item Transversal propagation: If $\gamma \cap ( \mathcal{C}_{\tindex} \times \mathbb{R}^{n} ) \neq \emptyset$ and $\gamma \not\subset \mathcal{C}_{\tindex} \times \mathbb{R}^{n}$, then $\gamma$ intersects $\mathcal{C}_{\tindex} \times \mathbb{R}^{n}$ at a unique $( y_{\tindex,0}, \zeta_{0} ) \in \Sigma_{\mathrm{n}}$ such that ${^{\mathrm{sc}}\pi_{\mathcal{C}_{\tindex}}^{\perp} }( y_{\tindex,0}, \zeta_{0} ) = ( y_{\tindex,0}, \zeta_{\tindex,0} )$. Regularity propagates along $\gamma$ as in the free setting until $\gamma$ intersects $\mathcal{C}_{\tindex} \times \mathbb{R}^{n}$, at which point diffraction phenomena occur.
\item Tangential propagation: If $\gamma \subset \mathcal{C}_{\tindex} \times \mathbb{R}^{n}$, then it holds that $\gamma \subset \mathcal{C}_{\tindex} \times \mathbb{R}^{n_{\tindex}}$ since $\mathcal{C}_{\tindex}$ is totally geodesic. Thus ${^{\mathrm{sc}}\pi_{\mathcal{C}_{\tindex}}^{\perp}}( \gamma ) \subset \Sigma_{\mathrm{t}}$, and regularity propagates along $\gamma$ as if it were a free bicharacteristic curve on $\mathcal{C}_{\tindex}$, i.e., an integral curve of
\begin{equation*}
\mathsf{H}_{p_{\tindex}} \coloneq \frac{1}{x_{\tindex}} H_{|\zeta_{\tindex}|^2} \big|_{x_{\tindex} = 0}, \quad x_{\tindex} \coloneq |z_{\tindex}|^{-1}.
\end{equation*}
\end{enumerate}
Using methods inspired by the proof of the Mourre estimate, Vasy showed that propagation of regularity (not including radial set estimates at the lifts of $\mathcal{R}_{\mathrm{sc}, \pm}$) for solutions to $Pu=f$ can be completely described under the above interpretation, provided we work on the scale of Sobolev spaces with spatial indices $r=l$ (i.e., just the scattering Sobolev spaces). \par

However, it is equally natural, at least in cases (1) and (2) above, to adopt an alternate perspective in the presence of decoupled indices, i.e., when $r \neq l$. To see this, suppose we consider a different rescaling of the Hamiltonian vector field
\begin{equation} \label{restriction of the mf Hamiltonian vector field}
\mathsf{H}_{p,\mf} \coloneq \frac{1}{x^{\tindex}} H_{p} \big|_{x^{\tindex} = 0},
\end{equation}
where $x^{\tindex} \coloneq |z^{\tindex}|^{-1}$ is a local defining function for $\mf$. Then it can be checked that $\mathsf{H}_{p,\mf}$ is tangent to $\mathrm{mf} \times \mathbb{R}^{n}$. Moreover,
if $\beta_{\mathrm{3sc}}$ also denotes the extended blow-down map
\begin{equation} \label{intro; beta 3sc is a natural blow down map}
\beta_{\mathrm{3sc}}: [ \overline{\mathbb{R}^{n}} ; \mathcal{C} ] \times \overline{\mathbb{R}^{n}} \rightarrow  \overline{\mathbb{R}^{n}} \times \overline{ \mathbb{R}^{n} },
\end{equation}
then $\mathsf{H}_{p,\mf}$ defines a complete flow on the lift of $ \Sigma_{\mathrm{sc}}$ through $\beta_{\mathrm{3sc}}$, which is usually denoted by $\beta_{\mathrm{3sc}}^{\ast}( \Sigma_{\mathrm{sc}} )$, and that the free bicharacteristics lift to become integral curves of $\mathsf{H}_{p, \mathrm{mf}}$ as well. Here, it is worth remarking that 
\begin{equation*}
\overline{ ^{\mathrm{3sc}}T^{\ast} }[ \overline{\mathbb{R}^{n}} ; \mathcal{C} ] = [ \overline{ ^{\mathrm{3sc}}T^{\ast} }[ \overline{\mathbb{R}^{n}} ; \mathcal{C} ] ; \overline{ ^{\mathrm{sc}}T^{\ast}}_{\mathcal{C}} \overline{\mathbb{R}^{n}} ] \cong [ \overline{\mathbb{R}^{n}} ; \mathcal{C} ] \times \overline{\mathbb{R}^{n}}
\end{equation*}
is also the phase space for the conormal three-body operators (i.e., where these operators are naturally microlocalized on). In particular, (\ref{intro; beta 3sc is a natural blow down map}) is also a natural blow-down map. \par

\begingroup
\begin{center}
\begin{figure}
\includegraphics{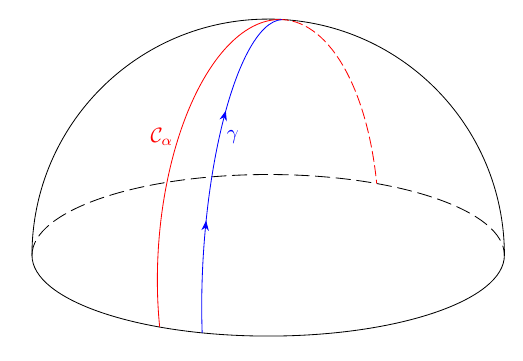}
\caption{A hemisphere in $\partial \overline{\mathbb{R}^{3}}$ when $\mathcal{C}_{\tindex}$ has dimension 1, i.e., a great circle. The curve $\gamma$ intersects $\mathcal{C}_{\tindex}$ at the radial set $\mathcal{R}_{\mathrm{n},+}$.}
\end{figure}
\begin{figure}
\includegraphics{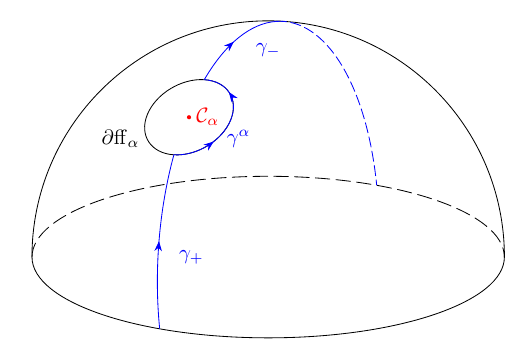}
\caption{The lift of a hemisphere in $\partial \overline{\mathbb{R}^{3}}$ to $[ \overline{\mathbb{R}^{3}} ; \mathcal{C}_{\tindex} ]$ when $\mathcal{C}_{\tindex}$ has dimension 0, i.e., just a point. If $Pu$ is sufficiently regular, then the microlocal regularity of $u$ propagates along $\gamma_{\pm}$, and upon reaching $\mathcal{C}_{\tindex}$, propagates to the `other side' of $\mathcal{C}_{\tindex}$ by traversing around $\partial \ff$ via curves like $\gamma^{\tindex}$.}
\label{figure3}
\end{figure}
\end{center}
\endgroup

Now, such a lift does not change the situation in case (1) above, since $\beta_{\mathrm{3sc}}$ restricts to the identity on $\gamma$ in this case. In case (2), it would be useful to introduce projective coordinates separately in the free and interaction position variables, i.e., we shall write $z_{\tindex} = ( x_{\tindex}, y_{\tindex} )$ and $z^{\tindex} = ( x^{\tindex}, y^{\tindex} )$ in the obvious senses. Moreover, let $\bdf \coloneq x_{\tindex}/x^{\tindex}$. Then it can be shown that
\begin{equation} \label{coordinates near 3sc bundle}
( \bdf, y_{\tindex}, x^{\tindex}, y^{\tindex}, \zeta_{\tindex} , \tau^{\tindex}, \mu^{\tindex} )
\end{equation}
are valid coordinates on $[ \overline{\mathbb{R}^{n}} ; \mathcal{C} ] \times \mathbb{R}^{n}$ near $\partial \ff \times \mathbb{R}^{n}$ with respect to the form 
\begin{equation*}
\zeta_{\tindex} \cdot d z_{\tindex} + \tau^{\tindex} \frac{dx^{\tindex}}{ (x^{\tindex})^2 } + \mu^{\tindex} \cdot \frac{dy^{\tindex}}{x^{\tindex}}.
\end{equation*}
Furthermore, $\bdf$ and $x^{\tindex}$ are local defining functions for $\ff \times \mathbb{R}^{n}$ and $\mf \times \mathbb{R}^n$ respectively. The advantage of working with (\ref{coordinates near 3sc bundle}) is that we can now define the \emph{normal} radial sets by
\begin{equation*}
\mathcal{R}_{\mathrm{n},\pm} \coloneq \{ \bdf = 0, x^{\tindex} = 0, \tau^{\tindex} = \pm ( \lambda^{2} - |\zeta_{\tindex}|^2 )^{\frac{1}{2}}, \mu^{\tindex} = 0  \} \subset \partial \ff \times \mathbb{R}^{n}.
\end{equation*}
It can be checked that $\mathcal{R}_{\mathrm{n},\pm}$ are saddle points for the flow of $\mathsf{H}_{p,\mf}$. Moreover, integral curves of $\mathsf{H}_{p,\mf}$ can only enter$(+)$/exit$(-)$ $\partial \ff \times \mathbb{R}^{n}$ via $\mathcal{R}_{\mathrm{n}, \pm}$. \par

Now, let us return to case (2) above, and let $\gamma$ be the bicharacteristic segment in question. Suppose we view $\ff \times \mathbb{R}^{n}_{\zeta}$ through the identification
\begin{equation*} \label{obvious fibration}
\ff \times \mathbb{R}^{n}_{\zeta} \cong \mathcal{C}_{\tindex} \times \mathbb{R}^{n_{\tindex}}_{\zeta_{\tindex}} \times {^{\mathrm{sc}}TX^{\tindex}},
\end{equation*}
and define the corresponding natural projection 
\begin{equation} \label{three-body projection}
{ ^{\mathrm{3sc}} \pi_{\ff}^{\perp} } \coloneq { ^{\mathrm{sc}}\pi_{\mathcal{C}_{\tindex}}^{\perp}} \circ  \beta_{\mathrm{3sc}}|_{\ff \times \mathbb{R}^{n} } : \mathcal{C}_{\tindex} \times \mathbb{R}^{n_{\tindex}}_{\zeta_{\tindex}} \times {^{\mathrm{sc}}TX^{\tindex}} \rightarrow \mathcal{C}_{\tindex} \times \mathbb{R}^{n_{\tindex}}_{\zeta_{\tindex}}.
\end{equation}
Then $\beta_{\mathrm{3sc}}^{\ast} \gamma$ gets broken at $(^{\mathrm{3sc}}\pi_{\ff}^{\perp})^{-1}( \Sigma_{\mathrm{n}} )$ into a pair of segments $\gamma_{\pm} \subset \mathrm{mf} \times \mathbb{R}^{n}$ such that $\gamma_{+} \cup \gamma_{-} = \gamma$ when restricted to $\mf^{\circ} \times \mathbb{R}^{n}$. Regularity propagates along $\gamma_{\pm}$ as in the free setting. If $\gamma$ intersects $\mathcal{C}_{\tindex} \times \mathbb{R}^{n}$ at a point whose projection under ${^{\mathrm{3sc}}\pi_{\ff}^{\perp}}$ is $( y_{\tindex,0}, \zeta_{\tindex,0} ) \in \Sigma_{\mathrm{n}}$, then upon viewing $\gamma_{\pm}$ as integral curves of $\mathsf{H}_{p,\mf}$, we have
\begin{equation*}
\lim_{t \rightarrow \pm \infty} \gamma_{\pm}(t) \in \mathcal{R}_{\mathrm{n}, \pm}( y_{\tindex,0}, \zeta_{\tindex,0} ),
\end{equation*}
where we are writing
\begin{equation*}
\mathcal{R}_{\mathrm{n}, \pm}( y_{\tindex}, \zeta_{\tindex} ) \coloneq \mathcal{R}_{\mathrm{n},\pm} \cap ( \{ ( y_{\tindex}, \zeta_{\tindex} ) \} \times { ^{\mathrm{sc}}T^{\ast}_{\mathcal{C}^{\tindex}} X^{\tindex} } ), 
\end{equation*}
i.e., the intersection is taking place at a fiber of (\ref{three-body projection}), which could also be written just as $(^{\mathrm{3sc}}\pi_{\ff}^{\perp})^{-1}(y_{\tindex}, \zeta_{\tindex})$, though we are using more explicit notation here. Moreover, it can be shown that 
\begin{equation} \label{restriction of the mf Hamiltonian flow to a fiber}
\mathsf{H}_{p, \mf}|_{  \{ ( y_{\tindex}, \, \zeta_{\tindex} ) \} \times { ^{\mathrm{sc}}T^{\ast}_{\mathcal{C}^{\tindex}} X^{\tindex} }   } = \mathsf{H}_{|\zeta^{\tindex}|^2}, \quad \sH_{|\zeta^{\tindex}|^2} \coloneq \frac{1}{x^{\tindex}} H_{|\zeta^{\tindex}|^{2}} \big|_{x^{\tindex} = 0}
\end{equation}
for every $( y_{\tindex}, \zeta_{\tindex} ) \in \mathcal{C}_{\tindex} \times \mathbb{R}^{n_{\tindex}}$. Thus, one can also connect the forward time asymptotics of $\gamma_{+}$ with the backward time asymptotics of $\gamma_{-}$ via integral curves of $\mathsf{H}_{|\zeta^{\tindex}|^2}$ (which are still integral curves of $\mathsf{H}_{p,\mf}$ in view of (\ref{restriction of the mf Hamiltonian flow to a fiber})), say $\gamma^{\tindex}$, such that $\gamma^{\tindex}$ lives on $\Sigma_{\partial \ff}(y_{\tindex,0}, \zeta_{\tindex,0})$, where
\begin{equation}
\label{introduction normal characteristic set}
\Sigma_{ \partial \ff } ( y_{\tindex}, \zeta_{\tindex} ) \coloneq \{ ( y_{\tindex}, \zeta_{\tindex},  y^{\tindex} , \zeta^{\tindex} ) \in \{ ( y_{\tindex}, \zeta_{\tindex} ) \} \times { ^{\mathrm{sc}}T^{\ast}_{\mathcal{C}^{\tindex}} X^{\tindex} }  : |\zeta^{\tindex}|^{2} = \lambda^{2} - |\zeta_{\tindex}|^2 \}.
\end{equation}
Notice that $\Sigma_{\partial \ff} (y_{\tindex},\zeta_{\tindex})$ is contained in $\beta_{\mathrm{3sc}}^{\ast}( \Sigma_{\mathrm{sc}} )$, and can be identified as a submanifold of $^{\mathrm{sc}}T^{\ast}_{\mathcal{C}^{\tindex}}X^{\tindex}$. From this perspective, $\Sigma_{\partial \ff}( y_{\tindex}, \zeta_{\tindex} )$ is simply the characteristic set of
\begin{equation} \label{indicial operator of P}
\hat{P}_{\tindex}( |\zeta_{\tindex}| ) \coloneq \Delta_{z^{\tindex}} + V^{\tindex} - ( \lambda^{2} - |\zeta_{\tindex}|^2 )
\end{equation}
if we view (\ref{indicial operator of P}) as an element of $\Psi_{\mathrm{sc}}^{2,0}(X^{\tindex})$. Moreover, we can write $\mathcal{R}_{\mathrm{n},\pm}(y_{\tindex}, \zeta_{\tindex})$ as
\begin{equation*}
 \{ ( y_{\tindex}, \zeta_{\tindex},  y^{\tindex} , \utau, \umu ) \in \{ ( y_{\tindex}, \zeta_{\tindex} ) \} \times { ^{\mathrm{sc}}T^{\ast}_{\mathcal{C}^{\tindex}} X^{\tindex} }  : \tau^{\tindex} = \pm ( \lambda^{2} - |\zeta_{\tindex}|^2 )^{\frac{1}{2}}, \mu^{\tindex} = 0  \},
\end{equation*}
which are the incoming($+$)/outgoing($-$) radial sets of $\hat{P}_{\tindex}( |\zeta_{\tindex}| )$ in this sense as well.  \par

Since we have assumed that $|\zeta_{\tindex}| < \lambda$ (by the criteria for transversal propagation), the operator (\ref{indicial operator of P}) defines a \emph{finite} energy, two-body problem. Thus, we may expect microlocal regularity to propagate also from the `end point' of $\gamma_{+}$ (i.e., the point $\lim_{t \rightarrow \infty} \gamma_{+}(t)$) to the `starting point' of $\gamma_{-}$ (i.e., the point $\lim_{t \rightarrow -\infty} \gamma_{-}(t)$) via bicharacteristics of (\ref{indicial operator of P}) as in the two-body setting. See also Figure \ref{figure3} for a graphical illustration. We remark that this is exactly the setup of \emph{diffraction} as presented in \cite{JaredDiffraction}. \par


Now, one could attempt to carry out propagation estimates following the above dynamical considerations. However, one obstacle, and indeed a major difference between the three-body and the scattering calculi, is that the leading order behavior of a general $A \in \Psi^{m,r,l}_{\mathrm{3sc}}( [ \overline{\mathbb{R}^{n}} ; \mathcal{C} ] )$ at $\mathrm{ff}$ is no longer captured by its principal symbol. In particular, for operators whose (full) symbols are partially classical at $\mf \times \mathbb{R}^{n}$, the principal symbol map would correspond to a restriction at $\mf \times \mathbb{R}^{n}$ upon normalizing with respect to some boundary defining function of $\mf$. In fact, this is the reason why the rescalling (\ref{restriction of the mf Hamiltonian vector field}) was considered. \par

To capture decay at $\mathrm{ff}$, we need to consider the \emph{indicial operators}
\begin{equation} \label{membership of indicial operators}
{ ^{\mathrm{3sc}}\hat{N}_{\ff,l} }(A) \in \mathcal{C}^{\infty}( \mathcal{C}_{\tindex} \times \mathbb{R}^{n_{\tindex}}_{\zeta_{\tindex}} ; \Psi^{m,r-l}_{\mathrm{sc}}( X^{\tindex} ) ), \quad \tindex \in \ind,
\end{equation}
with a large-parameter behavior as $|\zeta_{\tindex}| \rightarrow \infty$ (see \S \ref{parameters-dependent families subsection} for the precise description of large-parameter scattering operators). The indicial operators are multiplicative just like the principal symbol, and we can understand them as `operator-valued symbols' in a sense. In fact, it is not hard to see that
\begin{equation*}
{ ^{\mathrm{3sc}}\hat{N}_{\ff,0}} (P) ( y_{\tindex}, \zeta_{\tindex} ) = \hat{P}_{\tindex}( |\zeta_{\tindex}| ),
\end{equation*}
i.e., the indicial operator of $P$ at $\ff$ is simply the operator (\ref{indicial operator of P}). Hence, we also expect that the two-body phenomena which occur in the three-body bicharacteristic dynamic to be prominent at the level of indicial operators as well.  \par

Let $H^{m,r,l}( \mathbb{R}^{n} )$ be the Sobolev spaces defined with respect to the three-body operators, where $m,r,l \in \mathbb{R}$ measure differential regularity and decay at $\mf$, $\mathrm{ff}$ respectively. Here the presence of variable orders are also possible, though we do not consider them for the current discussion. Then it can be checked that the inclusions
\begin{equation*}
H^{M,N,L}( \mathbb{R}^{n} ) \subset H^{m,r,l}( \mathbb{R}^{n} )
\end{equation*}
are compact only when $M < m$, $N < r$ and $L < l$ (and likewise in the variable orders case). \par

As a consequence, propagation estimates which are made using only the principal symbol map will not be sufficient in obtaining an analogy to (\ref{2 body Fredholm estimate}), as we would not be able to obtain a compact remainder in this way, i.e., the indicial operators must also be involved. The latter is exactly what was done in \cite{AndrasThesis}: positive commutator calculations were carried out using both the principal symbol and the indicial operators. However, such calculations necessitate $r = l$. Thus, in the settings where $r \neq l$, the arguments of \cite{AndrasThesis} cease to be valid, and we would need to make calculations with respects to the principal symbol and the indicial operators separately.

One way to overcome this challenge is to apply techniques similar to that of \cite[\S 7]{Peter3b}, the key of which is the uniform invertibility for the parametrized two-body operators $\hat{P}_{\tindex}( |\zeta_{\tindex}| )${\ep}at least for $|\zeta_{\tindex}| < \lambda$, i.e., the case of transversal propagation. Heuristically, this aligns with the definition that ellipticity at $\ff$ corresponds to the pointwise invertibility of $\hat{P}_{\tindex}(y_{\tindex}, \zeta_{\tindex})$ for every $(y_{\tindex}, \zeta_{\tindex})$, with an inverse that belongs to $\Psi_{\mathrm{sc}}^{-2,0}( \overline{\mathbb{R}^{n}} )$. Unfortunately, the aforementioned invertibility result does not lead to a family of inverses in the correct class. Indeed, we only know that $\hat{P}_{\tindex}$ is uniformly invertible between suitable anisotropic Sobolev spaces, the inverse of which may not be a family of scattering operators. But this can be supplemented by concrete estimates and the fact that $\hat{P}_{\tindex}$ is truly the partial Fourier transform of $P$ in the free variables. We will not go into more details on this.

If the above discussions can be made rigorous, then we would have understood propagation of regularity along free bicharacteristic which intersect $\mathcal{C}_{\tindex} \times \mathbb{R}^{n}$ transversally, and by extension any free bicharacteristics that are away from $\mathcal{C}_{\tindex} \times \mathbb{R}^{n}$ as well (i.e., cases (1) and (2) above). It remains to consider the case of tangential propagation (i.e., case (3) above). The problem is how propagation phenomena should be understood \emph{as they transition from transversal propagation to tangential propagation}. \par

Indeed, recall that propagation estimates in the transversal directions rely crucially on the understanding of the bicharacteristic flow for $\hat{P}_{\tindex}(|\zeta_{\tindex}|)$ where $(y_{\tindex}, \zeta_{\tindex}) \in \Sigma_{\mathrm{t}}$, i.e., when $|\zeta_{\tindex}| < \lambda$. The transition from transversal to tangential propagation thus corresponds to the transition from $(y_{\tindex}, \zeta_{\tindex}) \in \Sigma_{\mathrm{n}}$ to $( y_{\tindex}, \zeta_{\tindex} ) \in \Sigma_{\mathrm{t}}$, or equivalently, the process of letting $|\zeta_{\tindex}| \rightarrow \lambda$ from below. \par

This gives rises to two problems:
\begin{enumerate}
    \item For principal symbol consideration: The sets $\Sigma_{\partial \ff}(y_{\tindex}, \zeta_{\tindex})$ defined in (\ref{introduction normal characteristic set}) degenerate for $( y_{\tindex}, \zeta_{\tindex} ) \in \Sigma_{\mathrm{t}}$ in the sense that they becomes zero-dimensional in the interaction momentum variables, i.e., if $(y_{\tindex}, \zeta_{\tindex}) \in \Sigma_{\mathrm{t}}$, then $\Sigma_{\partial \ff}( y_{\tindex}, \zeta_{\tindex} ) \subset \{ ( y_{\tindex}, \zeta_{\tindex} ) \} \times o_{\mathcal{C}^{\tindex}}$ when viewed as a submanifold of $ \{ ( y_{\tindex}, \zeta_{\tindex} ) \} \times {^{\mathrm{sc}}T^{\ast}_{\mathcal{C}^{\tindex}}}X^{\tindex}$, where $o_{\mathcal{C}^{\tindex}}$ denotes the zero section in ${^{\mathrm{sc}}T^{\ast}_{\mathcal{C}^{\tindex}}}X^{\tindex}$ (i.e., $\mathcal{C}^{\tindex} \times \{ 0 \}$). Thus, the flow of $\sH_{p,\mf}$ is no longer `nice' on $\Sigma_{\partial \ff}(y_{\tindex}, \zeta_{\tindex})$ as it consists entirely of stationary points. Moreover, it can be checked that
   \begin{equation*}
   \beta_{\mathrm{3sc}}^{\ast}( \mathcal{R}_{\mathrm{sc}, \pm} ) \cap ( \ff \times \mathbb{R}^{n} ) \subset \mathcal{R}_{\mathrm{n}, \pm} \cap ( { ^{\mathrm{3sc}} \pi_{\ff}^{\perp} })^{-1}( \Sigma_{\mathrm{t}} ). 
   \end{equation*}
   Thus in particular, by understanding the aforementioned degeneracy, we also understand better the flow of $\sH_{p,\mf}$ at $\beta_{\mathrm{3sc}}^{\ast}( \mathcal{R}_{\mathrm{sc}, \pm} )$ in the process.

    \item For indicial operator consideration: Suppose we view $\hat{P}_{\tindex}(|\zeta_{\tindex}|)$ as a family of scattering operators (which is necessary if $P$ were to be viewed as an element of $\Psi_{\mathrm{3scc}}^{2,0,0}([ \overline{\mathbb{R}^{n}} ; \mathcal{C} ])$). Then we no longer have a uniform solvability theory of $\hat{P}_{\tindex}(|\zeta_{\tindex}|)$ that is analogous to the invertibility (\ref{Fredholm statement two-body case}) as $|\zeta_{\tindex}| \rightarrow \lambda$ from below. 
    \end{enumerate}
This is where second microlocalization comes in.

\subsection{Second microlocalization for the two-body problem}
\label{an overview of second microlocalization subsection}
In \cite{AndrasSM}, Vasy considered the zero-energy problem for the Helmholtz operator, i.e., the uniform behavior of 
\begin{equation} \label{the zero energy Helmhotz operator}
P( \sigma ) \coloneq \Delta + V - \sigma^{2}, \quad \sigma \geq 0
\end{equation}
as $\sigma \rightarrow 0$ from above, and where $V \in S^{-2-\delta}( \overline{\mathbb{R}^{n}} )$. Its characteristic set 
\begin{equation*}
\Sigma_{\mathrm{sc}}( 0 ) = \mathbb{S}^{n-1} \times \{ 0 \}
\end{equation*}
degenerates in the sense that it is zero-dimensional in the momentum variables (i.e., the degeneracy for the flow of $\sH_{p,\mf}$ at $\Sigma_{\partial \ff}(y_{\tindex}, \zeta_{\tindex})$, $(y_{\tindex}, \zeta_{\tindex}) \in \Sigma_{\mathrm{t}}$ discussed above), and therefore consists entirely of stationary bicharacteristics (i.e., points), which is a useless Hamiltonian dynamic in terms of microlocal propagation estimates.  \par

To overcome this problem, Vasy considered the phase space blow-up
\begin{equation} \label{phase space for second microlocalization in the introduction}
[ \overline{\mathbb{R}^{n}} \times \overline{\mathbb{R}^{n}} ; \mathbb{S}^{n-1} \times \{ 0 \} ],
\end{equation}
so polar coordinates are introduced about $\Sigma_{\mathrm{sc}}( 0 )$. As $\sigma \rightarrow 0$, the family of radial sets $\mathcal{R}_{\mathrm{sc}, \pm}(\sigma)$ converge onto the front face created by (\ref{phase space for second microlocalization in the introduction}), and remain to be disjoint there. In particular, upon carrying out the blow-up (\ref{phase space for second microlocalization in the introduction}), the bicharacteristic flow at (the lift of) $\Sigma_{\mathrm{sc}}( 0 )$ becomes useful once again. This resolves the degeneracy at the dynamical level in a sense. \par

However, to perform concrete calculations, one must also consider how such a resolution makes sense at the analytic level of operators. In terms of microlocal analysis, this amounts to constructing an algebra of operators whose elements can be microlocalized at (\ref{phase space for second microlocalization in the introduction}){\ep}a process that is difficult to execute directly. \par

Vasy's approach, therefore, is to adopt a `converse perspective', which starts from the spaces of conormal b-operators $\Psi_{\mathrm{bc}}^{m,l}( \overline{\mathbb{R}^{n}} )$ for $m,l \in \mathbb{R}$. Here, $l$ measures spatial decay, while $m$ is the index of microlocal differential regularity, now measured with respect to the Lie algebra of b-vector fields $\mathcal{V}_{\mathrm{b}}( \overline{\mathbb{R}^{n}} )$, which is the span of $\{ \langle  z \rangle \partial_{z_1}, ... , \langle z \rangle \partial_{z_{n}} \}$ over $\mathcal{C}^{\infty}( \overline{\mathbb{R}^{n}} )$. In particular, if $m \in \mathbb{N}_0$, then we can understand $\Psi_{\mathrm{bc}}^{m,0}( \overline{\mathbb{R}^{n}} )$ as the natural microlocalization for those differential operators of order $m$ generated by $\mathcal{V}_{\mathrm{b}}( \overline{\mathbb{R}^{n}} )$ with coefficients in $\mathcal{C}^{\infty}( \overline{\mathbb{R}^{n}} )$.  \par

The conormal b-operators are microlocalized on the fiber-compactified \emph{b-cotangent bundle} $\overline{ ^{\mathrm{b}}T^{\ast} }\overline{\mathbb{R}^{n}}$ (which needs to be constructed separately). Moreover, it can be shown that 
\begin{equation} 
\label{second microlocalization diffeomorphism in the introduction}
[ \overline{ ^{\mathrm{sc}}T^{\ast}} \overline{\mathbb{R}^{n}} ; o_{\mathbb{S}^{n-1}} ] \cong [ \overline{ ^{\mathrm{b}}T^{\ast} }\overline{\mathbb{R}^{n}} ;  {^{\mathrm{b}}S^{\ast}_{ \mathbb{S}^{n-1} }} \overline{\mathbb{R}^{n}} ].
\end{equation}
Here, $o_{\mathbb{S}^{n-1}}$ is simply the zero section in ${^{\mathrm{sc}}T^{\ast}_{\mathbb{S}^{n-1}}}\overline{\mathbb{R}^{n}}$ (i.e., $\mathbb{S}^{n-1} \times \{ 0 \}$), and thus $[ \overline{ ^{\mathrm{sc}}T^{\ast}} \overline{\mathbb{R}^{n}} ; o_{\mathbb{S}^{n-1}} ]$ is the same manifold as (\ref{phase space for second microlocalization in the introduction}). Moreover, the diffeomorphism in (\ref{second microlocalization diffeomorphism in the introduction}) restricts to the identity map in the interiors. Thus the microlocal structures are changed only at the boundaries.  \par

Now, an important observation is that the blow-up (\ref{second microlocalization diffeomorphism in the introduction}) occurs at a corner face of $\overline{ ^{\mathrm{b}}T^{\ast} }\overline{\mathbb{R}^{n}}$. The characterization for conormality is therefore insensitive to such a blow-up. Indeed, a natural subclass for the second microlocalized operators can already be defined by
\begin{equation} \label{second microlocalized algebra tautology case}
\Psi^{ m, m + l , l }_{\mathrm{sc,b}}( \overline{\mathbb{R}^{n}} ) \coloneq \Psi^{ m,l }_{\mathrm{bc}}( \overline{\mathbb{R}^{n}} ),
\end{equation}
where the main idea is as follows: b-differential regularity can be microlocalized for conormal b-operators, and this corresponds to decay at the fiber infinity ${ ^{\mathrm{b}}S^{\ast} } \overline{\mathbb{R}^{n}}$ for symbols which are conormal to $\overline{ ^{\mathrm{b}}T^{\ast}} \overline{\mathbb{R}^{n}}$. By the above comment on conormality, any amount of decay at ${ ^{\mathrm{b}}S^{\ast}} \overline{\mathbb{R}^{n}}$ is also transferred to the new face created by blowing up ${ ^{\mathrm{b}}S^{\ast}_{\mathbb{S}^{n-1}}} \overline{\mathbb{R}^{n}}${\ep}which we may refer to as the \emph{scattering face} for brevity. Moreover, we can check that the lifts of ${^{\mathrm{sc}}S^{\ast}} \overline{\mathbb{R}^{n}}$ and $\overline{ ^{\mathrm{sc}}T^{\ast} }_{\mathbb{S}^{n-1}} \overline{ \mathbb{R}^{n} }$ to the left hand side of (\ref{second microlocalization diffeomorphism in the introduction}) can be respectively identified with the lift of ${ ^{\mathrm{b}}S^{\ast} } \overline{\mathbb{R}^{n}}$ to the right hand side of (\ref{second microlocalization diffeomorphism in the introduction}) and the scattering face. Thus, regularity at both the lifts of ${^{\mathrm{sc}}S^{\ast}} \overline{\mathbb{R}^{n}}$ and $\overline{ ^{\mathrm{sc}}T^{\ast} }_{\mathbb{S}^{n-1}} \overline{ \mathbb{R}^{n} }$, which can only be understood microlocally, are indeed captured by symbolic calculations. In (\ref{second microlocalized algebra tautology case}), the indices $m$ and $m+l$ measure microlocal decay respectively at these faces. 
\par

On the other hand, $l$ measures spatial decay more globally in the b-sense (typically requiring either the normal of indicial operators). But as the lift of $\overline{ ^{\mathrm{b}}T^{\ast} } \overline{\mathbb{R}^{n}}$ identifies with the new face created by blowing up the left hand side of (\ref{second microlocalized algebra tautology case}){\ep}which we may refer to as the \emph{b-face}, we can understand $l$ as the index which microlocally measures decay at the b-face as well. In particular, this confirms that the microlocal nature of the b-face is non-symbolic.  \par

Thus morally speaking, elements of (\ref{second microlocalized algebra tautology case}) can be viewed as being defined through modifying conormal b-operators purely at the symbolic level, in such a way that the symbols are required to be conormal on (\ref{second microlocalization diffeomorphism in the introduction}), though this is really a tautology at the level of (\ref{second microlocalized algebra tautology case}). This idea generalizes easily to define the class
\begin{equation} \label{introduction: sc,b operators all indices are arbitrary}
\Psi^{m , r, l}_{\mathrm{sc,b}}( \overline{\mathbb{R}^{n}} ), \quad m,r,l \in \mathbb{R},
\end{equation}
where $r$ measures microlocal decay at the scattering face. One could also incorporate variable orders $\mathsf{m}$ and $\mathsf{r}$. However, the index which measure decay at the b-spatial infinity (i.e., $l$), which microlocally also includes the b-face, can only be constant. \par

Correspondingly, we can define a family of Sobolev spaces $H^{ m ,\mathsf{r},l}_{\mathrm{sc,b}}( \overline{\mathbb{R}^{n}} )$. Then Vasy proved that we have a pair of uniform (in the sense that the constant $C > 0$ below is independent of $\sigma$) estimates
\begin{equation} \label{intro; two-body second microlocal Fredholm estimates}
\| u \|_{H^{m_{\pm}, \mathsf{r}_{\pm}, l_{\pm}}_{\mathrm{sc,b}} } \leq C ( \| P( \sigma ) \|_{H_{\mathrm{sc,b}}^{m_{\pm}-2, \mathsf{r}_{\pm}+ 1, l_{\pm} + 2}} + \| u \|_{H_{\mathrm{sc,b}}^{M, N ,L} } )
\end{equation}
for all $u \in \mathcal{X}_{\mathrm{sc,b}}^{m_{\pm}, \vor_{\pm}, l_{\pm}}$ and all $\sigma \in S$, where $S \subset [0,\infty)$ is compact; the orders $\mathsf{r}_{\pm}$, $l_{\pm}$ satisfy suitable monotonicity (required for $\vor_{\pm}$ only) and threshold conditions, in particular
\begin{equation} \label{intro; threshold conditions for l}
| l_{\pm} + 1 | \leq \frac{n-2}{2},
\end{equation}
and that $M, N, L \in \mathbb{R}$ are sufficiently negative (see \cite[Proposition 5.6]{AndrasSM} for the precise statement of these conditions). Here, we also define
\begin{equation*}
\begin{gathered}
\mathcal{X}_{\mathrm{sc,b}}^{m_{\pm}, \vor_{\pm}, l_{\pm}} =\mathcal{X}_{\mathrm{sc,b}}^{m_{\pm}, \vor_{\pm}, l_{\pm}}(\sigma) \coloneq \{ u \in H_{\mathrm{sc,b}}^{m_{\pm}, \vor_{\pm}, l_{\pm}} ( \overline{\mathbb{R}^{n}} ) : P( \sigma ) u \in H_{\mathrm{sc,b}}^{m_{\pm} - 2, \vor_{\pm} + 1, l_{\pm} + 2} (\overline{\mathbb{R}^{n}}) \}, \\
\mathcal{Y}_{\mathrm{sc,b}}^{m_{\pm}, \vor_{\pm}, l_{\pm}} \coloneq H_{\mathrm{sc,b}}^{m_{\pm} - 2, \vor_{\pm} + 1, l_{\pm} + 2}( \overline{\mathbb{R}^{n}} ).
\end{gathered}
\end{equation*}
Then the maps 
\begin{equation} \label{intro; two-body second microlocal Fredholm map}
P(\sigma): \mathcal{X}_{\mathrm{sc,b}}^{m_{\pm}, \vor_{\pm}, l _{\pm}} \rightarrow \mathcal{Y}_{\mathrm{sc,b}}^{m_{\pm}, \vor_{\pm}, l_{\pm}}
\end{equation}
are Fredholm, and in fact invertible if there are no bound state nor half-bound states at zero energy. In other words, an analogue to (\ref{a quick review of the two-body problem subsection}) is satisfied for the zero-energy problem \emph{provided we work on the scale of second microlocalized, sc,b-Sobolev spaces}. Better yet, the invertibility holds uniformly as $\sigma \rightarrow 0$, in the sense that the inverses are uniformly bounded as maps $H_{\mathrm{sc,b}}^{m_{\pm} - 2, \vor_{\pm} + 1, l_{\pm} + 2}( \overline{\mathbb{R}^{n}} ) \rightarrow H_{\mathrm{sc,b}}^{m_{\pm}, \vor_{\pm}, l_{\pm}}( \overline{\mathbb{R}^{n}} )$. So in particular, we also understand how the Fredholm framework transitions from positive to zero-energy.  \par

\begin{remark}
Let $H_{\mathrm{b}}^{m,l}( \overline{\mathbb{R}^{n}} )$, $m,l \in \mathbb{R}$ be the b-Sobolev spaces such that spatial decay is (unconventionally) measured with respect to the scattering, or in this case simply the Euclidean $L^{2}$ (see \S \ref{b-calculus subsection} below for clarifications). Then it is a result of \cite[\S 6.2]{MelroseIndex} that for every $m \in \mathbb{R}$ and $| l + 1 | < (n-2)/2${\ep}exactly the threshold condition (\ref{intro; threshold conditions for l}), the map
\begin{equation} \label{intro; Fredholm map zero energy b setting}
P(0) : H_{\mathrm{b}}^{m, l}( \overline{\mathbb{R}^{n}} ) \rightarrow H_{\mathrm{b}}^{m-2, l+2} ( \overline{\mathbb{R}^{n}} )
\end{equation}
is also Fredholm of index $0$. Moreover, we can check that 
\begin{equation*}
H_{\mathrm{b}}^{m,l}( \overline{\mathbb{R}^{n}} ) = H_{\mathrm{sc,b}}^{m, m +l, l}( \overline{\mathbb{R}^{n}} ) , \quad 
H_{\mathrm{b}}^{m-2, l + 2}( \overline{\mathbb{R}^{n}} ) = H_{\mathrm{sc,b}}^{m-2, m + l, l + 2}( \overline{\mathbb{R}^{n}} ).
\end{equation*}
Thus, in a sense, the results of \cite{AndrasSM} constitute a theory that connects the Fredholm maps (\ref{intro; Fredholm map zero energy b setting}) and (\ref{Fredholm statement two-body case}) in a uniform manner. See also \cite{Andrewbpaper1, Andrewbpaper2} for an alternate approach via directly constructing a parametrix for the resolvent family.
\end{remark}

The study of second microlocalization dates back to the work of Bony \cite{BonySM84}, where it was introduced at a homogenous Lagrangian $\Lambda$ of $T^{\ast} M \backslash o$, $M$ is a smooth manifold. Thus, $\Lambda$ can also be understood as a Legendrian submanifold of $S^{\ast} M$. More recently, this study is extended to the semiclassical setting in \cite{AndrasJaredSM09}, where the construction was done locally by writing $\Lambda$ -- now a general Lagrangian submanifold, in the model case $\mathbb{R}^{n} \times \{ 0 \}$. \par

In the setting of Euclidean scattering calculus, it is well known (see \cite[Proposition 8]{MelroseMaciej96}) that conjugation by Fourier transform defines an isomorphism $\Psi^{m,r}_{\mathrm{sc}}( \overline{\mathbb{R}^{n}} ) \cong \Psi^{r,m}_{\mathrm{sc}}( \overline{\mathbb{R}^{n}} )$. Thus, it is equally natural to consider second microlocalization at a Legendrian submanifold $\mathcal{L} \subset  { ^{\mathrm{sc}}T^{\ast}_{\mathbb{S}^{n-1}}} \overline{\mathbb{R}^{n}}$ (or more generally, a Legendrian submanifold $\mathcal{L} \subset { ^{\mathrm{sc}}T^{\ast}_{\partial M} }M$, where $M$ is a smooth, compact manifold with boundary). Thus, $\Psi^{m,r,l}_{\mathrm{sc,b}}( \overline{\mathbb{R}^{n}} )$ is simply the local model
\begin{equation}
\label{introduction second microlocal operators at o}
\Psi^{m,r,l}_{\mathrm{sc,2}}( \overline{\mathbb{R}^{n}} ; o_{\mathbb{S}^{n-1}} ),
\end{equation}
where the subscript `sc,2' indicates directly that members of (\ref{introduction second microlocal operators at o}) are the scattering operators second microlocalized at $o_{\mathbb{S}^{n-1}}$. It is then likely (though not written down anywhere to the best of our knowledge) that (\ref{introduction second microlocal operators at o}) can be used to define the more general class 
\begin{equation}
\label{introduction second microlocal operators at L}
\Psi^{m,r,l}_{\mathrm{sc}, 2}( \overline{\mathbb{R}^{n}} ; \mathcal{L} ), \quad m,r,l \in \mathbb{R},
\end{equation}
where elements of (\ref{introduction second microlocal operators at L}) are required to be microlocalized on the phase space $[ \overline{ ^{\mathrm{sc}}T^{\ast}} \overline{\mathbb{R}^{n}} ; \mathcal{L} ]$, i.e., their symbols should be degenerate at $\mathcal{L}$ in the geometric sense that they are conormal to the front face of $[ \overline{ ^{\mathrm{sc}}T^{\ast}} \overline{\mathbb{R}^{n}} ; \mathcal{L} ]$. 
\begin{remark}
By the same idea, we speculate that Vasy's construction can also be used for the second microlocalization at a Legendrian submanifold of $\mathcal{L} \subset S^{\ast} \overline{\mathbb{R}^{n}}$ (or more generally for $S^{\ast} M$). Thus, one should be able to define second microlocalization more precisely in this case as well. Again, this is not written anywhere to the best of our knowledge.
\end{remark}
\begin{remark}
The zero-energy problem is also tackled by Vasy through another, \emph{Lagrangian regularity} approach, which still makes use of second microlocalization, but not (directly) variable orders. We refer the interested readers to \cite{AndrasLagrangian1,AndrasLagrangian2} for more details.
\end{remark}

\subsection{Second microlocalization for the three-body problem}
Let us now return to the three-body problem. Then the discussion in \S \ref{motivation subsection} suggests that we should define a new class of operators
\begin{equation} \label{introduction second microlocalized operators in the three-body setting}
\Psi_{\mathrm{3sc,2}}^{m,r,l,b} \big( [ \overline{\mathbb{R}^{n}} ; \mathcal{C} ] ; \bigcup_{\tindex \in \ind} \overline{{^{\mathrm{3sc}}\pi_{\ff}^{-1}}( o_{\mathcal{C}^{\tindex}} )} \, \big), \quad m,r,l,b \in \mathbb{R},
\end{equation}
whose elements are microlocalized on the blow-up
\begin{equation} \label{the correct blowup}
\big[  [ \overline{\mathbb{R}^{n}} ; \mathcal{C} ] \times \overline{\mathbb{R}^{n}} ; 
\bigcup_{\tindex \in \ind} \overline{{^{\mathrm{3sc}}\pi_{\ff}^{-1}}( o_{\mathcal{C}^{\tindex}} )} \,  \big] ,
\end{equation}
where we recall that $o_{\mathcal{C}^{\tindex}}$ is the zero section in ${ ^{\mathrm{sc}}T^{\ast}_{\mathcal{C}^{\tindex}}}X^{\tindex}$. Here ${^{\mathrm{3sc}}}\pi_{\ff}$ is the natural projection 
\begin{equation*}
{ ^{\mathrm{3sc}}\pi_{\ff} } : \ff \times \mathbb{R}^{n} \cong \mathcal{C}_{\tindex} \times \mathbb{R}^{n_{\tindex}}_{\zeta_{\tindex}} \times {^{\mathrm{sc}}T^{\ast}} X^{\tindex} \rightarrow { ^{\mathrm{sc}}T^{\ast} X^{\tindex} }, \quad \tindex \in \ind,
\end{equation*}
though we could also view ${^{\mathrm{3sc}}\pi_{\ff}}$ as being orthogonal to the projection in (\ref{three-body projection}), which the notation suggests already. We remark that ${^{\mathrm{3sc}}\pi_{\ff}^{-1}}(o_{\mathcal{C}^{\tindex}})$ is simply a more compact notation for $\mathcal{C}_{\tindex} \times \mathbb{R}^{n_{\tindex}}_{\zeta_{\tindex}} \times o_{\mathcal{C}^{\tindex}}$. \par

Now, the reason why (\ref{the correct blowup}) is the correct blow-up to make is because it resolves both problems outlined at the end of \S \ref{motivation subsection} in the following senses: 
\begin{itemize}
\item Resolution of problem (1): Upon blowing up $\mathcal{C}_{\tindex} \times \mathbb{R}^{n_{\tindex}}_{\zeta_{\tindex}} \times o_{\mathcal{C}^{\tindex}}$, we introduce projective coordinates near $\Sigma_{\partial \ff} ( y_{\tindex}, \zeta_{\tindex} ) \subset \{ ( y_{\tindex}, \zeta_{\tindex} )\} \times o_{\mathcal{C}^{\tindex}}$, $( y_{\tindex}, \zeta_{\tindex} ) \in \Sigma_{\mathrm{t}}$. We can check that the flow of $\sH_{p,\mf}$ lifts to become non-trivial again at the lift of all such $\Sigma_{\partial \ff} ( y_{\tindex}, \zeta_{\tindex} )$, which is akin to the effects of blowing up $o_{\mathbb{S}^{n-1}} \subset {^{\mathrm{sc}}T^{\ast}} \mathbb{R}^{n}$ for the zero-energy problem in the two-body case. Indeed, we have already mentioned that $\Sigma_{\partial \ff}(y_{\tindex}, \zeta_{\tindex})$ is simply the scattering characteristic set for $\hat{P}_{\tindex}$ at $(y_{\tindex}, \zeta_{\tindex})$. In particular, the lifts of $\mathcal{R}_{\mathrm{n},\pm}$ and $\mathcal{R}_{\mathrm{sc}, \pm}$ remain disjoint upon intersecting with $({^{\mathrm{3sc}}\pi_{\ff}^{\perp}})^{-1}( \Sigma_{\mathrm{t}} )$ as well.
\item Resolution of problem (2): Recall that for finite energy in the free variables (i.e., if we ignore the large-parameter behavior as $|\zeta_{\tindex}| \rightarrow \infty$), indicial operators of three-body operators are in a sense microlocalized on $\mathcal{C}_{\tindex} \times \mathbb{R}^{n_{\tindex}}_{\zeta_{\tindex}} \times { ^{\mathrm{sc}}T^{\ast} X^{\tindex} }$. Thus, upon carrying out the blow-up at $\mathcal{C}_{\tindex} \times \mathbb{R}^{n_{\tindex}}_{\zeta_{\tindex}} \times o_{\mathcal{C}^{\tindex}}$, we expect that indicial operators for the second microlocalized operators be microlocalized on the blow-up 
\begin{equation} \label{intro: first appearance; phase space of the indicial operator at dff}
\mathcal{C}_{\tindex} \times \mathbb{R}^{n_{\tindex}}_{\zeta_{\tindex}} \times [ {^{\mathrm{sc}}T^{\ast}}X^{\tindex} ; o_{\mathcal{C}^{\tindex}} ],
\end{equation}
and so form a parametrized family of operators in $\Psi_{\mathrm{sc,b}}(\overline{\mathbb{R}^{}}) \coloneq \bigcup_{m,r,l \in \mathbb{R}} \Psi_{\mathrm{sc,b}}^{m,r,l}( \overline{\mathbb{R}^{n}} )$. It follows from the above discussion that one now has a uniformly Fredholm, in fact invertibility theory for $\hat{P}_{\tindex}(|\zeta_{\tindex}|)$, even as $|\zeta_{\tindex}| \rightarrow \lambda$ from below. Hence, the transition from transversal to tangential phenomena is understood.
\end{itemize}

We will now briefly discuss how to construct the class of second microlocalized operators in question. By disjointness, henceforth we are allowed to consider the case where $\mathcal{C}$ has just one element, i.e., $\mathcal{C} = \mathcal{C}_{\tindex}$. \par

Our strategy will be to follow the philosophy of \cite{AndrasSM}. Thus, we will start by establishing a `converse' perspective in the three-body setting, which will take the form of a new algebra of operators acting on the position space defined by the iterated blow-up
\begin{equation} \label{iterated blow-up}
X\coloneq [ [ \overline{\mathbb{R}^{n}} ; \mathcal{C}_{\tindex} ] ; \ff \cap \mf ].
\end{equation}
The boundary of (\ref{iterated blow-up}) consists of three faces: $\mathrm{dmf}$, which is the lift of $\mf$; $\dff$, which is the lift of $\ff$; and a new face $\cf$, which is the lift of $\mf \cap \ff$, and has the structure of a fibered cone connecting the first two faces. We will refer the new algebra as the conormal, `three-cone' algebra, which is filtered over spaces of conormal three-cone operators. \par

We will write $\Psi_{\mathrm{3coc}}^{m,r,l,b}(\Xd)$, $m,r,l,b \in \mathbb{R}$ for the spaces of conormal three-cone operators. Here $m$ is the order of microlocal differential regularity, which is measured with respect to the Lie algebra of three-cone vector fields $\mathcal{V}_{\mathrm{3co}}(X)$, where $\mathcal{V}_{\mathrm{3co}}(\Xd)$ is the span over $\mathcal{C}^{\infty}(\Xd)$ by $\{ \partial_{z_{\tindex, 1}}, ... , \partial_{z_{\tindex,n_{\tindex}}}, \langle z^{\tindex} \rangle \partial_{z^{\tindex}_1}, ... , \langle z^{\tindex} \rangle \partial_{z^{\tindex}_{n^{\tindex}}}   \} $ in a neighborhood of $\cf$, and agrees everywhere else with the standard Euclidean vector fields. On the other hand, $r,l,b$ measure spatial decay respectively at $\dmf$, $\dff$ and $\cf$. In particular, if $m \in \mathbb{N}_0$, then we can understand $\Psi_{\mathrm{3coc}}^{m,0,0,0}(\Xd)$ as the natural microlocalization for those differential operators of order $m$ which are generated by $\mathcal{V}_{\mathrm{3co}}(\Xd)$ with coefficients in $\mathcal{C}^{\infty}(\Xd)$. \par

Furthermore, conormal three-cone operators are microlocalized on the fiber-compactified \emph{three-cone cotangent bundle} $\overline{^{\mathrm{3co}}T^{\ast}} \Xd$. However, only regularity in the differential sense and decay at $\dmf$ is microlocal in the symbolic sense, i.e., they are captured by symbolic calculations at ${ ^{\mathrm{3co}}S^{\ast} }\Xd$ and $\overline{^{\mathrm{3co}}T^{\ast}}_{\dmf} \Xd$ respectively. In order to capture decay at $\cf$ and $\dff$, one needs more global considerations, for which indicial operators are again required.

The spaces of conormal three-cone operators will be constructed explicitly in \S \ref{the conormal three-cone operators section} below. We will see that $A \in \Psi_{\mathrm{3coc}}^{m,r,l,b}(\Xd)$ restricts away from $\cf$ to an element of $\Psi_{\mathrm{3scc}}^{m,r,l}( [ \overline{\mathbb{R}^{n}} ; \mathcal{C}_{\tindex} ] )$. In fact, $A$ has a scattering structure at $\dmf$. Moreover, we can locally write $A$ near $\dff$ as the partial quantization in the free variables (i.e., $(z_{\tindex},\zeta_{\tindex})$) of a family of conormal b-operators with a large-parameter behavior. Similarly, we can also write $A$ near $\cf$ as the partial quantization in the free variables of a family of conormal cone operators, and with a large-parameter behavior as well. By interpreting the three-cone operators locally as partial quantizations, basic operations such as composition can also be easily understood.\par

Here, the spaces of conormal cone operators can be understood as operators acting on a cone with a scattering structure at the cone-end and a conormal b-structure at the cone-tip. When such operators arise in the context of the conormal three-cone operators, the cone in question will often by written as $[ \hat{X}^{\tindex} ; \{ 0 \} ]$, where $\hat{X}^{\tindex} \cong \overline{\mathbb{R}^{n^{\tindex}}} $, though $\hat{X}^{\tindex}$ is also explicitly determined, for we have
\begin{equation} \label{introduction equation: the canonical identification of cf}
\cf \cong \mathcal{C}_{\tindex} \times [ \hat{X}^{\tindex} ; \{ 0 \} ]
\end{equation}
as a canonical identification. Then $[ \hat{X}^{\tindex} ; \{ 0 \} ]$ is a smooth, compact manifold with two boundary faces: $\mathcal{C}_{\infty}^{\tindex}$, which is the lift of the original boundary of $\hat{X}^{\tindex}$; and $\mathcal{C}_{0}^{\tindex}$, which is the lift of $\{ 0 \}$. We will interpret $\mathcal{C}^{\tindex}_{\infty}$ at the cone-end and $\mathcal{C}^{\tindex}_{0}$ the cone-tip. Moreover, under identification (\ref{introduction equation: the canonical identification of cf}), $\mathcal{C}_{\tindex} \times \mathcal{C}^{\tindex}_{\infty}$ intersects $\dmf$ and $\mathcal{C}_{\tindex} \times \mathcal{C}_{0}^{\tindex}$ intersects $\dff$. In fact, this is exactly how we connect the structures of conormal three-cone operators at $\dmf$ and $\dff$. \par

We will write $\Psi_{\mathrm{coc}}^{m,r,l}( [ \hat{X}^{\tindex} ; \{ 0 \} ] )$, $m,r,l \in \mathbb{R}$ for the spaces of conormal cone operators on $[ \hat{X}^{\tindex} ; \{ 0 \} ]$. Here $m$ is the index of microlocal differential regularity measured with respect to the Lie algebra of co-vector fields $\mathcal{V}_{\mathrm{co}}( [ \hat{X}^{\tindex} ; \{ 0 \} ] )$. Such vector fields agree with $\mathcal{V}_{\mathrm{b}}([ \hat{X}^{\tindex} ; \{ 0 \} ])$ and the Euclidean (or scattering) vector fields on $\hat{X}^{\tindex}$ respectively away from $\mathcal{C}^{\tindex}_{\infty}$ and $\mathcal{C}^{\tindex}_{0}$. Meanwhile, $r$ and $l$ measure spatial decay respectively at $\mathcal{C}^{\tindex}_{0}$ and $\mathcal{C}^{\tindex}_{\infty}$. We refer the readers to \S \ref{subsection the cone calculus} for a detailed introduction to the conormal cone algebra.  \par

To connect the conormal three-cone operators with the conormal three-body operators (similar to how scattering operators and the conormal b-operators are connected in \S \ref{an overview of second microlocalization subsection}, i.e., via blow-ups in phase spaces), it will be conceptually convenient to resolve the conormal three-body operators separately at the spatial corner of $[ \overline{\mathbb{R}^{n}} ; \mathcal{C}_{\tindex} ]$, i.e., $\mf \cap \ff$. Upon doing so, we will construct new spaces of conormal, \emph{decoupled} (in the sense that $\mf$ and $\ff$ are now separated by $\cf$) three-body operators $\Psi_{\mathrm{d3scc}}^{m,r,l,\nu}(\Xd)$, $m,r,l,\nu \in \mathbb{R}$, where $m$ is the order of standard differential regularity, while $r, l,\nu$ measure spatial decay at $\dmf$, $\dff$ and $\cf$ respectively. \par

By construction, elements of $\Psi_{\mathrm{d3scc}}^{m,r,l,\nu}(\Xd)$ are operators acting on $X$, and as such have the same base space as the conormal three-cone operators. Moreover, they can be microlocalized on the phase space
\begin{equation*}
\overline{^{\mathrm{d3sc}}T^{\ast}}\Xd  =  \big[ \overline{^{\mathrm{3sc}}T^{\ast}} [ \overline{\mathbb{R}^{n}} ; \mathcal{C}_{\tindex} ]; \overline{^{\mathrm{3sc}}T^{\ast}}_{\mf \cap \ff} [ \overline{\mathbb{R}^{n}} ; \mathcal{C}_{\tindex} ]  \big] \cong X \times \overline{\mathbb{R}^{n}}.
\end{equation*} 
A crucial observation in this paper is the following diffeomorphism:
\begin{equation} 
\label{introduction second microlocal diffeomorphism}
\left[ \overline{^{\mathrm{d3sc}}T^{\ast}} \Xd ; \beta_{\mathrm{d3sc}}^{\ast} ( \overline{\tscf^{-1}( 
{\sco} )} ) \right] \cong 
\left[ \overline{ ^{\mathrm{3co}} T^{\ast}} \Xd  ; \itccf ( \mathcal{C}_{\tindex} \times {^{\mathrm{co}}S^{\ast} [ \hat{X}^{\tindex} ; \{ 0 \} ]} ) \right],
\end{equation}
which is the analogue of (\ref{second microlocalization diffeomorphism in the introduction}) in the three-body setting. Much like (\ref{second microlocalization diffeomorphism in the introduction}), the diffeomorphism (\ref{introduction second microlocal diffeomorphism}) again restricts to the identity map in the interiors. \par

Here in the introduction, we shall not attempt to explain the meaning of
\begin{equation} \label{introduction 3co corner where we blow up}
\itccf ( \mathcal{C}_{\tindex} \times {^{\mathrm{co}}S^{\ast} [ \hat{X}^{\tindex} ; \{ 0 \} ]} )
\end{equation}
on the right hand side of (\ref{introduction second microlocal diffeomorphism}). See \S \ref{subsection diffeomorphism of the phase spaces}, and in particular the discussions before Proposition \ref{second microlocal diffeomorphism} for clarification. On the left hand side of (\ref{introduction second microlocal diffeomorphism}), we are simply writing 
\begin{equation*}
\beta_{\mathrm{d3sc}}:  \overline{^{\mathrm{d3sc}}T^{\ast}} \Xd \rightarrow  \overline{ ^{\mathrm{3sc}}T^{\ast} } [ \overline{\mathbb{R}^{n}} ; \mathcal{C}_{\tindex} ]
\end{equation*}
for the natural blow-down map of $\overline{^{\mathrm{d3sc}}T^{\ast}} \Xd$. In \S \ref{subsection definition of the second microlocalized algebra}, we will show that diffeomorphism (\ref{introduction second microlocal diffeomorphism}) gives rise to the spaces of second microlocalized, decoupled three-body operators
\begin{equation}
\label{introduction d3sc 2 notation}
\Psi_{\mathrm{d3sc,2}}^{m,r,l,\nu,b} \big( \Xd ;   \beta_{\mathrm{d3sc}}^{\ast} ( \overline{{^{\mathrm{3sc}}\pi_{\ff}^{-1}}( o_{\mathcal{C}^{\tindex}} )} )  \big) =  \Psi_{\mathrm{d3sc,3co}}^{m,r,l, \nu, b }(\Xd), \quad m,r,l, \nu, b \in \mathbb{R}.
\end{equation} 
This procedure is akin to how (\ref{second microlocalization diffeomorphism in the introduction}) gives rise to the sc,b-operators (\ref{introduction: sc,b operators all indices are arbitrary}), i.e., one modifies the conormal three-cone operators $\Psi_{\mathrm{3coc}}^{m,r,l,b}(\Xd)$ at the symbolic level. However, this modification is now more involved since the right hand side of (\ref{introduction second microlocal diffeomorphism}) is no longer a corner blow-up. Nevertheless, it is still a blow-up at a submanifold of ${ ^{\mathrm{3co}}S^{\ast}_{\cf}}\Xd$ (which is a corner), i.e., we can show that (\ref{introduction 3co corner where we blow up}) is contained in ${ ^{\mathrm{3co}}S^{\ast}_{\cf}}\Xd$. Thus, through careful, symbolic calculations, one could show that (\ref{introduction d3sc 2 notation}) is indeed well-defined. 
\par

Let us also remark that either of the notations in (\ref{introduction d3sc 2 notation}) is a sensible choice: the subscript `d3sc,2' on the left hand side of (\ref{introduction d3sc 2 notation}) indicates that it is a space of decoupled three-body operators which are second microlocalized at $\beta_{\mathrm{d3sc}}^{\ast} ( \overline{{^{\mathrm{3sc}}\pi_{\ff}^{-1}}( o_{\mathcal{C}^{\tindex}} )} ) $; the subscript `d3sc,3co' on the right hand side of (\ref{introduction d3sc 2 notation}) indicates that it is a space of operators with simultaneously the decoupled three-body structure and the three-cone structure.  \par

In fact, since $ \overline{\tscf^{-1}( \sco )} \subset \overline{ ^{\mathrm{3sc}}T^{\ast}}_{\mathrm{mf} \cap \ff } [ \overline{\mathbb{R}^{n}} ; \mathcal{C}_{\tindex} ] $, it is known that
\begin{equation}
\label{introduction blow up orders dont matter}
\left[ \overline{^{\mathrm{d3sc}}T^{\ast}} \Xd ; \beta_{\mathrm{d3sc}}^{\ast} ( \overline{\tscf^{-1}( 
{\sco} )} ) \right] = \big[ \overline{^{\mathrm{3sc}}T^{\ast}} [ \overline{\mathbb{R}^{n}} ; \mathcal{C}_{\tindex} ] ;   \overline{\tscf^{-1}( 
{\sco} )} ; \overline{^{\mathrm{3sc}}T^{\ast}}_{\mf \cap \ff} [ \overline{\mathbb{R}^{n}} ; \mathcal{C}_{\tindex} ] \big],
\end{equation}
in the sense that we could also carry out the blow-up at $\overline{\tscf^{-1}( \sco )}$ first instead. Moreover, since $\overline{^{\mathrm{3sc}}T^{\ast}}_{\mf \cap \ff} [ \overline{\mathbb{R}^{n}} ; \mathcal{C}_{\tindex} ]$ lifts to a corner face of 
\begin{equation*}
\big[ \overline{^{\mathrm{3sc}}T^{\ast}} [ \overline{\mathbb{R}^{n}} ; \mathcal{C}_{\tindex} ] ;   \overline{\tscf^{-1}( 
{\sco} )} \, \big] = \big[ [ \overline{\mathbb{R}^{n}} ; \mathcal{C}_{\tindex} ] \times \overline{\mathbb{R}^{n}} ; \overline{\tscf^{-1}( 
{\sco} )} \, \big],
\end{equation*}
conormality on the latter space is again insensitive to the second blow-up in (\ref{introduction blow up orders dont matter}). Thus, the same argument wchih motivated us to define (\ref{second microlocalized algebra tautology case}) shows that it is reasonable to define
\begin{equation*}
\Psi_{\mathrm{3sc,2}}^{m,r,l,b} \left(  [ \overline{\mathbb{R}^{n}} ; \mathcal{C}_{\tindex} ] ; \overline{\tscf^{-1}(\sco)} \, \right) \coloneq \Psi_{\mathrm{d3sc,3co}}^{m,r,l,r+l, b} (\Xd)
\end{equation*}
as well, yileding the desired algebra of operators. \par

Nevertheless, a further, though relatively mild resolution at the (lift of the) fiber infinity 
\begin{align}
\begin{split} 
\label{second microlocal diffeomorphism resol intro}
& \left[ \overline{^{\mathrm{d3sc}}T^{\ast}} \Xd ; \beta_{\mathrm{d3sc}}^{\ast} ( \overline{\tscf^{-1}( 
{\sco} )} ) , \idtscdff( \mathcal{C}_{\tindex} \times {^{\mathrm{sc}}S^{\ast}} X^{\tindex} ) \right] \\
& \quad \quad \cong
\left[ \overline{ ^{\mathrm{3co}} T^{\ast}} \Xd  ; \itccf (\mathcal{C}_{\tindex} \times {^{\mathrm{co}}S^{\ast} [ \hat{X}^{\tindex} ; \{ 0 \} ]} ); \itcodff( \mathcal{C}_{\tindex} \times { ^{\mathrm{b}}S^{\ast} X^{\tindex} } ) \right],
\end{split}
\end{align}
is necessary to incorporate variable orders. Here, we will again not elaborate on the meanings of $\idtscdff( \mathcal{C}_{\tindex} \times {^{\mathrm{sc}}S^{\ast}} X^{\tindex} )$ and $ \itcodff( \mathcal{C}_{\tindex} \times { ^{\mathrm{b}}S^{\ast} X^{\tindex} } )$ in (\ref{second microlocal diffeomorphism resol intro}), and just remark that they are contained in ${^{\mathrm{d3sc}}S^{\ast}_{\dff}}\Xd$ and ${ ^{\mathrm{3co}}S^{\ast}_{\dff}}\Xd$ respectively. Thus ultimately, we would actually like to study the spaces of more refined operators
\begin{equation}
\label{further resolved second microlocalized operators introduction section}
\Psi_{\mathrm{d3sc,3co,res}}^{m,r,l,\nu,b,s}(\Xd), \quad m, r, l, \nu, b, s \in \mathbb{R},
\end{equation}
such that elements of (\ref{further resolved second microlocalized operators introduction section}) are microlocalized on (\ref{second microlocal diffeomorphism resol intro}), with the point being that \emph{every} index above can now be replaced by an \emph{admissible} variable order. Moreover, it can also be checked that $\Psi_{\mathrm{d3sc,3co}}^{m,r,l,\nu,b}(\Xd)$ is a natural subspace of $ \Psi_{\mathrm{d3sc,3co,res}}^{m,r,l,\nu,b,m+l}(\Xd)$. See \S\S \ref{a further resolution at fiber infinity}--\ref{Variable orders compatibility subsection} below.
 \par

Now, suppose we denote (\ref{second microlocal diffeomorphism resol intro}) by 
\begin{equation*}
\psf \Xd.
\end{equation*}
Moreover, let us write
\begin{equation*}
\text{$^{\mathrm{d3sc,3co,res}}S^{\ast}X$, $\overline{^{\mathrm{d3sc,3co,res}}T^{\ast}}_{\mathrm{dmf}} \Xd$, $\overline{^{\mathrm{d3sc,3co,res}}T^{\ast}}_{\dff} \Xd$, $\dtsccf$, $\tcocf$, $\rf$}
\end{equation*}
respectively for the lifts of 
\begin{equation*}
\text{${ ^{\mathrm{d3sc}}S^{\ast} }X$, $\overline{ ^{\mathrm{d3sc}}T^{\ast} }_{\dmf}X$, $\overline{^{\mathrm{d3sc}}T^{\ast}}_{\dff}X$, $\beta_{\mathrm{d3sc}}^{\ast}( \overline{{ ^{\mathrm{3sc}}\pi_{\ff}^{-1} }( o_{\mathcal{C}_{\tindex}} )} )$, $\overline{ ^{\mathrm{d3sc}}T^{\ast} }_{\cf} X$, $\idtscdff( \mathcal{C}_{\tindex} \times {^{\mathrm{sc}}S^{\ast}} X^{\tindex} )$}
\end{equation*}
to the first line of (\ref{second microlocal diffeomorphism resol intro}), or equivalently 
\begin{equation*}
\begin{gathered}
\text{${ ^{\mathrm{3co}}S^{\ast} }X$, $\overline{ ^{\mathrm{3co}}T^{\ast} }_{\dmf}X$, $\overline{^{\mathrm{3co}}T^{\ast}}_{\dff}X$, $\overline{ ^{\mathrm{3co}}T^{\ast} }_{\cf} X$,} \\ 
\text{$ \itccf ( \mathcal{C}_{\tindex} \times {^{\mathrm{co}}S^{\ast} [ \hat{X}^{\tindex} ; \{ 0 \} ]} )$, $\itcodff( \mathcal{C}_{\tindex} \times { ^{\mathrm{b}}S^{\ast} X^{\tindex} } )$}
\end{gathered}
\end{equation*}
to the second line of (\ref{second microlocal diffeomorphism resol intro}). Then for elements of $\Psi_{\mathrm{d3sc,3co,res}}^{m,r,l,\nu,b,s}(\Xd)$, the notion of `regularity', even microlocally, is complicated (but is similar to how regularity is understood for elements of $\Psi_{\mathrm{sc,b}}( \overline{\mathbb{R}^{n}} )$). Indeed, it no longer makes sense to distinguish between `differential regularity' and `spatial decay', since, for instance, $\psf \Xd$ is not even a fiber bundle (although we will keep this notation throughout this paper for convenience). \par

From the symbolic viewpoint, a natural substitute is to consider decay of symbols at the boundary faces of $\psf \Xd$. However, it turns out that 
\begin{equation} \label{introduction: microlocally symbolic faces}
\text{$^{\mathrm{d3sc,3co,res}}S^{\ast}X$, $\psf_{\dmf} \Xd$, $\dtsccf$, $\rf$}
\end{equation}
are the only reasonable `microlocally symbolic' faces in the following sense: the order of decay for the (principal) symbol of any d3sc,3co,res-operator at (\ref{introduction: microlocally symbolic faces}) determines completely the corresponding order of the operator. Moreover, the leading order decay at (\ref{introduction: microlocally symbolic faces}) is captured simultaneously by a principal symbol map, which we will introduce in \S \ref{subsection principal symbol}. In (\ref{further resolved second microlocalized operators introduction section}), we use the indices $m,r, \nu, s$ to measure microlocal decay at the faces (\ref{introduction: microlocally symbolic faces}), respectively.

In a different, more global sense, the indices $l,b$ measure spatial decay respectively at $\dff$ and $\cf$. Roughly speaking, such notions of decay can be understood as being lifted from the spatial decay at $\dff$, $\cf$ in the conormal three-cone setting{\ep}much like how for elements of $\Psi_{\mathrm{sc,b}}^{m,r,l}( \overline{\mathbb{R}^{n}} )$, the index $l$ measures not only microlocal decay at the b-face, but also more globally at the lifted spatial decay (i.e., at $\partial \overline{\mathbb{R}^{n}}$) in the b-sense as well. \par

To capture leading order decay at $\dff$ and $\cf$, we will require the introduction of indicial operators at these faces. If $A \in \Psi_{\mathrm{d3sc,3co,res}}^{m, r, l, \nu, b, s}( \Xd )$ satisfies suitable classicality conditions, then the indicial operators at $\dff$ and $\cf$ will be respectively given by
\begin{gather}
{^{\mathrm{d3sc,3co,res}}\hat{N}_{\dff, l}}(A) \in \mathcal{C}^{\infty}( \mathcal{C}_{\tindex} \times \mathbb{R}^{n_{\tindex}}_{\zeta_{\tindex}} ; \Psi_{\mathrm{sc,b}}^{s-l, r - l, b - l} ( X^{\tindex} )  ), \label{intro; indicial operator at dff} \\
{^{\mathrm{d3sc,3co,res}}\hat{N}_{\cf, b}}(A) \in \mathcal{C}^{\infty}( \mathcal{C}_{\tindex} \times \mathbb{R}^{n_{\tindex}}_{\zeta_{\tindex}} ; \Psi_{\mathrm{bc}}^{\nu - b/2, r - b/2, l - b/2} ( [ \hat{X}^{\tindex} ; \{ 0 \} ] ) ). \label{intro; indicial operator at cf}
\end{gather}
In particular, (\ref{intro; indicial operator at cf}) can be mirolocalized on (\ref{intro: first appearance; phase space of the indicial operator at dff}) as required. These indicial operators will be introduced in \S \ref{subsection partial classicality and indicial operators}. However, unlike the indicial operators at $\ff$ for three-body operators, the behavior of (\ref{intro; indicial operator at dff}) and (\ref{intro; indicial operator at cf}) as $|\zeta_{\tindex}| \rightarrow \infty$ will no longer be large-parameter. Instead, they will also exhibit second-microlocal structures, which we will describe in \S \S \ref{subsection second microlocalization for the indicial operators dff} and \ref{subsection second microlocalization for the indicial operators cf}. 

\par

In \S \ref{subsection composition and adjoint}, we will show that 
\begin{equation} \label{intro; d3sc,3co,res calculus}
 \Psi_{\mathrm{d3sc,3co,res}}( \Xd ) \coloneq \bigcup_{m,r,l,\nu,b,s \in \mathbb{R}} \Psi_{\mathrm{d3sc,3co,res}}^{m,r,l,\nu,b,s}(\Xd)
\end{equation}
is indeed a graded algebra, and that the principal symbol and indicial operator maps are multiplicative. Moreover, $\Psi_{\mathrm{d3sc,3co,res}}^{m,r,l,\nu,b,s}(\Xd)$ is closed under taking adjoints (in various senses). These properties thus promote (\ref{intro; d3sc,3co,res calculus}) into a `calculus'. \par

We will also investigate several standard results regarding microlocalization in the second microlocal framework. In particular, we will introduce the operator wavefront set in \S \ref{Operator wavefront set subsection}, as well as the elliptic and characteristic sets in \S \ref{symbolic version of elliptic and characterstic sets}. In \S \ref{subsection Sobolev spaces under second microlocalization}, we will introduce a new scale of Sobolev spaces $H_{\mathrm{d3sc,3co,res}}^{m,r,l,\nu,b,s}(\Xd)$ corresponding to (\ref{further resolved second microlocalized operators introduction section}), and then establish microlocal elliptic regularity estimate (in various senses) in \S \ref{subsection;microlocal elliptic regularity estimates}.

Finally, let us remark that the above discussion remains valid even in the case where we replace $m,r,l,\nu,b,s \in \mathbb{R}$ by \emph{admissible} variable orders $\vom, \vor, \vol, \vov, \vob, \vos \in \mathcal{C}^{\infty}( \psf \Xd )$. See \S \ref{the presence of variable orders section} below for clarification on admissibility and \S \ref{subsection construction of variable order operators} on the construction of the corresponding spaces of operators $\Psf^{\vom, \vor, \vol, \vov, \vob, \vos}(\Xd)$, which depend on some arbitrarily small $\delta > 0$.

\subsection{Rescaled Hamiltonian and propagation estimates}
In what remains of this introduction, we will briefly explain how to obtain the required Fredholm maps within the second microlocal framework. \par

Returning to the discussion on Hamiltonian flow. Let 
\begin{equation*}
\rho_{\dmf}, \rho_{\dtsccf}, \rho_{\tcocf} \in \mathcal{C}^{\infty}( \psf \Xd )
\end{equation*}
be defining functions for $\psf_{\dmf}\Xd$, $\dtsccf$ and $\tcocf$ respectively. Moreover, let 
\begin{equation*}
\bcv \subset \psf_{\dmf} \Xd \cup \dtsccf
\end{equation*}
denote the `symbolic' characteristic set of $P \in \Psi_{\mathrm{d3sc,3co,res}}^{2,0,0,0,0,2}(\Xd)$, in the sense that the principal symbol of $P$ in the  d3sc,3co,res-calculus vanishes on $\Sigma_{\sigma}$. Then we can show that 
\begin{equation}  \label{intro cal 6}
\rho_{\dmf}^{-1} \rho_{\dtsccf}^{-1} \rho_{\tcocf}^{-2} H_{p}
\end{equation}
defines a complete flow on $\bcv$. The dynamic of this flow is detailedly studied in \S \ref{subsection characterization for the flow of the second microlocal dynamic}.  \par

Following the discussion in \S \ref{motivation subsection}, when attempting to propagate microlocal regularity for solutions to $Pu = f$ in $\beta_{\mathrm{3sc}}^{\ast}( \Sigma_{\mathrm{sc}} )$ along the flow of $\sH_{p,\mf}$, the primary challenges emerged with tangential propagation phenomena. In the second microlocal framework, we will replace the flow of $\sH_{p,\mf}$ by that of $\rho_{\dmf}^{-1} \rho_{\dtsccf}^{-1} \rho_{\tcocf}^{-2} H_{p}$.

A central part of this paper is to show that microlocal propagation of regularity for solutions to $Pu = f$ on $\Sigma_{\sigma}$ can be completely understood via the flow of $\rho_{\dmf}^{-1} \rho_{\dtsccf}^{-1} \rho_{\tcocf}^{-2} H_{p}${\ep}much as in the two-body setting, which we have explained briefly in \S \ref{a quick review of the two-body problem subsection}{\ep}so long as $f$ is sufficiently regular. Thus, when this is combined with microlocal elliptic regularity estimate, we would be able to conclude microlocal regularity for $u$ at all of the `symbolic' boundary faces of $\psf_{\dmf} \Xd$, i.e., (\ref{introduction: microlocally symbolic faces}). \par

\begin{remark}
Notice that for this discussion, the `res' part of the operators is actually unnecessary. Indeed, $\bcv$ is compactly contained away from both the fiber infinity and $\rf$, the latter face being the new front face introduced by (\ref{second microlocal diffeomorphism resol intro}). Moreover, we will only be interested in microlocal propagation at the symbolic faces. Thus, our attention here will be restricted to just $\psf_{\dmf} \Xd$ and $\dtsccf$ (though nevertheless including their intersections with $\tcocf$ and $\psf_{\dff}\Xd$). 
\end{remark}

In \S \ref{principal type propagation section}, we will show that standard principal type propagation of regularity estimate holds even along the flow of $\rho_{\dmf}^{-1} \rho_{\dtsccf}^{-1} \rho_{\tcocf}^{-2} H_{p}$. Thus, it suffices to consider the radial sets, which we define as the sets where $\rho_{\dmf}^{-1} \rho_{\dtsccf}^{-1} \rho_{\tcocf}^{-2} H_{p}$ vanishes. \par

In \S \S \ref{subsection rescaled hamiltonian in the second microlocal framework} and \ref{subsection radial points in the second microlocal framework}, we find that there are the following radial sets:
\begin{itemize}
\item $\mathcal{R}_{\mathrm{n}, \dmf, \pm} \subset \psf_{\dmf}\Xd \cap \dtsccf$, which are saddles;
\item $\mathcal{R}_{\mathrm{n},\dff,\pm} \subset \psf_{\dff}\Xd \cap \dtsccf$, which are saddles;
\item $\mathcal{R}_{0,\pm} \subset \tcocf^{\circ} \cap \psf_{\dmf}\Xd$, which are saddles; 
\item $\brpm \subset \psf_{\dmf} \Xd$, which are respectively source (corresponding to the $+$ sign) and sink (corresponding to the $-$ sign). 
\end{itemize} \par

Here, the radial sets $\brpm$ are simply the lifts of the two-body radial sets $\mathcal{R}_{\mathrm{sc}, \pm}$ from ${^{\mathrm{sc}}T^{\ast}_{\mathbb{S}^{n-1}} \overline{\mathbb{R}^{n}} }$ to $\psf_{\dmf} \Xd$. In \S \ref{global radial point estimate section} below, we will show that the usual radial point estimates at $\mathcal{R}_{\mathrm{sc},\pm}$ essentially extend to hold for $\brpm$ as well. In particular, under suitable above threshold requirement, we can conclude microlocal regularity for $u$ at $\brp$; while under suitable below threshold requirement, and moreover assuming that $u$ is microlocally regular in a punctured neighborhood of $\brm$, we can conclude microlocal regularity for $u$ at $\brm$. \par

On the other hand, the radial sets $\mathcal{R}_{\mathrm{n},\dmf, \pm}$, $\mathcal{R}_{\mathrm{n}, \dff, \pm}$ together replace the roles of $\mathcal{R}_{\mathrm{n},\pm}$ in the discussion from \S \ref{motivation subsection}. Indeed, suppose for the moment that we are away from $\tcocf$ (i.e., away from zero-energy in the interaction variables, where only transversal propagation phenomena occur). Then an integral curve $\tilde{\gamma}$ of $\sH_{p,\mf}$ which intersects ${^{\mathrm{3sc}}T^{\ast}_{\ff} [ \overline{\mathbb{R}^{n}} ; \mathcal{C}_{\tindex} ]}$ transversally in forward time lifts to an integral curve $\gamma$ of $\rho_{\dmf}^{-1} \rho_{\dtsccf}^{-1} \rho_{\tcocf}^{-2} H_{p}$ which converges to $\dtsccf \backslash \tcocf$ as $t \rightarrow \infty$. In fact, we must have $\lim_{t \rightarrow \infty} \gamma(t) \in \mathcal{R}_{\mathrm{n}, \dmf, +} \backslash \tcocf$. \par

Moreover, we will find that $\rho_{\dmf}^{-1} \rho_{\dtsccf}^{-1} \rho_{\tcocf}^{-2} H_{p}$ is actually tangent to $\tcocf$, and the flow of $\rho_{\dmf}^{-1} \rho_{\dtsccf}^{-1} \rho_{\tcocf}^{-2} H_{p}$ enables us to propagate microlocal regularity for $u$ on $\dtsccf \backslash \tcocf$ from $\mathcal{R}_{\mathrm{n}, \dmf,+} \backslash \tcocf$ to $\mathcal{R}_{\mathrm{n}, \dff, +} \backslash \tcocf$, where we will meet $\psf_{\dff} \Xd$. Then the `two-body' propagation phenomena occur once again on $\psf_{\dff} \Xd${\ep}much like how they occurred on ${^{\mathrm{3sc}}T^{\ast}_{\ff} [ \overline{\mathbb{R}^{n}} ; \mathcal{C}_{\tindex} ] }$ in the discussion from \S \ref{motivation subsection}. It follows that microlocal regularity for $u$ can be further propagated from $\mathcal{R}_{\mathrm{n}, \dff, +} \backslash \tcocf$ to $\mathcal{R}_{\mathrm{n}, \dff, - } \backslash \tcocf$. \par

Finally, we can propagate microlocal regularity from $\mathcal{R}_{\mathrm{n}, \dff, -} \backslash \tcocf$ to $\mathcal{R}_{ \mathrm{n}, \dmf, +} \backslash \tcocf$, i.e., the flow of $\rho_{\dmf}^{-1} \rho_{\dtsccf}^{-1} \rho_{\tcocf}^{-2} H_{p}$ must eventually return to $\psf_{\dmf} \Xd$. However, we also remark that there are complete integral curves of $\rho_{\dmf}^{-1} \rho_{\dtsccf}^{-1} \rho_{\tcocf}^{-2} H_{p}$ which connect $\mathcal{R}_{\mathrm{n}, \dmf, +} \backslash \tcocf$ to $\mathcal{R}_{\mathrm{n}, \dmf,-} \backslash \tcocf$ directly, and propagation along these curves is necessary to conclude microlocal regularity of $u$ at $\mathcal{R}_{\mathrm{n},\dmf, -} \backslash \tcocf$. \par

The usage of a radial point estimate is required each time we wish to conclude microlocal regularity of $u$ at a radial set, which also requires suitable threshold conditions to be satisfied. See \S \ref{transversal propagation section} below for details. This roundabout analysis is strictly more refined than the transversal propagation analysis carried out in \S \ref{motivation subsection}. \par

Now, the behavior of the $\rho_{\dmf}^{-1} \rho_{\dtsccf}^{-1} \rho_{\tcocf}^{-2} H_{p}$-flow at $\mathcal{R}_{0,\pm}$ is much more interesting. To explain the dynamical significance of $\mathcal{R}_{0,\pm}$, recall that integral curves $\tilde{\gamma}$ of $\sH_{p}$ in $\Sigma_{\mathrm{sc}}$ always reach the two-body radial sets $\mathcal{R}_{\mathrm{sc}, \pm}$ in either forward($-$) or backward($+$) time asymptotics. Suppose that $\lim_{t \rightarrow \infty} \tilde{\gamma}(t)$ reaches precisely the intersection of $\mathcal{R}_{\mathrm{sc}, -}$ and ${^{\mathrm{sc}}T^{\ast}_{\mathcal{C}_{\tindex}}} \overline{\mathbb{R}^{n}}$. Then we will show in \S \ref{subsection: characterization for the flow} that $\beta_{\mathrm{3sc}}^{\ast} \tilde{\gamma}$ must reach ${ ^{\mathrm{3sc}}T^{\ast}_{\ff} [ \overline{\mathbb{R}^{n}} ; \mathcal{C}_{\tindex} ] }$ at zero interaction energy (i.e., when $\zeta^{\tindex} = 0$). Thus, let $\gamma$ be the lift of $\tilde{\gamma}$ to $\psf_{\dmf}\Xd$; then we must have $\lim_{t \rightarrow \infty} \gamma(t) \in \tcocf$. In fact, we will show that $\lim_{t \rightarrow  \infty} \gamma(t) \in \mathcal{R}_{0, -}$. Likewise, if $\tilde{\gamma}$ is such that $\lim_{t \rightarrow -\infty} \tilde{\gamma}(t)$ reaches the intersection of $\mathcal{R}_{\mathrm{sc},+}$ and ${^{\mathrm{sc}}T^{\ast}_{\mathcal{C}_{\tindex}}} \overline{\mathbb{R}^{n}}$, then we must have $\lim_{t \rightarrow -\infty} \gamma(t) \in \mathcal{R}_{0,+}$.  \par

It follows that we have a very different dynamical picture on $\psf_{\dmf} \cap \tcocf$. For instance, in the forward time asymptotics, there are integral curves of $\rho_{\dmf}^{-1} \rho_{\dtsccf}^{-1} \rho_{\tcocf}^{-2} H_{p}$ which travel from $\mathcal{R}_{0, +}$ to $\mathcal{R}_{\mathrm{n}, \dmf, +} \cap \tcocf$. Moreover, the above discussion for the flow of $\rho_{\dmf}^{-1} \rho_{\dtsccf}^{-1} \rho_{\tcocf}^{-2} H_{p}$ on $\dtsccf \backslash \tcocf$ applies verbatim for the flow at $\dtsccf \cap \tcocf$ as well, and we conclude that microlocal regularity of $u$ can also be propagated on $\dtsccf \cap \tcocf$ from $\mathcal{R}_{\mathrm{n}, \dmf, +} \cap \tcocf$ to $\mathcal{R}_{\mathrm{n}, \dmf, -} \cap \tcocf$ through $\mathcal{R}_{\mathrm{n}, \dff, \pm} \cap \tcocf$. This amount of microlocal regularity can then be propagated further to $\brm$. The radial point estimates required for the discussion here will be presented in \S \ref{section: radial point estimates in the tangential directions}.
\par

The complete situation is much more complicated, with the point being that \emph{microlocal regularity of $u$ can always be propagated from $\mathcal{R}_{\mathrm{2sc,dmf},+}$ to $\mathcal{R}_{\mathrm{2sc,dmf},-}$}. We refer eto Figure \ref{figure 6} below for a graphical illustration. The precise description is deferred to the body of this paper, specifically \S \ref{almost semi-Fredholm estimates with symbolic decay section}, where we will establish the following result:

\begin{proposition} [Proposition \ref{fredholm estimate with symbolic decay}] \label{fredholm estimate with symbolic decay introduction}
Let $P$ be defined by (\ref{Helmhotz operator}) with $n_{\tindex} \geq 2$ and $n^{\tindex} \geq 3$. Then we can construct variable orders $\vor_{\pm}, \vol_{\pm}, \vob_{\pm} \in \mathcal{C}^{\infty}( \psf \Xd )$ and constant orders $m_{\pm}, s_{\pm} \in \mathbb{R}$ explicitly, such that for every $M, S \in \mathbb{R}$ and sufficiently small $\delta > 0$, we have
\begin{align} \label{almost semi-Fredholm estimates with symbolic decay introduction}
\begin{split}
 \| u \|_{ H_{\mathrm{d3sc,3co,res}}^{ m_{ \pm }, \vor_{ \pm }, \vol_{ \pm }, \vor_{ \pm } + \vol_{ \pm }, \vob_{ \pm }, s_{ \pm }  } } & \leq  C ( \| P u \|_{ H_{\mathrm{d3sc,3co,res}}^{ m_{ \pm } - 2, \vor_{ \pm } + 1, \vol_{ \pm }, \vor_{ \pm } + \vol_{\pm  } + 1, \vob_{ \pm } + 2, s_{ \pm } - 2 } } \\
 & \quad + \| u \|_{H_{\mathrm{d3sc,3co,res}}^{M, \vor_{\pm} - \delta, \vol_{\pm} , \vor_{\pm} + \vol_{\pm} - \delta, \vob_{\pm} ,S} } )
\end{split}
\end{align}
in the strong sense that if the right hand sides of (\ref{almost semi-Fredholm estimates with symbolic decay introduction}) are finite, then so are the left hand sides, and the estimates hold. \par

The error terms on the right hand side of (\ref{almost semi-Fredholm estimates with symbolic decay introduction}) improve on the left hand sides (i.e., by being of strictly lower orders) at all of the `symbolic' boundary faces of $\psf \Xd$, which include ${^{\mathrm{d3sc,3co,res}}S^{\ast}}\Xd$, $\psf_{\dmf}\Xd$, $\dtsccf$ and $\rf$, but they do not improve on the left hand sides at the global faces $\cf$ and $\dff$.
\end{proposition}

For the construction of the variable orders in the above proposition, see \S \ref{variable order construction section} below. \par

The strategy we employ in proving Proposition \ref{fredholm estimate with symbolic decay introduction} is simply to follow the flow of (\ref{intro cal 6}), along which we will propagate microlocal regularity. However, the error terms appearing in (\ref{almost semi-Fredholm estimates with symbolic decay introduction}) are not compact, since we have not yet obtained decay at the additional, global faces $\cf$ and $\dff$. Indeed, for two sets of variable orders 
\begin{equation*}
\vom, \vor, \vol, \vov, \vob, \vos \in \mathcal{C}^{\infty}( \psf \Xd ), \quad \vom', \vor', \vol', \vov', \vob', \vos' \in \mathcal{C}^{\infty}( \psf \Xd )
\end{equation*}
 which are admissible, we only have compact inclusions
\begin{equation*}
\text{$H_{\mathrm{d3sc,3co,res}}^{\vom, \vor, \vol, \vov, \vob, \vos}(\Xd)  \subset H_{\mathrm{d3sc,3co,res}}^{\vom', \vor', \vol', \vov', \vob', \vos'}(\Xd)$ for $\vom' < \vom$, $\vor' < \vor$, $\vol' < \vol$, $\vov' < \vov$, $\vob' < \vob$, $\vos' < \vos$,}
\end{equation*}
which is problematic for the purpose of obtaining Fredholm maps.

\subsection{Decay estimates at the global faces}
To obtain further decay at $\cf$, we will make heavy usage of Vasy's construction in \cite[\S 4]{AndrasSM}, where decay at the b-face was obtained in the second microlocalized, two-body setting. The main ingredient is a positive commutator estimate at the level of b-normal operators. Such an estimate will then automatically translate into the framework of second microlocalized scattering operators due to (\ref{second microlocalized algebra tautology case}).  \par

Suppose $A \in \Psi_{\mathrm{d3sc,3co,res}}(\Xd)$ is suitably classical. Then we have already stated that the indicial operator of $A$ at $\cf$ is a family of conormal co-operators on $[\hat{X}^{\tindex} ; \{ 0 \}  ]$ parametrized by $\mathcal{C}_{\tindex} \times \mathbb{R}^{n_{\tindex}}_{\zeta_{\tindex}}$. See (\ref{intro; indicial operator at dff}). Recall that conormal co-operators are scattering-like near $\mathcal{C}^{\tindex}_{\infty}$ (the cone-end) and b-like near $\mathcal{C}^{\tindex}_{0}$ (the cone-tip). Thus Vasy's construction applies naturally near $\mathcal{C}^{\tindex}_{0}$ but not $\mathcal{C}^{\tindex}_{\infty}$. In fact, the argument works everywhere away from $\mathcal{C}^{\tindex}_{\infty}$.  This suggests Vasy's construction should allow us to obtain decay at $\cf \cong \mathcal{C}_{\tindex} \times [ \hat{X}^{\tindex} ; \{ 0 \} ]$ everywhere except for an arbitrarily small neighborhood of $\mathcal{C}_{\tindex} \times \mathcal{C}^{\tindex}_{\infty}$.
 \par

As it turns out, the scattering structure which occurs parametrically at $\mathcal{C}_{\tindex} \times \mathcal{C}_{\infty}^{\tindex}$ microlocally corresponds exactly to the structure at $\psf_{\dmf}\Xd${\ep}at least when we are away from ${^{\mathrm{d3sc,3co,res}}S^{\ast}} \Xd$ (i.e., $|\zeta_{\tindex}| < \infty$ if we also spatially restrict to $\cf$), which is the case here since we only work near the characteristic set of $P${\ep}albeit now in a different sense at $\cf$. Indeed, our construction in this paper will show that 
\begin{equation}  \label{intro cal 7}
\tcocf \backslash  { ^{\mathrm{d3sc,3co,res}}S^{\ast}}\Xd \cong \mathcal{C}_{\tindex} \times \mathbb{R}^{n_{\tindex}}_{\zeta_{\tindex}} \times \overline{ ^{\mathrm{co}}T^{\ast} }[ \hat{X}^{\tindex} ; \{ 0 \} ],
\end{equation}
where we can identify 
\begin{equation*}
\mathcal{C}_{\tindex} \times \mathbb{R}^{n_{\tindex}}_{\zeta_{\tindex}} \times { ^{\mathrm{co}}S^{\ast}}[ \hat{X}^{\tindex} ; \{  0 \} ],  \ \mathcal{C}_{\tindex} \times \mathbb{R}^{n_{\tindex}}_{\zeta_{\tindex}} \times \overline{ ^{\mathrm{co}}T^{\ast} }_{\mathcal{C}^{\tindex}_{\infty}}[ \hat{X}^{\tindex} ; \{ 0 \} ], \  \mathcal{C}_{\tindex} \times \mathbb{R}^{n_{\tindex}}_{\zeta_{\tindex}} \times \overline{ ^{\mathrm{co}}T^{\ast} }_{\mathcal{C}^{\tindex}_0}[ \hat{X}^{\tindex} ; \{ 0 \} ]
\end{equation*}
with $\dtsccf \cap \tcocf$, $\psf_{\dmf} \Xd \cap \tcocf$ and $\psf_{\dff} \Xd \cap \tcocf$ respectively. However, since $\psf_{\dmf} \Xd$ is a `symbolic' face (i.e., one of (\ref{introduction: microlocally symbolic faces})), we have already obtained a $\delta$-small amount of decay at this face in the error terms of (\ref{almost semi-Fredholm estimates with symbolic decay introduction}). Fortunately, we can take advantage of this even at the level of indicial operator, which will allow us to obtain decay at $\cf$ also near $\mathcal{C}_{\tindex} \times \mathcal{C}^{\tindex}_{\infty}$, where Vasy's construction no longer applies. \par

The above calculations will be carried out in \S \ref{decay at the corner face section}, the result of which is the following:
\begin{proposition}[Proposition \ref{decay at the corner face proposition}] \label{decay at the corner face proposition introduction}
Let $P$ be defined by (\ref{Helmhotz operator}) with $n_{\tindex} \geq 2$ and $n^{\tindex} \geq 3$. Then we can construct variable orders $\vor_{\pm}, \vol_{\pm}, \vob_{\pm} \in \mathcal{C}^{\infty}( \psf \Xd )$ and constant orders $m_{\pm}, s_{\pm} \in \mathbb{R}$ explicitly, such that for every $M, S \in \mathbb{R}$ and sufficiently small $\delta > 0$, we have
\begin{align} \label{corner face decay equation 1 introcution}
\begin{split}
 \| u \|_{ H_{\mathrm{d3sc,3co,res}}^{ m_{ \pm }, \vor_{ \pm }, \vol_{ \pm }, \vor_{ \pm } + \vol_{ \pm }, \vob_{ \pm }, s_{ \pm }  } } & \leq  C ( \| P u \|_{ H_{\mathrm{d3sc,3co,res}}^{ m_{ \pm } - 2, \vor_{ \pm } + 1, \vol_{ \pm }, \vor_{ \pm } + \vol_{\pm  } + 1, \vob_{ \pm } + 2, s_{ \pm } - 2 } } \\
 & \quad + \| u \|_{H_{\mathrm{d3sc,3co,res}}^{M, \vor_{\pm} - \delta, \vol_{\pm} , \vor_{\pm} + \vol_{\pm} - \delta , \vob_{\pm} - \delta ,S} } )
\end{split}
\end{align}
for all $u \in H_{\mathrm{d3sc,3co,res}}^{m_{\pm}, \vor_{\pm}, \vol_{\pm}, \vor_{\pm} + \vol_{\pm}, \vob_{\pm}, s_{\pm}}$ such that $Pu \in H_{\mathrm{d3sc,3co,res}}^{ m_{ \pm } - 2, \vor_{ \pm } + 1, \vol_{ \pm }, \vor_{ \pm } + \vol_{\pm  } + 1, \vob_{ \pm } + 2, s_{ \pm } - 2 }$. \par

The error terms on the right hand side of (\ref{corner face decay equation 1 introcution}) improve on the left hand sides (i.e., by being of strictly lower orders) at all of the `symbolic' boundary faces of $\psf \Xd$ and the global face $\cf$, but they do not improve on the left hand sides at the global face $\dff$.

\end{proposition}

It remains to improve decay for the error terms in (\ref{corner face decay equation 1 introcution}) at $\dff$, which will be the goal in \S \ref{decay at the front face section}.  In doing so, we note that if $P$ is viewed as an element of $\Psi_{\mathrm{d3sc,3co,res}}^{2,0,0,0,0,2}(\Xd)$, then 
\begin{equation*}
{ ^{\mathrm{d3sc,3co,res}}\hat{N}_{\dff, 0}}(P) ( y_{\tindex}, \zeta_{\tindex} ) = \hat{P}_{\tindex}( | \zeta_{\tindex} | ) = \Delta_{z^{\tindex}} + V^{\tindex} - ( \lambda^{2} - |\zeta_{\tindex}|^2 ),
\end{equation*}
where we remark that $\hat{P}_{\tindex}$ has already been considered in (\ref{indicial operator of P}). Recall that if $P$ is instead viewed as a three-body operator, then $\hat{P}_{\tindex}$ is also the indicial operator of $P$ at $\ff$, and it has the natural interpretation as a family of scattering operators on $X^{\tindex}$. Thus, an advantage of the second microlocalization is that $\hat{P}_{\tindex}$ can now be viewed systematically as a family of sc,b-operators (cf. (\ref{intro; indicial operator at dff})). Then as discussed in \S \ref{an overview of second microlocalization subsection}, there exists a pair of invertible maps (i.e., (\ref{intro; two-body second microlocal Fredholm map}), albeit with $\hat{P}_{\tindex}(|\zeta_{\tindex}|)$ replacing $P(\sigma)$) for the family $\hat{P}_{\tindex}( |\zeta_{\tindex}| )$ which holds uniformly as $|\zeta_{\tindex}| \rightarrow \lambda$ from below.  \par

To make use of the the maps (\ref{intro; two-body second microlocal Fredholm map}) in this context, it is crucial that
\begin{equation*} \label{intro; normal operator at ff}
\Delta + V^{\tindex} - \lambda^{2} = \mathcal{F}_{  z_{\tindex} \rightarrow \zeta_{\tindex} }^{-1} \hat{P}_{\tindex}( |\zeta_{\tindex}| ) \mathcal{F}_{z_{\tindex} \rightarrow \zeta_{\tindex}},
\end{equation*}
where $\mathcal{F}_{z_{\tindex} \rightarrow \zeta_{\tindex}}$ is the Fourier transform which sends $z_{\tindex}$ to its dual variable $\zeta_{\tindex}$. Moreover, the invertibility of (\ref{intro; two-body second microlocal Fredholm map}) allows us to remove the error terms in (\ref{intro; two-body second microlocal Fredholm estimates}). In fact, the d3sc,3co,res-Sobolev spaces we consider in this paper are all measured with respect to the Euclidean (or scattering) $L^{2}$. Thus, it turns out that we can use Plancherel theorem in the free variables, combined with the aforementioned errorless version of (\ref{intro; two-body second microlocal Fredholm estimates}){\ep}now applied to the operators $\hat{P}(|\zeta_{\tindex}|)$ acting on $X^{\tindex}$, to improve the error terms of (\ref{corner face decay equation 1 introcution}) further at $\dff$ as well. \par

\subsection{Statements of the main theorems} \label{subsection; statements of the main theorem}
Finally, we can define clearly the spaces $\mathcal{X}_{\pm}$ and $\mathcal{Y}_{\pm}$ which appeared in the statement of Theorem \ref{main theorem 1}. We will set
\begin{align*}
\mathcal{X}_{\pm} & \coloneq  \mathcal{X}_{\mathrm{d3sc,3co,res}}^{m_{\pm}, \vor_{\pm}, \vol_{\pm}, \vob_{\pm}, s_{\pm}} \\
& \coloneq \{ u \in H_{\mathrm{d3sc,3co,res}}^{m_{\pm}, \vor_{\pm}, \vol_{\pm}, \vor_{\pm} + \vol_{\pm}, \vob_{\pm}, s_{\pm} }( \Xd ) : Pu \in  H_{\mathrm{d3sc,3co,res}}^{m_{\pm} - 2, \vor_{\pm} + 1, \vol_{\pm}, \vor_{\pm} + \vol_{\pm} + 1, \vob_{\pm} + 2, s_{\pm} - 2 }( \Xd )  \}, \\
\mathcal{Y}_{\pm} & \coloneq \mathcal{Y}_{\mathrm{d3sc,3co,res}}^{m_{\pm}, \vor_{\pm}, \vol_{\pm}, \vob_{\pm}, s_{\pm}} \coloneq H_{\mathrm{d3sc,3co,res}}^{ m_{\pm}, \vor_{\pm}, \vol_{\pm}, \vor_{\pm} + \vol_{\pm}, \vob_{\pm}, s_{\pm} }(\Xd),
\end{align*}
where the variable orders are constructed in \S \ref{variable order construction section}, and are as in the statements of Propositions \ref{fredholm estimate with symbolic decay introduction} and \ref{decay at the corner face proposition introduction}. It follows that Proposition \ref{decay at the corner face proposition introduction} holds for every $u \in \mathcal{X}_{\mathrm{d3sc,3co,res}}^{m_{\pm}, \vor_{\pm}, \vol_{\pm}, \vob_{\pm}, s_{\pm}}$, so that the above discussions immediately lead to the following theorem: 
\begin{main theorem} \label{intro; main theorem 2}
Let $P$ be defined by (\ref{Helmhotz operator}) with $n_{\tindex} \geq 2$ and $n^{\tindex} \geq 3$. Moreover, suppose that $\Delta_{z^{\tindex}} + V^{\tindex}$ has no bound state nor half-bound state for every $\tindex \in \ind$. Then we can construct variable orders $\vor_{\pm}, \vol_{\pm}, \vob_{\pm} \in \mathcal{C}^{\infty}( \psf \Xd )$ and constant orders $m_{\pm}, s_{\pm} \in \mathbb{R}$ explicitly, such that for every $M, S \in \mathbb{R}$ and sufficiently small $\delta > 0$, we have
\begin{align} \label{front face decay equation  introcution}
\begin{split}
 \| u \|_{ H_{\mathrm{d3sc,3co,res}}^{ m_{ \pm }, \vor_{ \pm }, \vol_{ \pm }, \vor_{ \pm } + \vol_{ \pm }, \vob_{ \pm }, s_{ \pm }  } } & \leq  C ( \| P u \|_{ H_{\mathrm{d3sc,3co,res}}^{ m_{ \pm } - 2, \vor_{ \pm } + 1, \vol_{ \pm }, \vor_{ \pm } + \vol_{\pm  } + 1, \vob_{ \pm } + 2, s_{ \pm } - 2 } } \\
 & \quad + \| u \|_{H_{\mathrm{d3sc,3co,res}}^{M, \vor_{\pm} - \delta, \vol_{\pm} - \delta , \vor_{\pm} + \vol_{\pm} - \delta , \vob_{\pm} - \delta ,S} } )
\end{split}
\end{align}
for all $u \in \mathcal{X}_{\mathrm{d3sc,3co,res}}^{m_{\pm}, \vor_{\pm}, \vol_{\pm}, \vob_{\pm}, s_{\pm}}$.
\end{main theorem}
Now, $ \mathcal{X}_{\mathrm{d3sc,3co,res}}^{m_{\pm}, \vor_{\pm}, \vol_{\pm}, \vob_{\pm}, s_{\pm}} $, $\mathcal{Y}_{\mathrm{d3sc,3co,res}}^{m_{\pm}, \vor_{\pm}, \vol_{\pm}, \vob_{\pm}, s_{\pm}}$ are naturally Hilbert spaces. Moreover, we will have
\begin{equation*}
\begin{gathered}
H_{\mathrm{d3sc,3co,res}}^{ m_{-}, \vor_{-}, \vol_{-}, \vor_{-} + \vol_{-}, \vob_{-}, s_{-} } (\Xd) = \left( H_{\mathrm{d3sc,3co,res}}^{ m_{+} - 2, \vor_{+} + 1, \vol_{+}, \vor_{+} + \vol_{+} + 1, \vob_{+} + 2, s_{+} - 2 } (\Xd) \right)^{\ast}, \\
H_{\mathrm{d3sc,3co,res}}^{ m_{-} - 2, \vor_{-} + 1, \vol_{-}, \vor_{-} + \vol_{-} + 1, \vob_{-} + 2, s_{-} - 2 } (\Xd) = \left( H_{\mathrm{d3sc,3co,res}}^{m_{+}, \vor_{+}, \vol_{+}, \vor_{+} + \vol_{+}, \vob_{+}, s_{+} } (\Xd) \right)^{\ast},
\end{gathered}
\end{equation*}
where the adjoints are taken with respect to the standard $L^{2}$ pairing. Then by a standard argument, which can be found for instance in \cite[\S 3]{AndrewNSC}, we can conclude that Theorem \ref{intro; main theorem 2} in fact implies Theorem \ref{main theorem 1}. \par

The proof of Theorem \ref{intro; main theorem 2} will be given in \S \ref{section 10 proof of main theorem subsection}.

\subsection{Summary of main innovations and future directions} \label{subsection; Summary of main innovations and future directions}
Before we conclude this introduction, let us remark that our main innovations in this paper are the following:
\begin{itemize}
\item The construction of the calculus $\Psf(\Xd)$ through the introduction of a new, conormal three-cone algebra $\Psi_{\mathrm{3coc}}(\Xd) \coloneq \bigcup_{ m,r,l \in \mathbb{R} } \Psi_{\mathrm{3coc}}^{m,r,l}( \Xd )$, which is the microlocalization of a new Lie algebra $\mathcal{V}_{\mathrm{3co}}(\Xd) $. 
\item The incorporation of the second microlocalized scattering algebra $\Psi_{\mathrm{sc,b}}(X^{\tindex})$ into the $\Psi_{\mathrm{d3sc,3co,res}}(\Xd)$ algebra (as the indicial operators at $\dff$ for elements in the latter algebra), which allows for Vasy's uniform low energy resolvent estimate from \cite{AndrasSM} to be used as a `black box'. 
\item The analysis of non-tangential (i.e., free or transversal) microlocal propagation via saddle-point radial estimates, which also allows for a symbolic understanding of the transition from transversal to tangential microlocal propagation phenomena. 
\end{itemize}

Some natural future directions which remain to be investigated, but are beyond the scope of this paper, include the following:
\begin{itemize}
\item Invertibility of the Fredholm maps (\ref{introduction; main maps}) constructed in Theorem \ref{main theorem 1}.
\item Incorporation of bound states for the subsystems.
\item Refined asymptotic expansions for generalized eigenfunctions. 
\item Scattering matrix behavior in the non-tangential limit.
\item The Sch\"ordinger propagator.
\end{itemize}

In fact, if we were to follow the standard methods (see \cite{Melrose94}, \cite[\S 3.3]{AndrewNSC}) to prove the invertibility of (\ref{introduction; main maps}), then it appears that obtaining asymptotic expansions for solutions to $Pu \in \mathcal{S}( \mathbb{R}^{n} )$ would already be necessary. Indeed, in the two-body setting, one typically starts with an asymptotic expansion for $u$, which is followed by the application of a `boundary pairing' lemma. Consequentially, we believe it is wiser to leave the question of invertibility to a future paper, where we will treat asymptotic expansions as well.

\subsection{Acknowledgements}
This paper is based on the author's PhD thesis \cite{YilinThesis}. \par 

The author expresses deepest gratitude to Andrew Hassell and And\'as Vasy for the many fruitful discussions that shaped this work. The author is especially indebted to Andrew Hassell for suggesting such an interesting problem, for his generous financial support, and for his valuable assistance with the writing of \S \ref{decay at the front face section}. The author also benefited from communications with Dean Baskin, Gong Chen, Moritz Doll, Jesse Gell-Redman, Peter Hintz, Qiuye Jia, Richard Melrose, Pierre Portal, Jared Wunsch, and Po-Lam Yung.

Throughout his studies, the author has been supported by an Australian Research Training Program (RTP) Scholarship, and he is grateful for having been selected for this.

\section{Microlocal preliminaries}
\label{section microlocal preliminaries}
In this section, we will review several well-known pseudodifferential algebras in the literature. These include the conormal, or small b-algebra and scattering algebra, which will be presented in \S \ref{b-calculus subsection} and \S \ref{subsection the scattering calculus} respectively, as well as Vasy's second microlocalized (sc,b) algebra, which will be discussed in \S \ref{subsection Vasy's second microlocalized calculus}. The three-body algebra, which is reviewed carefully in \S \ref{subsection Vasy's second microlocalized calculus}, and which have already been defined in \cite{AndrasThesis}, will now be considered with decoupled indices. Additionally, we will introduce the conormal `cone algebra', which is a combination of the scattering algebra and the b-algebra on a cone, and thus require no new theoretical understanding, and will be introduced in \S \ref{subsection the cone calculus}. Finally, in \S \ref{parameters-dependent families subsection}, we will consider parameterized families of pseudodifferential operators which exhibit specific large-parameter behaviors. \par

We shall introduce some central concepts which will be relevant throughout our constructions below. The first is the notion of \emph{conormal symbols}, which can be defined on any smooth manifold $M$ with corners: Let $\{ H_{j} \}_{j=1}^{N}$ be the collection of boundary hypersurfaces of $M$, $\{ \rho_{H_{j}} \}_{j=1}^{N}$ be their defining functions, and let also $\mathcal{V}(M)$ be the Lie algebra of smooth vector fields on $M$. Then we will define
\begin{equation} \label{conormal symbol class}
S^{m_{1}, ..., m_{N}}(M), \quad m_1, ..., m_{N} \in \mathbb{R}
\end{equation}
by those functions $a \in \mathcal{C}^{\infty}( M^{\circ} )$ such that 
\begin{equation*}
V_{1} \cdots V_{K} a \in \rho_{H_1}^{-m_1} \cdots \rho_{H_{N}}^{-m_{N}} L^{\infty}(M)
\end{equation*}
for any choices of $V_{1}, ..., V_{K} \in \mathcal{V}_{\mathrm{b}}(M)$ and $K \in \mathbb{N}_0$, where $\mathcal{V}_{\mathrm{b}}(M)$ consists of $V \in \mathcal{V}(M)$ such that $V$ is tangent to $H_{j}$, $j =1,...,N$. \par

In applications, the manifold appearing in (\ref{conormal symbol class}) often comes from a \emph{blow-up}, which is the second central concept we will introduce. Let $M$ be as above, which we now also assume has dimension $n$. Then $L \subset M$ is called a \emph{$p$-submanifold} if locally we can always find coordinates $( x_{1}, ..., x_{k}, y_{1}, ... , y_{n-k} ) \in [ 0 , \infty]^{k} \times \mathbb{R}^{n-k}$ such that $L$ is given by the vanishing of a subset of these coordinates. The \emph{blow-up} of $M$ at $S$ is then given by
\begin{equation} \label{definition of a blow-up}
[ M ; L ] \coloneq ( M \backslash L ) \sqcup S^{+}L,
\end{equation}
where $S^{+}L$ is the inward pointing spherical normal bundle of $L$ in $M$. Associated to (\ref{definition of a blow-up}) is always a \emph{blow-down} map
\begin{equation*}
\beta :[M ; L] \rightarrow M,
\end{equation*}
which is defined by the identity on $M \backslash L$ and the collapsing of $S^{+}L$ onto $L$. One typically employs projective coordinates near $S^{+}L$. We will refer the readers to \cite{MelroseBook} for a more thorough introduction of blow-ups. \par

Finally, if $M$ is a compact manifold with boundary, then we will write 
\begin{equation*}
\dot{\mathcal{C}}^{\infty}(M)
\end{equation*}
for the set of $\varphi \in \mathcal{C}^{\infty}(M^{\circ})$ which vanish to infinite orders at $\partial M$, which is clearly a generalization of the Schwartz functions. In particular, if $M = \overline{\mathbb{R}^{n}}$, then we will write
\begin{equation*}
\dot{\mathcal{C}}^{\infty}( \overline{\mathbb{R}^{n}} ) = \mathcal{S}(\mathbb{R}^{n}).
\end{equation*}

\subsection{The conormal b-calculus}
\label{b-calculus subsection}
The b-algebra is the cornerstone of Melrose's program for studying pseudodifferential operators which are `degenerate' at spatial infinity. In general, one can introduce the b-algebra on arbitrary smooth, compact manifolds $M$ with corners. For the sake of defining the normal operators more easily below, we will only consider the case when $M$ is a smooth, compact manifold with boundary $\partial M$. We mainly need to keep in mind of the case of $M = \overline{\mathbb{R}^{n}}$ in this paper. \par

We first introduce the spaces of b-pseudodifferential operators on $M$. These operators are designed to quantize differential operators generated by the b-vector fields $\mathcal{V}_{\mathrm{b}}( M )$, where $\mathcal{V}_{\mathrm{b}}( M )$ is defined by all $V \in \mathcal{V}( M )$ such that $V$ is tangent to $\partial M$. Thus, locally in a collar neighborhood of $\partial M$ with coordinates $(x,y)$ such that $x$ is a local defining function for $\partial M$, members of $\mathcal{V}_{\mathrm{b}}( M )$ are simply given by the span of $\{ x \partial_{x}, \partial_{y} \}$ over $\mathcal{C}^{\infty}( M )$. \par

Now, one can construct a smooth vector bundle ${ ^{\mathrm{b}}T^{\ast} M }$ such that there exists an identification $\mathcal{C}^{\infty}( M ; { ^{\mathrm{b}}T^{\ast} M } ) = \mathcal{V}_{\mathrm{b}}( M )$. Thus naturally, we have $^{\mathrm{b}}T^{\ast}_{M^{\circ}} M  \cong T^{\ast} M^{\circ}$, while near the boundary ${ ^{\mathrm{b}}T^{\ast}_{\partial M}  M }$ we can find a local frame $\{ x \partial_{x}, \partial_{y} \}$ in the coordinates $(x,y)$. It follows that the dual bundle of ${ ^{\mathrm{b}}T }M$, which we denote by ${ ^{\mathrm{b}}T^{\ast} } M$, will satisfy ${^{\mathrm{b}}T^{\ast}_{M^{\circ}} M} \cong {T^{\ast} M^{\circ} }$, and that $^{\mathrm{b}}T^{\ast}M$ has a local frame given by $\{ dx/x, dy \}$ near ${^{\mathrm{b}}T^{\ast}_{ \partial M } M }$. One can then find local coordinates $(x,y, \tau_{\mathrm{b}}, \mu_{\mathrm{b}} )$ on ${ ^{\mathrm{b}}T^{\ast} } M $ near ${ ^{\mathrm{b}}T^{\ast}}_{ \partial M } M$ with respect to the form
\begin{equation} \label{cannonical one form in the b case}
\tau_{\mathrm{b}} \frac{dx}{x} + \mu_{\mathrm{b}} \cdot dy,
\end{equation}
while over $^{\mathrm{b}}T^{\ast}_{M^{\circ}}M$ we can simply everywhere use Euclidean coordinates $(z,\zeta)$ where $\zeta$ is dual to $z$. Finally, we can compactify ${ ^{\mathrm{b}}T^{\ast} M }$ radially in the fibers, thereby giving rise to $\overline{ ^{\mathrm{b}}T^{\ast} } M$, which is a smooth compact manifold with corners. \par

Following the construction used in \cite{AndrasBook}, we will define the b-operators by specifying their kernels locally everywhere. Since we would like these operators to act on functions, we will request that they are sections of the right b-density bundle ${^{\mathrm{b}} \Omega_{R}}M \coloneq \pi_{R}^{\ast} { ^{\mathrm{b}} \Omega M } $, where $^{\mathrm{b}}\Omega M \rightarrow M$ is the bundle of b-densities arising naturally from $\mathcal{C}^{\infty}( M ; { ^{\mathrm{b}}T^{\ast} M } )$ and $\pi_{R}$ is the right-projection $M^{2} \rightarrow M$. This complication will typically be ignored implicitly by fixing a global trivialization of ${ ^{\mathrm{b}} \Omega M }$, which can easily be done with respect to a choice of strictly positive b-density $\nu_{\mathrm{b}}$. Notably, suppose we choose $\nu_{\mathrm{b}}$ to be such that 
\begin{equation*}
\nu_{\mathrm{b}} = \Big| \frac{dx'}{x'} dy' \hspace{0.5mm} \Big|
\end{equation*}
in the coordinates $(x,y)$, then upon changing to the variables $(t, y)$ where $x = e^{-t}$, we have
\begin{equation*}
\nu_{\mathrm{b}} = | dt' dy' |,
\end{equation*}
i.e., $\nu_{\mathrm{b}}$ gets transformed into the Euclidean right density in the coordinates $(t,y)$. Moreover, notice that $dx/x = - dt$. Therefore in the coordinate system $(x,y,\tau_{\mathrm{b}},\mu_{\mathrm{b}})$ defined with respect to (\ref{cannonical one form in the b case}), one actually has that $-\tau_{\mathrm{b}}$ is dual to $dt$. In particular, this suggests that the near-diagonal part of a general b-operator, which will be defined via quantization, should act on distributions taking values from $[ 0, \infty )_{t} \times \partial M$ in a way similar to the standard setting.
\par

For $m,l \in \mathbb{R}$, the space of conormal symbols $S^{m,l}( \overline{ ^{\mathrm{b}}T^{\ast} } M )$ can be defined as in the beginning of this section. Here $m$ measures decay at $^{\mathrm{b}}S^{\ast} M$ and $l$ measures decay at $\overline{ ^{\mathrm{b}}T^{\ast}}_{ \partial M } M$. More concretely, $S^{m,l}( \overline{ ^{\mathrm{b}}T^{\ast} } M )$ consists of those $a \in \mathcal{C}^{\infty}( T^{\ast} M^{\circ} )$ such that spatially in a neighborhood of $\partial M$, and in local coordinates $(x,y, \tau_{\mathrm{b}}, \mu_{\mathrm{b}})$, we have
\begin{equation}
\label{b-conormal estimates 1}
| ( x \partial_{x} )^{j} \partial_{y}^{\alpha} \partial_{\tau_{\mathrm{b}} }^{k} \partial_{\mu_{\mathrm{b}}}^{\beta} a | \leq C_{j \alpha k \beta} x^{-l} \langle \tau_{\mathrm{b}}, \mu_{\mathrm{b}} \rangle^{m - k - |\beta|}. 
\end{equation}
The above estimates can also be expressed in coordinates $(t, y, \tau_{\mathrm{b}}, \mu_{\mathrm{b}})$, in which case we have
\begin{equation}
\label{b-conormal estimates 2}
| \partial_{t}^{j} \partial_{y}^{\alpha} \partial_{\tau_{\mathrm{b}}}^{k} \partial_{\mu_{\mathrm{b}}}^{\beta} a | \leq C_{j \alpha k \beta} e^{lt} \langle \tau_{\mathrm{b}}, \mu_{\mathrm{b}} \rangle^{m - k - |\beta|}.
\end{equation}
In particular, $a$ can locally be viewed as belonging to H\"ormander's class of uniform symbols with an additional, exponential growth in $t$. Meanwhile, spatially away from $\partial M$, $a$ simply restricts to a standard symbol of order $m$ on $T^{\ast}M^{\circ}$.
\par

We now come back to the definition of b-operators. In fact, we will only consider the conormal classes of these operators, the resulting calculus (the precise meaning of this terminology will be made clear shortly) is known as the `small calculus'. 
\par

Thus for $m,l \in \mathbb{R}$, we are interested in the construction of a space
\begin{equation*}
\Psi^{m,l}_{\mathrm{bc}}( M ).
\end{equation*}
We will first and foremost require that each $A \in \Psi_{\mathrm{bc}}^{m,l}(M)$ be a continuous linear map
\begin{equation*}
A : \dot{\mathcal{C}}^{\infty}( M ) \rightarrow \dot{\mathcal{C}}^{\infty}( M ).
\end{equation*}
In particular, $A$ must have a distributional Schwartz kernel, which we will view as a section of ${^{\mathrm{b}}\Omega}_{R}M$. Next, let $\psi \in \mathcal{C}^{\infty}( M )$ denote an arbitrary cut-off function at some point of $\partial M$. Then we will require that
\begin{equation*}
\psi A \psi \coloneq  \tilde{B}_{\psi} + R_{\psi} .
\end{equation*}
Here, $\tilde{B}_{\psi}$ is given by a `b-quantization'
\begin{equation}
\label{the b-quantization in section 2}
\tilde{B}_{\psi} \coloneq \frac{1}{(2\pi)^{n}} \int_{  \mathbb{R}^{n}  } e^{ i \frac{x-x'}{x} \tau_{\mathrm{b}} + i ( y - y' ) \cdot \mu_{\mathrm{b}}   }\tilde{\varphi} \big( \frac{x-x'}{x} \big) \varphi( |y - y'| ) \tilde{b}_{\psi}( x, y, \tau_{\mathrm{b}}, \mu_{\mathrm{b}} ) d\tau_{\mathrm{b}} d\mu_{\mathrm{b}} \big| \frac{dx'}{x'} dy' \big|,
\end{equation}
where $\tilde{b}_{\psi} \in S^{m,l}( \overline{ ^{\mathrm{b}}T^{\ast} } M )$, $\varphi \in \mathcal{C}^{\infty}_{c}( \mathbb{R} )$ is supported in a small neighborhood of $0$, and $\tilde{\varphi}(s) \coloneq \varphi( 1 - e^{s} )$. However, it is often instructive to work instead in the coordinates $(t,y, \tau_{\mathrm{b}}, \mu_{\mathrm{b}})$, in which case one could also replace $\tilde{B}_{\psi}$ with the standard quantization
\begin{equation} 
\label{the real b-quantisation in time variables}
 B_{\psi} \coloneq \frac{1}{(\pi)^{n}} \int_{\mathbb{R}^{2n}} e^{-i ( t - t' ) \cdot \tau_{\mathrm{b}} + i ( y -y' ) \cdot \mu_{\mathrm{b}} }  \varphi( t - t' ) \varphi( |y-y'| ) b_{\psi} (t, y, \tau_{\mathrm{b}}, \mu_{\mathrm{b}}) d\tau_{\mathrm{b}} d\mu_{\mathrm{b}} | dt' dy' |,
\end{equation}
where $b_{\psi}$ is the representation of $\tilde{b}_{\psi}$ in the new coordinates. In particular, this allows one to understand $B_{\psi}$ via the standard theory. Meanwhile, the remainder term $R_{\psi}$ is a smooth kernel which satisfies
\begin{equation} \label{the real b-R decay term}
| \partial_{t}^{j} \partial_{y}^{\alpha} \partial_{t'}^{j'} \partial_{y'}^{\alpha'} R_{\psi} | \leq C_{j \alpha j' \alpha' NM} e^{tl} \langle y - y' \rangle^{-N} e^{-M| t - t' |}
\end{equation}
for all $M > 0$. \par

 If instead $\psi \in \mathcal{C}^{\infty}(M)$ is a cut-off function at some point in the interior of $M$, then we will simply require that $\psi A \psi$ be a standard pseudodifferential operator with compact support of order $m$ on $M^{\circ}$. In particular, the symbol of $\psi A \psi$ must belong to $S^{m}( \overline{T^{\ast}}M^{\circ} )$ (as specified at the beginning of this section), and it must also be spatially supported away from $\partial M$. We remark that this makes sense because $\overline{ ^{\mathrm{b}}T^{\ast} }_{M^{\circ}}M \cong \overline{T^{\ast}} M^{\circ}$. See \cite{AndrasBook} for a thorough treatment of this material.
 \par

Now, let $\phi \in \mathcal{C}^{\infty}( M )$ be another cut-off function. Then we are interested in the behaviors of $\phi A \psi$ for various choice of $\phi$ and $\psi$, whenever they have disjoint supports. We will need to consider three cases:
\begin{enumerate}
\item $\supp \phi$ and $\supp \psi$ are both disjoint from $ \partial M $, 
\item $\supp \phi$ is disjoint from $\partial M$ but $\supp \psi$ is not, or vice versa,
\item both $\supp \phi$ and $\supp \psi$ intersect $\partial M$, but are still disjoint. 
\end{enumerate}
Let $K_{\phi, \psi}$ be the kernels of $\phi A \psi$ in each case. \par 
Firstly, we shall require that $K_{\phi, \psi}$ to always be smooth. Secondly, $K_{\phi, \psi}$ must be rapidly decreasing in case (2). Finally, in case (3), suppose we work in the coordinates $(t, y, t', y')$ where $y$ and $y'$ might not necessarily be related (i.e., they might not be coordinates of the same chart). Then we will require that the estimates
\begin{equation} \label{b-kernel K_3}
| \partial_{t}^{j} \partial_{y}^{\alpha} \partial_{t'}^{j'} \partial_{y'}^{\alpha'} K_{\phi, \psi} | \leq C_{j \alpha j' \alpha' M} e^{tl} e^{-M| t - t' |}
\end{equation}
be satisfied for all $M > 0$. \par

This concludes the construction of $\Psi^{m,l}_{\mathrm{bc}}(M)$. \par

Now, it is straightforward to show that operator composition defines a map
\begin{equation*}
\Psi^{ m_1 , l_1 }_{\mathrm{bc}} (M ) \times \Psi^{ m_2 ,l_2 }_{\mathrm{bc}} ( M ) \rightarrow \Psi^{ m_1 + m_2, l_1 + l_2}_{\mathrm{bc}} ( M )
\end{equation*}
for every $m_{j}, l_{j} \in \mathbb{R}$, $j=1,2$. Thus the conormal b-alegbra is the filtered algebra defined by
\begin{equation*}
\Psi_{\mathrm{bc}}(M) \coloneq \bigcup_{m,l \in \mathbb{R}} \Psi_{\mathrm{bc}}^{m,l}(M)
\end{equation*}
Moreover, the usual principal symbol map defines a short exact sequence
\begin{equation*}
0 \rightarrow \Psi^{ m - 1 , l - 1 }_{\mathrm{bc}}  ( M ) \rightarrow \Psi^{ m, l }_{\mathrm{bc}} ( M ) \longrightarrow_{^{\mathrm{b}}\sigma_{ m , l }} ( S^{ m, l } / S^{ m - 1, l  } ) ( \overline{ ^{\mathrm{b}}T^{\ast} } M ) \rightarrow 0,
\end{equation*}
and thus descends to an isomorphism
\begin{equation*}
{ ^{\mathrm{b}}\sigma_{m,l} } : ( \Psi_{\mathrm{bc}}^{m,l} / \Psi^{m-1, l }_{\mathrm{bc}} ) ( M ) \xrightarrow{\sim} ( S^{ m , l } / S^{ m - 1, l  } ) ( \overline{ ^{\mathrm{b}}T^{\ast} } M )
\end{equation*}
which is multiplicative in the sense that 
\begin{equation*}
^{\mathrm{b}}\sigma_{ m_1 + m_2 , l_1 + l_2}(A_1 A_2) = {^{\mathrm{b}}\sigma_{ m_1, l_1 }}( A_1 ) {^{\mathrm{b}}\sigma_{ m_2, l_2 }}( A_2 )
\end{equation*} 
whenever $A_{j} \in \Psi^{m_j,l_j}_{\mathrm{bc}}( M )$, $j=1,2$. \par

Let $L^{2}_{\mathrm{b}}(M) \coloneq L^{2}(M , \nu_{\mathrm{b}})$ be the space of $L^2$ functions defined with respect to a fixed, strictly positive b-density $\nu_{\mathrm{b}}$. Let $A^{\ast_{\mathrm{b}}}$ denote the adjoint of $A \in \Psi^{m,l}_{\mathrm{bc}}(M)$ with respect to $L^2_{\mathrm{b}}(M)$. Then it can be shown that $A^{\ast_{\mathrm{b}}} \in \Psi^{m,r}_{\mathrm{bc}}(M)$ as well. Moreover, we have
\begin{equation} 
\label{b algebra principal symbol and indicial operators commutation properties}
{^{\mathrm{b}}\sigma_{m,l}}(A^{\ast_{\mathrm{b}}}) = \overline{{^{\mathrm{b}}\sigma_{m,l}}(A)}.
\end{equation} 
Thus following the standard convention, one could say that the principal symbol map promotes the conormal b-algebra into the \emph{conormal, or small b-calculus}.

One can now carry out the standard microlocalization construction with respect to the principal symbol. Thus, we can define the notions of elliptic, characteristic and wavefront sets, and all of the usual results apply. One could also define suitable notions of b-Sobolev spaces. However, in the discussion of second microlocalization, there are actually two definitions of b-Sobolev spaces (yielding distinct spaces) which can be considered. \par

In the first definition, which is only natural from the perspective of second microlocalization (see \S \ref{subsection Vasy's second microlocalized calculus} below), b-Sobolev spaces will be defined with respect to the weighted space $L^{2}_{\mathrm{sc}}(M) \coloneq x^{n/2} L^{2}_{\mathrm{b}}(M)$. Here, the subscript `sc' indicates that $L^{2}_{\mathrm{sc}}(M)$ is the natural $L^{2}$ space on $M$ when one works with the \emph{scattering calculus}. See also \S \ref{subsection the scattering calculus} below for further discussion on this in the case where $M = \overline{\mathbb{R}^{n}}$. \par

Thus for $m \geq 0$, we will define
\begin{equation*}
H_{\mathrm{b}}^{m,0}( M) \coloneq \{ u \in L^{2}_{\mathrm{sc}}( M ) :  \Lambda u \in L^{2}_{\mathrm{sc}}( M ) \},
\end{equation*}
where $\Lambda \in \Psi_{\mathrm{b}}^{m,0}( M )$ is a fixed elliptic operator, in the sense that the principal symbol of $\Lambda$ is elliptic. Notice that this construction is independent of the choices of $\Lambda \in \Psi^{m,0}_{\mathrm{bc}}(M)$ and $\nu_{\mathrm{b}}$ (i.e., the strictly positive density used in the definition of $L^{2}_{\mathrm{b}}(M)$), in the sense that different choices give rises to spaces with equivalent norms. \par

For $l \in \mathbb{R}$, we also set
\begin{equation*}
H_{\mathrm{b}}^{m,l}( M ) \coloneq x^{l} H_{\mathrm{b}}^{m,0}( M ).
\end{equation*}
It follows that $H_{\mathrm{b}}^{m,l}( M )$ is a Hilbert space with the square norm
\begin{equation*}
\| u \|_{H_{\mathrm{b}}^{m,l}}^2 \coloneq \| x^{-l} u \|_{L^{2}_{\mathrm{sc}}}^{2} + \| x^{-l} \Lambda u \|_{L^{2}_{\mathrm{sc}}}^{2}. 
\end{equation*} 
The cases of $m \leq 0$ can be defined by duality with respect to $L^{2}_{\mathrm{sc}}(M)$.  See also \cite{AndrasLagrangian1, AndrasLagrangian2} for the same details on this definition of the b-Sobolev spaces.
\par

Alternatively,  we could define b-Sobolev spaces with respect to $L^{2}_{\mathrm{b}}(M)$ for a fixed choice (and thus all choices) of $\nu_{\mathrm{b}}$. Such a construction follows from the above procedure verbatim, except that one must replace $L^{2}_{\mathrm{sc}}(M)$ by $L^{2}_{\mathrm{b}}(M)$ everywhere. We refer to \cite{AndrasBook} for these materials. For our actual application in mind (i.e., the three-body problem), we will primarily be working with the b-Sobolev spaces defined with respect to $L^{2}_{\mathrm{sc}}(M)$. \par

It is also possible to assume that the regularity order $m$ is variable dependent, i.e., $m$ is replaced by some $\mathsf{m} \in \mathcal{C^{\infty}}( ^{\mathrm{b}}S^{\ast} M )$. However, the spatial order $l$ must remain constant. For $\delta > 0$ sufficiently small, the spaces of conormal symbols $S^{\vom, l}_{\delta}( \overline{ ^{\mathrm{b}}T^{\ast}} M )$ now consist of those $a \in \mathcal{C}^{\infty}( T^{\ast}M^{\circ} )$ such that, upon spatially restricted to a neighborhood of $\partial M$, estimates (\ref{b-conormal estimates 2}) are replaced by 
\begin{equation*}
| \partial_{t}^{j} \partial_{y}^{\alpha} \partial_{\tau_{\mathrm{b}}}^{k} \partial_{\mu_{\mathrm{b}}}^{\beta} a | \leq C_{j \alpha k \beta} e^{lt} \langle \tau_{\mathrm{b}}, \mu_{\mathrm{b}} \rangle^{m - k - |\beta| + \delta | (j, \alpha, k , \beta ) |}.
\end{equation*}
Meanwhile, spatially away from $\partial M$, and in Euclidean local coordinates $(z,\zeta) \in T^{\ast} M^{\circ}$, $a$ is required to satisfy
\begin{equation*}
| \partial_{z}^{\alpha} \partial_{\zeta}^{\beta} a | \leq C_{\alpha \beta} \langle \zeta \rangle^{m - |\beta| + \delta |\beta |}.
\end{equation*} 
In other words, we lose a $\delta > 0$ order of decay at fiber infinity upon differentiating in any of the variables.
\par

Correspondingly, the spaces of operators $\Psi_{\mathrm{bc},\delta}^{\vom, l}(M)$ can now be defined by the same construction as in the case of $\Psi^{m,l}_{\mathrm{bc}}(M)$, except that one replaces $S^{m,l}( \overline{ ^{\mathrm{b}}T^{\ast}}M )$ by $S^{m,l}_{\delta}( \overline{ ^{\mathrm{b}}T^{\ast}}M )$ everywhere. The principal symbol map captures the leading differential order of $\Psi_{\mathrm{bc}, \delta}^{\vom, l}(M)$ modulo $\Psi_{\mathrm{bc},\delta}^{\vom - 1 + 2 \delta, l}(M)$, and it is multiplicative as before. Moreover, if $A \in \Psi_{\mathrm{bc}, \delta}^{\vom, l}(M)$, then we have $A^{\ast_{\mathrm{b}}} \in \Psi_{\mathrm{bc},\delta}^{\vom,l}(M)$ as well. \par

Finally, b-Sobolev spaces can be defined with respect to variable orders. We simply set
\begin{equation*}
H_{\mathrm{b}}^{\vom, l}(M) \coloneq \{ u \in H_{\mathrm{b}}^{M, l}(M) :  \Lambda u \in L^{2}_{\mathrm{sc}}(M) \}, \quad \| u \|_{H_{\mathrm{b}}^{\vom,l}}^2 \coloneq \| u \|_{H_{\mathrm{b}}^{M,l}}^2  + \| Au \|_{L^{2}}^{2},
\end{equation*}
where $A \in \Psi_{\mathrm{bc},\delta}^{\vom, l}(M)$ is elliptic in the symbolic sense and $M < \vom$. The specific choices of $A$, $M$ and $\nu_{\mathrm{b}}$ (for the definition of $L^{2}_{\mathrm{sc}}(M)$) is irrelevant for this construction. \par

We will refer the readers to \cite{AndrasBook} for a detailed exposition of this material.

\subsection{The scattering calculus} 
\label{subsection the scattering calculus}
The scattering calculus was originally introduced by Parenti and Shubin \cite{ParentiScattering, ShubinScattering} in the non-geometric setting, and then by Melrose \cite{Melrose94} in the geometric setting. It has since been applied to various aspects of the scattering theory. In this subsection, we will review the scattering algebra on the radially compactified Euclidean space $\overline{\mathbb{R}^{n}}$, though we remark that the scattering algebra can be defined more generally on any smooth, compact manifold $M$ with boundary. We will refer the readers to the \cite{AndrasBook} for a thorough exposition of the latter construction. \par

Following the general program of Melrose, the scattering algebra is defined to quantize those differential operators generated by the scattering vector fields 
\begin{equation*}
\mathcal{V}_{\mathrm{sc}}( \overline{\mathbb{R}^{n}} ) \coloneq x \mathcal{V}_{\mathrm{b}}( \overline{\mathbb{R}^{n}} )
\end{equation*}
where $x$ is any global defining function of $\partial \overline{\mathbb{R}^{n}} \cong \mathbb{S}^{n-1}$. Thus initially, locally in any collar neighborhood of $\mathbb{S}^{n-1}$ with coordinates $(x,y)$, $\mathcal{V}_{\mathrm{sc}}( \overline{\mathbb{R}^{n}} )$ must be spanned over $\mathcal{C}^{\infty}( \overline{\mathbb{R}^{n}} )$ by the vector fields $\{ x^{2} \partial_{x}, x \partial_{y}\}$. In Euclidean coordinates, this can also be realized more globally as the span over $\mathcal{C}^{\infty}(\overline{\mathbb{R}^{n}})$ of $\{ \partial_{z_1}, ... , \partial_{z_{n}} \}$. One can then construct a vector bundle ${ ^{\mathrm{sc}}T \overline{\mathbb{R}^{n}} }$ such that $ \mathcal{C}^{\infty} ( \overline{\mathbb{R}^{n}} ; { ^{\mathrm{sc}}T \overline{\mathbb{R}^{n}} } ) = \mathcal{V}_{\mathrm{sc}}( \overline{\mathbb{R}^{n}} )$. Moreover, ${ ^{\mathrm{sc}}T_{\mathbb{R}^{n}} \overline{\mathbb{R}^{n}} } \cong T \mathbb{R}^{n}$, while near ${ ^{\mathrm{sc}}T^{\ast}_{\mathbb{S}^{n-1}} \overline{\mathbb{R}^{n}}}$, we have that ${ ^{\mathrm{sc}}T \overline{\mathbb{R}^{n}}  }$ is spanned by the local frame $\{ x^{2} \partial_{x}, x \partial_{y} \}$ in coordinates $(x,y)$. It follows that if ${ ^{\mathrm{sc}}T^{\ast} \overline{\mathbb{R}^{n}}}$ denotes the dual bundle of ${ ^{\mathrm{sc}}T \overline{\mathbb{R}^{n}} }$, then ${ ^{\mathrm{sc}}T^{\ast}_{\mathbb{R}^{n}} \overline{\mathbb{R}^{n}} } = T^{\ast} \mathbb{R}^{n}$, and $\{ dx/x^{2}, dy/x \}$ is a local frame for ${ ^{\mathrm{sc}}T^{\ast} \overline{\mathbb{R}^{n}} }$ near ${ ^{\mathrm{sc}}T^{\ast}_{ \partial \overline{\mathbb{R}^{n}} } \overline{\mathbb{R}^{n}} }$. In particular, we can also use the Euclidean frames $\{ d{z_1}, ... , d{z_n} \}$ over the interior of ${ ^{\mathrm{sc}}T^{\ast} \overline{\mathbb{R}^{n}} }$. \par
  Spatially in a neighborhood of $\partial \overline{\mathbb{R}^{n}}$, we can find coordinates $(x,y, \tau, \mu )$ on ${ ^{\mathrm{sc}}T^{\ast} \overline{\mathbb{R}^{n}} }$ with respect to the form
\begin{equation*}
\tau \frac{dx}{x^2} + \mu \cdot dy,
\end{equation*}
while over the interior (i.e., just $\mathbb{R}^{n}$) we can simply use $( z , \zeta)$, where $\zeta$ is dual to $dz$. It is a fact which only applies in the Euclidean case that we have
\begin{equation} \label{special diffeomorphism of the scattering bundle}
 ^{\mathrm{sc}}T^{\ast}  \overline{\mathbb{R}^{n}}  \cong \overline{\mathbb{R}^{n}} \times \mathbb{R}^{n},
\end{equation} 
and this is reflected by the change of variables from $(z, \zeta)$ to $( x, y, \tau, \mu )$ being uniformly valid as $x \rightarrow 0$. \par 
Now, let $\overline{^{\mathrm{sc}}T^{\ast}}  \overline{\mathbb{R}^{n}} \cong \overline{\mathbb{R}^{n}} \times \overline{\mathbb{R}^{n}}$ denote the fiber compactification of $^{\mathrm{sc}}T^{\ast}  \overline{\mathbb{R}^{n}}$. Then for $m,r \in \mathbb{R}$, we can consider the space of conormal symbols $S^{m,r}( \overline{^{\mathrm{sc}}T^{\ast}} \overline{\mathbb{R}^{n}}  )$, where $m$ measures decay at $^{\mathrm{sc}}S^{\ast} \overline{\mathbb{R}^{n}}$ and $r$ measures decay at $\overline{ ^{\mathrm{sc}}T^{\ast} }_{\mathbb{S}^{n-1}} \overline{\mathbb{R}^{n}}$. One can check that elements of $S^{m,r}( \overline{ ^{\mathrm{sc}}T^{\ast} } \overline{\mathbb{R}^{n}} )$ consists of those $a \in \mathcal{C}^{\infty}( T^{\ast} \mathbb{R}^{n} )$ such that 
\begin{equation}
\label{Euclidean scattering symbolic estimates}
| \partial_{z}^{\alpha} \partial_{\zeta}^{\beta} a  | \leq C_{\alpha \beta} \langle z \rangle^{r - |\alpha|} \langle \zeta \rangle^{m - |\beta|},
\end{equation}
and this characterizes $S^{m,r}( \overline{ ^{\mathrm{sc}}T^{\ast}} \overline{\mathbb{R}^{n}} )$ completely in view of (\ref{special diffeomorphism of the scattering bundle}). Nevertheless, we will note that in coordinates $( x, y, \tau , \mu )$, (\ref{Euclidean scattering symbolic estimates}) can also be written as 
\begin{equation}
\label{Melrosian scattering symbolic estimates}
|( x \partial_{x} )^{j} \partial_{y}^{\alpha} \partial_{\tau}^{k} \partial_{\mu}^{\beta} a | \leq C_{j \alpha k \beta} x^{-r} \langle \tau, \mu \rangle^{m - k - |\beta|} .
\end{equation}
\par

The space of operators 
\begin{equation*}
\Psi^{ m , r }_{\mathrm{sc}}( \overline{\mathbb{R}^{n}} )  
\end{equation*}
will be defined by the Euclidean quantizations of $S^{m,r}( \overline{ ^{\mathrm{sc}}T^{\ast} } \overline{\mathbb{R}^{n}}$), i.e., $\Psi_{\mathrm{sc}}^{m,r}( \overline{\mathbb{R}^{n}} )$ consists of those operators $A$ such that
\begin{equation} 
\label{scattering quantization}
A = \frac{1}{(2\pi)^{n}} \int_{\mathbb{R}^{n}} e^{i( z - z' ) \cdot \zeta} a(z,\zeta) d\zeta |dz'|
\end{equation}
for some $a \in S^{m,r}( \overline{ ^{\mathrm{sc}}T^{\ast} } \overline{\mathbb{R}^{n}} )$. \par

It can be shown that operator composition is a mapping
\begin{equation*}
\Psi^{ m_1 , r_1 }_{\mathrm{sc}} ( \overline{\mathbb{R}^{n}} ) \times \Psi^{ m_2 ,r_2 }_{\mathrm{sc}} ( \overline{\mathbb{R}^{n}} ) \rightarrow \Psi^{ m_1 + m_2, r_1 + r_2}_{\mathrm{sc}} ( \overline{\mathbb{R}^{n}} ).
\end{equation*}
The scattering algebra is then the filtered algebra defined by the union 
\begin{equation*}
\Psi_{\mathrm{sc}}( \overline{\mathbb{R}^{n}} ) \coloneq \bigcup_{m,r \in \mathbb{R}} \Psi_{\mathrm{sc}}^{m,r}( \overline{\mathbb{R}^{n}} ).
\end{equation*}
Moreover, the principal symbol map $ ^{\mathrm{sc}}\sigma_{ m, r }(A) = a$ gives rise to a short exact sequence
\begin{equation*}
0 \rightarrow \Psi^{ m - 1 , r - 1 }_{\mathrm{sc}}  ( \overline{\mathbb{R}^{n}} ) \rightarrow \Psi^{ m, r }_{\mathrm{sc}} ( \overline{\mathbb{R}^{n}} ) \longrightarrow_{^{\mathrm{sc}}\sigma_{ m ,r }} ( S^{ m, r } / S^{ m - 1, r - 1  } ) ( \overline{ ^{\mathrm{sc}}T^{\ast} } \overline{\mathbb{R}^{n}} ) \rightarrow 0,
\end{equation*}
such that the descended isomorphism
\begin{equation*}
{^{\mathrm{sc}}\sigma_{m,r}} : ( \Psi_{\mathrm{sc}}^{m,r} / \Psi^{m-1, r-1}_{\mathrm{sc}} ) ( \overline{\mathbb{R}^{n}} ) \xrightarrow{\sim} ( S^{ m, r } / S^{ m - 1, r - 1  } ) ( \overline{ ^{\mathrm{sc}}T^{\ast} } \overline{\mathbb{R}^{n}} )
\end{equation*}
is multiplicative in the sense that
\begin{equation*}
^{\mathrm{sc}}\sigma_{ m_1 + m_2 , r_1 + r_2}(A_1 A_2) = {^{\mathrm{sc}}\sigma_{ m_1, r_1 }}( A_1 ) {^{\mathrm{sc}}\sigma_{ m_2, r_2 }}( A_2 )
\end{equation*} 
whenever $A_{j} \in \Psi^{m_j,r_j}_{\mathrm{sc}}( \overline{\mathbb{R}^{n}} )$, $j=1,2$. \par

As already mentioned in the discussion of the b-Sobolev spaces in \S \ref{b-calculus subsection}, the natural $L^2$ spaces one works with in the context of the scattering calculus are $L^{2}_{\mathrm{sc}}( \overline{\mathbb{R}^{n}}) = L^{2}( \overline{\mathbb{R}^{n}} , \nu_{\mathrm{sc}} )$. Here $\nu_{\mathrm{sc}} \in {^{\mathrm{sc}} \Omega \overline{\mathbb{R}^{n}}}$ is any fixed, strictly positive scattering density, where ${^{\mathrm{sc}} \Omega \overline{\mathbb{R}^{n}}} \rightarrow \overline{\mathbb{R}^{n}}$ is the bundle defined by the natural densities arising from $\mathcal{C}^{\infty}( \overline{\mathbb{R}^{n}} ; { ^{\mathrm{sc}}T^{\ast} \overline{\mathbb{R}^{n}} } )$. For example, given any strictly positive b-density $\nu_{\mathrm{b}}$, we know that $x^{-n} \nu_{\mathrm{b}}$ is a strictly positive scattering density whenever $x$ is a global defining function for $\mathbb{S}^{n-1}$. In particular, this justifies the notation used in the definition of b-Sobolev spaces in \S \ref{b-calculus subsection}. Moreover, in the Euclidean case, a canonical choice of strictly positive scattering density is simply $|dz|$, so the resulting $L^{2}_{\mathrm{sc}}(\overline{\mathbb{R}^{n}})$ space actually agrees with the standard $L^{2}(\mathbb{R}^{n})$ space. \par

Let $A^{\ast}$ be the adjoint of $A \in \Psi^{m,r}_{\mathrm{sc}}(\overline{\mathbb{R}^{n}})$ taken with respect to $L^{2}(\mathbb{R}^{n})$. Then we can show that $A^{\ast} \in \Psi^{m,r}_{\mathrm{sc}}(\overline{\mathbb{R}^{n}})$ as well, and we have
\begin{equation}
\label{sc algebra principal symbol commutation properties}
{^{\mathrm{sc}}\sigma_{m,r}}(A^{\ast}) = \overline{ ^{\mathrm{sc}}\sigma_{m,r}(A) }.
\end{equation}
Thus, when we combine (\ref{sc algebra principal symbol commutation properties}) with the multiplicative property of the principal symbol map discussed above, one could say that the scattering algebra is endowed with a symbol calculus. It is therefore also referred to as the \emph{scattering calculus}. 
\par

One could also discuss microlocalization, in particular the notions of elliptic, characteristic and wavefront sets in this setting. Heuristically speaking, since the principal symbol captures principal decay microlocally at all boundary faces of $\overline{^{\mathrm{sc}}T^{\ast}} \overline{\mathbb{R}^{n}} \cong \overline{\mathbb{R}^{n}} \times \overline{\mathbb{R}^{n}}$, these objects can be viewed as direct extensions from the standard (i.e., H\"ormander's uniform algebra $\Psi_{\infty}( \mathbb{R}^{n} )$, see \cite{AndrasBook} for where this notation is used) case, and is thus straightforward in a sense. Moreover, the Sobolev spaces $H_{\mathrm{sc}}^{m,r}( \overline{\mathbb{R}^{n}} )$ defined using scattering elliptic operators turn out to be the same as the usual Sobolev spaces with spatial weights $H^{m,r}( \mathbb{R}^{n} )$, except that we now have an alternate characterization given by
\begin{equation*}
H^{m,r}( \mathbb{R}^{n} )  = \{ u \in \mathcal{S}'( \mathbb{R}^{n} ) :  \Lambda u \in L^{2}( \mathbb{R}^{n} ) \},
\end{equation*}
where $\Lambda \in \Psi_{\mathrm{sc}}^{m,r}( \overline{\mathbb{R}^{n}} )$ is elliptic, i.e., its principal symbol is elliptic. This definition is independent of the choice of $\Lambda$, though in practice, one typically takes $\Lambda$ to be invertible (and thus $\Lambda^{-1} \in \Psi_{\mathrm{sc}}^{-m,-r}(\overline{\mathbb{R}^{n}})$ by ellipticity), and define
\begin{equation*}
\| u \|_{H^{m,r}} \coloneq \| \Lambda u \|_{L^{2}}.
\end{equation*}
\par

We can also extend $m$ and $r$ to be variable dependent, and require that they be smooth functions $\mathsf{m} \in \mathcal{C}^{\infty}( ^{\mathrm{sc}}S^{\ast} \overline{\mathbb{R}^{n}} )$ and $\mathsf{r} \in \mathcal{C}^{\infty}( \overline{^{\mathrm{sc}}T^{\ast}}_{\mathbb{S}^{n-1}} \overline{\mathbb{R}^{n}} )$ respectively. Let $\delta > 0$ be sufficiently small. Then we will define $S^{\vom, \vor}_{\delta}( \overline{ ^{\mathrm{sc}}T^{\ast}} \overline{\mathbb{R}^{n}} )$ by those $a \in \mathcal{C}^{\infty}( T^{\ast} \mathbb{R}^{n} )$ such that \begin{equation*}
\begin{gathered}
| \partial_{z}^{\alpha} \partial_{\zeta}^{\beta} a  | \leq C_{\alpha \beta} \langle z \rangle^{r - |\alpha| + \delta |( \alpha, \beta )|} \langle \zeta \rangle^{m - |\beta| + \delta |(\alpha, \beta)|},
\end{gathered}
\end{equation*}
i.e., we lose a $\delta > 0$ order of decay at both $\overline{^{\mathrm{sc}}T^{\ast}}_{\mathbb{S}^{n-1}} \overline{\mathbb{R}^{n}}$ and ${^{\mathrm{sc}}S^{\ast}}\overline{\mathbb{R}^{n}}$ upon differentiating in any of the variables. Correspondingly, the space of operators $\Psi^{\mathsf{m}, \mathsf{r}}_{\mathrm{sc}, \delta}( \overline{\mathbb{R}^{n}} )$ will be defined by directly quantizing elements of $S^{\vom,\vor}_{\delta}( \overline{ ^{\mathrm{sc}} }T^{\ast} \overline{\mathbb{R}^{n}} )$ through (\ref{scattering quantization}). Moreover, operator composition still holds, with a symbol calculus at the level of $\Psi^{\mathsf{m}, \mathsf{r}}_{\mathrm{sc}, \delta}( \overline{\mathbb{R}^{n}} )$ modulo $\Psi^{\mathsf{m} - 1 +2\delta, \mathsf{r} - 1 + 2\delta}_{\mathrm{sc}, \delta}( \overline{\mathbb{R}^{n}} )$. Additionally, (\ref{sc algebra principal symbol commutation properties}) holds as well. Finally, scattering Sobolev spaces $H_{\mathrm{sc}}^{\vom,\vor}( \overline{\mathbb{R}^{n}} ) = H^{\vom, \vor}(\mathbb{R}^{n})$ with variable orders can also be defined with respect to variable orders, i.e.,
\begin{equation*}
\begin{gathered}
H^{\vom, \vor}( \mathbb{R}^{n} ) \coloneq \{ u \in H^{M,N}( \mathbb{R}^{n} ) : \Lambda u \in L^2 ( \mathbb{R}^{n} )  \}, \quad  \| u \|_{H^{\vom,\vor}}^2 \coloneq \| u \|_{H^{M,N}}^2 + \| \Lambda u \|_{L^2}^2,
\end{gathered}
\end{equation*}
where $\Lambda \in \Psi_{\mathrm{sc},\delta}^{\vom, \vor}(\overline{\mathbb{R}^{n}})$ is elliptic and that $M \leq \vom$, $N \leq \vor$. This definition is independent of the choice of $\Lambda$, $M$ and $N$. Again, see \cite{AndrasBook,AndrewNSC} for more details.

\subsection{The conormal three-body calculus} 
\label{subsection the three-body calculus}
We now review the conormal three-body algebra and some of its basic properties. We will assume that $\mathcal{C} =  \mathcal{C}_{\tindex}$ has just one component. Hence the position space in question is just the blow-up
\begin{equation*}
\beta_{\mathrm{3sc}} : [ \overline{\mathbb{R}^{n}} ; \mathcal{C}_{\tindex} ] \rightarrow \overline{\mathbb{R}^{n}}. 
\end{equation*} \par

By construction, $[ \overline{\mathbb{R}^{n}} ; \mathcal{C}_{\tindex} ]$ is a smooth manifold with codimension two corners. The boundary hypersurfaces of $[ \overline{\mathbb{R}^{n}}; \mathcal{C}_{\tindex} ]$ are given by $\mf$ and $ \ff $, where $\mf$ is the lift of $\mathbb{S}^{n-1}$ and $ \ff $ the new front face (or equivalently the lift of $\mathcal{C}_{\tindex}$). 
\par

\begin{figure}
\centering
\includegraphics{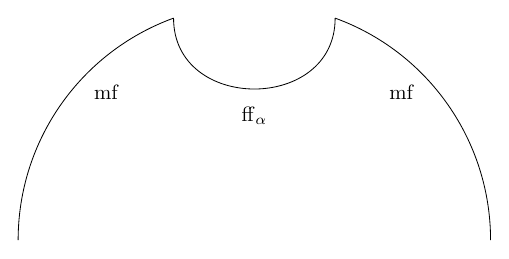}
\caption{The space $[ \overline{\mathbb{R}^{2}} ; \mathcal{C}_{\tindex} ]$ near the lift of $\mathcal{C}_{\tindex}$ when $\mathcal{C}_{\tindex}$ is a point.}
\end{figure}

Away from $\text{ff}_{\alpha}$ there is a diffeomorphism
\begin{equation*}
[ \overline{\mathbb{R}^{n}}; \mathcal{C}_{\alpha} ] \backslash \text{ff}_{\alpha} \cong_{\text{Id}} \overline{\mathbb{R}^{n}} \backslash \mathcal{C}_{\alpha}, 
\end{equation*}
so the smooth structure of $[\overline{\mathbb{R}^{n}} ; \mathcal{C}_{\tindex} ]$ on any compact subsets of this region is clear. In order to understand the smooth structure near $\ff$, we will derive projective coordinates near $\mathrm{ff}_{\tindex}$. To this end, in a small neighborhood of $\mathcal{C}_{\alpha}$, we will consider regions of the form
\begin{equation} \label{condition near the parallel blow-up}
1 \leq c |z_{\alpha}|, \ |z^{\alpha}| \leq c |z_{\alpha}|, \ | z_{\tindex,j} | \leq c | z_{\tindex,N_{\tindex}} |, \ j \neq N_{\tindex}, \ N_{\tindex} = 1, ..., n_{\tindex}
\end{equation}
for some constant $c > 0$. Thus, in regions where (\ref{condition near the parallel blow-up}) is satisfied, we can initially introduce
\begin{equation} \label{usual scattering coordinates}
 x_{\alpha} \coloneq \frac{1}{|z_{\alpha}|}, \ y_{\tindex,j} \coloneq \frac{z_{\tindex,j}}{|z_{\alpha}|}, \ j \neq N_{\tindex}, \ Y^{\alpha} \coloneq \frac{z^{\alpha}}{|z_{\alpha}|}
\end{equation}
which are valid coordinates on $\overline{\mathbb{R}^{n}}$ near $\mathcal{C}_{\tindex}$, where $\mathcal{C}_{\tindex} = \{ x_{\tindex} = 0, Y^{\tindex} = 0 \}$. By introducing projective coordinates at $\mathcal{C}_{\tindex}$, it follows that we can find coordinates
\begin{equation} \label{coordinates near ff}
x_{\alpha}, \, y_{\alpha} , \, z^{\alpha} \coloneq \frac{Y^{\alpha}}{x_{\alpha}}
\end{equation}
in $\mathrm{ff}_{\alpha}^{\circ}$, where $z^{\alpha} \in (X^{\tindex})^{\circ}$ and $x_{\alpha}$ is a local defining function for $\ff^{\circ}$. On the other hand, in a small neighborhood of $\ff \cap \mf$, we can find coordinates
\begin{equation} \label{coordinates near mf cap ff}
\bdf\coloneq |Y^{\alpha}| = \frac{x_{\alpha}}{x^{\alpha}}, \ x^{\alpha} \coloneq \frac{x_{\alpha}}{|Y^{\alpha}|} = \frac{1}{|z^{\tindex}|}, \ y_{\alpha}, \ y^{\alpha}_{k} \coloneq \frac{ Y^{\alpha}_{k} }{|Y^{\alpha}|}, \ k \neq N^{\tindex}, \ N^{\tindex} = 1 , ... , n^{\tindex},
\end{equation}
which can be introduced whenever $|z^{\alpha}_{k}| \leq c |z^{\alpha}_{N^{\tindex}}|$ for all $k \neq N^{\tindex}$ and $1 \leq c|z^{\alpha}|$. Here $\bdf$ is a local defining function for $\ff $, while $x^{\alpha}$ is again a local defining function of $\mf $. \par

Having described the geometry of $[ \overline{\mathbb{R}^{n}} ; \mathcal{C}_{\tindex} ]$, we can now introduce the Lie algebra of three-body vector fields $\mathcal{V}_{\mathrm{3sc}}( [ \overline{\mathbb{R}^{n}} ; \mathcal{C}_{\tindex} ] )$ simply by the lift of $\mathcal{V}_{\mathrm{sc}}( \overline{\mathbb{R}^{n}} )$, i.e., 
\begin{equation*}
\mathcal{V}_{\mathrm{3sc}} ( [ \overline{\mathbb{R}^{n}} ; \mathcal{C}_{\tindex} ] ) \coloneq \beta_{\mathrm{3sc}}^{\ast} \mathcal{V}_{\mathrm{sc}} ( \overline{\mathbb{R}^{n}} ). 
\end{equation*}
In local coordinates (\ref{usual scattering coordinates}), we find that $\mathcal{V}_{\text{sc}}( \overline{\mathbb{R}^{n}} )$ is spanned over $\mathcal{C}^{\infty}( [ \overline{\mathbb{R}^{n}} ; \mathcal{C}_{\tindex} ] )$ by 
\begin{alignat}{2} \label{scattering vector fields special cases}
x_{\alpha}^{2} \partial_{x_{\alpha}}, \, x_{\alpha} \partial_{y_{\alpha}}, \, x_{\tindex} \partial_{Y^{\alpha}}.
\end{alignat}
In local coordinates (\ref{coordinates near ff}) and (\ref{coordinates near mf cap ff}), elements of (\ref{scattering vector fields special cases}) lift to, in the same order
\begin{alignat}{2}
x_{\alpha}^2 \partial_{x_{\alpha}} - x_{\alpha} z^{\alpha} \cdot \partial_{z^{\alpha}} , \, x_{\alpha} \partial_{y_{\alpha}}, \, \partial_{z^{\alpha}}
\end{alignat}
and
\begin{equation}
\begin{gathered}
\bdf (x^{\alpha})^2 \partial_{x^{\alpha}}, \ \bdf x^{\alpha} \partial_{y_{\alpha}}, \vspace{0.5mm} \\ 
\bdf x^{\alpha} y^{\alpha}_{k} \partial_{ \bdf } - (x^{\tindex})^2 y^{\tindex}_{k} \partial_{x^{\tindex}} + x^{\alpha} \partial_{y^{\alpha}_{k}} - x^{\alpha} y^{\alpha}_{k} y^{\alpha} \cdot \partial_{y^{\alpha}}, \ k \neq N^{\tindex}, \ N^{\tindex} = 1, ... , n^{\tindex}, \\
\bdf x^{\alpha} ( 1 - |y^{\alpha}|^{2} )^{1/2} \partial_{\bdf} - (x^{\tindex})^2 ( 1 - |y^{\tindex}|^2 )^{1/2} \partial_{x^{\tindex}}  - x^{\alpha} (1-|y^{\alpha}|^2)^{1/2} y^{\alpha} \cdot \partial_{y^{\alpha}}
\end{gathered}
\end{equation}
respectively. Thus, $\mathcal{V}_{\mathrm{3sc}}( [ \overline{\mathbb{R}^{n}} ; \mathcal{C}_{\tindex} ] )$ is spanned over $\mathcal{C}^{\infty}( [ \overline{\mathbb{R}^{n}} ; \mathcal{C}_{\tindex} ] )$ by
\begin{equation} \label{three body tangent bundle frame 1}
x_{\alpha}^2 \partial_{x_{\alpha}}, \, x_{\alpha} \partial_{y_{\alpha}}, \,  \partial_{z^{\alpha}}
\end{equation}
and
\begin{equation} 
\label{three body tangent bundle frame 2}
\begin{gathered}
\bdf (x^{\tindex})^2 \partial_{x^{\tindex}}, \, \bdf x^{\alpha} \partial_{\bdf} - (x^{\tindex})^2 \partial_{x^{\tindex}}, \, \bdf x^{\alpha} \partial_{y_{\alpha}}, \, x^{\alpha} \partial_{y^{\alpha}},
\end{gathered}
\end{equation}  
in coordinates (\ref{coordinates near ff}) and (\ref{coordinates near mf cap ff}) respectively, while $\mathcal{V}_{\mathrm{3sc}}( [ \overline{\mathbb{R}^{n}} ; \mathcal{C}_{\tindex} ] ) = \mathcal{V}_{\mathrm{sc}}( \overline{\mathbb{R}^{n}} )$ away from $\ff$. 
\par

Now, the Melrose program produces a compressed vector bundle ${ ^{\mathrm{3sc}}T [ \overline{\mathbb{R}^{n}} ; \mathcal{C}_{\tindex} ] }$ such that ${ ^{\mathrm{3sc}}T_{\mathbb{R}^{n}} [ \overline{\mathbb{R}^{n}} ; \mathcal{C}_{\tindex} ] } = T \mathbb{R}^{n}$ and $\mathcal{C}^{\infty}( [ \overline{\mathbb{R}^{n}} ; \mathcal{C}_{\tindex} ] ; { ^{\mathrm{3sc}}T [ \overline{\mathbb{R}^{n}} ; \mathcal{C}_{\tindex} ] } ) = \mathcal{V}_{\mathrm{3sc}}( [ \overline{\mathbb{R}^{n}} ; \mathcal{C}_{\tindex} ] )$. Then in a small neighborhood of $^{\mathrm{3sc}}T_{\ff}[ \overline{\mathbb{R}^{n}} ; \mathcal{C}_{\tindex} ]$, there exists local frames for ${ ^{\mathrm{3sc}}T [ \overline{\mathbb{R}^{n}} ; \mathcal{C}_{\tindex} ] }$ given by (\ref{three body tangent bundle frame 1}) and (\ref{three body tangent bundle frame 2}), while ${ ^{\mathrm{3sc}}T_{U} [ \overline{\mathbb{R}^{n}} ; \mathcal{C}_{\tindex} ] } = { ^{\mathrm{sc}}T_{U} \overline{\mathbb{R}^{n}} }$ whenever $U \subset [ \overline{\mathbb{R}^{n}} ; \mathcal{C}_{\tindex} ] $ is an open neighborhood away from $\ff$. \par

Equivalently, we can use
\begin{equation*}
\begin{gathered}
\rho_{\tindex}^2 x^{\tindex} \partial_{\rho_{\tindex}}, \, \bdf x^{\alpha} \partial_{\bdf} - (x^{\tindex})^2 \partial_{x^{\tindex}}, \, \bdf x^{\alpha} \partial_{y_{\alpha}}, \, x^{\alpha} \partial_{y^{\alpha}},
\end{gathered}
\end{equation*}  
as a basis of vector fields in place of (\ref{three body tangent bundle frame 2}). Moreover, the presence of the `mixed' vector field $\rho_{\tindex} x^{\tindex} \partial_{\bff} - (x^{\tindex})^2 \partial_{x^{\tindex}}$ can be eliminated by considering slightly degenerate coordinate systems $( x_{\tindex}, y_{\tindex}, x^{\tindex}, y^{\tindex} )$ and $( x_{\tindex}, y_{\tindex}, \bff, y^{\tindex})$. These coordinate systems do not extend to the the entire boundary of $ [ \overline{\mathbb{R}^{n}} ; \mathcal{C}_{\tindex} ]$, even near $\mathcal{C}_{\tindex}$, but are nevertheless valid over the interior $\mathbb{R}^{n}$ of $[ \overline{\mathbb{R}^{n}} ; \mathcal{C}_{\tindex} ]$. More precisely, coordinates $(x_{\tindex}, y_{\tindex}, x^{\tindex}, y^{\tindex})$ are valid up to $\ff$ but not $\mf$, while coordinates $( x_{\tindex}, y_{\tindex}, \bff, y^{\tindex} )$ are valid up to $\mf$ but not $\ff$. Additionally, in coordinates $(x_{\tindex}, y_{\tindex}, x^{\tindex}, y^{\tindex})$, a basis of vector fields which spans $\mathcal{V}_{\mathrm{3sc}}([ \overline{\mathbb{R}^{n}} ; \mathcal{C}_{\tindex} ])$ over $\mathcal{C}^{\infty}( [ \overline{\mathbb{R}^{n}} ; \mathcal{C}_{\tindex} ] )$ is given by
\begin{equation*}
x_{\tindex}^2 \partial_{x_{\tindex}}, \, x_{\tindex} \partial_{y_{\tindex}}, \, (x^{\tindex})^2 \partial_{x^{\tindex}}, \, x^{\tindex} \partial_{y^{\tindex}},
\end{equation*}
while in coordinates $(x_{\tindex}, y_{\tindex}, \bff, y^{\tindex})$, such a basis is given by
\begin{equation*}
x_{\tindex}^2 \partial_{x_{\tindex}}, \, x_{\tindex} \partial_{y_{\tindex}}, \, x^{\tindex} \bdf \partial_{\bdf}, \, x^{\tindex} \partial_{y^{\tindex}}.
\end{equation*}
These facts will also be used implicitly in the introduction of momentum variables below. The usage of degenerate coordinates will be the focus in the discussion of three-cone vector fields in \S \ref{subsection definition of the three-cone bundle and the variable changes} as well.

Alternatively, one could easily see that
\begin{equation*}
{ ^{\mathrm{3sc}}T[ \overline{ \mathbb{R}^{n} } ; \mathcal{C} ] } = \beta_{\mathrm{3sc}}^{\ast} ({ ^{\mathrm{sc}}T \overline{ \mathbb{R}^{n} }  }).
\end{equation*}
Namely, ${ ^{\mathrm{3sc}}T[ \overline{ \mathbb{R}^{n} } ; \mathcal{C} ] }$ is nothing but the pull-back bundle of ${ ^{\mathrm{sc}}T\overline{ \mathbb{R}^{n} }  }$ through $\beta_{\mathrm{3sc}}$. \par

As usual, we will let ${ ^{\mathrm{3sc}}T^{\ast} [ \overline{\mathbb{R}^{n}} ; \mathcal{C}_{\tindex} ] }$ denote the bundle dual to ${ ^{\mathrm{3sc}}T [ 
\overline{\mathbb{R}^{n}} ; \mathcal{C}_{\tindex} ] }$, though one could also realize $^{\mathrm{3sc}}T^{\ast}[ \overline{\mathbb{R}^{n}} ; \mathcal{C}_{\tindex} ]$ as a pulled back bundle, i.e.,
\begin{equation} \label{bundle pull back definition of the three-body cotangent bundle}
^{\mathrm{3sc}}T^{\ast}[ \overline{\mathbb{R}^{n}} ; \mathcal{C}_{\tindex} ] = \beta_{\mathrm{3sc}}^{\ast}( ^{\mathrm{sc}}T^{\ast} \overline{\mathbb{R}^{n}} ). 
\end{equation}
It follows that ${ ^{\mathrm{3sc}}T^{\ast}_{\mathbb{R}^{n}} } [ \overline{\mathbb{R}^{n}} ; \mathcal{C}_{\tindex} ] = { T^{\ast} \mathbb{R}^{n} }$. Moreover, we have ${ ^{\mathrm{3sc}}T^{\ast}_{U} [ \overline{\mathbb{R}^{n}} ; \mathcal{C}_{\tindex} ] } = { ^{\mathrm{sc}}T^{\ast}_{U} \overline{\mathbb{R}^{n}} }$ whenever $U \subset [ 
\overline{\mathbb{R}^{n}} ; \mathcal{C}_{\tindex} ]$ is an open neighborhood away from $\ff$. \par

However, for the purpose of computing local frames near ${^{\mathrm{3sc}}T^{\ast}_{\ff} [ \overline{\mathbb{R}^{n}} ; \mathcal{C}_{\tindex} ]}$, it would actually be more useful to start from (\ref{bundle pull back definition of the three-body cotangent bundle}), in which case we need to pull back the scattering frame 
$dx_{\tindex}/x_{\tindex}^2$, $dy_{\tindex}/x_{\tindex}$, $dY^{\tindex}/x_{\tindex}$ through $\beta_{\mathrm{3sc}}$. Doing so in coordinates (\ref{coordinates near ff}), we only need observe
\begin{equation*}
d z^{\tindex}_{k} = \frac{dY^{\tindex}_{k}}{x_{\tindex}} - x_{\tindex}  z^{\tindex}_{k} \frac{dx_{\tindex}}{x_{\tindex}^2}, \ k = 1, ... , n^{\tindex}
\end{equation*}
to conclude that 
\begin{equation*}
\frac{dx_{\tindex}}{x_{\tindex}^2}, \ \frac{dy_{\tindex}}{x_{\tindex}}, \ dz^{\tindex}
\end{equation*}
is a local frame for $^{\mathrm{3sc}}T^{\ast}[ \overline{\mathbb{R}^{n}} ; \mathcal{C} ]$ if we exclude $\mf$ spatially. This determines new coordinate $( x_{\tindex}, y_{\tindex}, z^{\tindex}, \tau_{\tindex}, \mu_{\tindex}, \zeta^{\tindex} )$ with respect to the form
\begin{equation*}
\tau_{\tindex} \frac{dx_{\tindex}}{x_{\tindex}^{2}} + \mu_{\tindex} \cdot \frac{dy_{\tindex}}{x_{\tindex}} + \zeta^{\tindex} \cdot dz^{\tindex}.
\end{equation*} \par

Meanwhile, we will use coordinates (\ref{coordinates near mf cap ff}) spatially in a small neighborhood of $\mf \cap \ff$. Then by applying the same pull-back procedure as before, we observe that
\begin{equation} \label{three-body near corner local frame 1}
\frac{dx_{\tindex}}{x_{\tindex}^2}, \ \frac{dy_{\tindex}}{x_{\tindex}}, \ \frac{d\bdf}{x^{\tindex} \bdf }, \ \frac{dy^{\tindex}}{x^{\tindex}}
\end{equation}
is a local frame for ${^{\mathrm{3sc}}T^{\ast}}[\overline{\mathbb{R}^{n}}; \mathcal{C}]$. Indeed, in such coordinates, we have
\begin{equation*}
\begin{gathered}
\frac{d \bdf}{x^{\tindex} \bdf } = \sum_{j=1}^{n^{\tindex}-1} y^{\tindex}_{k} \frac{dY^{\tindex}_{k}}{x_{\tindex}} + ( 1 - |y^{\tindex}|^{1/2} ) \frac{dY^{\tindex}_{n^{\tindex}}}{x_{\tindex}}, \\
\frac{dy^{\tindex}_{k}}{x^{\tindex}} = \frac{dY^{\tindex}_{k}}{x_{\tindex}} - y^{\tindex}_{k} \frac{d \bdf}{x^{\tindex} \bdf}, \ k \neq N^{\tindex}.
\end{gathered}
\end{equation*} \par

Note also that since
\begin{equation*}
\frac{dx^{\tindex}}{(x^{\tindex})^2} = \bdf \frac{dx_{\tindex}}{x_{\tindex}^2} - \frac{d\bdf}{x^{\tindex} \bdf },
\end{equation*}
we could use
\begin{equation} \label{three-body near corner local frame 2}
\frac{dx_{\tindex}}{x_{\tindex}^2}, \ \frac{dy_{\tindex}}{x_{\tindex}}, \ \frac{dx^{\tindex}}{(x^{\tindex})^2}, \ \frac{dy^{\tindex}}{x^{\tindex}}
\end{equation}
equivalently as a valid frame as well. Thus in this region, we can also introduce coordinates $( \bdf, y_{\tindex}, x^{\tindex}, y^{\tindex}, \tau_{\tindex}, \mu_{\tindex}, \tau^{\tindex}, \mu^{\tindex} )$ with respect to the form
\begin{equation} \label{three-body covector fields}
\tau_{\tindex} \frac{dx_{\tindex}}{x_{\tindex}^{2}} + \mu_{\tindex} \cdot \frac{dy_{\tindex}}{x_{\tindex}} + \tau^{\tindex} \frac{dx^{\tindex}}{(x^{\tindex})^2} + \mu^{\tindex} \cdot \frac{dy^{\tindex}}{x^{\tindex}}.
\end{equation} \par

By changing the basis back from (\ref{three-body near corner local frame 2}) to (\ref{three-body near corner local frame 1}), one shows that (\ref{three-body covector fields}) becomes
\begin{equation*}
\tau_{\tindex}^{\mathrm{3sc,sf}} \frac{dx_{\tindex}}{x_{\tindex}^{2}} + \mu_{\tindex}^{\mathrm{3sc,sf}} \cdot \frac{dy_{\tindex}}{x_{\tindex}} + \tau^{\tindex}_{\mathrm{3sc,sf}} \frac{d\bdf}{x^{\tindex}\bdf} + \mu^{\tindex}_{\mathrm{3sc,sf}} \cdot \frac{dy^{\tindex}}{x^{\tindex}}.
\end{equation*}
Here, the subscript `sf' abbreviates `scattering-fibred' in the sense of \cite{AAScatteringFibred01,AAScatteringFibred99}. See also the Remark \ref{scattering fibered structure remark} below. Moreover, the change of variables is given by
\begin{equation} \label{OG tau}
\tau_{\tindex}^{\mathrm{3sc,sf}} = \tau_{\tindex} + \bdf \tau^{\tindex}, \ \mu_{\tindex}^{\mathrm{3sc,sf}} = \mu_{\tindex}, \ \tau^{\tindex}_{\mathrm{3sc,sf}} = - \tau^{\tindex}, \ \mu^{\tindex}_{\mathrm{3sc,sf}} = \mu^{\tindex}.
\end{equation}
Thus in particular, we can use
\begin{equation}
\label{three-body covector fields from the other direction}
 \bdf, y_{\tindex}, x^{\tindex},   y^{\tindex}, \tau_{\tindex}^{\mathrm{3sc,sf}},  \mu_{\tindex},   \tau^{\tindex}_{\mathrm{3sc,sf}} , \mu^{\tindex} 
\end{equation}
as equally valid coordinates near ${ ^{\mathrm{3sc}}T^{\ast}_{\mf \cap \ff} } [ \overline{\mathbb{R}^{n}} ; \mathcal{C}_{\tindex} ]$. \par

In view of of pull-back definition (\ref{bundle pull back definition of the three-body cotangent bundle}) and the Euclidean structure (\ref{special diffeomorphism of the scattering bundle}), we can also understand the three-body cotangent bundle as 
\begin{equation*} \label{Euclidean three-body blow up}
^{\mathrm{3sc}}T^{\ast} [ \overline{\mathbb{R}^{n}} ; \mathcal{C}_{\tindex} ] = [ \overline{\mathbb{R}^{n}} ; \mathcal{C}_{\tindex} ] \times \mathbb{R}^{n}.
\end{equation*}
As before, this is reflected by the change of variables from $( z_{\tindex}, z^{\tindex} )$ to any one of the above coordinates being uniformly valid up to the boundary of $[ \overline{\mathbb{R}^{n}} ; \mathcal{C}_{\tindex} ]$. In fact, one could even do this partially in the fiber variables, i.e., one could replace $\zeta_{\tindex}$, $\zeta^{\tindex}$ by $(\tau_{\tindex}, \mu_{\tindex})$, $(\tau^{\tindex}, \mu^{\tindex})$ respectively whenever necessary. The resulting local coordinates will always be valid. \par

Another way to understand $^{\mathrm{3sc}}T^{\ast} [ \overline{\mathbb{R}^{n}} ; \mathcal{C}_{\tindex} ]$ is to look at it as a blow-up. This is best understood if we also compactify $^{\mathrm{3sc}}T^{\ast} [ \overline{\mathbb{R}^{n}} ; \mathcal{C}_{\tindex} ]$ radially in the fibers, thereby creating the manifold with corner $\overline{ ^{\mathrm{3sc}}T^{\ast} } [ \overline{\mathbb{R}^{n}} ; \mathcal{C}_{\tindex} ] \cong [ \overline{\mathbb{R}^{n}} ; \mathcal{C}_{\tindex} ] \times \overline{\mathbb{R}^{n}}$. Then we have
\begin{equation*}
\overline{ ^{\mathrm{3sc}}T^{\ast} } [ \overline{\mathbb{R}^{n}} ; \mathcal{C}_{\tindex} ] = [ \overline{^{\mathrm{sc}}T^{\ast}}\overline{\mathbb{R}^{n}} ; \overline{ ^{\mathrm{sc}}T^{\ast} }_{\mathcal{C}_{\tindex}} \overline{\mathbb{R}^{n}} ],
\end{equation*}
where the blow-down map is simply $\beta_{\mathrm{3sc}}$ (i.e., the blow-down map of $[ \overline{\mathbb{R}^{n}} ; \mathcal{C}_{\tindex} ]$) acting on the first factor of $[ \overline{\mathbb{R}^{n}} ; \mathcal{C}_{\tindex} ] \times \overline{\mathbb{R}^{n}}$.

\begin{remark}
\label{scattering fibered structure remark}
There is a more general, and perhaps also structurally more illuminating way of understanding the three-body structure, that is to follow the program of Hassell-Vasy \cite{AAScatteringFibred01,AAScatteringFibred99}, where they realized the three-body structure as a special case of the scattering-fibred structure on $[\overline{\mathbb{R}^{n}}; \mathcal{C}]$, with the fibrations being just the blow-down map $\beta_{\mathrm{3sc}}: \mathrm{ff}_{\tindex} \rightarrow \mathcal{C}_{\tindex}$. The readers are encouraged to read their papers to understand this more general perspective. \
\end{remark}

Below we shall characterize conormal symbols on $\overline{ ^{\mathrm{3sc}}T^{\ast} }[ \overline{\mathbb{R}^{n}} ; \mathcal{C}_{\tindex} ]$ by estimates. Suppose that $m,r,l \in \mathbb{R}$ measure decay at ${^{\mathrm{3sc}} S^{\ast}}[\overline{\mathbb{R}^{n}}; \mathcal{C}]$, $\overline{ ^{\mathrm{3sc}} T^{\ast}}_{ \mf }[\overline{\mathbb{R}^{n}}; \mathcal{C}]$ and $\overline{ ^{\mathrm{3sc}} T^{\ast}}_{ \ff }[\overline{\mathbb{R}^{n}}; \mathcal{C}]$ respectively. In the same order, we will also write $\rho_{\infty}, \rho_{\mathrm{mf}}$ and $\rho_{\mathrm{ff}_{\tindex } }$ for boundary defining functions of these boundary faces. Then 
\begin{equation*}
S^{m,r,l}( \overline{ ^{\mathrm{3sc}} T^{\ast}}[ \overline{\mathbb{R}^{n}} ; \mathcal{C} ] )
\end{equation*}
consists of those $a \in \mathcal{C}^{\infty}( T^{\ast} \mathbb{R}^{n} )$ such that in a small neighborhood of $\overline{ ^{\mathrm{3sc}} T^{\ast} }_{\mathrm{ff}_{\tindex}} [ \overline{\mathbb{R}^{n}} ; \mathcal{C} ]$, we have
\begin{equation} \label{symbol estimate three-body ordinary constant order}
| \partial_{z_{\alpha}}^{\beta_{\alpha}} \partial_{z^{\alpha}}^{\beta^{\alpha}} \partial_{\zeta}^{\gamma} a | \leq C_{\beta_{\alpha} \beta^{\alpha} \gamma} \rho_{\mathrm{mf}}^{-r + |\beta_{\tindex}| + |\beta^{\tindex}|} \bdf^{-l + |\beta_{\tindex}|}  \rho_{\infty}^{-m + | \gamma |}.
\end{equation}
Note that in coordinates $(x_{\tindex}, y_{\tindex}, z^{\tindex}, \tau_{\tindex}, \mu_{\tindex}, \zeta^{\tindex} )$, the above becomes
\begin{equation} \label{symbol estimate three-body ordinary constant order boundary coordiates interior}
| (x_{\tindex} \partial_{x_{\tindex}})^{j} \partial_{y_{\tindex}}^{\beta_{\tindex}} \partial_{z^{\tindex}}^{\beta^{\tindex}} \partial_{\tau_{\tindex}}^{k} \partial_{\mu_{\tindex}}^{\gamma_{\tindex}} \partial_{\zeta^{\tindex}}^{\gamma^{\tindex}} a | \leq C_{j  \beta_{\tindex} \beta^{\tindex} k \gamma_{\tindex} \gamma^{\tindex} } \rho_{\mathrm{mf}}^{-r +  |\beta^{\tindex}|} \rho_{\mathrm{ff}_{\tindex}}^{-l } \rho_{\infty}^{-m + k + |\gamma_{\tindex}| + |\gamma^{\tindex}| },
\end{equation}
while in the coordinates $( \bdf , y_{\tindex}, x^{\tindex}, y^{\tindex}, \tau_{\tindex}, \mu_{\tindex}, \tau^{\tindex}, \mu^{\tindex} )$, they become
\begin{align} \label{symbol estimate three-body ordinary constant order boundary coordiates corner}
\begin{split}
& | ( \bdf \partial_{ \bdf })^{j_{\tindex}} \partial_{y_{\tindex}}^{\beta_{\tindex}} ( 
x^{\tindex} \partial_{x^{\tindex}})^{j^{\tindex}} \partial_{y^{\tindex}}^{\beta^{\tindex}} \partial_{\tau_{\tindex}}^{k_{\tindex}} \partial_{\mu_{\tindex}}^{\gamma_{\tindex}} \partial_{\tau^{\tindex}}^{k^{\tindex}} \partial_{\mu^{\tindex}}^{\gamma^{\tindex}} a  | \\
& \qquad \leq C_{j_{\tindex} \beta_{\tindex} j^{\tindex} (\beta^{\tindex}) k_{\tindex} \gamma_{\tindex} k^{\tindex} (\gamma^{\tindex}) } \rho_{\mathrm{mf}}^{-r } \rho_{\mathrm{ff}_{\tindex}}^{-l_{\tindex} } \rho_{\infty}^{-m + k_{\tindex} + k^{\tindex} + |\gamma_{\tindex}| + |(\gamma^{\tindex})'|}.
\end{split}
\end{align}
While away from $\overline{^{\mathrm{3sc}}T^{\ast}}\Xd$, we simply require that $a$ satisfies the estimates (\ref{Euclidean scattering symbolic estimates}).

The spaces of conormal three-body pseudodifferential operators
\begin{equation*}
\Psi^{m,r,l}_{\mathrm{3scc}}( [ \overline{\mathbb{R}^{n}} ; \mathcal{C}_{\tindex} ] )
\end{equation*}
will now be defined by the direct quantization of symbols $a \in S^{m,r,l}( \overline{ 
^{\mathrm{3sc}}T^{\ast} } [ \overline{\mathbb{R}^{n}} ; \mathcal{C}_{\tindex} ]  )$, i.e., the same formula (\ref{scattering quantization}) as in the definition for scattering operators. It can be shown that the operators thus defined are closed under composition:
\begin{equation*}
\Psi^{m_1,r_1,l_1}_{\mathrm{3scc}}( [ \overline{\mathbb{R}^{n}} ; \mathcal{C}_{\tindex} ] ) \times \Psi^{m_2,r_2,l_2}_{\mathrm{3scc}}( [ \overline{\mathbb{R}^{n}} ; \mathcal{C}_{\tindex} ] ) \rightarrow \Psi^{m_1 + m_2,r_1 +r_2,l_1+l_2}_{\mathrm{3scc}}( [ \overline{\mathbb{R}^{n}} ; \mathcal{C}_{\tindex} ] ).
\end{equation*}
We three-body algebra is then the filtered algebra defined by the union
\begin{equation*}
\Psi_{\mathrm{3scc}}( [ \overline{\mathbb{R}^{n}} ; \mathcal{C}_{\tindex} ] ) \coloneq \bigcup_{m,r,l \in \mathbb{R}} \Psi_{\mathrm{3scc}}^{m,r,l}( [ \overline{\mathbb{R}^{n}} ; \mathcal{C}_{\tindex} ] ).
\end{equation*}
Moreover, the map ${ ^{\mathrm{3sc}} \sigma }_{m,r,l}(A) = a$ defines a short exact sequence
\begin{align*}
& 0 \rightarrow \Psi^{ m - 1 , \mathsf{r} - 1, l }_{\mathrm{3scc}}  ( [ \overline{\mathbb{R}^{n}} ; \mathcal{C}_{\tindex} ] ) \rightarrow \Psi^{ m, r }_{\mathrm{3scc}} ( [ \overline{\mathbb{R}^{n}} ; \mathcal{C}_{\tindex} ] ) \\
& \qquad \qquad \longrightarrow_{^{\mathrm{3sc}}\sigma_{ m ,r,l }} ( S^{ m, r, l } / S^{ m - 1, r - 1, l  } ) ( \overline{ ^{\mathrm{3sc}}T^{\ast} } [ \overline{\mathbb{R}^{n}} ; \mathcal{C}_{\tindex} ] ) \rightarrow 0
\end{align*}
such that the descended isomorphism
\begin{equation*}
{ ^{\mathrm{3sc}} \sigma_{m,r,l} } : ( \Psi_{\mathrm{3scc}}^{m,r,l} / \Psi^{m-1, r-1, l}_{\mathrm{3scc}} ) ( [ \overline{\mathbb{R}^{n}} ; \mathcal{C}_{\tindex} ] ) \xrightarrow{\sim} ( S^{ m, r, l } / S^{ m - 1, r - 1, l  } ) ( \overline{ ^{\mathrm{3sc}}T^{\ast} } [ \overline{\mathbb{R}^{n}} ; \mathcal{C}_{\tindex} ] )
\end{equation*}
is multiplicative in the sense that
\begin{equation*}
^{\mathrm{3sc}}\sigma_{ m_1 + m_2 , r_1 + r_2, l_1 + l_2 }(A_1 A_2) = {^{\mathrm{3sc}}\sigma_{ m_1, r_1, l_1 }}(A_1) {^{\mathrm{3sc}}\sigma_{ m_2, r_2, l_2 }}(A_2)
\end{equation*} 
whenever $A_{j} \in \Psi^{m_j,r_j,l_j}_{\mathrm{3sc}}( [\overline{\mathbb{R}^{n}}; \mathcal{C}_{\tindex}] )$, $j=1,2$. \par

To capture principal order decay at $\overline{^{\mathrm{3sc}}T^{\ast}}_{\ff} [ \overline{\mathbb{R}^{n}} ; \mathcal{C} ]$, we will need to consider the \emph{indicial operator}. To this end, let $A = \mathrm{Op}(a) \in \Psi^{m,r,l}_{\mathrm{3scc}}( [ \overline{\mathbb{R}^{n}} ; \mathcal{C}_{\tindex} ] )$, where $a \in S^{m,r,l}(  \overline{ ^{\mathrm{3sc}}T^{\ast} } [ \overline{\mathbb{R}^{n}} ; \mathcal{C}_{\tindex} ] )$ is \emph{partially classical} at $\overline{ ^{\mathrm{3sc}}T^{\ast}}_{\ff}[ \overline{\mathbb{R}^{n}} ; \mathcal{C}_{\tindex} ]$, in the sense that
\begin{equation*}
a - \sum_{j=0}^{J-1} x_{\tindex}^{-l+j} a_{j} \in S^{m , r - J, l - J}(   \overline{ ^{\mathrm{3sc}}T^{\ast} } [ \overline{\mathbb{R}^{n}} ; \mathcal{C}_{\tindex} ] )
\end{equation*}
is valid locally near $\overline{ ^{\mathrm{3sc}}T^{\ast}}_{\ff}[ \overline{\mathbb{R}^{n}} ; \mathcal{C}_{\tindex} ]$, where for every $j \geq 0$, we have
\begin{equation*}
a_{j} \in \mathcal{C}^{\infty}( \mathcal{C}_{\tindex} \times \mathbb{R}^{n_{\tindex}} ; S^{m, r - l - j}(\overline{ ^{\mathrm{sc}}T^{\ast} X^{\tindex} } ) )
\end{equation*}
with a large-parameter behavior as $|\zeta_{\tindex}| \rightarrow \infty$. See \S\ref{parameters-dependent families subsection} below for a precise description of this statement. Then the indicial operator of $A$ is defined by
\begin{equation*}
{ ^{\mathrm{3sc}} \hat{N}_{\ff, l} }(A) ( y_{\tindex}, \tau_{\tindex}, \mu_{\tindex} ) \coloneq \frac{1}{(2\pi)^{n^{\tindex}}} \int_{\mathbb{R}^{n^{\tindex}}} e^{ i ( z^{\tindex} - (z^{\tindex})' ) \cdot \zeta^{\tindex} }  a_0 ( y_{\tindex}, z^{\tindex}, \tau_{\tindex}, \mu_{\tindex}, \zeta^{\tindex} ) d\zeta^{\tindex} |dz^{\tindex}|.
\end{equation*} \par

Now, it can be shown that
\begin{equation} \label{membership of the three-body indicial operator}
{ ^{\mathrm{3sc}} \hat{N}_{\ff, l} }(A) \in \mathcal{C}^{\infty}( \mathcal{C}_{\tindex} \times \mathbb{R}^{n_{\tindex}} ; \Psi^{m,r-l}_{\mathrm{sc}}( X^{\tindex} ) )
\end{equation}
with a large-parameter behavior as $|\zeta_{\tindex}| \rightarrow \infty$, see \S \ref{parameters-dependent families subsection} again. Moreover, indicial operators are multiplicative in the sense that 
\begin{equation*}
{ ^{\mathrm{3sc}} \hat{N}_{\ff, l_1 + l_2} }(A_1 A_2) = 
{ ^{\mathrm{3sc}} \hat{N}_{\ff, l_1} }(A_1) { ^{\mathrm{3sc}} \hat{N}_{\ff, l_2} }(A_2)
\end{equation*}
whenever $A_{j} \in \Psi^{m_j,r_j,l_j}_{\mathrm{3sc}}( [\overline{\mathbb{R}^{n}}; \mathcal{C}_{\tindex}] )$, $j=1,2$, are partially classical at $\overline{^{\mathrm{3sc}}T^{\ast}}_{\ff} [ \overline{\mathbb{R}^{n}} ; \mathcal{C}_{\tindex} ]$. \par

Let $A^{\ast}$ be the adjoint of $A$ taken with respect to $L^{2}( \mathbb{R}^{n} )$ for $A \in \Psi^{m,r,l}_{\mathrm{3sc}}( [ \overline{\mathbb{R}^{n}} ; \mathcal{C}_{\tindex} ] )$. Then it can be shown that $A^{\ast} \in \Psi^{m,r,l}_{\mathrm{3sc}}( [ \overline{\mathbb{R}^{n}} ; \mathcal{C}_{\tindex} ] )$ as well, and we have
\begin{equation} \label{three-body adjoint}
{^{\mathrm{3sc}}\sigma_{m,r,l}}(A) = \overline{ {^{\mathrm{3sc}}\sigma_{m,r,l}}(A) }, \quad {^{\mathrm{3sc}}\hat{N}_{\ff, l}}(A^{\ast}) = {^{\mathrm{3sc}}\hat{N}_{\ff,l} }(A)^{\ast}.
\end{equation}
The adjoint of the rightmost term of (\ref{three-body adjoint}) can be interpret as being pointwise with respect to $L^{2}(X^{\tindex})$, and it is not be hard to see that ${^{\mathrm{3sc}}\hat{N}^{\ast}}(A)$ also has a large-parameter structure as $|\zeta_{\tindex}| \rightarrow 0$. In particular, the conormal three-body algebra is endowed with a symbol calculus. It is therefore also referred to as the \emph{three-body calculus}. 
\par

There is again a suitable notion of microlocalization with respect to the principal symbol. However, one could now even define microlocalization at $\overline{ 
^{\mathrm{3sc}}T^{\ast}}_{\ff} [ \overline{\mathbb{R}^{n}} ; \mathcal{C}_{\tindex} ]$, though this is not quite a `complete' picture, and would require further analysis, such as via a blow-up at the fiber infinity{\ep}much life the resolution considered in \S \ref{a further resolution at fiber infinity} below. See \cite{AndrasNbody} for more details on this incompleteness. \par

The three-body Sobolev spaces $H^{m,r,l}( \mathbb{R}^{n} )$ can also be defined. For $r=l=0$, they are simply defined by $H^{m,0}$ as the scattering case, and for $r, l \neq 0$ by factoring in the spatial weights, i.e., 
\begin{equation*}
H^{m,r,l}( \mathbb{R}^{n} ) = x_{\mf}^{r} x_{\ff}^{l} H^{m,0,0}( \mathbb{R}^{n} )
\end{equation*}
where $x_{\mf}, x_{\ff} \in \mathcal{C}^{\infty}( [ \overline{\mathbb{R}^{n}} ; \mathcal{C}_{\tindex} ] )$ are boundary defining functions for $\mf$ and $\ff$ respectively. \par

Finally, we can replace $m$ and $r$ by variable orders 
\begin{equation*}
\vom \in \mathcal{C}^{\infty}( ^{\mathrm{3sc}}S^{\ast}[ \overline{\mathbb{R}^{n}} ; \mathcal{C_{\tindex}} ] ), \quad \vor \in \mathcal{C}^{\infty}( \overline{^{\mathrm{3sc}}T^{\ast}}_{\mf} [ \overline{\mathbb{R}^{n}} ; \mathcal{C}_{\tindex} ] )
\end{equation*}
respectively, though not $l$, which requires the fiber blow-up mentioned above (see \cite{JesseWave} for this discussion in the Klein-Gordon case). All of the above results have natural analogies in the variable orders case, but we do not go into further details on this.

\subsection{Vasy's second microlocal calculus} 
\label{subsection Vasy's second microlocalized calculus}
We have already discussed the main idea of Vasy's second microlocalized algebra in the introduction. For a more thorough construction, the readers are immediately referred to Vasy's original papers \cite{AndrasSM,AndrasLagrangian1,AndrasLagrangian2}. In general, Vasy's second microlocalization can be done on any smooth, compact manifold $M$ with boundary. As in the case of the scattering algebra, here we just give a summarized version of this material in the case where $M = \overline{\mathbb{R}^{n}}$. \par

There are two equivalent perspectives one could adopt in understanding Vasy's method of second microlocalization. We will first introduce the `direct' perspective. As mentioned already in \S \ref{an overview of second microlocalization subsection}, our primary goal for considering second microlocalization is to study the zero-energy problem for the Helmhotz operator $P(\sigma)$ (as defined in (\ref{the zero energy Helmhotz operator})), $\sigma \geq 0$, whose characteristic set degenerates at $o_{\mathbb{S}^{n-1}}$ in ${ \overline{^{\mathrm{sc}}T^{\ast}}  } \overline{\mathbb{R}^{n}}$ as $\sigma \rightarrow 0$. Here, $o_{\mathbb{S}^{n-1}}$ denotes the zero section in ${ ^{\mathrm{sc}}T^{\ast}}\overline{\mathbb{R}^{n}}$. Thus the direct perspective is to define an algebra of operators which can be microlocalized on the phase space $[ \overline{ ^{\mathrm{sc}}T^{\ast} } \overline{\mathbb{R}^{n}} ; o_{\mathbb{S}^{n-1}} ]$, and this is indeed `direct' in the sense that it tackles the problem which motivates us directly. \par

However, it is very difficult to work from the direct perspective. We will therefore consider a `converse' perspective, in which case one instead starts from the space of conormal b-algebra $\Psi_{\mathrm{bc}}( \overline{\mathbb{R}^{n}} )$, and blow up its natural phase space, namely the fiber-compactified b-cotangent bundle $\overline{ ^{\mathrm{b}}T^{\ast} } \overline{\mathbb{R}^{n}}$, at the corner, i.e., we will consider the blow-up $[ \overline{^{\mathrm{b}}T^{\ast} } \overline{\mathbb{R}^{n}} ; {^{\mathrm{b}}S^{\ast}_{\mathbb{S}^{n-1}} \overline{\mathbb{R}^{n}}} ]$. Such a blow-up will then be highly manageable, since the resolution occurs only at fiber infinity, and is thus relevant only at the symbolic level. \par

One can connect the direct perspective with the converse perspective by observing that
\begin{equation} 
\label{second microlocalization diffeomorphism in the introduction section 2}
[ \overline{ ^{\mathrm{sc}}T^{\ast}} \overline{\mathbb{R}^{n}} ; o_{\mathbb{S}^{n-1}} ] \cong [ \overline{ ^{\mathrm{b}}T^{\ast} }\overline{\mathbb{R}^{n}} ;  {^{\mathrm{b}}S^{\ast}_{ \mathbb{S}^{n-1} }} \overline{\mathbb{R}^{n}} ].
\end{equation} 
Here, the diffeomorphism (\ref{second microlocalization diffeomorphism in the introduction section 2}) is realized by extending the identify map from the interior. Moreover, the lifts of ${ ^{\mathrm{sc}}S^{\ast}} \overline{\mathbb{R}^{n}}$ and ${^{\mathrm{b}}S^{\ast} } \overline{\mathbb{R}^{n}}$ are identified; the lift of $\overline{^{\mathrm{sc}} T^{\ast}}_{\mathbb{S}^{n-1}} \overline{\mathbb{R}^{n}}$ to the left hand side of (\ref{second microlocalization diffeomorphism in the introduction section 2}) identifies with the lift of $ {^{\mathrm{b}}S^{\ast}}_{\mathbb{S}^{n-1}} \overline{\mathbb{R}^{n}} $ to the right hand side of (\ref{second microlocalization diffeomorphism in the introduction section 2}); and the lift of $\overline{^{\mathrm{b}}T^{\ast}}_{\mathbb{S}^{n-1}} \overline{\mathbb{R}^{n}}$ to the right hand side of (\ref{second microlocalization diffeomorphism in the introduction section 2}) identifies with the lift of $o_{\mathbb{S}^{n-1}}$ to the left hand side of (\ref{second microlocalization diffeomorphism in the introduction section 2}).
\par

It will be convenient to write (\ref{second microlocalization diffeomorphism in the introduction section 2}) using the new notation
\begin{equation} \label{second microlocalization diffeomorphism in the introduction section 2.1}
\overline{^{\mathrm{sc,b}}T^{\ast}} \overline{\mathbb{R}^{n}}.
\end{equation}
However, one must be careful as (\ref{second microlocalization diffeomorphism in the introduction section 2.1}) will \emph{no longer be a fiber bundle}. The notation (\ref{second microlocalization diffeomorphism in the introduction section 2.1}) is therefore slightly abusive, but will nonetheless be kept since it resembles the standard conventions used in the literature. We will let $\mathrm{sf}$ (the `scattering face') and $\mathrm{bf}$ (the `b-face') denote the lifts of $\overline{ ^{\mathrm{sc}}T^{\ast} }_{\mathbb{S}^{n-1}} \overline{\mathbb{R}^{n}}$ and $\overline{ ^{\mathrm{b}}T^{\ast} }_{\mathbb{S}^{n-1}} \overline{\mathbb{R}^{n}}$ respectively. We will also let ${ ^{\mathrm{sc,b}}S^{\ast} }\overline{\mathbb{R}^{n}}$ denote the life of ${ ^{\mathrm{sc}}S^{\ast} } \overline{\mathbb{R}^{n}}$, or equivalently the lift of ${ ^{\mathrm{b}} }S^{\ast} \overline{\mathbb{R}^{n}} $, which we shall often refer to as just the `fiber infinity'.
\par

Next, we will consider the spaces of conormal symbols $S^{m,r,l}( \overline{^{\mathrm{sc,b}}T^{\ast}}\overline{\mathbb{R}^{n}} )$, where $m,r,l \in \mathbb{R}$ measure decay respectively at $\mathrm{sf}$, $\mathrm{bf}$ and ${^{\mathrm{sc,b}}S^{\ast}} \overline{\mathbb{R}^{n}}$. In the same order, let $\rho_{\sfa}$, $\rho_{\bfa}$ and $\rho_{\infty}$ be some defining functions for these boundary faces. Then $S^{m,r,l}( \overline{^{\mathrm{sc,b}}T^{\ast}}\overline{\mathbb{R}^{n}} )$ consists of all those $a \in \mathcal{C}^{\infty}( T^{\ast} \mathbb{R}^{n} )$ such that, in a small neighborhood of $\mathrm{bf}$ where the coordinates $(t, y, \tau_{\mathrm{b}}, \mu_{\mathrm{b}})$, $x = e^{-t}$ are valid, we have
\begin{equation}
\label{symbol estimates in the Vasy second microlocalized setting}
| \partial_{t}^{j} \partial_{y}^{\alpha} \partial_{\tau_{\mathrm{b}} }^{k} \partial_{\mu_{\mathrm{b}}}^{\beta} a | \leq C_{j \alpha k \beta} \rho_{\infty}^{-m + k + |\beta|} \rho_{\sfa}^{-r + k + |\beta|} \rho_{\bfa}^{-l},
\end{equation}
and the restrictions of $a$ to regions away from $\mathrm{bf}$ identify with elements of $S^{m,r}(\overline{^{\mathrm{sc}}T^{\ast}} \overline{\mathbb{R}^{n}} )$. Notice in particular that we indeed have 
\begin{equation*}
S^{m,m + l,l}( \overline{^{\mathrm{sc,b}}T^{\ast}}\overline{\mathbb{R}^{n}} ) = S^{m,l}( \overline{ ^{\mathrm{b}}T^{\ast} } \overline{\mathbb{R}^{n}} ),
\end{equation*}
as one could arrange so that $x \simeq \rho_{\sfa} \rho_{\infty}$, where $x$ is some global defining function for $\mathbb{S}^{n-1}$. \par 
The space of second microlocalized operators
\begin{equation*}
\Psi_{\mathrm{sc,b}}^{m,r,l}( \overline{\mathbb{R}^{n}} )
\end{equation*}
is now defined by the following procedure: Recall from the construction of the conormal b-algebra that any $ \tilde{A} \in \Psi^{m,l}_{\mathrm{bc}}( \overline{\mathbb{R}^{n}} )$ is defined by specifying the kernels of $\psi \tilde{A} \psi, \phi \tilde{A} \psi$ for various choices of cut-off functions $\phi, \psi \in \mathcal{C}^{\infty}( \overline{\mathbb{R}^{n}} )$ such that $\phi$, $\psi$ have disjoint supports. Moreover, only the kernels of $\psi \tilde{A} \psi$ are defined by quantizing elements of $S^{m,l}( \overline{ ^{\mathrm{b}}T^{\ast} } \overline{\mathbb{R}^{n}} )$. It turns out that we can define $A \in \Psi^{m,r,l}_{\mathrm{sc,b}}( \overline{\mathbb{R}^{n}} )$ by the same procedure, i.e., we will require that $A : \mathcal{S}( \mathbb{R}^{n} ) \rightarrow \mathcal{S}( \mathbb{R}^{n} )$ be a continuous linear mapping, and then specify the kernels of $\psi A \psi$, $\phi A \psi$ (still viewed as sections of ${ ^{\mathrm{b}}\Omega_{R} } \overline{\mathbb{R}^{n}} $), where $\phi$ and $\psi$ are as before. \par

In fact, in view of (\ref{second microlocalization diffeomorphism in the introduction section 2}), we will simply require that $\phi A \psi \in \Psi_{\mathrm{bc}}^{-\infty,l}(\overline{\mathbb{R}^{n}})$. Namely, we will require that the kernels of $\phi A \psi$ be smooth, and satisfy additionally that $\phi A \psi$ is rapidly decreasing, resp. the decay estimates (\ref{b-kernel K_3}) if $\supp \phi$, $\supp \psi$ are as in cases (2), resp. (3) of that discussion. Similarly, if $\psi$ is supported away from $\mathbb{S}^{n-1}$, then we will require that $\psi A \psi \in \Psi_{\mathrm{bc}}^{m, -\infty}( \overline{\mathbb{R}^{n}} )$, or in other words, $\psi A \psi \in \Psi^{m}_{\infty}( \mathbb{R}^{n} )$ with compact support (again, see \cite{AndrasBook} for where this notation is used). Notice that these definitions make sense since the second blow-up in (\ref{second microlocalization diffeomorphism in the introduction section 2}) occurs only at the spatial infinity of $\overline{ ^{\mathrm{b}}T^{\ast} } \overline{\mathbb{R}^{n}}$.  
\par

The only modification in our definition will therefore occur to the kernels of $\psi A \psi$, where $\psi$ cuts off at some point of $\mathbb{S}^{n-1}$. In this case, we will again require that 
\begin{equation*}
\psi A \psi = B_{\psi} + R_{\psi},
\end{equation*}
and let $R_{\psi}$ be defined exactly as if $A$ were an element of $\Psi^{-\infty,l}_{\mathrm{bc}}( \overline{\mathbb{R}^{n}} )$, i.e., we will require that $R_{\psi}$ has a smooth kernel, and satisfy additionally the decay estimates (\ref{the real b-R decay term}). We then modify the construction of $B_{\psi}$ so that it is still defined by the quantization (\ref{the real b-quantisation in time variables}), except that we now require its symbol $b_{\psi}$ to belong to $S^{m,r,l}( \overline{ ^{\mathrm{sc,b}}T^{\ast} } \overline{\mathbb{R}^{n}} )$ instead. \par

This completes the definition of $A$ and therefore $\Psi^{m,r,l}_{\mathrm{sc,b}}( \overline{\mathbb{R}^{n}} )$. \par

Now, it can be shown that operator composition is a map
\begin{equation*}
\Psi^{ m_1 , r_1 , l_1 }_{\mathrm{sc,b}} ( \overline{\mathbb{R}^{n}} ) \times \Psi^{ m_2 , r_2 , l_2 }_{\mathrm{sc,b}} ( \overline{\mathbb{R}^{n}} ) \rightarrow \Psi^{ m_1 + m_2, r_1 + r_2 , l_1 + l_2}_{\mathrm{sc,b}} ( \overline{\mathbb{R}^{n}} )
\end{equation*}
for every $m_{j}, r_{j}, l_{j} \in \mathbb{R}$, $j=1,2$. Vasy's second microlocalized algebra is then defined by the filtered algebra
\begin{equation*}
\Psi_{\mathrm{sc,b}}( \overline{\mathbb{R}^{n}} ) = \bigcup_{m,r,l \in \mathbb{R}} \Psi_{\mathrm{sc,b}}^{m,r,l}( \overline{\mathbb{R}^{n}} ).
\end{equation*}
Moreover, the principal symbol map descends to an isomorphism
\begin{equation*}
{ ^{\mathrm{sc,b}}\sigma_{m,r,l} }:  ( \Psi_{\mathrm{sc,b}}^{m,r,l} / \Psi^{m-1, r-1, l}_{\mathrm{sc,b}} ) ( \overline{\mathbb{R}^{n}} ) \xrightarrow{\sim} ( S^{m, r, l } / S^{ m - 1, r - 1 , l  } ) ( \overline{ ^{\mathrm{sc,b}}T^{\ast} } \overline{\mathbb{R}^{n}} ),
\end{equation*}
and it is multiplicative in the sense that 
\begin{equation*}
^{\mathrm{sc,b}}\sigma_{ m_1 + m_2 , r_1 + r_2, l_1 + l_2}(A_1 A_2) = {^{\mathrm{sc,b}}\sigma_{ m_1, r_1, l_1 }}(A_1) {^{\mathrm{sc,b}}\sigma_{ m_2, r_2, l_2 }}(A_2)
\end{equation*} 
whenever $A_{j} \in \Psi^{m_j,r_j,l_j}_{\mathrm{sc,b}}(\overline{\mathbb{R}^{n}})$, $j=1,2$.  \par

Let $A^{\ast_{\mathrm{b}}}$ be the adjoint of $A \in \Psi^{m,r,l}_{\mathrm{sc,b}}( \overline{\mathbb{R}^{n}} )$ taken with respect to $L^{2}_{\mathrm{b}}( \overline{\mathbb{R}^{n}} )$. Then by construction (i.e., the converse perspective) we clearly have $A^{\ast_{\mathrm{b}}} \in \Psi^{m,r,l}_{\mathrm{sc,b}}( \overline{\mathbb{R}^{n}} )$. However, since we can also understand $A$ as the modification at zero-energy of some scattering operator, it is natural to consider the Euclidean, $L^{2}( \mathbb{R}^{n} )$-adjoint of $A$ as well, which we shall denote by $A^{\ast}$. Then it is not hard to see that
\begin{equation*}
{^{\mathrm{sc,b}}\sigma_{m,r,l}(A^{\ast_{\mathrm{b}}})} = {^{\mathrm{sc,b}}\sigma_{m,r,l}(A^{\ast})} = \overline{ ^{\mathrm{sc,b}}\sigma_{m,r,l}(A) }.
\end{equation*} 
Thus, when combined with the multiplicative property discussed above, one could say that the principal symbol map promotes the second microlocalized algebra into a calculus.
\par

The usual microlocalization procedure also applies for $\Psi^{m,r,l}_{\mathrm{sc,b}}(\overline{\mathbb{R}^{n}})$, which can be done at the scattering face and the fiber infinity separately. Additionally, variable order operators can also be introduced. Namely, one can replace $m$, $r$ by smooth functions $\mathsf{m} \in \mathcal{C}^{\infty}( ^{\mathrm{sc,b}}S^{\ast} \overline{\mathbb{R}^{n}} )$ and $\mathsf{r} \in \mathcal{C}^{\infty}( \overline{^{\mathrm{sc,b}}T^{\ast}}_{\mathbb{S}^{n-1}} \overline{\mathbb{R}^{n}} )$ respectively. However, the order $l$ needs to remain constant as in the conormal $\mathrm{b}$-case. The corresponding space of conormal symbols
\begin{equation*}
S^{\vom, \vor, l}( \overline{ ^{\mathrm{sc,b}}T^{\ast}} \overline{\mathbb{R}^{n}} )
\end{equation*}
now consists of those $a \in \mathcal{C}^{\infty}( T^{\ast} \mathbb{R}^{n} )$ such that, in a small neighborhood of $\mathrm{bf}$ where the coordinates $(t,y, \tau_{\mathrm{b}}, \mu_{\mathrm{b}})$ are valid, we can replace (\ref{symbol estimates in the Vasy second microlocalized setting}) by
\begin{equation*}
| \partial_{t}^{j} \partial_{y}^{\alpha} \partial_{\tau_{\mathrm{b}} }^{k} \partial_{\mu_{\mathrm{b}}}^{\beta} a | \leq C_{j \alpha k \beta} \rho_{\infty}^{- \vom - \delta |  |  + k + |\beta|} \rho_{\sfa}^{-r + k + |\beta|} \rho_{\bfa}^{-l}.
\end{equation*}
Here $\delta > 0$ must be taken sufficiently small. \par

Finally, we can define a new scale of Sobolev spaces $H^{\vom,\vor,l}_{\mathrm{sc,b}}( \overline{\mathbb{R}^{n}} )$. Much like in our discussion of the b-Sobolev spaces, integrability for elements of $H_{\mathrm{sc,b}}^{\vom,\vor,l}( \overline{\mathbb{R}^{n}} )$ will be measured with respect to the Euclidean density (equivalently the scattering density), instead of some choice of b-density. In fact, having established the definition of $H_{\mathrm{b}}^{m,l}( \overline{\mathbb{R}^{n}} )$, it suffices to set
 \begin{equation*}
 H_{\mathrm{sc,b}}^{\vom,\vor,l}( \overline{\mathbb{R}^{n}} ) \coloneq \{ u \in H_{\mathrm{b}}^{M, l}( \overline{\mathbb{R}^{n}} ) : \Lambda u \in L^{2}( \mathbb{R}^{n} ) \},
 \end{equation*}
where $\Lambda \in \Psi_{\mathrm{sc,b}, \delta}^{\vom, \vor, l}( \overline{\mathbb{R}^{n}} )$ is elliptic in the symbolic sense, and $M$ is sufficiently negative. One can then show that $H_{\mathrm{sc,b}}^{\vom,\vor,l}( \overline{\mathbb{R}^{n}} )$ is a Hilbert space with the square norm
 \begin{equation*}
 \| u \|_{H_{\mathrm{sc,b}}^{\vom,\vor,l}}^2 \coloneq \| u \|_{H_{\mathrm{b}}^{M, l}}^2 + \| \Lambda u \|_{L^2}^2,
 \end{equation*}
though we also remark that the choices of $\Lambda \in \Psi_{\mathrm{sc,b}, \delta}^{\vom,\vor,l}(\overline{\mathbb{R}^{n}})$ and $M$ are unimportant, so long as they satisfy the required conditions.

\subsection{The conormal cone calculus} 
\label{subsection the cone calculus}
We now introduce the conormal cone algebra, which has at least implicitly appeared in \cite{AndrewPeijie} already. We will only consider such an algebra when it is defined on an exact cone, which we now realize as a blow-up $[ \mathbb{R}^{n} ; \{ 0 \} ]$. Upon further compactifying $\mathbb{R}^{n}$ radially, we also obtain $[ \overline{\mathbb{R}^{n}} ; \{ 0 \} ]$. This is a smooth, compact manifold with two disjoint boundary faces: the front face of $[ \overline{\mathbb{R}^{n}} ; \{ 0 \} ]$, which is identified with the `tip' of the cone, and the lift of $\partial \overline{\mathbb{R}^{n}}$, which is identified with the `end' of the cone. Both these faces are diffeomorphic to $\mathbb{S}^{n-1}$, so we will use additional subscripts $\mathbb{S}^{n-1}_{0}$, resp. $\mathbb{S}^{n-1}_{\infty}$ to distinguish the tip, resp. end of the cone.    \par

Let $x_{\infty} \in \mathcal{C}^{\infty}( [ \overline{\mathbb{R}^{n}} ; \{ 0 \} ] )$ be a global defining function for $\mathbb{S}^{n-1}_{\infty}$. Then the Lie algebra of cone vector fields is defined as
\begin{equation*}
\mathcal{V}_{\mathrm{co}}( [ \overline{\mathbb{R}^{n}} ; \{ 0 \} ] ) \coloneq x_{\infty} \mathcal{V}_{\mathrm{b}}( \econ ).
\end{equation*}
In other words, elements of $\mathcal{V}_{\mathrm{co}}( [ \overline{\mathbb{R}^{n}} ; \{ 0 \} ] ) $ are required to be scattering-like near $\mathbb{S}^{n-1}_{\infty}$ and b-like near $\mathbb{S}^{n-1}_{0}$. It follows from the program of Melrose that we can construct a compressed tangent bundle ${ ^{\mathrm{co}}T}[ \overline{\mathbb{R}^{n}} ; \{ 0 \} ]$ such that $\mathcal{C}^{\infty}( [ \overline{\mathbb{R}^{n}} ; \{ 0 \} ] ;  { ^{\mathrm{co}}T}[ \overline{\mathbb{R}^{n}} ; \{ 0 \} ] ) = \mathcal{V}_{\mathrm{co}}( [ \overline{\mathbb{R}^{n}} ; \{ 0 \} ] )$ and ${^{\mathrm{co}}T}_{\mathbb{R}^{n}}[ \overline{\mathbb{R}^{n}} ; \{ 0 \} ] = T \mathbb{R}^{n}$. The dual bundle of ${^{\mathrm{co}}T}\econ$ is denoted $^{\mathrm{co}}T^{\ast} \econ$. \par

Suppose we consider polar variables $( r , y ) \in \mathbb{R}^{n}$, where $r = |z|$, $y = z/|z|$ under the standard abuse of notations, and where $z$ denotes the Euclidean coordinates on $\mathbb{R}^{n}$. Then $r$ extends to be a defining function for $\mathbb{S}^{n-1}_{0}$ in $[ \overline{\mathbb{R}^{n}} ; \{ 0 \} ]$. Meanwhile, if we write $x = r^{-1}$, then $x$ extends to be a defining function for $\mathbb{S}^{n-1}_{\infty}$. It follows that we have
\begin{equation} \label{cone algebra cal 0.09}
[ \overline{\mathbb{R}^{n}} ; \{ 0 \} ] \backslash \mathbb{S}^{n-1}_{0} \cong  [0, \infty)_{x} \times \mathbb{S}^{n-1}_{\infty}, \quad [ \overline{\mathbb{R}^{n}} ; \{ 0 \} ] \backslash \mathbb{S}^{n-1}_{\infty} \cong [ 0, \infty )_{r} \times \mathbb{S}^{n-1}_{0}.
\end{equation}
Additionally, under the identifications (\ref{cone algebra cal 0.09}), we have
\begin{equation} \label{cone algebra cal 0.1}
\begin{gathered}
{^{\mathrm{co}}T^{\ast}_{ [ 0, \infty )_{x} \times \mathbb{S}^{n-1}_{\infty} }} \econ \cong {^{\mathrm{sc}}T^{\ast}_{ [ 0, \infty )_{x} \times \mathbb{S}^{n-1}_{\infty} } }\overline{\mathbb{R}^{n}}, \\
{^{\mathrm{co}}T^{\ast}_{[ 0, \infty )_{r} \times \mathbb{S}^{n-1}_{0}}} \econ \cong { ^{\mathrm{b}}T^{\ast}_{ [0, \infty)_{r} \times \mathbb{S}^{n-1}_{0} } } \econ
\end{gathered}
\end{equation}
as bundles. Together (\ref{cone algebra cal 0.1}) determines the full structure of ${^{\mathrm{co}}T^{\ast}}\econ$. \par

Let $\overline{ ^{\mathrm{co}}T^{\ast} } \econ$ denote the radial compactification of ${^{\mathrm{co}}T^{\ast}} \econ$ in the fibers. Ten we can consider the space of conormal symbols 
\begin{equation} 
\label{constant order conormal symbol cone cotangent bundle}
S^{m,r,l}( \overline{ ^{\mathrm{co}}T^{\ast} } \econ ).
\end{equation}
Here, $m,r ,l \in \mathbb{R}$ measure decay at ${ ^{\mathrm{co}}S^{\ast} } [ \overline{\mathbb{R}^{n}} ; \{ 0 \} ]$, $\overline{ ^{\mathrm{co}}T^{\ast} }_{\mathbb{S}^{n-1}_{\infty}} \econ$ and $\overline{ ^{\mathrm{co}}T^{\ast} }_{\mathbb{S}^{n-1}_{0}} \econ$ respectively. It follows from (\ref{cone algebra cal 0.09}), (\ref{cone algebra cal 0.1}) that (\ref{constant order conormal symbol cone cotangent bundle}) can be characterized by those $a \in \mathcal{C}^{\infty}( T^{\ast} ( \mathbb{R}^{n} \backslash \{ 0 \} ) )$ such that 
\begin{equation*}
\begin{gathered}
\text{$a$ restricts to a $S^{m, r}( \overline{^{\mathrm{sc}}T^{\ast}} \overline{\mathbb{R}^{n}} )$ symbol near $\overline{^{\mathrm{co}}T^{\ast}}_{\mathbb{S}^{n-1}_{\infty}} \overline{\mathbb{R}^{n}}$, } \\
\text{$a$ restricts to a $S^{m, l}(\overline{^{\mathrm{b}}T^{\ast}} \econ)$ symbol near $\overline{^{\mathrm{co}}T^{\ast}}_{\mathbb{S}^{n-1}_{0}} \econ$.}
\end{gathered}
\end{equation*}
\par

We can now define the spaces of conormal cone operators
\begin{equation}
\label{constant orders conormal cone operators first definition}
\Psi^{m,r,l}_{\mathrm{coc}}( \econ )
\end{equation}
by specifying their kernels locally everywhere{\ep}much as in the conormal b-case. As such, elements of (\ref{constant orders conormal cone operators first definition}) are first and foremost continuous linear maps 
\begin{equation*}
A : \dot{\mathcal{C}}^{\infty}( \econ ) \rightarrow \dot{\mathcal{C}}^{\infty}( \econ ).
\end{equation*}
Thus, $A$ can be identified with kernels which are sections of the right cone density bundle ${^{\mathrm{co}}\Omega_{R}}\econ \coloneq \pi_{R}^{\ast} { ^{\mathrm{co}}\Omega } \econ $. Here $^{\mathrm{co}}\Omega [ \overline{\mathbb{R}^{n}} ; \{ 0 \} ]$ denotes the space of cone density arising naturally from $\mathcal{C}^{\infty}( [ \overline{\mathbb{R}^{n}} ; \{ 0 \} ] ; { ^{\mathrm{co}}T^{\ast} } [ \overline{\mathbb{R}^{n}} ; \{ 0 \} ] )$, and $\pi_{R} : \econ^2 \rightarrow \econ$ is the right projection. The dependency on ${ ^{\mathrm{co}}\Omega }_{R} \econ$ can be omitted upon choosing a global trivialization, which can be done by fixing a strictly positive $\nu_{\mathrm{co}} \in {^{\mathrm{co}}\Omega}\econ$. Without loss of generality, spatially in a small neighborhood of $\mathbb{S}^{n-1}_{+}$, resp. $\mathbb{S}^{n-1}_{0}$, we may assume that
\begin{equation*}
\nu_{\mathrm{co}} = |dz|, \ \text{resp.} \ \nu_{\mathrm{co}} = | dt dy  |
\end{equation*}
where $z$ is the Euclidean coordinates on $\mathbb{R}^{n}$ and $(t,y)$ is obtained by taking $r = e^{-t}$, where $r = |z|$ is a local defining function for $\mathbb{S}^{n-1}_{0}$. \par

Let $\psi \in \mathcal{C}^{\infty}( \econ )$ a cut-off function at some point of $\mathbb{S}^{n-1}_{0}$. Then we will require that
\begin{equation} 
\label{cone algebra b parts}
\psi A \psi \in \Psi_{\mathrm{bc}}^{m,r}( \econ ).
\end{equation}
On the other hand, if $\psi \in \mathcal{C}^{\infty}( \econ )$ is supported away from $\mathbb{S}^{n-1}_{0}$ (possibly intersecting $\mathbb{S}^{n-1}_{\infty}$), then we will require that
\begin{equation} 
\label{cone algebra scattering part}
\psi A \psi  \in \Psi^{m,r}_{\mathrm{sc}}( \overline{\mathbb{R}^{n}} ).
\end{equation}
\par

 It remains to specify $\phi A \psi$, where $\phi \in \mathcal{C}^{\infty}( \econ )$ is another cut-off such that $\supp \phi \cap \supp \psi = \emptyset$. There is a total of six cases to be considered:
\begin{enumerate}
    \item $ \supp \phi $ and $ \supp \psi $ are both disjoint from $\mathbb{S}^{n-1}_{0} \cup \mathbb{S}^{n-1}_{\infty}$,
    \item $\supp \phi$ intersects $\mathbb{S}^{n-1}_{0}$ and $\supp \psi$ intersects $\mathbb{S}^{n-1}_{\infty}$, or vice versa,
    \item $\supp \phi$ is disjoint from $ \mathbb{S}^{n-1}_{0} \cup \mathbb{S}^{n-1}_{\infty}$ but $\supp \psi$ intersects $\mathbb{S}^{n-1}_{0}$, or vice versa,
    \item $\supp \phi$ is disjoint from $ \mathbb{S}^{n-1}_0 \cup \mathbb{S}^{n-1}_{\infty} $ but $\supp \psi$ intersects $\mathbb{S}^{n-1}_{\infty}$, or vice versa,
    \item both $\supp \phi$ and $\supp \psi$ intersect $\mathbb{S}^{n-1}_{\infty}$, but are still disjoint,
    \item both $\supp \phi$ and $\supp \psi$ intersect $\mathbb{S}^{n-1}_{0}$, but are still disjoint.
\end{enumerate}
Let $K_{\phi, \psi}$ be the kernels of $\phi A \psi$ in each case.
\par

Firstly, we shall require that $K_{\phi, \psi}$ to always be smooth. Secondly, $K_{\phi, \psi}$ must be rapidly decreasing in cases (1)--(5). The only exception is in case (6), where we require that estimates (\ref{b-kernel K_3}) be satisfied. \par

This concludes the construction of $\Psi^{m,r,l}_{\mathrm{coc}}( \econ )$. \par 
As expected, composition is a map
\begin{equation*}
\Psi^{m_1,r_1,l_1}_{\mathrm{coc}}( \econ ) \times \Psi^{m_2,r_2,l_2}_{\mathrm{coc}}( \econ)  \rightarrow \Psi^{m_1+m_2, r_1 + r_2, l_1 + l_2}_{\mathrm{coc}}( \econ )
\end{equation*}
for every $m_{j}, r_{j}, l_{j} \in \mathbb{R}$, $j=1,2$. This can be easily shown since $\Psi^{m,r}_{\mathrm{sc}}( \overline{\mathbb{R}^{n}} )$, $\Psi^{m,l}_{\mathrm{bc}}( \econ )$ both have this property. The conormal cone algebra is then the filtered algebra defined by the union
\begin{equation*}
\Psi_{\mathrm{coc}}( \econ ) \coloneq \bigcup_{m,r,l \in \mathbb{R}} \Psi_{\mathrm{coc}}^{m,r,l}( \econ ).
\end{equation*}
Moreover, by the standard argument, one can show that the principal symbol map defines a short exact sequence
\begin{align*}
 0 & \rightarrow \Psi^{ m - 1 , r - l, 1 }_{\mathrm{coc}}  ( \econ) \rightarrow \Psi^{ m, r, l }_{\mathrm{coc}} ( \econ ) \\ 
& \qquad \longrightarrow_{^{\mathrm{co}}\sigma_{ m , r,  l }} ( S^{ m, r, l } / S^{ m - 1, r - 1,  l  } ) ( \overline{ ^{\mathrm{co}}T^{\ast} } \econ ) \rightarrow 0.
\end{align*}
Thus, $^{\mathrm{co}}\sigma_{m,r,l}$ descends to an isomorphism
\begin{equation*}
{ ^{\mathrm{co}}\sigma_{m,r,l} }:  ( \Psi_{\mathrm{coc}}^{m,r,l} / \Psi^{m-1, r-1, l}_{\mathrm{coc}} ) ( \econ ) \xrightarrow{\sim} ( S^{m, r, l } / S^{ m - 1, r - 1 , l  } ) ( \overline{ ^{\mathrm{co}}T^{\ast} } \econ ),
\end{equation*}
which is multiplicative in the sense that
\begin{equation*}
^{\mathrm{co}}\sigma_{ m_1 + m_2 , r_1 + r_2, l_1 + l_2 }(A_1 A_2) = {^{\mathrm{co}}\sigma_{ m_1, r_1, l_1  }}(A_1) {^{\mathrm{co}}\sigma_{ m_2, r_2, l_2 }}(A_2)
\end{equation*} 
whenever $A_{j} \in \Psi_{\mathrm{coc}}^{m_j,r_j,l_j}(\econ)$, $j=1,2$. \par

Let $L^{2}_{\mathrm{co}}( [ \overline{\mathbb{R}^{n}} ; \{ 0 \} ] ) \coloneq L^{2}( [ \overline{\mathbb{R}^{n}} ; \{ 0 \} ] , \nu_{\mathrm{co}} )$ be the space of $L^2$ functions defined with respect to a strictly positive cone density $\nu_{\mathrm{co}}$, and let $A^{\ast_{\mathrm{co}}}$ be the adjoint of $A \in \Psi_{\mathrm{co}}^{m,r,l}( [ \overline{\mathbb{R}^{n}} ; \{ 0 \} ] )$ with respect to $L^{2}_{\mathrm{co}}( [ \overline{\mathbb{R}^{n}} ; \{ 0 \} ] )$. Then it can be shown easily that $A^{\ast_{\mathrm{co}}} \in \Psi^{m,r,l}_{\mathrm{co}}( [ \overline{\mathbb{R}^{n}} ; \{ 0 \} ] )$ as well. Moreover, we again have 
\begin{equation*}
{^{\mathrm{co}}\sigma_{m,r,l}}(A^{\ast_{\mathrm{co}}}) = \overline{{^{\mathrm{co}}\sigma_{m,r,l}}(A)}.
\end{equation*} 
Thus in particular, with the principal symbol map, one could also refer to $\Psi_{\mathrm{coc}}( [ \overline{\mathbb{R}^{n}} ; \{ 0 \} ] )$ as the \emph{conormal, or small cone calculus}. 
\par

We can also replace $m$ and $r$ by variable orders 
\begin{equation*}
\mathsf{m} \in \mathcal{C}^{\infty}( ^{\mathrm{co}}S^{\ast} \econ  ), \quad \mathsf{r} \in \mathcal{C}^{\infty}( \overline{^{\mathrm{co}}T^{\ast}}_{\mathbb{S}^{n-1}_{\infty}} \econ )
\end{equation*}
respectively. However, the order $l$ must remain constant. If $\delta > 0$ is sufficiently small, then the space of variable orders conormal symbols 
\begin{equation*}
S^{\vom, \vor, l}_{\delta}( \overline{ ^{\mathrm{co}}T^{\ast}} \econ ),
\end{equation*}
can now defined by those $a \in \mathcal{C}^{\infty}( T^{\ast}  ( \mathbb{R}^{n} \backslash \{ 0 \} ) )$ such that
\begin{equation*}
\begin{gathered}
\text{$a$ restricts to a $S^{\vom, \vor}_{\delta}( \overline{^{\mathrm{sc}}T^{\ast}} \overline{\mathbb{R}^{n}} )$ symbol near $\overline{^{\mathrm{co}}T^{\ast}}_{\mathbb{S}^{n-1}_{\infty}} \overline{\mathbb{R}^{n}}$, } \\
\text{$a$ restricts to a $S^{\vom, l}_{\delta}(\overline{^{\mathrm{b}}T^{\ast}} \econ)$ symbol near $\overline{^{\mathrm{co}}T^{\ast}}_{\mathbb{S}^{n-1}_{0}} \econ$.}
\end{gathered}
\end{equation*}
\par

The corresponding space of operators $\Psi_{\mathrm{coc},\delta}^{\vom,\vor,l}( \econ )$ can now be defined in exactly the same way as $\Psi^{m,r,l}_{\mathrm{coc}}( \econ )$, except for minor modifications: Let $\psi \in \mathcal{C}^{\infty}( \econ )$ be a cut-off function at some point of $\mathcal{C}^{\infty}_{0}$. Then we will require that
\begin{equation*}
\psi A \psi \in \Psi_{\mathrm{bc}, \delta}^{\vom, l}( \econ ).
\end{equation*}
If instead $\psi \in \mathcal{C}^{\infty}( \econ )$ is supported away from $\mathbb{S}^{n-1}_{0}$, then we will require that
\begin{equation*}
\psi A \psi \in \Psi_{\mathrm{sc}, \delta}^{\vom, \vor}( \overline{\mathbb{R}^{n}} ).
\end{equation*}
Here the variable orders are implicitly understood as their restrictions.
\subsection{Large-parameter families of operators} 
\label{parameters-dependent families subsection}
In this subsection, we will consider families of conormal b-, scattering and cone operators which are smoothly parametrized by a vector bundle over a compact, smooth manifold without boundary, with a large-parameter behavior (in the sense of Shubin, see \cite{ShubinLP}) in the fibers. Although this construction can be done generally, it will be sufficient for our purpose to only consider the cases where the vector bundles in question are trivial. In fact, it will be convenient to radially compactify the fibers of such bundles.\par

To this end, let $\mathcal{M}$ be a compact, $d$-dimensional smooth manifold without boundary. Let also $M$ be a compact, smooth manifold with boundary. Then for $m,l \in \mathbb{R}$, we will at first introduce the space of large-parameter conormal b-operators
\begin{equation} \label{large-parameter b algebra}
\Psi_{\mathrm{bc,lp}}^{m,l}( M ; \mathcal{M} \times {\mathbb{R}^{k}} ),
\end{equation}
where $\mathcal{M} \times {\mathbb{R}^{k}}$ is the set of parameters. Following the construction of $\Psi_{\mathrm{bc}}^{m,l}( M )$, we will require elements $A \in \Psi_{\mathrm{bc,lp}}^{m,l}(M; \mathcal{M} \times {\mathbb{R}^{k}} )$ to be continuous linear maps $\dot{\mathcal{C}}^{\infty}(M) \rightarrow \dot{\mathcal{C}}^{\infty} ( M )$. Thus $A$ has a distributional Schwartz kernel, which we will view as a section of the right b-density bundle ${^{\mathrm{b}}\Omega_{R}} M$. \par

Let $\psi \in \mathcal{C}^{\infty}(M)$ be a cut-off function at some point of $\partial M$. Then we will require that
\begin{equation}
\label{large-parameters b operator construction 0.1}
\psi A \psi \coloneq B_{\psi} + R_{\psi}.
\end{equation}
Here $B_{\psi}$ takes the form (\ref{the real b-quantisation in time variables}), except that its symbol $b_{\psi}$ now belongs to 
\begin{equation} \label{large-parameters b-symbol spaces.}
S^{m,l}( \mathcal{M} \times  \overline{ ^{\mathrm{b}}T^{\ast} \oplus \mathbb{R}^{k} } M  ).
\end{equation}
The notation $\overline{^{\mathrm{b}}T^{\ast} \oplus \mathbb{R}^{k}} M$ indicates that we are direct summing to each fiber of ${^{\mathrm{b}}T^{\ast}} M$ the Euclidean space $\mathbb{R}^{k}$, and then compactifying the resulting fibers radially. The index $m$ then measures decay at the `joint fiber infinity' of $\mathcal{M} \times \overline{ ^{\mathrm{b}}T^{\ast} \oplus \mathbb{R}^{k}}M$, while the index $l$ measures decay at $\mathcal{M} \times \overline{ ^{\mathrm{b}}T^{\ast} \oplus \mathbb{R}^{k}}_{\partial M}M$. \par

Concretely, if $b_{\psi} \in S^{m,l}( \mathcal{M} \times  \overline{ ^{\mathrm{b}}T^{\ast} \oplus \mathbb{R}^{k} } M )$, then this is equivalent to saying that $b_{\psi} \in \mathcal{C}^{\infty}( \mathcal{M} \times \mathbb{R}^{k} \times T^{\ast} M^{\circ} )$, such that in a neighborhood of $\mathcal{M} \times \overline{ ^{\mathrm{b}}T^{\ast} \oplus \mathbb{R}^{k} }_{\partial M}M$ where the coordinates $( w, \eta, t, y, \tau_{\mathrm{b}}, \mu_{\mathrm{b}} )$ (cf. discussion on the conormal b-symbols in \S \ref{b-calculus subsection}) are valid, and where $(w,\eta) \in M \times \mathbb{R}^{k}$, we have the estimates
\begin{equation} \label{large-parameter b estimates}
| \partial_{w}^{\alpha} \partial^{\sigma}_{\eta} \partial_{t}^{j} \partial_{y}^{\beta} \partial^{\kappa}_{\tau_{\mathrm{b}}} \partial^{\gamma}_{\mu_{\mathrm{b}}} b_{\psi} | \leq C_{ \alpha \sigma j \beta \kappa \gamma}  e^{lt} \langle \tau_{\mathrm{b}}, \mu_{\mathrm{b}}, \eta \rangle^{m - \kappa - |\gamma| - |\sigma| }.
\end{equation}
Meanwhile, over the interior $\mathcal{M} \times \mathbb{R}^{k} \times T^{\ast}M^{\circ}$, and in Euclidean coordinates $( w, \eta, z, \zeta )$, we have the estimates
\begin{equation*}
| \partial_{w}^{\alpha} \partial_{\eta}^{\sigma} \partial_{z}^{\beta} \partial_{\zeta}^{\gamma} b_{\psi} | \leq C_{ \alpha \sigma \beta \gamma } \langle \zeta, \eta \rangle^{m - |\sigma| - |\gamma|}.
\end{equation*} \par

Returning now to (\ref{large-parameters b operator construction 0.1}), we will also require that $R_{\psi}$ be smooth in all of the variables, and additionally satisfies the decay estimates
\begin{equation} \label{estimates large-parameters R_b term}
| \partial_{w}^{\alpha} \partial^{\sigma}_{\eta} \partial_{t}^{j} \partial_{y}^{\beta} \partial_{t'}^{j'} \partial_{y'}^{\beta'} R_{\psi} | \leq C_{\alpha \sigma j \beta j' \beta'}  e^{lt} \langle \eta \rangle^{-N} \langle y - y' \rangle^{-M} e^{-L|t-t'|}
\end{equation}
for any $N,M, L \in \mathbb{R}$, i.e., $R_{\psi}$ remains trivial in the regularity sense, even in the presence of the large-parameter $\eta$. \par

If instead $\psi \in \mathcal{C}^{\infty}(M)$ is supported away from $\partial M$, then we will require that $\psi A \psi$ be given by a standard quantization (i.e., locally everywhere in Euclidean coordinates), with the exception that its symbol now belongs to $S^{m,l}( \mathcal{M} \times \overline{ ^{\mathrm{b}}T^{\ast} \oplus \mathbb{R}^{k} } M )$ as before. \par

Now, let $\phi \in \mathcal{C}^{\infty}(M)$ be another cut-off  function such that $\supp \phi \cap \supp \psi = \emptyset$. Then we can consider the off-diagonal terms $\phi A \psi$ with kernels $K_{\phi, \psi}$. We will also require that $K_{\phi, \psi}$ be smooth in all of the variables, and they must satisfy decay estimates as in the conormal b-cases, which depend on the support properties of $\phi$ and $\psi$ (cf. cases (1)--(3) in \S \ref{b-calculus subsection}). Additionally, $K_{\phi, \psi}$ must decay to infinite orders as $|\eta| \rightarrow \infty$ as well. \par

This completes the construction of $\Psi^{m,0}_{\mathrm{bc,lp}}(M ; \mathcal{M} \times {\mathbb{R}^{k}} )$. \par

Next, for $m,r \in \mathbb{R}$, we will introduce the space of large-parameter scattering operators
\begin{equation} \label{large-parameters operator class for the scattering case}
\Psi^{m,r}_{\mathrm{sc,lp}}( \overline{\mathbb{R}^{n}} ; \mathcal{M} \times {\mathbb{R}^{k}} ).
\end{equation}
Such a space consists of families of operators defined by the direct, Euclidean quantizations (i.e., as in (\ref{scattering quantization})) of symbols in 
\begin{equation} \label{scattering large-parameters symbols}
S^{m,r}( \mathcal{M} \times \overline{ ^{\mathrm{sc}}T^{\ast} \oplus \mathbb{R}^{k} } \overline{\mathbb{R}^{n}} ).
\end{equation}
Here, $m$ measures decay at the joint fiber infinity of $\overline{ ^{\mathrm{sc}}T^{\ast} \oplus \mathbb{R}^{k} } \overline{\mathbb{R}^{n}}$, while $r$ measures decay at the base infinity $\mathcal{M} \times \overline{ ^{\mathrm{sc}}T^{\ast} \oplus \mathbb{R}^{k} }_{\mathbb{S}^{n-1}} \overline{\mathbb{R}^{n}}$. Equivalently, $a$ belongs to (\ref{scattering large-parameters symbols}) if $a \in \mathcal{C}^{\infty}( \mathcal{M} \times \mathbb{R}^{k} \times T^{\ast} \mathbb{R}^{n} )$, and that the estimates
\begin{equation} \label{scattering large-parameters estimate}
| \partial_{w}^{\alpha} \partial^{\sigma}_{\eta} \partial_{z}^{\beta} \partial_{\zeta}^{\gamma} a | \leq C_{ \alpha \sigma \beta \gamma } \langle z \rangle^{r - |\beta|} \langle \zeta , \eta \rangle^{m - |\gamma| - |\sigma| }
\end{equation} 
are satisfied. In fact, in the context of the three-body algebra, with $ \mathcal{M} = \mathcal{C}_{\tindex}$, $k = n_{\tindex}$, and also replacing $\overline{\mathbb{R}^{n}}$ by $X^{\tindex}$ and $r$ by $r - l$. We can show that (\ref{large-parameters operator class for the scattering case}) is exactly the space of operators for which indicial operators of operators (which must also be classical at $\ff$ in some suitable sense) in $\Psi^{m,r,l}_{\mathrm{3scc}}( [ \overline{\mathbb{R}^{n}} ; \mathcal{C}_{\tindex} ] )$ belong to. \par

Finally, for $m,r,l \in \mathbb{R}$, we will also introduce the space of large-parameter families of conormal cone operators
\begin{equation} 
\label{large-parameters cone operators}
\Psi^{m,r,l}_{\mathrm{coc,lp}}( [ \overline{\mathbb{R}^{n}} ; \{ 0 \} ] ; \mathcal{M} \times {\mathbb{R}^{k}} ).
\end{equation}
As before, the space (\ref{large-parameters cone operators}) can be defined in essentially the same way as how we defined $\Psi_{\mathrm{coc}}^{m,r,l}(\econ)$. As such, we will first require that (\ref{large-parameters cone operators}) consists of continuous linear maps $A : \dot{\mathcal{C}}^{\infty}( [ \overline{\mathbb{R}^{n}} ; \{ 0 \} ] ) \rightarrow \dot{\mathcal{C}}^{\infty}( [ \overline{\mathbb{R}^{n}} ; \{ 0 \} ] )$. Thus, each $A$ has a distributional Schwartz kernel, which we will view as a section of ${^{\mathrm{co}}\Omega}_{R} \econ$. \par

Let $\psi \in \mathcal{C}^{\infty}( \econ )$ be a cut-off function at some point of $\mathbb{S}^{n-1}_{0}$ (recall that this is the tip of the cone $\econ$, which is realized as the lift of $\{ 0 \}$). Then we will require that
\begin{equation}
\psi A \psi \in \Psi_{\mathrm{bc,lp}}^{m,l}( \econ ; \mathcal{M} \times {\mathbb{R}^{k}} ).
\end{equation}
On the other hand, if $\psi$ is supported away from $\mathbb{S}^{n-1}_{0}$, possibly intersecting $\mathbb{S}^{n-1}_{\infty}$ (recall that this is the end of the cone $\econ$, which is realized as the lift of $\partial \overline{\mathbb{R}^{n}}$), then we will require that
\begin{equation}
\psi A \psi \in \Psi_{\mathrm{sc,lp}}^{m,r}( \overline{\mathbb{R}^{n}} ; \mathcal{M} \times {\mathbb{R}^{k}}  ).
\end{equation} 
\par

By the above construction, the symbol of $A$ belong to the space
\begin{equation*} 
S^{m,r,l}( \mathcal{M} \times \overline{^{\mathrm{co}}T^{\ast} \oplus \mathbb{R}^{k} } [ \overline{\mathbb{R}^{n}} ; \{ 0 \} ] ).
\end{equation*}
Here, the index $m$ measures decay at the joint fiber infinity of $\overline{ ^{\mathrm{co}}T^{\ast} \oplus \mathbb{R}^{k} } \econ$, $r$ measures decay at $\mathcal{M} \times \overline{ ^{\mathrm{co}}T^{\ast} \oplus \mathbb{R}^{k}}_{\mathbb{S}^{n-1}_{\infty}} \econ$, and $l$ measures decay at $\mathcal{M} \times \overline{ ^{\mathrm{co}}T^{\ast} \oplus \mathbb{R}^{k} }_{\mathbb{S}^{n-1}_{0}} \econ$. Equivalently, $b \in S^{m,r,l}( \mathcal{M} \times \overline{ ^{\mathrm{co}}T^{\ast} \oplus \mathbb{R}^{k} } \econ )$ if and only if $b \in \mathcal{C}^{\infty}( \mathcal{M} \times \mathbb{R}^{k} \times T^{\ast} \mathbb{R}^{n} \backslash \{ 0 \}  )$, and $b$ satisfies the $S^{m,r}( \mathcal{M} \times \overline{ ^{\mathrm{sc}}T^{\ast} \oplus \mathbb{R}^{k} } \overline{\mathbb{R}^{n}} )$ estimates spatially in a neighborhood of $\mathbb{S}^{n-1}_{\infty}$, as well as the $S^{m,l}( \mathcal{M} \times \overline{ ^{\mathrm{b}}T^{\ast} \oplus \mathbb{R}^{k} } \econ )$ estimates spatially in a neighborhood of $\mathbb{S}^{n-1}_{0}$.
 \par
Lastly, if $\phi \in \mathcal{C}^{\infty}( \econ )$ is another cut-off function such that $\supp \phi \cap \supp \psi = \emptyset$. Let $K_{\phi, \psi}$ denotes the kernel of $\phi A \psi$. Then we will require that $K_{\phi, \psi}$ be smooth in all of the variables. They must also satisfy decay estimates as in the conormal cone cases, which depend on the support properties of $\phi$ and $\psi$ (cf. cases (1)--(6) in \S \ref{subsection the cone calculus}). Additionally, $K_{\phi, \psi}$ must always decay to infinite orders as $|\eta| \rightarrow \infty$ (where $\eta$ is the fiber variable for $\mathcal{M} \times \mathbb{R}^{k}$).
\par

The spaces of large-parameter conormal b-, scattering and cone operators also compose in the expected manners, thereby giving rises to filtered algebras of operators as discussed in \S \S \ref{b-calculus subsection}, \ref{subsection the scattering calculus} and \ref{subsection the cone calculus} respectively. Moreover, following essentially the same procedures as those illustrated in the above subsections, one could also discuss the principal symbol maps, adjoints as well as microlocalization in the large-parameter settings. Variable orders can also be introduced in the same ways where possible. For brevity, we will omit giving a detailed presentation of these materials.

\section{The conormal three-cone operators} 
\label{the conormal three-cone operators section}
For brevity, starting from now we shall always assume that
\begin{equation*}
\mathcal{C} = \mathcal{C}_{\tindex},
\end{equation*}
although the construction below generalizes directly to the case where $\mathcal{C} = \{ \mathcal{C}_{\tindex} : \tindex \in \ind \}$ due to disjointness. \par

Recall that we have a natural projection
\begin{equation} \label{important 3sc projection}
{^{\mathrm{3sc}}\pi_{\ff}} : \mathcal{C}_{\tindex} \times \mathbb{R}^{n_{\tindex}}_{\zeta_{\tindex}} \times {^{\mathrm{sc}}T^{\ast}} X^{\tindex} \rightarrow { ^{\mathrm{sc}}T^{\ast} X^{\tindex} }.
\end{equation}
Let $o_{\mathcal{C}^{\tindex}}$ denote the zero section in ${^{\mathrm{sc}}T^{\ast}_{\mathcal{C}^{\tindex}}} X^{\tindex}$. Then our goal is to construct a class operators whose phase space (i.e., where these operators are microlocalized on) is the blow-up
\begin{equation} \label{three-body second microlocalized blow-up}
\left[ \hspace{0.5mm} \overline{^{\mathrm{3sc}}T^{\ast}} [ \overline{\mathbb{R}^{n}} ; \mathcal{C}_{\alpha} ] ; \overline{^{\mathrm{3sc}}\pi_{\ff}^{-1}( o_{\mathcal{C}^{\tindex}} )} \hspace{0.5mm} \right],
\end{equation}
where we are now using more systematic notations in comparison to (\ref{the correct blowup}). \par

As already mentioned, it would be difficult to work directly with (\ref{three-body second microlocalized blow-up}). To get around this technicality, in this section we will construct a `converse' perspective{\ep}similar to how Vasy approaches second microlocalization in the two-body (i.e., scattering) case (see \S \ref{subsection Vasy's second microlocalized calculus} above), in the form of the introduction of a new class of operators, which we will refer to as the conormal \emph{three-cone} operators below. \par

By design, the three-cone operators will act on the position space
\begin{equation*}
\Xd \coloneq [ [ \overline{\mathbb{R}^{n}} ; \mathcal{C}_{\alpha} ] ; \ff \cap \mf ],
\end{equation*}
although at this stage, it is not immediately clear why we require the additional blow-up at the corner of $[ \overline{\mathbb{R}^{n}} ; \mathcal{C}_{\alpha} ]$. This will be clarified in \S \ref{section second microlocalization} below. It can be seen that $\Xd$ is a manifold with corners having three boundary hypersurfaces. These are the lifts of the two boundary hypersurfaces of $[ \overline{\mathbb{R}^{n}} ; \mathcal{C}_{\alpha} ]$, which will be denoted by $\dmf$ (the `decoupled' main face) for the lift of $\mathrm{mf}$, and $\dff$ (the `decoupled' front face) for the lift of $\mathrm{ff}_{\tindex}$, as well as the lift of $\ff \cap \mf$, i.e., the new face that has just been created by the blow-up, which will be denoted by $\mathrm{cf}_{\tindex}$ (for the `corner face'). \par

In the discussion below, we will continue to follow Melrose's program. Thus, in \S \ref{subsection definition of the three-cone bundle and the variable changes}, we will by defining the Lie algebra of three-cone vector fields on $\Xd$, and from there construct the three-cone cotangent bundle, on which we can define a class of conormal symbols. In \S\S \ref{overview of three-cone bundle subsection}--\ref{subsection construction of cf}, we construct spaces of conormal three-cone operators via local cut-offs{\ep}much like in the cases of the conormal b- and cone algebras. In fact, the latter two are heavily related to the conormal three-cone algebra. Indeed, as we will show in \S \ref{subsection partial quantization near dff} and \S \ref{subsection partial quantization near cf}, one could write any conormal three-cone operator near $\dff$ as the partial quantization in the free variables of a family of conormal b-operators,  and near $\cf$ as the partial quantization of a family of conormal cone operators, and away from these faces as a scattering operator. Finally, we will also discuss briefly why the local constructions we present in this section are compatible in \S \ref{compatibility of three-cone operators subsection}.

\subsection{The three-cone vector fields, cotangent bundle and symbols}
\label{subsection definition of the three-cone bundle and the variable changes}
Observe that since $\Xd$ arises from blowing up $\ff \cap \mf$ (i.e., the corner) in $[ \overline{\mathbb{R}^{n}} ; \mathcal{C}_{\alpha} ]$, the geometry of $\Xd \backslash \cf$ is the same as that of $[ \overline{\mathbb{R}^{n}} ; \mathcal{C}_{\alpha} ] \backslash ( \ff \cap \mf )$, and as such requires no further discussion. Thus, we shall henceforth restrict our attention to a neighborhood of $\cf$ (i.e., the lift of $\ff \cap \mf$). \par

Starting from coordinates $(\bdf, x^{\tindex}, y_{\tindex}, y^{\tindex})$ as introduced in (\ref{coordinates near mf cap ff}), we can find local coordinates on $\Xd$ by introducing projective coordinates about $\{ \bdf = x^{\tindex} = 0\}$. Thus, there are two choices of such coordinates: in a neighborhood of $\partial \dff$ (i.e., $\cf \cap \dff$), we have coordinates
\begin{equation} \label{three-cone position variables near dff}
\bbdf \coloneq \frac{\bdf}{x^{\tindex}}, \, x^{\tindex}, \, y_{\tindex}, y^{\tindex},
\end{equation}
where $\bbdf$, resp. $x^{\tindex}$ are local boundary defining functions for $\dff$, resp. $\mathrm{cf}_{\tindex}$; in a neighborhood of $\partial \dmf$ (i.e., $\cf \cap \dmf$), we have coordinates
\begin{equation} \label{three-cone position variable near dmf}
\hat{x}^{\tindex} \coloneq \frac{x^{\tindex}}{\bdf}, \, \bdf, \, y_{\tindex}, \, y^{\tindex},
\end{equation}
where $\hat{x}^{\tindex}$, resp. $\bdf$ are local boundary defining functions for $\dmf$, resp. $\mathrm{cf}_{\tindex}$. In particular, there is a natural fibration for $\mathrm{cf}_{\tindex}$ with base $\mathcal{C}_{\tindex}$ and fibers which can be realized as follows: away from $\dff$ we can make a change of variable from $( \hat{x}^{\tindex}, y^{\tindex
} )$ to $\hat{z}^{\tindex}$, where
\begin{equation} \label{fiber of cf coordinates}
\hat{x}^{\tindex} = \frac{1}{|\hat{z}^{\tindex}|}, \, y_{j}^{\tindex} = \frac{\hat{z}^{\tindex}_{j}}{|\hat{z}^{\tindex}|}, \quad  j =1,...,n^{\tindex},
\end{equation}
and these coordinates are valid everywhere except when $\hat{z}^{\tindex} = 0$ (in which case $\hat{x}^{\tindex} = \infty$, so coordinates (\ref{three-cone position variable near dmf}) no longer make sense). At the point $\{ \hat{z}^{\alpha} = 0 \}$, we will glue another sphere $\mathcal{C}^{\tindex} \cong \mathbb{S}^{n^{\alpha}-1}$, which is defined by $\bbdf = 1/\hat{x}^{\tindex} = 0$ and $x^{\tindex} = 0$ in coordinates (\ref{three-cone position variables near dff}). Thus, if $\hat{X}^{\tindex}$ denotes the canonically chosen, radially compactiied vector space with interior coordinates $\hat{z}^{\tindex}$, then the fibers of $\mathrm{cf}_{\tindex}$ can be realized simply as the blow up $[ \hat{X}^{\tindex} ; \{ 0 \} ]$.
\par

\begin{figure}
\centering
\includegraphics{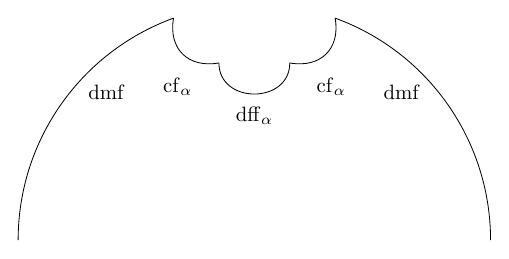}
\caption{The space $[ \overline{\mathbb{R}^{2}} ; \mathcal{C}_{\tindex} ; \mf \cap \ff ]$ near the lift of $\mathcal{C}_{\tindex}$ when $\mathcal{C}_{\tindex}$ is a point. The three-cone structure connects the microlocal structures at $\dmf$ and $\dff$ through $\cf$. }
\end{figure}

However, as it is the case for the three-body structure, it would be useful to introduce somewhat degenerate coordinates systems, which are valid over the interior of $\Xd$. 
\begin{itemize}
\item Near $\mathrm{cf}_{\tindex} \cap \dff$, such coordinates are either $( x_{\tindex}, y_{\tindex}, x^{\tindex}, y^{\tindex} )$ or $( x_{\tindex}, y_{\tindex}, \bbdf, y^{\tindex} )$, where we always have $x_{\tindex} = \bbdf (x^{\tindex})^2$. The former coordinates are valid up to $\dff$ but not $\cf$. We must assume that $ x^{\tindex} \neq 0 $. Moreover, $\dff^{\circ} = \{ x_{\tindex} = 0 \}$. The latter coordinates are valid up to $\cf$ but not $\dff$. We must assume that $\bbdf \neq 0$. Moreover, $\mathrm{cf}_{\tindex}^{\circ} = \{ x_{\tindex} = 0 \}$, though $x_{\alpha}$ is only a quadratic defining function for $\cf^{\circ}$ in this case. 
\item Near $\mathrm{cf}_{\tindex} \cap \dmf$, such coordinates are either $( x_{\tindex}, y_{\tindex}, \hat{x}^{\tindex}, y^{\tindex} )$ or $(x_{\tindex}, y_{\tindex}, \bdf, y^{\tindex})$, where we always have $x_{\tindex} = \bdf^{2} \hat{x}^{\tindex}$. The former coordinates are valid up to  $\cf$ but not $\dmf$. We must assume that $\hat{x}^{\tindex} \neq 0$. Moreover, $\mathrm{cf}_{\tindex} = \{ x_{\tindex} = 0 \}$ to the second order. The latter coordinates (which will not be of major importance here) are valid up to $\dmf$ but not $\cf$. We must assume that $\bdf \neq 0$. Moreover, $\dmf = \{ \bdf = 0 \}$. 
\end{itemize}
We remark that often times, it is even better to further change $(x_{\tindex}, y_{\tindex})$ into the Euclidean coordinates $z_{\tindex}$. We shall change between these coordinates rather freely in the calculations below as we see fit, without further elaborations.
\par

The three-cone vector fields (i.e., the vector fields which we will microlocalize on $\Xd$) will now be defined by modifying the three-body vector fields at $\cf$ in such a way that we have a cone structure at each fiber $[ \hat{X}^{\tindex} ; \{  0 \} ]$ of $\mathrm{cf}_{\tindex}$. As in the discussion of \S \ref{subsection the cone calculus}, we shall understand $[ \hat{X}^{\alpha} ; \{ 0 \} ]$ as a cone, where the lift of $\partial \hat{X}^{\tindex}$ is the end of the cone and the lift of $\{ 0 \}$ is the tip of the cone. We will denote the lift of $\partial \hat{X}^{\alpha}$ by $\mathcal{C}^{\tindex}_{\infty}$ and the lift of $\{ 0 \}$ by $\mathcal{C}^{\tindex}_{0}$. Both faces are of course diffeomorphic to $\mathbb{S}^{n^{\tindex}-1}$. Then the space three-cone vector fields will be defined as those smooth vector fields $V$ on $X$ such that
\begin{itemize}
\item $V$ is an element of $\mathcal{V}_{\mathrm{sc}}(X_{\tindex}) \oplus_{\mathcal{C}^{\infty}(\Xd)} \mathcal{V}_{\mathrm{b}}(X^{\tindex})$ in a small neighborhood of $\dff$;
\item $V$ is an element of $\mathcal{V}_{\mathrm{sc}}(X_{\tindex}) \oplus_{\mathcal{C}^{\infty}(\Xd)} \mathcal{V}_{\mathrm{co}}( \econ ) $ in a small neighborhood of $\cf$;
\item $V$ is an element of $\mathcal{V}_{\mathrm{sc}}(\overline{\mathbb{R}^{n}})$ away from $\dff \cup \cf$.
\end{itemize}
These conditions can be captured globally by defining
\begin{equation*}
\mathcal{V}_{\mathrm{3co}}( \Xd ) \coloneq \mathcal{V}_{\mathrm{sc}}(X_{\tindex}) \oplus_{\mathcal{C}^{\infty}(\Xd)} x_{\cf}^{-1} \mathcal{V}_{\mathrm{sc}}(X^{\tindex}),
\end{equation*}
where $x_{\cf} \in \mathcal{C}^{\infty}( \Xd )$ is a global defining function for $\cf$. From this definition, it is straightforward to check that $\mathcal{V}_{\mathrm{3co}}(\Xd)$ is also a Lie algebra. In particular, elements of $\mathcal{V}_{\mathrm{3co}}(\Xd)$ are reduced naturally to the three-body vector fields away from $\Xd \backslash \cf$. \par

Explicitly in local coordinates, in a small neighborhood of $\cf \cap \dff$ where $x_{\cf} \simeq x^{\tindex}$, elements of $\mathcal{V}_{\mathrm{3co}}( \Xd )$ can be written in coordinates $(x_{\tindex}, y_{\tindex}, x^{\tindex}, y^{\tindex})$ as the $\mathcal{C}^{\infty}( \Xd )$ span of
\begin{equation} \label{definition of three-cone vector fields 1}
x_{\tindex}^2 \partial_{x_{\tindex}}, \,  x_{\tindex} \partial_{y_{\tindex}}, \, x^{\tindex} \partial_{x^{\tindex}}, \, \partial_{y^{\tindex}}.
\end{equation}
Alternatively, they can be written in coordinates $( x_{\tindex}, y_{\tindex}, \bbdf, y^{\tindex} )$ as the $\mathcal{C}^{\infty}( \Xd )$ span of
\begin{equation} \label{definition of three-cone vector fields 2}
x_{\tindex}^2 \partial_{x_{\tindex}}, \, x_{\tindex} \partial_{y_{\tindex}}, \, \bbdf \partial_{\bbdf}, \,\partial_{y^{\tindex}}.
\end{equation}
On the other hand, in a small neighborhood of $\mathrm{cf}_{\tindex} \cap \dmf$ where $x_{\cf} \simeq \rho_{\ff}$, elements of $\mathcal{V}_{\mathrm{3co}}( \Xd )$ can be written in coordinates $( x_{\tindex}, y_{\tindex}, \hat{x}^{\tindex}, y^{\tindex} )$ as the $\mathcal{C}^{\infty}( \Xd )$ span of
\begin{equation} \label{definition of three-cone vector fields 3}
x_{\tindex}^2 \partial_{x_{\tindex}}, \, x_{\tindex} \partial_{y_{\tindex}}, \, (\hat{x}^{\tindex})^2 \partial_{\hat{x}^{\tindex}}, \, \hat{x}^{\tindex} \partial_{y^{\tindex}}.
\end{equation}
Alternatively, they can be written in coordinates $ ( x_{\tindex}, y_{\tindex}, \bdf, y^{\tindex} ) $ as the $\mathcal{C}^{\infty}( \Xd )$ span of
\begin{equation} \label{definition of three-cone vector fields 4}
x_{\tindex}^2 \partial_{x_{\tindex}}, \, x_{\tindex} \partial_{y_{\tindex}}, \, \hat{x}^{\tindex} \bdf \partial_{\bdf}, \,  \hat{x}^{\tindex} \partial_{y^{\tindex}}.
\end{equation} \par

By Melrose's program, it follows that we can find a tangent bundle ${^{\mathrm{3co}}T} \Xd$ such that ${^{\mathrm{3co}}T_{\mathbb{R}^{n}}\Xd} = T\mathbb{R}^{n}$ and $\mathcal{C}^{\infty}( \Xd ; {^{\mathrm{3co}}T \Xd} ) = \mathcal{V}_{\mathrm{3co}}(\Xd)$. Moreover, by construction, we must have ${^\mathrm{3co}T}_{U} \Xd = {^{\mathrm{3sc}}T}_{U} X$ whenever $U \subset \Xd$ is away from $\cf$. The three-cone cotangent bundle $^{\mathrm{3co}}T^{\ast}\Xd$ is then defined by the dual bundle of ${^\mathrm{3co}}T \Xd$. This bundle satisfies ${^\mathrm{3co}T^{\ast}_{\mathbb{R}^{n}}}\Xd = T^{\ast}\mathbb{R}^{n}$ and ${^{\mathrm{3co}}T^{\ast}_{U}} \Xd = {^{\mathrm{3sc}}T^{\ast}}_{U} X$, whenever $U$ is as above. \par

Coordinates on ${^\mathrm{3co}T^{\ast}\Xd}$ are defined by local frames of covector fields, which, when written with respect to (\ref{definition of three-cone vector fields 1}) and (\ref{definition of three-cone vector fields 2}), are given by
\begin{equation} \label{covector fields 1}
\tau_{\tindex}^{\mathrm{3co,b}} \frac{dx_{\tindex}}{x_{\tindex}^2} + \mu_{\tindex}^{\mathrm{3co,b}} \cdot \frac{dy_{\tindex}}{x_{\tindex}} + \tau^{\tindex}_{\mathrm{3co,b}} \frac{dx^{\tindex}}{x^{\tindex}} + \mu^{\tindex}_{\mathrm{3co,b}} \cdot dy^{\tindex},
\end{equation}
and 
\begin{equation} \label{covector fields 2}
\tau_{\tindex}^{\mathrm{3co,co, b}} \frac{dx_{\tindex}}{x_{\tindex}^2} + \mu_{\tindex}^{\mathrm{3co,co, b}} \cdot \frac{dy_{\tindex}}{x_{\tindex}} + \tau^{\tindex}_{\mathrm{3co,co, b}} \frac{d \bbdf }{ \bbdf } + \mu^{\tindex}_{\mathrm{3co,co, b}} \cdot dy^{\tindex}
\end{equation}
respectively. It follows that (\ref{covector fields 1}) gives rise to the coordinates
\begin{equation} \label{three cone phase varibales 1}
\bbdf,  y_{\tindex}, x^{\tindex},  y^{\tindex}, \tau_{\tindex}^{\mathrm{3co,b}}, \mu_{\tindex}^{\mathrm{3co,b}},   \tau^{\tindex}_{\mathrm{3co,b}},  \mu^{\tindex}_{\mathrm{3co,b}}, 
\end{equation}
while (\ref{covector fields 2}) gives rise to the coordinates
\begin{equation} \label{three cone phase variables 2}
\bbdf, y_{\tindex},  x^{\tindex}, y^{\tindex}, \tau_{\tindex}^{\mathrm{3co,co,b}}, \mu_{\tindex}^{\mathrm{3co,co, b}},  \tau^{\tindex}_{\mathrm{3co,co,b}},  \mu^{\tindex}_{\mathrm{3co,co,b}} .
\end{equation}
Both (\ref{three cone phase varibales 1}) and (\ref{three cone phase variables 2}) are valid in a small neighborhood of ${ ^{\mathrm{3co}}T^{\ast}_{\cf \cap \dff} \Xd }$. The change of variable formulae between them are given by
\begin{equation*}
\tau^{\mathrm{3co,co,b}}_{\tindex} = \tau_{\tindex}^{\mathrm{3co,b}} + \frac{\bbdf (x^{\tindex})^2}{2} \tau^{\tindex}_{\mathrm{3co,b}}, \ \mu_{\tindex}^{\mathrm{3co,co,b}} = \mu_{\tindex}^{\mathrm{3co,b}}, \ \tau^{\tindex}_{\mathrm{3co,co,b}} = - \frac{\tau^{\tindex}_{\mathrm{3co,b}}}{2}, \ \mu^{\tindex}_{\mathrm{3co,co,b}} = \mu^{\tindex}_{\mathrm{3co,b}}.
\end{equation*} \par
On the other hand, covector fields can be written with respect to (\ref{definition of three-cone vector fields 3}) and (\ref{definition of three-cone vector fields 4}) as 
\begin{equation} \label{covector fields 3}
\tau_{\tindex}^{\mathrm{3co,co, sc}} \frac{dx_{\tindex}}{x_{\tindex}^2} + \mu_{\tindex}^{\mathrm{3co,co, sc}} \cdot \frac{dy_{\tindex}}{x_{\tindex}} + \tau^{\tindex}_{\mathrm{3co,co,sc}} \frac{d \hat{x}^{\tindex} }{ ( \hat{x}^{\tindex} )^2 } + \mu^{\tindex}_{\mathrm{3co,co,sc}} \cdot \frac{dy^{\tindex}}{\hat{x}^{\tindex}},
\end{equation}
and 
\begin{equation} \label{covector fields 4}
\tau_{\tindex}^{\te} \frac{dx_{\tindex}}{x_{\tindex}^2} + \mu_{\tindex}^{ \te } \cdot \frac{dy_{\tindex}}{x_{\tindex}} - \tau^{\tindex}_{\te} \frac{d\bdf}{\hat{x}^{\tindex} \bdf} + \mu^{\tindex}_{\te} \cdot \frac{dy^{\tindex}}{\hat{x}^{\tindex}}.
\end{equation}
Here, (\ref{covector fields 3}) gives rise to the coordinates 
\begin{equation} \label{three cone phase variables 3}
\bdf,  y_{\tindex},   \hat{x}^{\tindex},   y^{\tindex},  \tau_{\tindex}^{\mathrm{3co,sc}},   \mu^{\mathrm{3co,sc}}_{\tindex},  \tau^{\tindex}_{\mathrm{3co,sc}},  \mu^{\tindex}_{\mathrm{3co,sc}}, 
\end{equation}
while (\ref{covector fields 4}) gives rise to the coordinates
\begin{equation} \label{three cone phase variables 4}
\bdf,  y_{\tindex},   \hat{x}^{\tindex},  y^{\tindex},  \tau_{\tindex}^{\mathrm{3co,sf}},  \mu_{\tindex}^{\mathrm{3co,sf}}, \tau^{\tindex}_{\mathrm{3co,sf}}, \mu^{\tindex}_{\mathrm{3co,sf}}. 
\end{equation} 
Note that we have chosen in (\ref{covector fields 4}) the variable dual to $d \bdf / \hat{x}^{\tindex} \bdf$ with a negative sign because this gives rise to the convenient scaling property (\ref{revisted notation 3sc res 1}) below. Both (\ref{three cone phase variables 3}) and (\ref{three cone phase variables 4}) are valid in a small neighborhood of $^{\mathrm{3co}}T^{\ast}_{\cf \cap \dmf} \Xd$. The change of variables formulae between them are given by
\begin{equation*}
\tau_{\tindex}^{\mathrm{3co,sc}} = \tau^{\te}_{\tindex} - \frac{\bdf^2}{2} \tau^{\tindex}_{\te} , \ \mu_{\tindex}^{\mathrm{3co,sc}} = \mu_{\tindex}^{\te}, \ \tau^{\tindex}_{\mathrm{3co,sc}} =  \frac{ \tau^{\tindex}_{\te} }{2}, \ \mu^{\tindex}_{\mathrm{3co,sc}} = \mu^{\tindex}_{ \te }.
\end{equation*}  \par 
It is convenient to work with canonical coordinates. To this end, notice that in either position coordinates $(x_{\tindex}, y_{\tindex}, x^{\tindex}, y^{\tindex})$ and $(x_{\tindex}, y_{\tindex}, \bbdf, y^{\tindex})$, we can always write
\begin{equation*}
\tau_{\tindex}^{\bullet} \frac{dx_{\tindex}}{x_{\tindex}^{2}} + \mu_{\tindex}^{\bullet} \cdot dy^{\tindex} = \zeta_{\tindex}^{\bullet} \cdot dz_{\tindex},
\end{equation*}
where $\bullet$ could be any one of the subscripts `$\text{3co,b}$', `$\text{3co,co,b}$', `$\text{3co,co,sc}$' or `$\text{3co,sf}$'. Hence we could replace every $( x_{\tindex}, y_{\tindex}, \tau_{\tindex}^{\bullet}, \mu_{\tindex}^{\bullet} )$ with $( z_{\tindex}, \zeta_{\tindex}^{\bullet})$, with the latter coordinates being the canonical Euclidean coordinates in the free variables. Moreover, as in the b-case, we could also write $\bbdf = e^{-\hat{t}_{\ff}}$, $x^{\tindex} = e^{-t^{\tindex}}$. Combined with the above discussion, this implies that we can find canonical coordinates
\begin{equation*}
z_{\tindex},  t^{\tindex}, y^{\tindex}, \zeta_{\tindex}^{\mathrm{3co,b}}, - \tau^{\tindex}_{\mathrm{3co,b}},  \mu^{\tindex}_{\mathrm{3co,b}}
\end{equation*}
in place of (\ref{three cone phase varibales 1}), i.e., $-\tau^{\tindex}_{\mathrm{3sc,b}}$ is dual to $t^{\tindex}$. Likewise, the coordinates
\begin{equation*}
z_{\tindex}, \hat{t}_{\tindex},  y^{\tindex},  \zeta_{\tindex}^{\mathrm{3co,co, b}},  - \tau^{\tindex}_{\mathrm{3co,co,b}},  \mu^{\tindex}_{\mathrm{3co,co, b}}
\end{equation*}
are also canonical, i.e., $-\tau^{\tindex}_{\mathrm{3co,b}}$ is dual to $\hat{t}_{\tindex}$, and can be used in place of (\ref{three cone phase variables 2}). Furthermore, by making the variable change from $(\hat{x}^{\tindex}, y^{\tindex})$ to $\hat{z}^{\tindex}$, and by defining $\zeta_{\mathrm{3co,co,sc}}^{\tindex}$ with respect to 
\begin{equation*}
\tau^{\tindex}_{\mathrm{3co,co, sc}} \frac{ d \hat{x}^{\tindex} }{ ( \hat{x}^{\tindex} )^2 } + \mu_{\mathrm{3co,co, sc}}^{\tindex} \cdot \frac{dy^{\tindex}}{\hat{x}^{\tindex}} = \zeta^{\tindex}_{\mathrm{3co,co, sc}} \cdot d \hat{z}^{\tindex},
\end{equation*}
we can replace (\ref{three cone phase variables 3}) with the canonical coordinates $(z_{\tindex}, \hat{z}^{\tindex}, \zeta_{\tindex}^{\mathrm{3co,co,sc}}, \zeta^{\tindex}_{\mathrm{3co,co,sc}})$ as well.
\par

Now, the notations introduced above, while systematic, are quite burdening. So it is very natural to have them relaxed at a few places. First, in view of our motivation for defining the three-cone structure, which is to understand second microlocalization for the three-body structure, one should be tempted to compare (\ref{covector fields 1}) with (\ref{three-body covector fields}) and (\ref{covector fields 4}) with (\ref{three-body covector fields from the other direction}). These coordinates are related by
\begin{gather} 
\tau_{\tindex} = \tau_{\tindex}^{\mathrm{3co,b}}, \ \mu_{\tindex} = \mu_{\tindex}^{\mathrm{3co,b}}, \ \frac{\tau^{\tindex}}{x^{\tindex}} =  \tau^{\tindex}_{\mathrm{3co,b}}, \  \frac{\mu^{\tindex}}{x^{\tindex}} = \mu^{\tindex}_{\mathrm{3co,b}},  \label{revisted notation 3sc b} \\
\tau_{\tindex}^{\mathrm{3sc,sf}} = \tau_{\tindex}^{\mathrm{3co,sf}}, \ \mu_{\tindex} = \mu_{\tindex}^{\mathrm{3co,sf}}, \ \frac{\tau^{\tindex}}{\bdf} = \tau^{\tindex}_{\mathrm{3co,sf}}, \ \frac{\mu^{\tindex}}{\bdf} = \mu^{\tindex}_{\mathrm{3co,sf}} \label{revisted notation 3sc res 1}
\end{gather}
respectively. Hence, by drawing analogy from Vasy's sc,b-notations (i.e., writing $\tau_{\mathrm{b}} = \tau/x$, $\mu_{\mathrm{b}} = \mu/x$ in the two-body case), we are encouraged to suppress the subscripts `$\mathrm{3co}$' which appears in (\ref{revisted notation 3sc b}), and simply write
\begin{equation} \label{revised notation 3sc b 2}
\tau^{\tindex}_{\mathrm{b}} \coloneq \frac{\tau^{\tindex}}{x^{\tindex}} = \tau^{\tindex}_{\mathrm{3co,b}} , \ \mu^{\tindex}_{\mathrm{b}} \coloneq \frac{\mu^{\tindex}}{x^{\tindex}} = \mu^{\tindex}_{\mathrm{3co,b}}.
\end{equation}
Likewise, we will also write, for brevity
\begin{equation} \label{revised notation 3sc res 2}
\tau^{\tindex}_{\mathrm{sf}} \coloneq \frac{\tau^{\tindex}}{\bdf} = \tau^{\tindex}_{\mathrm{3co,sf}}, \ \mu^{\tindex}_{\mathrm{sf}} \coloneq \frac{\mu^{\tindex}}{\bdf} = \mu^{\tindex}_{\mathrm{3co,sf}},
\end{equation}
as well as 
\begin{equation}
\tau^{\tindex}_{\mathrm{co, b}} \coloneq \tau^{\tindex}_{\mathrm{3co,co,b}}, \ \mu^{\tindex}_{\mathrm{co,b}} \coloneq \mu^{\tindex}_{\mathrm{3co,co,sc}}, \quad \tau^{\tindex}_{\mathrm{co,sc}} \coloneq \tau^{\tindex}_{\mathrm{3co,co,sc}}, \ \mu^{\tindex}_{\mathrm{co,sc}} \coloneq \mu^{\tindex}_{\mathrm{3co,co,sc}} .
\end{equation}
Moreover, since the variables change between (\ref{covector fields 2}) and (\ref{covector fields 3}) is given by
\begin{equation*}
\tau_{\tindex}^{\mathrm{3co,co,sc}} = \tau_{\tindex}^{\mathrm{3co,co,b}}, \ \mu_{\tindex}^{\mathrm{3co,co,sc}} = \mu_{\tindex}^{\mathrm{3co,co,b}}, \ \tau^{\tindex}_{\mathrm{co,sc}} = - \frac{ \tau^{\tindex}_{\mathrm{co,b}} }{ \hat{\rho}_{\tindex} }, \ \mu^{\tindex}_{\mathrm{co,sc}} = \frac{\mu^{\tindex}_{\mathrm{co,b}}}{\hat{\rho}_{\tindex}},
\end{equation*}
it also makes sense to just write
\begin{equation} \label{revised notation 3co}
\tau_{\tindex}^{\mathrm{3co}} \coloneq \tau_{\tindex}^{\mathrm{3co,co,sc}} = \tau_{\tindex}^{\mathrm{3co,co,b}}, \  \mu_{\tindex}^{\mathrm{3co}} \coloneq \mu_{\tindex}^{\mathrm{3co,co,sc}} = \mu_{\tindex}^{\mathrm{3co,co,b}}.
\end{equation}
Additionally, we set
\begin{equation} \label{revised notation 3co 1.5}
\tau_{\tindex}^{\mathrm{sf}} \coloneq  \tau_{\tindex}^{\mathrm{3sc,sf}} = \tau_{\tindex}^{\mathrm{3co,sf}}.
\end{equation}
In terms of Euclidean coordinates, we will write
\begin{equation} \label{revised notation Euclidean}
\zeta_{\tindex} = \zeta_{\tindex}^{\mathrm{3co,b}}, \ \zeta_{\tindex}^{\mathrm{3co}} \coloneq \zeta_{\tindex}^{\mathrm{3co,co,sc}} = \zeta_{\tindex}^{\mathrm{3co,b}}, \ \zeta^{\tindex}_{\mathrm{co,sc}} \coloneq \zeta^{\tindex}_{\mathrm{3co,sc}}.
\end{equation}
Henceforth, we will always be adopting notations (\ref{revisted notation 3sc b})--(\ref{revised notation Euclidean}) without further comments. \par

Finally, let $\overline{^{\mathrm{3co}}T^{\ast}} \Xd$ denote the compactification of $^{\mathrm{3co}}T^{\ast} \Xd$ radially in the fibers. We shall now give a characterization of the conormal three-cone symbol spaces in terms of estimates. To this end, let $m,r,l,b \in \mathbb{R}$ measure decay respectively at ${^{\mathrm{3co}}S^{\ast}}\Xd$, $\overline{^{\mathrm{3co}}T^{\ast} }_{\dmf} \Xd$, $\overline{^{\mathrm{3co}}T^{\ast} }_{\dff} \Xd$ and $\overline{^{\mathrm{3co}}T^{\ast} }_{\cf} \Xd$. In the same orders, let $\rho_{\infty}, \rho_{\dmf}, \rho_{\dff}, \rho_{\cf} \in \mathcal{C}^{\infty}( \overline{^{\mathrm{3co}}T^{\ast}} \Xd )$ be defining functions for these boundary faces, Moreover, let $U_{H}$ be some small neighborhood of $\overline{^{\mathrm{3co}}T^{\ast}}_{H}\Xd$ for $H \in \{ \dmf, \dff, \cf \}$. Then
\begin{equation*}
S^{m,r,l,b}( \overline{^{\mathrm{3co}}T^{\ast} }\Xd )
\end{equation*}
will consist of those functions $a \in \mathcal{C}^{\infty}( T^{\ast} \mathbb{R}^{n} )$ such that:
\begin{itemize}
\item Near every point of $U_{\cf} \backslash U_{\dmf}$, we either have
\begin{equation}
\label{symbol estimates constant order three cone 1}
| \partial_{z_{\tindex}}^{\beta_{\tindex}} \partial_{t^{\tindex}}^{j} \partial_{y^{\tindex}}^{\beta^{\tindex}} \partial_{ \tscblz }^{\gamma_{\tindex}} \partial_{ \tscbut }^{k} \partial_{ \tscbum }^{\gamma^{\tindex}} a  | \leq C_{\beta_{\tindex} j \beta^{\tindex} \gamma_{\tindex} k \gamma^{\tindex}}  \rho_{\infty}^{-m + |\gamma_{\tindex}| + k + |\gamma^{\tindex}|}  \rho_{\cf}^{-b+ 2 |\beta_{\tindex}|} \rho_{\dff}^{-l + |\beta_{\tindex}|}
\end{equation}
in some coordinates of the form $( z_{\tindex}, t^{\tindex}, y^{\tindex}, \tscblz, \tscbut, \tscbum)$; or equivalently
\begin{equation} \label{symbol estimates constant order three cone 2}
| \partial_{z_{\tindex}}^{\beta_{\tindex}} \partial_{ \hat{t}_{\tindex} }^{j} \partial_{y^{\tindex}}^{\beta^{\tindex}} \partial_{ \tcoblz }^{\gamma_{\tindex}} \partial_{ \tcobut }^{k} \partial_{ \tcobum }^{\gamma^{\tindex}} a | \leq C_{\beta_{\tindex} j \beta^{\tindex} \gamma_{\tindex} k \gamma^{\tindex}}  \rho_{\infty}^{-m + |\gamma_{\tindex}| + k + |\gamma^{\tindex}|}   \rho_{\cf}^{-b+ 2 |\beta_{\tindex}|} \rho_{\dff}^{-l + |\beta_{\tindex}|}
\end{equation}
in some coordinates of the form $(z_{\tindex}, \hat{t}_{\tindex}, y^{\tindex}, \tcoblz, \tcobut, \tcobum)$.
\item In $U_{\cf} \backslash U_{\dff}$, we have
\begin{equation}
\label{symbol estimates constant order three cone 3}
| \partial_{z_{\tindex}}^{\beta_{\tindex}} \partial_{ \hat{z}^{\tindex} }^{\beta^{\tindex}} \partial_{ \tcosclz }^{\gamma_{\tindex}} \partial_{ \tcoscuz }^{\gamma^{\tindex}} a  | \leq C_{\beta_{\tindex} \beta^{\tindex} \gamma_{\tindex} \gamma^{\tindex}} \rho_{\infty}^{-m + |\gamma_{\tindex}| + |\gamma^{\tindex}|} \rho_{\dmf}^{-r + |\beta_{\tindex}| + |\beta^{\tindex}|} \rho_{\cf}^{-b+ 2 |\beta_{\tindex}|}
\end{equation}
in the coordinates $( z_{\tindex}, \hat{z}^{\tindex}, \tcosclz, \tcoscuz)$.
\item In any region that is away from $U_{\cf}$, we will require that $a$ restricts to an element of $S^{m,r,l}( \overline{ ^{\mathrm{3sc}}T^{\ast} } X )$. 
\end{itemize}
This concludes the characterization of $S^{m,r,l,b}( \overline{^{\mathrm{3co}}T^{\ast}} \Xd )$.

\begin{remark}
\label{product type description remark three-cone case}
Notice that if $a \in S^{m,r,l,b}( \overline{ ^{\mathrm{3co}}T^{\ast} }X )$, then locally near $\overline{^{\mathrm{3co}}T^{\ast}}_{\dff}\Xd$ and $\overline{^{\mathrm{3co}}T^{\ast}}_{\cf}\Xd$, we can understand $a$ as having a `product-like' symbolic structures in the base variables, as well as a `joint' symbolic structure in the fiber variables. \par

Here, the product structure near $\dff$ is that of a product structure for $\overline{\mathbb{R}^{n_{\tindex}}} \oplus X^{\tindex}$. Indeed, estimates (\ref{symbol estimates constant order three cone 1}) can also be written as
\begin{equation} 
\label{symbol estimates constant order three cone 4}
| \partial_{z_{\tindex}}^{\beta_{\tindex}} \partial_{t^{\tindex}}^{j} \partial_{y^{\tindex}}^{\beta^{\tindex}} \partial_{ \tscblz }^{\gamma_{\tindex}} \partial_{ \tscbut }^{k} \partial_{ \tscbum }^{\gamma^{\tindex}} a  | \leq C_{\beta_{\tindex} j \beta^{\tindex} \gamma_{\tindex} k \gamma^{\tindex}} \langle z_{\tindex} \rangle^{l - |  \beta_{\tindex}| }  e^{t^{\tindex} ( b - 2l ) } \langle \zeta_{\tindex}, \utaub, \umub \rangle^{ m - | \gamma_{\tindex} | - k - |\gamma^{\tindex}| }.
\end{equation}
Thus, by combining (\ref{symbol estimates constant order three cone 4}) with the symbol estimates in the three-body calculus, we see that locally near $\overline{^{\mathrm{3co}}T^{\ast}}_{\dff}\Xd$, $a$ can be thought of as having a scattering structure of order $l$ in the base variables along $\overline{\mathbb{R}^{n_{\tindex}}}$, and a b-structure of order $b-2l$ in the base variables along $X^{\tindex}$. \par

Likewise, the product structure near $\cf$ is that of a product structure for $\overline{\mathbb{R}^{n_{\tindex}}} \times [ \hat{X}^{\tindex} ; \{ 0 \} ]$. To see this, note that estimates (\ref{symbol estimates constant order three cone 5}) can also be written as, respectively
\begin{align}
\label{symbol estimates constant order three cone 5}
\begin{split}
& | \partial_{z_{\tindex}}^{\beta_{\tindex}} \partial_{ \hat{t}_{\tindex} }^{j} \partial_{y^{\tindex}}^{\beta^{\tindex}} \partial_{ \tcoblz }^{\gamma_{\tindex}} \partial_{ \tcobut }^{k} \partial_{ \tcobum }^{\gamma^{\tindex}} a | \\
& \qquad \leq C_{\beta_{\tindex} j \beta^{\tindex} \gamma_{\tindex} k \gamma^{\tindex}} \langle z_{\tindex} \rangle^{b/2 - |\beta_{\tindex}|}  e^{\hat{t}_{\tindex} ( l - b/2 ) } \langle \zeta_{\tindex}^{\mathrm{3co}}, \utaucob, \umucob \rangle^{m - |\gamma_{\tindex}| - k - \gamma^{\tindex} },
\end{split}
\end{align}
while  (\ref{symbol estimates constant order three cone 5}) can be written as
\begin{equation}
\label{symbol estimates constant order three cone 6}
| \partial_{z_{\tindex}}^{\beta_{\tindex}} \partial_{ \hat{z}^{\tindex} }^{\beta^{\tindex}} \partial_{ \tcosclz }^{\gamma_{\tindex}} \partial_{ \tcoscuz }^{\gamma^{\tindex}} a  | \leq C_{\beta_{\tindex} \beta^{\tindex} \gamma_{\tindex} \gamma^{\tindex}} \langle z_{\tindex} \rangle^{b/2 - |\beta_{\tindex}|} \langle \hat{z}^{\tindex} \rangle^{r - b/2 - |\beta^{\tindex}|} \langle \zeta_{\tindex}^{\mathrm{3co}}, \zeta^{\tindex}_{\mathrm{co,sc}} \rangle^{m - |\gamma_{\tindex}| - |\gamma^{\tindex}|}.
\end{equation}
Thus together, (\ref{symbol estimates constant order three cone 5}), (\ref{symbol estimates constant order three cone 6}) imply that locally near $\overline{^{\mathrm{3co}}T^{\ast}}_{\cf}\Xd$, $a$ can be thought of as having a scattering structure of order $b/2$ in the base variables along $X_{\tindex}$, as well as a cone structure of orders $l - b/2$ at $\mathcal{C}^{\tindex}_{0}$, $r - b/2$ at $\mathcal{C}^{\tindex}_{\infty}$, in the base variables along $[ \hat{X}^{\tindex} ; \{ 0 \} ]$.
\end{remark}
\subsection{Definition of the three-cone algebra: an overview}
\label{construction of three-cone overview} \label{overview of three-cone bundle subsection}
Let $m,r,l \in \mathbb{R}$. Starting from this subsection, we will define the space of conormal three-cone operators
\begin{equation}
\label{conormal three-cone operators OG}
\Psi^{m,r, l, b}_{ \mathrm{3coc} }( \Xd ).
\end{equation}
It will be instructive to first give an overview of this construction. \par 
As usual, operators $A \in \Psi^{m,r,l,b}_{\mathrm{3coc}}( \Xd )$ are first and foremost continuous linear maps 
\begin{equation*}
A : \mathcal{S}( \mathbb{R}^{n} ) \rightarrow \mathcal{S}(\mathbb{R}^{n}),
\end{equation*}
so in particular, each $A$ has a distributional Schwartz kernel, which we will realize as a section of the right three-cone density bundle ${ ^{\mathrm{3co}}\Omega_{R}} \Xd \coloneq \pi_{R}^{\ast} { ^{\mathrm{3co}} \Omega \Xd }$. Here, $ { ^{\mathrm{3co}} \Omega \Xd}$ is the bundle of densities arising naturally from $\mathcal{C}^{\infty}(\Xd ; { ^{\mathrm{3co}}T^{\ast} \Xd })$, while $\pi_{R}: X^{2} \rightarrow X$ is the natural projection to the right factor of $X^2$. However, this dependency will usually be omitted upon choosing a global trivialization of ${ ^{\mathrm{3co}} \Omega_{R} \Xd }$, which can be done by fixing a strictly positive three-cone density $\nu_{\mathrm{3co}}$. \par

In the construction below, we will often work with coordinates systems in which $\nu_{\mathrm{3co}}$ is assumed to take the form of a standard, Euclidean density. This can be done without loss of generality, since:
\begin{itemize}
\item In the coordinates $(z_{\tindex}, t^{\tindex}, y^{\tindex})$ near $\dff \cap \cf$, we have $\nu_{\mathrm{3co}} \simeq |dz_{\tindex} dt^{\tindex} dy^{\tindex}|$;
\item In the coordinates $(z_{\tindex}, \hat{t}_{\tindex}, y^{\tindex})$ near $\dff \cap \cf$, we have $\nu_{\mathrm{3co}} \simeq | dz_{\tindex} d\hat{t}_{\tindex} dy^{\tindex} |$; 
\item In the coordinates $(z_{\tindex}, \hat{z}^{\tindex})$ near $\dmf \cap \cf$, we have $\nu_{\mathrm{3co}} \simeq | dz_{\tindex} d\hat{z}^{\tindex} |$. 
\end{itemize}  \par

As in the constructions of the conormal b- and cone operators, the three-cone operators are also global in nature. Thus, quantization alone will not be sufficient in capturing the properties of three-cone operators. Our construction will therefore be done locally everywhere (even though we are essentially working on $\mathbb{R}^{n}$) in such a way that the local pieces can be patched together in the end. 
\par

We also remark that another way to construct the three-cone operators is through a double space construction, which has the advantage that the resulting construction is automatically invariant. However, following Vasy's second microlocalization for the scattering calculus \cite{AndrasBook}, we will nevertheless choose to present the aforementioned local construction. In particular, this local construction will have the advantage that it can be more easily converted to the second microlocalized settings when variable orders, even at $\dff$ or $\cf$, are present.  \par 
We will work from the perspective of Schwartz kernels. The standard idea is to consider partitions of unity. In fact, it will be enough if we just choose a sum
\begin{equation*}
1 = \tilde{\psi}_{0} +\tilde{\psi}_{\dff} + \tilde{\psi}_{\cf}
\end{equation*}
of $\mathcal{C}^{\infty}( \Xd )$ functions, such that $\tilde{\psi}_{\dff}$, resp. $ \tilde{\psi}_{\cf}$ are supported in a small neighborhood of $\dff$, resp. $\cf$, while $\tilde{\psi}_{0}$ is supported away from $\dff \cup \cf$. Notice that if $\tilde{\psi}_{\dff}$ is identically $1$ at $\dff$, then $\tilde{\psi}_{\cf}$ cannot be identically $1$ at $\cf$, and vice versa.  \par

Now, let $\psi_{0}, \psi_{\dff}, \psi_{\cf} \in \mathcal{C}^{\infty}( \Xd )$ be identically $1$ on the supports of $\tilde{\psi}_{0}$, $\tilde{\psi}_{\dff}$ and $\tilde{\psi}_{\cf}$ respectively. Then we can write
\begin{align} 
\label{standard operator decomposition using partition of unity}
\begin{split}
 A = {} & {\psi}_{0} A {\psi}_{0} \tilde{\psi}_{0} + {\psi}_{\cf} A {\psi}_{\cf} \tilde{\psi}_{\cf} + {\psi}_{\dff} A {\psi}_{\dff} \tilde{\psi}_{\dff} \\
& + (1 - {\psi}_{0} ) A \tilde{\psi}_{0} + (1 - {\psi}_{\cf} ) A \tilde{\psi}_{\cf} + (1 - {\psi}_{\dff} ) A \tilde{\psi}_{\dff}.
\end{split}
\end{align}
Thus, it would be enough if we can make sense of each term in (\ref{standard operator decomposition using partition of unity}). In particular, we can now assume that $\psi_{\dff}$, resp. $\psi_{\cf}$ are identically $1$ at $\dff$, resp. $\cf$. \par 
Since three-cone vector fields restrict to scattering vector fields at $\dmf$, it is immediately clear that we should define
\begin{equation} \label{three-cone scattering part of the definition}
\psi_{0} A \psi_{0} \in \Psi_{\mathrm{sc}}^{m , r}(  \overline{\mathbb{R}^{n}} ).
\end{equation}
On the other hand, we will require the last three terms in (\ref{standard operator decomposition using partition of unity}) to have rapidly decreasing, smooth kernels. In \S \S \ref{subsection construction of dff}--\ref{subsection construction of cf} below, we will construct the terms $\psi_{\dff} A \psi_{\dff}$, $\psi_{\cf} A \psi_{\cf}$ via explicit constructions. 

\subsection{Construction of the three-cone operators near $\dff$}
\label{subsection construction of dff}
Let us first restrict our attention to a small neighborhood of $\dff$, i.e., we will construct the operator $\psi_{\dff} A \psi_{\dff}$, where $\psi_{\dff} \in \mathcal{C}^{\infty}(\Xd)$ is a cut-off function at $\dff$. In fact, consider a further (finite) partition $\{ \tilde{\psi}_{j} \}_{j=1}^{J} \subset \mathcal{C}^{\infty}(X)$ such that $\sum_{j=1}^{J} \tilde{\psi}_{j} = 1$ on $\supp \psi_{\dff}$. For each $j = 1, ..., J$, let also ${\psi}_{j} \in \mathcal{C}^{\infty}(X)$ be chosen such that ${\psi}_{j}$ identically $1$ on $\supp \tilde{\psi}_{j}$. Then we have
\begin{equation} \label{local decomposition near dff}
\psi_{\dff} A \psi_{\dff} = \sum_{j = 1}^{J} \psi_{\dff} {\psi}_{j} A \psi_{\dff} {\psi}_{j}  \tilde{\psi}_{j} + \sum_{j,k = 1}^{J} ( 1 - {\psi}_{j} ) \tilde{\psi}_{k}  \psi_{\dff} A \varphi_{j} \psi_{\dff}.
\end{equation}
Thus, in order to define $\psi_{\dff} A \psi_{\dff}$, one only needs to define each term in (\ref{local decomposition near dff}). \par

More generally, what we need is to define all operators of the forms $\psi A \psi$, $\phi A \psi$, where $\psi, \phi \in \mathcal{C}^{\infty} ( \Xd )$ are cut-off functions such that:
\begin{itemize}
\item $\supp \psi$, $\supp \phi$ are as small as needed (i.e., in such a way that the local coordinates used in the definitions of $\psi A \psi$, $\phi A \psi$ are valid),
\item $\supp \psi$, $\supp \phi$ are contained in a small neighborhood of $\dff$,
\item $\supp \psi \cap \supp \phi = \emptyset$.
\end{itemize} \par

Consider the operators $\psi A \psi$ with $\psi$ as above. It is only worth considering in details the cases where $\supp \psi $ intersects $\mathrm{cf}_{\tindex} \cap \dff$, e.g., $\psi$ is a cut-off function at a point of $\cf \cap \dff$. Otherwise, $\supp \psi$ must intersect the boundary of $\Xd$ only at $\dff^{\circ}$ (which we remark is also the same as $\mathrm{ff}_{\tindex}^{\circ}$), where $A$ will be required to have a three-body structure, i.e., we will define
\begin{equation}
\label{the interior of a second-microlocalised operator that is three-body scattering}
\psi A \psi  \coloneq B_{\psi,\mathrm{3sc}}
\end{equation}
for some $B_{\psi,\mathrm{3sc}} \in \Psi^{m,-\infty, l }_{\mathrm{3scc}}( [ \overline{\mathbb{R}^{n}} ; \mathcal{C}_{\tindex} ] )$. 

 \begin{remark}
Clearly, one can also combine (\ref{the interior of a second-microlocalised operator that is three-body scattering}) with (\ref{three-cone scattering part of the definition}) to see that if $\psi \in \mathcal{C}^{\infty}(X)$ is supported away from $\cf$, then we must have $\psi A \psi \in \Psi_{\mathrm{3scc}}^{m,r,l}( [ \overline{\mathbb{R}^{n}} ; \mathcal{C}_{\tindex} ] )$. 
 \end{remark}

Thus, suppose that $\supp \psi$ indeed intersects $\mathrm{cf}_{\tindex} \cap \dff$. Then due to the non-local nature of the three-cone operators as mentioned before, the operator $\psi A \psi$ will only partially be determined by a quantization. \par

Consider first the `near-diagonal' part (see Remark \ref{stretched diagonal remark} below) of $\psi A \psi$, which will be given by a quantization. Then by working spatially in the coordinates $( z_{\tindex}, x^{\tindex}, y^{\tindex} )$ such that $\dff$ is given by $|z_{\tindex}| \rightarrow \infty$, this quantization is initially given by
\begingroup
\begin{align} 
\begin{split} \label{on-diagonal part of the second microlocalized operator}
\widetilde{B}_{\psi, \mathrm{3co,b}}  \coloneq {} & \frac{1}{(2\pi)^{n}} \int_{ \mathbb{R}^{n} }  e^{ i (z_{\tindex} - z_{\tindex}') \cdot \tscblz + i \frac{ (x^{\tindex} - ( x^{\tindex} )') }{x^{\tindex}} \tscbut + i ( y^{\tindex} - (y^{\tindex})' ) \cdot \tscbum }  \tilde{\varphi} \Big( \frac{ x^{\tindex} - (x^{\tindex})'  }{x^{\tindex}} \Big) \varphi( |y^{\tindex} - ( y^{\tindex} )'| ) \\
& \times  \tilde{b}_{\psi, \mathrm{3co,b} } ( z_{\tindex}, x^{\tindex}, y^{\tindex}, \tscblz , \tscbut , \tscbum  ) d\tscblz d\tscbut d\tscbum \Big| dz_{\tindex}' \frac{d(x^{\tindex})'  }{ ( x^{\tindex} )' } d(y^{\tindex})'  \Big|,
\end{split}
\end{align}
\endgroup
where $\tilde{b}_{\psi, \mathrm{3co,b}} \in S^{m, -\infty ,r,l}(  \overline{^{\mathrm{3co}}T^{\ast}}\Xd )$, $\varphi \in \mathcal{C}^{\infty}_{c}( \mathbb{R} )$ is supported in a small neighborhood of $0$, and $\tilde{\varphi}(s) \coloneq \varphi( 1 - e^{s} )$. However, for $x^{\tindex} = e^{-t^{\tindex}}$, one could also switch to the coordinates $( z_{\tindex}, t^{\tindex}, y^{\tindex} )$, and consider equivalently the standard quantization
\begingroup
\begin{align} 
\label{the 3scb quantization written in terms of t}
\begin{split}
B_{\psi, \mathrm{3co,b}} \coloneq {} & \frac{1}{(2\pi)^{n}} \int_{ \mathbb{R}^{n} }  e^{ i (z_{\tindex} - z_{\tindex}') \cdot \tscblz - i ( t^{\tindex} - (t^{\tindex})' ) \tscbut + i ( y^{\tindex} - (y^{\tindex})' ) \cdot \tscbum } \varphi(t-t') \varphi( |y^{\tindex} - ( y^{\tindex} )'| ) \\
& \times  b_{\psi, \mathrm{3co,b} } ( z_{\tindex}, t^{\tindex}, y^{\tindex}, \tscblz , \tscbut , \tscbum  )  d\zeta_{\tindex}^{\mathrm{3sc,b}} d\tau^{\tindex}_{\mathrm{3sc,b}} d\mu^{\tindex}_{\mathrm{3sc,b}} | dz_{\tindex}' d(t^{\tindex})' d(y^{\tindex})'  |,
\end{split}
\end{align}
\endgroup
where $b_{\psi, \mathrm{3co,b}}$ is the representation of $\tilde{b}_{\psi, \mathrm{3co,b}}$ in the new coordinates. \par

 On the other hand, we will also need to consider an `off-diagonal' part of $\psi A \psi $, which is defined by a smooth kernel $R_{\psi, \mathrm{3co,b}}$ satisfying the estimates
\begin{align} 
\label{3scb part near diagonal off diagonal part}
\begin{split}
& |  \partial_{z_{\tindex}}^{\beta_{\tindex}}  \partial_{z_{\tindex}'}^{\beta_{\tindex}'}  \partial_{t^{\tindex}}^{j} \partial_{y^{\tindex}}^{\gamma_{\tindex}} \partial_{ (t^{\tindex})' }^{j'}  \partial_{ (y^{\tindex})' }^{\gamma_{\tindex}'} R_{\psi, \mathrm{3co,b}} | \\
& \qquad  \leq C_{\beta_{\tindex} \beta_{\tindex}' j \gamma_{\tindex} j' \gamma_{\tindex}'NML} x_{\dff}^{-l} x_{\cf}^{-b}  \langle z_{\tindex} \rangle^{-|\beta_{\tindex}|} \langle z_{\tindex}' \rangle^{-|\beta_{\tindex}'|} \langle z_{\tindex} - z_{\tindex}' \rangle^{-N} \langle y^{\tindex} - (y^{\tindex})' \rangle^{-L} e^{-M| t^{\tindex} - (t^{\tindex})' |}
\end{split}
\end{align}
for all $N, M, L \in \mathbb{R}$. Here $x_{\dff}, x_{\cf} \in \mathcal{C}^{\infty}(\Xd)$ are some boundary defining functions for $\dff$ and $\cf$ respectively. We also let $x_{\dff}'$, $x_{\cf}'$ denote the representations of $x_{\dff}$, $x_{\cf}$ in the right variables. Thus, $R_{\psi, \mathrm{3co,b}}$ is residual only in the regularity sense, with an additional, super-exponential decay as $|t^{\tindex} - (t^{\tindex})'| \rightarrow \infty$. Notice that $x_{\cf} \simeq e^{-t^{\tindex}}$ on the support of $R_{\psi, \mathrm{3co,b}}$. Thus, the latter exponential decay condition is equivalent to saying that $x_{\cf}/x_{\cf}'$ vanishes to infinite orders at it tends to $0$ or $\infty$.
\par

Together, we then require that
\begin{equation} \label{on-diagonal part of the operator focusing on dff}
\psi A \psi \coloneq   B_{\psi, \mathrm{3co,b}} + R_{\psi, \mathrm{3co,b}}, 
\end{equation}
where $B_{\psi, \mathrm{3co,b}}$, $R_{\psi, \mathrm{3co,b}}$ are constructed as above.   \par 

\begin{remark}
\label{stretched diagonal remark}
In the above discussion, the notions of `near-diagonal' and `off-diagonal' are used only heuristically: the `diagonal' in question should really be understood as the lifted diagonal in some suitable stretch double space, in the sense of Melrose.
\end{remark}

We now move onto the specifications of $\phi A \psi$, where $\phi, \psi \in \mathcal{C}^{\infty}(\Xd)$ are as specified as in the beginning of this subsection. Then there are three cases to be considered (compare these with cases (1)--(3) in \S \ref{b-calculus subsection}):
\begin{enumerate}
    \item $ \supp \phi $ and $\supp \psi $ are both disjoint from $\mathrm{cf}_{\tindex}$,
    \item $\supp \phi $ is disjoint from $\mathrm{cf}_{\tindex}$ but $\supp \psi $ is not, or vice versa, 
    \item both $\supp \phi $ and $\supp \psi $ intersect $\mathrm{cf}_{\tindex}$, but are still disjoint.
\end{enumerate}
Let $K_{\phi, \psi, \mathrm{3co,b}}$ be the kernel of $\phi A \psi$ in each case. \par 
We shall require $K_{\phi, \psi, \mathrm{3co,b}}$ to always be smooth. In case (1), the coordinates $(z_{\tindex}, z^{\tindex}, z_{\tindex}', (z^{\tindex})')$ must be valid on $\supp K_{\phi, \psi, \mathrm{3co,b}}$, and we will require additionally that
\begin{equation} 
\label{3sc b kernel estimate 1}
| \partial_{z_{\tindex}}^{\beta_{\tindex}} \partial_{z_{\tindex}'}^{\beta_{\tindex}'}  \partial_{z^{\tindex}}^{\beta^{\tindex}} \partial_{(z^{\tindex})'}^{ (\beta^{\tindex})' } K_{\phi, \psi, \mathrm{3co,b}} | \leq C_{ \beta_{\tindex} \beta_{\tindex}' \beta^{\tindex} ( \beta^{\tindex} )' N } x_{\dff}^{-l} \langle z_{\tindex} \rangle^{-|\beta_{\tindex}|}  \langle z_{\tindex}' \rangle^{-|\beta_{\tindex}'|}  \langle z_{\tindex} - z_{\tindex}' \rangle^{-N}
\end{equation}
for all $N \in \mathbb{R}$. In case (2), suppose that $\supp \phi$ is disjoint from $\cf$ such that the coordinates $( z_{\tindex}, z^{\tindex} , z_{\tindex}, (t^{\tindex})', (y^{\tindex})' )$ are valid on $\supp K_{\phi, \psi, \mathrm{3co,b}}$. Then we will require that 
\begin{align} \label{3sc b kernel estimate 2}
\begin{split}
& |\partial_{z_{\tindex}}^{\beta_{\tindex}}  \partial_{z_{\tindex}'}^{\beta_{\tindex}'}   \partial_{ z^{\tindex} }^{\beta^{\tindex}} \partial_{(t^{\tindex})'}^{j'} \partial_{(y^{\tindex})'}^{(\beta^{\tindex})'}  K_{\phi, \psi, \mathrm{3co, b}} | \\
& \qquad \leq C_{\beta_{\tindex} \beta_{\tindex}' j \gamma^{\tindex} (\beta^{\tindex})' NM } x_{\dff}^{-l} ( x_{\cf}' )^{M} \langle z_{\tindex} \rangle^{-|\beta_{\tindex}|}  \langle z_{\tindex}' \rangle^{-|\beta_{\tindex}'|}  \langle z_{\tindex} - z_{\tindex}' \rangle^{-N} 
\end{split}
\end{align}
for all $N, M \in \mathbb{R}$. A similar (and indeed obvious) requirement must be met if instead $\supp \psi$ is disjoint from $\cf$. \par

 Finally, in case (3), the coordinates $( z_{\tindex}, t^{\tindex}, y^{\tindex}, z_{\tindex}', t^{\tindex}, (y^{\tindex})' )$ are valid on $\supp K_{\phi, \psi, \mathrm{3co,b}}$, and we will require that $K_{\phi, \psi, \mathrm{3co,b}}$ be vanishing to infinite orders super-exponentially as $|t^{\tindex} - (t^{\tindex})'| \rightarrow \infty$, or equivalently as $x_{\cf}/x_{\cf}'$ tends to $0$ or $\infty$, though there is no decay in the diagonal along the fibers of $\dff$ (since this does not make sense by disjointedness, i.e., the coordinates $y^{\tindex}$ amd $(y^{\tindex})'$ are unrelated). In particular, the estimates
\begin{align} 
\label{3sc b kernel estiamte 3}
\begin{split}
& | \partial_{z_{\tindex}}^{\beta_{\tindex}}  \partial_{z_{\tindex}'}^{\beta_{\tindex}'}  \partial_{t^{\tindex}}^{j} \partial_{y^{\tindex}}^{\gamma^{\tindex}} \partial_{(t^{\tindex})'}^{j'}  \partial_{ (y^{\tindex})' }^{(\gamma^{\tindex})'} K_{\phi, \psi, \mathrm{3co,b}} |\\
& \qquad \leq  C_{\beta_{\tindex} \beta_{\tindex}' j \gamma_{\tindex} j' \gamma_{\tindex}'NM} x_{\dff}^{-l} x_{\cf}^{-b} \langle z_{\tindex} \rangle^{-|\beta_{\tindex}|} \langle z_{\tindex}' \rangle^{-|\beta_{\tindex}'|}  \langle z_{\tindex} - z_{\tindex}' \rangle^{-N} e^{-M|t^{\tindex} - (t^{\tindex})'|}
\end{split}
\end{align} 
for all $N, M \in \mathbb{R}$.
\begin{remark}
\label{remark left and right variable decays are the same}
As in the conormal b- or cone cases, one could replace $x_{\cf}^{-b}$ by $(x_{\cf}')^{-b}$ (or vice versa) in estimates (\ref{3scb part near diagonal off diagonal part}) or (\ref{3sc b kernel estiamte 3}) due to the presence of infinite orders of decay as $x_{\cf}/x_{\cf}$ tends to $0$ or infinity. Moreover, in the present case, one could also replace $x_{\dff}^{-l}$ by $(x_{\dff}')^{-l}$ (or vice versa) in estimates (\ref{3scb part near diagonal off diagonal part}) and (\ref{3sc b kernel estimate 1})--(\ref{3sc b kernel estiamte 3}) as well. To see this, simply observe that $x_{\dff}^{-l} \simeq |z_{\tindex}|^{l} x_{\cf}^{2l}$. Then either $x_{\cf}^{2l} \simeq 1$ (if $K_{\phi, \psi, \mathrm{3co,b}}$ decays to infinite order at $\cf$ in the left variables) or $x_{\cf}^{2l} \simeq (x_{\cf}')^{2l}$ (due to the argument at the beginning of this remark). As for the $\langle z_{\tindex} \rangle^{l}$ term, we can use Peetre's inequality $\langle z_{\tindex} - z_{\tindex}' \rangle^{-N} \leq C_{N} ( \langle z_{\tindex}' \rangle / \langle z_{\tindex} \rangle )^{-N}$ and the fact $\langle z_{\tindex} \rangle^{l} = \langle z_{\tindex}' \rangle^{l} ( \langle z_{\tindex} \rangle / \langle z_{\tindex}' \rangle )^{l}$ to see that $\langle z_{\tindex} \rangle^{l}$ can be replaced by $\langle z_{\tindex}' \rangle^{l}$, and vice versa.
\end{remark}
\begin{remark}
In particular, since $\langle z_{\tindex} \rangle^{-1} \simeq x_{\dff} x_{\cf}^{2}$, estimates (\ref{3scb part near diagonal off diagonal part}) and (\ref{3sc b kernel estimate 1})--(\ref{3sc b kernel estiamte 3}) contain the following information: the kernels $R_{\psi, \mathrm{3co,b}}$ and $K_{\phi, \psi, \mathrm{3co,b}}$ enjoy a gain or order $1$, resp. $2$ at $\dff$, resp. $\cf$ upon differentiating in the $z_{\tindex}$ variable. The same statement is also true if the left variables are replaced with the right variables. In fact, due to Remark \ref{remark left and right variable decays are the same} above, decay in the left and right variables are interchangeable. 
\end{remark}
\subsection{Construction of the three-cone operators near $\cf$}
\label{subsection construction of cf}

Next, we will carry out a similar construction in a neighborhood of $\mathrm{cf}_{\tindex}$. In fact, to some extend the situation near $\mathrm{cf}_{\tindex} \cap \dff$ has already been considered in \S \ref{subsection construction of dff} above. Thus, our construction below must also satisfy appropriate compatibility conditions in this region. This will be investigated in \S \ref{compatibility of three-cone operators subsection} below. \par

As before, we will define all operators of the forms $\psi A \psi$, $\phi A \psi$, where $\psi, \phi \in \mathcal{C}^{\infty}(X)$ are cut-off functions such that:
\begin{itemize}
\item $\supp \psi$, $\supp \phi$ are as small as needed (i.e., in such a way that the local coordinates used in the definitions of $\psi A \psi$, $\phi A \psi$ are valid),
\item $\supp \psi$, $\supp \phi$ are contained in a small neighborhood of $\cf$,
\item $\supp \psi \cap \supp \phi = \emptyset$.
\end{itemize} \par

First, let us again consider the behavior of $A$ in a neighborhood of $\cf \cap \dff$. Therefore, we will proceed to define $\psi A \psi$, where $\supp \psi$ intersects $\cf \cap \dff$, e.g., when $\psi$ is a cut-off function at some point of $\cf \cap \dff$. Initially, it will be natural to work in the coordinates $( z_{\tindex}, \bbdf, y^{\tindex} )$ such that $\mathrm{cf}_{\tindex}$ is given by $|z_{\tindex}| \rightarrow \infty$. Then we will require that the `near-diagonal' part of the $\psi A \psi$ be given by the quantization
\begingroup
\begin{align} \label{complicated 3co,b quantization}
\begin{split}
\widetilde{B}_{\psi, \mathrm{3co,co,b}} \coloneq {} & \frac{1}{(2\pi)^{n}} \int_{ \mathbb{R}^{n} }  e^{ i (z_{\tindex} - z_{\tindex}') \cdot \tcoblz + i \frac{ \bbdf - \bbdf' }{ \bbdf } \tcobut + i ( y^{\tindex} - (y^{\tindex})' ) \cdot \tcobum }  \tilde{\varphi}\Big( \frac{ \bbdf - \bbdf '  }{ \bbdf } \Big) \varphi( |y^{\tindex} - ( y^{\tindex} )'| )  \\
& \times \tilde{b}_{\psi, \mathrm{3co,co,b} } ( z_{\tindex}, \bbdf, y^{\tindex}, \tcoblz , \tcobut, \tcobum  )  d \tcoblz d\tcobut d\tcobum \Big| dz_{\tindex}' \frac{ d \bbdf' }{ \bbdf' } d(y^{\tindex})'  \Big|,
\end{split}
\end{align}
\endgroup
where $\tilde{b}_{\psi, \mathrm{3co,co, b}} \in S^{m,-\infty, l , b}( \overline{^{\mathrm{3co}}T^{\ast}} \Xd )$, $\varphi \in \mathcal{C}^{\infty}_{c}( \mathbb{R} )$ is supported in a small neighborhood of $0$, and $\tilde{\varphi}(s) \coloneq \varphi( 1 - e^{s} )$. However, for $\bbdf = e^{- \hat{t}_{\tindex} }$, one could make a further coordinates change to $( z_{\tindex}, \hat{t}_{\tindex},  y^{\tindex})$. Thus equivalently, and equivalently consider the standard quantization
\begingroup
\begin{align} 
\label{the 3cob quantization written in terms of t}
\begin{split}
B_{\psi, \mathrm{3co,co, b}} \coloneq {} & \frac{1}{(2\pi)^{n}} \int_{ \mathbb{R}^{n} }  e^{ i (z_{\tindex} - z_{\tindex}') \cdot \tcoblz - i ( \hat{t}_{\tindex} - \hat{t}_{\tindex}' ) \tcobut + i ( y^{\tindex} - (y^{\tindex})' ) \cdot \tcobum }  \varphi( \hat{t}_{\tindex} - \hat{t}_{\tindex}' ) \varphi( |y^{\tindex} - ( y^{\tindex} )'| )  \\
& \times b_{\psi, \mathrm{3co, co, b} } ( z_{\tindex}, \hat{t}_{\tindex}, y^{\tindex}, \tcoblz , \tcobut , \tcobum  )   d\tcoblz d\tcobut d\tcobum  |dz_{\tindex}' d\hat{t}_{\tindex}' d(y^{\tindex})' |,
\end{split}
\end{align}
\endgroup
where $b_{\psi, \mathrm{3co,co,b}}$ is the representation of $\tilde{b}_{\psi, \mathrm{3co,co,b}}$ in the new coordinates. \par

On the other hand, we will require that the `off-diagonal' part of $\psi A \psi$ be a smooth kernel $R_{\psi, \mathrm{3co,co,b}}$ which satisfies the decay estimates
\begin{align} \label{estimate on R 3b b}
\begin{split}
& | \partial_{z_{\tindex}}^{\beta_{\tindex}} \partial_{z_{\tindex}'}^{\beta_{\tindex}'} \partial_{\hat{t}_{\tindex} }^{j} \partial_{\hat{t}_{\tindex}'}^{j'} \partial_{y^{\tindex}}^{\gamma_{\tindex}} \partial_{ (y^{\tindex})' }^{\gamma_{\tindex}'} R_{\psi, \mathrm{3co,co, b}} |   \\
& \qquad \leq C_{\beta_{\tindex} \beta_{\tindex}' j j' \gamma_{\tindex} \gamma_{\tindex}'NML}  x_{\cf}^{-b} x_{\dff}^{-l}  \langle z_{\tindex} \rangle^{-|\beta_{\tindex}|}  \langle z_{\tindex}' \rangle^{-|\beta_{\tindex}'|}  \langle z_{\tindex} - z_{\tindex}' \rangle^{-N}  \langle y^{\tindex} - ( y^{\tindex} )' \rangle^{-L} e^{-M| \hat{t}_{\tindex} - \hat{t}_{\tindex}' |}.
\end{split}
\end{align}
for all $N, M, L \in \mathbb{R}$. The infinite orders of super-exponential decay as $|\hat{t}_{\tindex} - \hat{t}_{\tindex}| \rightarrow \infty$ can again be replaced by the infinite orders of polynomial decay as $x_{\dff}/x_{\dff}' \rightarrow 0$ or infinity. \par

Together, we then require that
\begin{equation}
\label{on-diagonal part focuing on cf}
\psi A \psi \coloneq B_{\psi, \mathrm{3co,co,b}} +  R_{\psi, \mathrm{3co,co, b}},
\end{equation} 
where $B_{\psi, \mathrm{3co, co, b}}$, $R_{\psi, \mathrm{3co,co, b}}$ are constructed as above. \par 

\begin{remark}
The content of Remark \ref{stretched diagonal remark} applies to the above discussion as well.
\end{remark}

Next, we will consider the behavior of $A$ in a neighborhood of $\mathrm{cf}_{\tindex} \cap \dmf$, i.e., the definition of $\psi A \psi$ where $\supp \psi$ intersects $\cf \cap \dmf$. In fact, in this case it would be enough if $\psi$ is supported away from $\cf \cap \dff$. It would be convenient to first switch to the spatial coordinates $(z_{\tindex}, \hat{x}^{\tindex}, y^{\tindex})$, which we recall can then be further changed into the variables $( z_{\tindex}, \hat{z}^{\tindex} )$, where the relationships are determined by (\ref{fiber of cf coordinates}). Then we will define
\begin{equation}
\label{3cosc quantization -1}
\psi A \psi \coloneq B_{\psi, \mathrm{3co, co, sc}},
\end{equation}
where $B_{\psi, \mathrm{3co,co, sc}}$ is determined by the quantization
\begin{equation} 
\label{3cosc quantization}
\frac{1}{(2\pi)^{n}} \int_{\mathbb{R}^{n}} e^{ i( z_{\tindex} - z_{\tindex}' ) \cdot \tcosclz + i ( 
\hat{z}^{\tindex} - (\hat{z}^{\tindex})' ) \cdot \tcoscuz }   b_{\psi, \mathrm{3co,co, sc}}( z_{\tindex}, \hat{z}^{\tindex}, \tcosclz , \tcoscuz ) d \tcosclz d \tcoscuz | dz_{\tindex}' d ( \hat{z}^{\tindex} )' |
\end{equation}
for some $b_{\psi, \mathrm{3co, co, sc}} \in S^{m,r,-\infty, b}( \overline{ ^{\tc}T^{\ast} } \Xd  )$. \par 
Finally, we will need to consider the cross terms. Let $\phi, \psi \in \mathcal{C}^{\infty}(\Xd)$ be specified as in the beginning of this subsection. Then under the most thorough classifications, there are now six cases to be considered (compare these with cases (1)--(6) in \S \ref{subsection the cone calculus}):
\begin{enumerate}
    \item $ \supp \phi $ and $ \supp \psi $ are both disjoint from $\dff \cup \dmf$,
    \item $\supp \phi$ intersects $\dff$ and $\supp \psi$ intersects $\dmf$, or vice versa,
    \item $\supp \phi$ is disjoint from $ \dff \cup \dmf$ but $\supp \psi$ intersects $\dff$, or vice versa,
    \item $\supp \phi$ is disjoint from $ \dff \cup \dmf $ but $\supp \psi$ intersects $\dmf$, or vice versa,
    \item both $\supp \phi$ and $\supp \psi$ intersect $\dmf$, but are still disjoint,
    \item both $\supp \phi$ and $\supp \psi$ intersect $\dff$, but are still disjoint.
\end{enumerate}
Let $K_{\phi, \psi, \mathrm{3co,co}}$ be the kernel of $\phi A \psi$ in each case. \par 
As always, we shall first require $K_{\phi, \psi, \mathrm{3co,co}}$ to always be smooth. In case (1), the coordinates $(z_{\tindex}, \hat{z}^{\tindex}, z_{\tindex}',  (\hat{z}^{\tindex})' )$ are valid on $\supp K_{\phi, \psi, \mathrm{3co,co}}$, and we will require that
\begin{equation} \label{3con cross term estimaets 1}
| \partial_{z_{\tindex}}^{\beta_{\tindex}}  \partial_{z_{\tindex}'}^{\beta_{\tindex}'}  \partial_{ \hat{z}^{\tindex}}^{\beta^{\tindex}} \partial_{( \hat{z}^{\tindex})'}^{ (\beta^{\tindex})' } K_{\phi, \psi, \mathrm{3co,co}} | \leq C_{ \beta_{\tindex} \beta_{\tindex}' \beta^{\tindex} ( \beta^{\tindex} )' N }  x_{\cf}^{-b} \langle z_{\tindex} \rangle^{-|\beta_{\tindex}|} \langle z_{\tindex}' \rangle^{-|\beta_{\tindex}'|}  \langle z_{\tindex} - z_{\tindex}' \rangle^{-N}
\end{equation}
for all $N \in \mathbb{R}$. In case (2), suppose for definiteness that $\supp \phi$ intersects $\dff$ and $\supp \psi$ intersects $\dmf$. Then the coordinates $( z_{\tindex}, \hat{t}_{\tindex}, y^{\tindex}, z_{\tindex}', (\hat{z}^{\tindex})' )$ are valid on $\supp K_{\phi, \psi, \mathrm{3co,co}}$, and we will require that
\begin{align} \label{3con cross term estimates 2}
\begin{split}
& |  \partial_{z_{\tindex}}^{\beta_{\tindex}}  \partial_{z_{\tindex}'}^{ \beta_{\tindex}' } \partial_{\hat{t}_{\tindex}}^{j} \partial_{y^{\tindex}}^{\gamma^{\tindex}} \partial_{ (\hat{z}^{\tindex})' }^{ (\beta^{\tindex})' } K_{\phi, \psi, \mathrm{3co,co}} | \\
& \qquad \leq C_{\beta_{\tindex} \beta_{\tindex}' j \gamma^{\tindex} ( \beta^{\tindex} )'} x_{\cf}^{-b} x_{\dff}^{M} (x_{\dmf}')^{L} \langle z_{\tindex} \rangle^{-|\beta_{\tindex}|}  \langle z_{\tindex}' \rangle^{-|\beta_{\tindex}'|}  \langle z_{\tindex} - z_{\tindex}' \rangle^{-N}\rangle^{-L}
\end{split}
\end{align}
for all $N, M, L \in \mathbb{R}$. Here, $x_{\dmf} \in \mathcal{C}^{\infty}(\Xd)$ is a defining function for $\dmf$, and we will also let $x_{\dmf}'$ denote the representation of $x_{\dmf}$ in the right variables. Notice that $x_{\dff} \simeq e^{-\hat{t}_{\tindex}}$, $x_{\dmf}' \simeq \langle \hat{z}^{\tindex} \rangle^{-1}$ on $\supp K_{\phi, \psi, \mathrm{3co,co}}$ as well. \par

 In case (3), and assume for definiteness that $\supp \psi$ intersects $\dff$. Then the required estimates for $K_{\phi, \psi, \mathrm{3co,co}}$ are exactly the same as (\ref{3con cross term estimates 2}), except that we now omit the term $(x_{\dmf}')^{L}$ and replace $x_{\dff}^{M}$ with $(x_{\dff}')^{M}$. In case (4), and assuming that $\supp \psi$ intersects $\dmf$, then we can work in the coordinates $( z_{\tindex}, \hat{z}^{\tindex}, z_{\tindex}', ( \hat{z}^{\tindex} )' )$ again, and require that
\begin{equation} \label{3con cross term estimates 4}
| \partial_{z_{\tindex}}^{\beta_{\tindex}}  \partial_{z_{\tindex}'}^{\beta_{\tindex}'}  \partial_{ \hat{z}^{\tindex}}^{\beta^{\tindex}} \partial_{( \hat{z}^{\tindex})'}^{ (\beta^{\tindex})' } K_{\phi, \psi, \mathrm{3co,co}} | \leq C_{ \beta_{\tindex} \beta_{\tindex}' \beta^{\tindex} ( \beta^{\tindex} )' N } x_{\cf}^{-b} (x_{\dmf}')^{M} \langle z_{\tindex} \rangle^{-|\beta_{\tindex}|} \langle z_{\tindex}' \rangle^{-|\beta_{\tindex}'|}  \langle z_{\tindex} - z_{\tindex}' \rangle^{-N} 
\end{equation}
for all $N, M \in \mathbb{R}$. In case (5), and working again in coordinates $( z_{\tindex}, \hat{z}^{\tindex}, z_{\tindex}', ( \hat{z}^{\tindex} )' )$, we will impose the even stronger conditions that
\begin{align} \label{3con cross term estimates 5}
\begin{split}
& | \partial_{z_{\tindex}}^{\beta_{\tindex}}  \partial_{z_{\tindex}'}^{\beta_{\tindex}'}  \partial_{ \hat{z}^{\tindex}}^{\beta^{\tindex}} \partial_{( \hat{z}^{\tindex})'}^{ (\beta^{\tindex})' } K_{\phi, \psi, \mathrm{3co,co}} | \\
& \qquad \leq C_{ \beta_{\tindex} \beta_{\tindex}' \beta^{\tindex} ( \beta^{\tindex} )' N } x_{\cf}^{-b} x_{\dmf}^{M}   (x_{\dmf}')^{M'} \langle z_{\tindex} \rangle^{-|\beta_{\tindex}|} \langle z_{\tindex}' \rangle^{-|\beta_{\tindex}'|}  \langle z_{\tindex} - z_{\tindex}' \rangle^{-N}
\end{split}
\end{align}
for all $N, M, M' \in \mathbb{R}$. Finally, in case $(6)$, the coordinates $( z_{\tindex}, \hat{t}_{\tindex}, y^{\tindex}, z_{\tindex}', \hat{t}_{\tindex}', (y^{\tindex})' )$ are valid, and we will require as in (\ref{3sc b kernel estiamte 3}) that 
\begin{align}  
\label{3con cross term estimates 6}
\begin{split}
& | \partial_{z_{\tindex}}^{\beta_{\tindex}}  \partial_{z_{\tindex}'}^{\beta_{\tindex}'}  \partial_{ \hat{t}_{\tindex} }^{j} \partial_{y^{\tindex}}^{\beta^{\tindex}} \partial_{ \hat{t}_{\tindex}' }^{j'}  \partial_{ (y^{\tindex})' }^{(\beta^{\tindex})'} K_{\phi, \psi, \mathrm{3co,co}} | \\
& \qquad \leq  C_{\beta_{\tindex} \beta_{\tindex}' j \beta^{\tindex} j' (\beta^{\tindex})'NM} x_{\cf}^{-b} x_{\dff}^{-l} \langle z_{\tindex} \rangle^{-|\beta_{\tindex}|} \langle z_{\tindex}' \rangle^{-|\beta_{\tindex}'|}  \langle z_{\tindex} - z_{\tindex}' \rangle^{-N} e^{-M| \hat{t}_{\tindex} - \hat{t}_{\tindex}' |}.
\end{split}
\end{align} 
for all $N, M \in \mathbb{R}$. We remark again that the infinite orders of super-exponential decay as $| \hat{t}_{\tindex} - \hat{t}_{\tindex} | \rightarrow \infty$ in (\ref{3con cross term estimates 6}) is equivalent to the infinite orders of polynomial decay of $x_{\dff}/x_{\dff}'$ as it tends to $0$ or infinity.

\subsection{Compatibility of the constructions} 
\label{compatibility of three-cone operators subsection}
We now discuss the compatibility properties of our constructions above. This is not required to show that (\ref{conormal three-cone operators OG}) is well-defined, but will be necessary if one wishes to construct specific elements of $\Psi_{\mathrm{3coc}}^{m,r,l,b}(\Xd)$ starting from local pieces. It is also important for the considerations of principal symbol and indicial operators, which will not be of interest until \S \ref{Section: Microlocal properties of the second microlocalized algebra} below. \par

What is worth showing is that our definitions are compatible with each others near the corners $\cf \cap \dmf$ and $\cf \cap \dff$. However, this is automatic near $\cf \cap \dmf$, i.e., the terms $\psi_{0} A \psi_0$, $\psi A \psi$, when $\psi \in \mathcal{C}^{\infty}(\Xd)$ is supported near $\cf \cap \dmf$, are naturally compatible in the usual sense (i.e., by the Kuranishi trick, so that their symbols are the same on the intersection between $\psi$ and $\psi_0$, modulo a term of lower order at both $\overline{^{\mathrm{3co}}T^{\ast}}_{\dmf} \Xd$ and fiber infinity), since they are both defined by quantizations.
\par

It remains to check compatibility at $\cf \cap \dff$. To this end, we will need to check compatibility between the terms in the formulae (\ref{on-diagonal part of the operator focusing on dff}) and (\ref{on-diagonal part focuing on cf}) for a fixed $\psi$, which is a cut-off function at some point of $\dff \cap \cf$, as well as the remainder terms $K_{\phi, \psi}$, where $\phi, \psi$ are supported in regions where compatibility is of relevance (i.e., case (3) in the discussion of \S \ref{subsection construction of dff}, against case (6) in the discussion of \S \ref{subsection construction of cf}) . \par

Overall, compatibility near $\cf \cap \dff$ will follow from a simple change of variables between $( z_{\tindex}, x^{\tindex}, y^{\tindex} )$ and $( z_{\tindex}, \bbdf, y^{\tindex} )$ (or equivalently between $ ( z_{\tindex}, t^{\tindex}, y^{\tindex} ) $ and $( z_{\tindex}, \hat{t}_{\tindex}, y^{\tindex} )$), which is determined by the relation
\begin{equation*}
\bbdf = \frac{1}{|z_{\tindex}| (x^{\tindex})^2 }.
\end{equation*}
It is standard to look at the near- and off-diagonal behaviors separately (where the `diagonal' in this context is the lift of the usual diagonal to some corner resolution of $X^2$). \par

Note again that compatibility between the terms $B_{\mathrm{3co,b}}$ and $B_{\mathrm{3co,co, b}}$, which are defined by quantizations, follows from the Kuranishi trick. Thus, it suffices to restrict our attention to the question of compatibility between $R_{\psi, \mathrm{3co,b}}$ and $R_{\psi, \mathrm{3co,co,b}}$.  \par

For this, suppose that $\varphi \in \mathcal{C}_{c}^{\infty}(\mathbb{R})$ is cut-off function at $1$, and write
\begin{equation} \label{definition of compatibility cut-off functions}
\varphi_{1} \coloneq \varphi\Big( \frac{|z_{\tindex}|}{|z_{\tindex}'|} \Big), \quad \varphi_{2} \coloneq ( 1 - \varphi ) \Big( \frac{|z_{\tindex}|}{|z_{\tindex}'|} \Big).
\end{equation} 
Then it is easy to verify that the defining estimates for $R_{\psi, \mathrm{3co,b}}$ and $R_{\psi, \mathrm{3co,co,b}}$ must agree on $\supp \varphi_1$. Indeed, one only needs to observe that, for instance
\begin{equation*}
\frac{\bbdf}{\bbdf'} = \frac{|z_{\tindex}|}{|z_{\tindex}'|} \frac{x^{\tindex}}{(x^{\tindex})'} \simeq \frac{x^{\tindex}}{(x^{\tindex})'}
\end{equation*}
on $\supp \varphi_1$. On the other hand, it is obvious that $R_{\mathrm{3co,b}}$, $R_{\mathrm{3co,co,b}}$ are rapidly decreasing on $\supp \varphi_{2}$. Thus, we can conclude that (\ref{on-diagonal part of the operator focusing on dff}) and (\ref{on-diagonal part focuing on cf}) are indeed compatible as anticipated. \par

In fact, the same argument applies verbatim for the relevant terms $K_{\phi, \psi, \mathrm{3co,b}}$, $K_{\phi, \psi, \mathrm{3co,co}}$. This concludes our discussion on compatibility.

\subsection{Characterization near $\dff$ through partial quantization}
\label{subsection partial quantization near dff}
In this subsection, we will show that if $\psi_{\dff} \in \mathcal{C}^{\infty}( \Xd )$ is a cut-off function at $\dff$, then for any $A \in \Psi^{m,r, l, b}_{\mathrm{3coc}}(X)$, $\psi_{\dff} A \psi_{\dff}$ can be characterized as the partial quantization in the free variables of some `operator-valued symbol', which is more precisely an element of $S^{m}( \overline{\mathbb{R}^{n_{\tindex}}} ; \Psi_{\mathrm{bc,lp}}^{m, b - 2l}(X^{\tindex}; {\mathbb{R}^{n_{\tindex}}} ))$. See Proposition \ref{membership of operator valued symbol near dff} below for the exact statements. Such a characterization will be useful when we discuss for instance composition in \S \ref{subsection composition and adjoint}. \par 

At first we will proceed slightly more generally. Thus, let $\phi, \psi \in \mathcal{C}^{\infty}( \Xd )$ be cut-off functions defined as in the beginning of \S \ref{subsection construction of dff}. Then depending on the locations of their supports, we will realize all operators of the forms 
\begin{equation}
\label{partial quantization subsection operators of interest}
\text{$B_{\psi, \mathrm{3sc}}$, $B_{\psi, \mathrm{3co,b}}$, $R_{\psi, \mathrm{3co,b}}$, $K_{\phi, \psi, \mathrm{3co,b}}$}
\end{equation}
as kernels on $\overline{\mathbb{R}^{n_{\tindex}}} \times X^{\tindex}$, which we view as sections of the bundle spanned by $| dz_{\tindex}' | \nu_{\mathrm{b}}' $. Here $\nu_{\mathrm{b}}'$ is a strictly positive b-density which is written in the right variables, and thus we have $\nu_{\mathrm{b}}' \simeq |d(t^{\tindex})' d(y^{\tindex})'|$ near $\partial X^{\tindex}$. 
\par

\begingroup
\allowdisplaybreaks
Now, for $Z_{\tindex} \coloneq z_{\tindex} - z_{\tindex}'$ (and hence $|dz_{\tindex}'| \nu_{\mathrm{b}}' = |dZ_{\tindex}| \nu_{\mathrm{b}}'$), suppose we define
\begin{align}
\hat{B}_{\psi, \mathrm{3sc}}  & \coloneq  \int_{\mathbb{R}^{n_{\tindex}}} e^{-iZ_{\tindex} \cdot \zeta_{\tindex}}  B_{\psi, \mathrm{3sc}} ( z_{\tindex}, z_{\tindex} - Z_{\tindex} ) dZ_{\tindex}  \nu_{\mathrm{b}}', \label{quantization term is a partial quantization B3sc}  \\
\hat{B}_{\psi, \mathrm{3co,b}} & \coloneq  \int_{\mathbb{R}^{n_{\tindex}}} e^{-iZ_{\tindex} \cdot \zeta_{\tindex}}  B_{\psi , \mathrm{3co,b}} ( z_{\tindex}, z_{\tindex} - Z_{\tindex} ) dZ_{\tindex} \nu_{\mathrm{b}}', \label{quantization term is a partial quantization B3cob} \\
\hat{R}_{\psi, \mathrm{3co,b}} & \coloneq \int_{\mathbb{R}^{n_{\tindex}}} e^{-iZ_{\tindex} \cdot \zeta_{\tindex}}  R_{\psi, \mathrm{3co,b}} ( z_{\tindex}, z_{\tindex} - Z_{\tindex} ) dZ_{\tindex}  \nu_{\mathrm{b}}', \label{remainder term is a partial quantization R3cob} \\
\hat{K}_{\phi, \psi, \mathrm{3co,b}} & \coloneq\int_{\mathbb{R}^{n_{\tindex}}} e^{-iZ_{\tindex} \cdot \zeta_{\tindex}}  K_{\phi, \psi, \mathrm{3co,b}}  ( z_{\tindex}, z_{\tindex} - Z_{\tindex} ) dZ_{\tindex} \nu_{\mathrm{b}}' \label{remainder term is a partial quantization K3cob}
\end{align}
as parametrized families of operators acting on $X^{\tindex}$, where in every case the space of parameters is ${\mathbb{R}^{n_{\tindex}}_{z_{\tindex}}} \times {\mathbb{R}^{n_{\tindex}}_{\zeta_{\tindex}}}$. Then it is immediate that we have
\begin{align}
B_{\psi, \mathrm{3sc}} & = \frac{1}{( 2 \pi )^{n_{\tindex}} } \int_{\mathbb{R}^{n_{\tindex}}} e^{ i ( z_{\tindex} - z_{\tindex}' ) \cdot \zeta_{\tindex} } \hat{B}_{\psi, \mathrm{3sc}} ( z_{\tindex}, \zeta_{\tindex} ) d \zeta_{\tindex} | d z_{\tindex}' |, \label{B 3sc as partial quantization} \\
B_{\psi, \mathrm{3co,b}} & = \frac{1}{( 2 \pi )^{n_{\tindex}} } \int_{\mathbb{R}^{n_{\tindex}}} e^{ i ( z_{\tindex} - z_{\tindex}' ) \cdot \zeta_{\tindex} } \hat{B}_{\psi,\mathrm{3co,b}} ( z_{\tindex}, \zeta_{\tindex} ) d \zeta_{\tindex} | d z_{\tindex}' |, \label{B 3co b as partial quantization} \\
R_{\psi, \mathrm{3co,b}} & = \frac{1}{( 2 \pi )^{n_{\tindex}} } \int_{\mathbb{R}^{n_{\tindex}}} e^{ i ( z_{\tindex} - z_{\tindex}' ) \cdot \zeta_{\tindex} } \hat{R}_{\psi, \mathrm{3co,b}} ( z_{\tindex}, \zeta_{\tindex} ) d \zeta_{\tindex} | d z_{\tindex}' |, \label{R 3co b as partial quantization} \\
{K}_{\phi, \psi, \mathrm{3co,b}} & = \frac{1}{( 2 \pi )^{n_{\tindex}} } \int_{\mathbb{R}^{n_{\tindex}}} e^{ i ( z_{\tindex} - z_{\tindex}' ) \cdot \zeta_{\tindex} } \hat{K}_{\phi, \psi, \mathrm{3co,b}} ( z_{\tindex}, \zeta_{\tindex} ) d \zeta_{\tindex} | d z_{\tindex}' |. \label{K 3co b as partial quantization}
\end{align}
In other words, the operators in (\ref{partial quantization subsection operators of interest}) can respectively be written as partial quantizations of (\ref{quantization term is a partial quantization B3sc})--(\ref{remainder term is a partial quantization K3cob}) in the free variables. Here, the dependencies of (\ref{partial quantization subsection operators of interest}) on the interaction variables (i.e., the variable which belong to $(X^{\tindex})^2$) have been suppressed in (\ref{quantization term is a partial quantization B3sc})--(\ref{remainder term is a partial quantization K3cob}). 
\endgroup

The operators (\ref{quantization term is a partial quantization B3sc})--(\ref{remainder term is a partial quantization K3cob}) can also be characterized more explicitly. Indeed, it is easy to see from the definition of the three-body operators that (\ref{quantization term is a partial quantization B3sc}) can be written as
\begin{equation} \label{definition of hat psi 3sc}
\hat{B}_{\psi, \mathrm{3sc}} = \frac{1}{(2\pi)^{n^{\tindex}}} \int_{\mathbb{R}^{n^{\tindex}}} e^{ i( z^{\tindex} - (z^{\tindex})' ) } b_{\psi, \mathrm{3sc}}( z_{\tindex}, z^{\tindex}, \zeta_{\tindex}, \zeta^{\tindex} ) d\zeta^{\tindex} |d(z^{\tindex})'|
\end{equation}
for some $b_{\psi, \mathrm{3sc}} \in S^{m, -\infty, l}( \overline{ ^{\mathrm{3sc}}T^{\ast} } [ \overline{\mathbb{R}^{n_{\tindex}}} ; \mathcal{C}_{\tindex} ] )$. Likewise, we can write (\ref{quantization term is a partial quantization B3cob}) using (\ref{the 3scb quantization written in terms of t}) as
\begin{align}
\begin{split} \label{definition of Bb}
\hat{B}_{\psi, \mathrm{3co,b}} =  {} & \frac{1}{(2\pi)^{n^{\tindex}}} \int_{\mathbb{R}^{n^{\tindex}}}  e^{ - i (t^{\tindex} - (t^{\tindex})') \tscbut + i ( y^{\tindex} - (y^{\tindex})'  ) \cdot \tscbum  }  \varphi(t-t') \varphi( |y^{\tindex} - ( y^{\tindex} )'| )  \\
& \times b_{\psi_{\dff} \psi^{\tindex}, \mathrm{3co,b} } ( z_{\tindex}, t^{\tindex}, y^{\tindex}, \tscblz , \tscbut , \tscbum  )  d \tscbut d\tscbum | d(t^{\tindex})' d(y^{\tindex})' |
\end{split}
\end{align} 
for some $b_{\psi, \mathrm{3co,b}} \in S^{m,-\infty, l, b}( \overline{ ^{\mathrm{3co}}T^{\ast} } X )$. 
\par

On the other hand, the integral (\ref{remainder term is a partial quantization R3cob}) converges absolutely, and in particular must decay to infinite order as $|\zeta_{\tindex}| \rightarrow \infty$, since $R_{\psi, \mathrm{3co,b}}$ also has this property in $|Z_{\tindex}|$. Thus, it is easy to see from (\ref{3scb part near diagonal off diagonal part}) that $\hat{R}_{\psi, \mathrm{3co,b}}$ will be smooth in all of the variables, and moreover must satisfy the estimates
\begin{align} 
\label{quote and quote symbolic estiamtes}
\begin{split}
& | \partial_{z_{\tindex}}^{\beta_{\tindex}} \partial_{\tscblz}^{\gamma_{\tindex}} \partial_{t^{\tindex}}^{j} \partial_{y^{\tindex}}^{\beta^{\tindex}} \partial_{(t^{\tindex})'}^{j'} \partial_{(y^{\tindex})'}^{ ( \beta^{\tindex} )' } \hat{R}_{\psi, \mathrm{3co,b}} | \\
& \qquad \leq C_{ \beta_{\tindex} \gamma_{\tindex} j \beta^{\tindex} j' (\beta^{\tindex})' } x_{\dff}^{-l} x_{\mathcal{C}^{\tindex}}^{-b} \langle z_{\tindex} \rangle^{-|\beta_{\tindex}|} \langle \tscblz \rangle^{-N} \langle y^{\tindex} - (y^{\tindex})' \rangle^{-L} e^{- M | t^{\tindex} - (t^{\tindex})' | }
\end{split}
\end{align}
for all $N,M \in \mathbb{R}$. Here $x_{\mathcal{C}^{\tindex}} \in \mathcal{C}^{\infty}( X^{\tindex} )$ is a boundary defining function, and this can be used to replace $x_{\cf}$ since $x_{\cf} \simeq x_{\mathcal{C}^{\tindex}}$ in a small neighborhood of $\dff$. We also let $x_{\mathcal{C}^{\tindex}}'$ denote the same function written in the right interaction variables. In particular, neither $x_{\mathcal{C}^{\tindex}}$ nor $x_{\mathcal{C}^{\tindex}}'$ depend on the free variables. \par

The above discussion for $\hat{R}_{\psi, \mathrm{3co,b}}$ extends essentially verbatim to the studies of $\hat{K}_{\phi, \psi, \mathrm{3co,b}}$. It is therefore easy to see that each $\hat{K}_{\phi, \psi, \mathrm{3co,b}}$ is smooth in all of the variables. Moreover, they must satisfy the estimates
\begin{align}
\label{estimates for operator valued symbol near dff constant order K}
\begin{split}
| \partial_{z_{\tindex}}^{\beta_{\tindex}} \partial_{\tscblz}^{\gamma_{\tindex}} \partial_{z^{\tindex}}^{\beta^{\tindex}} \partial_{ (z^{\tindex})' }^{ ( \beta^{\tindex} )' } \hat{K}_{\phi, \psi, \mathrm{3co,b}} | & \leq C_{ \beta_{\tindex} \beta_{\tindex}' \beta^{\tindex} ( \beta^{\tindex} )' N } x_{\dff}^{-l} \langle z_{\tindex} \rangle^{-|\beta_{\tindex}|}   \langle \lz \rangle^{-N}, \\
| \partial_{z_{\tindex}}^{\beta_{\tindex}} \partial_{\tscblz}^{\gamma_{\tindex}} \partial_{t^{\tindex}}^{j} \partial_{y^{\tindex}}^{\gamma^{\tindex}} \partial_{(z^{\tindex})'}^{(\beta^{\tindex})'}  \hat{K}_{\phi, \psi, \mathrm{3co,b}}  | & \leq C_{  \beta_{\tindex} \gamma_{\tindex} j \gamma^{\tindex} ( \beta^{\tindex} )' NM }  x_{\dff}^{-l} ( x_{\mathcal{C}^{\tindex}}' )^{M} \langle z_{\tindex} \rangle^{ -|\beta_{\tindex}| }   \langle \lz \rangle^{-N} , \\
| \partial_{z_{\tindex}}^{\beta_{\tindex}} \partial_{\tscblz}^{\gamma_{\tindex}} \partial_{t^{\tindex}}^{rr} \partial_{y^{\tindex}}^{\beta^{\tindex}} \partial_{(t^{\tindex})'}^{j'} \partial_{(y^{\tindex})'}^{(\beta^{\tindex})'}   \hat{K}_{\phi, \psi, \mathrm{3co,b}}  | & \leq C_{ \beta_{\tindex} \gamma_{\tindex}' j \beta^{\tindex} j' (\beta^{\tindex})' NM } x_{\dff}^{-l} x_{\mathcal{C}^{\tindex}}^{-b} \langle z_{\tindex} \rangle^{ -|\beta_{\tindex}| }   \langle \lz \rangle^{-N} e^{  - M|t^{\tindex} - (t^{\tindex})' |}
\end{split}
\end{align}
for all $N,M \in \mathbb{R}$, corresponding respectively to cases (1)--(3) in the discussion of \S \ref{subsection construction of dff}, except for the `vice versa' part, in which case one simply switch the roles of the interaction variables in the obvious way. \par

Finally, these implications are clearly reversible, i.e., (\ref{definition of hat psi 3sc})--(\ref{estimates for operator valued symbol near dff constant order K}) can also be used to define (\ref{quantization term is a partial quantization B3sc})--(\ref{remainder term is a partial quantization K3cob}), and subsequently the operators in (\ref{partial quantization subsection operators of interest}). As such, the above discussions also provide an alternate characterization for the conormal three-cone operators near $\dff$. \par

In fact, suppose we write
\begin{equation*}
A_{\psi_{\dff}} \coloneq \psi_{\dff} A \psi_{\dff},
\end{equation*}
where $\psi_{\dff} \in \mathcal{C}^{\infty}( \Xd )$ is some cut-off function at $\dff$. Then we have the following observation regarding the membership of $\hat{A}_{\psi_{\dff}}$.
\begin{proposition}
\label{membership of operator valued symbol near dff}
Let $A \in \Psi_{\mathrm{3coc}}^{m,r,l,b}( \Xd )$, and set
\begin{equation}
\label{operator valued symbol near dff 1}
\hat{A}_{\psi_{\dff}} \coloneq  \int_{\mathbb{R}^{n_{\tindex}}} e^{-iZ_{\tindex} \cdot \zeta_{\tindex}}  A_{\psi_{\dff}} ( z_{\tindex}, z_{\tindex} - Z_{\tindex} ) dZ_{\tindex}  \nu_{\mathrm{b}}', \quad Z_{\tindex} = z_{\tindex} - z_{\tindex}'.
\end{equation}
Then we have
\begin{equation}
\label{operator valued symbol near dff 1.5}
\hat{A}_{\psi_{\dff}} \in S^{l}( \overline{\mathbb{R}^{n_{\tindex}}_{z_{\tindex}}} ; \Psi_{\mathrm{bc, lp}}^{m,  b - 2l} ( X^{\tindex} ;  {\mathbb{R}^{n_{\tindex}}_{\zeta_{\tindex}}} ) ).
\end{equation}
Moreover, it holds that
\begin{equation}
\label{operator valued symbol near dff 2}
A_{\psi_{\dff}} = \frac{1}{( 2 \pi )^{n_{\tindex}} } \int_{\mathbb{R}^{n_{\tindex}}} e^{ i( z_{\tindex} - z_{\tindex}' ) } \hat{A}_{\psi_{\dff}} ( z_{\tindex}, \zeta_{\tindex} ) d\zeta_{\tindex} |dz_{\tindex}'|.
\end{equation}
\end{proposition}
\begin{proof}
Formula (\ref{operator valued symbol near dff 2}) follows clearly from (\ref{operator valued symbol near dff 1}). To get (\ref{operator valued symbol near dff 1.5}), note that if $\psi^{\tindex} \in \mathcal{C}^{\infty}(X^{\tindex})$, then it is straightforward to see that
\begin{equation*} \label{construction of the three-cone operators product cut-off functions near dff}
\psi_{\dff} \psi^{\tindex} \in \mathcal{C}^{\infty}( X ).
\end{equation*}
Thus, we can apply the discussions earlier in this subsection with cut-offs $\psi_{\dff} \psi^{\tindex}$ in places of $\psi$. Concretely, we will consider all operators of the forms $\psi^{\tindex} \hat{A}_{\dff} \psi^{\tindex}$, $\phi^{\tindex} \hat{A}_{\dff} \psi^{\tindex}$, where $\phi^{\tindex} \in \mathcal{C}^{\infty}(X^{\tindex})$ is some cut-off function with sufficiently small support. We can then compare our discussion above with the constructions made in \S \ref{parameters-dependent families subsection} for the space $\Psi_{\mathrm{bc.lp}}^{m,r-l,b-2l}( X^{\tindex} ; { \mathbb{R}^{n_{\tindex}} } )$ (i.e., when $\mathcal{M}$ is a point) to see that (\ref{operator valued symbol near dff 1.5}) is true at least in the case when $l = 0$. For $l \neq 0$, we simply observe that $x_{\dff}^{-l} \simeq \langle z_{\tindex} \rangle^{l} x_{\mathcal{C}^{\tindex}}^{2l} $ on the support of $\hat{A}_{\psi_{\dff}}$. \par

This concludes the proof of the proposition.
\end{proof}

\subsection{Characterization near $\cf$ through partial quantization}
\label{subsection partial quantization near cf}
A similar procedure can be carried out near $\cf$. Thus in this subsection, we will show that for any $A \in \Psi_{\mathrm{3coc}}^{m,r,l,b}(\Xd)$, $\psi_{\cf} A \psi_{\cf}$ can also be characterized as the partial quantization in the free variables of some `operator-valued symbol', which now lives in $S^{b/2}( \overline{\mathbb{R}^{n_{\tindex}}} ; \Psi_{\mathrm{coc,lp}}^{m, r - b/2, l - b/2} ( [ \hat{X}^{\tindex} ; \{ 0 \} ] ; \mathbb{R}^{n_{\tindex}} ) )$. See Proposition \ref{membership of operator valued symbol near cf} below for the precise statements.
\par

We will now proceed as in \S \ref{subsection partial quantization near dff}. Thus, let $\phi, \psi \in \mathcal{C}^{\infty}( \Xd )$ be cut-off functions as in the beginning of \S \ref{subsection construction of cf}. Then we will realize all operators of the forms
\begin{equation}
\label{partial quantization subsection operators of interest cone operators}
\text{$B_{\psi, \mathrm{3co,co,sc}}$, $B_{\psi, \mathrm{3co,co,b}}$, $R_{\psi, \mathrm{3co,co,b}}$, $K_{\phi, \psi, \mathrm{3co,co,b}}$}
\end{equation}
as Schwartz kernels on $\overline{\mathbb{R}^{n_{\tindex}}} \times [ \hat{X}^{\tindex} ; \{ 0 \} ]$, which we view as sections of the bundle spanned by $|dz_{\tindex}'| \nu_{\mathrm{co}}'$, where $\nu_{\mathrm{co}}'$ is a strictly positive cone density that is written in the right variables. Thus we have $\nu_{\mathrm{co}}' \simeq | d(t^{\tindex})' d(y^{\tindex})' |$ near $\mathcal{C}^{\tindex}_{0}$ (i.e., the lift of $\{ 0 \}$) and $\nu_{\mathrm{co}}' \simeq | d( \hat{z}^{\tindex} )' |$ near $\mathcal{C}^{\tindex}_{\infty}$ (i.e., the lift of $\partial \hat{X}^{\tindex}$). \par

For $Z_{\tindex} = z_{\tindex} - z_{\tindex}'$ (and hence $dz_{\tindex}' \nu_{\mathrm{co}}' = | dZ_{\tindex} | \nu_{\mathrm{co}}'$)), we again define 
\begingroup
\allowdisplaybreaks
\begin{align}
\hat{B}_{\psi, \mathrm{3co,co,sc}} &\coloneq \big( \int_{\mathbb{R}^{n_{\tindex}}} e^{-iZ_{\tindex} \cdot \zeta_{\tindex}^{\mathrm{3co}} }  B_{ \psi , \mathrm{3co,co,sc}} ( z_{\tindex}, z_{\tindex} - Z_{\tindex} ) |dZ_{\tindex}| \big) \nu_{\mathrm{co}}', \label{quantization term is a partial quantization B 3co co sc} , \\
\hat{B}_{\psi, \mathrm{3co,co,b}} &\coloneq \big( \int_{\mathbb{R}^{n_{\tindex}}} e^{-iZ_{\tindex} \cdot \zeta_{\tindex}^{\mathrm{3co}} }  B_{ \psi , \mathrm{3co,co,b}} ( z_{\tindex}, z_{\tindex} - Z_{\tindex} ) |dZ_{\tindex}| \big) \nu_{\mathrm{co}}',  \label{quantization term is a partial quantization B 3co co b} \\
\hat{R}_{\psi, \mathrm{3co,co,b}} &\coloneq \big( \int_{\mathbb{R}^{n_{\tindex}}} e^{-iZ_{\tindex} \cdot \zeta_{\tindex}^{\mathrm{3co}} }  R_{ \psi , \mathrm{3co,co,b}} ( z_{\tindex}, z_{\tindex} - Z_{\tindex} ) |dZ_{\tindex}| \big) \nu_{\mathrm{co}}',  \label{quantization term is a partial quantization R 3co co b} \\
\hat{K}_{\phi, \psi, \mathrm{3co,co}} &\coloneq \big( \int_{\mathbb{R}^{n_{\tindex}}} e^{-iZ_{\tindex} \cdot \zeta_{\tindex}^{\mathrm{3co}} }  K_{\phi, \psi, \mathrm{3co,co}} ( z_{\tindex}, z_{\tindex} - Z_{\tindex} ) |dZ_{\tindex}| \big) \nu_{\mathrm{co}}'  \label{quantization term is a partial quantization K 3co co}
\end{align}
as parametrized families of operators acting on $[ \hat{X}^{\tindex} ; \{ 0 \} ]$, where in every case the space of parameters is $\mathbb{R}^{n_{\tindex}}_{z_{\tindex}} \times \mathbb{R}^{n_{\tindex}}_{\zeta_{\tindex}^{\mathrm{3co}}}$. Then it is immediate that we have
\begin{align}
B_{\psi, \mathrm{3co,co,sc}} & = \frac{1}{( 2 \pi )^{n_{\tindex}}} \int_{\mathbb{R}^{n_{\tindex}}} e^{ i ( z_{\tindex} - z_{\tindex}' ) \cdot \zeta_{\tindex}^{\mathrm{3co}} } \hat{B}_{ \psi, \mathrm{3co,co,sc} } ( z_{\tindex}, \zeta_{\tindex}^{\mathrm{3co}} )  d\zeta_{\tindex}^{\mathrm{3co}} | dz_{\tindex}' |, \label{B 3co co sc as partial quantization}  \\
B_{\psi, \mathrm{3co,co,b}} & = \frac{1}{( 2 \pi )^{n_{\tindex}}} \int_{\mathbb{R}^{n_{\tindex}}} e^{ i ( z_{\tindex} - z_{\tindex}' ) \cdot \zeta_{\tindex}^{\mathrm{3co}} } \hat{B}_{ \psi, \mathrm{3co,co,b} } ( z_{\tindex}, \zeta_{\tindex}^{\mathrm{3co}} )  d\zeta_{\tindex}^{\mathrm{3co}} | dz_{\tindex}' |, \label{B 3co co b as partial quantization} \\
R_{\psi, \mathrm{3co,co,b}} & = \frac{1}{( 2 \pi )^{n_{\tindex}}} \int_{\mathbb{R}^{n_{\tindex}}} e^{ i ( z_{\tindex} - z_{\tindex}' ) \cdot \zeta_{\tindex}^{\mathrm{3co}} } \hat{R}_{ \psi, \mathrm{3co,co,b} } ( z_{\tindex}, \zeta_{\tindex}^{\mathrm{3co}} )  d\zeta_{\tindex}^{\mathrm{3co}} | dz_{\tindex}' |, \label{R 3co co b as partial quantization} \\
K_{\phi, \psi, \mathrm{3co,co}  } & = \frac{1}{( 2 \pi )^{n_{\tindex}}} \int_{\mathbb{R}^{n_{\tindex}}} e^{ i ( z_{\tindex} - z_{\tindex}' ) \cdot \zeta_{\tindex}^{\mathrm{3co}} } \hat{K}_{ \phi, \psi, \mathrm{3co,co} } ( z_{\tindex}, \zeta_{\tindex}^{\mathrm{3co}} )  d\zeta_{\tindex}^{\mathrm{3co}} | dz_{\tindex}' |. \label{K 3co co as partial quantization}
\end{align}
\endgroup
As in the case near $\dff$, dependencies of the operators in (\ref{partial quantization subsection operators of interest cone operators}) on the interaction variables, which are now the variables belonging to $[ \hat{X}^{\tindex} ; \{ 0 \} ]^2$, have been suppressed in (\ref{quantization term is a partial quantization B 3co co sc})--(\ref{quantization term is a partial quantization K 3co co}).
\par 

We now characterize (\ref{quantization term is a partial quantization B 3co co sc})--(\ref{quantization term is a partial quantization K 3co co}) more explicitly. First, by using (\ref{3cosc quantization}), it is easy to see that (\ref{quantization term is a partial quantization B 3co co sc}) can also be written as
\begin{equation}
\label{partial quantization hat B 3co co sc}
\hat{B}_{\psi, \mathrm{3co, co, sc}} = \frac{1}{(2\pi)^{n^{\tindex}}} \int_{\mathbb{R}^{n^{\tindex}}}  e^{ i( \hat{z}^{\tindex} - (\hat{z}^{\tindex})' ) \cdot \zeta^{\tindex}_{\mathrm{co,sc}} } b_{ \psi , \mathrm{3co, co,sc} }( z_{\tindex}, \hat{z}^{\tindex}, \zeta_{\tindex}^{\mathrm{3co}}, \zeta^{\tindex}_{\mathrm{co,sc}} ) d\zeta^{\tindex}_{\mathrm{co,sc}} |d ( \hat{z}^{\tindex} )' |
\end{equation}
for some $b_{\psi, \mathrm{3co, co,sc}} \in S^{m,r,-\infty, b}( \overline{ ^{\mathrm{3co}}T^{\ast} } \Xd )$. Likewise, by using (\ref{the 3cob quantization written in terms of t}), we can write (\ref{quantization term is a partial quantization B 3co co b}) as
\begin{align} 
\label{3co,b operator-valued symbol definittion}
\begin{split}
\hat{B}_{\psi, \mathrm{3co,co, b}} \coloneq {} & \frac{1}{( 2\pi )^{n^{\tindex}}}  \int_{\mathbb{R}^{n^{\tindex}}}  e^{ - i ( \hat{t}_{\tindex} - \hat{t}_{\tindex} ) \tcobut + i ( y^{\tindex} - (y^{\tindex})' ) \cdot \tcobum  } \varphi( \hat{t}_{\tindex} - \hat{t}_{\tindex}' ) \varphi( |y^{\tindex} - ( y^{\tindex} )'| )  \\
& \times b_{ \psi, \mathrm{3co,co, b} } ( z_{\tindex}, \hat{t}_{\ff}, y^{\tindex}, \tcoblz , \tcobut , \tcobum  )  d \tcobut d\tcobum | d(\hat{t}_{\tindex})' d(y^{\tindex})' |
\end{split}
\end{align}
for some $b_{\psi , \mathrm{3co,co,b}} \in S^{m, -\infty, l ,b}( \overline{ ^{\mathrm{3co}}T^{\ast} } X )$. \par

Next, we note from (\ref{estimate on R 3b b}) that $R_{\psi ,\mathrm{3co,co,b}}$ decays to infinite orders as $|Z_{\tindex}| \rightarrow \infty$. Hence the integral in (\ref{quantization term is a partial quantization R 3co co b}) converges absolutely. In particular, $\hat{R}_{\psi, \mathrm{3co,co,b}}$ is smooth in all of the variables, and must also decay to infinite orders as $|\zeta_{\tindex}^{\mathrm{3co}}| \rightarrow \infty$, i.e., we have the estimates
\begin{align}
\label{symbolic estimate operator valued symbol R near cf}
\begin{split}
 & | \partial_{z_{\tindex}}^{\beta_{\tindex}} \partial_{ \tcottlz }^{\gamma_{\tindex}} \partial_{\hat{t}_{\tindex} }^{j} \partial_{y^{\tindex}}^{\beta^{\tindex}} \partial_{\hat{t}_{\tindex}'}^{j'}  \partial_{ (y^{\tindex})' }^{(\beta^{\tindex})'} \hat{R}_{\psi, \mathrm{3co,co,b}} | \\ 
 & \qquad \leq C_{\beta_{\tindex} \gamma_{\tindex} j j' \beta^{\tindex} (\beta^{\tindex})'NML}  x_{\cf}^{-b} x_{\mathcal{C}^{\tindex}_{0}}^{-l}    \langle z_{\tindex} \rangle^{-|\beta_{\tindex}| } \langle \tcottlz \rangle^{-N} \langle y^{\tindex} - ( y^{\tindex} )' \rangle^{-L} e^{ -M| \hat{t}_{\tindex} - \hat{t}_{\tindex}' |}
\end{split}
\end{align}
for all $N,L, M \in \mathbb{R}$. Here $x_{\mathcal{C}^{\tindex}_{0}} \in \mathcal{C}^{\infty}( [\hat{X}^{\tindex} ; \{ 0 \} ] )$ is a defining function for $\mathcal{C}^{\tindex}_{0}$. We also let $x_{\mathcal{C}^{\tindex}_{\infty}} \in \mathcal{C}^{\infty}( [ \hat{X}^{\tindex} ; \{ 0 \} ] )$ be a defining function for $\mathcal{C}^{\tindex}_{\infty}$. These functions satisfy $x_{\dff} \simeq x_{\mathcal{C}^{\tindex}_{0}}$, $x_{\dmf} \simeq x_{\mathcal{C}^{\tindex}_{\infty}}$ in a small neighborhood of $\cf$. Moreover, let $x_{\mathcal{C}^{\tindex}_{0}}'$, $x_{\mathcal{C}^{\tindex}_{\infty}}'$ denote respectively the representations of $x_{\mathcal{C}^{\tindex}_{0}}$, $x_{\mathcal{C}^{\tindex}_{\infty}}$ in the right interaction variables. In particular, none of the the functions we have just introduced depend on the free variables.
\par

The argument that was used in the discussion for $\hat{R}_{\psi, \mathrm{3co,co,b}}$ can be applied to the studies of $\hat{K}_{\phi, \psi , \mathrm{3co,co}}$ verbatim. From this, it is easy to see that each $\hat{K}_{\phi, \psi, \mathrm{3co,co}}$ is smooth in all of the variables. Moreover, corresponding to cases (1)--(6) in \S {\ref{subsection construction of cf}}, they respectively satisfy
\begingroup
\allowdisplaybreaks
\begin{align}
\label{symbolic estimate operator valued symbol K near cf}
\begin{split}
| \partial_{z_{\tindex}}^{\beta_{\tindex}}  \partial_{\tcottlz}^{ \gamma_{\tindex}} \partial_{ \hat{z}^{\tindex}}^{\beta^{\tindex}} \partial_{( \hat{z}^{\tindex})'}^{ (\beta^{\tindex})' } \hat{K}_{\phi, \psi, \mathrm{3co,co}} |  \leq {} & C_{ \beta_{\tindex} \gamma_{\tindex} \beta^{\tindex} ( \beta^{\tindex} )' N } x_{\cf}^{-b} \langle z_{\tindex} \rangle^{-|\beta_{\tindex}| }  \langle \tcottlz \rangle^{-N} , \\
| \partial_{z_{\tindex}}^{\beta_{\tindex}}  \partial_{ \tcottlz }^{ \beta_{\tindex}' } \partial_{\hat{t}_{\tindex}}^{j} \partial_{y^{\tindex}}^{\gamma^{\tindex}} \partial_{ (\hat{z}^{\tindex})' }^{ (\beta^{\tindex})' } \hat{K}_{\phi, \psi, \mathrm{3co,co}} | \leq {} & C_{\beta_{\tindex} \gamma_{\tindex} j \gamma^{\tindex} ( \beta^{\tindex} )' NML} x_{\cf}^{-b} x_{\mathcal{C}^{\tindex}_{0}}^{M} ( x_{\mathcal{C}^{\tindex}_{\infty}}' )^{L} \langle z_{\tindex} \rangle^{-|\beta_{\tindex}|}   \langle \tcottlz \rangle^{-N},  \\ 
| \partial_{z_{\tindex}}^{\beta_{\tindex}}  \partial_{ \tcottlz }^{ \beta_{\tindex}' } \partial_{\hat{t}_{\tindex}}^{j} \partial_{y^{\tindex}}^{\gamma^{\tindex}} \partial_{ (\hat{z}^{\tindex})' }^{ (\beta^{\tindex})' } \hat{K}_{\phi, \psi, \mathrm{3co,co}} | \leq {} & C_{\beta_{\tindex} \gamma_{\tindex} j \gamma^{\tindex} ( \beta^{\tindex} )' NM} x_{\cf}^{-b} (x_{\mathcal{C}^{\tindex}_{0}}')^{M} \langle z_{\tindex} \rangle^{-|\beta_{\tindex}|  }   \langle \tcottlz \rangle^{-N}  , \\ 
| \partial_{z_{\tindex}}^{\beta_{\tindex}}  \partial_{ 
\tcottlz }^{\gamma_{\tindex}}  \partial_{ \hat{z}^{\tindex}}^{\beta^{\tindex}} \partial_{( \hat{z}^{\tindex})'}^{ (\beta^{\tindex})' }\hat{K}_{\phi, \psi, \mathrm{3co,co}} | \leq {} & C_{ \beta_{\tindex} \gamma_{\tindex} \beta^{\tindex} ( \beta^{\tindex} )' NM } x_{\cf}^{-b} (x_{\mathcal{C}^{\tindex}_{\infty}}')^{M}  \langle z_{\tindex} \rangle^{-|\beta_{\tindex}|  }  \langle \tcottlz \rangle^{-N}, \\
| \partial_{z_{\tindex}}^{\beta_{\tindex}}  \partial_{ \tcottlz }^{\gamma_{\tindex}}  \partial_{ \hat{z}^{\tindex}}^{\beta^{\tindex}} \partial_{( \hat{z}^{\tindex})'}^{ (\beta^{\tindex})' } \hat{K}_{\phi, \psi, \mathrm{3co,co}} | \leq {} & C_{ \beta_{\tindex} \gamma_{\tindex} \beta^{\tindex} ( \beta^{\tindex} )' N M M' } x_{\cf}^{-b} x_{\mathcal{C}^{\tindex}_{\infty}}^{M}  (x_{\mathcal{C}^{\tindex}_{\infty}}')^{M'}  \langle z_{\tindex} \rangle^{-|\beta_{\tindex}| }  \langle \tcottlz \rangle^{-N} , \\
 | \partial_{z_{\tindex}}^{\beta_{\tindex}}  \partial_{ \tcottlz }^{ \gamma_{\tindex} }   \partial_{ \hat{t}_{\tindex} }^{j} \partial_{y^{\tindex}}^{\beta^{\tindex}} \partial_{ \hat{t}_{\tindex}' }^{j'}  \partial_{ (y^{\tindex})' }^{(\beta^{\tindex})'} \hat{K}_{\phi, \psi, \mathrm{3co,co}} | \leq {} &  C_{\beta_{\tindex} \gamma_{\tindex} j \beta^{\tindex} j' (\beta^{\tindex})' NM} x_{\cf}^{-b} x_{\mathcal{C}^{\tindex}_{0}}^{-l}  \langle z_{\tindex} \rangle^{-|\beta_{\tindex}| }  \langle \tcottlz \rangle^{-N} e^{ -M| \hat{t}_{\tindex} - \hat{t}_{\tindex}' |}
 \end{split}
\end{align}
\endgroup
for all $N,M, M', L \in \mathbb{R}$, except for the `vice versa' parts, in which cases one simply switch the roles of the interaction variables in the obvious ways. \par

Finally, these implications are again reversible, i.e., (\ref{partial quantization hat B 3co co sc})--(\ref{symbolic estimate operator valued symbol K near cf}) can also be used to define (\ref{quantization term is a partial quantization B 3co co sc})--(\ref{quantization term is a partial quantization K 3co co}), and subsequently the operators in (\ref{partial quantization subsection operators of interest cone operators}). As such, the above discussions also provide an alternate characterization for the conormal three-cone operators near $\cf$. \par

In fact, suppose we write
\begin{equation*}
A_{\psi_{\cf}} \coloneq \psi_{\cf} A \psi_{\cf},
\end{equation*}
where $\psi_{\cf} \in \mathcal{C}^{\infty}(\Xd)$ is some cut-off function at $\cf$. Then we have the following observation regarding membership of $\hat{A}_{\psi_{\cf}}$. 
\begin{proposition}
\label{membership of operator valued symbol near cf}
Let $A \in \Psi_{\mathrm{3coc}}^{m,r,l,b}( \Xd )$, and set
\begin{equation} 
\label{three-cone case writing A psicf as a partial quantization -1}
\hat{A}_{\psi_{\cf}} \coloneq  \int_{\mathbb{R}^{n_{\tindex}}} e^{-iZ_{\tindex} \cdot \zeta_{\tindex}^{\mathrm{3co}} }  A_{\psi_{\cf}} ( z_{\tindex}, z_{\tindex} - Z_{\tindex} ) dZ_{\tindex}  \nu_{\mathrm{co}}', \quad Z_{\tindex} = z_{\tindex} - z_{\tindex}'.
\end{equation}
Then we have
\begin{equation*}
\hat{A}_{\psi_{\cf}} \in S^{b/2}( \overline{\mathbb{R}^{n_{\tindex}}_{z_{\tindex}}} ; \Psi_{\mathrm{coc,lp}}^{m, r - b/2, l - b/2} ( [ \hat{X}^{\tindex} ; \{ 0 \} ] ; \mathbb{R}^{n_{\tindex}}_{\zeta_{\tindex}^{\mathrm{3co}}} ) ).
\end{equation*}
Moreover, it holds that
\begin{equation} 
\label{three-cone case writing A psicf as a partial quantization}
A_{\psi_{\cf}} = \frac{1}{( 2 \pi )^{n_{\tindex}} } \int_{\mathbb{R}^{n_{\tindex}}} e^{ i( z_{\tindex} - z_{\tindex}' ) \cdot \zeta_{\tindex}^{\mathrm{3co}} } \hat{A}_{\psi_{\cf}} ( z_{\tindex}, \zeta_{\tindex}^{\mathrm{3co}} ) d\zeta_{\tindex}^{\mathrm{3co}} |dz_{\tindex}'|.
\end{equation}
\end{proposition}
\begin{proof}
This proposition is the direct analogy of Proposition \ref{membership of operator valued symbol near dff} near $\cf$, and thus their proofs are similar. The main difference is that we now have $\psi_{\cf} \psi^{\tindex} \in \mathcal{C}^{\infty}(\Xd)$ whenever $\psi^{\tindex} \in \mathcal{C}^{\infty}( [ \hat{X}^{\tindex} ; \{ 0 \} ] )$. Moreover, $x_{\cf}^{-b} \simeq \langle z_{\tindex} \rangle^{b/2} x_{\mathcal{C}^{\tindex}_{0}}^{b/2} x_{\mathcal{C}^{\tindex}_{\infty}}^{b/2}$ on the support of $\hat{A}_{\psi_{\cf}}$.
\end{proof}

\section{Second microlocalization} 
\label{section second microlocalization}
Having constructed the spaces of conormal three-cone operators $\Psi^{m,r,l, b}_{\mathrm{3coc}}( \Xd )$, we now proceed to second microlocalize the three-body algebra at $\tscf^{-1}( \sco )$ as announced. Here, let us recall that $\tscf$ is defined by the natural projection
\begin{equation} \label{three-body cotangent bundle natural projection}
\tscf : {^{\mathrm{3sc}}T^{\ast}_{\ff}} [ \overline{\mathbb{R}^{n}} ; \mathcal{C}_{\tindex} ] \cong \mathcal{C}_{\tindex} \times \mathbb{R}^{n_{\tindex}} \times {^{\mathrm{sc}}T^{\ast}} X^{\tindex} \rightarrow {^{\mathrm{sc}}T^{\ast}} X^{\tindex}
\end{equation}
while $o_{\mathcal{C}^{\tindex}} \subset {^{\mathrm{sc}}T^{\ast}_{\mathcal{C}^{\tindex}}} X^{\tindex}$ is the zero section. The spaces of second microlocalized operators, once properly defined, will be denoted by
\begin{equation} 
\label{the OG second microlocalized algebra}
\Psi^{m,r,l,b}_{\mathrm{3sc,2}} \big( [ \overline{\mathbb{R}^{n}} ; \mathcal{C}_{\tindex} ] ; \overline{\pi_{\ff}^{-1}( o_{\mathcal{C}^{\tindex}} )} \, \big),
\end{equation}
though we will actually consider a slightly more general class below (with different notation).
 \par

As mentioned in the introduction, our construction will follow the philosophy laid out in \cite{AndrasSM}. Thus, there will be a `direct' perspective, as well as a `converse' perspective{\ep}much like what was done in \S \ref{subsection Vasy's second microlocalized calculus}. The converse perspective will be  based on the conormal three-cone operators constructed in \S \ref{the conormal three-cone operators section}. We will explain second microlocalization from both the direct and converse perspectives, and then connect them via a blow-up at the end. It will become obvious how (\ref{the OG second microlocalized algebra}) should be defined once this connection is established.

\subsection{Diffeomorphism of the phase spaces} \label{subsection diffeomorphism of the phase spaces}
We first explain the direct perspective. In fact, it will be useful to proceed more generally. Thus, in addition to considering the three-body structure, we will also define the \emph{decoupled} three-body cotangent bundle by a (somewhat trivial) corner blow-up
\begin{equation} \label{construction of the d3sc bundle}
\overline{^{\mathrm{d3sc}}T^{\ast}} \Xd \coloneq \big[ \overline{ ^{\mathrm{3sc}}T^{\ast} } [ \overline{\mathbb{R}^{n}} ; \mathcal{C}_{\tindex} ] ; \overline{ ^{\mathrm{3sc}}T^{\ast} }_{ \mathrm{mf} \cap \ff } [ \overline{\mathbb{R}^{n}} ; \mathcal{C}_{\tindex} ] \big].
\end{equation}
Let ${^{\mathrm{d3sc}}S^{\ast}}\Xd$, $\overline{^{\mathrm{d3sc}}T^{\ast}}_{\dmf}\Xd$, $\overline{^{\mathrm{d3sc}}T^{\ast}}_{\dff} \Xd$ denote the lifts of ${^{\mathrm{3sc}}S^{\ast}} [ \overline{\mathbb{R}^{n}} ; \mathcal{C}_{\tindex} ]$, $\overline{^{\mathrm{3sc}}T^{\ast}}_{\dmf} [ \overline{\mathbb{R}^{n}} ; \mathcal{C}_{\tindex} ]$ and $\overline{^{\mathrm{3sc}}T^{\ast}}_{\dff} [ \overline{\mathbb{R}^{n}} ; \mathcal{C}_{\tindex} ]$ respectively. We also let $\overline{^{\mathrm{d3sc}}T^{\ast}}_{\cf} \Xd$ denote the new front face created by the blow-up (\ref{construction of the d3sc bundle}), or equivalently, the lift of $\overline{ ^{\mathrm{3sc}}T^{\ast} }_{ \mathrm{mf} \cap \ff } [ \overline{\mathbb{R}^{n}} ; \mathcal{C}_{\tindex} ]$. \par 
It is indeed easy to see that there exist well-behaved spaces of operators
\begin{equation} \label{algebra definition of the decoupled 3sc operators}
\Psi_{\mathrm{d3scc}}^{m,r,l,\nu}( \Xd )
\end{equation}
that can be defined globally by quantising conormal symbols on $\overline{^{\mathrm{d3sc}}T^{\ast}} \Xd$, i.e., the spaces
\begin{equation} \label{decoupled 3sc symbols}
S^{m,r,l,\nu} ( \overline{^{\mathrm{d3sc}}T^{\ast}} \Xd ),
\end{equation}
in the sense that each element $A \in \Psi_{\mathrm{d3scc}}^{m,r,l,\nu}(X)$ can be defined in Euclidean coordinates $(z,\zeta)$ by an operator of the form
\begin{equation}
\label{the standard quantization again}
\frac{1}{(2\pi)^{n}} \int_{\mathbb{R}^{n}} e^{i ( z - z' ) \cdot \zeta} a(z,\zeta) d\zeta |dz'|
\end{equation}
for some $a \in S^{m,r,l,\nu}( \overline{ ^{\mathrm{d3sc}}T^{\ast} }X )$. \par

Here, the indices $m,r,l, \nu \in \mathbb{R}$ measure decay at ${^{\mathrm{d3sc}}S^{\ast}}\Xd$, $\overline{^{\mathrm{d3sc}}T^{\ast}}_{\dmf}\Xd$, $\overline{^{\mathrm{d3sc}}T^{\ast}}_{\dff}\Xd$ and $\overline{^{\mathrm{d3sc}}T^{\ast}}_{\cf}\Xd$ respectively. As a consequence of this construction, it is straightforward to verify that one could microlocalize elements of (\ref{algebra definition of the decoupled 3sc operators}) even at the face ${ \overline{^{\mathrm{d3sc}}T^{\ast}}_{\cf} \Xd }$, in the sense that the principal symbol map extends to capture principal order decay at this face. \par 
Since (\ref{construction of the d3sc bundle}) is a blow-up of $\overline{ ^{\mathrm{3sc}}T^{\ast}} [ \overline{\mathbb{R}^{n}} ; \mathcal{C}_{\tindex} ]$ at a corner face, its effects are extremely mild in the microlocal sense that conormality is insensitive under this blow-up. Namely, it is easy to verify that
\begin{equation} \label{blow up is insensitive from 3sc to d3sc}
S^{m,r,l}( \overline{ ^{\mathrm{3sc}}T^{\ast}} [ \overline{\mathbb{R}^{n}}  ; \mathcal{C}_{\tindex} ] ) = S^{m,r,l,r+l} ( \overline{^{\mathrm{d3sc}}T^{\ast}} \Xd ),
\end{equation}
which directly implies that
\begin{equation*}
\Psi^{m,r,l}_{\mathrm{3scc}}( [ \overline{\mathbb{R}^{n}} ; \mathcal{C}_{\tindex} ] ) = \Psi_{\mathrm{d3scc}}^{m,r,l,r+l}( \Xd ).
\end{equation*}
Thus, $\Psi_{\mathrm{d3scc}}^{m,r,l,\nu}( \Xd )$ can indeed be viewed as a refinement of $\Psi_{\mathrm{3scc}}^{m,r,l}( [ \overline{\mathbb{R}^{n}} ; \mathcal{C}_{\tindex} ] )$. \par

Let us recall that our main objective of second microlocalization, herustically speaking, is as follows: An operator $A$ which belongs to $\Psi^{m,r,l,b}_{\mathrm{3sc,2}}( X ; \overline{ \tscf^{-1}( o_{\mathcal{C}^{\tindex}} )} )$, if somehow defined, must be microlocalized on the phase space
\begin{equation} 
\label{second microlocalized 3sc operators must be microlocalized here}
\big[ \, \overline{^{\mathrm{3sc}}T^{\ast}} [ \overline{\mathbb{R}^{n}} ; \mathcal{C}_{\tindex} ] ; \overline{\tscf^{-1}( \sco )} \, \big].
\end{equation}
We now make this statement slightly more precise: the symbol of $A$ must belong to
\begin{equation*}
S^{m,r,l,b}\big( \big[ \, \overline{^{\mathrm{3sc}}T^{\ast}} [ \overline{\mathbb{R}^{n}} ; \mathcal{C}_{\tindex} ] ; \overline{\tscf^{-1}( \sco )} \, \big] \big).
\end{equation*} 
Here, the indices $m,r,l \in \mathbb{R}$ measure decay at the lifts of ${^{\mathrm{3sc}}S^{\ast}} [ \overline{\mathbb{R}^{n}} ; \mathcal{C}_{\tindex} ]$, $\overline{ ^{\mathrm{3sc}}T^{\ast} }_{\mf} [ \overline{\mathbb{R}^{n}} ; \mathcal{C}_{\tindex} ] $ and $\overline{ ^{\mathrm{3sc}}T^{\ast} }_{\ff} [ \overline{\mathbb{R}^{n}} ; \mathcal{C}_{\tindex} ]$ respectively, while $b \in \mathbb{R}$ measures decay at the new front face of (\ref{second microlocalized 3sc operators must be microlocalized here}), or equivalently, the lift of $ \overline{\tscf^{-1}( \sco )}$. \par

In fact, it will be convenient to blow up (\ref{second microlocalized 3sc operators must be microlocalized here}) further at the lift of $\overline{ ^{\mathrm{3sc}}T^{\ast}}_{\mathrm{mf} \cap \ff } [ \overline{\mathbb{R}^{n}} ; \mathcal{C}_{\tindex} ]$. Moreover, since $ \overline{\tscf^{-1}( \sco )} \subset \overline{ ^{\mathrm{3sc}}T^{\ast}}_{\mathrm{mf} \cap \ff } [ \overline{\mathbb{R}^{n}} ; \mathcal{C}_{\tindex} ] $, we could carry out the blow-up at $\overline{ ^{\mathrm{3sc}}T^{\ast}}_{\mathrm{mf} \cap \ff }X$ first instead. More precisely, let 
\begin{equation*}
\beta_{\mathrm{d3sc}}:  \overline{^{\mathrm{d3sc}}T^{\ast}} \Xd \rightarrow  \overline{ ^{\mathrm{3sc}}T^{\ast} } [ \overline{\mathbb{R}^{n}} ; \mathcal{C}_{\tindex} ]
\end{equation*}
denote the blow-down map of (\ref{construction of the d3sc bundle}). Then we have
\begin{equation} \label{second microlocal blow up 3s tp d3sc}
\big[ \overline{^{\mathrm{3sc}}T^{\ast}} [ \overline{\mathbb{R}^{n}} ; \mathcal{C}_{\tindex} ] ; \overline{\tscf^{-1}( \sco )}, \overline{ ^{\mathrm{3sc}}T^{\ast}}_{\mathrm{mf} \cap \ff } [ \overline{\mathbb{R}^{n}} ; \mathcal{C}_{\tindex} ] \big] = \big[ \overline{^{\mathrm{d3sc}}T^{\ast}} \Xd ; \beta_{\mathrm{d3sc}}^{\ast} ( \overline{\tscf^{-1}( 
{\sco} )} ) \big].
\end{equation} 
Notice that $\overline{ ^{\mathrm{3sc}}T^{\ast}}_{\mathrm{mf} \cap \ff } [ \overline{\mathbb{R}^{n}} ; \mathcal{C}_{\tindex} ]$ lifts to a corner face of $[ \overline{^{\mathrm{3sc}}T^{\ast}} [ \overline{\mathbb{R}^{n}} ; \mathcal{C}_{\tindex} ] ; \overline{\tscf^{-1}( \sco )}]$ as well, and so (\ref{second microlocal blow up 3s tp d3sc}) is again the blow-up of $[ \overline{^{\mathrm{3sc}}T^{\ast}} [ \overline{\mathbb{R}^{n}} ; \mathcal{C}_{\tindex} ] ; \overline{\tscf^{-1}( \sco )} ]$ at a corner face. Conormality is therefore insensitive under this blow-up, and we have as in (\ref{blow up is insensitive from 3sc to d3sc}) that
\begin{align}
\begin{split}
\label{blow up is insensitive from 3sc to d3sc second microlocalized}
& S^{m,r,l,b}\big( \big[ \overline{^{\mathrm{3sc}}T^{\ast}} [ \overline{\mathbb{R}^{n}} ; \mathcal{C}_{\tindex} ] ; \overline{\tscf^{-1}( \sco )} \, \big] \big) = S^{m,r,l, r+l, b} \left( \big[ \overline{^{\mathrm{d3sc}}T^{\ast}} \Xd ; \beta_{\mathrm{d3sc}}^{\ast} ( \overline{\tscf^{-1}( 
{\sco} )} ) \big] \right),
\end{split}
\end{align}
where the right hand side of (\ref{blow up is insensitive from 3sc to d3sc second microlocalized}) could also be defined more generally as
\begin{equation*}
S^{m,r,l, \nu, b} \left( \big[ \overline{^{\mathrm{d3sc}}T^{\ast}} \Xd ; \beta_{\mathrm{d3sc}}^{\ast} ( \overline{\tscf^{-1}( 
{\sco} )} ) \big] \right).
\end{equation*}
Here, the indices $m,r,l,\nu \in \mathbb{R}$ measure decay at the lifts of ${^{\mathrm{d3sc}}S^{\ast}} \Xd$,  $\overline{^{\mathrm{d3sc}}T^{\ast}}_{\dmf} \Xd$, $\overline{^{\mathrm{d3sc}}T^{\ast}}_{\dff} \Xd$ and $\overline{^{\mathrm{d3sc}}T^{\ast}}_{\cf} \Xd$ respectively, while $b \in \mathbb{R}$ measures decay at the new front face of the right hand side of (\ref{second microlocal blow up 3s tp d3sc}), or equivalently the lift of $\overline{\tscf^{-1}( \sco )}$ (or $\beta_{\mathrm{d3sc}}^{\ast}( \overline{\tscf^{-1}( \sco )})$) .\par

Now, (\ref{blow up is insensitive from 3sc to d3sc second microlocalized}) naturally motivates us to ask whether we could instead second microlocalize the decoupled space of operators $\Psi^{m,r,l,\nu}_{\mathrm{d3scc}}( \Xd )$ at $\beta_{\mathrm{d3sc}}^{\ast} ( \overline{\tscf^{-1}( 
{\sco} )} )$, and define a slightly larger space of operators
\begin{equation} 
\label{formally written decoupled second microlocalized operators}
\Psi_{\mathrm{d3sc,2}}^{m,r,l,\nu, b} \left( \Xd ; \beta_{\mathrm{d3sc}}^{\ast}(\overline{\tscf^{-1}(\sco)}) \right)
\end{equation}
in such a way that, heuristically speaking, the right hand side of (\ref{second microlocal blow up 3s tp d3sc}) is the phase space for (\ref{formally written decoupled second microlocalized operators}). This procedure, if done, would have the potential to be useful (and indeed it is, as we will see below shortly), for suppose that we have somehow defined (\ref{formally written decoupled second microlocalized operators}), then we would also have defined space (\ref{the OG second microlocalized algebra}) simply by setting
\begin{equation}
\label{formally written decoupled second microlocalized operators 1}
\Psi_{\mathrm{3sc,2}}^{m,r,l,b} \left(  [ \overline{\mathbb{R}^{n}} ; \mathcal{C}_{\tindex} ] ; \overline{\tscf^{-1}(\sco)} \, \right) \coloneq \Psi_{\mathrm{d3sc,2}}^{m,r,l,r+l, b} \left( \Xd ; \beta_{\mathrm{d3sc}}^{\ast}(\overline{\tscf^{-1}(\sco)}) \right),
\end{equation}
which must be reasonable in view of (\ref{blow up is insensitive from 3sc to d3sc second microlocalized}). In particular, the space (\ref{formally written decoupled second microlocalized operators}) is again a refinement of $\Psi_{\mathrm{3sc,2}}^{m,r,l,b} ( [ \overline{\mathbb{R}^{n}} ; \mathcal{C}_{\tindex} ] ; \overline{\tscf^{-1}(\sco)})$, and one pays only very little price in order to work with these more general operators (since one gets $\Psi^{m,r,l,\nu}_{\mathrm{d3scc}}( \Xd )$ almost `for free').
\par

On the other hand, let us consider the `converse' perspective, which starts from the three-cone structure. Recall that $\cf$ admits a trivial fibration with base $\mathcal{C}_{\tindex}$ and fibers $[ \hat{X}^{\tindex} ; \{ 0 \} ]$. Then one finds easily that the three-cone cotangent bundle restricts to a product structure
\begin{equation} \label{product identification at cf 3co sense}
^{\mathrm{3co}}T^{\ast}_{\mathrm{cf}_{\tindex}} \Xd \cong \mathcal{C}_{\tindex} \times \mathbb{R}^{n_{\tindex}} \times { ^{\mathrm{co}}T^{\ast} }[ \hat{X}^{\tindex} ; \{ 0 \} ].
\end{equation}
We are interested in extending (\ref{product identification at cf 3co sense}) to its closure. In fact, it is not hard to see that
\begin{equation} \label{a simple identification of the sphere bundle at ff in the 3co sense}
\overline{^{\mathrm{3co}}T^{\ast}}_{\mathrm{cf}_{\tindex}} \Xd \cong \mathcal{C}_{\tindex} \times \overline{ ^{\mathrm{co}}T^{\ast} \oplus \mathbb{R}^{n_{\tindex}} } [ \hat{X}^{\tindex} ; \{  0 \} ],
\end{equation}
where the notation on the right hand side of (\ref{a simple identification of the sphere bundle at ff in the 3co sense}) means that we are direct summing to each fiber of $^{\mathrm{co}}T^{\ast}[ \hat{X}^{\tindex} ; \{ 0 \} ]$ the Euclidean space $\mathbb{R}^{n_{\tindex}}$, and then compactifing the resulting space radially. Under identification (\ref{a simple identification of the sphere bundle at ff in the 3co sense}), the restricted fiber infinity ${^{\mathrm{3co}}S^{\ast}_{\mathrm{cf}_{\tindex}}}\Xd$ agrees with the Cartesian product between $\mathcal{C}_{\tindex}$ and the fiber infinity of $ \overline{ ^{\mathrm{co}}T^{\ast} \oplus \mathbb{R}^{n_{\tindex}} } [ \hat{X}^{\tindex} ; \{  0 \} ] $. Therefore, one can identify $\mathcal{C}_{\tindex} \times {^{\mathrm{co}}S^{\ast}[ \hat{X}^{\tindex} ; \{ 0 \} ]}$ naturally as a natural submanifold of ${^{\mathrm{3co}}S^{\ast}_{\cf}} \Xd$. In fact, if $\mathbb{R}^{n_{\tindex}}$ has coordinates $\zeta_{\tindex}^{\mathrm{3co}}$, then $\mathcal{C}_{\tindex} \times { ^{\mathrm{co}}S^{\ast} }[ \hat{X}^{\tindex} ; \{ 0 \} ]$ can be more explicitly identified as 
\begin{equation*}
\mathcal{C}_{\tindex} \times \partial \overline{ \{ \zeta_{\tindex}^{\mathrm{3co}} = 0 \} },
\end{equation*}
where the closure is taken in $\overline{ ^{\mathrm{co}}T^{\ast} \oplus \mathbb{R}^{n_{\tindex}} }[ \hat{X}^{\tindex} ; \{ 0 \} ]$.
 \par  
Locally, starting from coordinates $(\bbdf, x^{\tindex}, y_{\tindex}, y^{\tindex}, \tcoblt, \tcoblm, \tcobut, \tcobum)$, and also in a neighborhood of the fibers where $| ( \tcoblt, \tcoblm ) | \leq c | ( \tcobut, \tcobum ) |$, we can introduce projective coordinates of the form
\begin{equation*}
\frac{ \tcoblt }{ | ( \tcobut, \tcobum ) | }, \frac{ \tcoblm }{ | ( \tcobut, \tcobum ) | },  \frac{ \tcobut }{ | ( \tcobut, \tcobum ) | },   \frac{ \tcobum } {| ( \tcobut, \tcobum ) |},    \frac{1}{| ( \tcobut, \tcobum ) |},
\end{equation*}
such that if $\itccf$ denotes the natural inclusion
\begin{equation*}
\itccf : \mathcal{C}_{\tindex} \times \overline{ ^{\mathrm{co}}T^{\ast} \oplus \mathbb{R}^{n_{\tindex}} } [ \hat{X}^{\tindex} ; \{  0 \} ] \rightarrow \overline{^{\mathrm{3co}}T^{\ast}} \Xd
\end{equation*}
as determined by $(\ref{a simple identification of the sphere bundle at ff in the 3co sense})$, then $\itccf(  \mathcal{C}_{\tindex} \times { ^{\mathrm{co}}S^{\ast} }[ \hat{X}^{\tindex} ; \{ 0 \} ] ) ]$ can be written as 
\begin{equation} \label{three cone blow up where 1}
\left\{ x^{\tindex} = 0,  \frac{1}{ | ( \tcobut, \tcobum ) | } = 0,  \frac{ \tcoblt }{ |(\tcobut, \tcobum)| } = 0, \frac{ \tcoblm }{ |(\tcobut, \tcobum)| } = 0 \right\}.
\end{equation}
\par

Alternatively, starting from the coordinates $( \bdf, \hat{x}^{\tindex}, y_{\tindex}, y^{\tindex}, \tcosclt , \tcosclm , \tcoscut, \tcoscum )$, and in a neighborhood of the fibers where $| ( \tcosclt, \tcosclm ) | \leq c | ( \tcoscut, \tcoscum ) |$, we can introduce projective coordinates of the form
\begin{equation*}
\frac{ \tcosclt }{ | ( \tcoscut, \tcoscum ) | }, \frac{ \tcosclm }{ | ( \tcoscut, \tcoscum ) | }, \frac{ \tcoscut }{ | ( \tcoscut, \tcoscum ) | },   \frac{ \tcoscum }{| ( \tcoscut, \tcoscum ) |},  \frac{1}{| ( \tcoscut, \tcoscum ) |},
\end{equation*}
in which case we can also write $\itccf ( \mathcal{C}_{\tindex} \times { ^{\mathrm{co}}S^{\ast} }[ \hat{X}^{\tindex} ; \{ 0 \} ] ) $ as
\begin{equation} \label{three cone blow up where 2}
\left\{ \bdf = 0, \frac{1}{ | ( \tcoscut, \tcoscum ) |} = 0,  \frac{ \tcosclt }{ | ( \tcoscut, \tcoscum ) | } = 0,  \frac{ \tcosclm }{ | ( \tcoscut, \tcoscum ) | } = 0 \right\}.
\end{equation}
This concludes the description of $\itccf ( \mathcal{C}_{\tindex} \times { ^{\mathrm{co}}S^{\ast} }[ \hat{X}^{\tindex} ; \{ 0 \} ] )$ in local coordinates. \par
We now demonstrate the connection between the direct and converse perspectives.
\begin{proposition}
The identity map extends from the the interior to a diffeomorphism
\begin{equation} 
\label{second microlocal diffeomorphism}
\left[ \overline{^{\mathrm{d3sc}}T^{\ast}} \Xd ; \beta_{\mathrm{d3sc}}^{\ast} ( \overline{\tscf^{-1}( 
{\sco} )} ) \right] \cong 
\left[ \overline{ ^{\mathrm{3co}} T^{\ast}} \Xd  ; \itccf ( \mathcal{C}_{\tindex} \times {^{\mathrm{co}}S^{\ast} [ \hat{X}^{\tindex} ; \{ 0 \} ]} ) \right],
\end{equation}
where the lifts of $ { ^{\mathrm{d3sc}}S^{\ast} \Xd }$, $\overline{ ^{\mathrm{d3sc}}T^{\ast}}_{\dmf}\Xd$, $\overline{ ^{\mathrm{d3sc}}T^{\ast}}_{\dff}\Xd$ to the left hand side of (\ref{second microlocal diffeomorphism}) identify with the lifts of ${ ^{\mathrm{3co}}S^{\ast} \Xd }$, $\overline{ ^{\mathrm{3co}}T^{\ast} }_{\dmf} \Xd$, $\overline{ ^{\mathrm{3co}}T^{\ast} }_{\dff} \Xd$ to the right hand side of (\ref{second microlocal diffeomorphism}). Moreover, the front face for the left hand side of (\ref{second microlocal diffeomorphism}) is identified with the lift of $\overline{ ^{\mathrm{3co}}T^{\ast} }_{\cf} \Xd$ to the right hand side, while the front face for the right hand side of (\ref{second microlocal diffeomorphism}) is identified with the lift of $\overline{ ^{\mathrm{d3sc}}T^{\ast} }_{\cf} \Xd$ in the left hand side.
\end{proposition}
\begin{proof}
The proof of this result will follow directly from computing the necessary local coordinates on both manifolds, and then realizing that they can be identified. 
\par

However, before we proceed to the actual calculations, let us remark that the interiors of $\overline{ ^{\mathrm{d3sc}}T^{\ast} } \Xd$ and $\overline{ ^{\mathrm{3co}}T^{\ast} } \Xd$, both being $T^{\ast} \mathbb{R}^{n}$, are of course naturally identified. Moreover, the lifts of $\overline{^{\mathrm{d3sc}}T^{\ast}}_{\dmf^{\circ}} \Xd$ and $\overline{^{\mathrm{3co}}T^{\ast}}_{\dmf^{\circ}} \Xd$ are both naturally identified with $\overline{ ^{\mathrm{3sc}}T^{\ast} }_{\mf^{\circ}} X$, while the lifts of $\overline{ ^{\mathrm{d3sc}}T^{\ast}
_{\dff^{\circ}}} \Xd$ and $\overline{ ^{\mathrm{3co}}T^{\ast} }_{\dff^{\circ}} \Xd$ are both naturally identified with $\overline{ ^{\mathrm{3sc}}T^{\ast} }_{\ff^{\circ}}X$. The situation is thus simplified in such a way that we need only be concerned with the situation spatially in a neighborhood of $\cf$, which will be the focus below.  
\par

We will first work from the perspective of blowing up $ \overline{^{\mathrm{d3sc}}T^{\ast}} \Xd$ at $\beta_{\mathrm{d3sc}}^{\ast} \overline{( \pi_{\ff}^{-1} ( o_{\mathcal{C}^{\tindex}} ) )}$. In doing so, we will first derive coordinates on $ \overline{ ^{\mathrm{d3sc}}T^{\ast}} \Xd$ near $\overline{ ^{\mathrm{d3sc}}T^{\ast}}_{\cf} \Xd$. \par

This is easy, since by starting from the coordinates $( \bdf, x^{\tindex}, y_{\tindex}, y^{\tindex}, \tau_{\tindex}, \mu_{\tindex}, \tau^{\tindex}, \mu^{\tindex} )$, we can construct coordinates system on $\overline{ ^{\mathrm{d3sc}}T^{\ast} } \Xd$ by introducing projective coordinates about $\{ \bdf = 0 , x^{\tindex} = 0 \}$, from which we see that
\begin{equation} \label{d3sc scattering coordiantes 1}
\bbdf,   x^{\tindex},  y_{\tindex},  y^{\tindex},  \tau_{\tindex},   \mu_{\tindex}, \tau^{\tindex},  \mu^{\tindex} 
\end{equation}
are valid coordinates in a neighborhood of ${ ^{\mathrm{d3sc}}T^{\ast}_{\cf \cap \dff} } \Xd$. Here $\bbdf$ is a local defining function for ${ ^{\mathrm{d3sc}}T^{\ast}_{\dff}} \Xd$ and $x^{\tindex}$ is a local defining function for ${^{\mathrm{d3sc}}T^{\ast}_{\cf}} \Xd$. \par

To derive coordinates which are valid near ${^{\mathrm{d3sc}}T^{\ast}_{\cf \cap \dmf}}\Xd$, it would be better to start from the coordinates $( \bdf, x^{\tindex}, y_{\tindex}, y^{\tindex}, \tau_{\tindex}^{\mathrm{sf}} , \mu_{\tindex}, \tau^{\tindex}, \mu^{\tindex} )$ introduced in (\ref{three-body covector fields from the other direction}), where we have also used (\ref{OG tau}) and (\ref{revised notation 3co 1.5}) to simplify the notations. Then by introducing projective coordinates about $\{ \bdf = 0, x^{\tindex} = 0 \}$, we see that
\begin{equation} \label{d3sc scattering coordinates 2}
\bdf,  \hat{x}^{\tindex}, y_{\tindex},  y^{\tindex},  \tau_{\tindex}^{\mathrm{sf}}, \mu_{\tindex},   \tau^{\tindex},  \mu^{\tindex}
\end{equation}
are valid coordinates in a neighborhood of ${ ^{\mathrm{d3sc}}T^{\ast}_{\cf \cap \dmf} }\Xd$. Here $\hat{x}^{\tindex}$ is a defining function for ${ ^{\mathrm{d3sc}}T^{\ast}_{\dmf}} \Xd$ and $\bdf$ is a defining function for ${ ^{\mathrm{d3sc}}T^{\ast}_{\cf}} \Xd$. 
\par

Now, notice that $o_{\mathcal{C}^{\tindex}} \subset { ^{\mathrm{sc}}T^{\ast} X^{\tindex} }$  is given in local coordinates by
\begin{equation*}
o_{\mathcal{C}^{\tindex}} = \{ x^{\tindex} = 0, | (\tau^{\tindex}, \mu^{\tindex} )| = 0 \},
\end{equation*}
thus we locally have
\begin{equation} \label{where to blow up 1}
 \tscf^{-1}( o_{\mathcal{C}^{\tindex}} ) = \{ \bdf = x^{\tindex} = 0, |(\tau^{\tindex}, \mu^{\tindex} )| = 0 \},
\end{equation} 
where we recall that ${ ^{\mathrm{3sc}}\pi_{\ff} }$ is the projection defined in (\ref{three-body cotangent bundle natural projection}). Then by lifting (\ref{where to blow up 1}) further to $\overline{ ^{\mathrm{d3sc}}T^{\ast} } \Xd$, we have that 
\begin{gather}
\beta_{\mathrm{d3sc}}^{\ast}( \tscf^{-1} ( o_{\mathcal{C}^{\tindex}} ) ) = \{ x^{\tindex} = 0, \tau^{\tindex} = 0, \mu^{\tindex} = 0 \}, \ \text{resp} \label{first lifted zero section} \\
\beta_{\mathrm{d3sc}}^{\ast}( \tscf^{-1} ( o_{\mathcal{C}^{\tindex}} ) ) = \{ \bdf = 0 , \tau^{\tindex} = 0, \mu^{\tindex} = 0 \}. \label{second lifted zero section}
\end{gather}
in coordinates (\ref{d3sc scattering coordiantes 1}) resp (\ref{d3sc scattering coordinates 2}). 
\par 
We proceed to introduce projective coordinates about (\ref{first lifted zero section}). By the standard abuse of notations, this will produce coordinates of the forms
\begin{gather} 
 \bbdf, x^{\tindex}, y_{\tindex},  y^{\tindex},  \tau_{\tindex},  \mu_{\tindex},  \frac{\tau^{\tindex}}{x^{\tindex}},  \frac{\mu^{\tindex}}{x^{\tindex}}, \tag{D.1} \label{D.1} \\ 
\bbdf, \frac{x^{\tindex}}{| ( \tau^{\tindex}, \mu^{\tindex} ) |}, y_{\tindex}, y^{\tindex}, \tau_{\tindex}, \mu_{\tindex}, \frac{\tau^{\tindex}}{| ( \tau^{\tindex}, \mu^{\tindex} ) |}, \frac{\mu^{\tindex}}{| ( \tau^{\tindex}, \mu^{\tindex} ) |},  | ( \tau^{\tindex}, \mu^{\tindex} ) |. 
\tag{D.2} \label{D.2}
\end{gather}
On the other hand, introducing projective coordinates about (\ref{second lifted zero section}) will produce coordinates of the forms
\begin{gather}
 \bdf, \hat{x}^{\tindex}, y_{\tindex},  y^{\tindex},  \tau_{\tindex}^{\mathrm{sf}},  \mu_{\tindex},  \frac{\tau^{\tindex}}{ \bdf },  \frac{\mu^{\tindex}}{ \bdf }, \tag{D.3}  \label{D.3} \\ 
\frac{ \bdf }{| ( \tau^{\tindex}, \mu^{\tindex} ) |}, \hat{x}^{\tindex}, y_{\tindex}, y^{\tindex}, \tau_{\tindex}^{\mathrm{sf}}, \mu_{\tindex}, \frac{\tau^{\tindex}}{| ( \tau^{\tindex}, \mu^{\tindex} ) |}, \frac{\mu^{\tindex}}{| ( \tau^{\tindex}, \mu^{\tindex} ) |}, |( \tau^{\tindex}, \mu^{\tindex} )|. \tag{D.4} \label{D.4}
\end{gather} \par 
Next, we will consider coordinates near the fiber infinity. Starting from the coordinates system $( \bdf, x^{\tindex}, y_{\tindex}, y^{\tindex}, \tau_{\tindex}, \mu_{\tindex}, \tau^{\tindex}, \mu^{\tindex} )$ once again, and assume that $|( \tau_{\tindex}, \mu_{\tindex} )|$ is dominating, then there are natural coordinates on $\overline{ ^{\mathrm{3sc}}T^{\ast}}X$ of the form
\begin{equation*}
\bdf, x^{\tindex},  y_{\tindex},  y^{\tindex},  \frac{1}{|( \tau_{\tindex}, \mu_{\tindex} )|}, \frac{\tau_{\tindex} }{|( \tau_{\tindex}, \mu_{\tindex} )|},  \frac{\mu_{\tindex} }{|( \tau_{\tindex}, \mu_{\tindex} )|},  \frac{\tau^{\tindex} }{|( \tau_{\tindex}, \mu_{\tindex} )|},  \frac{\mu^{\tindex} }{|( \tau_{\tindex}, \mu_{\tindex} )|},
\end{equation*}
where $|( \tau_{\tindex}, \mu_{\tindex} )|^{-1}$ defines fiber infinity. Thus, coordinates on $\overline{ ^{\mathrm{d3sc}}T^{\ast}}\Xd$ near $\overline{ ^{\mathrm{d3sc}}T^{\ast} }_{\cf} \Xd$ can be found by introducing projective coordinates about $\{ \bdf = 0 , x^{\tindex} = 0 \}$. It follows that
\begin{equation} \label{d3sc scattering coordiantes 3}
\bbdf, x^{\tindex},  y_{\tindex},  y^{\tindex},  \frac{1}{|( \tau_{\tindex}, \mu_{\tindex} )|}, \frac{\tau_{\tindex} }{|( \tau_{\tindex}, \mu_{\tindex} )|},  \frac{\mu_{\tindex} }{|( \tau_{\tindex}, \mu_{\tindex} )|},  \frac{\tau^{\tindex} }{|( \tau_{\tindex}, \mu_{\tindex} )|},  \frac{\mu^{\tindex} }{|( \tau_{\tindex}, \mu_{\tindex} )|}
\end{equation}
are valid coordinates in a neighborhood of $\overline{ ^{\mathrm{d3sc}}T^{\ast}}_{\cf \cap \dff} \Xd$. Here $\bbdf$ is a local defining function for $\overline{ ^{\mathrm{d3sc}}T^{\ast}}_{\dff} \Xd$ and $x^{\tindex}$ is a local defining function for $\overline{ ^{\mathrm{d3sc}}T^{\ast}}_{\cf} \Xd$. \par

Alternatively, if we instead start from the coordinates system $( \bdf, x^{\tindex}, y_{\tindex}, y^{\tindex}, \tau_{\tindex}^{\mathrm{sf}}, \mu_{\tindex}, \tau^{\tindex}, \mu^{\tindex} )$, and suppose that $|( \tau_{\tindex}^{\mathrm{sf}}, \mu_{\tindex} )|$ is dominating, then there are natural coordinates on $\overline{ ^{\mathrm{3sc}}T^{\ast} }X$ of the form
\begin{equation*} 
\bdf, x^{\tindex},  y_{\tindex},  y^{\tindex},  \frac{1}{|( \gtau , \mu_{\tindex} )|}, \frac{\gtau }{|( \gtau, \mu_{\tindex} )|},  \frac{\mu_{\tindex} }{|( \gtau, \mu_{\tindex} )|},  \frac{\tau^{\tindex} }{|( \gtau, \mu_{\tindex} )|},  \frac{\mu^{\tindex} }{|( \gtau, \mu_{\tindex} )|},
\end{equation*}
where $|( \gtau, \mu_{\tindex} )|^{-1}$ defines fiber infinity. By introducing projective coordinates about $\{ \bdf = 0, x^{\tindex} = 0 \}$, we find that
\begin{equation} \label{d3sc scattering coordiantes 4}
\bdf, \hat{x}^{\tindex},  y_{\tindex},  y^{\tindex},  \frac{1}{|( \gtau , \mu_{\tindex} )|}, \frac{\gtau }{|( \gtau, \mu_{\tindex} )|},  \frac{\mu_{\tindex} }{|( \gtau, \mu_{\tindex} )|},  \frac{\tau^{\tindex} }{|( \gtau, \mu_{\tindex} )|},  \frac{\mu^{\tindex} }{|( \gtau, \mu_{\tindex} )|}.
\end{equation}
are valid coordinates in a neighborhood of $\overline{ ^{\mathrm{d3sc}}T^{\ast} }_{\cf \cap \dmf}$. Here $\hat{x}^{\tindex}$ is a local defining function for $\overline{ ^{\mathrm{d3sc}}T^{\ast} }_{\dmf} \Xd$ and $\bdf$ is a local defining function for $\overline{ ^{\mathrm{d3sc}}T^{\ast} }_{\cf} \Xd$. 
\par

Locally, it is easy to see that
\begin{gather*}
\overline{ \tscf^{-1}( o_{\mathcal{C}^{\tindex}} ) } = \left\{ \bdf = x^{\tindex} = 0, \frac{ \tau^{\tindex} }{|( \tau_{\tindex}, \mu_{\tindex} )|} = 0, \frac{ \mu^{\tindex} }{|( \tau_{\tindex}, \mu_{\tindex} )|} = 0 \right\}, \ \text{resp} \\
\overline{ \tscf^{-1}( o_{\mathcal{C}^{\tindex}} ) } = \left\{ \bdf = x^{\tindex} = 0, \frac{ \gtau }{|( \tau_{\tindex}, \mu_{\tindex} )|} = 0, \frac{ \mu^{\tindex} }{|( \gtau, \mu_{\tindex} )|} = 0 \right\}
\end{gather*}
in coordinates (\ref{d3sc scattering coordiantes 3}) resp (\ref{d3sc scattering coordiantes 4}). Thus we also have
\begin{gather}
\beta_{\mathrm{d3sc}}^{\ast} ( \overline{ \tscf^{-1}( o_{\mathcal{C}^{\tindex}} ) } ) = \left\{ x^{\tindex}  = 0, \frac{ \tau^{\tindex} }{|( \tau_{\tindex}, \mu_{\tindex} )|} = 0, \frac{ \mu^{\tindex} }{|( \tau_{\tindex}, \mu_{\tindex} )|} = 0  \right\}, \ \text{resp} \label{third lifted zero section} \\
\beta_{\mathrm{d3sc}}^{\ast} ( \overline{ \tscf^{-1}( o_{\mathcal{C}^{\tindex}} ) } ) = \left\{ \bdf  = 0, \frac{ \tau^{\tindex} }{|( \gtau, \mu_{\tindex} )|} = 0, \frac{ \mu^{\tindex} }{|( \gtau, \mu_{\tindex} )|} = 0  \right\} \label{fourth lifted zero section}
\end{gather}
in coordinates (\ref{d3sc scattering coordiantes 3}) resp (\ref{d3sc scattering coordiantes 4}). \par

We proceed to introduce projective coordinates about (\ref{third lifted zero section}). Then by the standard abuse of notations once again, this will produce coordinates of the form 
\begin{equation}
\bbdf, x^{\tindex}, y_{\tindex}, y^{\tindex}, \frac{1}{|( \tau_{\tindex}, \mu_{\tindex} )|}, \frac{\tau_{\tindex} }{|( \tau_{\tindex}, \mu_{\tindex} )|},  \frac{\mu_{\tindex} }{|( \tau_{\tindex}, \mu_{\tindex} )|}, \frac{\tau^{\tindex} }{ x^{\tindex} |( \tau_{\tindex}, \mu_{\tindex} )|} ,  \frac{\mu^{\tindex} }{ x^{\tindex} |( \tau_{\tindex}, \mu_{\tindex} )|}, \tag{D.5} \label{D.5} 
\end{equation}
and
\begin{equation}
\begin{gathered}
\bbdf, \frac{x^{\tindex} | ( \tau_{\tindex}, \mu_{\tindex} ) | }{ | ( \tau^{\tindex}, \mu^{\tindex} ) | }, y_{\tindex}, y^{\tindex}, \\
\frac{1}{|( \tau_{\tindex}, \mu_{\tindex} )|}, \frac{\tau_{\tindex}}{|(\tau_{\tindex}, \mu_{\tindex})|}, \frac{\mu_{\tindex}}{|(\tau_{\tindex},\mu_{\tindex})|}, \frac{\tau^{\tindex}}{|(\tau^{\tindex}, \mu^{\tindex})|}, \frac{\mu^{\tindex}}{|(\tau^{\tindex}, \mu^{\tindex})|}, \frac{|(\tau^{\tindex}, \mu^{\tindex})|}{|(\tau_{\tindex}, \mu_{\tindex})|}. \tag{D.6} \label{D.6}
\end{gathered}
\end{equation}
On the other hand, introducing projective coordinates about (\ref{fourth lifted zero section}) will produce coordinates of the forms
\begin{equation}
\bdf, \hat{x}^{\tindex}, y_{\tindex}, y^{\tindex}, \frac{1}{|( \gtau, \mu_{\tindex} )|}, \frac{\gtau }{|( \gtau, \mu_{\tindex} )|},  \frac{\mu_{\tindex} }{|( \gtau, \mu_{\tindex} )|}, \frac{\tau^{\tindex} }{ \bdf |( \gtau, \mu_{\tindex} )|} ,  \frac{\mu^{\tindex} }{ \bdf |( \gtau, \mu_{\tindex} )|}, \tag{D.7} \label{D.7} 
\end{equation}
and 
\begin{equation}
\begin{gathered}
\frac{ \bdf | ( \gtau, \mu_{\tindex} ) | }{ | ( \tau^{\tindex},  \mu^{\tindex} ) | }, \hat{x}^{\tindex}, y_{\tindex}, y^{\tindex}, \\
\frac{1}{|( \gtau, \mu_{\tindex} )|}, \frac{\gtau}{|(\gtau, \mu_{\tindex})|}, \frac{\mu_{\tindex}}{|(\gtau,\mu_{\tindex})|}, \frac{\tau^{\tindex}}{|(\tau^{\tindex}, \mu^{\tindex})|}, \frac{\mu^{\tindex}}{|(\tau^{\tindex}, \mu^{\tindex})|}, \frac{|(\tau^{\tindex}, \mu^{\tindex})|}{|(\gtau, \mu_{\tindex})|}. \tag{D.8} \label{D.8}
 \end{gathered}
\end{equation} \par

Alternatively, suppose we start from the coordinates systems $( \bbdf, x^{\tindex}, y_{\tindex}, y^{\tindex}, \tau_{\tindex}, \mu_{\tindex}, \tau^{\tindex}, \mu^{\tindex} )$ or $( \bdf, \hat{x}^{\tindex}, y_{\tindex}, y^{\tindex}, \gtau, \mu_{\tindex}, \tau^{\tindex}, \mu^{\tindex} )$, then we can respectively consider regions near the fiber infinity where $|(\tau^{\tindex}, \mu^{\tindex})|$ is dominating. In particular we must be away from $\beta_{\mathrm{d3sc}}^{\ast}( \overline{ \tscf^{-1}( o_{\mathcal{C}^{\tindex}} ) } )$, and so the coordinates systems
\begin{gather} \tag{D.9} \label{D.9}
\bbdf, x^{\tindex},  y_{\tindex},  y^{\tindex},  \frac{1}{|( \tau^{\tindex}, \mu^{\tindex} )|}, \frac{\tau_{\tindex} }{|( \tau^{\tindex}, \mu^{\tindex} )|},  \frac{\mu_{\tindex} }{|( \tau^{\tindex}, \mu^{\tindex} )|},  \frac{\tau^{\tindex} }{|( \tau^{\tindex}, \mu^{\tindex} )|},  \frac{\mu^{\tindex} }{|( \tau^{\tindex}, \mu^{\tindex} )|},  \\
\bdf, \hat{x}^{\tindex},  y_{\tindex},  y^{\tindex},  \frac{1}{|( \tau^{\tindex}, \mu^{\tindex} )|}, \frac{\gtau }{|( \tau^{\tindex}, \mu^{\tindex} )|},  \frac{\mu_{\tindex} }{|( \tau^{\tindex}, \mu^{\tindex} )|},  \frac{\tau^{\tindex} }{|( \tau^{\tindex}, \mu^{\tindex} )|},  \frac{\mu^{\tindex} }{|( \tau^{\tindex}, \mu^{\tindex} )|} \tag{D.10} \label{D.10}
\end{gather}
are equally valid on the blow-up as well.

\par 
This gives a complete description of the coordinates systems for the left hand side of (\ref{second microlocal diffeomorphism}) in a neighborhood of the lift of $\overline{ ^{\mathrm{d3sc}}T^{\ast}}_{\cf}\Xd$.
\par 
Next, we will switch to the perspective of blowing up $\overline{ ^{\mathrm{3co}}T^{\ast} }\Xd$ at $\itccf ( \mathcal{C}_{\tindex} \times {^{\mathrm{co}}S^{\ast} [ \hat{X}^{\tindex} ; \{ 0 \} ]} )$. Recall that the latter set has already been described in local coordinates. However, for the proof of this proposition, when working in a neighborhood of $\overline{ ^{\mathrm{3co}}T^{\ast} }_{\cf \cap \dff} \Xd$, it would actually be more convenient to start with the coordinates
\begin{equation} \tag{C.1} \label{C.1}
\bbdf, x^{\tindex}, y_{\tindex}, y^{\tindex}, \tscblt, \tscblm, \tscbut, \tscbum
\end{equation}
instead. The procedure is then the same: at regions of relevance, we can change the fiber coordinates to 
\begin{equation} \label{phase spaces are the same fiber change}
\frac{ \tscblt }{ | ( \tscbut, \tscbum ) | },  \frac{ \tscblm }{ | ( \tscbut, \tscbum ) | },  \frac{ \tscbut }{ | ( \tscbut, \tscbum ) | },  \frac{ \tscbum }{| ( \tscbut, \tscbum ) |},  \frac{1}{| ( \tscbut, \tscbum ) |},
\end{equation}
and note that $\itccf ( \mathcal{C}_{\tindex} \times {^{\mathrm{co}}S^{\ast} [ \hat{X}^{\tindex} ; \{ 0 \} ]} )$ is equivalently given in these coordinates by
\begin{equation} \label{where to blow up fiber infinity 3sc b}
\Big\{ x^{\tindex} = 0,  \frac{ \tscblt }{ |(\tscbut, \tscbum )| } = 0,  \frac{ \tscblm }{ |(\tscbut, \tscbum )| } = 0, \frac{1}{ | ( \tscbut, \tscbum ) | } = 0 \Big\} .
\end{equation} \par

Likewise, when working in a neighborhood of $\overline{ ^{\mathrm{3co}}T^{\ast} }_{\cf \cap \dmf} \Xd$, it would be more convenient to start with the coordinates
\begin{equation} \tag{C.2} \label{C.2}
\bdf, \hat{x}^{\tindex}, y_{\tindex}, y^{\tindex}, \tscreslt, \tscreslm, \tscresut, \tscresum ,
\end{equation}
after which we can make the same fiber variables change to
\begin{equation} \label{fiber variables change 3sc res}
\frac{ \tscreslt }{ | ( \tscresut, \tscresum ) | },  \frac{ \tscreslm }{ | ( \tscresut, \tscresum ) | },  \frac{ \tscresut }{ | ( \tscresut, \tscresum ) | },  \frac{ \tscresum }{| ( \tscresut, \tscresum ) |},  \frac{1}{| ( \tscresut, \tscresum ) |}.
\end{equation}
Then $\itccf ( \mathcal{C}_{\tindex} \times {^{\mathrm{co}}S^{\ast} [ \hat{X}^{\tindex} ; \{ 0 \} ]} )$ is written in (\ref{fiber variables change 3sc res}) as 
\begin{equation} \label{where to blow up fiber infinity 3sc res}
\Big\{ \bdf = 0,  \frac{ \tscreslt }{ |(\tscresut, \tscresum )| } = 0,  \frac{ \tscreslm }{ |(\tscresut, \tscresum )| } = 0,  \frac{1}{ | ( \tscresut, \tscresum ) | } = 0 \Big\} .
\end{equation}
\par 
It follows from introducing projective coordinates about (\ref{where to blow up fiber infinity 3sc b}) that there are coordinates of the forms 
\begin{gather}
\bbdf, x^{\tindex}, y_{\tindex}, y^{\tindex}, \frac{\tscblt}{ x^{\tindex} | ( \tscbut, \tscbum ) |}, \frac{\tscblm}{ x^{\tindex} | ( \tscbut, \tscbum ) | }, \frac{\tscbut}{| ( \tscbut, \tscbum ) |}, \frac{\tscbum}{| ( \tscbut, \tscbum ) |}, \frac{1}{x^{\tindex} | ( \tscbut, \tscbum ) |}, \tag{C.3} \label{C.3} \\
\bbdf, x^{\tindex} | ( \tscbut, \tscbum ) |, y_{\tindex}, y^{\tindex},  \tscblt, \tscblm, \frac{\tscbut}{| ( \tscbut, \tscbum ) |}, \frac{\tscbum}{| ( \tscbut, \tscbum ) |}, \frac{1}{| ( \tscbut, \tscbum ) |}, \tag{C.4} \label{C.4} 
\end{gather}
and 
\begin{equation}
\begin{gathered}
\bbdf, \frac{x^{\tindex} | ( \tscbut, \tscbum ) | }{ | ( \tscblt, \tscblm ) | }, y_{\tindex}, y^{\tindex}, \\
\frac{\tscblt}{| ( \tscblt, \tscblm ) |}, \frac{\tscblm}{| ( \tscblt, \tscblm ) |}, \frac{|( \tscblt, \tscblm )|}{ | ( \tscbut, \tscbum ) | }, \frac{\tscbut}{| ( \tscbut, \tscbum ) |}, \frac{\tscbum}{| ( \tscbut, \tscbum ) |}, \frac{1}{|(\tscblt, \tscblm)|}. \tag{C.5} \label{C.5}
\end{gathered}
\end{equation} 
On the other hand, by introducing projective coordinates about (\ref{where to blow up fiber infinity 3sc res}), we find that there are coordinates of the forms
\begin{equation}
\begin{gathered}
\bdf, \hat{x}^{\tindex}, y_{\tindex}, y^{\tindex}, \\ \frac{\tscreslt}{ \bdf | ( \tscresut, \tscresum ) |}, \frac{\tscreslm}{ \bdf | ( \tscresut, \tscresum ) | }, \frac{\tscresut}{| ( \tscresut, \tscresum ) |}, \frac{\tscresum}{| ( \tscresut, \tscresum ) |}, \frac{1}{ \bdf | ( \tscresut, \tscresum ) |}, \tag{C.6} \label{C.6} 
\end{gathered}
\end{equation}
and 
\begin{equation}
\bdf | ( \tscresut, \tscresum ) |, \hat{x}^{\tindex}, y_{\tindex}, y^{\tindex},  \tscreslt, \tscreslm, \frac{\tscresut}{| ( \tscresut, \tscresum ) |}, \frac{\tscresum}{| ( \tscresut, \tscresum ) |}, \frac{1}{| ( \tscresut, \tscresum ) |}, \tag{C.7} \label{C.7}
\end{equation}
as well as 
\begin{equation}  \tag{C.8} \label{C.8}
\begin{gathered}
\frac{ \bdf | ( \tscresut, \tscresum ) | }{ | ( \tscreslt, \tscreslm ) | }, \hat{x}^{\tindex},  y_{\tindex}, y^{\tindex}, \\
\frac{\tscreslt}{| ( \tscreslt, \tscreslm ) |}, \frac{\tscreslm}{| ( \tscreslt, \tscreslm ) |}, \frac{|( \tscreslt, \tscreslm )|}{ | ( \tscresut, \tscresum ) | }, \frac{\tscresut}{| ( \tscresut, \tscresum ) |}, \frac{\tscresum}{| ( \tscresut, \tscresum ) |}, \frac{1}{|(\tscreslt, \tscreslm)|}.
\end{gathered}
\end{equation} \par

Alternatively, if we start from the coordinates systems $( \bbdf, x^{\tindex}, y_{\tindex}, y^{\tindex}, \tscblt, \tscblm, \tscbut, \tscbum )$ or $( \bdf, \hat{x}^{\tindex}, y_{\tindex}, y^{\tindex}, \tscreslt, \tscreslm, \tscresut, \tscresum )$, then we can respectively consider regions near the fiber infinity where $|( \tscblt, \tscblm )|$ or $|( \tscreslt, \tscreslm )|$ dominates. Since in such regions, we must be away from $\itccf ( \mathcal{C}_{\tindex} \times {^{\mathrm{co}}S^{\ast} [ \hat{X}^{\tindex} ; \{ 0 \} ]} )$, the coordinates systems
\begin{gather}
\bbdf, x^{\tindex}, y_{\tindex}, y^{\tindex}, \frac{ \tscblt }{ | ( \tscblt, \tscblm ) | },  \frac{ \tscblm }{ | ( \tscblt, \tscblm ) | },  \frac{ \tscbut }{ | ( \tscblt, \tscblm ) | },  \frac{ \tscbum }{| ( \tscblt, \tscblm ) |},  \frac{1}{| ( \tscblt, \tscblm ) |},   \tag{C.9} \label{C.9} \\
\bdf, \hat{x}^{\tindex}, y_{\tindex}, y^{\tindex}, \frac{ \tscreslt }{ | ( \tscreslt, \tscreslm )  | },  \frac{ \tscreslm }{ | ( \tscreslt, \tscreslm )  | },  \frac{ \tscresut }{ | ( \tscreslt, \tscreslm )  | },  \frac{ \tscresum }{| ( \tscreslt, \tscreslm )  |},  \frac{1}{| ( \tscreslt, \tscreslm ) |} \tag{C.10} \label{C.10}
\end{gather}
are equality valid on the blow-up as well. \par

This gives a complete description of the coordinates systems for the right hand side of (\ref{second microlocal diffeomorphism}) in a neighborhood of the lift of $\overline{ ^{\mathrm{3co}}T^{\ast}}_{\cf}\Xd$. \par

We now proceed to connect coordinates (\ref{D.1})--(\ref{D.10}) and (\ref{C.1})--(\ref{C.10}) together. In fact, all we need is to making the scalings
\begin{gather*}
 \tscbut = \frac{\tau^{\tindex}}{x^{\tindex}} , \,  \tscbum = \frac{\mu^{\tindex}}{x^{\tindex}}, \quad \tscresut = \frac{\tau^{\tindex}}{\bdf}, \, \tscresum = \frac{\mu^{\tindex}}{\bdf},
\end{gather*}
which were also done in \S \ref{subsection definition of the three-cone bundle and the variable changes}, in particular (\ref{revised notation 3sc b 2}) and (\ref{revised notation 3sc res 2}). Then it is easy to see that the aforementioned coordinate systems correspond as follows:
\begin{gather*}
(\mathrm{\ref{D.1}}) \leftrightarrow (\mathrm{\ref{C.1}}), \ (\mathrm{\ref{D.2}}) \leftrightarrow (\mathrm{\ref{C.4}}), \ (\mathrm{\ref{D.3}}) \leftrightarrow (\mathrm{\ref{C.2}}), \ (\mathrm{\ref{D.4}}) \leftrightarrow (\mathrm{\ref{C.7}}), \ (\mathrm{\ref{D.5}}) \leftrightarrow (\mathrm{\ref{C.9}}), \\
(\mathrm{\ref{D.6}}) \leftrightarrow (\mathrm{\ref{C.5}}), \ (\mathrm{\ref{D.7}}) \leftrightarrow (\mathrm{\ref{C.10}}), \ (\mathrm{\ref{D.8}}) \leftrightarrow (\mathrm{\ref{C.8}}), \ (\mathrm{\ref{D.9}}) \leftrightarrow (\mathrm{\ref{C.3}}), \ (\mathrm{\ref{D.10}}) \leftrightarrow (\mathrm{\ref{C.6}}).
\end{gather*} \par

This concludes the proof of the proposition.
\end{proof}
Fro brevity, henceforth we will write
\begin{equation*}
\overline{^{\mathrm{d3sc,3co}}T^{\ast}} \Xd
\end{equation*}
to denote (\ref{second microlocal diffeomorphism}). Here, we emphasize that $\overline{^{\mathrm{d3sc,3co}}T^{\ast}} \Xd$ is no longer a fiber bundle in spite of the suggestive (and indeed absuive) notation. We will also write 
\begin{equation*}
\text{$^{\mathrm{d3sc,3co}}S^{\ast} \Xd$, $\overline{ ^{\mathrm{d3sc,3co}}T^{\ast}}_{\dmf} \Xd$, $ \overline{ ^{\mathrm{d3sc,3co}}T^{\ast}}_{\dff} \Xd$, $\dtsccf$, $\tcocf$}
\end{equation*}
respectively for the lifts of either
\begin{equation*}
\text{${ ^{\mathrm{d3sc}}S^{\ast} }X$, $\overline{ ^{\mathrm{d3sc}}T^{\ast} }_{\dmf}X$, $\overline{^{\mathrm{d3sc}}T^{\ast}}_{\dff}X$, $\beta_{\mathrm{d3sc}}^{\ast}( \overline{{ ^{\mathrm{3sc}}\pi_{\ff}^{-1} }( o_{\mathcal{C}_{\tindex}} )} )$, $\overline{ ^{\mathrm{d3sc}}T^{\ast} }_{\cf} X$}
\end{equation*}
to the left hand side of (\ref{second microlocal diffeomorphism}), or equivalently
\begin{equation*}
\text{${ ^{\mathrm{3co}}S^{\ast} }X$, $\overline{ ^{\mathrm{3co}}T^{\ast} }_{\dmf}X$, $\overline{^{\mathrm{3co}}T^{\ast}}_{\dff}X$, $\overline{ ^{\mathrm{3co}}T^{\ast} }_{\cf} X$, $ \itccf ( \mathcal{C}_{\tindex} \times {^{\mathrm{co}}S^{\ast} [ \hat{X}^{\tindex} ; \{ 0 \} ]} )$ }
\end{equation*}
to the right hand side of (\ref{second microlocal diffeomorphism}). We remark that the notations $\dtsccf$ and $\tcocf$ are so chosen since they can both be identified as the spatial restriction of a bundle at $\cf$ (in particular, the notation `$\overline{^{\mathrm{d3sc,3co}}T^{\ast}}_{\cf}X$' alone would not be sufficient).

\subsection{Construction of the second microlocalized operators} 
\label{subsection definition of the second microlocalized algebra}
Thus far, we have shown how one can relate $\Psi_{\mathrm{d3scc}}^{m,r,l,\nu}( \Xd )$ and $\Psi_{\mathrm{3coc}}^{m,r,l,b}( \Xd )$ at the phase space level through a blow-up. The resulting manifold (\ref{second microlocal diffeomorphism}) is then on the one hand, the required phase space for the second microlocalized operators (\ref{formally written decoupled second microlocalized operators}), which also leads to the definition of (\ref{the OG second microlocalized algebra}) through (\ref{formally written decoupled second microlocalized operators 1}); and on the other hand, an easily manageable resolution of the three-cone cotangent bundle. \par

This suggests that the operators in (\ref{formally written decoupled second microlocalized operators}) should have features of both the decoupled three-body structure as well as the conormal three-cone structure{\ep}much like how the operators in $\Psi_{\mathrm{sc,b}}^{m,r,l}( \overline{\mathbb{R}^{n}} )$, as defined in \S \ref{subsection Vasy's second microlocalized calculus}, have features of both the scattering structure as well as the conormal b-structure. \par

In this subsection, we will construct the operators (\ref{formally written decoupled second microlocalized operators}), and subsequently (\ref{the OG second microlocalized algebra}) through (\ref{formally written decoupled second microlocalized operators 1}) as well. Somewhat tautologically, and indeed following the convention established in the two-body setting \cite{AndrasSM}, we will henceforth write
\begin{equation} 
\label{deliberate mixed structure notation d3sc 3co}
\Psi_{\mathrm{d3sc,2}}^{m,r,l,\nu, b} \left( \Xd ; \beta_{\mathrm{d3sc}}^{\ast}(\overline{\tscf^{-1}(\sco)}) \right) =  \Psi_{\mathrm{d3sc,3co}}^{m,r,l,\nu,b} (X),
\end{equation}
where the subscript `d3sc,3co' on the right hand side of (\ref{deliberate mixed structure notation d3sc 3co}) reflects the aforementioned, mixed structure of the operators.
\par 

Recall from \S\S \ref{construction of three-cone overview}--\ref{subsection construction of cf} the construction of the conormal three-cone operators $\Psi^{m,r,l,b}_{\mathrm{3coc}}(\Xd)$. We will now construct (\ref{deliberate mixed structure notation d3sc 3co}) essentially by the same procedure. Such a construction will also be analogous to the construction of $\Psi^{m,r,l}_{\mathrm{sc,b}}( \overline{\mathbb{R}^{n}} )$ in \S \ref{subsection Vasy's second microlocalized calculus}, though the latter follows from the construction of $\Psi_{\mathrm{bc}}^{m,l}( \overline{\mathbb{R}^{n}} )$ instead. \par

To this end, recall that elements $\tilde{A} \in \Psi_{\mathrm{3coc}}^{m,r,l,b}(X)$ are continuous linear maps $\tilde{A} : \mathcal{S}( \mathbb{R}^{n} ) \rightarrow \mathcal{S}( \mathbb{R}^{n} )$ such that $\tilde{A}$ is determined by operators of the forms $\psi \tilde{A} \psi$, $\phi \tilde{A} \psi$, where $\phi, \psi \in \mathcal{C}^{\infty}(X)$ are cut-off functions whose supports are disjoint. Moreover, only the terms $\psi \tilde{A} \psi$ are determined by quantizations of elements in $S^{m,r,l,b}( \overline{ ^{\mathrm{3co}}T^{\ast} }X )$. As mentioned above, we will construct elements $A \in \Psi_{\mathrm{d3sc,3co}}^{m,r,l,\nu,b}(X)$ via the same procedure: We will require that each $A$ be a continuous linear map $\mathcal{S}( \mathbb{R}^{n} ) \rightarrow \mathcal{S}( \mathbb{R}^{n} )$, then we will specify operators of the forms $\psi A \psi$, $\phi A \psi$, where $\phi, \psi \in \mathcal{C}^{\infty}(X)$ are again cut-off functions with disjoint supports. We will always identify $A$ by its Schwartz kernel, which we view as some section of $^{\mathrm{3co}}\Omega_{R}X$. \par

In fact, in view of (\ref{second microlocal diffeomorphism}), we will only modify the specifications of $\psi A \psi$ (where quantizations of symbols are involved) when $\psi$ is supported in a small neighborhood of $\cf$. The remaining terms will be defined as if $A$ were an element of $\Psi_{\mathrm{3coc}}^{m,r,l,b}(X)$. \par

First, if $\psi$ is instead supported away from $\cf$, then since the right hand side of (\ref{second microlocal diffeomorphism}) is a resolution of $\overline{ ^{\mathrm{3co}}T^{\ast} }X$ only at $\overline{ ^{\mathrm{3co}}T^{\ast} }_{\cf}X$, we will still require that $\psi A \psi \in \Psi_{\mathrm{3coc}}^{m,r,l,-\infty}(X)$, or indeed more explicitly $\psi A \psi \in \Psi^{m,r, l}_{\mathrm{3scc}}( [ \overline{\mathbb{R}^{n}} ; \mathcal{C}_{\tindex} ] )$. \par

Next, we will also require that $\phi A \psi \in \Psi_{\mathrm{3coc}}^{-\infty, -\infty, l, b}(X)$. Equivalently, $\phi A \psi$ must have a smooth kernel $K_{\phi \psi}$, such that we have the following cases:
\begin{itemize}
\item If $\phi$, $\psi$ are as in the last three terms of (\ref{standard operator decomposition using partition of unity}), then $K_{\phi \psi}$ will also be rapidly decreasing;
\item If $\supp \phi$, $\supp \psi$ are contained in a small neighborhood of $\dff$, then depending on the precise regions in which $\supp \phi$, $\supp \psi$ are contained in, $K_{\phi \psi}$ must satisfy the characterizing estimates in one of the cases (1)--(3) of \S \ref{subsection construction of dff};
\item If $\supp \phi$, $\supp \psi$ are contained in a small neighborhood of $\cf$, then depending on the precise regions in which $\supp \phi$, $\supp \psi$ are contained in, $K_{\phi \psi}$ must satisfy the characterizing estimates in one of the cases (1)--(6) of \S \ref{subsection construction of cf}.
\end{itemize} \par 

Now, if $\supp \psi$ intersects $\cf \cap \dff$, then we can define $\psi A \psi$ in two (equivalent) ways:
\begin{itemize}
\item Following (\ref{the 3scb quantization written in terms of t}), we can define 
\begin{equation}
\label{d3sc3co 3cob}
\psi A \psi \coloneq B_{\psi, \mathrm{3co,b}} + R_{\psi, \mathrm{3co,b}}.
\end{equation}
Here $B_{\psi,\mathrm{3co,b}}$ is again given by the quantization (\ref{the 3scb quantization written in terms of t}), except its symbol $b_{\psi, \mathrm{3co,b}}$ must now be an element of $S^{m,-\infty,l,\nu,b}( \overline{ ^{\mathrm{d3sc,3co}}T^{\ast} } X)$. The kernel of $R_{\psi,\mathrm{3co,b}}$ is required to be smooth, and moreover must satisfy estimates (\ref{3scb part near diagonal off diagonal part}).
\item Following (\ref{on-diagonal part focuing on cf}), we can define
\begin{equation}
\label{d3sc3co 3cocob}
\psi A \psi \coloneq B_{\psi, \mathrm{3co,co,b}} + R_{\psi, \mathrm{3co,co,b}}.
\end{equation}
Here $B_{\psi, \mathrm{3co,co,b}}$ is again given by the quantization (\ref{the 3cob quantization written in terms of t}), except its symbol $b_{\psi,\mathrm{3co,co,b}}$ must now be an element of $S^{m, -\infty, l, \nu, b}( \overline{ ^{\mathrm{d3sc,3co}}T^{\ast} } X )$. The kernel of $R_{\psi, \mathrm{3co,co,b}}$ is again required to be smooth, and moreover must satisfy estimates
\end{itemize}
These two definitions are again equivalent in view of the procedures carried out in \S \ref{compatibility of three-cone operators subsection}.  \par 

Finally, if $\supp \psi$ intersects $\cf \cap \dmf$, then we will define
\begin{equation}
\label{d3sc3co 3cocosc}
\psi A \psi \coloneq B_{\psi, \mathrm{3co,co,sc}},
\end{equation}
where $B_{\psi, \mathrm{3co,co,sc}}$ is defined as in (\ref{3cosc quantization}), except that its symbol $b_{\psi,\mathrm{3co,co,sc}}$ must now belong to $S^{m, r, -\infty, \nu, b}( \overline{ ^{\mathrm{d3sc,3co}}T^{\ast}}X )$.

\par

For completeness, we will also provide concrete symbolic characterization for elements of $S^{m,r,l,\nu,b}( \overline{^{\mathrm{d3sc,3co}}T^{\ast}} \Xd )$. Even in the second microlocalized case, it would still be more natural to start from the perspective of the three-cone algebra (i.e., the converse perspective). \par

Suppose we let 
\begin{equation*}
\rho_{\infty}, \rho_{\dmf} ,  \rho_{\dff} , \rho_{\dtsccf},  \rho_{\tcocf} \in \mathcal{C}^{\infty}( \overline{ ^{\mathrm{d3sc,3co}}T^{\ast}}X )
\end{equation*}
be some boundary defining functions for 
\begin{equation*}
\text{$^{\mathrm{d3sc,3co}}S^{\ast} \Xd$, $\overline{ ^{\mathrm{d3sc,3co}}T^{\ast}}_{\dmf} \Xd$,  $\overline{ ^{\mathrm{d3sc,3co}}T^{\ast}}_{\dff} \Xd$,  $\dtsccf$,  
$\tcocf$}
\end{equation*}
respectively. For $H \in \{ \dmf, \cf, \dff \}$, let $U_{H} \subset \overline{ ^{\mathrm{3coc}}T^{\ast} } \Xd$ be some sufficiently small neighborhood of $\overline{ ^{\mathrm{3co}}T^{\ast} }_{H} \Xd$, and let $\widetilde{U}_{H}$ be the lift of $U_{H}$ to $\overline{^{\mathrm{d3sc,3co}}T^{\ast}}\Xd$. Thus in particular, $\widetilde{U}_{\cf}$ could cover $\tcocf$ as well. Then $S^{m,r,l,\nu,b}( \overline{ ^{\mathrm{d3sc,3co}}T^{\ast} } \Xd )$ consists of those $a \in \mathcal{C}^{\infty}( T^{\ast} \mathbb{R}^{n} )$ such that the following conditions are satisfied:
\begin{itemize}
    \item Near every point of $\widetilde{U}_{\cf} \backslash \widetilde{U}_{\dmf}$, we either have
    \begin{align} 
    \label{symbolc estimates at cf away from dmf for first second microlocal}
    \begin{split}
    & | \partial_{z_{\tindex}}^{\beta_{\tindex}} \partial_{t^{\tindex}}^{j} \partial_{y^{\tindex}}^{ \tscblz } \partial_{ \tscblz }^{\gamma_{\tindex}} \partial_{ \tscbut }^{k} \partial_{ \tscbum }^{\gamma^{\tindex}} a  | \\
    & \qquad \leq C_{\beta_{\tindex} j \beta^{\tindex} \gamma_{\tindex} k \gamma^{\tindex}} \rho_{\infty}^{ - m + |\gamma_{\tindex}|+k+|\gamma^{\tindex}|} \rho_{\dff}^{-l + |\beta_{\tindex}|} \rho_{\dtsccf}^{- \nu + 2 |\beta_{\tindex}| + k + |\gamma^{\tindex}| } \rho_{\zf}^{ - b + 2|\beta_{\tindex}| }
    \end{split}
\end{align}
in some coordinates of the form $( z_{\tindex}, t^{\tindex}, y^{\tindex}, \tscblz, \tscblt, \tscbut, \tscbum )$; or equivalently
    \begin{align*} 
    \label{d3sc,3co symbol estimates 2}
    \begin{split}
    &| \partial_{z_{\tindex}}^{\beta_{\tindex}} \partial_{\hat{t}_{\tindex}}^{j} \partial_{y^{\tindex}}^{\beta^{\tindex}} \partial_{\tcolz}^{\gamma_{\tindex}} \partial_{\tcobut}^{k} \partial_{\tcobum}^{\gamma^{\tindex}} a | \\
    & \qquad C_{\beta_{\tindex} j \beta^{\tindex} \gamma_{\tindex} k \gamma^{\tindex}} \rho_{\infty}^{ - m + |\gamma_{\tindex}|+k+|\gamma^{\tindex}|} \rho_{\dff}^{-l + |\beta_{\tindex}|} \rho_{\dtsccf}^{- \nu + 2 |\beta_{\tindex}| + k + |\gamma^{\tindex}| } \rho_{\zf}^{ - b + 2|\beta_{\tindex}| }
    \end{split}
    \end{align*}
    in some coordinates of the form $(z_{\tindex}, \hat{t}_{\tindex}, y^{\tindex}, \tcoblz, \tcobut, \tcobum)$.
    \item In $\widetilde{U}_{\cf} \backslash \widetilde{U}_{\dff}$, we have
    \begin{equation*}
     \label{d3sc,3co symbol estimate 3}
    | \partial_{z_{\tindex}}^{\beta_{\tindex}} \partial_{\hat{z}^{\tindex}}^{\beta^{\tindex}} \partial_{\tcolz}^{\gamma_{\tindex}} \partial_{\tcoscuz}^{\gamma^{\tindex}} a | \leq  C_{\beta_{\tindex} \beta^{\tindex} \gamma_{\tindex} \gamma^{\tindex}} \rho_{\infty}^{ - m + |\gamma_{\tindex}| + |\gamma^{\tindex}|} \rho_{\dmf}^{- r + |\beta_{\tindex}| + |\beta^{\tindex}| } \rho_{\dtsccf}^{-\nu + 2|\beta_{\tindex}| + |\gamma^{\tindex}| } \rho_{\zf}^{-b + 2 |\beta_{\tindex}|}.
    \end{equation*}
     in the coordinates $( z_{\tindex}, \hat{z}^{\tindex}, \tcolz, \tcoscuz )$.
    \item In any region that is away from $\widetilde{U}_{\cf}$, we will require that $a$ restricts to an element of $S^{m,r,l}( \overline{ ^{\mathrm{3sc}}T^{\ast} } [ \overline{\mathbb{R}^{n}} ; \mathcal{C}_{\tindex} ] )$. 
\end{itemize} \par

Finally, we are in a position to make the following definitions:
\begin{definition}[Second microlocalization of the three-body operators] \label{definition of the second microlocalized algebra} Let $m,r,l,\nu,b \in \mathbb{R}$, then we will define the space of second microlocalized three-body operators by 
\begin{equation*}
\Psi^{m,r,l,b}_{\mathrm{3sc,2}} \left( X ; \overline{\tscf^{-1}( o_{\mathcal{C}^{\tindex}} )} \, \right) \coloneq \Psi_{\mathrm{d3sc,3co}}^{m,r,l,r+l,b}( \Xd )
\end{equation*}
\end{definition}
Notice that the above definition is indeed reasonable in view of (\ref{blow up is insensitive from 3sc to d3sc second microlocalized}), i.e., by construction, the elements of $\Psi^{m,r, l, r+l, b}_{\mathrm{d3sc,3co}}(X)$ are expected to have principal symbols which are conormal to $[ \overline{\mathbb{R}^{n}} ; \mathcal{C}_{\tindex} ] ; \overline{\tscf^{-1}( \sco )} ] )$ of order $(m,r,l,b)$.

\subsection{Further resolution at fiber infinity} 
\label{a further resolution at fiber infinity}
As it turns out, there is one more resolution that we must carry out at fiber infinity in order to include operators with the desired microlocal properties. In a way, such a blow-up is designed to fix a `flaw' present in Vasy's three-body calculus, which causes a few issues. \par

Notably, Vasy needed to consider symbols which depend only on the free variables when restricted to $\overline{ ^{\mathrm{3sc}}T^{\ast} }_{\ff}[ \overline{\mathbb{R}^{n}} ; \mathcal{C}_{\tindex} ]$. This is advantageous since the indicial operators at $\ff$ of the quantizations of such symbols will then be scalar-valued, and thus  commute with any other operator. However, it is easy to see that such operators do not naturally belong to the three-body algebra. Thus in \cite[Chapter 13]{AndrasThesis}, the inclusion of such operators were shown to be very technical, making strong usage of the functional calculus $\psi( P ) \in \Psi_{\mathrm{3scc}}^{-\infty,0,0}( [ \overline{\mathbb{R}^{n}} ; \mathcal{C}_{\tindex} ] )$ (as shown in \cite[Appendix C]{AndrasThesis}). Here the smoothing property of $\psi(P)$ is crucial.  \par

The `further resolution' we consider in this subsection is precisely designed to remove the aforementioned technicality, in the sense that the indicial operators at $\dff$ (defined below in \S \ref{subsection principal symbol}) of elements in the new class of operators will naturally allow for $\mathcal{C}^{\infty}( \mathcal{C}_{\tindex} \times \overline{\mathbb{R}^{n_{\tindex}}} )$ functions. See \S \ref{subsection second microlocalization for the indicial operators dff}, and in particular Remark \ref{further resolved large-parameters microlocalized operators 2} for further explanations. In other words, one now has an appropriate sense of `microlocalization at $\dff$'. This will in particular allow for variable orders to be defined at $\dff$. See also Remark \ref{membership of variable orders at cf and dff thesis}. \par

In view of (\ref{second microlocal diffeomorphism}), the resolution we are interested here can again be understood from either the three-body perspective or the three-cone perspective. We shall consider two obvious identifications
\begin{equation}
\begin{gathered} \label{identification 3co at dff}
 \overline{ ^{\mathrm{3co}}T^{\ast} }_{\mathrm{dff}_{\tindex}} \Xd \cong \mathcal{C}_{\tindex} \times \overline{^\mathrm{b}T^{\ast} \oplus \mathbb{R}^{n_{\tindex}} }  X^{\tindex},  \\
 \overline{^{\mathrm{d3sc}}T^{\ast}}_{\mathrm{dff}_{\tindex}} \Xd \cong \mathcal{C}_{\tindex} \times \overline{ ^{\mathrm{sc}}T^{\ast} \oplus \mathbb{R}^{n_{\tindex}} } X^{\tindex}
\end{gathered}
\end{equation}
which can be understood as in the case of (\ref{a simple identification of the sphere bundle at ff in the 3co sense}). These respectively have inclusion maps
\begin{equation*}
\begin{gathered}
\itcodff: \mathcal{C}_{\tindex} \times \overline{^\mathrm{b}T^{\ast} \oplus  \mathbb{R}^{n_{\tindex}} }  X^{\tindex} \rightarrow \overline{ ^{\mathrm{3co}}T^{\ast} } \Xd, \\ 
\idtscdff : \mathcal{C}_{\tindex} \times \overline{ ^{\mathrm{sc}}T^{\ast} \oplus \mathbb{R}^{n_{\tindex}} } X^{\tindex} \rightarrow \overline{ ^{\mathrm{d3sc}}T^{\ast} } \Xd.
\end{gathered}
\end{equation*}
Thus, $\mathcal{C}_{\tindex} \times {^{\mathrm{b}}S^{\ast}} X^{\tindex}$, $\mathcal{C}_{\tindex} \times { ^{\mathrm{sc}}S^{\ast} X^{\tindex} }$ can be realized naturally as submanifolds of $\overline{ ^{\mathrm{3co}}T^{\ast} }_{\dff} \Xd$, $\overline{^{\mathrm{d3sc}}T^{\ast}}_{\dff}\Xd$ respectively{\ep}much like what was done above in \S \ref{subsection diffeomorphism of the phase spaces}. In fact, upon passing to the blow-up (\ref{second microlocal diffeomorphism}), one easily finds that the lifts of $\overline{ ^{\mathrm{3co}}T^{\ast} }_{\mathrm{dff}_{\tindex}} \Xd$ and $\overline{^{\mathrm{d3sc}}T^{\ast}}_{\mathrm{dff}_{\tindex}} \Xd$ must be diffeomorphic to each other. Subsequently, the lifts of the inclusions $\itccf( \mathcal{C}_{\tindex} \times {^{\mathrm{b}}S^{\ast}} X^{\tindex} )$, $\idtscdff( \mathcal{C}_{\tindex} \times { ^{\mathrm{sc}}S^{\ast} X^{\tindex} } )$ must be diffeomorphic to each other as well. \par

It follows that we have
\begin{align}
\begin{split} 
\label{second microlocal diffeomorphism resol}
& \left[ \overline{^{\mathrm{d3sc}}T^{\ast}} \Xd ; \beta_{\mathrm{d3sc}}^{\ast} ( \overline{\tscf^{-1}( 
{\sco} )} ) , \idtscdff( \mathcal{C}_{\tindex} \times {^{\mathrm{sc}}S^{\ast}} X^{\tindex} ) \right] \\
& \quad \quad \cong
\left[ \overline{ ^{\mathrm{3co}} T^{\ast}} \Xd  ; \itccf (\mathcal{C}_{\tindex} \times {^{\mathrm{co}}S^{\ast} [ \hat{X}^{\tindex} ; \{ 0 \} ]} ); \itcodff( \mathcal{C}_{\tindex} \times { ^{\mathrm{b}}S^{\ast} X^{\tindex} } ) \right],
\end{split}
\end{align}
where we remark that the two resolutions in the first line of (\ref{second microlocal diffeomorphism resol}) occur at disjoint sets, and therefore commute. We will refer to the new face thus created by $\rf$ (the `resolved face'). \par

For notational brevity, in the below we will simply write
\begin{equation}
\label{d3sc,3co,res cotangent bundle}
 \overline{^{\mathrm{d3sc,3co,res}}T^{\ast}} \Xd
 \end{equation}
 to denote (\ref{second microlocal diffeomorphism resol}). In spite of the notation, we will again emphasize that $\overline{^{\mathrm{d3sc,3co,res}}T^{\ast}} \Xd$ is no longer a fiber bundle. ${^{\mathrm{d3sc,3co,res}}S^{\ast}} \Xd $. Moreover, we will write
\begin{equation*}
\text{$^{\mathrm{d3sc,3co,res}}S^{\ast}X$, $\overline{^{\mathrm{d3sc,3co,res}}T^{\ast}}_{\mathrm{dmf}} \Xd$, $\overline{^{\mathrm{d3sc,3co,res}}T^{\ast}}_{\dff} \Xd$, $\dtsccf$, $\tcocf$}
\end{equation*}
respectively for the lifts of either
\begin{equation*}
\text{${ ^{\mathrm{d3sc}}S^{\ast} }X$, $\overline{ ^{\mathrm{d3sc}}T^{\ast} }_{\dmf}X$, $\overline{^{\mathrm{d3sc}}T^{\ast}}_{\dff}X$, $\beta_{\mathrm{d3sc}}^{\ast}( \overline{{ ^{\mathrm{3sc}}\pi_{\ff}^{-1} }( o_{\mathcal{C}_{\tindex}} )} )$, $\overline{ ^{\mathrm{d3sc}}T^{\ast} }_{\cf} X$}
\end{equation*}
to the first line of (\ref{second microlocal diffeomorphism resol}), or equivalently
\begin{equation*}
\text{${ ^{\mathrm{3co}}S^{\ast} }X$, $\overline{ ^{\mathrm{3co}}T^{\ast} }_{\dmf}X$, $\overline{^{\mathrm{3co}}T^{\ast}}_{\dff}X$, $\overline{ ^{\mathrm{3co}}T^{\ast} }_{\cf} X$, $ \itccf ( \mathcal{C}_{\tindex} \times {^{\mathrm{co}}S^{\ast} [ \hat{X}^{\tindex} ; \{ 0 \} ]} )$ }
\end{equation*}
to the second line of (\ref{second microlocal diffeomorphism resol}). Here, the notations $\dtsccf$ and $\tcocf$ are the same as those introduced in the discussion of the boundary faces of $\overline{^{\mathrm{d3sc,3co}}T^{\ast}}\Xd$. However, this overlap will not cause any serious confusion, since we will primarily be interested in $\overline{ ^{\mathrm{d3sc,3co,res}}T^{\ast}}\Xd$ below as our phase space below. \par

We now introduce the space of operators
\begin{equation}
\label{constant orders d3sc 3co res operators}
\Psi_{\mathrm{d3sc,3co,res}}^{m,r,l,\nu,b,s}(X)
\end{equation}
following the constructions of $\Psi_{\mathrm{d3sc,3co}}^{m,r,l, \nu, b}(X)$ in \S \ref{subsection definition of the second microlocalized algebra} (and thus the construction of $\Psi_{\mathrm{3coc}}^{m,r,l,b}(X)$ in \S\S \ref{construction of three-cone overview}--\ref{subsection construction of cf} as well), such that elements of (\ref{constant orders d3sc 3co res operators}) have symbols which belong to 
\begin{equation*}
S^{m,r, l, \nu, b,s} ( \psf \Xd ).
\end{equation*}
Here, the indices $m, r, l, \nu, b, s \in \mathbb{R}$ measure decay at 
\begin{equation*}
\text{${ ^{\mathrm{d3sc,3co,res}}S^{\ast} }\Xd$, $\psf_{\dmf} \Xd$, $\psf_{\dff}\Xd$, $\dtsccf$, $\tcocf$, $\rf$}
\end{equation*}
respectively. We will provide slightly less details for the construction of (\ref{constant orders d3sc 3co res operators}) to avoid overly repeating the discussion of \S \ref{subsection definition of the second microlocalized algebra} . \par

Thus first of all, elements $A \in \Psi^{m,r,l,\nu,b}_{\mathrm{d3sc,3co,res}}(X)$ are continuous linear maps $\mathcal{S}( \mathbb{R}^{n} ) \rightarrow \mathcal{S}( \mathbb{R}^{n} )$, which we will identify as Schwartz kernels valued in $^{\mathrm{3co}}\Omega_{R}X$. Next, let $\phi, \psi \in \mathcal{C}^{\infty}(X)$ be cut-off functions whose supports are disjoint. Then we will define $\phi A \psi \in \Psi_{\mathrm{3coc}}^{-\infty, -\infty, l ,b}(X)$ as in the construction of $\Psi_{\mathrm{d3sc,3co}}^{m,r,l,\nu,b}(X)$ \par 

For the definitions of $\psi A \psi$, we have the following cases:
\begin{itemize}
\item If $\supp \psi$ does not intersect $\cf$, then $\psi A \psi$ is given by the standard quantization (i.e., the formula (\ref{the standard quantization again})), with a symbol belonging to $S^{m,-\infty,l,-\infty,-\infty,s}(\overline{ ^{\mathrm{d3sc,3co,res}}T^{\ast} }X)$. In particular, if $\supp \psi$ does not intersect both $\cf$ and $\dff$, then $\psi A \psi \in \Psi_{\mathrm{sc}}^{m,r}(\overline{\mathbb{R}^{n}})$.
\item If $\supp \psi$ intersects $\cf \cap \dmf$, then $\psi A \psi \coloneq B_{\psi, \mathrm{3co,co,sc}}$ as in (\ref{3cosc quantization -1}) (and also (\ref{d3sc3co 3cocosc})), except that $b_{\psi, \mathrm{3co,co,sc}}$ must now belongs to $S^{m,r,-\infty,\nu,b,-\infty}( \overline{ ^{\mathrm{d3sc,3co,res}}T^{\ast} }X )$.
\item If $\supp \psi$ intersects $\cf \cap \dmf$, then $\psi A \psi$ can be equivalently defined by either 
\begin{equation*}
B_{\psi, \mathrm{3co,b}} + R_{\psi, \mathrm{3co,b}} \  \  \text{or} \ \  B_{\psi, \mathrm{3co,co,b}} + R_{\psi, \mathrm{3co,co,b}}
\end{equation*}
corresponding to (\ref{the 3scb quantization written in terms of t}) and (\ref{on-diagonal part focuing on cf}) respectively (and also (\ref{d3sc3co 3cob}), (\ref{d3sc3co 3cocob}) respectively), except that $b_{\psi, \mathrm{3co,b}}$, $b_{\psi,\mathrm{3co,co,sc}}$ must now belong to $S^{m,r,l,\nu,b,s}(\overline{ ^{\mathrm{d3sc,3co,res}}T^{\ast} }X)$.
\end{itemize}
\par

We can also write down the symbol estimates which characterizes $S^{m,r,l,\nu,b,s}( \overline{ ^{\mathrm{d3sc,3co,res}}T^{\ast} }X )$ explicitly in local coordinates. However, this will be delayed to \S \ref{the presence of variable orders section} below, where we consider more general symbol estimates with variable orders. See Definition \ref{definition of variable orders symbols} and Remark \ref{remark recover constant order symbol estimates from the variable orders one}.

\subsection{Requirements on variable orders}
\label{the presence of variable orders section}
Let us now move onto the consideration of variable orders in the second microlocalized setting, which are typically necessary if one wishes to obtain a Fredholm result for non-elliptic operators. \par

Typically at the phase space level, it would be sufficient to just specify the values of the variable orders at the respective boundary faces where the symbol map measures principal decay. The corresponding spaces of conormal symbols can then be defined independently of the smooth extensions of these variable orders. However, as we shall see below, more constraints will need to be imposed in our setting for a global (i.e., non-symbolic, as captured by the indicial operators) consideration. \par

We will proceed very generally. Thus in \S \ref{subsection construction of variable order operators} below, we will define a space of operators
\begin{equation} \label{general variable orders operators}
\Psi^{ \mathsf{m} , \mathsf{r} , \mathsf{l} , \mathsf{v}, \mathsf{b}, \mathsf{s} }_{\mathrm{d3sc,3co,res},\delta}( \Xd )
\end{equation}
depending on some sufficiently small $\delta > 0$. Here, the standard assumptions are that
\begin{equation}  
\label{not so important variable order condition}
\begin{gathered}
\mathsf{m}  \in \mathcal{C}^{\infty}( { ^{\mathrm{d3sc,3co,res}}S^{\ast}} \Xd ),  \quad \mathsf{r} \in \mathcal{C}^{\infty}( \overline{ ^{\mathrm{d3sc,3co,res}}T^{\ast} }_{\mathrm{dmf}} \Xd ),  \\
 \mathsf{v} \in \mathcal{C}^{\infty}( \dtsccf ), \quad \mathsf{s}   \in \mathcal{C}^{\infty}( \mathrm{rf}_{\tindex} ).
\end{gathered}
\end{equation}
However, for the variable orders $\mathsf{l}$ and $\mathsf{b}$, we will require much stricter, global conditions that they depend only on the free variables, i.e.,
\begin{equation}
\label{important variable order condition}
\mathsf{l}, \mathsf{b}, \in \mathcal{C}^{\infty}( \mathcal{C}_{\tindex} \times \overline{\mathbb{R}^{n_{\tindex}}} ).
\end{equation}
We will often use the same letters to denote arbitrary $\mathcal{C}^{\infty}(\psf X)$ extensions of these functions (if there is no risk of confusion) without further elaborations. 
\begin{remark} 
\label{membership of variable orders at cf and dff thesis}
Notice that the smoothness conditions (\ref{important variable order condition}) are crucial, and are really required due to the blow-ups introduced in the definition of $\overline{^{\mathrm{d3sc,3co,res}}T^{\ast}} \Xd$. Indeed, one could check this quite easily in local coordinates, though it is not hard to even see that
\begin{equation}
\begin{gathered} \label{boundary structures of the blown up dff and zf faces}
\overline{ ^{\mathrm{d3sc,3co,res}}T^{\ast} }_{\dff} \Xd \cong \mathcal{C}_{\tindex} \times [ \overline{ ^{\mathrm{sc}}T^{\ast} \oplus \mathbb{R}^{n_{\tindex}} } X^{\tindex} ; \overline{o \oplus \mathbb{R}^{n_{\tindex}}}X^{\tindex}, {^{\mathrm{sc}}S^{\ast}} X^{\tindex} ], \\
\zf \cong \mathcal{C}_{\tindex} \times [ \overline{ ^{\mathrm{co}}T^{\ast} \oplus \mathbb{R}^{n_{\tindex}} } [ \hat{X}^{\tindex} ; \{ 0 \} ]; {^{\mathrm{co}}S^{\ast}} [ \hat{X}^{\tindex} ; \{ 0 \} ] ],
\end{gathered}
\end{equation}
Now, (\ref{boundary structures of the blown up dff and zf faces}) could also be viewed as blow-ups at the translations of $ \mathcal{C}_{\tindex} \times \mathbb{R}^{n_{\tindex}}$ in the directions of the fibers of ${ ^{\mathrm{sc}}T^{\ast} }X^{\tindex}$ and ${ ^{\mathrm{co}}T^{\ast} }[ \hat{X}^{\tindex} ; \{ 0 \} ]$. Hence, smooth functions of $\mathcal{C}_{\tindex} \times \overline{\mathbb{R}^{n_{\tindex}}}$ can be naturally identified as smooth functions on $\overline{ ^{\mathrm{d3sc,3co,res}}T^{\ast}}_{\dff} \Xd$ and $\zf$ as well. 
\end{remark} 
\begin{remark} 
\label{remark different extension of variable orders l and b}
In fact, it turns out that one needs to consider different extensions of $\vol$ and $\vob$ in the construction of operators with variable orders. Namely, it would not be enough to just extend $\vol, \vob$ to be elements of $\mathcal{C}^{\infty}( \psf X )$. Instead, one must also extend $\vol, \vob$ to be elements of $\mathcal{C}^{\infty}( \overline{\mathbb{R}^{n_{\tindex}}} \times \overline{\mathbb{R}^{n_{\tindex}}} )$. See \S \ref{subsection construction of variable order operators} below.
\end{remark}
\par

We will now define the space of conormal symbols with variable orders that satisfy conditions (\ref{not so important variable order condition}) and (\ref{important variable order condition}). Suppose we let
\begin{equation*}
\rho_{\infty}, \rho_{\mathrm{dmf}}, \rho_{\dff}, \rho_{\dtsccf}, \rho_{\tcocf} ,\rho_{\mathrm{rf}_{\tindex}} \in \mathcal{C}^{\infty}( \overline{ 
^{\mathrm{d3sc,3co,res}}T^{\ast} } \Xd )
\end{equation*}
respectively be some defining functions for 
\begin{equation*}
\text{$^{\mathrm{d3sc,3co}}S^{\ast} \Xd$, $\overline{ ^{\mathrm{d3sc,3co}}T^{\ast}}_{\dmf} \Xd$, $\overline{ ^{\mathrm{d3sc,3co}}T^{\ast}}_{\dff} \Xd$, $\dtsccf$, $\tcocf$, $\rf$.} 
\end{equation*}
Then the idea is to define a space of $\mathcal{C}^{\infty}( T^{\ast} \mathbb{R}^{n} )$ functions satisfying the `minimal' symbolic estimates which are met by
\begin{equation} \label{canonical variable order symbol}
\rho_{\infty}^{- \vom } \rho_{\mathrm{dmf}}^{- \vor } \rho_{\dff}^{- \vol } \rho_{\dtsccf}^{- \vov } \rho_{\tcocf}^{- \vob } \rho_{\mathrm{rf}_{\tindex}}^{ - \vos }.
\end{equation} \par

For $H \in \{ \dmf, \cf, \dff \}$, let $U_{H} \subset \overline{ ^{\mathrm{3coc}}T^{\ast} } \Xd$ be some sufficiently small neighborhood of $\overline{ ^{\mathrm{3co}}T^{\ast} }_{H} \Xd$, and let $\widetilde{U}_{H}$ be the lift of $U_{H}$ to $\overline{^{\mathrm{d3sc,3co,res}}T^{\ast}}\Xd$. Thus in particular, $\widetilde{U}_{\cf}$ and $\widetilde{U}_{\dff}$ could respectively cover $\zf$ and $\rf$ as well.

\begin{definition}[Variable orders symbols]
\label{definition of variable orders symbols}
Let $\vom, \vor, \vov,  \vos \in \mathcal{C}^{\infty}( \psf X )$, $\vol, \vob \in \mathcal{C}^{\infty}( \mathcal{C}_{\tindex} \times \overline{\mathbb{R}^{n_{\tindex}}} )$. Then for $\delta > 0$ sufficiently small, we will define
\begin{equation*}
S^{\vom, \vor, \vol, \vov, \vob, \vos}_{\delta}(\psf X)
\end{equation*}
by those $a \in \mathcal{C}^{\infty}( T^{\ast} \mathbb{R}^{n} )$ for which the following conditions are satisfied:
\begin{itemize}
\item Near every point of $\widetilde{U}_{\cf} \backslash \widetilde{U}_{\dmf}$, we either have 
\begin{align*}
\begin{split}
&  | \partial_{z_{\tindex}}^{\beta_{\tindex}} \partial_{t^{\tindex}}^{j} \partial_{y^{\tindex}}^{\beta^{\tindex}} \partial_{ \tscblz }^{  \gamma_{\tindex} } \partial_{ \tscbut }^{k} \partial_{ \tscbum }^{\gamma^{\tindex}} a | \\
& \qquad \leq C_{\beta_{\tindex} j \beta^{\tindex} \gamma_{\tindex} k \gamma^{\tindex}} \rho_{\infty}^{ - \vom + | \gamma_{\tindex} | + k + |\gamma^{\tindex}| - \delta | ( \beta_{\tindex}, j, \beta^{\tindex}, \gamma_{\tindex}, k, \gamma^{\tindex} ) | }  \rho_{\dff}^{ - \vol + |\beta_{\tindex}| - \delta | ( \beta_{\tindex}, \gamma_{\tindex} ) |  } \\
& \qquad \quad \times \rho_{\dtsccf}^{ - \vov + 2 |\beta_{\tindex}| + k + |\gamma^{\tindex}| - \delta | ( \beta_{\tindex}, j, \beta^{\tindex}, \gamma_{\tindex}, k , \gamma^{\tindex} ) |  } \rho_{\tcocf}^{ - \vob+ 2 |\beta_{\tindex}| - \delta | ( \beta_{\tindex}, \gamma_{\tindex} ) | } \rho_{\rf}^{ - \vos + |\beta_{\tindex}| + k + |\gamma^{\tindex}| - \delta | ( \beta_{\tindex}, j , \beta^{\tindex}, \gamma_{\tindex}, k , \gamma^{\tindex} ) | }
\end{split}
\end{align*}
in some coordinates of the form $( z_{\tindex}, t^{\tindex}, y^{\tindex}, \tscblz, \tscbut, \tscbum )$; or equivalently 
\begin{align*}
\begin{split}
&  | \partial_{z_{\tindex}}^{\beta_{\tindex}} \partial_{ \hat{t}_{\tindex} }^{j} \partial_{y^{\tindex}}^{\beta^{\tindex}} \partial_{ \tcoblz }^{  \gamma_{\tindex} } \partial_{ \tcobut }^{k} \partial_{ \tcobum }^{\gamma^{\tindex}} a | \\
& \qquad \leq C_{\beta_{\tindex} j \beta^{\tindex} \gamma_{\tindex} k \gamma^{\tindex}} \rho_{\infty}^{ - \vom  + | \gamma_{\tindex} | + k + |\gamma^{\tindex}| - \delta | ( \beta_{\tindex}, j, \beta^{\tindex}, \gamma_{\tindex}, k, \gamma^{\tindex} ) | }  \rho_{\dff}^{ - \vol + |\beta_{\tindex}| - \delta | ( \beta_{\tindex}, \gamma_{\tindex} ) |  } \\
& \qquad  \quad \times \rho_{\dtsccf}^{ - \vov + 2 |\beta_{\tindex}| + k + |\gamma^{\tindex}| - \delta | ( \beta_{\tindex}, j, \beta^{\tindex}, \gamma_{\tindex}, k , \gamma^{\tindex} ) |  } \rho_{\tcocf}^{ - \vob + 2 |\beta_{\tindex}| - \delta | ( \beta_{\tindex}, \gamma_{\tindex} ) | } \rho_{\rf}^{ - \vos + |\beta_{\tindex}| + k + |\gamma^{\tindex}| - \delta | ( \beta_{\tindex}, j , \beta^{\tindex}, \gamma_{\tindex}, k , \gamma^{\tindex} ) | }
\end{split}
\end{align*}
in some coordinates of the form $(z_{\tindex}, \hat{t}_{\tindex}, y^{\tindex}, \tcoblz, \tcobut, \tcobum)$.
\item In $\widetilde{U}_{\cf} \backslash \widetilde{U}_{\dff}$, we have
\begin{align*}
\begin{split}
& | \partial_{z_{\tindex}}^{\beta_{\tindex}} \partial_{\hat{z}^{\tindex}}^{\beta^{\tindex}} \partial_{\tcosclz}^{\gamma_{\tindex}} \partial_{\tcoscuz}^{\gamma^{\tindex}} a | \\
& \qquad \leq C_{\beta_{\tindex} \beta^{\tindex} \gamma_{\tindex} \gamma^{\tindex}} \rho_{\infty}^{ 
- \vom + |\gamma_{\tindex}| + |\gamma^{\tindex}| - \delta | ( \beta_{\tindex}, \beta^{\tindex}, \gamma_{\tindex}, \gamma^{\tindex} ) | } \rho_{\dmf}^{ - \vor + |\beta_{\tindex}| + |\beta^{\tindex}| - \delta | ( \beta_{\tindex}, \beta^{\tindex}, \gamma_{\tindex}, \gamma^{\tindex} ) | } \\
& \qquad \quad \times \rho_{\dtsccf}^{- \vov + 2 |\beta_{\tindex}| + |\gamma^{\tindex}| - \delta | ( \beta_{\tindex}, \beta^{\tindex}, \gamma_{\tindex}, \gamma^{\tindex} ) | } \rho_{\tcocf}^{ - \vob + 2|\beta_{\tindex}| - \delta | ( \beta_{\tindex},  \gamma_{\tindex} ) | }
\end{split}
\end{align*}
in the coordinates $( z_{\tindex}, \hat{z}^{\tindex}, \tcolz, \tcoscuz )$.
\item In $\tilde{U}_{\dff} \backslash \tilde{U}_{\cf}$, we have
\begin{align*}
\begin{split}
 | \partial_{z_{\tindex}}^{\beta_{\tindex}} \partial_{z^{\tindex}}^{\beta^{\tindex}} \partial_{\tsclz}^{\gamma_{\tindex}} \partial_{\tscuz}^{\gamma^{\tindex}} a | & \leq C_{\beta_{\tindex} \beta^{\tindex} \gamma_{\tindex} \gamma^{\tindex}} \rho_{\infty}^{ - \vom+ |\gamma_{\tindex}| + |\gamma^{\tindex}| - \delta | ( \beta_{\tindex}, \beta^{\tindex}, \gamma_{\tindex}, \gamma^{\tindex} ) | } \\
& \quad \times \rho_{\dff}^{- \vol + |\beta_{\tindex}| - \delta |( 
\beta_{\tindex}, \gamma_{\tindex} )| } \rho_{\rf}^{ - \vos +  |\beta_{\tindex}| + |\gamma^{\tindex}| - \delta | ( \beta_{\tindex}, \beta^{\tindex}, \gamma_{\tindex}, \gamma^{\tindex} ) | }
\end{split}
\end{align*}
in the coordinates $(z_{\tindex}, z^{\tindex}, \zeta_{\tindex}, \zeta^{\tindex})$.
\item In any region that is away from both $\tilde{U}_{\cf}$ and $\tilde{U}_{\dff}$, we will require that $a$ restricts to an element of $S^{\vom,\vor}( \overline{ ^{\mathrm{sc}}T^{\ast} }  \overline{\mathbb{R}^{n}} )$. The variable orders $\vom$ and $\vor$ are implicitly understood as their restrictions to the region in question, and can naturally be realized as elements of $\mathcal{C}^{\infty}( \overline{ ^{\mathrm{sc}}T^{\ast} } \overline{\mathbb{R}^{n}} )$.
\end{itemize}
\end{definition}
\begin{remark}
\label{remark recover constant order symbol estimates from the variable orders one}
If $\vom = m$, $\vor = r$, $\vol = l$, $\vov = \nu$, $\vob = v$ and $\vos = s$ are all constants in $\mathbb{R}$, then the symbol estimates which characterize $S^{m,r,l,\nu,b,s}(\psf X)$ can be recovered by setting $\delta = 0$ (which is not allowed in the variable orders case) in the estimates appearing in Definition \ref{definition of variable orders symbols} above.
\end{remark}
\begin{remark}
The distinction between $S^{ \mathsf{m}, \mathsf{r}, \mathsf{l}, \mathsf{v}, \mathsf{b}, \mathsf{s} }_{\delta} ( \overline{^{\mathrm{d3sc,3co,res}}T^{\ast}} \Xd )$ and $S^{ m, r, l, \nu, b, s }( \overline{^{\mathrm{d3sc,3co,res}}T^{\ast}} \Xd )$ is that we require a loss of order $\delta > 0$ at ${ ^{\mathrm{d3sc,3co,res}}S^{\ast} \Xd }, { \overline{^{\mathrm{d3sc,3co,res}}T^{\ast}}_{\dmf} \Xd}$, $\dtsccf$ and $\rf$ upon differentiating the symbols in any variables (which is in line with the usual variable order symbolic estimates), but only a loss of order $\delta$ at ${ \overline{^{\mathrm{d3sc,3co,res}}T^{\ast}}_{\dff} \Xd}$ and $\zf$ upon differentiating any free variables. This is necessary due to the additional requirements (\ref{important variable order condition}).
\end{remark}
\begin{remark}
It is easy to see that (\ref{canonical variable order symbol}) indeed belongs to $S^{ \mathsf{m}, \mathsf{r} , \mathsf{l}, \mathsf{v}, \mathsf{b} , \mathsf{s} }_{\delta}( \overline{ 
^{\mathrm{d3sc,3co,res}}T^{\ast} } \Xd )$. Moreover, we have
\begin{equation*}
S^{ \mathsf{m} , \mathsf{r}, \mathsf{l}, \mathsf{v} , \mathsf{b}, \mathsf{s} }_{\delta}( \overline{ 
^{\mathrm{d3sc,3co,res}}T^{\ast} } \Xd ) =   \rho_{\infty}^{- \vom } \rho_{\mathrm{dmf}}^{- \vor } \rho_{\dff}^{- \vol } \rho_{\dtsccf}^{- \vov } \rho_{\tcocf}^{-\vob} \rho_{\mathrm{rf}_{\tindex}}^{-\vos} S^{0,0,0,0,0}_{\delta}( \overline{ 
^{\mathrm{d3sc,3co,res}}T^{\ast} } \Xd ).
\end{equation*}
In fact, the above could also be taken as an equivalent definition of variable order symbols, assuming that the right hand side has already been defined. 
\end{remark}

\begin{remark}
In fact, it is not necessary to have the same small constant $\delta > 0$ corresponding to each boundary face. Thus in the most general case, one can define variable order symbols with respect to six independent, sufficiently small constants, $\delta_1, ..., \delta_{6} > 0$. These two classes of symbols are then equivalent in the sense that by choosing the constants to be small enough, one can always be included in the other. One can eliminate this arbitrariness by requiring that the symbolic estimates hold for all choices of small enough $\delta > 0$, i.e., by intersecting the variable orders spaces over all small $\delta > 0$. We will not pursue such a construction.
\end{remark}

\begin{remark}
It is also not necessary to have all of the orders to be variable dependent, particularly for the purpose of just constructing \emph{a} Fredholm map. For example, and indeed as mentioned in the introduction, typically in the two-body scattering theory, one takes $\mathsf{m} = m$ to be constant. The same can be done in the present, three-body case as well. Moreover, if a variable order corresponding to a specific boundary face of $\psf \Xd$ is in fact constant, then one could also remove the loss of small $\delta > 0$ decay at that face when differentiating. 
\end{remark}

\begin{remark}
\label{the serious product structure in the second microlocalized, resolved case remark}
In Remark \ref{product type description remark three-cone case}, we have already explained how elements of $S^{m,r,l,b}( \overline{ ^{\mathrm{3co}}T^{\ast}} \Xd )$ can be thought of as having `product-type' symbolic structures in the base variables near $\overline{^{\mathrm{3co}}T^{\ast}}_{\dff} \Xd$ and $\overline{^{\mathrm{3co}}T^{\ast}}_{\cf} \Xd$. The `further resolution' introduced in \S \ref{a further resolution at fiber infinity} improves these properties. Indeed, it is a straightforward (albeit cumbersome, and therefore be omitted) exercise to check in local coordinates that the elements $a \in S^{m, r, l, b, s}( \psf \Xd )$ (i.e., just the constant orders cases), upon restricting to small neighborhoods of $\psf_{\dff} \Xd$ and $\psf_{\cf} \Xd$, enjoy suitable product-type symbolic behaviors in both the base and frequency variables.  \par

Concretely, in a neighborhood of $\psf_{\dff} X$, the symbolic structure of $a$ is that of the conormal structure for $\overline{\mathbb{R}^{n_{\tindex}}} \times \overline{\mathbb{R}^{n_{\tindex}}} \times \overline{ ^{\mathrm{sc,b}}T^{\ast} }X^{\tindex}$. Here $\overline{\mathbb{R}^{n_{\tindex}}} \times \overline{\mathbb{R}^{n_{\tindex}}}$ should be thought of as a typical scattering cotangent bundle. Likewise, in a neighborhood of $\psf_{\cf} \Xd$, the symbolic structure of $a$ is that of the product structure for $\overline{\mathbb{R}^{n_{\tindex}}} \times \overline{\mathbb{R}^{n_{\tindex}}} \times \overline{ ^{\mathrm{co}} T^{\ast}} [ \hat{X}^{\tindex} ; \{ 0 \} ]$. The situation becomes slightly more complicated in the presences of variable orders. 
\end{remark}

\subsection{Second microlocalized operators with variable orders} 
\label{subsection construction of variable order operators}
In this subsection, we will construct the space of operators (\ref{general variable orders operators}). This construction will be more involved than that of $\Psi^{m,r,l, \nu, b}_{\mathrm{d3sc,3co}}( \Xd )$ and $\Psi^{m,r,l,\nu,b,s}_{\mathrm{d3sc,3co,res}}( \Xd )$. Indeed, due to the presences of the variable orders $\vol, \vob$, we now have to modify the three-cone operators at the global level (as opposed to just modifying symbols of quantizations, which only occurs `near-diagonally') as well. \par

Let $\vom, \vor, \vov,  \vos \in \mathcal{C}^{\infty}( \psf X )$, $\vol, \vob \in \mathcal{C}^{\infty}( \mathcal{C}_{\tindex} \times \overline{\mathbb{R}^{n_{\tindex}}} )$, and $\delta > 0$ be sufficiently small. Then we will first of all require that elements of 
\begin{equation*}
A \in \Psi^{ \mathsf{m} , \mathsf{r} , \mathsf{l} , \mathsf{v}, \mathsf{b}, \mathsf{s} }_{\mathrm{d3sc,3co,res},\delta}(X)
\end{equation*}
be continuous linear maps 
\begin{equation*}
A : \mathcal{S}( \mathbb{R}^{n} ) \rightarrow \mathcal{S}( \mathbb{R}^{n} ).
\end{equation*}
It follows that we can henceforth identify $A$ with its Schwartz kernel, which we will again assume to be a section of ${ ^{\mathrm{3co}}\Omega_{R} }X$. \par 

Let $\tilde{\psi}_{0}, \tilde{\psi}_{\dff}, \tilde{\psi}_{\cf} \in \mathcal{C}^{\infty}( \Xd )$ be chosen as in \S \ref{overview of three-cone bundle subsection}, i.e., we will assume that $\tilde{\psi}_{\dff} + \tilde{\psi}_{\cf} + \tilde{\psi}_0 = 1$, $\tilde{\psi}_{\dff}$, $\tilde{\psi}_{\cf}$ are supported near $\dff$ and $\cf$ respectively, and $\tilde{\psi}_0$ is supported away from $\dff \cup \cf$. Let also ${\psi}_0, {\psi}_{\dff}, {\psi}_{\cf} \in \mathcal{C}^{\infty}(\Xd)$ be cut-off functions with the same properties as $\tilde{\psi}_{0}$, $\tilde{\psi}_{\dff}$ and $\tilde{\psi}_{\cf}$ respectively, and assume additionally that they are respectively identically $1$ on $\supp \tilde{\psi}_0$, $\supp \tilde{\psi}_{\dff}$ and $\supp \tilde{\psi}_{\cf}$. In particular, $\psi_{\dff}$, $\psi_{\cf}$ can be (and will indeed be chosen to be) cut-off functions at $\dff$ and $\cf$ respectively.

Now, as expected, we will require that
\begin{equation*}
\psi_0 A \psi_0 \in \Psi_{\mathrm{sc},\delta}^{\vom, \vor}( \overline{\mathbb{R}^{n}} ).
\end{equation*}
Here, the variable orders $\vom$ and $\vor$ are implicitly understood as their restrictions to the support of $\psi_0$, and can thus naturally be identified as elements of $\mathcal{C}^{\infty}( \overline{ ^{\mathrm{sc}}T^{\ast} } \overline{\mathbb{R}^{n}} )$. We will also require that
\begin{equation*} 
\text{$( 1 - {\psi}_{\bullet}) A \tilde{\psi}_{\bullet}$ has rapidly decaying smooth kernels}, \quad \bullet = 0, \dff, \cf.
\end{equation*}
Thus in view of (\ref{standard operator decomposition using partition of unity}), it will be enough for the definition of $A$ if we can construct the terms $\psi_{\dff} A \psi_{\dff}$, $\psi_{\cf} A \psi_{\cf}$. We will approach these constructions respectively based on the procedures outlined in \S\S \ref{subsection partial quantization near dff},  \ref{subsection partial quantization near cf} (as opposed to \S\S \ref{subsection construction of dff}, \ref{subsection construction of cf}, which would be necessary in order to incorporate the variable orders). \par 

Consider the case near $\dff$. To this end, let $\phi, \psi \in \mathcal{C}^{\infty}(\Xd)$ be chosen as in the beginning of \S \ref{subsection construction of dff}. Then we will need to define all terms of the forms $\psi A \psi$, $\phi A \psi$. If $\supp \psi$ intersects $\cf \cap \dff$, e.g., if $\psi$ is a cut-off function at some point of 
$\cf \cap \dff$, then we will define
\begin{equation} 
\label{variable orders definition of the operator near dff 1}
\psi A \psi \coloneq B_{\mathrm{3co,b}, \psi} + R_{\mathrm{3co,b}, \psi}.
\end{equation}
Here, the term $B_{\mathrm{3co,b}, \psi}$ will again be defined by the formula (\ref{the 3scb quantization written in terms of t}), with the exception that its symbol $b_{\mathrm{3co,b}, \psi}$ now belongs to $S^{\vom, -\infty , \vol, \vov, \vob, \vos}_{\delta}( \psf \Xd )$. Likewise, if we instead assume that $\supp \psi$ intersects $\cf \cap \dmf$, then we will define 
\begin{equation*}
\psi A \psi \coloneq B_{\mathrm{3sc}, \psi},
\end{equation*}
where $B_{\mathrm{3sc}, \psi}$ is the standard quantization (i.e., in Euclidean coordinates $(z,\zeta)$) of a symbol $b_{\mathrm{3sc}, \psi} \in S^{\vom, -\infty, \vol, -\infty, \vob, \vos}_{\delta}( \psf \Xd )$. Alternatively, we can also view $B_{\mathrm{3co,b},\psi}$ and $B_{\mathrm{3sc}, \psi}$ respectively as partial quantizations through (\ref{B 3co b as partial quantization}) and (\ref{B 3sc as partial quantization}), with their operator-valued symbols given by (\ref{quantization term is a partial quantization B3cob}) and (\ref{quantization term is a partial quantization B3sc}).  \par

Consider now the `off-diagonal' terms of $A$ near $\dff$. These include the term $R_{\mathrm{3co,b}, \psi}$ which appears in (\ref{variable orders definition of the operator near dff 1}), as well as the terms $K_{\mathrm{3co,b}, \phi, \psi} \coloneq \phi A \psi$. We will define $R_{\mathrm{3co,b}, \psi}$, $K_{\mathrm{3co,b}, \phi, \psi}$ by (\ref{R 3co b as partial quantization}) and (\ref{K 3co b as partial quantization}) respectively. In particular, the operator-valued symbols of $R_{\mathrm{3co,b}}$, $K_{\mathrm{3co,b}, \phi, \psi}$ will be denoted by $\hat{R}_{\mathrm{3co,b}, \psi}$, $\hat{K}_{\mathrm{3co,b}, \phi, \psi}$ respectively, which we now proceed to construct. \par

As discussed already in Remark \ref{remark different extension of variable orders l and b}, a major complication which occurs in the considerations of $R_{\mathrm{3co,b}, \psi}$, $K_{\mathrm{3co,b}, \phi, \psi}$ in the variable order setting is that we must now consider different extensions of variable orders, i.e., we must now extend $\vol, \vob, \in \mathcal{C}^{\infty}( \mathcal{C}_{\tindex} \times \overline{\mathbb{R}^{n_{\tindex}}} )$ into elements of $\mathcal{C}^{\infty}( \overline{\mathbb{R}^{n_{\tindex}}} \times \overline{\mathbb{R}^{n_{\tindex}}} )$, instead of elements of $\mathcal{C}^{\infty}(\psf \Xd)$. \par

We first remark that the variable orders $\vom, \vor, \vol, \vos$ are irrelevant in the considerations of $R_{\mathrm{3co,b}, \psi}$ and $K_{\mathrm{3co,b}, \phi, \psi}$, since these terms are by construction `trivial' at the boundary faces of $\psf \Xd$ where $\vom, \vor, \vol, \vos$ measure principal decay. \par

However, $\hat{R}_{\mathrm{3co,b},\psi}$, $\hat{K}_{\mathrm{3co,b}, \phi, \psi}$ will indeed be relevant in measuring principal decay at the (slightly more global, i.e., not in the phase space sense) faces $\dff$ (in the sense that their restrictions as $|z_{\tindex}|^{-1} = x_{\tindex} \rightarrow 0$ are components of the `indicial operator at $\dff$') and $\cf$. \par

Nevertheless, since $\hat{R}_{\mathrm{3co,b}, \psi}$, $\hat{K}_{\mathrm{3co,b}, \phi, \psi}$ are not microlocalized in the interaction variables, it no longer make sense to extend $\vol, \vob$ into smooth functions on $\psf \Xd$, whenever they are discussed in relation to $\hat{R}_{\mathrm{3co,b},\psi}$, $\hat{K}_{\mathrm{3co,b}, \phi, \psi}$. Instead, we will always extend $\vol, \vob$ into elements of $\mathcal{C}^{\infty}( \overline{\mathbb{R}^{n_{\tindex}}} \times \overline{ \mathbb{R}^{n_{\tindex}} } )$ in this context without further elaborations.
\par

We finally present the symbolic conditions which $\hat{R}_{\mathrm{3co,b}, \psi}$, $\hat{K}_{\mathrm{3co,b}, \phi, \psi}$ must satisfy. First, they must be smooth in all of the variables. Moreover, let $x_{\dff} \in \mathcal{C}^{\infty}(\Xd)$ be some global defining functions for $\dff$, and let $x_{\mathcal{C}^{\tindex}} \in \mathcal{C}^{\infty}(X^{\tindex})$ be a boundary defining function. Then we will require that $\hat{R}_{\mathrm{3co,b}, \psi}$ satisfies the estimates
\begin{align}
\label{construction of the second microlocalized operators with variable order R 3co b estimates}
\begin{split}
| \partial_{z_{\tindex}}^{\beta_{\tindex}} \partial_{\tscblz}^{\gamma_{\tindex}} \partial_{t^{\tindex}}^{j} \partial_{y^{\tindex}}^{\beta^{\tindex}} \partial_{(t^{\tindex})'}^{j'} \partial_{(y^{\tindex})'}^{ ( \beta^{\tindex} )' } \hat{R}_{\mathrm{3co,b}, \psi} | \leq {} & C_{ \beta_{\tindex} \gamma_{\tindex} j \beta^{\tindex} j' (\beta^{\tindex})' N LM } x_{\dff}^{ - \vol  -\delta | ( \beta_{\tindex}, \gamma_{\tindex} ) |} x_{\mathcal{C}^{\tindex}}^{ - \vob -\delta | ( \beta_{\tindex}, \gamma_{\tindex} ) |} \\
&  \langle z_{\tindex} \rangle^{ -|\beta_{\tindex}| } \langle \lz \rangle^{-N} \langle y^{\tindex} - (y^{\tindex})' \rangle^{-L} e^{  - M | t^{\tindex} - (t^{\tindex})' | }
\end{split}
\end{align}
for all $N, L, M \in \mathbb{R}$. These estimates can be thought of as the variable orders equivalences of (\ref{quote and quote symbolic estiamtes}) (i.e., they are the same estimates, except now one incurs a loss of order $\delta$ with respect to the weights $x_{\dff}$, $x_{\mathcal{C}^{\tindex}}$ upon differentiating in $(z_{\tindex}, \tscblz)$). \par

\begingroup
\allowdisplaybreaks
Likewise, corresponding respectively to the cases (1)--(3) discussed in \S \ref{subsection partial quantization near dff}, we will require that $\hat{K}_{\mathrm{3co,b}, \phi, \psi}$ satisfies the estimates 
\begin{align*}
| \partial_{z_{\tindex}}^{\beta_{\tindex}} \partial_{\tscblz}^{\gamma_{\tindex}} \partial_{z^{\tindex}}^{\beta^{\tindex}} \partial_{ (z^{\tindex})' }^{ ( \beta^{\tindex} )' } \hat{K}_{\mathrm{3co,b}, \phi, \psi} | \leq {} & C_{ \beta_{\tindex} \beta_{\tindex}' \beta^{\tindex} ( \beta^{\tindex} )' N } x_{\dff}^{- \vol -\delta | ( \beta_{\tindex}, \gamma_{\tindex} ) |} \langle z_{\tindex} \rangle^{-|\beta_{\tindex}|}   \langle \lz \rangle^{-N} \\
| \partial_{z_{\tindex}}^{\beta_{\tindex}} \partial_{\tscblz}^{\gamma_{\tindex}} \partial_{t^{\tindex}}^{j} \partial_{y^{\tindex}}^{\gamma^{\tindex}} \partial_{(z^{\tindex})'}^{(\beta^{\tindex})'} \hat{K}_{\mathrm{3co,b}, \phi, \psi} | \leq {} & C_{  \beta_{\tindex} \gamma_{\tindex} j \gamma^{\tindex} ( \beta^{\tindex} )' NM } x_{\dff}^{- \vol -\delta | ( \beta_{\tindex}, \gamma_{\tindex} ) | } (x_{\mathcal{C}^{\tindex}}')^{M} \langle z_{\tindex} \rangle^{ -|\beta_{\tindex}| }   \langle \lz \rangle^{-N}  \\
| \partial_{z_{\tindex}}^{\beta_{\tindex}} \partial_{\tscblz}^{\gamma_{\tindex}} \partial_{t^{\tindex}}^{j} \partial_{y^{\tindex}}^{\gamma^{\tindex}} \partial_{(t^{\tindex})'}^{j'} \partial_{(y^{\tindex})'}^{(\gamma^{\tindex})'} \hat{K}_{\mathrm{3co,b}, \phi, \psi} | \leq {} & C_{ \beta_{\tindex} \gamma_{\tindex}' j \gamma^{\tindex} j' (\gamma^{\tindex})' NM } \\
& \times x_{\dff}^{-\vol -\delta | ( \beta_{\tindex}, \gamma_{\tindex} ) |} x_{\mathcal{C}^{\tindex}}^{-\vob -\delta | ( \beta_{\tindex}, \gamma_{\tindex} ) |}  \langle z_{\tindex} \rangle^{ -|\beta_{\tindex}| }   \langle \lz \rangle^{-N} e^{  - M|t^{\tindex} - (t^{\tindex})' |}.
\end{align*}
for all $N,M, L \in \mathbb{R}$. These estimates are the variable order equivalences of (\ref{estimates for operator valued symbol near dff constant order K}).  \par 
\endgroup

Next we consider the situation near $\cf$. Let $\phi, \psi \in \mathcal{C}^{\infty}( \Xd )$ be chosen as in the beginning of \S \ref{subsection construction of cf}. We then need to define terms of the forms $\psi A \psi$, $\phi A \psi$. Assume that $\supp \psi$ intersects $\cf \cap \dff$, e.g., if $\psi$ is a cut-off function at some point of $\cf \cap \dff$. Then we will define
\begin{equation}
\label{variable orders definition of the operator near cf 1}
\psi A \psi \coloneq B_{\mathrm{3co,co,b}, \psi} + R_{\mathrm{3co,co,b}, \psi}.
\end{equation}
Here, the term $B_{\mathrm{3co,co,b}, \psi}$ will again be defined by the formula (\ref{the 3cob quantization written in terms of t}), with the exception that its symbol $b_{\mathrm{3co,co,b},\psi}$ now belongs to $S^{\vom, -\infty , \vol, \vov, \vob, \vos}_{\delta}( \psf \Xd )$. Likewise, if we instead assume that $\supp \psi$ is supported away from $\cf \cap \dff$, then we will define
\begin{equation*}
\psi A \psi \coloneq B_{\mathrm{3co,co,sc}, \psi},
\end{equation*}
where $B_{\mathrm{3co,co,sc}, \psi}$ is defined by the formula (\ref{3cosc quantization}), with a symbol $b_{\mathrm{3co,co,sc},\psi}$ that now belongs to $S^{\vom, \vor, -\infty, \vov, \vob, - \infty }( \psf \Xd )$. Alternatively, we can also view $B_{\mathrm{3co,co,b},\psi}$ and $B_{\mathrm{3co,co,sc},\psi}$ respectively as partial quantizations through (\ref{B 3co co b as partial quantization}) and (\ref{B 3co co sc as partial quantization}), with their operator-valued symbols given by (\ref{quantization term is a partial quantization B 3co co b}) and (\ref{quantization term is a partial quantization B 3co co sc}).

Finally, we will consider the `off-diagonal' terms of $A$ near $\cf$. These include the term $\hat{R}_{\mathrm{3co,co,b},\psi}$ which appears in (\ref{variable orders definition of the operator near cf 1}), as well as the terms $K_{\mathrm{3co,co}, \phi, \psi} \coloneq \phi A \psi$. We will define $R_{\mathrm{3co,co,b},\psi}$, $K_{\mathrm{3co,co}, \phi, \psi}$ by (\ref{R 3co co b as partial quantization}) and (\ref{K 3co co as partial quantization}) respectively. In particular, the operator-valued symbols of $R_{\mathrm{3co,co,b}, \psi}$, $K_{\mathrm{3co,co}, \phi, \psi}$ will be denoted by $\hat{R}_{\mathrm{3co,co,b}, \psi}$ and $\hat{K}_{\mathrm{3co,co}, \phi, \psi}$ respectively, which we now proceed to construct. \par

As before, we will let $\vol, \vob$ be extended into elements of $\mathcal{C}^{\infty}( \overline{\mathbb{R}^{n_{\tindex}}} \times \overline{\mathbb{R}^{n_{\tindex}}} )$, and this will henceforth be assumed without further elaboration in any discussion concerning $\hat{R}_{\mathrm{3co,co,b}, \psi}$ and $\hat{K}_{\mathrm{3co,co}, \phi, \psi}$. Now, we will require that $\hat{R}_{\mathrm{3co,co,b}, \psi}$, $\hat{K}_{\mathrm{3co,co}, \phi, \psi}$ be smooth in all the variables. Moreover, $\hat{R}_{\mathrm{3co,co,b}, \psi}$ must satisfy the estimates 
\begin{align}
\label{symbolic requirement for R hat 3co b}
\begin{split}
& | \partial_{z_{\tindex}}^{\beta_{\tindex}} \partial_{ \tcottlz }^{\gamma_{\tindex}} \partial_{\hat{t}_{\tindex} }^{j} \partial_{y^{\tindex}}^{\gamma_{\tindex}}  \partial_{\hat{t}_{\tindex}'}^{j'}  \partial_{ (y^{\tindex})' }^{\gamma_{\tindex}'} \hat{R}_{\mathrm{3co,co,b}, \psi}  | \\
& \qquad \leq C_{\beta_{\tindex} \gamma_{\tindex} j j' \gamma_{\tindex} \gamma_{\tindex}'NML} x_{\cf}^{ - \vob -\delta | ( \beta_{\tindex}, \gamma_{\tindex} ) |} x_{\mathcal{C}^{\tindex}_{0}}^{- \vol -\delta | ( \beta_{\tindex}, \gamma_{\tindex} ) |}  \langle z_{\tindex} \rangle^{-|\beta_{\tindex}| } \langle \tcottlz \rangle^{-N} \langle y^{\tindex} - ( y^{\tindex} )' \rangle^{-L} e^{ -M| \hat{t}_{\tindex} - \hat{t}_{\tindex}' |}
\end{split}
\end{align}
for all $N,M, L \in \mathbb{R}$. These estimates are the variable order equivalences of (\ref{symbolic estimate operator valued symbol R near cf}).
\par 

\begingroup
\allowdisplaybreaks

Likewise, corresponding respectively to the cases (1)--(6) discussed in \S \ref{subsection partial quantization near cf}, we will require that $\hat{K}_{\mathrm{3co,co}, \phi, \psi}$ satisfies the estimates  
\begin{align*}
| \partial_{z_{\tindex}}^{\beta_{\tindex}}  \partial_{\tcottlz}^{ \gamma_{\tindex}} \partial_{ \hat{z}^{\tindex}}^{\beta^{\tindex}} \partial_{( \hat{z}^{\tindex})'}^{ (\beta^{\tindex})' } \hat{K}_{\mathrm{3co,co}, \phi, \psi} | \leq {} & C_{ \beta_{\tindex} \gamma_{\tindex} \beta^{\tindex} ( \beta^{\tindex} )' N } x_{\cf}^{-\vob -\delta |( \beta_{\tindex}, \gamma_{\tindex} )| } \langle z_{\tindex} \rangle^{-|\beta_{\tindex}| }  \langle \tcottlz \rangle^{-N} , \\
| \partial_{z_{\tindex}}^{\beta_{\tindex}}  \partial_{ \tcottlz }^{ \beta_{\tindex}' } \partial_{\hat{t}_{\tindex}}^{j} \partial_{y^{\tindex}}^{\gamma^{\tindex}} \partial_{ (\hat{z}^{\tindex})' }^{ (\beta^{\tindex})' }\hat{K}_{\mathrm{3co,co}, \phi, \psi} | \leq {} & C_{\beta_{\tindex} \gamma_{\tindex} j \gamma^{\tindex} ( \beta^{\tindex} )' NML} \\
& \times x_{\cf}^{ -\vob- \delta | ( \beta_{\tindex}, \gamma_{\tindex} ) | } x_{\mathcal{C}^{\tindex}_{0}}^{M} ( x_{\mathcal{C}^{\tindex}_{\infty}}' )^{L}  \langle z_{\tindex} \rangle^{-|\beta_{\tindex}|}   \langle \tcottlz \rangle^{-N} ,  \\ 
| \partial_{z_{\tindex}}^{\beta_{\tindex}}  \partial_{ \tcottlz }^{ \beta_{\tindex}' } \partial_{\hat{t}_{\tindex}}^{j} \partial_{y^{\tindex}}^{\gamma^{\tindex}} \partial_{ (\hat{z}^{\tindex})' }^{ (\beta^{\tindex})' }\hat{K}_{\mathrm{3co,co}, \phi, \psi}| \leq {} & C_{\beta_{\tindex} \gamma_{\tindex} j \gamma^{\tindex} ( \beta^{\tindex} )' N M}  x_{\cf}^{- \vob - \delta | ( \beta_{\tindex}, \gamma_{\tindex} ) | } (x_{\mathcal{C}^{\tindex}_{0}}')^{M} \langle z_{\tindex} \rangle^{-|\beta_{\tindex}|  }   \langle \tcottlz \rangle^{-N} , \\ 
| \partial_{z_{\tindex}}^{\beta_{\tindex}}  \partial_{ 
\tcottlz }^{\gamma_{\tindex}}  \partial_{ \hat{z}^{\tindex}}^{\beta^{\tindex}} \partial_{( \hat{z}^{\tindex})'}^{ (\beta^{\tindex})' } \hat{K}_{\mathrm{3co,co}, \phi, \psi}| \leq {} & C_{ \beta_{\tindex} \gamma_{\tindex} \beta^{\tindex} ( \beta^{\tindex} )' N M}  x_{\cf}^{ -\vob - \delta | ( \beta_{\tindex}, \gamma_{\tindex} ) | } (x_{\mathcal{C}^{\tindex}_{\infty}}')^{M} \langle z_{\tindex} \rangle^{-|\beta_{\tindex}|  }  \langle \tcottlz \rangle^{-N}, \\
| \partial_{z_{\tindex}}^{\beta_{\tindex}}  \partial_{ \tcottlz }^{\gamma_{\tindex}}  \partial_{ \hat{z}^{\tindex}}^{\beta^{\tindex}} \partial_{( \hat{z}^{\tindex})'}^{ (\beta^{\tindex})' } \hat{K}_{\mathrm{3co,co}, \phi, \psi} |  \leq {} & C_{ \beta_{\tindex} \gamma_{\tindex} \beta^{\tindex} ( \beta^{\tindex} )' N M M' } \\
&   \times  x_{\cf}^{ - \vob - \delta | ( \beta_{\tindex}, \gamma_{\tindex} ) | }  x_{\mathcal{C}^{\tindex}_{\infty}}^{M} (x_{\mathcal{C}^{\tindex}_{\infty}}')^{M'} \langle z_{\tindex} \rangle^{-|\beta_{\tindex}| }  \langle \tcottlz \rangle^{-N} ,  \\
 | \partial_{z_{\tindex}}^{\beta_{\tindex}}  \partial_{ \tcottlz }^{ \gamma_{\tindex} }   \partial_{ \hat{t}_{\tindex} }^{j} \partial_{y^{\tindex}}^{\gamma^{\tindex}} \partial_{ \hat{t}_{\tindex}' }^{j'}  \partial_{ (y^{\tindex})' }^{(\gamma^{\tindex})'} \hat{K}_{\mathrm{3co,co}, \phi, \psi} | \leq {} & C_{\beta_{\tindex} \gamma_{\tindex} j \gamma_{\tindex} j' \gamma_{\tindex}'NM} \\
 & \times x_{\cf}^{ - \vob -\delta | ( \beta_{\tindex}, \gamma_{\tindex} ) |}  x_{\dff}^{ - \vol -\delta | ( \beta_{\tindex}, \gamma_{\tindex} ) |}  \langle z_{\tindex} \rangle^{-|\beta_{\tindex}| }  \langle \tcottlz \rangle^{-N} e^{ -M| \hat{t}_{\tindex} - \hat{t}_{\tindex}' |}.
\end{align*}
for any $N,L,M, M' \in \mathbb{R}$. These estimates are the variable order equivalences of (\ref{symbolic estimate operator valued symbol K near cf}). \par 
\endgroup

\begin{remark}
If $\vom = m$, $\vor = r$, $\vol = l$, $\vov = \nu$, $\vob = v$ and $\vos = s$ are all constants in $\mathbb{R}$, then by setting $\delta = 0$ in the above construction (in particular replacing $S_{\delta}^{m,r,l,\nu,b,s}( \psf \Xd )$ with $S^{m,r,l,\nu,b,s}( \psf \Xd )$, see Remark \ref{remark recover constant order symbol estimates from the variable orders one}), we would have recovered exactly the construction of the operators $\Psi_{\mathrm{d3sc,3co,res}}^{m,r,l,\nu,b,s}(\Xd)$.
\end{remark}

\begin{remark}
\label{remark after the construction of the calculus variable order}
It will also be convenience to package the above constructions slightly more densely. Namely, suppose that we write
\begin{equation*}
A_{\psi_{\dff}} \coloneq \psi_{\dff} A \psi_{\dff}, \quad A_{\psi_{\cf}} \coloneq \psi_{\cf} A \psi_{\cf},
\end{equation*}
where $\psi_{\dff}, \psi_{\cf} \in \mathcal{C}^{\infty}(\Xd)$ are cut-off functions at $\dff$, $\cf$ respectively. Then $A_{\psi_{\dff}}$, $A_{\psi_{\cf}}$ can also be written as partial quantizations of some $\hat{A}_{\psi_{\dff}}$, $\hat{A}_{\psi_{\cf}}$ through (\ref{operator valued symbol near dff 2}) and (\ref{three-cone case writing A psicf as a partial quantization}) respectively. However, unlike in Propositions \ref{membership of operator valued symbol near dff} and \ref{membership of operator valued symbol near cf}. The memberships of $\hat{A}_{\psi_{\dff}}$ and $\hat{A}_{\psi_{\cf}}$ are more difficult to characterize, as they belong to families of highly resolved operators. \par

For example, as we have discussed already in Remark \ref{the serious product structure in the second microlocalized, resolved case remark}, the symbol of $\hat{A}_{\psi_{\dff}}$ can be locally identified as those which are conormal to $\overline{\mathbb{R}^{n_{\tindex}}} \times \overline{\mathbb{R}^{n_{\tindex}}} \times \overline{^{\mathrm{sc,b}}T^{\ast}}X^{\tindex}$. However, away from the (lifted) diagonal in the interaction variables (i.e., the variables in $(X^{\tindex})^2$), $\hat{A}_{\psi_{\dff}}$ becomes residual in the total regularity sense. In other words, the residual terms of $\hat{A}_{\psi_{\dff}}$, e.g., terms of the form $\phi^{\tindex} \hat{A}_{\psi_{\dff}} \psi^{\tindex}$, where $\phi^{\tindex}, \psi^{\tindex} \in \mathcal{C}^{\infty}(X^{\tindex})$ are such that $\supp \phi^{\tindex} \cap \supp \psi^{\tindex} = \emptyset$, no longer respect the aforementioned product structure. This discussion also extends to the operators $\hat{A}_{\psi_{\cf}}$ in the obvious way. For these reasons, we shall reframe from characterizing the memberships of $\hat{A}_{\psi_{\dff}}$, $\hat{A}_{\psi_{\cf}}$ in details.
\end{remark}

\begin{remark}
Let $A \in \Psf^{\vom, \vor, \vol, \vov, \vob, \vos}(\Xd)$. Then it is also easy to see that
\begin{equation} 
\label{a different definition of operator valued symbols in the variable order setting}
\hat{A}_{\psi_{\dff}} \coloneq x_{\dff}^{-\vol} x_{\cf}^{-\vob} (\hat{A_0})_{\psi_{\dff}}, \ \hat{A}_{\psi_{\cf}} \coloneq x_{\dff}^{-\vol} x_{\cf}^{-\vob} (\hat{A_0})_{\psi_{\cf}},
\end{equation}
for some $A _{0} \in \Psf^{\vom, \vor, 0, \vov - \vob, 0, \vos - \vol}( \Xd )$. Here the membership of $A_0$ is due to the fact that 
\begin{equation*}
x_{\dff} \simeq \rho_{\dff} \rho_{\rf}, \quad x_{\cf} \simeq \rho_{\tcocf} \rho_{\dtsccf}.
\end{equation*}
Moreover, we are assuming that $\vol, \vob \in \mathcal{C}^{\infty}( \overline{\mathbb{R}^{n_{\tindex}}} \times \overline{\mathbb{R}^{n_{\tindex}}} )$, though it would be enough if we just extend $\vol, \vob \in \mathcal{C}^{\infty}( \mathcal{C}_{\tindex} \times \overline{\mathbb{R}^{n_{\tindex}}} )$ by dilation in $z_{\tindex}$ (thus the extensions only belong to $\mathcal{C}^{\tindex} ( \overline{\mathbb{R}^{n_{\tindex}}} \backslash \{ 0  \} \times \overline{\mathbb{R}^{n_{\tindex}}} )$). Indeed, by construction, it is obvious that $\hat{A}_{\psi_{\dff}}$, $\hat{A}_{\psi_{\cf}}$ must be supported in a small neighborhood of $\mathcal{C}_{\tindex} \times \overline{\mathbb{R}^{n_{\tindex}}}$ in their parameters. 
\end{remark}

\subsection{Compatibility for the variable orders operators} \label{Variable orders compatibility subsection}
As in the discussion of \S \ref{compatibility of three-cone operators subsection}, there is again the questions of whether or not the above constructions satisfy suitable compatibility conditions. Moreover, by the Kuranishi trick, it is only worth investigating such questions for those terms that are not defined by a quantization. \par 

In fact, for brevity, we will only discuss how compatibility can be shown between the terms $R_{\mathrm{3sc,b}, \psi}$ and $R_{\mathrm{3co,b}, \psi}$. Compatibility for the rest of the terms will follow from a similar calculation. \par 

As kernels, it is more convenient (and indeed illustrative) to write $R_{\mathrm{3co,b}, \psi}$ in coordinates $( z_{\tindex}, z_{\tindex}', x^{\tindex}, y^{\tindex}, (x^{\tindex})', (y^{\tindex})' )$ and $R_{\mathrm{3co,co,b}, \psi}$ in coordinates $( z_{\tindex}, z_{\tindex}', \bbdf, y^{\tindex}, \bbdf', (y^{\tindex})' )$. Recall that the transformation laws between these two coordinates are given by
\begin{equation*}
x^{\tindex} = \Big( \frac{ 1 }{ |z_{\tindex}| \bbdf} \Big)^{1/2}, \quad (x^{\tindex})' = \Big( \frac{ 1 }{|z_{\tindex}'| \bbdf'} \Big)^{1/2}.
\end{equation*} 
Now, let $\varphi \in \mathcal{C}^{\infty}_{c}( [0, \infty) )$ be a cut-off function at $1$, and write
\begin{equation*}
\varphi_{1} \coloneq \varphi\Big( \frac{|z_{\tindex}|}{|z_{\tindex}'|} \Big), \quad \varphi_{2} \coloneq ( 1 - \varphi ) \Big( \frac{|z_{\tindex}|}{|z_{\tindex}'|} \Big).
\end{equation*}
Suppose that we make the above coordinates change to $\varphi_{1} R_{\mathrm{3co,b}, \psi}$, and set
\begin{equation*}
{G}_{\mathrm{3co,b}, \psi} ( z_{\tindex}, z_{\tindex}', \bbdf, y^{\tindex}, \bbdf', (y^{\tindex})' ) \coloneq  \varphi_1 R_{\mathrm{3co,b}, \psi} \Big( z_{\tindex}, z_{\tindex}', \Big( \frac{ 1 }{ |z_{\tindex}| \bbdf} \Big)^{1/2}, y^{\tindex}, \Big( \frac{ 1 }{|z_{\tindex}'| \bbdf'} \Big)^{1/2}, (y^{\tindex})'  \Big).
\end{equation*}
Then we need to show that $G_{\mathrm{3co,b}, \psi}$ satisfies the defining properties of $R_{\mathrm{3co,co,b}, \psi}$ as in (\ref{symbolic requirement for R hat 3co b}). \par

In doing so, let $\hat{R}_{\mathrm{3co,b}, \psi}$ be the operator-valued symbol of $R_{\mathrm{3co,b}, \psi}$, as defined in (\ref{construction of the second microlocalized operators with variable order R 3co b estimates}). Then by relabeling $\tscodlz$ into $\tcoodlz$, we see that
\begin{equation} \label{compatibility variable changed formula}
{G}_{\mathrm{3co,b}, \psi} = \frac{1}{(2\pi)^{n_{\tindex}}} \int_{\mathbb{R}^{n_{\tindex}}} e^{ i (z_{\tindex} - z_{\tindex}') \cdot \tcoodlz } \hat{G}_{\mathrm{3co,b}, \psi}  ( z_{\tindex}, z_{\tindex}', \tcoodlz ) d\tcoodlz,
\end{equation}
where we have defined here that 
\begin{equation*}
\hat{G}_{\mathrm{3co,b}, \psi}( z_{\tindex}, z_{\tindex}', \tcolz ) = \varphi_1 \hat{R}_{\mathrm{3co,b}, \psi} \Big( z_{\tindex}, \tcoodlz , \Big( \frac{ 1 }{ |z_{\tindex}| \bbdf} \Big)^{1/2}, y^{\tindex}, \Big( \frac{ 1 }{|z_{\tindex}'| \bbdf'} \Big)^{1/2}, (y^{\tindex})'  \Big).
\end{equation*}
Motivated by the methods in treating the scattering calculus, we will apply a version of the usual `symbolic left reduction' to (\ref{compatibility variable changed formula}) as a quantization in the partial variables. \par

To explain how this can be carried out, let us note that in more conventional settings, such as in the setting of the scattering calculus, what one really requires is a reasonable Frech\'et structure on the space of symbols, such that a notion of convergence can be discussed in this topology. In the current context, we instead have the operator-valued symbols. Nevertheless, locally we can still find natural Frech\'et structures which can be imposed on these operator-valued symbols. \par

For the current discussion, we are working in the regions where $R_{\mathrm{3co,b}, \psi}$, $R_{\mathrm{3co,co,b}, \psi}$ are relevant. Then locally, we can find a natural Frech\'et structure of orders $(\vol, \vob)$ for the class of operator-valued symbols $\hat{R}_{\mathrm{3co,b}, \psi}$. The topology of this Frech\'et structure is determined by those semi-norms defined by the optimal constants in (\ref{construction of the second microlocalized operators with variable order R 3co b estimates}). \par

It follows that the usual left reduction procedure can be carried out in this topology. This allows us to construct some
\begin{equation} \label{left reduction 3sc,b partial symbolic operator}
\hat{G}_{\mathrm{3co,b}, \psi}^{L} = \hat{G}_{\mathrm{3co,b}, \psi}^{L}( z_{\tindex}, \tcoodlz, \bbdf , y^{\tindex}, \bbdf'  , (y^{\tindex})' )
\end{equation}
which belongs to the aforementioned local Frech\'et space. Moreover, we have
\begin{equation} \label{left reduction 3sc,b partial symbolic operator condtion 2}
\hat{G}_{\mathrm{3co,b}, \psi}^{L} \sim \sum_{ \beta_{\tindex} \in \mathbb{N}_{0}^{n_{\tindex}} } \frac{1}{\beta_{\tindex}!} \left( \partial^{\beta_{\tindex}}_{\tcoodlz} D_{z_{\tindex}'}^{\beta_{\tindex}} \hat{G}_{\mathrm{3co,b}, \psi} ( z_{\tindex}, z_{\tindex}', \tcoodlz ) \right)\big|_{ z_{\tindex}' = z_{\tindex} }
\end{equation}
in the following sense: for every integer $J \geq 1$, the difference
\begin{equation*}
\hat{G}_{\mathrm{3co,b}, \psi}^{L} - \sum_{|\beta_{\tindex}| = 0}^{J-1} \frac{1}{\beta_{\tindex}!} \left( \partial^{\beta_{\tindex}}_{\tcoodlz} D_{z_{\tindex}'}^{\beta_{\tindex}} \hat{G}_{\mathrm{3co,b}, \psi} ( z_{\tindex}, z_{\tindex}', \tcoodlz ) \right)\big|_{ z_{\tindex}' = z_{\tindex} }
\end{equation*} 
belongs to the local Frech\'et space of order $( \vol - J ( 1 - 2\delta ), \vob - J ( 2 - 2 \delta. ) )$. For each $\beta_{\tindex} \in \mathbb{N}_{0}^{n_{\tindex}}$, the above construction also require
\begin{equation*}
 \left( \partial^{\beta_{\tindex}}_{\tcoodlz} D_{z_{\tindex}'}^{\beta_{\tindex}} \hat{G}_{\mathrm{3co,b}, \psi} ( z_{\tindex}, z_{\tindex}', \tcoodlz ) \right)\big|_{ z_{\tindex}' = z_{\tindex} }
\end{equation*}
to belong to the local Frech\'et space of order $( \vol - |\beta_{\tindex}| ( 1 - 2\delta ), \vob - |\beta_{\tindex}| ( 2 - 2\delta ) )$. However, this can be checked easily by direct computations. \par 

More systematically, suppose that we define
\begin{equation*}
\begin{gathered}
G_{\mathrm{3sc,b}, \psi}^{L} = \frac{1}{(2\pi)^{n_{\tindex}}} \int_{\mathbb{R}^{n_{\tindex}}} e^{i ( z_{\tindex} - z_{\tindex}' ) \cdot \tcoodlz } \hat{G}_{\mathrm{3co,b}, \psi}^{L}( z_{\tindex}, \tcoodlz ) d\tcoodlz, \\
G_{\mathrm{3co,b}, \psi, j} \coloneq \sum_{|\beta_{\tindex}| = j } \frac{1}{(2\pi)^{n_{\tindex}}} \int_{\mathbb{R}^{n_{\tindex}}} e^{ i ( z_{\tindex} - z_{\tindex}' ) \cdot \tcottlz } \frac{1}{\beta_{\tindex}!}  \left( \partial^{\beta_{\tindex}}_{\tcoodlz} D_{z_{\tindex}'}^{\beta_{\tindex}} \hat{G}_{\mathrm{3co,b}, \psi, j} ( z_{\tindex}, z_{\tindex}', \tcoodlz ) \right)\big|_{ z_{\tindex}' = z_{\tindex} } d \tcottlz.
\end{gathered}
\end{equation*}
Then we have
\begin{equation*}
G^{L}_{\mathrm{3co,b}, \psi } \sim \sum_{ j = 0 }^{\infty} G_{\mathrm{3co,b}, \psi, j}
\end{equation*}
in the sense that
\begin{equation*}
G_{\mathrm{3co,b}, \psi }^{L} - \sum_{j=0}^{J-1} G_{\mathrm{3co,b}, \psi, j} \in \Psi^{-\infty, -\infty, \mathsf{l} - J( 1 - 2\delta ), - \infty, \mathsf{b} - J(2 - 2 \delta ), -\infty }_{\mathrm{d3sc,3co,res},\delta} ( \Xd ). 
\end{equation*} \par 

This shows that the definition of $\varphi_1 R_{\mathrm{3co,b}, \psi}$ is compatible with the definition of $R_{\mathrm{3co,co,b},\psi}$. On the other hand, we can show that $\varphi_2 R_{\mathrm{3co,b},\psi}$ is rapidly decreasing. Indeed, let us write $R_{\mathrm{3co,b}, \psi, 2} \coloneq \varphi_2 R_{\mathrm{3co,b}, \psi}$. Then on the support of $R_{\mathrm{3co,b}, \psi, 2}$, we can compute that
\begin{equation*}
R_{\mathrm{3co,b}, \psi} = \frac{1}{(2\pi)^{n_{\tindex}}} \int_{\mathbb{R}^{n_{\tindex}}} e^{ i ( z_{\tindex} - z_{\tindex}' ) \cdot \zeta_{\tindex} } |z_{\tindex} - z_{\tindex}'|^{-K}  \Delta_{\zeta_{\tindex}}^{K} \hat{R}_{\mathrm{3co,b},\psi, 2} ( z_{\tindex}, \zeta_{\tindex} ) d\zeta_{\tindex}
\end{equation*}
for any positive integer $K$. The rest of the argument is standard.

\subsection{Second microlocalizing the indicial operators at $\dff$}
\label{subsection second microlocalization for the indicial operators dff} 
Thus far, we have defined the spaces of operators $\Psi_{\mathrm{d3sc,3co,res}}^{m,r,l,\nu,b,s}(\Xd)$ and their variable orders analogues. In \S \ref{Section: Microlocal properties of the second microlocalized algebra} below, we will study the microlocal properties of these operators. In particular, assuming that suitable classicality conditions are satisfied, we will define the indicial operators at $\dff$ and $\cf$. In this subsection and the next, we will describe the structures of the indicial operators.    \par

First, we will consider the indicial operators at $\dff$. As we have already mentioned in \S \ref{parameters-dependent families subsection}, indicial operators for the three-body operators in $\Psi^{m,r,l}_{\mathrm{3scc}}( [ \overline{\mathbb{R}^{n}} ; \mathcal{C}_{\tindex} ] )$ (assuming that they are defined) are families of large-parameter scattering operators in $\Psi^{m,r-l}_{\mathrm{sc,lp}}( X^{\tindex} ; \mathcal{C}_{\tindex} \times {\mathbb{R}^{n_{\tindex}}} )$. \par 

One particular characteristic for such operators is that they are `fully microlocal' in the sense that they can be defined by directly quantizing conormal symbols on
\begin{equation} \label{indicial operators large parameters scattering phase space}
\overline{ ^{\mathrm{d3sc}}T^{\ast} }_{\dff} \Xd \cong \overline{ ^{\mathrm{3sc}}T^{\ast} }_{\ff} [ \overline{\mathbb{R}^{n}} ; \mathcal{C}_{\tindex} ] \cong \mathcal{C}_{\tindex} \times \overline{ ^{\mathrm{sc}}T^{\ast} \oplus \mathbb{R}^{n_{\tindex}}}X^{\tindex}.
\end{equation}
Thus, by restricting the left hand side of (\ref{second microlocal diffeomorphism}) to the lift of $\overline{ ^{\mathrm{d3sc}}T^{\ast} }_{\dff}\Xd$, we expect that the indicial operators at $\dff$ for elements of $\Psi_{\mathrm{d3sc,3co}}^{m,r,l, \nu, b}(\Xd)$ should be microlocalized on 
\begin{equation}
\label{scattering large-parameters second microlocalized blow-up}
\overline{^{\mathrm{d3sc,3co}}T^{\ast}}_{\dff} \Xd  \cong \mathcal{C}_{\tindex} \times [ 
\overline{^{\mathrm{sc}}T^{\ast} \oplus \mathbb{R}^{n_{\tindex}} } X^{\tindex} ; \overline{ o \oplus \mathbb{R}^{n_{\tindex}} }_{\mathcal{C}^{\tindex}} X^{\tindex} ].
\end{equation}
More generally, by restricting the first line of (\ref{second microlocal diffeomorphism resol}) to the lift of $\overline{ ^{\mathrm{d3sc}}T^{\ast}}_{\dff} \Xd$, we also expect that the indicial operators at $\dff$ for elements of the `further-resolved' spaces of operators $\Psi_{\mathrm{d3sc,3co,res}}^{m,r,l, \nu, b, s}(\Xd)$ should be microlocalized on
\begin{equation}
\label{boundary class diffeomorphism 1}
\begin{gathered}
\overline{^{\mathrm{d3sc,3co,res}}T^{\ast}}_{\dff} \Xd \cong \mathcal{C}_{\tindex} \times [ \overline{^{\mathrm{sc}}T^{\ast} \oplus \mathbb{R}^{n_{\tindex}}} X^{\tindex} ; \overline{o \oplus \mathbb{R}^{n_{\tindex}}}_{\mathcal{C}^{\tindex}} X^{\tindex} , {^{\mathrm{sc}}S^{\ast}} X^{\tindex} ].
\end{gathered}
\end{equation}
However, in both cases, the new indicial operators will no longer be `fully microlocal' as in the three-body case. In fact, the resolutions which appear in (\ref{scattering large-parameters second microlocalized blow-up}) and (\ref{boundary class diffeomorphism 1}) suggest that they should also be second microlocal in nature.
\par 

We first consider the operators whose elements are microlocalized on (\ref{scattering large-parameters second microlocalized blow-up}). Since (\ref{scattering large-parameters second microlocalized blow-up}) is a resolution of (\ref{indicial operators large parameters scattering phase space}), the construction of these operators should be understood as a process of second microlocalizing $\Psi_{\mathrm{sc,lp}}( X^{\tindex} ; \mathcal{C}_{\tindex} \times {\mathbb{R}^{n_{\tindex}}} )$ at $\overline{ o \oplus \mathbb{R}^{n_{\tindex}} }_{\mathcal{C}^{\tindex}} X^{\tindex}$. Thus, with the notations we have introduced above, we are interested in constructing spaces of operators
\begin{equation}
\label{second microlocalized large-parameters scattering b space}
\Psi^{m,r,l}_{\mathrm{sc,lp,2}}( X^{\tindex} ; \mathcal{C}_{\tindex} \times {\mathbb{R}^{n_{\tindex}}} ; \overline{o \oplus \mathbb{R}^{n_{\tindex}}}_{\mathcal{C}^{\tindex}} X^{\tindex} )
\end{equation}
such that elements of (\ref{second microlocalized large-parameters scattering b space}) have symbols which belong to
\begin{equation*}
S^{m,r,l}( \overline{ ^{\mathrm{d3sc,3co}} T^{\ast}}_{\dff} \Xd ).
\end{equation*}
Under identification (\ref{scattering large-parameters second microlocalized blow-up}), $m,r, l \in \mathbb{R}$ measure decay at the lifts of the `joint fiber infinity' of $\mathcal{C}_{\tindex} \times \overline{ ^{\mathrm{sc}}T^{\ast} \oplus \mathbb{R}^{n_{\tindex}} }X^{\tindex}$, $\mathcal{C}_{\tindex} \times \overline{ ^{\mathrm{sc}}T^{\ast} \oplus \mathbb{R}^{n_{\tindex}} }_{ \mathcal{C}^{\tindex} } X^{\tindex}$, and $ \overline{ o \oplus \mathbb{R}^{n_{\tindex}} }_{\mathcal{C}^{\tindex}} X^{\tindex}$ respectively.  \par

Now, one can construct (\ref{second microlocalized large-parameters scattering b space}) by adopting the philosophy used for second microlocalizating the scattering or the three-body algebras. Namely, we can consider a `converse perspective' to the problem, which would be to instead resolve the large-parameter conormal b-operators $\Psi^{m,l}_{\mathrm{bc,lp}}( X^{\tindex} ; \mathcal{C}_{\tindex} \times {\mathbb{R}^{n_{\tindex}}} )$ in the current setting. \par

Indeed, it is not hard to see that there exists a diffeomorphism 
\begin{equation}
\label{large-parameters phase space diffeomorphism}
 [ \overline{^{\mathrm{sc}}T^{\ast} \oplus \mathbb{R}^{n_{\tindex}}} X^{\tindex} ; \overline{o \oplus \mathbb{R}^{n_{\tindex}}}_{ \mathcal{C}^{\tindex} } X^{\tindex} ] \cong [ \overline{^{\mathrm{b}}T^{\ast} \oplus \mathbb{R}^{n_{\tindex}}} X^{\tindex} ; {^{\mathrm{b}}S^{\ast}_{ \mathcal{C}^{\tindex} }} X^{\tindex} ].
\end{equation}
Here, the lifts of the `joint fiber infinities' of $\overline{ ^{\mathrm{sc}}T^{\ast} \oplus \mathbb{R}^{n_{\tindex}} } X^{\tindex}$, $\overline{ ^{\mathrm{b}}T^{\ast} \oplus \mathbb{R}^{n_{\tindex}} } X^{\tindex}$ are identified with each other; the lift of $\overline{o \times \mathbb{R}^{n_{\tindex}}} X^{\tindex}$ to the left hand side of (\ref{large-parameters phase space diffeomorphism}) identifies with the lift of $\overline{ ^{\mathrm{b}}T^{\ast} \oplus \mathbb{R}^{n_{\tindex}} }_{ \mathcal{C}^{\tindex} }X^{\tindex}$ to the right hand side of (\ref{large-parameters phase space diffeomorphism}); and the lift of $^{\mathrm{b}}S^{\ast}_{\mathcal{C}^{\tindex}}X^{\tindex}$ to the right hand side of (\ref{large-parameters phase space diffeomorphism}) identifies with the lift of $\overline{ ^{\mathrm{sc}}T^{\ast} \oplus \mathbb{R}^{n_{\tindex}} }_{ \mathcal{C}^{\tindex} }X^{\tindex}$ to the left hand side of (\ref{large-parameters phase space diffeomorphism}).  \par 

Alternatively, we could note that
\begin{equation}
\label{second microlocalized large-parameters scattering b space 1}
\overline{^{\mathrm{3co}}T^{\ast}}_{\dff}\Xd \cong \mathcal{C}_{\tindex} \times  
\overline{^{\mathrm{b}}T^{\ast} \oplus \mathbb{R}^{n_{\tindex}} }X^{\tindex}. 
\end{equation}
Thus, we expect that the indicial operators at $\dff$ for elements $A \in \Psi^{m,r,l,b}_{\mathrm{3coc}}(\Xd)$ (assuming they are defined) should naturally be microlocalized on the phase space (\ref{second microlocalized large-parameters scattering b space 1}), and moreover live in the large-parameter conormal b-algebra. See also Proposition \ref{membership of operator valued symbol near dff} for how this statement can be realized concretely by letting $|z_{\tindex}| \rightarrow \infty$ in $\hat{A}_{\psi_{\dff}}$, where $\psi_{\dff} \in \mathcal{C}^{\infty}(\Xd)$ is some cut-off function at $\dff$. \par

Thus, from the perspective of resolving $\Psi_{\mathrm{3coc}}^{m,r,l,b}(\Xd)$ and its phase space, and by restricting the right hand side of (\ref{second microlocal diffeomorphism}) to the lift of $\overline{ ^{\mathrm{3co}}T^{\ast} }_{\dff} \Xd$, it is also natural that the indicial operators at $\dff$ for elements of $\Psi_{\mathrm{d3sc,3co}}^{m, r ,l ,\nu, b}(X)$ are microlocalized on the Cartesian product between $\mathcal{C}_{\tindex}$ and the right hand side of (\ref{large-parameters phase space diffeomorphism}). \par

It follows that we can apply essentially the same procedures as adopted in \S \ref{subsection definition of the second microlocalized algebra} to construct (\ref{second microlocalized large-parameters scattering b space}). The resulting spaces of operators will thus have characteristics of both the large-parameter scattering operators as well as the large-parameter conormal b-operators, where the parameter space is $\mathcal{C}_{\tindex} \times \overline{\mathbb{R}^{n_{\tindex}}}$ in both cases. For this reason, and following the convention established by Vasy in \cite{AndrasSM}, which was also used in \S \ref{subsection definition of the second microlocalized algebra} above, we can instead write
\begin{equation} \label{second microlocalized large-parameters scattering b space 2}
\Psi^{m,r,l}_{\mathrm{sc,lp,2}}( X^{\tindex} ; \mathcal{C}_{\tindex} \times {\mathbb{R}^{n_{\tindex}}} ; \overline{o \oplus \mathbb{R}^{n_{\tindex}}}_{\mathcal{C}_{\tindex}} X^{\tindex} ) = \Psi_{\mathrm{sc,b,lp}}^{m,r,l}( X^{\tindex} ; \mathcal{C}_{\tindex} \times { \mathbb{R}^{n_{\tindex}} } ).
\end{equation} \par

However, since our primary interest is in the `further-resolved' spaces of operators, possibly with variable orders, as defined respectively in \S \ref{a further resolution at fiber infinity} and \S \ref{subsection construction of variable order operators}, we will omit discussing in details the construction of (\ref{second microlocalized large-parameters scattering b space 2}).   \par

Indeed, although (\ref{second microlocalized large-parameters scattering b space 2}) suffices for a precise characterization of the indicial operators at $\dff$ for elements in $\Psi^{m,r,l,\nu,b}_{\mathrm{d3sc,3co}}(\Xd)$, it is insufficient for characterizing those operators which are microlocalized on the phase space (\ref{boundary class diffeomorphism 1}). Nevertheless, such an analysis can still be carried out through the `converse perspective' introduced above. \par

To this see, let us note that, upon viewing $\psf X$ from the perspective of resolving $\overline{ ^{\mathrm{3co}}T^{\ast} }X$, and then restricting the second line in (\ref{second microlocal diffeomorphism resol}) to the lift of $\overline{ ^{\mathrm{d3sc}}T^{\ast} }_{\dff} \Xd$, we have
\begin{equation}
\label{boundary class diffeomorphism 2}
\psf_{\dff} X \cong \mathcal{C}_{\tindex} \times [ \overline{^{\mathrm{b}}T^{\ast} \oplus \mathbb{R}^{n_{\tindex}}} X^{\tindex} ; {^{\mathrm{b}}S^{\ast}_{\mathcal{C}^{\tindex} }} X^{\tindex}; {^{\mathrm{b}}S^{\ast}X^{\tindex} } ].
\end{equation}
Thus, (\ref{boundary class diffeomorphism 1}) and (\ref{boundary class diffeomorphism 2}) together imply that
\begin{equation}
\label{boundary class diffeomorphism 3}
\mathcal{C}_{\tindex} \times [ \overline{^{\mathrm{sc}}T^{\ast} \oplus \mathbb{R}^{n_{\tindex}}} X^{\tindex} ; \overline{o \oplus \mathbb{R}^{n_{\tindex}}}_{\mathcal{C}^{\tindex}} X^{\tindex} , {^{\mathrm{sc}}S^{\ast}} X^{\tindex} ] \cong \mathcal{C}_{\tindex} \times  [ \overline{^{\mathrm{b}}T^{\ast} \oplus \mathbb{R}^{n_{\tindex}}} X^{\tindex} ; {^{\mathrm{b}}S^{\ast}_{\mathcal{C}^{\tindex} }} X^{\tindex}; {^{\mathrm{b}}S^{\ast}X^{\tindex} } ].
\end{equation}
Here, the lifts of the `joint fiber infinity' of $ \mathcal{C_{\tindex}} \times \overline{ ^{\mathrm{sc}}T^{\ast} \oplus \mathbb{R}^{n_{\tindex}} } X^{\tindex}$ and $\mathcal{C}_{\tindex} \times \overline{ ^{\mathrm{b}}T^{\ast} \oplus \mathbb{R}^{n_{\tindex}} } X^{\tindex}$ are identified with each other; the lift of $\mathcal{C}_{\tindex} \times \overline{o \times \mathbb{R}^{n_{\tindex}}} X^{\tindex}$ to the left hand side of (\ref{boundary class diffeomorphism 3}) identifies with the lift of $\mathcal{C}_{\tindex} \times \overline{ ^{\mathrm{b}}T^{\ast} \oplus \mathbb{R}^{n_{\tindex}} }_{ \mathcal{C}^{\tindex} }X^{\tindex}$ to the right hand side of (\ref{boundary class diffeomorphism 3}); the lift of $\mathcal{C}_{\tindex} \times {^{\mathrm{b}}S^{\ast}_{\mathcal{C}^{\tindex}}X^{\tindex}}$ to the right hand side of (\ref{boundary class diffeomorphism 3}) identifies with the lift of $\mathcal{C}_{\tindex} \times \overline{ ^{\mathrm{sc}}T^{\ast} \oplus \mathbb{R}^{n_{\tindex}} }_{ \mathcal{C}^{\tindex} }X^{\tindex}$ to the left hand side of (\ref{boundary class diffeomorphism 3}); finally, the lifts of $\mathcal{C}_{\tindex} \times {^{\mathrm{sc}}S^{\ast}X^{\tindex}}$ and $\mathcal{C}_{\tindex} \times {^{\mathrm{b}}S^{\ast}X^{\tindex}}$ are identified as well. \par

More globally, the aforementioned boundary faces of (\ref{boundary class diffeomorphism 3}) can also be respectfully identified as the restrictions to $\psf_{\dff} \Xd$ of ${ ^{\mathrm{d3sc,3co,res}}S^{\ast}} \Xd$, $\tcocf$, $\dtsccf$ and $\rf$.

We now focus on the construction of spaces of operators 
\begin{equation} 
\label{further resolved large-parameters microlocalized operators}
\Psi_{\mathrm{sc,b, lp, res}}^{m,r,l,s}( X^{\tindex} ; \mathcal{C}_{\tindex} \times {\mathbb{R}^{n_{\tindex}}} )
\end{equation}
which are microlocalized on (\ref{boundary class diffeomorphism 1}), i.e., elements of (\ref{further resolved large-parameters microlocalized operators}) have symbols which belong to
\begin{equation*}
S^{m,r,l,s}( \psf_{\dff} X ).
\end{equation*}
Under identification (\ref{boundary class diffeomorphism 1}), $m,r, l,s  \in \mathbb{R}$ measure decay at the lifts of the `joint fiber infinity' of $\mathcal{C}_{\tindex} \times \overline{ ^{\mathrm{sc}}T^{\ast} \oplus \mathbb{R}^{n_{\tindex}} }X^{\tindex}$, $\mathcal{C}_{\tindex} \times \overline{ ^{\mathrm{sc}}T^{\ast} \oplus \mathbb{R}^{n_{\tindex}} }_{ \mathcal{C}^{\tindex} } X^{\tindex}$, $\mathcal{C}_{\tindex} \times \overline{ o \oplus \mathbb{R}^{n_{\tindex}} }_{\mathcal{C}^{\tindex}} X^{\tindex}$ and $\mathcal{C}_{\tindex} \times { ^{\mathrm{sc}}S^{\ast} } X^{\tindex}$ respectively. Equivalently under identification (\ref{boundary class diffeomorphism 2}), they measure decay at the lifts of of `joint fiber infinity' of $\mathcal{C}_{\tindex} \times \overline{ ^{\mathrm{b}}T^{\ast} \oplus \mathbb{R}^{n_{\tindex}}} X^{\tindex}$, $\mathcal{C_{\tindex}} \times { ^{\mathrm{b}}S^{\ast}_{\mathcal{C}^{\tindex}} }X^{\tindex}$, $\mathcal{C}_{\tindex} \times \overline{ ^{\mathrm{b}}T^{\ast} \oplus \mathbb{R}^{n_{\tindex}} }_{\mathcal{C}_{\tindex}} X^{\tindex}$ and $\mathcal{C}_{\tindex} \times { ^{\mathrm{b}}S^{\ast} X^{\tindex}}$ respectively.
\par

First, we will require that elements of (\ref{further resolved large-parameters microlocalized operators}) be continuous linear mappings $\hat{A}: \dot{\mathcal{C}}^{\infty}( X^{\tindex} ) \rightarrow \dot{\mathcal{C}}^{\infty}( X^{\tindex} )$. In particular, they can be identified with their distributional kernels, which we will realize as sections of ${ ^{\mathrm{b}}\Omega}_{R} X^{\tindex}$. \par

Next, we will construct $\hat{A}$ following the procedure outlined in \S \ref{parameters-dependent families subsection}. Let $\psi^{\tindex} \in \mathcal{C}^{\infty}( X^{\tindex} )$ be a cut-off at some point of $\partial X^{\tindex}$. Then following (\ref{large-parameters b operator construction 0.1}), we will require that 
\begin{equation}
\label{B in the constant order large parameter indicial operator b calculus -1}
\psi^{\tindex} \hat{A} \psi^{\tindex} \coloneq \hat{B}_{\psi^{\tindex}} + \hat{R}_{\psi^{\tindex}}.
\end{equation}
Here, and in the local coordinates $(y_{\tindex}, t^{\tindex}, y^{\tindex}, \tscblz , \tscbut , \tscbum) \in \mathcal{C}_{\tindex} \times ( {^{\mathrm{b}}}T^{\ast} \oplus \mathbb{R}^{n_{\tindex}} )(X^{\tindex})^{\circ}$, we define
\begin{align} 
\label{B in the constant order large parameter indicial operator b calculus}
\begin{split} 
\hat{B}_{\psi^{\tindex}} \coloneq  {} & \frac{1}{(2\pi)^{n^{\tindex}}} \int_{\mathbb{R}^{n^{\tindex}}}  e^{ - i (t^{\tindex} - (t^{\tindex})') \tscbut + i ( y^{\tindex} - (y^{\tindex})'  ) \cdot \tscbum  }  \\
& \times \varphi(t^{\tindex} - (t^{\tindex})') \varphi( |y^{\tindex} - ( y^{\tindex} )'| )  b_{ \psi^{\tindex} } ( y_{\tindex}, t^{\tindex}, y^{\tindex}, \tscblz , \tscbut , \tscbum  )  d \tscbut d\tscbum | d(t^{\tindex})' d(y^{\tindex})' |
\end{split}
\end{align} 
for some $b_{ \psi^{\tindex} } \in S^{m,r,l,s}( \psf_{\dff}\Xd )$; meanwhile, we will require the kernel of $\hat{R}_{\psi^{\tindex}}$ to be smooth in all of the variables, and satisfy additionally the estimates
\begin{equation}
\label{dff second microlocalizing indicial operator estimate on hat R}
| \partial_{y_{\tindex}}^{\beta_{\tindex}} \partial^{\gamma_{\tindex}}_{\zeta_{\tindex}} \partial_{t^{\tindex}}^{j} \partial_{y^{\tindex}}^{\beta^{\tindex}} \partial_{(t^{\tindex})'}^{j'} \partial_{(y^{\tindex})'}^{(\beta^{\tindex})'} \hat{R}_{\psi^{\tindex}} | \leq C_{\alpha \sigma j \beta^{\tindex} j' (\beta^{\tindex})'} x_{\mathcal{C}^{\tindex}}^{-l} \langle \zeta_{\tindex} \rangle^{-N} \langle y^{\tindex} - (y^{\tindex})' \rangle^{-M} e^{-L| t^{\tindex} - (t^{\tindex})' |}
\end{equation}
for all $N,M, L \in \mathbb{R}$, where $x_{\mathcal{C}^{\tindex}} \in \mathcal{C}^{\infty}( X^{\tindex} )$ is a boundary defining function. \par 

If instead $\psi^{\tindex} \in \mathcal{C}^{\infty}(X^{\tindex})$ is supported away from $\mathcal{C}^{\tindex}$, then we will require $\psi^{\tindex} \hat{A} \psi^{\tindex}$ to be given by a standard quantization in the variables $(z^{\tindex}, \zeta^{\tindex})$ (omitting the parameter $(y_{\tindex}, \zeta_{\tindex})$), with a symbol which belongs to $S^{m,-\infty, -\infty, s}( \mathcal{C}_{\tindex} \times \overline{ ^{\mathrm{b}}T^{\ast} \oplus \mathbb{R}^{n_{\tindex}} } X^{\tindex} )$. Finally, let $\phi^{\tindex} \in \mathcal{C}^{\infty}(X^{\tindex})$ be another cut-off function such that $\supp \phi^{\tindex} \cap \supp \psi^{\tindex} = \emptyset$. Then we will simply require that $\phi^{\tindex} \hat{A} \psi^{\tindex} \in \Psi_{\mathrm{bc,lp}}^{-\infty, l}( X^{\tindex} ; \mathcal{C}_{\tindex} \times {\mathbb{R}^{n_{\tindex}}} )$. \par

 This concludes the construction of (\ref{further resolved large-parameters microlocalized operators}).

\begin{remark} 
\label{further resolved large-parameters microlocalized operators 2}
An important distinction between $\Psi_{\mathrm{d3sc,3co,res}}^{m,r,l, \nu, b, s}(\Xd)$ and $\Psi_{\mathrm{d3sc,3co}}^{m ,r , l ,\nu, b}(\Xd)$, and indeed the motivation for defining the former spaces of operators, is that for any $A \in \Psi_{\mathrm{d3sc,3co,res}}^{m,r,l,\nu,b,s}(\Xd)$, the indicial operator of $A$ at $\dff$, which will be denoted by $\hat{N}_{\dff,l}(A)$ in \S\ref{subsection principal symbol} below, lives in
\begin{equation*}
\Psi_{\mathrm{sc,b,lp,res}}^{m ,r - l, b - l, s-l} ( X^{\tindex} ; \mathcal{C}_{\tindex} \times {\mathbb{R}^{n_{\tindex}}} ).
\end{equation*}
For bounded $\zeta_{\tindex} \in \mathbb{R}^{n_{\tindex}}$ (i.e., the parameter), this allows us to view $\hat{N}_{\dff,l}(A)$ as a smooth family of operators in $\Psi_{\mathrm{sc,b}}^{s-l, r - l, b - l}(X^{\tindex})$. Thus in particular, these operators are in a sense `independent of one another'. This is not the true in the non-resolved case: In the non-resolved case, these operators would have the same principal symbol. This is because in the radial compactification in $(\zeta_{\tindex}, \zeta^{\tindex})$, the closure of $\{ \zeta_{\tindex} = \zeta_{\tindex}^{\ast} \} \times \mathbb{R}^{n^{\tindex}}_{\zeta^{\tindex}}$ meets the joint fiber infinity at a point independent of $\zeta_{\tindex}^{\ast}$.
\end{remark}

We will next define the spaces of operators with variable orders
\begin{equation}
\label{variable order operators scattering large-parameters second microlocalized resolved}
\Psi^{\mathsf{m}, \mathsf{r}, \mathsf{l}, \mathsf{s}}_{\mathrm{sc, b, lp, res},\delta}( X^{\tindex} ; \mathcal{C}_{\tindex} \times {\mathbb{R}^{n_{\tindex}}} )
\end{equation}
depending on some sufficiently small $\delta > 0$. Here, the variable orders are required to satisfy
\begin{equation*}
\begin{gathered}
\mathsf{m} \in \mathcal{C}^{\infty}( \psf_{\dff} \Xd \cap  {^{\mathrm{d3sc, 3co, res}}S^{\ast} } \Xd ), \\
\mathsf{r} \in \mathcal{C}^{\infty}( \overline{^{\mathrm{d3sc, 3co, res}}T^{\ast}}_{\dff} \Xd \cap \dtsccf ), \quad \mathsf{s} \in \mathcal{C}^{\infty}( \overline{^{\mathrm{d3sc, 3co ,res}}T^{\ast}}_{\dff} \Xd \cap \rf ).
\end{gathered}
\end{equation*}
We will also use the same symbols to denote any $\mathcal{C}^{\infty}( \psf_{\dff} \Xd )$ extensions of these functions without further elaboration. On the other hand, $\mathsf{l}$ is required to satisfy the more restrictive, global condition that
\begin{equation}
\label{variable order l large-parameters}
\mathsf{l} \in \mathcal{C}^{\infty}( \mathcal{C}_{\tindex} \times \overline{\mathbb{R}^{n_{\tindex}}} ).
\end{equation}
It is also not hard to see that (\ref{variable order l large-parameters}) implies
\begin{equation} 
\label{variable order l large-parameters realization}
\mathsf{l} \in \mathcal{C}^{\infty}( \psf_{\dff} \Xd \cap \tcocf ).
\end{equation}
Thus in particular, $\vol$ always admits a $\mathcal{C}^{\infty}(  \psf_{\dff} \Xd )$ extension. Such extensions will still be denoted by $\vol$ without further elaboration.
\par

We now consider the spaces of variable orders symbols
\begin{equation}
\label{symbol class large-parameters at dff}
S^{\mathsf{m}, \mathsf{r}, \mathsf{l}, \mathsf{s}}_{\delta}( \overline{^{\mathrm{d3sc,3co,res}}T^{\ast}}_{\dff} \Xd )
\end{equation}
depending on some sufficiently small $\delta > 0$. Abusing notations slightly, let
\begin{equation*}
\rho_{\infty}, \rho_{\dtsccf}, \rho_{\zf}, \rho_{\rf} \in \mathcal{C}^{\infty}( \overline{^{\mathrm{d3sc,3co,res}}T^{\ast}}_{\dff} \Xd )
\end{equation*}
be some boundary defining functions for 
\begin{equation*}
\begin{gathered}
\text{$\psf_{\dff} \Xd \cap {^{\mathrm{d3sc,3co,res}}S^{\ast}}\Xd$,  $\overline{^{\mathrm{d3sc,3co,res}}T^{\ast}}_{\dff} \Xd 
\cap \dtsccf$}, \\
\text{$\overline{^{\mathrm{d3sc,3co,res}}T^{\ast}}_{\dff}\Xd \cap \tcocf$, $\overline{^{\mathrm{3co,res}}T^{\ast}}_{\dff} \Xd \cap \rf$}
\end{gathered}
\end{equation*}
respectively. Then (\ref{symbol class large-parameters at dff}) consists of all functions $a \in \mathcal{C}^{\infty}( \mathcal{C}_{\tindex} \times \mathbb{R}^{n_{\tindex}} \times T^{\ast} \mathbb{R}^{n^{\tindex}} )$ such that:
\begin{itemize}
\item Let $U \subset \mathcal{C}_{\tindex} \times \overline{ ^{\mathrm{b}}T^{\ast} \oplus \mathbb{R}^{n_{\tindex}} }X^{\tindex}$ be a small neighborhood of $\mathcal{C}_{\tindex} \times \overline{ ^{\mathrm{b}}T^{\ast} \oplus \mathbb{R}^{n_{\tindex}} }_{\mathcal{C}^{\tindex}}X^{\tindex}$, and let $\widetilde{U}$ be the lift of $U$ to (\ref{boundary class diffeomorphism 3}). Then near every point of $\widetilde{U}$, we have
\begin{align*}
& | \partial_{y_{\tindex}}^{\beta_{\tindex}} \partial_{t^{\tindex}}^{j} \partial_{y^{\tindex}}^{\beta^{\tindex}} \partial_{ \tscblz }^{  \gamma_{\tindex} } \partial_{ \tscbut }^{k} \partial_{ \tscbum }^{\gamma^{\tindex}} a | \\
& \qquad \leq C_{\beta_{\tindex} j \beta^{\tindex} \gamma_{\tindex} k \gamma^{\tindex}} \rho_{\infty}^{ - \vom + | \gamma_{\tindex} | + k + |\gamma^{\tindex}| - \delta | ( \beta_{\tindex}, j, \beta^{\tindex}, \gamma_{\tindex}, k, \gamma^{\tindex} ) | } \\
& \qquad \quad \times \rho_{\dtsccf}^{ - \vor + k + |\gamma^{\tindex}| - \delta | ( \beta_{\tindex}, j, \beta^{\tindex}, \gamma_{\tindex}, k , \gamma^{\tindex} ) |  } \rho_{\zf}^{ - \vol  - \delta | ( \beta{\tindex}, \gamma_{\tindex} ) | } \rho_{\rf}^{ - \vos + k + |\gamma^{\tindex}| - \delta | ( \beta_{\tindex}, j , \beta^{\tindex}, \gamma_{\tindex}, k , \gamma^{\tindex} ) | }
\end{align*}
in some coordinates of the form $(y_{\tindex}, \tau^{\tindex}, y^{\tindex}, \zeta_{\tindex}, \utaub, \umub )$. 
\item Near any point that is outside of $\widetilde{U}$, we have
\begin{align*}
 | \partial_{y_{\tindex}}^{\beta_{\tindex}} \partial_{z^{\tindex}}^{\beta^{\tindex}} \partial_{\tsclz}^{\gamma_{\tindex}} \partial_{\zeta^{\tindex}}^{\gamma^{\tindex}} a | & \leq C_{\beta_{\tindex} \beta^{\tindex} \gamma_{\tindex} \gamma^{\tindex}} \rho_{\infty}^{ - \vom + |\gamma_{\tindex}| + |\gamma^{\tindex}| - \delta | ( \beta_{\tindex}, \beta^{\tindex}, \gamma_{\tindex}, \gamma^{\tindex} ) | } \rho_{\rf}^{ - \vos  + |\gamma^{\tindex}| - \delta | ( \beta_{\tindex}, \beta^{\tindex}, \gamma_{\tindex}, \gamma^{\tindex} ) | }
\end{align*}
in some coordinates of the form $(y_{\tindex}, z^{\tindex}, \zeta_{\tindex}, \zeta^{\tindex})$.
\end{itemize} \par

The rest of this subsection will be devoted to the construction of (\ref{variable order operators scattering large-parameters second microlocalized resolved}). Suppose that $\hat{A}$ is an element of (\ref{variable order operators scattering large-parameters second microlocalized resolved}). Then $\hat{A}$ is first of all a continuous linear map $\dot{\mathcal{C}}^{\infty}( X^{\tindex} ) \rightarrow \dot{\mathcal{C}}^{\infty}(X^{\tindex})$. Next, let $\psi^{\tindex} \in \mathcal{C}^{\infty}( X^{\tindex} )$ be a cut-off function at some point of $\mathcal{C}^{\tindex}$. Then $\psi^{\tindex} \hat{A} \psi^{\tindex}$ will be required to take the form $\hat{B}_{\psi^{\tindex}} + \hat{R}_{\psi^{\tindex}}$, where $\hat{B}_{\psi^{\tindex}}$ is defined again by (\ref{B in the constant order large parameter indicial operator b calculus}), except that its symbol $b_{\psi^{\tindex}}$ now belongs to $S^{\vom, \vor, \vol, \vos}_{\delta}( \psf_{\dff} \Xd )$. Likewise, if $\psi^{\tindex} \in \mathcal{C}^{\infty}(X^{\tindex})$ is supported away from $\mathcal{C}^{\tindex}$, then $\psi^{\tindex} \hat{A} \psi^{\tindex}$ is given by a standard quantization in the variables $(z^{\tindex}, \zeta^{\tindex})$, with the exception that its symbol now belongs to $S^{\vom, -\infty , -\infty , \vos}_{\delta}( \psf_{\dff} \Xd )$. \par 

Finally, $\hat{R}_{\psi^{\tindex}}$ is required to be smooth in all of the variables, and satisfy additionally the estimates   
\begin{align} \label{R estimate large parameters variable orders}
\begin{split}
& | \partial_{y_{\tindex}}^{\beta_{\tindex}} \partial_{\tscblz}^{\gamma_{\tindex}} \partial_{t^{\tindex}}^{j} \partial_{y^{\tindex}}^{\beta^{\tindex}} \partial_{(t^{\tindex})'}^{j'} \partial_{(y^{\tindex})'}^{ ( \beta^{\tindex} )' } \hat{R}_{\psi^{\tindex}}  | \\
& \qquad \leq C_{ \beta_{\tindex} \gamma_{\tindex} j \beta^{\tindex} j' (\beta^{\tindex})' N LM }   x_{\mathcal{C}^{\tindex}}^{-\vol -\delta | ( \beta_{\tindex}, \gamma_{\tindex} ) |} \langle \lz \rangle^{-N} \langle y^{\tindex} - (y^{\tindex})' \rangle^{-L} e^{  - M | t^{\tindex} - (t^{\tindex})' | }
\end{split}
\end{align} 
for all $N,L, M \in \mathbb{R}$. On the other hand, let $\phi^{\tindex} \in \mathcal{C}^{\infty}(X^{\tindex})$ be another cut-off function such that $\supp \phi^{\tindex} \cap \supp \psi^{\tindex} = \emptyset$, and let $\hat{K}_{\phi^{\tindex}, \psi^{\tindex}}$ denote the kernel of $\phi^{\tindex} \hat{A} \psi^{\tindex}$. Then we will refer again to cases (1)--(3) found in \S \ref{b-calculus subsection} for the characterizing conditions imposed on $\hat{K}_{\phi^{\tindex}, \psi^{\tindex}}$, which depend on the support properties of $\phi^{\tindex}$ and $\psi^{\tindex}$. In fact, these conditions remain the same in cases (1) and (2) (i.e. $\hat{K}_{\phi^{\tindex}, \psi^{\tindex}}$ must be smooth and rapidly decreasing in these cases), while in case (3), we will now require that
\begin{equation*}
| \partial_{y_{\tindex}}^{\beta_{\tindex}} \partial_{\tscblz}^{\gamma_{\tindex}} \partial_{t^{\tindex}}^{j} \partial_{y^{\tindex}}^{\beta^{\tindex}} \partial_{(t^{\tindex})'}^{j'} \partial_{(y^{\tindex})'}^{(\beta^{\tindex})'}  \hat{K}_{\phi^{\tindex}, \psi^{\tindex}}  |  \leq C_{ \beta_{\tindex} \gamma_{\tindex}' j \beta^{\tindex} j' (\beta^{\tindex})' NM } x_{\mathcal{C}^{\tindex}}^{- \vol -\delta | ( \beta_{\tindex}, \gamma_{\tindex} ) |}   \langle \lz \rangle^{-N} e^{  - M|t^{\tindex} - (t^{\tindex})' |}
\end{equation*}
for all $N,M \in \mathbb{R}$. This concludes the construction of (\ref{variable order operators scattering large-parameters second microlocalized resolved}).

\subsection{Second microlocalizing the indicial operators at $\cf$}
\label{subsection second microlocalization for the indicial operators cf}
We now consider second microlocalizing the indicial operators at $\cf$. Recall that we have
\begin{equation*}
\overline{^{\mathrm{3co}}T^{\ast}}_{\mathrm{cf}_{\tindex}} \Xd \cong \mathcal{C}_{\tindex} \times \overline{ ^{\mathrm{co}}T^{\ast} \oplus \mathbb{R}^{n_{\tindex}} } [ \hat{X}^{\tindex} ; \{  0 \} ].
\end{equation*}
Thus, the indicial operators at $\cf$ for elements of $\Psi^{m,r,l,b}_{\mathrm{3coc}}( \Xd )$ (once they are defined) should naturally live in the space of large-parameter conormal cone operators $\Psi_{\mathrm{coc, lp}}^{m,r,l}( [ \hat{X}^{\tindex} ; \{ 0 \} ] ; \mathcal{C}_{\tindex} \times \mathbb{R}^{n_{\tindex}} )$. See also Proposition \ref{membership of operator valued symbol near cf} for how this statement can be realized concretely by letting $|z_{\tindex}| \rightarrow \infty$ in $\hat{A}_{\psi_{\cf}}$, where $\psi_{\cf} \in \mathcal{C}^{\infty}(\Xd)$ is some cut-off function at $\cf$.

\par

Nevertheless, it will still be necessary to resolve $\Psi^{m,r,l}_{\mathrm{coc,lp}}( [ \hat{X}^{\tindex} ; \{ 0 \} ] ; \mathcal{C}_{\tindex} \times \mathbb{R}^{n_{\tindex}} )$ at fiber infinity, since we have
\begin{equation} 
\label{three cone boundary face blow-up at cf}
\tcocf \cong \mathcal{C}_{\tindex} \times [  \overline{ ^{\mathrm{co}}T^{\ast} \oplus \mathbb{R}^{n_{\tindex}} } [ \hat{X}^{\tindex} ; \{  0 \} ] ; {^{\mathrm{co}}S^{\ast}} [ \hat{X}^{\tindex} ; \{ 0 \} ] ].
\end{equation} 
For $m,r,l,s \in \mathbb{R}$, one is then motivated to construct spaces of operators
\begin{equation}
\label{further resolved large-parameters microlocalized operators cf}
\Psi^{m,r,l,s}_{\mathrm{coc,lp,res}}( [ 
\hat{X}^{\tindex} ; \{ 0 \} ] ; \mathcal{C}_{\tindex} \times {\mathbb{R}^{n_{\tindex}}} )
\end{equation}
consisting of those continuous linear maps $\hat{A} : \mathcal{C}^{\infty}( [ \hat{X}^{\tindex} ; \{ 0 \} ] ) \rightarrow \dot{\mathcal{C}}^{\infty}([ \hat{X}^{\tindex} ; \{ 0 \}  
])$ which can be constructed by modifying the construction of $\Psi^{m,r,l}_{\mathrm{coc,lp}}( [ \hat{X}^{\tindex} ; \{ 0 \} ] ; \mathcal{C}_{\tindex} \times {\mathbb{R}^{n_{\tindex}}} )$ in such a way that the symbols of $\hat{A}$ belong to 
\begin{equation*}
S^{m, r , l, s}(  \tcocf  ). 
\end{equation*}
Here, $m$ measures decay at the lift of the `joint fiber infinity' of $\mathcal{C}_{\tindex} \times \overline{ ^{\mathrm{co}}T^{\ast} \oplus \mathbb{R}^{n_{\tindex}} } [ \hat{X}^{\tindex} ; \{ 0 \} ]$; $r$ measures decay at the lift of $\mathcal{C}_{\tindex} \times \overline{ ^{\mathrm{co}}T^{\ast} \oplus \mathbb{R}^{n_{\tindex}} }_{\mathcal{C}^{\tindex}_{\infty}} [ \hat{X}^{\tindex} ; \{ 0 \} ]$; $l$ measures decay at the lift of $\mathcal{C}_{\tindex} \times \overline{ ^{\mathrm{co}}T^{\ast} \oplus \mathbb{R}^{n_{\tindex}} }_{\mathcal{C}^{\tindex}_{0}} [ \hat{X}^{\tindex} ; \{ 0 \} ]$; and $s$ measures decay at the lift of $\mathcal{C}_{\tindex} \times { ^{\mathrm{co}}S^{\ast}} [ \hat{X}^{\tindex} ; \{ 0 \} ]$. While from a more global perspective, these boundary faces of (\ref{three cone boundary face blow-up at cf}) can also be respectively realized as the restrictions to $\tcocf$ of ${^{\mathrm{d3sc,3co,res}}S^{\ast}\Xd}$, $\psf_{\dmf}\Xd$, $\psf_{\dff}\Xd$ and $\dtsccf$.
\par

To this end, let $\psi^{\tindex} \in \mathcal{C}^{\infty}( [ \hat{X}^{\tindex} ; \{ 0 \} ] )$ be a cut-off function at some point of $\mathcal{C}^{\tindex}_{0}$. Then following (\ref{cone algebra b parts}), we will require that 
\begin{equation}
\label{second microlocalizing the indicial operator at cf b part quantization - 1}
\psi^{\tindex} \hat{A} \psi^{\tindex} \coloneq \hat{B}_{\psi^{\tindex}, \mathrm{b}} + \hat{R}_{\psi^{\tindex}, \mathrm{b}},
\end{equation}
where 
\begin{align}
\label{second microlocalizing the indicial operator at cf b part quantization}
\begin{split} 
\hat{B}_{\psi^{\tindex}, \mathrm{b}} \coloneq  {} & \frac{1}{(2\pi)^{n^{\tindex}}} \int_{\mathbb{R}^{n^{\tindex}}}  e^{ - i ( \hat{t}_{\tindex} - (\hat{t}_{\tindex})') \utaucob + i ( y^{\tindex} - (y^{\tindex})'  ) \cdot \umucob  } \varphi(t^{\tindex} - (t^{\tindex})') \varphi( |y^{\tindex} - ( y^{\tindex} )'| )   \\
& \times b_{ \psi^{\tindex}, \mathrm{b} } ( y_{\tindex}, \hat{t}_{\tindex} , y^{\tindex}, \zeta_{\tindex}^{\mathrm{3co}} , \utaucob , \umucob )  d \utaucob d\umucob | d(\hat{t}_{\tindex})' d(y^{\tindex})' |
\end{split}
\end{align} 
for some $b_{\psi^{\tindex}, \mathrm{b}} \in S^{m, -\infty, l, s}( \overline{ ^{\mathrm{d3sc,3co,res}}T^{\ast}}_{\cf} \Xd )$; meanwhile, the kernel of $\hat{R}_{\psi^{\tindex}, \mathrm{b}}$ is required to be smooth in all the variables, and additionally satisfy the estimates 
\begin{equation}
\label{second microlocalizing the indicial operator at cf b part quantization 1}
| \partial_{y_{\tindex}}^{\beta_{\tindex}} \partial^{\gamma_{\tindex}}_{\zeta_{\tindex}^{\mathrm{3co}}} \partial_{\hat{t}_{\tindex}}^{j} \partial_{y^{\tindex}}^{\beta^{\tindex}} \partial_{\hat{t}_{\tindex}'}^{j'} \partial_{(y^{\tindex})'}^{(\beta^{\tindex})'} \hat{R}_{\psi^{\tindex}, \mathrm{b}} | \leq C_{ \beta_{\tindex} \gamma_{\tindex} j \beta^{\tindex} j ' (\beta^{\tindex})' }  x_{\mathcal{C}^{\tindex}_{0}}^{-l} \langle \zeta_{\tindex}^{\mathrm{3co}} \rangle^{-N} \langle y^{\tindex} - (y^{\tindex})' \rangle^{-M} e^{-L| \hat{t}_{\tindex} - \hat{t}_{\tindex}' |}
\end{equation}
for all $N,M, L \in \mathbb{R}$, where $x_{\mathcal{C}^{\tindex}_{0}} \in \mathcal{C}^{\infty}( [ \hat{X}^{\tindex} ; \{ 0 \} ] )$ is a defining function for $\mathcal{C}^{\tindex}_{0}$. 
\par

If instead $\psi^{\tindex} \in \mathcal{C}^{\infty}( [ \hat{X}^{\tindex} ; \{ 0 \} ] )$ is a cut-off function at $\mathcal{C}^{\tindex}_{\infty}$, then we will require that
\begin{equation*}
\psi^{\tindex} \hat{A} \psi^{\tindex} \coloneq \hat{B}_{\psi^{\tindex}, \mathrm{sc}},
\end{equation*}
where 
\begin{equation}
\label{second microlocalizing the indicial operator at cf sc part quantization}
\hat{B}_{\psi^{\tindex}, \mathrm{sc}} \coloneq \frac{1}{(2\pi)^{n^{\tindex}}} \int_{\mathbb{R}^{n^{\tindex}}}  e^{ i( \hat{z}^{\tindex} - (\hat{z}^{\tindex})' ) \cdot \zeta^{\tindex}_{\mathrm{co,sc}} } b_{ \psi^{\tindex}, \mathrm{sc} }( z_{\tindex}, \hat{z}^{\tindex}, \zeta_{\tindex}^{\mathrm{3co}}, \zeta^{\tindex}_{\mathrm{co,sc}} ) d\zeta^{\tindex}_{\mathrm{co,sc}} |d ( \hat{z}^{\tindex} )' |
\end{equation}
for some $b_{\psi^{\tindex}, \mathrm{sc}} \in S^{m, r, -\infty}( \psf_{\cf} \Xd )$. 
\par

Finally, if $\phi^{\tindex} \in \mathcal{C}^{\infty}( [ \hat{X}^{\tindex} ; \{ 0 \} ] )$ is another cut-off function such that $\supp \phi^{\tindex} \cap \supp \psi^{\tindex} = \emptyset$, then we will simply require that $\phi^{\tindex} \hat{A} \psi^{\tindex} \in \Psi_{\mathrm{coc,lp}}^{-\infty, -\infty, l }( [ \hat{X}^{\tindex} ; \{ 0 \} ] ; \mathcal{C}_{\tindex} \times \mathbb{R}^{n_{\tindex}} )$. This concludes the construction of (\ref{further resolved large-parameters microlocalized operators cf}).

\begin{remark} \label{further resolved large-parameters microlocalized operators 2 corner face case}
The analogy of Remark \ref{further resolved large-parameters microlocalized operators 2} applies to the indicial operators at $\cf$ as well. However, in this case, the resolution at the joint fiber infinity of $\mathcal{C}_{\tindex} \times \overline{ ^{\mathrm{co}}T^{\ast} \oplus \mathbb{R}^{n_{\tindex}} } [ \hat{X}^{\tindex} ; \{  0 \} ]$ is no longer a consequence of the `resolved' blow-up, but rather the fiber infinity blow-up from the 3co-perspective, i.e., the blow-up which creates $\dtsccf$.
\end{remark}

We will now allow for variable orders, and define the spaces of operators  
\begin{equation}
\label{large-parameters conic resolved algebra variable orders}
\Psi_{\mathrm{coc,lp,res}, \delta}^{\vom, \vor, \vol, \vos}( [ \hat{X}^{\tindex} ; \{ 0 \} ] ; \mathcal{C}_{\tindex} \times \mathbb{R}^{n_{\tindex}} ),
\end{equation}
depending on some sufficiently small $\delta >0$. Here, the variable orders are required to satisfy
\begin{equation*}
\begin{gathered}
\vom \in \mathcal{C}^{\infty}( \tcocf \cap {^{\mathrm{d3sc,3co,res}}S^{\ast}}\Xd ), \\
\vor \in \mathcal{C}^{\infty}( \tcocf \cap \psf_{\dmf} \Xd ), \quad \vos \in \mathcal{C}^{\infty}( \tcocf \cap \dtsccf ).
\end{gathered}
\end{equation*}
We will also use the same symbols to denote any $\mathcal{C}^{\infty}( \tcocf )$ extensions of these functions without further elaboration. On the other hand, $\vol$ is required to satisfy
\begin{equation}
\label{variable order l large-parameters near cf}
\vol \in \mathcal{C}^{\infty}( \mathcal{C}_{\tindex} \times \overline{\mathbb{R}^{n_{\tindex}}} ).
\end{equation}
However, it is again easy to see that (\ref{variable order l large-parameters near cf}) implies 
\begin{equation*}
\vol \in \mathcal{C}^{\infty}( \tcocf \cap \psf_{\dff} \Xd ).
\end{equation*}
In particular, $\vol$ always admits a $\mathcal{C}^{\infty}(\tcocf)$ extension. Such extensions will still be denoted by $\vol$ without further elaboration.
\par

We also need to define the spaces of symbols with variable orders
\begin{equation} 
\label{symbol class restriction at cf variable orders}
S^{\mathsf{m},\mathsf{r}, \mathsf{l}, \mathsf{s}}_{\delta}( \tcocf )
\end{equation}
depending on some sufficiently small $\delta > 0$. With some abuse of notations, we will let 
\begin{equation*}
\rho_{\infty}, \rho_{\dmf}, \rho_{\dff} , \rho_{\dtsccf} \in \mathcal{C}^{\infty}( \tcocf )
\end{equation*}
be some boundary defining functions for 
\begin{equation*}
\begin{gathered}
\text{$\tcocf \cap { ^{\mathrm{d3sc,3co,res}}S^{\ast} \Xd }$, $\tcocf \cap \psf_{\dmf} \Xd$,} \\
\text{$\tcocf \cap \psf_{\dff} \Xd$, $\tcocf \cap \dtsccf$}
\end{gathered}
\end{equation*}
respectively. Then (\ref{symbol class restriction at cf variable orders}) consists of all functions $a \in \mathcal{C}^{\infty}( \mathcal{C}_{\tindex} \times \mathbb{R}^{n_{\tindex}} \times T^{\ast} \mathbb{R}^{n^{\tindex}} )$ such that:

\begin{itemize}
\item Let $U \subset \overline{ ^{\mathrm{co}}T^{\ast} \oplus \mathbb{R}^{n_{\tindex}} } [ \hat{X}^{\tindex} ; \{ 0 \} ]$ be a neighborhood of $\mathcal{C}_{\tindex} \times \overline{ ^{\mathrm{co}}T^{\ast} \oplus \mathbb{R}^{n_{\tindex}} }_{\mathcal{C}^{\tindex}_{0}} [ \hat{X}^{\tindex} ; \{ 0 \} ]$ excluding $\mathcal{C}_{\tindex} \times \overline{ ^{\mathrm{co}}T^{\ast} \oplus \mathbb{R}^{n_{\tindex}} }_{\mathcal{C}^{\tindex}_{\infty}} [ \hat{X}^{\tindex} ; \{ 0 \} ]$, and let $\widetilde{U}$ be the lift of $U$ to (\ref{three cone boundary face blow-up at cf}). Then near every point of $\widetilde{U}$, we have
\begin{align*}
| \partial_{y_{\tindex}}^{\beta_{\tindex}} \partial_{ \hat{t}_{\tindex} }^{j} \partial_{y^{\tindex}}^{\beta^{\tindex}} \partial_{ \tcoblz }^{  \gamma_{\tindex} } \partial_{ \tcobut }^{k} \partial_{ \tcobum }^{\gamma^{\tindex}} a | \leq {} & C_{\beta_{\tindex} j \beta^{\tindex} \gamma_{\tindex} k \gamma^{\tindex}} \rho_{\infty}^{ - \vom + | \gamma_{\tindex} | + k + |\gamma^{\tindex}| - \delta | ( \beta_{\tindex}, j, \beta^{\tindex}, \gamma_{\tindex}, k, \gamma^{\tindex} ) | }  \\
& \times \rho_{\dff}^{ - \vol - \delta | ( \beta_{\tindex}, \gamma_{\tindex} ) |  }  \rho_{\dtsccf}^{ - \vor + k + |\gamma^{\tindex}| - \delta | ( \beta_{\tindex}, j, \beta^{\tindex}, \gamma_{\tindex}, k , \gamma^{\tindex} ) |  }
\end{align*}
in some coordinates of the form $( y_{\tindex}, \hat{t}_{\tindex}, y^{\tindex}, \tcottlz, \utaucob, \umucob )$.
\item Let $U \subset \overline{ ^{\mathrm{co}}T^{\ast} \oplus \mathbb{R}^{n_{\tindex}} } [ \hat{X}^{\tindex} ; \{ 0 \} ]$ be a neighborhood of $\mathcal{C}_{\tindex} \times \overline{ ^{\mathrm{co}}T^{\ast} \oplus \mathbb{R}^{n_{\tindex}} }_{\mathcal{C}^{\tindex}_{\infty}} [ \hat{X}^{\tindex} ; \{ 0 \} ]$ excluding $\mathcal{C}_{\tindex} \times \overline{ ^{\mathrm{co}}T^{\ast} \oplus \mathbb{R}^{n_{\tindex}} }_{\mathcal{C}^{\tindex}_{0}} [ \hat{X}^{\tindex} ; \{ 0 \} ]$, and let $\widetilde{U}$ be the lift of $U$ to (\ref{three cone boundary face blow-up at cf}). Then near every point of $\widetilde{U}$, we have
\begin{align*}
| \partial_{y_{\tindex}}^{\beta_{\tindex}} \partial_{\hat{z}^{\tindex}}^{\beta^{\tindex}} \partial_{\tcosclz}^{\gamma_{\tindex}} \partial_{\tcoscuz}^{\gamma^{\tindex}} a | \leq {} &  C_{\beta_{\tindex} \beta^{\tindex} \gamma_{\tindex} \gamma^{\tindex}} \rho_{\infty}^{ 
- \vom + |\gamma_{\tindex}| + |\gamma^{\tindex}| - \delta | ( \beta_{\tindex}, \beta^{\tindex}, \gamma_{\tindex}, \gamma^{\tindex} ) | } \\
& \times \rho_{\dmf}^{ - \vor + |\beta_{\tindex}|  - \delta | ( \beta_{\tindex}, \beta^{\tindex}, \gamma_{\tindex}, \gamma^{\tindex} ) | } \rho_{\dtsccf}^{- \vov + |\gamma^{\tindex}| - \delta | ( \beta_{\tindex}, \beta^{\tindex}, \gamma_{\tindex}, \gamma^{\tindex} ) | } 
\end{align*}
in some coordinates of the form $(y_{\tindex}, \hat{z}^{\tindex}, \tcosclz, \utaucosc, \umucosc )$.
\end{itemize}

The rest of this subsection will be devoted to the construction of (\ref{variable order operators scattering large-parameters second microlocalized resolved}). Suppose that $\hat{A}$ is an element of (\ref{variable order operators scattering large-parameters second microlocalized resolved}). Then $\hat{A}$ is first of all a continuous linear map $\dot{\mathcal{C}}^{\infty}( [ \hat{X}^{\tindex} ; \{ 0 \} ] ) \rightarrow \dot{\mathcal{C}}^{\infty}( [ \hat{X}^{\tindex} ; \{ 0 \} ] )$. Next, let $\psi^{\tindex} \in \mathcal{C}^{\infty}( [ \hat{X}^{\tindex} ; \{ 0 \} ] )$ be a cut-off function at some point of $\mathcal{C}^{\tindex}_{0}$. Then $\psi^{\tindex} \hat{A} \psi^{\tindex}$ must take the form $\hat{B}_{\psi^{\tindex}, \mathrm{b}} + \hat{R}_{\psi^{\tindex}, \mathrm{b}}$, where $\hat{B}_{\psi^{\tindex}, \mathrm{b}}$ is defined again by (\ref{second microlocalizing the indicial operator at cf b part quantization}), except that its symbol $b_{\psi^{\tindex}, \mathrm{b}}$ now belongs to $S^{\vom, -\infty, \vol, \vos}_{\delta}( \tcocf )$. Likewise, if $\psi^{\tindex} \in \mathcal{C}^{\infty}( [ \hat{X}^{\tindex} ; \{ 0 \} ] )$ is supported away from $\mathcal{C}^{\tindex}_{0}$, then $\psi^{\tindex} \hat{A} \psi^{\tindex} \coloneq \hat{B}_{\psi^{\tindex}, \mathrm{sc}}$, where $\hat{B}_{\psi^{\tindex}, \mathrm{sc}}$ is defined again by (\ref{second microlocalizing the indicial operator at cf sc part quantization}), except that its symbol $b_{\psi^{\tindex}, \mathrm{sc}}$ now belongs to $S_{\delta}^{\vom, \vor, -\infty, \vos}( \tcocf )$.
\par

Finally, $\hat{R}_{\psi^{\tindex}, \mathrm{b}}$ is required to be smooth in all of the variables, and satisfy additionally the estimates
\begin{align*} 
\begin{split}
& | \partial_{y_{\tindex}}^{\beta_{\tindex}} \partial_{\tcottlz}^{\gamma_{\tindex}} \partial_{ \hat{t}_{\tindex} }^{j} \partial_{y^{\tindex}}^{\beta^{\tindex}} \partial_{ \hat{t}_{\tindex}' }^{j'} \partial_{(y^{\tindex})'}^{ ( \beta^{\tindex} )' } \hat{R}_{\psi^{\tindex}, \mathrm{b}}  | \\
& \qquad \leq C_{ \beta_{\tindex} \gamma_{\tindex} j \beta^{\tindex} j' (\beta^{\tindex})' N LM }  x_{\mathcal{C}^{\tindex}_{0}}^{ - \vol -\delta | ( \beta_{\tindex}, \gamma_{\tindex} ) |} \langle \tcottlz \rangle^{-N} \langle y^{\tindex} - (y^{\tindex})' \rangle^{-L} e^{  - M | \hat{t}_{\tindex} - \hat{t}_{\tindex}' | } 
\end{split}
\end{align*}
for all $N,M, L \in \mathbb{R}$. On the other hand, let $\phi^{\tindex} \in \mathcal{C}^{\infty}( [ \hat{X}^{\tindex} ; \{ 0 \} ] )$ be another cut-off function such that $\supp \phi^{\tindex} \cap \supp \psi^{\tindex} = \emptyset$, and let $\hat{K}_{\phi^{\tindex}, \psi^{\tindex}}$ denote the kernel of $\phi^{\tindex} \hat{A} \psi^{\tindex}$. Then we will refer again to cases (1)--(6) found in \S \ref{subsection the cone calculus} for the characterizing conditions imposed on $\hat{K}_{\phi^{\tindex}, \psi^{\tindex}}$, which depend on the support properties of $\phi^{\tindex}$ and $\psi^{\tindex}$. In fact, these conditions remain the same in cases (1)--(5) (i.e., $\hat{K}_{\phi^{\tindex}, \psi^{\tindex}}$ must be smooth and rapidly decreasing in these cases), while in case (6), we will now require that
\begin{equation*}
| \partial_{y_{\tindex}}^{\beta_{\tindex}} \partial_{\tcoblz}^{\gamma_{\tindex}} \partial_{ \hat{t}_{\tindex} }^{j} \partial_{y^{\tindex}}^{\beta^{\tindex}} \partial_{ \hat{t}_{\tindex}' }^{j'} \partial_{(y^{\tindex})'}^{(\beta^{\tindex})'}  \hat{K}_{\phi^{\tindex}, \psi^{\tindex}} | \leq C_{ \beta_{\tindex} \gamma_{\tindex}' j \beta^{\tindex} j' (\beta^{\tindex})' NM } x_{\mathcal{C}^{\tindex}_{0}}^{-\vol -\delta | ( \beta_{\tindex}, \gamma_{\tindex} ) |}   \langle \tcottlz \rangle^{-N} e^{  - M| \hat{t}_{\tindex} - \hat{t}_{\tindex}' |}
\end{equation*}
for all $N,M \in \mathbb{R}$. This concludes the construction of (\ref{large-parameters conic resolved algebra variable orders}).

\section{Basic microlocal properties}
\label{Section: Microlocal properties of the second microlocalized algebra}

\subsection{Principal symbol maps} 
\label{subsection principal symbol}
We now move onto the discussions of the central tools used to microlocally capture operator decay at infinities. For both the conormal three-cone as well as the second microlocalized operators (which arise by resolving the conormal three-cone oeprators), such tools will come in two flavors: the principal symbol, which is standard, and thus easily understood; and the indicial operators, which are inherently global objects in nature, and therefore warrants a more detailed discussion. \par

We will first consider the principal symbol, which is a universal object in microlocal analysis, and indeed well-understood in many settings. For a typical pseudodifferential algebra, the principal symbol measures leading order differential regularity, or equivalently leading order decay at the fiber infinity from the phase space perspective. However, it is also known that the principal symbol, when suitably interpreted, can be used to measure leading order decays jointly at other boundary faces of the phase space as well, provided certain properties are satisfied (see \cite[\S 2]{PeterLinearWave1} for a brief discussion on this). \par

For examples, in our review of the scattering calculus in \S \ref{subsection the scattering calculus}, we have seen that the principal symbol map captures leading order decay also at the spatial infinity $\overline{^{\mathrm{sc}}T^{\ast}}_{\mathbb{S}^{n-1}} \overline{\mathbb{R}^{n}}$. Thus overall, the principal symbol captures leading order decay \emph{simultaneously} at both spatial and fiber infinities in this setting. Meanwhile, in the context of the conormal b-calculus, i.e., \S \ref{b-calculus subsection}, the principal symbol map captures only leading order differential regularity. \par

In the case of the three-cone operators, there exists a principal symbol map which captures leading order decay at $^{\mathrm{3co}}S^{\ast}\Xd$ and $\overline{ ^{\mathrm{3co}}T^{\ast} }_{\dmf}\Xd$ simultaneously. Roughly speaking, this follows from the fact that the conormal three-cone operators are designed to quantize a Lie algebra of vector fields, i.e., $\mathcal{V}_{\mathrm{3co}}(\Xd)$, which has a scattering structure at $\dmf$.    \par

This motivates us to state the following proposition:

\begin{proposition}[Principal symbol map for the conormal three-cone operators]
\label{proposition principal symbol for the conormal three-cone operators}
Suppose that $m,r,l,b \in \mathbb{R}$. Then there exists a map
\begin{equation*}
{ ^{\mathrm{3co}} \sigma_{m,r,l,b} }: \Psi_{\mathrm{3co}}^{m,r,l,b} ( \Xd ) \rightarrow ( S^{m, r, l , b} / S^{m-1,r-1,l,b}  )( \overline{  { ^{\mathrm{3co}}T^{\ast} } } \Xd ),
\end{equation*}
which gives rise to a short exact sequence
\begin{equation*}
0 \rightarrow \Psi_{\mathrm{3coc}}^{m - 1, r - 1 , l, b}( \Xd ) \rightarrow \Psi_{\mathrm{3coc}}^{m,r,l,b}( \Xd ) \rightarrow_{ { ^{\mathrm{3co}} }\sigma_{m,r,l,b} } ( S^{m,r,l,b} / S^{m-1,r-1,l,b} ) ( \overline{ ^{\mathrm{3co}} T^{\ast} } X ) \rightarrow 0. 
\end{equation*}
Thus ${ ^{\mathrm{3co}}\sigma_{m,r,l,b} }$ descends to an isomorphism 
\begin{equation*}
{ ^{\mathrm{3co}}\sigma_{m,r,l,b} } : ( \Psi^{m,r,l,b}_{\mathrm{3coc}}  / \Psi_{\mathrm{3coc}}^{m-1,r-1, l, b}  ) ( \Xd ) \xrightarrow{\sim}( S^{m,r,l,b} / S^{ m - 1, r - 1 , l , b }  ) ( \overline{ ^{\mathrm{3co}} T^{\ast} } \Xd ).
\end{equation*}
\end{proposition}

Now, since $\dtsccf$ and $\rf$ both arise from blowing up faces of $\psf \Xd$ which intersect $^{\mathrm{3co}}S^{\ast}\Xd$, it should not be surprising that there exists a principal symbol map for the class of further-resolved, second microlocalized operators which captures leading order decay at $^{\mathrm{d3sc,3co,res}}S^{\ast}\Xd$, $\psf_{\dmf} \Xd$, $\dtsccf$ and $\rf$ simultaneously.

\begin{proposition}[Principal symbol map for the further-resolved second microlocalized operators] 
\label{proposition principal symbol map for the further-resolved second microlocalized operators}
Suppose that $\vom,\vor, \vov, \vos \in \mathcal{C}^{\infty}( \psf \Xd)$, $\vol, \vob \in \mathcal{C}^{\infty}( \mathcal{C}_{\tindex} \times \overline{\mathbb{R}^{n_{\tindex}}} )$. Then there exists a map
\begin{align}
\label{symbol map d3sc 3co res}
\begin{split}
 { ^{\mathrm{d3sc,3co,res}}\sigma_{\vom, \vor, \vol, \vov, \vob, \vos, \delta} } & : \Psi^{ \mathsf{m}, \mathsf{r}, \mathsf{l}, \mathsf{v}, \mathsf{b}, \mathsf{s} }_{\mathrm{d3sc,3co,res}, \delta} (\Xd) \\
& \qquad  \longrightarrow ( S^{ \mathsf{m}, \mathsf{r} , \mathsf{l}, \mathsf{v}, \mathsf{b}, \mathsf{s} }_{\delta}  / S_{\delta}^{ \vom - 1 + 2\delta, \vor - 1 + 2\delta, \vol, \vov - 1 + 2\delta, \vob , \vos - 1 + 2\delta } ) ( \overline{^{\mathrm{d3sc,3co,res}}T^{\ast}} \Xd ),
\end{split}
\end{align}
which gives rise to a short exact sequence
\begin{align*}
& 0  \rightarrow  \Psi^{ \vom -1 + 2 \delta, \vor - 1 + 2\delta, \vol , \vov - 1 + 2 \delta,  \vob , \vos - 1 +2 \delta  }_{\mathrm{d3sc,3co,res}, \delta} ( \Xd ) \rightarrow \Psi^{ \vom , \vor , \vol , \vov , \vob , \vos }_{\mathrm{d3sc,3co,res}, \delta}( \Xd ) \\
& \qquad \longrightarrow_{{ ^{\mathrm{d3sc,3co,res}}\sigma_{\vom, \vor, \vol, \vov, \vob, \vos, \delta} } } (  { S^{ \vom , \vor , \vol , \vov , \vob , \vos }_{\delta} } / { S^{ \vom - 1 + 2 \delta, \vor - 1 + 2\delta, \vol , \vov - 1 + 2 \delta, \vob , \vos - 1 + 2 \delta }_{\delta} } ) ( \overline{ ^{\mathrm{d3sc,3co,res}}T^{\ast}} \Xd ) \rightarrow 0.
\end{align*}
Thus ${ ^{\mathrm{d3sc,3co,res}}\sigma_{\vom, \vor, \vol, \vov, \vob, \vos, \delta} } $ descends to an isomorphism
\begin{align*}
{ ^{\mathrm{d3sc,3co,res}}\sigma_{\vom, \vor, \vol, \vov, \vob, \vos, \delta} }  & : ( { 
\Psi^{ \mathsf{m} , \mathsf{r}, \mathsf{l}, \mathsf{v}, \mathsf{b}, \mathsf{s} }_{\mathrm{d3sc,3co,res}, \delta} } / { \Psi^{ \mathsf{m} - 1 + 2 \delta, \mathsf{r} - 1 + 2\delta, \mathsf{l}, \mathsf{v} - 1 + 2 \delta,  \mathsf{b},  \mathsf{s} - 1 + 2 \delta }_{\mathrm{d3sc,3co,res},\delta} } ) ( \Xd ) \\
& \qquad \xlongrightarrow{\sim} (  { S^{ \mathsf{m} , \mathsf{r}, \mathsf{l}, \mathsf{v}, \mathsf{b} , \mathsf{s} }_{\delta} } / { S^{ \mathsf{m} - 1 + 2 \delta, \mathsf{r} - 1 + 2\delta, \mathsf{l}, \mathsf{v} - 1 + 2 \delta, \mathsf{b} , \mathsf{s} - 1 + 2 \delta }_{\delta} } ) ( \overline{ ^{\mathrm{d3sc,3co,res}}T^{\ast}} \Xd ).
\end{align*} \par

Moreover, if any of the variable orders $\vom$, $\vor$, $\vov$, or $\vos$ is actually constant, then the loss of $2\delta$ order corresponding to that index can be removed. In particular, if all of them are constants with $\vom = m$, $\vor = r$, $\vol = l$, $\vov = \nu$, $\vob = b$, $\vos = s$, then the above statements are valid with $\Psi_{\mathrm{d3sc,3co,res}}^{m,r,l,\nu,b,s}(\Xd)$ in places of $\Psi_{\mathrm{d3sc,3co,res}, \delta}^{\vom,\vor,\vol,\vov,\vob,\vos}(\Xd)$.
\end{proposition}
\begin{remark}
\label{remark no need to consider d3sc 3co principal symbol map}
For $m,r,l,\nu,b \in \mathbb{R}$, it is also possible to define the principal symbol map for $\Psi_{\mathrm{d3sc,3co}}^{m,r,l,\nu,b}(\Xd)$, which leads to an isomorphism
\begin{equation*}
( \Psi_{\mathrm{d3sc,3co}}^{m,r,l, s} / \Psi_{\mathrm{d3sc,3co}}^{m-1, r- 1, l, s- 1} )(X) \cong  ( S^{m, r, l, \nu, b} / S^{m-1, r- 1, l, \nu - 1, b} )( \overline{ ^{\mathrm{d3sc,3co}}T^{\ast}}\Xd ).
\end{equation*}
However, we will not be interested in the properties of $\Psi_{\mathrm{d3sc,3co}}^{m,r,l,\nu,b}(\Xd)$ for what remains of this paper, since it is in a sense inadequate at $\dff$. See Remark \ref{further resolved large-parameters microlocalized operators 2}. 
\end{remark}
\begin{proof}[Proofs of Proposition \ref{proposition principal symbol map for the further-resolved second microlocalized operators}]
Let $\psi_{j}, \tilde{\psi}_{j} \in \mathcal{C}^{\infty}(\Xd)$, $j=1,...,J$ be chosen such that:
\begin{itemize}
\item $\sum_{j=1}^{J} \tilde{\psi}_{j} = 1$ near $\partial \Xd$; 
\item $\psi_{j} = 1$ on the support of $\tilde{\psi}_{j}$, $j=1,...,J$;
\item the supports of $\psi_{j}$, $\tilde{\psi}_{j}$ are as small as we need them.
\end{itemize}
Then modulo the addition of an element in $\Psf^{-\infty, -\infty, \vol, -\infty, \vob, -\infty}(\Xd)$, we have
\begin{equation*}
A = \sum_{j=1}^{J} \psi_{j} A \psi_{j} \tilde{\psi}_{j}.
\end{equation*} \par

Now, depending on support of $\psi_{j}$, either we have $\psi_{j} A \psi_{j} \in \Psi_{\mathrm{sc}, \delta}^{\vom, \vor}( \overline{\mathbb{R}^{n}} )$, or they are given by one of the quantizations $B_{\psi_{j}, \mathrm{3co,b}}$, $B_{\psi_{j}, \mathrm{3co,co,b}}$, $B_{\psi_{j}, \mathrm{3co,co,sc}}$, $B_{\psi_{j}, \mathrm{3sc}}$ as defined in the previous sections. Here, we will also emphasize that the terms $B_{\psi_{j}, \mathrm{3co,b}}$, $B_{\psi_{j}, \mathrm{3co,co,b}}$ are interchangeable, as $\psi_{j} A \psi_{j}$ could be written in either forms modulo a term in $\Psf^{-\infty, -\infty, \vol, -\infty, -\vob, -\infty}(\Xd)$, whenever $\supp \psi_{j}$ intersects $\dff \cap \cf$. Moreover, terms of the forms $R_{\psi_{j}, \mathrm{3co,b}}, R_{\psi_{j}, \mathrm{3co,co,b}} \in \Psf^{-\infty, -\infty, \vol, -\infty, \vob, -\infty}(\Xd)$ have been modded out as well.  \par

Let $b_{\psi_{j}} \in S^{\vom, \vor, \vol, \vov, \vob, \vos}_{\delta} ( \psf \Xd )$ denote the symbol of $\psi_{j} A \psi_{j}$. Then the above procedure gives rise to a well-defined `full symbol' of $A$ which depend on the choice of the partition $\{ \psi_{j} \}_{j=1}^{J}$. Nevertheless, this construction can be made invariant by using the Kuranishi trick. We shall omit this calculation, and simply point out that the key observations are of follows: Firstly, the symbolic estimates presented in Definition \ref{definition of variable orders symbols} are (essentially, up to a sign) all written in canonical coordinates. Correspondingly, the quantizations $B_{\psi_{j}, \mathrm{3co,b}}$, $B_{\psi_{j}, \mathrm{3co,b}}$, $B_{\psi_{j}, \mathrm{3co,co,sc}}$ and $B_{\psi_{j}, \mathrm{3sc}}$ are also defined in such coordinates. Secondly, by simultaneously taking one order of derivatives in a spatial variable and its dual momentum variable, we gain at least $1-2\delta$ order of decay at $\psf_{\dmf}\Xd$, $\dtsccf$, $\rf$ and the fiber infinity. \par

The rest of the proof is standard.
\end{proof}

The proof of Proposition \ref{proposition principal symbol for the conormal three-cone operators} is similar.

Finally, principal symbols can also be defined for the classes of second microlocalized, large-parameter operators introduced in \S\S \ref{subsection second microlocalization for the indicial operators dff}--\ref{subsection second microlocalization for the indicial operators cf}.

\begin{proposition}[Principal symbol maps for the indicial operators] 
Suppose that $\vom, \vor, \vos \in \mathcal{C}^{\infty}( \psf_{\dff} \Xd )$, $\vol \in \mathcal{C}^{\infty}( \mathcal{C}_{\tindex} \times \overline{\mathbb{R}^{n_{\tindex}}} )$. Then there exists a map
\begin{align*}
{^{\mathrm{sc,b,lp,res}}\sigma_{\vom,\vor, \vol, \vos, \delta}} & : \Psi_{\mathrm{sc,b,lp,res}, \delta}^{\vom, \vor, \vol, \vos}( X^{\tindex} ; \mathcal{C}_{\tindex} \times {\mathbb{R}^{n_{\tindex}}} ) \\
& \qquad \longrightarrow ( S^{\vom, \vor, \vol, \vos}_{\delta} / S_{\delta}^{\vom - 1 + 2\delta, \vor - 1 + 2 \delta, \vol, \vos - 1 + 2 \delta} ) ( \psf_{\dff} \Xd )
\end{align*}
and a short exact sequence
\begin{align*}
& 0 \rightarrow  \Psi_{\mathrm{sc,b,lp,res},\delta}^{\vom, \vor, \vol, \vos} ( X^{\tindex} ; \mathcal{C}_{\tindex} \times {\mathbb{R}^{n_{\tindex}}} )\rightarrow  \Psi_{\mathrm{sc,b,lp,res},\delta}^{\vom - 1 + 2 \delta, \vor - 1 + 2 \delta, \vol, \vos - 1 + 2\delta}( X^{\tindex} ; \mathcal{C}_{\tindex} \times {\mathbb{R}^{n_{\tindex}}} ) \\
& \qquad \longrightarrow_{{^{\mathrm{sc,b,lp,res}}\sigma_{\vom,\vor, \vol, \vos, \delta}} } ( S^{\vom, \vor, \vol, \vos}_{\delta} / S^{\vom - 1 + 2\delta, \vor - 1 + 2\delta, \vol, \vos - 1 + 2\delta}_{\delta} ) ( \psf_{\dff} \Xd ) \rightarrow 0.
\end{align*}
Thus ${^{\mathrm{sc,b,lp,res}}\sigma_{\vom,\vor, \vol, \vos, \delta}} $ descends to a short exact sequence
\begin{align*}
{ ^{\mathrm{sc,b,lp,res}} \sigma_{\vom, \vor, \vol, \vos, \delta}} & : ( \Psi_{\mathrm{sc,b,lp,res}, \delta}^{\vom, \vor, \vol, \vos} / \Psi_{\mathrm{sc,b,lp,res}, \delta}^{\vom - 1 + 2 \delta, \vor - 1 + 2 \delta , \vol, \vos - 1 + 2 \delta} )( X^{\tindex} ; \mathcal{C}_{\tindex} \times {\mathbb{R}^{n_{\tindex}}} ) \\
& \qquad \xlongrightarrow{\sim} ( S^{\vom, \vor, \vol, \vos}_{\delta} / S^{\vom - 1 + 2 \delta , \vor - 1 + 2\delta , \vol, \vos - 1 + 2\delta}_{\delta} )( \psf_{\dff} \Xd ).
\end{align*} \par

Moreover, if any of the variable orders $\vom$, $\vor$, $\vol$, or $\vos$ is actually constant, then the loss of $2\delta$ order corresponding to that index can be removed. In particular, if all of them are constants with $\vom = m$, $\vor = r$, $\vol = l$, $\vos = s$, then the above results are valid with $\Psi_{\mathrm{sc,b,lp,res}}^{m,r,l,s}(X^{\tindex} ; \mathcal{C}_{\tindex} \times \mathbb{R}^{n_{\tindex}})$ in places of $\Psi_{\mathrm{sc,b,lp,res}, \delta}^{\vom,\vor,\vol,\vos}(X^{\tindex} ; \mathcal{C}_{\tindex} \times \mathbb{R}^{n_{\tindex}})$.

Likewise, suppose that $\vom,\vor, \vos \in \mathcal{C}^{\infty}( \tcocf )$, $\vol \in \mathcal{C}^{\infty}( \mathcal{C}_{\tindex} \times \overline{ \mathbb{R}^{n_{\tindex}} } )$. Then there exists a map
\begin{align*}
{^{\mathrm{co,lp,res}}\sigma_{\vom,\vor, \vol, \vos, \delta}} & : \Psi_{\mathrm{coc,lp,res}, \delta}^{\vom, \vor, \vol, \vos}( [ \hat{X}^{\tindex} ; \{  0\} ]  ; \mathcal{C}_{\tindex} \times {\mathbb{R}^{n_{\tindex}}} ) \rightarrow ( S^{\vom, \vor, \vol, \vos}_{\delta} / S_{\delta}^{\vom - 1 + 2\delta, \vor - 1 + 2 \delta, \vol, \vos - 1 + 2 \delta} ) ( \tcocf )
\end{align*}
and a short exact sequence
\begin{align*}
& 0 \rightarrow  \Psi_{\mathrm{coc,lp,res},\delta}^{\vom, \vor, \vol, \vos} ( [ \hat{X}^{\tindex} ; \{ 0 \} ] ; \mathcal{C}_{\tindex} \times {\mathbb{R}^{n_{\tindex}}} )\rightarrow  \Psi_{\mathrm{coc,lp,res},\delta}^{\vom - 1 + 2 \delta, \vor - 1 + 2 \delta, \vol, \vos - 1 + 2\delta}( [ \hat{X}^{\tindex} ; \{ 0 \} ] ; \mathcal{C}_{\tindex} \times {\mathbb{R}^{n_{\tindex}}} ) \\
& \qquad \qquad \qquad \qquad \qquad  \quad \longrightarrow_{{^{\mathrm{coc,lp,res}}\sigma_{\vom,\vor, \vol, \vos, \delta}} } ( S^{\vom, \vor, \vol, \vos}_{\delta} / S^{\vom - 1 + 2\delta, \vor - 1 + 2\delta, \vol, \vos - 1 + 2\delta}_{\delta} ) ( \tcocf ) \rightarrow 0.
\end{align*}
Thus ${^{\mathrm{coc,lp,res}}\sigma_{\vom,\vor, \vol, \vos, \delta}}$ descends to a short exact sequence
\begin{align*}
{ ^{\mathrm{coc,lp,res}} \sigma_{\vom, \vor, \vol, \vos, \delta}} & : ( \Psi_{\mathrm{coc,lp,res}, \delta}^{\vom, \vor, \vol, \vos} / \Psi_{\mathrm{coc,lp,res}, \delta}^{\vom - 1 + 2 \delta, \vor - 1 + 2 \delta , \vol, \vos - 1 + 2 \delta} )( [ \hat{X}^{\tindex} ; \{ 0 \} ] ; \mathcal{C}_{\tindex} \times {\mathbb{R}^{n_{\tindex}}} ) \\
& \qquad \qquad \qquad \qquad \xlongrightarrow{\sim} ( S^{\vom, \vor, \vol, \vos}_{\delta} / S^{\vom - 1 + 2 \delta , \vor - 1 + 2\delta , \vol, \vos - 1 + 2\delta}_{\delta} )( \tcocf ).
\end{align*}

Moreover, if any of the variable orders $\vom$, $\vor$, $\vol$, or $\vos$ is actually constant, then the loss of $2\delta$ order corresponding to that index can be removed. In particular, if all of them are constants with $\vom = m$, $\vor = r$, $\vol = l$, $\vos = s$, then the above results are valid with $\Psi_{\mathrm{coc,lp,res}}^{m,r,l,s}( [ \hat{X}^{\tindex} ; \{ 0 \} ] ; \mathcal{C}_{\tindex} \times \mathbb{R}^{n_{\tindex}} )$ in places of $\Psi_{\mathrm{coc,lp,res}, \delta}^{\vom, \vor, \vol, \vos}( [ \hat{X}^{\tindex} ; \{ 0 \} ] ; \mathcal{C}_{\tindex} \times \mathbb{R}^{n_{\tindex}} ))$.
\end{proposition}

For the sake of brevity, we will often denote the principal symbol map by $\sigma$ no matter the context{\ep}so long as it is clear which class of operators we are working with.

\subsection{Partial classicality and indicial operators}
\label{subsection partial classicality and indicial operators}
Next we consider indicial operators. In the contexts of either the conormal three-cone or the second microlocalized operators (which arise from resolving the conormal three-cone operators), there are two types of indicial operators we need to consider for a fully global analysis. They are separately used to measure leading order decay at $\dff$ and $\zf$, where the principal symbol fails to do so.  \par

Since conormality does not ensure smoothness (i.e., classicality), it should not be surprsiing that we cannot define indicial operators for arbitrary conormal three-cone or second microlocalized operators. Instead, we will need to assume that the operators in question be \emph{partially classical} at $\dff$ or $\cf$ in some appropriate senses. In fact, we will often make relatively weak assumption on partial classicality following \cite{AndrasSM}. \par

In making our discussion as motivating as possible, let us first consider classicality and indicial operators for the conormal three-cone operators. However, before we state their definitions, let us first observe the following: for any $m,r,l \in \mathbb{R}$, the families of operators
\begin{equation*}
\hat{A}_{\dff} \in \Psi_{\mathrm{bc,lp}}^{m,b-2l}( X^{\tindex} ; \mathcal{C}_{\tindex} \times \mathbb{R}^{n_{\tindex}} ), \quad \hat{A}_{\cf} \in \Psi_{\mathrm{coc,lp}}^{m,r - \frac{b}{2}, l - \frac{b}{2}}( [ \hat{X}^{\tindex} ; \{ 0 \} ] ; \mathcal{C}_{\tindex} \times \mathbb{R}^{n_{\tindex}} )
\end{equation*}
can be extended through dilation in $z_{\tindex}$ so that we have
\begin{gather}
\hat{A}_{\dff}  \in \mathcal{C}^{\infty}( \overline{\mathbb{R}^{n_{\tindex}}} \backslash \{ 0 \} ; \Psi_{\mathrm{bc,lp}}^{m, b- 2l } ( X^{\tindex} ; \mathbb{R}^{n_{\tindex}} ) ), \label{indicial operator subsection oeprator valued symbol dff} \\ 
 \hat{A}_{\cf}  \in \mathcal{C}^{\infty} ( \overline{\mathbb{R}^{n_{\tindex}}} \backslash \{ 0 \} ; \Psi_{\mathrm{coc,lp}}^{m,r - \frac{b}{2}, l - \frac{b}{2}} ( [ \hat{X}^{\tindex} ; \{ 0 \} ] ; \mathcal{C}_{\tindex} \times \mathbb{R}^{n_{\tindex}}  ) ). \label{indicial operator subsection oeprator valued symbol cf} 
\end{gather}
Thus, let $A_{\dff}$, $A_{\cf}$ be the partial quantizations of $x_{\tindex}^{-l} \hat{A}_{\dff}$, $x_{\tindex}^{-b/2} \hat{A}_{\cf}$ respectively, i.e.,
\begin{gather}
A_{\dff} \coloneq \frac{1}{(2\pi)^{n_{\tindex}}} \int_{\mathbb{R}^{n_{\tindex}}} e^{ i ( z_{\tindex} - z_{\tindex}' ) \cdot \zeta_{\tindex} } x_{\tindex}^{-l} \hat{A}_{\dff}( z_{\tindex}, \zeta_{\tindex} )  d\zeta_{\tindex} |dz_{\tindex}'|, \label{indicial operator subsection A dff}  \\
A_{\cf} \coloneq \frac{1}{(2\pi)^{n_{\tindex}}} \int_{\mathbb{R}^{n_{\tindex}}} e^{ i ( z_{\tindex} - z_{\tindex} ) \cdot \zeta_{\tindex}^{\mathrm{3co}} } x_{\tindex}^{-\frac{b}{2}} \hat{A}_{\cf}( z_{\tindex}, \zeta_{\tindex}^{\mathrm{3co}} ) d\zeta_{\tindex}^{\mathrm{3co}} |dz_{\tindex}'|. \label{indicial operator subsection A cf} 
\end{gather}
Then we have 
\begin{equation*}
\psi_{\dff} A_{\dff} \psi_{\dff}, \psi_{\cf} A_{\cf} \psi_{\cf} \in \Psi_{\mathrm{3coc}}^{m,r,l,b}(\Xd),
\end{equation*}
where $\psi_{\dff}, \psi_{\cf} \in \mathcal{C}^{\infty}( \Xd )$ are cut-off functions at $\dff$ and $\cf$ respectively.

\begin{definition}[Partial classicality and indicial operators for the three-cone operators]
\label{indicial operator definition for the three-cone operators}
Let $m,r,l,b \in \mathbb{R}$, and assume that $A \in \Psi_{\mathrm{3coc}}^{m, r,l,b}(\Xd)$. Assume further that $K > 0$. 
\begin{itemize}
\item We say that $A$ is classical at $\dff$ modulo $\Psi_{\mathrm{3coc}}^{m,r,l-K,b}(\Xd)$ if
\begin{equation*}
\psi_{\dff} ( A - A_{\dff} ) \psi_{\dff} \in \Psi_{\mathrm{3coc}}^{m, -\infty , l - K, b}( \Xd )
\end{equation*}
for some $A_{\dff}$ defined as in (\ref{indicial operator subsection A dff}), with an operator-valued symbol $\hat{A}_{\dff}$ as in (\ref{indicial operator subsection oeprator valued symbol dff}). \par

Moreover, let $\hat{A}_{\psi_{\dff}}$ be the operator-valued symbol of $\psi_{\dff} A \psi_{\dff}$ as determined through (\ref{operator valued symbol near dff 1}). Then we will define the indicial operator of $A$ at $\dff$ by
\begin{equation*}
{^{\mathrm{3co}} \hat{N}_{\dff, l}}(A) \coloneq \lim_{x_{\tindex} \rightarrow 0} x_{\tindex}^{l} \hat{A}_{\psi_{\dff}} = \hat{A}_{\dff}.
\end{equation*}
\item We say that $A$ is classical at $\cf$ modulo $\Psi_{\mathrm{3coc}}^{m,r,l,b-K}(\Xd)$ if
\begin{equation*}
\psi_{\cf} ( A - A_{\cf} ) \psi_{\cf} \in \Psi_{\mathrm{3coc}}^{m,r,l,b-K}(\Xd).
\end{equation*}
for some $A_{\cf}$ define as in (\ref{indicial operator subsection A cf}), with an operator-valued symbol $\hat{A}_{\cf}$ as in (\ref{indicial operator subsection oeprator valued symbol cf}). \par

Moreover, let $\hat{A}_{\psi_{\cf}}$ be the operator-valued symbol of $\psi_{\cf} A \psi_{\cf}$ as determined through (\ref{three-cone case writing A psicf as a partial quantization -1}). Then we will define the indicial operator of $A$ at $\cf$ by
\begin{equation*}
{ ^{\mathrm{3co}}\hat{N}_{\cf,b}} (A) \coloneq \lim_{x_{\tindex} \rightarrow 0} x_{\tindex}^{\frac{b}{2}} \hat{A}_{\psi_{\cf}} = \hat{A}_{\cf}.
\end{equation*}
\end{itemize}
\end{definition}
\begin{remark}
Notice that Definition \ref{indicial operator definition for the three-cone operators} require much weaker conditions than the standard notions of `full classicality at $\dff$'. Indeed, it is clear that there exists another natural notion of classicality at $\dff$ for $A \in \Psi_{\mathrm{3coc}}^{m,r,l,b}(\Xd)$ which is considerably stronger. Namely, let $\hat{A}_{\psi_{\dff}}$ be the operator-valued symbol of $\psi_{\dff} A \psi_{\dff}$ as determined by (\ref{operator valued symbol near dff 1}). Then for every integer $j \geq 0$, there exists 
\begin{equation*}
\hat{A}_{\dff,j} \in  \Psi_{\mathrm{bc,lp}}^{m,b-2l}( X^{\tindex} ; \mathcal{C}_{\tindex} \times \mathbb{R}^{n_{\tindex}} )
\end{equation*}
such that
\begin{equation}
\label{asymptotic of the operator-valued symbols three-cone case near dff}
\hat{A}_{\psi_{\dff}} \sim \sum_{j=0}^{\infty} x_{\tindex}^{-l + j} \hat{A}_{\dff,j}
\end{equation}
in the following sense: for every positive integer $J \geq 1$, we have
\begin{equation*}
\hat{A}_{\psi_{\dff}} - \sum_{j=0}^{J-1} x_{\tindex}^{-l+j}  \hat{A}_{\dff,j} \in x_{\tindex}^{-l+J} \mathcal{C}^{\infty} ( \overline{\mathbb{R}^{n_{\tindex}}} \backslash \{ 0 \} ; \Psi_{\mathrm{bc,lp}}^{ m, b - 2l } ( X^{\tindex} ; \mathbb{R}^{n_{\tindex}} ) ). 
\end{equation*}
Notice that such a notion of classicality at $\dff$ is also contained by Definition \ref{indicial operator definition for the three-cone operators}, i.e., if $A$ satisfies the above, stronger notion of classicality at $\dff$, then $A$ must also be classical at $\dff$ modulo $\Psi_{\mathrm{3coc}}^{m,r, l - 1, b}( \Xd )$. Moreover, we have ${ ^{\mathrm{3co}}\hat{N}}_{\dff}(A) = \hat{A}_{\dff,0}$. \par

However, it is not so obvious how the above definition should be extended to $\cf$. Namely, it seems that there are now two natural notions of `full classicality at $\cf$'. Indeed, let $\hat{A}_{\psi_{\cf}}$ be the operator-valued symbol of $\psi_{\cf} A \psi_{\cf}$ as determined by (\ref{three-cone case writing A psicf as a partial quantization -1}). Then on the one hand, by respecting the smooth structure of $\psf \Xd$, we may require that for every integer $j \geq 0$, there exists 
\begin{equation}
\label{natural full asymptotic at cf first case -1}
\hat{A}_{\cf,j} \in \Psi_{\mathrm{coc,lp}}^{m,r - \frac{b}{2}, l - \frac{b}{2}}( [ \hat{X}^{\tindex} ; \{ 0 \} ] ; \mathcal{C}_{\tindex} \times \mathbb{R}^{n_{\tindex}} )
\end{equation}
such that 
\begin{equation} 
\label{natural full asymptotic at cf first case}
\hat{A}_{\psi_{\cf}} \sim \sum_{j=0}^{\infty}  x_{\tindex}^{ - \frac{b}{2} + \frac{j}{2} } \hat{A}_{\cf,j}.
\end{equation}
In other words, we require that one gains $j$ order of decay at $\cf$ for the $j$th term in the asymptotic summation. Notice that this is the case since we have $x_{\tindex} \simeq x_{\dff} x_{\cf}^2$, i.e., $x_{\tindex}$ is a quadratic defining function for $\cf$, while the gain in decay at $\dff$ coming from orders of $x_{\tindex}$ is actually irrelevant because (\ref{natural full asymptotic at cf first case}) is only a meaningful statement near $\cf$. \par

On the other hand, recall from \S \ref{subsection partial quantization near cf} that every differentiation in the $z_{\tindex}$ variables (i.e., the `symbolic' variables of $\hat{A}_{\psi_{\cf}}$) produces one order of gain as $x_{\tindex} \rightarrow 0$, which contains two orders of gain at $\cf$. Thus, from the perspective of viewing $\hat{A}_{\psi_{\cf}}$ as an operator-valued symbol, the local product structure suggests that we should define $A$ to be fully classical at $\cf$ if for every integer $j \geq 0$, there exists operator $\hat{A}_{\dff,j}$ as in (\ref{natural full asymptotic at cf first case -1}), such that
\begin{equation} 
\label{natural full asymptotic at cf first case 1}
\hat{A}_{\psi_{\cf}} \sim \sum_{j=0}^{\infty}  x_{\tindex}^{ - \frac{b}{2} + j } \hat{A}_{\cf,j}.
\end{equation} \par

Here, the asymptotic summations (\ref{natural full asymptotic at cf first case}) and (\ref{natural full asymptotic at cf first case 1}) are understood in the following senses: for every positive integer $J \geq 1$, we have
\begin{equation*}
\begin{gathered}
\hat{A}_{\psi_{\cf}} - \sum_{j=0}^{J-1} \hat{A}_{\cf,j} \in x_{\tindex}^{ - \frac{b}{2} + \frac{j}{2} } \mathcal{C}^{\infty} ( \overline{\mathbb{R}^{n_{\tindex}}} \backslash \{ 0 \} ; \Psi_{\mathrm{coc,lp}}^{m, r - \frac{b}{2}, l - \frac{b}{2} } ( [ \hat{X}^{\tindex} ; \{ 0 \} ] ; \mathbb{R}^{n_{\tindex}} ) ), \\
\hat{A}_{\psi_{\cf}} - \sum_{j=0}^{J-1} \hat{A}_{\cf,j} \in x_{\tindex}^{ - \frac{b}{2} + j } \mathcal{C}^{\infty} ( \overline{\mathbb{R}^{n_{\tindex}}} \backslash \{ 0 \} ; \Psi_{\mathrm{coc,lp}}^{m, r - \frac{b}{2}, l - \frac{b}{2}} ( [ \hat{X}^{\tindex} ; \{ 0 \} ] ; \mathbb{R}^{n_{\tindex}} ) )
\end{gathered}
\end{equation*}
respectively. Thus, both notions of full classicality are again contained in Definition \ref{indicial operator definition for the three-cone operators}. Namely, they imply that $A$ is classical at $\cf$ modulo $\Psi_{\mathrm{3coc}}^{m, r, l, b - 1}(\Xd)$ and $\Psi_{\mathrm{3coc}}^{m,r,l,b-2}(\Xd)$ respectively. Moreover, in both cases we have ${ ^{\mathrm{3sc}} \hat{N}_{\cf} }(A) = \hat{A}_{\cf,0}$. \par

In particular, the above discussion suggests that Definition \ref{indicial operator definition for the three-cone operators} is indeed the more suitable notion of classicality in this context. In practice, one simply check carefully the precise structure of the operator at hand, and then choose $K > 0$ correspondingly.
\end{remark}

In the presence of variable orders, we will proceed via the same philosophy as in Definition \ref{indicial operator definition for the three-cone operators} to define indicial operators for the second microlocalized operators. However, suppose that $\vom, \vor, \vov, \vos \in \mathcal{C}^{\infty}( \psf \Xd )$, $\vol, \vob \in \mathcal{C}^{\infty}( \mathcal{C}_{\tindex} \times \overline{\mathbb{R}^{n_{\tindex}}} )$, and let us take
\begin{gather}
\hat{A}_{\dff} \in \Psi_{\mathrm{sc,bc,lp,res}, \delta}^{\vom, \vov - 2 \vol, \vob - 2 \vol, \vos - \vol}( X^{\tindex} ; \mathcal{C}_{\tindex} \times \mathbb{R}^{n_{\tindex}} ),  \label{hat A dff indicial operator subsection} \\ 
\hat{A}_{\cf} \in \Psi_{\mathrm{coc,lp,res}, \delta}^{ \vom, \vor - \frac{\vob}{2}, \vol - \frac{\vob}{2}, \vov - \vob }( [ \hat{X}^{\tindex} ; \{ 0 \} ] ; \mathcal{C}_{\tindex} \times \mathbb{R}^{n_{\tindex}} ).  \label{hat A cf indicial operator subsection} 
\end{gather}
Then we can no longer extend $\hat{A}_{\dff}$, $\hat{A}_{\cf}$ by dilation in $z_{\tindex}$, and naively expect the resulting families of operators to (even locally near $\dff$ and $\cf$ respectively) have the correct structures as operator-valued symbols for some elements of $\Psf^{\vom, \vor, \vol, \vov, \vob, \vos}(\Xd)$.
 \par

To see this, let us consider a simple example in the constant order case: Let
\begin{equation*}
\hat{A}_{\dff} \in \Psi_{\mathrm{sc,bc,lp,res}}^{0,-\infty , -\infty ,0}( X^{\tindex}; \mathcal{C}_{\tindex} \times \mathbb{R}^{n_{\tindex}}).
\end{equation*}
be chosen such that the support of $\hat{A}_{\dff}$ is contained in the interior of $X^{\tindex}$. Moreover, $\hat{A}_{\dff}$ is given by the standard quantization of some $a_{\dff} \in S^{0, - \infty, -\infty, 0}( \psf_{\dff} \Xd )$ in the Euclidean coordinates $( z^{\tindex}, \zeta^{\tindex} )$. \par

Suppose that we extend $a_{\dff}$ by dilation in $z_{\tindex}$, and denote this extension by $a$ still. Thus $a_{\dff}$ is now defined in the lift of a small neighborhood of $\overline{^{\mathrm{3co}}T^{\ast}}_{\dff} \Xd$ to $\psf \Xd$, which we might call $U$. Then we would like to know whether $a_{\dff}$ can be identified as an element of $S^{0, -\infty, 0, -\infty, -\infty, 0}( \psf \Xd )$ in $U$. \par

But it is clear that this is false. Indeed, notice that due to the support property of $a_{\dff}$, the only boundary faces of $\psf \Xd$ which $U$ intersects are $\psf_{\dff} \Xd$ and $\rf$. Moreover, we can find open subsets of $U$ which are disjoint from $\psf_{\dff} \Xd$, such that in the coordinates $( z_{\tindex}, z^{\tindex}, \zeta_{\tindex}, \zeta^{\tindex} )$, and with $|\zeta_{\tindex}| \geq c > 0$, we have
\begin{equation*}
\rho_{\rf} \simeq \frac{1}{|z_{\tindex}|} , \quad \rho_{\infty} \simeq \frac{|z_{\tindex}|}{ \langle \zeta^{\tindex} \rangle }. 
\end{equation*}
Thus the required symbol estimates are given by
\begin{align*}
 | \partial_{z_{\tindex}}^{\beta_{\tindex}} \partial_{z^{\tindex}}^{\beta^{\tindex}} \partial_{\tsclz}^{\gamma_{\tindex}} \partial_{\tscuz}^{\gamma^{\tindex}} a_{\dff} | \leq {} & C_{\beta_{\tindex} \beta^{\tindex} \gamma_{\tindex} \gamma^{\tindex} } \rho_{\infty}^{|\gamma_{\tindex}| + |\gamma^{\tindex}|} \rho_{\rf}^{|\beta_{\tindex}| + |\gamma^{\tindex}|} \\ 
 \leq {} & C_{\beta_{\tindex} \beta^{\tindex} \gamma_{\tindex} \gamma^{\tindex} } |z_{\tindex}|^{-|\beta_{\tindex}| + |\gamma_{\tindex}|} \langle \zeta^{\tindex} \rangle^{-|\gamma_{\tindex}| - |\gamma^{\tindex}|}.
\end{align*}
However, it is easy to check that, in fact we have
\begin{align*}
 | \partial_{z_{\tindex}}^{\beta_{\tindex}} \partial_{z^{\tindex}}^{\beta^{\tindex}} \partial_{\tsclz}^{\gamma_{\tindex}} \partial_{\tscuz}^{\gamma^{\tindex}} a_{\dff} | \leq {} &  C_{ \beta_{\tindex} \beta^{\tindex} \gamma_{\tindex} \gamma^{\tindex} } | z_{\tindex} |^{-|\beta_{\tindex}|} \langle \zeta^{\tindex} \rangle^{-|\gamma^{\tindex}|}  \\
 \leq {} & C_{ \beta_{\tindex} \beta^{\tindex} \gamma_{\tindex} \gamma^{\tindex} } |z_{\tindex}|^{-|\beta_{\tindex}| + |\gamma_{\tindex}|} \langle \zeta^{\tindex} \rangle^{-|\gamma_{\tindex}| - |\gamma^{\tindex}|} \rho_{\infty}^{-|\gamma_{\tindex}|}
\end{align*}
which degenerates as $\rho_{\infty} \rightarrow 0$.

To resolve this issue, let us note that the aforementioned degeneracy is only relevant at the level of symbols (in the sense that it is relevant only for those terms which are defined by a quantization) away from $\psf_{\dff} \Xd$. This can be checked in local coordinates, and indeed it is also reflected in Remark \ref{the serious product structure in the second microlocalized, resolved case remark}. \par

More concretely, suppose that 
\begin{equation*}
q_{\dff}, q_{\cf} \in \mathcal{C}^{\infty}( \psf \Xd )
\end{equation*}
are cut-off functions at $\psf_{\dff}\Xd$ and $\tcocf$ respectively. Let also
\begin{equation*}
a_{\dff} \in S^{\vom, \vov - 2 \vol, \vob - 2 \vol, \vos - \vol}_{\delta}( \psf_{\dff}\Xd ), \quad a_{\cf} \in S^{\vom, \vor - \frac{\vob}{2}, \vol - \frac{\vob}{2}, \vov - \vob}_{\delta} (\tcocf)
\end{equation*}
be the symbols of $\hat{A}_{\dff}$ and $\hat{A}_{\cf}$ as defined in (\ref{hat A dff indicial operator subsection}) and (\ref{hat A cf indicial operator subsection}) respectively. Then it is easy to see that we have 
\begin{equation*}
x_{\tindex}^{-\vol} q_{\dff} a_{\dff} \in S^{\vom, -\infty, \vol, \vov, \vob, \vos}_{\delta}( \psf \Xd ), \quad x_{\cf}^{-\frac{\vob}{2}} q_{\cf} a_{\cf} \in S^{\vom, \vor, \vol, \vov, \vob, -\infty }_{\delta} ( \psf \Xd ).
\end{equation*}
\par

It follows that we can define an operator $A_{\dff,q_{\dff}}$ as follows: Recall from \S \ref{subsection second microlocalization for the indicial operators dff} that for $\psi^{\tindex} \in \mathcal{C}^{\infty}(X^{\tindex})$ with sufficiently small support, the parts of $\hat{A}_{\dff}$ which are given directly by quantizations are either of the form $(\hat{B}_{\dff})_{\psi^{\tindex}}$ as in (\ref{B in the constant order large parameter indicial operator b calculus}) if $\supp \psi^{\tindex}$ intersects $\mathcal{C}^{\tindex}$, or they are of the form $\psi^{\tindex} \hat{A}_{\dff} \psi^{\tindex}$ if $\supp \psi^{\tindex}$ does not intersect $\mathcal{C}^{\tindex}$. Let $(b_{\dff})_{\psi^{\tindex}}$ denote the symbols of these terms in either cases. Then we will define $(B_{\dff, q_{\dff}})_{\psi^{\tindex}}$ and $\psi^{\tindex} A_{\dff, q_{\dff}} \psi^{\tindex}$ respectively by (\ref{the 3scb quantization written in terms of t}) and the standard Euclidean quantization (i.e., (\ref{scattering quantization})), with the exception that their symbols are now replaced by $x_{\tindex}^{-\vol} q_{\dff} (b_{\dff})_{\psi^{\tindex}}$. \par

On the other hand, let $\phi^{\tindex} \in \mathcal{C}^{\infty}( X^{\tindex} )$ be such that $\supp \psi^{\tindex} \cap \supp \phi^{\tindex} = \emptyset$, and assume that $\supp \phi^{\tindex}$ is small as well. Then the `off-diagonal' parts of $\hat{A}_{\dff}$ either take the form $(\hat{R}_{\dff})_{\psi^{\tindex}}$ (i.e., the second term which appears in (\ref{B in the constant order large parameter indicial operator b calculus -1}) and is defined by (\ref{dff second microlocalizing indicial operator estimate on hat R})), or they are given by $\phi^{\tindex} \hat{A}_{\dff} \psi^{\tindex}$, and we will define
\begin{align}
(R_{\dff})_{\psi^{\tindex}} & \coloneq \frac{1}{(2\pi)^{n_{\tindex}}} \int_{\mathbb{R}^{n_{\tindex}}} e^{ i ( z_{\tindex} - z_{\tindex}' ) \cdot \zeta_{\tindex} } x_{\tindex}^{-\vol} ( \hat{R}_{\dff} )_{\psi^{\tindex}}( z_{\tindex}, \zeta_{\tindex} ) d \zeta_{\tindex} |dz_{\tindex}'|, \label{indicial operator subsection part of Adff R} \\
\phi^{\tindex}( A_{\dff, q_{\dff}} ) \psi^{\tindex} & \coloneq \frac{1}{(2\pi)^{n_{\tindex}}} \int_{\mathbb{R}^{n_{\tindex}}} e^{ i ( z_{\tindex} - z_{\tindex}' ) \cdot \zeta_{\tindex} } x_{\tindex}^{- \vol} (\phi^{\tindex} \hat{A}_{\dff} \psi^{\tindex})  ( z_{\tindex}, \zeta_{\tindex} ) d \zeta_{\tindex} |dz_{\tindex}'|. \label{indicial operator subsection part of Adff K}
\end{align} 
Hence in fact, $(K_{\dff})_{\phi^{\tindex}, \psi^{\tindex}} \coloneq  \phi^{\tindex}( A_{\dff, q_{\dff}} ) \psi^{\tindex}$ is independent of $q_{\dff}$.
\par

It follows from the standard partition argument that if we set
\begin{equation*}
\psi^{\tindex} A_{\dff, q_{\dff}} \psi^{\tindex} \coloneq (B_{\dff, q_{\dff}})_{\psi^{\tindex}} + (R_{\dff})_{\psi^{\tindex}}
\end{equation*}
when $\supp \psi^{\tindex}$ intersects $\mathcal{C}^{\tindex}$,  then we can realize 
\begin{equation*}
\psi_{\dff} A_{\dff, q_{\dff}} \psi_{\dff} \in \Psf^{\vom, -\infty , \vol, \vov, \vob, \vos }(\Xd),
\end{equation*}
where $\psi_{\dff} \in \mathcal{C}^{\infty}(\Xd)$ is any cut-off function at $\dff$.

The analogous construction can also be made at $\cf$ as well. Indeed, recall from \S \ref{subsection second microlocalization for the indicial operators cf} that for $\psi^{\tindex} \in \mathcal{C}^{\infty}( [ \hat{X}^{\tindex} ; \{ 0 \} ] )$ with sufficiently small support, the parts of $\hat{A}_{\cf}$ which are given directly by quantizations are either of the form $(\hat{B}_{\cf})_{\psi^{\tindex},\mathrm{b}}$ as in (\ref{second microlocalizing the indicial operator at cf b part quantization}) if $\supp \psi^{\tindex}$ intersects $\mathcal{C}^{\tindex}_{0}$, or they are of the form $( \hat{B}_{\cf} )_{\psi^{\tindex}, \mathrm{sc}}$ as in (\ref{second microlocalizing the indicial operator at cf sc part quantization}) if $\supp \psi^{\tindex}$ does not intersect $\mathcal{C}^{\infty}_{0}$. Let $(b_{\cf})_{\psi^{\tindex}}$ denote the symbols of these terms in either cases. Then we will define $(B_{\cf, q_{\cf}})_{\psi^{\tindex}, \mathrm{b}}$ and $( B_{\cf, q_{\cf}} )_{\psi^{\tindex}, \mathrm{sc}}$ respectively by (\ref{the 3cob quantization written in terms of t}) and (\ref{3cosc quantization}), with the exception that their symbols are now replaced by $x_{\tindex}^{-\vob/2} q_{\cf} ( b_{\cf} )_{\psi^{\tindex}}$. \par

On the other hand, let $\phi^{\tindex} \in \mathcal{C}^{\infty}( [ \hat{X}^{\tindex} ; \{ 0 \} ] )$ be such that $\supp \psi^{\tindex} \cap \supp \phi^{\tindex} = \emptyset$, and assume that $\supp \phi^{\tindex}$ is small as well. Then the `off-diagonal' parts of $\hat{A}_{\cf}$ either take the form $(\hat{R}_{\cf})_{\psi^{\tindex}, \mathrm{b}}$ (i.e., the second term which appears in (\ref{second microlocalizing the indicial operator at cf b part quantization - 1}) and is defined by (\ref{second microlocalizing the indicial operator at cf b part quantization 1})), or they are given by $\phi^{\tindex} \hat{A}_{\cf} \psi^{\tindex}$, and we will define
\begin{align*}
(R_{\cf})_{\psi^{\tindex}, \mathrm{b}} & \coloneq \frac{1}{(2\pi)^{n_{\tindex}}} \int_{\mathbb{R}^{n_{\tindex}}} e^{ i ( z_{\tindex} - z_{\tindex}' ) \cdot \zeta_{\tindex}^{\mathrm{3co}} } x_{\tindex}^{-\vob/2} ( \hat{R}_{\cf} )_{\psi^{\tindex}, \mathrm{b}}( z_{\tindex}, \zeta_{\tindex}^{\mathrm{3co}} ) d\zeta_{\tindex}^{\mathrm{3co}} |dz_{\tindex}'|,  \\
\phi^{\tindex}( A_{\cf, q_{\cf}} ) \psi^{\tindex} & \coloneq \frac{1}{(2\pi)^{n_{\tindex}}} \int_{\mathbb{R}^{n_{\tindex}}} e^{ i ( z_{\tindex} - z_{\tindex}' ) \cdot \zeta_{\tindex}^{\mathrm{3co}} } x_{\tindex}^{-\vob/2} ( \phi^{\tindex} \hat{A}_{\cf} \psi^{\tindex} ) ( z_{\tindex}, \zeta_{\tindex}^{\mathrm{3co}} ) d \zeta_{\tindex}^{\mathrm{3co}} | d z_{\tindex}' |.
\end{align*}
Here, $(K_{\cf})_{\phi^{\tindex}, \psi^{\tindex}} \coloneq \phi^{\tindex}( A_{\cf, q_{\cf}} ) \psi^{\tindex}$ is again independent of $q_{\cf}$. \par

Finally, if we set 
\begin{equation*}
\psi^{\tindex} A_{\cf, q_{\cf}} \psi^{\tindex} \coloneq ( B_{\cf, q_{\cf}} )_{\psi^{\tindex}, \mathrm{b}} + (R_{\cf})_{\psi^{\tindex}, \mathrm{b}}
\end{equation*}
when $\supp \psi^{\tindex}$ intersects $\mathcal{C}^{\tindex}_{0}$, and
\begin{equation*}
\psi^{\tindex} A_{\cf,q_{\cf}} \psi^{\tindex} \coloneq ( B_{\cf,q_{\cf}} )_{\psi^{\tindex}, \mathrm{sc}}
\end{equation*}
otherwise, then the standard partition argument allows us to realize
\begin{equation*}
\psi_{\cf} A_{\cf,q_{\cf}} \psi_{\cf} \in \Psf^{\vom, \vor, \vol, \vov, \vob, -\infty }(\Xd),
\end{equation*}
where $\psi_{\cf} \in \mathcal{C}^{\infty}(\Xd)$ is any cut-off function at $\cf$. \par

By using the above constructions, we can now make the following definition:
\begin{definition}[Partial classicality and indicial operators for the further-resolved second microlocalized operators]
\label{definition of partial classicality}
Let $\vom, \vor, \vov, \vos \in \mathcal{C}^{\infty}( \psf \Xd )$, $\vol, \vob \in \mathcal{C}^{\infty}( \mathcal{C}_{\tindex} \times \overline{\mathbb{R}^{n_{\tindex}}} )$ and $A \in \Psi^{ \mathsf{m}, \mathsf{r}, \mathsf{l}, \mathsf{v}, \mathsf{b}, \mathsf{s} }_{\mathrm{d3sc,3co,res}, \delta} (\Xd)$. Assume further that $K > 0$. Then:
\begin{itemize}
\item We say that $A$ is classical at $\dff$ modulo $\Psf^{\vom, \vor, \vol - K, \vov, \vob, \vos}(\Xd)$ if
\begin{equation*}
\psi_{\dff} ( A - A_{\dff, q_{\dff}} ) \psi_{\dff} \in \Psf^{\vom, -\infty , \vol - K, \vov, \vob, \vos}(\Xd),
\end{equation*}
where $A_{\dff, q_{\dff}}$ is constructed as in the above discussion with respect to some $\hat{A}_{\dff}$ as in (\ref{hat A dff indicial operator subsection}) and $q_{\dff} \in \mathcal{C}^{\infty}(\psf \Xd)$ which cuts off at $\psf_{\dff} \Xd$. \par 

 Moreover, let $\hat{A}_{\psi_{\dff}}$ be the operator-valued symbol of $\psi_{\dff} A \psi_{\dff}$ as determined through (\ref{membership of operator valued symbol near dff}). Then we will define the indicial operator of $A$ at $\dff$ by
\begin{equation*}
{^{\mathrm{d3sc,3co,res}} \hat{N}_{\dff, \vol}}(A) \coloneq \lim_{x_{\tindex} \rightarrow 0} x_{\tindex}^{\vol} \hat{A}_{\psi_{\dff}} = \hat{A}_{\dff}.
\end{equation*}
\item We say that $A$ is classical at $\cf$ modulo $\Psf^{\vom, \vor, \vol, \vov, \vob - K, \vos}(\Xd)$ if
\begin{equation*}
\psi_{\cf} ( A - A_{\cf, q_{\cf}} ) \psi_{\cf} \in \Psf^{\vom, \vor, \vol, \vov, \vob - K, \vos}(\Xd),
\end{equation*}
where $A_{\cf, q_{\cf}}$ is constructed as in the above discussion with respect to some $\hat{A}_{\cf}$ as in (\ref{hat A dff indicial operator subsection}) and $q_{\cf} \in \mathcal{C}^{\infty}( \psf \Xd )$ which cuts off at $\tcocf$. \par

Moreover, let $\hat{A}_{\psi_{\cf}}$ be the operator-valued symbol of $\psi_{\cf} A \psi_{\cf}$ as determined through (\ref{membership of operator valued symbol near dff}). Then we will define the indicial operator of $A$ at $\cf$ by
\begin{equation*}
{^{\mathrm{d3sc,3co,res}} \hat{N}_{\cf, \vob}}(A) \coloneq \lim_{x_{\tindex} \rightarrow 0} x_{\tindex}^{\frac{\vob}{2}} \hat{A}_{\psi_{\cf}} = \hat{A}_{\cf}.
\end{equation*}
\end{itemize}

The same statements hold if $\vom = m$, $\vor = r$, $\vol = l$, $\vov = \nu$, $\vob = b$, $\vos = s$ are all constants, in which case we also set $\delta = 0$.

\end{definition}

\begin{remark}
\label{alternate characterization of partial classicality at dff}
Equivalently, $A$ is classical at $\dff$ modulo $\Psf^{\vom,\vor, \vol - K, \vov, \vob, \vos}(\Xd)$ if for some $\hat{A}_{\dff}$ as in (\ref{hat A dff indicial operator subsection}) the following conditions are satisfied:

\begin{itemize}
\item Suppose that $\psi^{\tindex} \in \mathcal{C}^{\infty}( X^{\tindex} )$ has sufficiently small support. Let $(b_{\psi_{\dff}})_{\psi^{\tindex}}$ denote the symbol of $\psi^{\tindex} A_{\psi_{\dff}} \psi^{\tindex}$ where $A_{\psi_{\dff}} = \psi_{\dff} A \psi_{\dff}$ (i.e., $(b_{\psi_{\dff}})_{\psi^{\tindex}}$ is the symbol of $B_{\psi_{\dff} \psi^{\tindex}, \mathrm{3co,b}}$ as defined in (\ref{variable orders definition of the operator near dff 1}) if $\supp \psi^{\tindex}$ intersects $\mathcal{C}^{\tindex}$). Then there exists some $q_{\dff} \in \mathcal{C}^{\infty}( \psf \Xd )$ which cuts off at $\psf_{\dff} \Xd$, such that 
\begin{equation}
\label{indicial operator vanishes condition 1}
( b_{\psi_{\dff}} )_{\psi^{\tindex}} - x_{\tindex}^{-\vol} q_{\dff} ( b_{\dff})_{\psi^{\tindex}}  \in S^{\vom, -\infty , \vol - K, \vov, \vob, \vos}_{\delta}( \psf \Xd ),
\end{equation}
where $(b_{\dff})_{\psi^{\tindex}}$ is defined as in the above discussion.
\item Suppose that $\psi^{\tindex} \in \mathcal{C}^{\infty}(X^{\tindex})$ has sufficiently small support such that $\supp \psi^{\tindex}$ does not intersect $\mathcal{C}^{\tindex}$. Let $(R_{\psi_{\dff}})_{\psi^{\tindex}}\coloneq R_{\psi_{\dff} \psi^{\tindex}, \mathrm{3co,b}}$ as defined in (\ref{variable orders definition of the operator near dff 1}) and $(R_{\dff})_{\psi^{\tindex}}$ be the operator defined in (\ref{indicial operator subsection part of Adff R}). Then we have
\begin{equation}
\label{indicial operator vanishes condition 2}
( R_{\psi_{\dff}} )_{\psi^{\tindex}} - \psi_{\dff} (R_{\dff})_{\psi^{\tindex}} \psi_{\dff} \in \Psf^{-\infty, -\infty, \vol - K, -\infty, \vob, -\infty}( \Xd ).
\end{equation}
\item Suppose that $\phi^{\tindex}, \psi^{\tindex} \in \mathcal{C}^{\infty}( X^{\tindex} )$ have sufficiently small supports, which moreover satisfy $\supp \phi^{\tindex} \cap \supp \psi^{\tindex} = \emptyset$. Let $(K_{\psi_{\dff}})_{\phi^{\tindex}, \psi^{\tindex}} \coloneq \phi^{\tindex} A_{\psi_{\dff}} \psi^{\tindex}$ and $(K_{\dff})_{\phi^{\tindex}, \psi^{\tindex}}$ be the operator defined in (\ref{indicial operator subsection part of Adff K}). Then we have
\begin{equation}
\label{indicial operator vanishes condition 3}
 (K_{\psi_{\dff}})_{\phi^{\tindex}, \psi^{\tindex}} - \psi_{\dff} ( K_{\dff} )_{\phi^{\tindex}, \psi^{\tindex}} \psi_{\dff}  \in \Psf^{-\infty, -\infty, \vol - K, -\infty, \vob, -\infty}( \Xd ).
\end{equation}
\end{itemize} \par

The analogous characterization, which we shall omit, is available at $\cf$ as well.
\end{remark}

\begin{remark}
We also remark that the peculiar choices of indices which appear in (\ref{hat A dff indicial operator subsection}) and (\ref{hat A cf indicial operator subsection}) are required because we have factored out $x_{\tindex} = |z_{\tindex}|^{-1}$ in both cases, which is a total boundary of $\Xd$ near $\mathcal{C}_{\tindex}$. In fact, by local calculations, it is easy to see that
\begin{equation*}
x_{\tindex} \simeq \rho_{\dff} \rho_{\dmf} \rho_{\mathrm{d3sccf}_{\tindex}}^2 \rho_{\mathrm{3cocf}}^2 \rho_{\rf}.
\end{equation*}
\end{remark}

For brevity, henceforth we will omit writing the subscripts which appear in the indicial operators, as well as the orders on which they depend. More precisely, if $A \in \Psi_{\mathrm{3co}}^{m,r,l,b}(\Xd)$, then we will write ${\hat{N}_{\dff} }(A) = { ^{\mathrm{3co}}\hat{N}_{\dff,l} }(A)$, ${ \hat{N}_{\cf} }(A) = { ^{\mathrm{3co}}\hat{N}_{\cf,b} }(A)$. Likewise, if $A \in \Psi_{\mathrm{3co}}^{m,r,l,b}(\Xd)$, then we will write ${\hat{N}_{\dff} }(A) = { ^{\mathrm{d3sc,3co,res}}\hat{N}_{\dff,l} }(A)$, ${ \hat{N}_{\cf} }(A) = { ^{\mathrm{d3sc,3co,res}}\hat{N}_{\cf,b} }(A)$.

The point of considering the indicial operators at $\dff$ and $\cf$ is that they capture principal order decay at these faces.
\begin{proposition}
\label{proposition vanishing of indicial operator implies lower order membership}
Let $\vom, \vor, \vov, \vos \in \mathcal{C}^{\infty}( \psf \Xd )$, $\vol, \vob \in \mathcal{C}^{\infty}( \mathcal{C}_{\tindex} \times \overline{\mathbb{R}^{n_{\tindex}}} )$, and assume that $A \in \Psi^{ \mathsf{m}, \mathsf{r}, \mathsf{l}, \mathsf{v}, \mathsf{b}, \mathsf{s} }_{\mathrm{d3sc,3co,res}, \delta} (\Xd)$. Assume further that $K > 0$.
\begin{enumerate}
\item Suppose that $A$ is classical at $\dff$ modulo $\Psf^{\vom, \vor, \vol - K, \vov, \vob, \vos}(\Xd)$ with $\hat{N}_{\dff}(A) = 0$. Then we have $A \in \Psf^{ \vom, \vor, \vol - K, \vov, \vob, \vos}(\Xd)$.
\item Suppose that $A$ is classical at $\cf$ modulo $\Psf^{\vom, \vor, \vol, \vov, \vob - K, \vos}(\Xd)$ with $\hat{N}_{\cf}(A) = 0$. Then we have $A \in \Psf^{\vom, \vor, \vol, \vov, \vob - K, \vos}(\Xd)$.
\end{enumerate}

The same statements hold if $\vom = m$, $\vor = r$, $\vol = l$, $\vov = \nu$, $\vob = b$, $\vos = s$ are all constants, in which case we also set $\delta = 0$.

\end{proposition}

\begin{proof}
We will only prove part (1) as the proof of part (2) is similar. \par

Assume that $A$ is classical as stated. Then the conditions outlined in Remark \ref{alternate characterization of partial classicality at dff} are satisfied. Assume additionally that $\hat{A}_{\dff} = 0$. Then we first have $(b_{\dff})_{\psi^{\tindex}} = 0$ for all $\psi^{\tindex} \in \mathcal{C}^{\infty}(X^{\tindex})$ with sufficiently small support. Next, we know that $(\hat{R}_{\dff})_{\psi^{\tindex}} = 0$ identically if $\supp \psi^{\tindex}$ intersects $\mathcal{C}^{\tindex}$, and so by (\ref{indicial operator subsection part of Adff R}) we must have $(R_{\dff})_{\psi^{\tindex}} = 0$ as well. Finally, if $\phi^{\tindex} \in \mathcal{C}^{\infty}(X^{\tindex})$ has sufficiently support and is such that $\supp \phi^{\tindex} \cap \supp \psi^{\tindex} = \emptyset$, then we know that $(\hat{K}_{\dff})_{\phi^{\tindex}, \psi^{\tindex}} = 0$ identically. Moreover, $(K_{\psi_{\dff}})_{\phi^{\tindex}, \psi^{\tindex}} =0$ by (\ref{indicial operator subsection part of Adff K}). \par

 It follows from (\ref{indicial operator vanishes condition 1})--(\ref{indicial operator vanishes condition 3}) that we have 
\begin{equation*}
\begin{gathered}
(b_{\psi_{\dff}})_{\psi^{\tindex}} \in S_{\delta}^{\vom, -\infty, \vol - K, \vov, \vob, \vos}( \psf \Xd ), \\
(R_{\psi_{\dff}})_{\psi^{\tindex}}, (K_{\psi_{\dff}})_{\phi^{\tindex}, \psi^{\tindex}} \in \Psf^{-\infty, -\infty, \vol - K, -\infty, \vob, -\infty}(\Xd)
\end{gathered}
\end{equation*}
whenever these terms are defined. Equivalently, the above argument shows that $A_{\dff, q_{\dff}} = 0$ by construction. Either way, we have
\begin{equation*}
A_{\psi_{\dff}} \in \Psf^{\vom, -\infty, \vol - K, \vov, \vob, \vos}( \Xd ),
\end{equation*}
from which the required membership of $A$ can be concluded as well.
\end{proof}

We also have the converse statement to Proposition \ref{proposition vanishing of indicial operator implies lower order membership}.

\begin{lemma}
\label{lemma constructing operator from prescribed indicial operator}
Suppose that $\vom, \vor, \vov, \vos \in \mathcal{C}^{\infty}( \psf \Xd )$ and $\vol, \vob \in \mathcal{C}^{\infty}( \mathcal{C}_{\tindex} \times \overline{\mathbb{R}^{n_{\tindex}}} )$. 
\begin{enumerate}
\item For every 
\begin{equation*}
\hat{A}_{\dff} \in \Psi^{\mathsf{m}, \mathsf{v} - 2 \mathsf{l}, \mathsf{b} - 2 \mathsf{l}, \mathsf{s} - \mathsf{l}}_{\mathrm{sc,b,lp,res},\delta}( 
X^{\tindex} ; \mathcal{C}_{\tindex} \times {\mathbb{R}^{n_{\tindex}}} ),
\end{equation*}
there exists $A \in \Psf^{\vom, -\infty, \vol, \vov, \vob, \vos}(\Xd)$ supported near $\dff$, such that 
\begin{equation*}
\text{$A$ is classical at $\dff$ modulo $\Psf^{\vom, \vor, -\infty, \vov, \vob, \vos}(\Xd)$ and $\hat{N}_{\dff}(A) = \hat{A}_{\dff}$.}
\end{equation*} 
\item For every 
\begin{equation*}
\hat{A}_{\cf} \in \Psi_{\mathrm{coc,lp,res}}^{\mathsf{m}, \mathsf{r} - \mathsf{b}/2, \mathsf{l} - \mathsf{b}/2, \mathsf{v} - \mathsf{b}/2 }( [ \hat{X}^{\tindex} ; \{ 0 \} ] ; \mathcal{C}_{\tindex} \times {\mathbb{R}^{n_{\tindex}}} ),
\end{equation*}
there exists $A \in \Psi^{\mathsf{m}, \mathsf{r}, \mathsf{l}, \mathsf{v}, \mathsf{b}, \mathsf{s}}_{\mathrm{d3sc,3co,res},\delta}( \Xd )$ supported near $\cf$, such that 
\begin{equation*}
\text{$A$ is classical at $\cf$ modulo $\Psi^{\mathsf{m}, \mathsf{r}, \mathsf{l}, \mathsf{v}, -\infty, \vos }_{\mathrm{d3sc,3co,res},\delta}(\Xd)$ and $\hat{N}_{\cf}(A) = \hat{A}_{\cf}$.}
\end{equation*}
\end{enumerate}
In fact, the operators in $\Psf^{\vom, -\infty, \vol, \vov, \vob, \vos}(\Xd)$, $\Psi^{\mathsf{m}, \mathsf{r}, \mathsf{l}, \mathsf{v}, -\infty, \vos }_{\mathrm{d3sc,3co,res},\delta}(\Xd)$ above can be the zero operator.
\end{lemma}
\begin{proof}
We simply take $A = \psi_{\dff} A_{\dff, q_{\dff}} \psi_{\dff}$ in case (1) and $A = \psi_{\cf} A_{\cf, q_{\cf}} \psi_{\cf}$ in case (2), where $A_{\dff, q_{\dff}}$, $A_{\cf, q_{\cf}}$ are constructed as in the discussion before Definition \ref{definition of partial classicality}.
\end{proof}

\subsection{Composition and adjoint}
\label{subsection composition and adjoint}
In this subsection, we will study compositions within the class of operators $\Psf^{\vom, \vor, \vol, \vov, \vob, \vos}(\Xd)$. Our main result is the following:

\begin{proposition}
\label{Composition proposition}
Let
\begin{equation*}
\begin{gathered}
\vom_{j}, \vor_{j}, \vov_{j}, \vos_{j} \in \mathcal{C}^{\infty}( \psf \Xd ), \quad \vol_{j}, \vob_{j} \in \mathcal{C}^{\infty}( \mathcal{C}_{\tindex} \times \overline{\mathbb{R}^{n_{\tindex}}} ), \\
A_{j} \in \Psi^{ \mathsf{m}_{j}, \mathsf{r}_{j}, \mathsf{l}_{j}, \mathsf{v}_{j}, \mathsf{b}_{j}, \mathsf{s}_{j} }_{\mathrm{d3sc,3co,res}, \delta} (\Xd), \quad j =1,2.
\end{gathered}
\end{equation*}
Then we have 
\begin{equation} 
\label{composition regularity result d3sc,3co,res}
A_{1} A_{2} \in \Psi_{\mathrm{d3sc,3co,res},\delta}^{ \mathsf{m}_{1} + \mathsf{m}_2, \mathsf{r}_{1} + \mathsf{r}_{2}, \mathsf{l}_{1} + \mathsf{l}_{2}, \mathsf{v}_{1} + \mathsf{v}_{2}, \mathsf{b}_{1} + \mathsf{b}_{2}, \mathsf{s}_{1} + \mathsf{s}_{2} }(\Xd),
\end{equation}
and principal symbol is multiplicative in the sense that 
\begin{equation} \label{multiplicative property for the principal symbol d3sc,3co,res}
{\sigma}( A_{1} A_{2} )  = {\sigma}(A_1) {\sigma}( A_2 ).
\end{equation} \par

Moreover, composition behaves well for the indicial operators too:
\begin{enumerate}
\item Let
\begin{equation*}
\begin{gathered}
\vom_{j}, \vor_{j}, \vos_{j} \in \mathcal{C}^{\infty}( \psf_{\dff} \Xd ), \quad \vol_{j} \in \mathcal{C}^{\infty}( \mathcal{C}_{\tindex} \times \overline{\mathbb{R}^{n_{\tindex}}} ), \\ 
\hat{A}_{\dff,j} \in \Psi^{\mathsf{m}_{j}, \mathsf{r}_{j}, \mathsf{l}_{j}, \mathsf{s}_{j}}_{\mathrm{sc,b, lp, res},\delta}( X^{\tindex} ; \mathcal{C}_{\tindex} \times {\mathbb{R}^{n_{\tindex}}} ), \quad j =1,2.
\end{gathered}
\end{equation*}
Then we have
\begin{equation} \label{composition symbol classes at dff}
\hat{A}_{\dff,1}\hat{A}_{\dff,2} \in \Psi^{\mathsf{m}_{1} + \mathsf{m}_2, \mathsf{r}_{1} + \mathsf{r}_2, \mathsf{l}_{1} + \mathsf{l}_2, \mathsf{s}_{1} + \mathsf{s}_2}_{\mathrm{sc,b,lp,res},\delta}( X^{\tindex} ; \mathcal{C}_{\tindex} \times {\mathbb{R}^{n_{\tindex}}} ),
\end{equation}
and principal symbol is multiplicative in the sense that
\begin{equation*}
\sigma( \hat{A}_{\dff,1} \hat{A}_{\dff,2} ) = \sigma( \hat{A}_{\dff,1} ) \sigma( \hat{A}_{\dff,2} ).
\end{equation*}
\item Let
\begin{equation*}
\begin{gathered}
\vom_{j}, \vor_{j}, \vos_{j} \in \mathcal{C}^{\infty}( \tcocf ), \quad \vol_{j} \in \mathcal{C}^{\infty}( \mathcal{C}_{\tindex} \times \overline{\mathbb{R}^{n_{\tindex}}} ), \\ 
\hat{A}_{\cf,j} \in
\Psi^{\mathsf{m}_j,\mathsf{r}_j,\mathsf{l}_j,\mathsf{s}_j}_{\mathrm{coc,lp,res},\delta}( [ 
\hat{X}^{\tindex} ; \{ 0 \} ] ; \mathcal{C}_{\tindex} \times {\mathbb{R}^{n_{\tindex}}} ), \quad j =1,2.
\end{gathered}
\end{equation*}
Then we have
\begin{equation} \label{composition symbol classes at cf}
\hat{A}_{\cf,1} \hat{A}_{\cf,2} \in \Psi^{\mathsf{m}_1 + \mathsf{m}_2 ,\mathsf{r}_1 + \mathsf{r}_2,\mathsf{l}_1 +\mathsf{l}_2 ,\mathsf{s}_1 + \mathsf{s}_2 }_{\mathrm{coc,lp,res},\delta}( [ 
\hat{X}^{\tindex} ; \{ 0 \} ] ; \mathcal{C}_{\tindex} \times {\mathbb{R}^{n_{\tindex}}} ),
\end{equation}
and principal symbol is multiplicative in the sense that
\begin{equation*}
\sigma( \hat{A}_{\cf,1} \hat{A}_{\cf,2} ) = \sigma( \hat{A}_{\cf,1} ) \sigma( \hat{A}_{\cf,2} ).
\end{equation*}
\end{enumerate} \par

Additionally, we have the following multiplicative results for the indicial operators:
\begin{enumerate}
    \item Suppose that 
    \begin{equation*}
    \text{$A_{j} \in \Psi^{m_{j}, \mathsf{r}_{j}, \mathsf{l}_{j}, \mathsf{v}_{j},  \mathsf{b}_{j}, \mathsf{s}_{j}}_{\mathrm{d3sc,3co,res},\delta}(\Xd)$ is classical at $\dff$ modulo $\Psi^{ \mathsf{m}_{j}, \mathsf{r}_{j}, \mathsf{l}_{j} - K , \mathsf{v}_{j} , \mathsf{b}_{j}, \mathsf{s}_{j}}_{\mathrm{d3sc,3co,res},\delta}( \Xd )$}
    \end{equation*}
for $j =1,2$ and some $K > 0$. Then 
    \begin{equation*}
    \text{$A_{1} A_{2}$ is classical at $\dff$ modulo $\Psi^{ \mathsf{m}_{1} + \mathsf{m}_{2}, \mathsf{r}_{1} + \mathsf{r}_{2}, \mathsf{l}_{1} + \mathsf{l}_{2} - K, \mathsf{v}_{1} + \mathsf{v}_{2},  \mathsf{b}_{1} + \mathsf{b}_{2}, \mathsf{s}_{1} + \mathsf{s}_{2} }_{\mathrm{d3sc,3co,res},\delta}(\Xd)$,}
    \end{equation*}
     and we have
\begin{equation} \label{multiplicative property for the indicial opeartor 3co,res, dff}
\hat{N}_{\dff} (A_{1} A_{2}) = \hat{N}_{\dff} (A_1) \hat{N}_{\dff} (A_2). 
\end{equation}
     \item Suppose that 
     \begin{equation*}
     \text{$A_{j} \in \Psi^{m_{j}, \mathsf{r}_{j}, \mathsf{l}_{j}, \mathsf{v}_{j},  \mathsf{b}_{j}, \mathsf{s}_{j}}_{\mathrm{d3sc,3co,res},\delta}(\Xd)$ is classical at $\cf$ modulo $\Psi^{ \mathsf{m}_{j}, \mathsf{r}_{j}, \mathsf{l}_{j} , \mathsf{v}_{j} , \mathsf{b}_{j} - K, \mathsf{s}_{j}}_{\mathrm{d3sc,3co,res},\delta}( \Xd )$}
     \end{equation*}
     for $j = 1,2$ and some $K > 0$. Then 
     \begin{equation*}
     \text{$A_{1} A_{2}$ is classical at $\cf$ modulo $\Psi^{ \mathsf{m}_{1} + \mathsf{m}_{2}, \mathsf{r}_{1} + \mathsf{r}_{2}, \mathsf{l}_{1} + \mathsf{l}_{2}, \mathsf{v}_{1} + \mathsf{v}_{2},  \mathsf{b}_{1} + \mathsf{b}_{2} - K, \mathsf{s}_{1} + \mathsf{s}_{2} }_{\mathrm{d3sc,3co,res},\delta}(\Xd)$,}
     \end{equation*}
     and we have
\begin{equation} \label{multiplicative property for the indicial operator 3co,res, cf}
\hat{N}_{\cf} (A_{1} A_{2}) = \hat{N}_{\cf} (A_1) \hat{N}_{\cf} (A_2). 
\end{equation}
\end{enumerate}

The same statements hold if $\vom = m$, $\vor = r$, $\vol = l$, $\vov = \nu$, $\vob = b$, $\vos = s$ are all constants, in which case we also set $\delta = 0$.
\end{proposition}

Now, recall that by construction, we can write any $A \in \Psf^{\vom, \vor, \vol, \vov, \vob, \vos}(\Xd)$ as a scattering operator near $\dmf$. Thus composition, as well as multiplicative property for the principal symbol map in this region will follow naturally from those of the scattering calculus. \par

Meanwhile, in a small neighborhood of either $\dff$ or $\cf$, we have constructed $A$ so that it takes the form of the partial quantization of some operator-valued symbol in the free variables. Thus in these regions, we can understand $A$ as being `partially scattering' in the free variables. Then composition can be understood partially as in the scattering case.

Nevertheless, there are minor issues in that one still has to consider how some of the usual techniques in scattering calculus, including asymptotic summations and symbol reductions, can be carried over to this context. In fact, this has already been considered in \S \ref{Variable orders compatibility subsection}. The discussion in that subsection then shows that there exists a Frech\'et structure for the local spaces of operator-valued symbols, and subsequently a Frech\'et structure for $\Psf^{\vom, \vor, \vol, \vov, \vob, \vos}(\Xd)$ as well. 
\par

By using the aforementioned Frech\'et structure for $\Psf^{\vom, \vor, \vol, \vov, \vob, \vos}(\Xd)$, we now make precise some of the notations we have used in \S \ref{Variable orders compatibility subsection}, where we considered asymptotic summations of operator-valued symbols.

\begin{lemma}[Asymptotic summations for operator-valued symbols]
\label{operator-valued symbols asymptotic definition} 
Suppose that $\vom, \vor, \vov, \vos \in \mathcal{C}^{\infty}( \psf \Xd )$ and $\vol, \vob \in \mathcal{C}^{\infty}( \mathcal{C}_{\tindex} \times \overline{\mathbb{R}^{n_{\tindex}}} )$.
\begin{enumerate}
\item For $j=1,2,3,...$, let
\begin{equation*}
A_{j} \in \Psf^{\vom, \vor, \vol - j ( 1 - 2 \delta ), \vov, \vob, \vos}(\Xd), \quad A_{\psi_{\dff},j} \coloneq \psi_{\dff} A_{j} \psi_{\dff},
\end{equation*}
where $\psi_{\dff} \in \mathcal{C}^{\infty}(\Xd)$ is any cut-off function at $\dff$. Moreover, let $\hat{A}_{\psi_{\dff},j}$ denote the operator-valued symbol of $A_{\psi_{\dff},j}$. Then there exists 
\begin{equation*}
\text{$A_{\psi_{\dff}} \in \Psf^{\vom, \vor, \vol, \vov, \vob, \vos}(\Xd)$ supported near $\dff$}
\end{equation*}
such that if $\hat{A}_{\psi_{\dff}}$ denotes the operator-valued symbol of $A_{\psi_{\dff}}$, then we have
\begin{equation}
\label{asymptotic summations lemma 1}
\hat{A}_{\psi_{\dff}} \sim \sum_{j=0}^{\infty} \hat{A}_{\psi_{\dff}, j}
\end{equation}
in the sense that
\begin{equation} 
\label{asymptotic summations lemma 2}
A_{\psi_{\dff}} - \sum_{j=0}^{N-1} A_{\psi_{\dff},j} \in \Psf^{\vom, -\infty, \vol - N ( 1 - 2\delta ), \vov, \vob, \vos}(\Xd).
\end{equation}
\item For $j=1,2,3,...$, let
\begin{equation*}
A_{j} \in \Psf^{\vom, \vor, \vol, \vov, \vob - j (1 - 2\delta), \vos}(\Xd), \quad A_{\cf,j} \coloneq \psi_{\cf} A_{j} \psi_{\cf},
\end{equation*}
where $\psi_{\cf} \in \mathcal{C}^{\infty}(\Xd)$ is any cut-off function at $\cf$. Moreover, let $\hat{A}_{\cf,j}$ denote the operator-valued symbol of $A_{\psi_{\cf},j}$. Then there exists 
\begin{equation*}
\text{$A \in \Psf^{\vom,\vor,\vol,\vov,\vob,\vos}(\Xd)$ supported near $\cf$}
\end{equation*}
such that if $\hat{A}_{\psi_{\cf}}$ denotes the operator-valued symbol of $A_{\psi_{\cf}}$, then we have
\begin{equation}
\label{asymptotic summations lemma 3}
\hat{A}_{\psi_{\cf}} \sim \sum_{j=0}^{\infty} \hat{A}_{\psi_{\cf},j}
\end{equation}
in the sense that
\begin{equation}
\label{asymptotic summations lemma 4}
A_{\psi_{\cf}} - \sum_{j=0}^{N-1} A_{\psi_{\cf}, j} \in \Psf^{\vom, \vor, \vol, \vov, \vob - N ( 1 - 2 \delta ), \vos} (\Xd).
\end{equation}
\end{enumerate}
\end{lemma}

\begin{proof}[Proof of Proposition \ref{Composition proposition}]
First of all, it is clear that
\begin{equation*}
A_{1} A_{2} : \mathcal{S}( \mathbb{R}^{n} ) \rightarrow \mathcal{S}( \mathbb{R}^{n} )
\end{equation*}
is a continuous linear map. \par

Now, let $\psi_{\bullet}, \tilde{\psi}_{\bullet} \in \mathcal{C}^{\infty}( \Xd )$, $\bullet = 0, \dff, \cf$ be chosen as in the beginning of \S \ref{subsection construction of variable order operators}, i.e., $\tilde{\psi}_{\dff}$ and $\tilde{\psi}_{\cf}$ are cut-off functions at $\dff$ and $\cf$ respectively, while $\tilde{\psi}_{0}$ is supported away from both $\dff$ and $\cf$. Moreover, each $\psi_{\bullet}$ has the same support requirement as that of $\tilde{\psi}_{\bullet}$, with the additional condition that $\psi_{\bullet} = 1$ on the support of $\tilde{\psi}_{\bullet}$. Then we need to show that each $\psi_{\bullet} A_1 A_2 \psi_{\bullet}$ has the correct structure, and that
\begin{equation}
\label{this term needs to be rapidly decreasing composition subsection} 
\text{$( 1 - {\psi}_{\bullet}) A_1 A_2 \tilde{\psi}_{\bullet}$ has rapidly decaying smooth kernels}, \quad \bullet = 0, \dff, \cf.
\end{equation}
\par

To show (\ref{this term needs to be rapidly decreasing composition subsection}), we simply choose $\phi_{\bullet} \in \mathcal{C}^{\infty}(\Xd)$ such that $\phi_{\bullet} = 1$ on $\supp \tilde{\psi}_{\bullet}$ and $\phi_{\bullet} = 1$ on $\supp \hat{\psi}_{\bullet}$. Then we can write
\begin{equation}
\label{composition calculation 1}
( 1 - \psi_{\bullet} ) A_{1} A_{2} \tilde{\psi}_{\bullet} = ( 1 - \psi_{\bullet} )A_{1} \phi_{\bullet} A_{2} \tilde{\psi}_{\bullet} + ( 1 - \psi_{\bullet} ) A_{1} ( 1 - \phi_{\bullet} ) A_2 \tilde{\psi}_{\bullet}
\end{equation}
from which (\ref{this term needs to be rapidly decreasing composition subsection}) follows. Indeed, the two terms on the right hand side of (\ref{composition calculation 1}) are both rapidly decaying and smooth, since $ ( 1 - \psi_{\bullet} ) A_{1} \phi_{\bullet}$ and $(1- \phi_{\bullet} ) A_{2} \tilde{\psi}_{\bullet}$ have this property. \par

Next, let $\varphi_{\bullet} \in \mathcal{C}^{\infty}(\Xd)$ be chosen such that $\varphi_{\bullet} = 1$ on $\supp \psi_{\bullet}$. Then we have
\begin{equation}
\label{composition calculation 2}
\psi_{\bullet} A_1 A_2 \psi_{\bullet} = \psi_{\bullet} \varphi_{\bullet} A_1  \varphi_{\bullet}^2  A_2 \varphi_{\bullet}  \psi_{\bullet} + \psi_{\bullet} A_1 (  1 - \varphi_{\bullet}^2 ) A_2 \psi_{\bullet}.
\end{equation}
Note that each $A_{1} ( 1- \varphi_{\bullet}^2 ) A_{2} {\psi}_{\bullet}$ has a rapidly decreasing smooth kernel since $( 1 - \varphi_{\bullet}^2 ) A_{2} {\psi}_{\bullet}$ has this property. Moreover, suppose that $\bullet = 0$. Then by construction, we have 
\begin{equation*}
\psi_{0}  \varphi_{0} A_1 \varphi_{0}^2 A_2 \varphi_{0}  \psi_{0} = \psi_{0} A_{ \varphi_{0}, \mathrm{sc}, 1} A_{ \varphi_0, \mathrm{sc},2} \psi_{0}
\end{equation*}
where $A_{ \varphi_0 , \mathrm{sc},j} \in \Psi^{ \mathsf{m}_{j}, \mathsf{r}_{j} }_{\mathrm{sc},\delta}( \overline{\mathbb{R}^{n}} )$, $j=1,2$. Thus by the standard composition results for scattering operators, we have 
\begin{equation} 
\label{d3sc,3co,res scattering components composition}
\psi_{0} A_{1} A_{2} \psi_{0} \in \Psi_{\mathrm{sc},\delta}^{\vom_1 + \vom_2, \vor_ 1 + \vor_2}( \overline{\mathbb{R}^{n}} )
\end{equation}
as required.

If instead we take $\bullet = \dff$, then we have 
\begin{equation} \label{composition calculation 3}
\psi_{\dff} \varphi_{\dff} A_{1} \varphi_{\dff}^2  A_{2} {\psi}_{\dff} \psi_{\dff} = \psi_{\dff} A_{ \varphi_{\dff} , 1} A_{ \varphi_{\dff} ,2} \psi_{\dff}.
\end{equation}
Here each $A_{ \varphi_{\dff} , j} = \varphi_{\dff} A \varphi_{\dff}$, $j=1,2$, is given by 
\begin{equation*}
A_{ \varphi_{\dff} ,j} = \frac{1}{(2\pi)^{n_{\tindex}}} \int_{\mathbb{R}^{n_{\tindex}}} e^{ i ( z_{\tindex} - z_{\tindex}' ) \cdot \zeta_{\tindex} } \hat{A}_{ \varphi_{\dff} , j} ( z_{\tindex}, \zeta_{\tindex} ) d \zeta_{\tindex} | dz_{\tindex}' |,
\end{equation*}
where $\hat{A}_{ \varphi_{\dff} , j}$ is the operator-valued symbol defined as in \S \ref{subsection construction of variable order operators}. See also Remark \ref{remark after the construction of the calculus variable order}. \par

Thus, we are naturally motivated to proceed as we did in the setting of compositions for scattering operators, albeit only partially in the free variables (in other words, we wish to carry out the usual procedure as if the operator-valued symbols were merely complex valued), in which case we need, in addition to the concept of asymptotic summation as introduced in Lemma \ref{operator-valued symbols asymptotic definition}, also suitable symbol reduction formuale for the operator-valued symbols. \par

For this purpose, we will also state the following lemma:

\begin{lemma}[Symbol reductions for operator-valued symbols]
\label{composition subsection symbol reduction formulae lemma}
Suppose that $\vom, \vor, \vov, \vos \in \mathcal{C}^{\infty}( \psf \Xd )$, $\vol, \vob \in \mathcal{C}^{\infty}( \mathcal{C}_{\tindex} \times \overline{\mathbb{R}^{n_{\tindex}}} )$ and let $A \in \Psf^{\vom, \vor, \vol, \vov, \vob, \vos}(\Xd)$. 
\begin{enumerate}
\item Let $\psi_{\dff} \in \mathcal{C}^{\infty}(\Xd)$ be a cut-off function at $\dff$, and let $\hat{A}_{\psi_{\dff}}$ be the operator-valued symbol of $A_{\psi_{\dff}} = \psi_{\dff} A \psi_{\dff}$. Then there exists $\hat{A}_{\psi_{\dff}}^{R}$ which lives in the same local Frech\'et space as $\hat{A}_{\psi_{\dff}}$, such that
\begin{equation}
\label{composition subsection right quantization formula near dff}
A_{\psi_{\dff}} = \frac{1}{(2\pi)} \int_{\mathbb{R}^{n_{\tindex}}} e^{i ( z_{\tindex} - z_{\tindex}' ) \cdot \zeta_{\tindex} } \hat{A}^{R}_{\psi_{\dff}} ( z_{\tindex}', \zeta_{\tindex} ) d\zeta_{\tindex} |dz_{\tindex}'|.
\end{equation} 
We will refer to (\ref{composition subsection right quantization formula near dff}) as the right quantization of $\hat{A}_{\psi_{\dff}}^{R}$. Moreover, it holds that 
\begin{equation}
\label{composition subsection right quantization formula near dff 1}
\hat{A}_{\psi_{\dff}}^{R}( z_{\tindex}, \zeta_{\tindex} )  \sim \sum_{\beta_{\tindex} \in \mathbb{N}_{0}^{n_{\tindex}} } \frac{(-1)^{|\beta_{\tindex}|}}{\beta_{\tindex}!} \partial_{\zeta_{\tindex}}^{\beta_{\tindex}} D_{z_{\tindex}}^{\beta_{\tindex}} \hat{A}_{\psi_{\dff}} ( z_{\tindex}, \zeta_{\tindex} ) 
\end{equation}
in the sense of Lemma \ref{operator-valued symbols asymptotic definition}, part (1). \par

If instead $A_{\psi_{\dff}}$ is defined by (\ref{composition subsection right quantization formula near dff}), then there exists $\hat{A}_{\psi_{\dff}}^{L}$ such that $\hat{A}_{\psi_{\dff}}^{L}$ is the operator-valued symbol of $A_{\psi_{\dff}}$, and we have
\begin{equation}
\label{composition subsection right quantization formula near dff 2}
\hat{A}_{\psi_{\dff}}^{L}( z_{\tindex}, \zeta_{\tindex} )  \sim \sum_{\beta_{\tindex} \in \mathbb{N}_{0}^{n_{\tindex}} } \frac{1}{\beta_{\tindex}!} \partial_{\zeta_{\tindex}}^{\beta_{\tindex}} D_{z_{\tindex}}^{\beta_{\tindex}} \hat{A}_{\psi_{\dff}} ( z_{\tindex}, \zeta_{\tindex} ).
\end{equation}
In particular, we must have $\hat{A}_{\psi_{\dff}} = \hat{A}_{\psi_{\dff}}^{L}$. \par

We will refer to $\hat{A}_{\psi_{\dff}}^{R}$ as the right operator-valued symbol of $A_{\psi_{\dff}}$ and $\hat{A}_{\psi_{\dff}}^{L}$ the left operator-valued symbol of $A_{\psi_{\dff}}$. If such distinctions are not made, then it will always be assumed that we are speaking of the left operator-valued symbol.
\item Let $\psi_{\cf} \in \mathcal{C}^{\infty}(\Xd)$ be a cut-off function at $\cf$, and let $\hat{A}_{\psi_{\cf}}$ be the operator-valued symbol of $A_{\psi_{\cf}} = \psi_{\cf} A \psi_{\cf}$. Then there exists $\hat{A}_{\psi_{\cf}}^{R}$ which lives in the same local Frech\'et space as $\hat{A}_{\psi_{\cf}}$, such that
\begin{equation}
\label{composition subsection right quantization formula near cf}
A_{\psi_{\cf}} = \frac{1}{(2\pi)} \int_{\mathbb{R}^{n_{\tindex}}} e^{i ( z_{\tindex} - z_{\tindex}' ) \cdot \zeta_{\tindex}^{\mathrm{3co}} } \hat{A}^{R}_{\psi_{\cf}} ( z_{\tindex}', \zeta_{\tindex}^{\mathrm{3co}} ) d\zeta_{\tindex}^{\mathrm{3co}} |dz_{\tindex}'|.
\end{equation} 
We will refer to (\ref{composition subsection right quantization formula near cf}) as the right quantization of $\hat{A}^{R}_{\psi_{\cf}}$. Moreover, it holds that 
\begin{equation*}
\hat{A}_{\psi_{\cf}}^{R}( z_{\tindex}, \zeta_{\tindex}^{\mathrm{3co}} )  \sim \sum_{\beta_{\tindex} \in \mathbb{N}_{0}^{n_{\tindex}} } \frac{(-1)^{|\beta_{\tindex}|}}{\beta_{\tindex}!}  \partial_{\zeta_{\tindex}^{\mathrm{3co}}}^{\beta_{\tindex}} D_{z_{\tindex}}^{\beta_{\tindex}} \hat{A}_{\psi_{\dff}} ( z_{\tindex}, \zeta_{\tindex} )
\end{equation*}
in the sense of Lemma \ref{operator-valued symbols asymptotic definition}, part (2). \par

If instead $A_{\psi_{\cf}}$ is defined by (\ref{composition subsection right quantization formula near dff}), then there exists $\hat{A}_{\psi_{\cf}}^{L}$ such that $\hat{A}_{\psi_{\cf}}^{L}$ is the operator-valued symbol of $A_{\psi_{\cf}}$, and we have 
\begin{equation*}
\hat{A}_{\psi_{\cf}}^{L}( z_{\tindex}, \zeta_{\tindex}^{\mathrm{3co}} )  \sim \sum_{\beta_{\tindex} \in \mathbb{N}_{0}^{n_{\tindex}} } \frac{1}{\beta_{\tindex}!}  \partial_{\zeta_{\tindex}^{\mathrm{3co}}}^{\beta_{\tindex}} D_{z_{\tindex}}^{\beta_{\tindex}} \hat{A}_{\psi_{\dff}} ( z_{\tindex}, \zeta_{\tindex}^{\mathrm{3co}} ).
\end{equation*}
In particular, we must have $\hat{A}_{\psi_{\cf}} = \hat{A}_{\psi_{\cf}}^{L}$. \par

We will refer to $\hat{A}_{\psi_{\cf}}^{R}$ as the right operator-valued symbol of $A_{\psi_{\cf}}$ and $\hat{A}_{\psi_{\cf}}^{L}$ the left operator-valued symbol of $A_{\psi_{\cf}}$. If such distinctions are not made, then it will always be assumed that we are speaking of the left operator-valued symbol.
\end{enumerate}
\end{lemma}

By using Lemma \ref{composition subsection symbol reduction formulae lemma}, we also have the following familiar formulae: 

\begin{lemma}
\label{composition subsection composition reduction lemma}
Let the assumptions of Proposition \ref{Composition proposition} be satisfied and set $A_3 \coloneq A_1 A_2$.
\begin{enumerate}
\item Let $\psi_{\dff}, \varphi_{\dff} \in \mathcal{C}^{\infty}(\Xd)$ be cut-off functions at $\dff$ such that $\varphi_{\dff} = 1$ on the support of $\psi_{\dff}$. Moreover, let 
\begin{equation*}
\begin{gathered}
\text{$\hat{A}_{\varphi_{\dff},j}$ be the operator-valued symbols of $A_{\varphi_{\dff},j} = \varphi_{\dff} A_{j} \varphi_{\dff}$, $j = 1,2$,} \\
\text{$\hat{A}_{\psi_{\dff},3}$ be the operator-valued symbol of $A_{\psi_{\dff},3} = \psi_{\dff} A_{3} \psi_{\dff}$.}
\end{gathered}
\end{equation*}
 Then we have
\begin{equation}
\label{composition subsection asymptotic of operator-valued symbol near dff}
\hat{A}_{\psi_{\dff},3} ( z_{\tindex}, \zeta_{\tindex} ) \sim \sum_{ \beta_{\tindex} \in \mathbb{N}^{n_{\tindex}}_{0}  } \frac{1}{\beta_{\tindex}!} (\partial_{\zeta_{\tindex}}^{\beta_{\tindex}} \hat{A}_{\varphi_{\dff},1} ) (z_{\tindex}, \zeta_{\tindex} ) ( D_{z_{\tindex}}^{\beta_{\tindex}} \hat{A}_{\varphi_{\dff},2} )( z_{\tindex}, \zeta_{\tindex} )
\end{equation}
in the sense of Lemma \ref{operator-valued symbols asymptotic definition}, part (1).
\item Let $\psi_{\cf}, \varphi_{\cf} \in \mathcal{C}^{\infty}(\Xd)$ be cut-off functions at $\cf$ such that $\varphi_{\cf} = 1$ on the support of $\psi_{\cf}$. Moreover, let
\begin{equation*}
\begin{gathered}
\text{$\hat{A}_{\varphi_{\cf},j}$ be the operator-valued symbols of $A_{\varphi_{\cf},j} \coloneq \varphi_{\cf} A_{j} \varphi_{\cf}$, $j = 1,2$,} \\
\text{$\hat{A}_{\psi_{\cf},3}$ be the operator-valued symbol of $A_{\psi_{\cf},3} \coloneq \psi_{\cf} A_{3} \psi_{\cf}$.}
\end{gathered}
\end{equation*}
Then we have
\begin{equation*}
\hat{A}_{\psi_{\cf},3} ( z_{\tindex}, \zeta_{\tindex}^{\mathrm{3co}} ) \sim \sum_{ \beta_{\tindex} \in \mathbb{N}^{n_{\tindex}}_{0}  } \frac{1}{\beta_{\tindex}!} (\partial_{\zeta_{\tindex}}^{\beta_{\tindex}} \hat{A}_{\varphi_{\dff},1} ) (z_{\tindex}, \zeta_{\tindex}^{\mathrm{3co}} ) ( D_{z_{\tindex}}^{\beta_{\tindex}} \hat{A}_{\varphi_{\cf},2} )( z_{\tindex}, \zeta_{\tindex}^{\mathrm{3co}} )
\end{equation*}
in the sense of Lemma \ref{operator-valued symbols asymptotic definition}, part (2).
\end{enumerate}
\end{lemma}
\begin{proof}[Proof of Lemma \ref{composition subsection composition reduction lemma}]
We will only prove part (1) as the proof of part (2) is similar. \par

By formula (\ref{composition calculation 2}), our question is reduced to the consideration of 
\begin{equation*}
\psi_{\dff} \varphi_{\dff} A_{1} \varphi_{\dff}^2 A_2 \varphi_{\dff} \psi_{\dff} = \psi_{\dff} A_{\varphi_{\dff}, 1} A_{\varphi_{\dff}, 2} \psi_{\dff}.
\end{equation*}
Let $\hat{A}_{\varphi_{\dff}, 2}^{R}$ be the right operator-valued symbol of $A_{\varphi_{\dff}, 2}$. Then by standard Fourier analysis (which will be applied in the free variable only), we have
\begin{align}
\begin{split}
\label{composition double symbol reduction formula}
& \psi_{\dff} A_{\varphi_{\dff}, 1} A_{\varphi_{\dff}, 2} \psi_{\dff} \\
& \qquad = \frac{1}{(2\pi)^{n_{\tindex}}} \int_{\mathbb{R}^{n_{\tindex}}} e^{ i ( z_{\tindex} - z_{\tindex}' ) \cdot \zeta_{\tindex} } (\psi_{\dff} \hat{A}_{\varphi_{\dff}, 1} ) ( z_{\tindex}, \zeta_{\tindex} ) ( \hat{A}_{\varphi_{\dff}, 2 }^{R} \psi_{\dff} ) ( z_{\tindex}', \zeta_{\tindex} ) d\zeta_{\tindex} |dz_{\tindex}'|.
\end{split}
\end{align}
In other words, (\ref{composition double symbol reduction formula}) is given by the quantization of a `double operator-valued symbol' in the free variables. \par

It follows from the standard procedure (see for instance \cite[\S 5.3.3]{AndrasSM}) that the method of symbol reduction can also be applied for the double symbols. Thus by using (\ref{composition subsection right quantization formula near dff 1}) and (\ref{composition subsection right quantization formula near dff 2}), we have
\begin{equation*}
\hat{A}_{\psi_{\dff},3}(z_{\tindex}, \zeta_{\tindex}) \sim  \sum_{ \beta_{\tindex} \in \mathbb{N}^{n_{\tindex}}_{0}  } \frac{1}{\beta_{\tindex}!} \psi_{\dff} ( \partial_{\zeta_{\tindex}}^{\beta_{\tindex}}  \hat{A}_{\varphi_{\dff},1} ) (z_{\tindex}, \zeta_{\tindex} ) \big( D_{z_{\tindex}}^{\beta_{\tindex}} ( \hat{A}_{\varphi_{\dff},2} \psi_{\dff} ) \big)( z_{\tindex}, \zeta_{\tindex} ),
\end{equation*}
from which the required asymptotics follow easily since $\psi_{\dff} = 1$ near $\dff$.
\end{proof}

In particular, Lemma \ref{composition subsection composition reduction lemma} implies that
\begin{equation*}
\psi_{\dff} A_1 A_2 \psi_{\dff}, \psi_{\cf} A_1 A_2 \psi_{\cf} \in \Psf^{\vom_{1} + \vom_{2}, \vor_{1} + \vor_{2}, \vol_{1} + \vol_{2}, \vov_{1} + \vov_{2} , \vob_{1} + \vob_{2}, \vos_{1} + \vos_{2}} (\Xd).
\end{equation*}
By combining this with the above discussion, we can conclude that (\ref{composition regularity result d3sc,3co,res}) must hold. Moreover, the lemma also shows that compositions preserve partial classicality as stated.

Meanwhile, composition for the indicial operators (\ref{composition symbol classes at dff}) and (\ref{composition symbol classes at cf}) will follow essentially from the same arguments as compositions in the large-parameter b- and cone calculi respectively. The only exception is that we now have to consider decay at the new faces. However, since the operators are microlocal at these faces (in the sense that the principal symbol maps capture leading order decay at these faces), compositions are easily seen to be well-behaved by symbol reductions. 
\par

To show (\ref{multiplicative property for the principal symbol d3sc,3co,res}), let $\psi_{j}, \tilde{\psi_{j}} \in \mathcal{C}^{\infty}(\Xd)$ be chosen as in the proof of Proposition \ref{proposition principal symbol map for the further-resolved second microlocalized operators} (thus in particular, $\psi_j = 1$ on the support of $\tilde{\psi}_j$) for $j = 1, ... , N$. Then we have
\begin{equation}
\label{composition symbolic product formula operator}
A_{1} A_2 = \sum_{j,k = 1}^{N} \tilde{\psi}_j \psi_{j} A_1 \psi_{j}  \psi_{k} A_2 \psi_{k} \tilde{\psi}_{k}
\end{equation} 
modulo the addition of an element in $\Psf^{-\infty, -\infty, \vol, -\infty, \vob, -\infty}(\Xd)$. If $\supp \psi_j \cap \supp \psi_{k} \neq \emptyset $, then for $i =1,2$, either we have $\psi_{j} A_{i} \psi_{j} \in \Psi_{\mathrm{sc},\delta}^{\vom_{i}, \vor_{i}}( \overline{\mathbb{R}^{n}} )$, or $\psi_{j} A_{i} \psi_{j}$ can be written as (still modding out $\Psf^{-\infty, -\infty, \vol, -\infty, \vob, -\infty}(\Xd)$) one of the quantizations $(B_i)_{\psi_{j}, \mathrm{3co,b}}$, $(B_{i})_{\psi_{j}, \mathrm{3co,co,b}}$, $(B_{i})_{\psi_{j}, \mathrm{3co,co,sc}}$ as defined in the earlier sections (with slight abuse of notations in the obvious sense). \par

Let $b_{i, \psi_{j}}$ be the quantization of $\psi_{j} A_1 \psi_{j}$ in all cases. Then by the standard symbol reduction calculations, combined with the same reasons outlined in the proof of Proposition \ref{proposition principal symbol map for the further-resolved second microlocalized operators}, we see that the full symbol $c_{\psi_{j}, \psi_{k}}$ of $\psi_{j} A_{1} \psi_{j} \psi_{k} A_2 \psi_{k} $ must be given by an asymptotic summation
\begin{equation*}
c_{\psi_{j}, \psi_{k}} \sim \sum_{ i = 0}^{\infty}  c_{\psi_{j}, \psi_{k}, i}
\end{equation*}
in the sense that 
\begin{equation*}
c_{\psi_{j}, \psi_{k}} - \sum_{i=0}^{I - 1} c_{\psi_{j}, \psi_{k}, i} \in S^{\vom - I ( 1 - 2 \delta ), \vor - I ( 1 - 2 \delta ), \vol, \vov - I ( 1 - 2 \delta), \vob, \vos - I ( 1 - 2 \delta ) }_{\delta} ( \psf \Xd ),
\end{equation*}
where, depending on the support of $\psi_{j}$ and $\psi_{k}$, we have defined $c_{\psi_{j}, \psi_{k}, i}$ by one of
\begin{equation*}
\begin{gathered}
\sum_{|\beta| = i} \frac{1}{\beta!} ( \partial_{( \zeta_{\tindex}, - \utaub , \umub )} ^{\beta} b_{1, \psi_{j}} ) ( D_{( z_{\tindex}, t^{\tindex}, y^{\tindex} )}^{\beta} b_{2, \psi_{k}}  ), \quad  \sum_{|\beta| = i} \frac{1}{\beta!} ( \partial_{( \zeta_{\tindex}^{\mathrm{3co}}, - \utaucob , \umucob )} ^{\beta} b_{1, \psi_{j}} ) ( D_{( z_{\tindex}, \hat{t}_{\tindex}, y^{\tindex} )}^{\beta} b_{2, \psi_{k}}  ), \\
\sum_{|\beta| = i} \frac{1}{\beta!}  ( \partial_{( \zeta_{\tindex}^{\mathrm{3co}}, \zeta^{\tindex}_{\mathrm{co,b}} )}^{\beta} b_{1, \psi_{j}} ) ( D_{( z_{\tindex}, \hat{z}^{\tindex} )}^{\beta} b_{2, \psi_{k}} ), \quad  \sum_{|\beta| = i} \frac{1}{\beta!} ( \partial_{ (z_{\tindex}, z^{\tindex} ) }^{\beta} b_{1, \psi_j} ) ( D_{ (\zeta_{\tindex}, \zeta^{\tindex} )}^{\beta} b_{2, \psi_{k}} ), 
\end{gathered}
\end{equation*}
corresponding respectively to the symbolic estimates required in Definition \ref{definition of variable orders symbols}. 
\par

As in the proof of Proposition \ref{proposition principal symbol map for the further-resolved second microlocalized operators}, it follows that the key observation now is that the combined differentiations in any position variable and its dual momentum variable of an element in $S^{\mathsf{m},\mathsf{r},\mathsf{l},\mathsf{v},\mathsf{b},\mathsf{s}}_{\delta}( \overline{ ^{\mathrm{d3sc,3co,res}}T^{\ast}}\Xd )$ produce at least $1- 2\delta$ order of decay at $\psf_{\dmf} \Xd$, $\dtsccf$, $\rf$ and the fiber infinity. Consequentially, we have
\begin{equation*}
c_{\psi_{j}, \psi_{k}, i} \in S^{\vom - i(1 - 2 \delta), \vor - i( 1 - 2\delta ), \vol, \vov - i( 1 - 2\delta ), \vob, \vos - i( 1 - 2\delta )}( \psf \Xd ).
\end{equation*}
Thus, it follows from (\ref{composition symbolic product formula operator}) that the principal symbol $\sigma(A_1 A_2)$ is given by
\begin{equation*}
\sum_{j,k=1}^{N} \tilde{\psi}_j b_{1, \psi_j} b_{2, \psi_j} \tilde{\psi}_k =  \Big( \sum_{j=1}^{N} \tilde{\psi}_j b_{1, \psi_{j}}  \Big) \Big( \sum_{k=1}^{N}  b_{1, \psi_{k}}  \tilde{\psi}_k \Big) = \sigma(A_1) \sigma(A_2) 
\end{equation*}
as required. 
\par

Finally, we will show (\ref{multiplicative property for the indicial opeartor 3co,res, dff}). This follows directly from (\ref{composition subsection asymptotic of operator-valued symbol near dff}), and we have
\begin{align*}
\hat{N}_{\dff}(A_{1} A_{2}) & = \lim_{x_{\tindex} \rightarrow 0} x_{\tindex}^{\mathsf{l}_1 + \mathsf{l}_2} \hat{A}_{\psi_{\dff},3} ( x_{\tindex}, y_{\tindex}, \tscblz ) \\
& = \lim_{x_{\tindex} \rightarrow 0} x_{\tindex}^{\mathsf{l}_1} \hat{A}_{\varphi_{\dff},1}( x_{\tindex}, y_{\tindex}, \tscblz ) \lim_{x_{\tindex} \rightarrow 0} x_{\tindex}^{\mathsf{l_2}} \hat{A}^{R}_{\varphi_{\dff}, 2}( x_{\tindex}, y_{\tindex}, \tscblz )  \\
& = \lim_{x_{\tindex} \rightarrow 0} x_{\tindex}^{\mathsf{l}_1} \hat{A}_{\varphi_{\dff},1}( x_{\tindex}, y_{\tindex}, \tscblz ) \lim_{x_{\tindex} \rightarrow 0} x_{\tindex}^{\mathsf{l}_2} \hat{A}_{\varphi_{\dff},2}( x_{\tindex}, y_{\tindex}, \tscblz )  \\
& = \hat{N}_{\dff}(A_{1}) \hat{N}_{\dff}(A_{2}). 
\end{align*}
The case of (\ref{multiplicative property for the indicial operator 3co,res, cf}) follows analogously. 
\end{proof}

\begin{remark}[Uniqueness of left/right reductions] 
\label{uniqueness of left and right reductions}
If $A \in \Psi^{ \mathsf{m}, \mathsf{r}, \mathsf{l}, \mathsf{v}, \mathsf{b} , \mathsf{s} }_{\mathrm{d3sc,3co,res},\delta}(\Xd)$, then by Lemma \ref{operator-valued symbols asymptotic definition}, locally we can always write $A_{\dff}$ as the left/right partial quantizations of operator-valued symbols $\hat{A}_{\psi_{\dff}}^{L}/\hat{A}_{\psi_{\dff}}^{R}$. It follows by taking the Fourier transform that these operator-valued symbols must be unique. This holds analogously at $\cf$.  
\end{remark}

Next we will discuss adjoints. Fix a smooth $\mathrm{3co}$-density $\nu_{\mathrm{3co}} \in \mathcal{C}^{\infty}( \Xd ; {^{\mathrm{3co}}\Omega}\Xd )$ such that $\nu_{\mathrm{3co}} > 0$ identically, and consider the space of $L^{2}$ functions with respect to $\nu_{\mathrm{3co}}$:
\begin{equation} \label{L^{2} space three-cone}
L^{2}_{\mathrm{3co}}( \Xd ) \coloneq L^{2}( \Xd , \nu_{\mathrm{3co}} ). 
\end{equation}
For every $A \in \Psi_{\mathrm{d3sc,3co,res},\delta}^{\vom,\vor,\vol,\vov,\vob,\vos}(\Xd)$, we let $A^{\ast}$ denote the adjoint of $A$ with respect to $L^{2}_{\mathrm{3co}}(\Xd)$. \par

Suppose we write $A$ as a sum
\begin{equation}
\label{adjoint decomposition of a generic operator}
A = A_{\psi_{0}, \mathrm{sc}} \tilde{\psi}_0 + A_{\psi_{\dff}} \tilde{\psi}_{\dff} + A_{\psi_{\cf}} \tilde{\psi}_{\cf} + K,
\end{equation}
where $K$ has rapidly decreasing smooth kernel, $A_{\psi_{0}, \mathrm{sc}} \in \Psi_{\mathrm{sc},\delta}^{\vom, \vol}( \overline{\mathbb{R}^{n}} )$, while $A_{\psi_{\dff}} = \psi_{\dff} A \psi_{\dff}$, $A_{\psi_{\cf}} = \psi_{\cf} A \psi_{\cf}$ are defined by partial quantizations. Then it would be enough if we can take the adjoint of each term in (\ref{adjoint decomposition of a generic operator}). \par

Let $\ast_{\mathrm{3co}}$ denote the adjoint taken with respect to the aforementioned $\mathrm{3co}$-density $\nu_{\mathrm{3co}}$. Then it is clear that $K^{\ast_{\mathrm{3co}}}$ will be another rapidly decreasing smooth kernel. Moreover, since any three-cone density will restrict to a scattering density away from $\dff \cup \cf$, we know that $A_{\psi_{0}, \mathrm{sc}}^{\ast_{\mathrm{3co}}}$ must also be another scattering operator, 
\par

It remains to consider the adjoints of $A_{\psi_{\dff}}$ and $A_{\psi_{\cf}}$, which only need be done locally near $\dff$ and $\cf$ respectively. Recall that $\nu_{\mathrm{3co}}$ is locally near $\dff$ equal to the density
\begin{equation*}
| dz_{\tindex} | \nu_{\mathrm{b}}
\end{equation*}
where $\nu_{\mathrm{b}} \in \mathcal{C}^{\infty}( X^{\tindex} ; {^{\mathrm{b}}\Omega} X^{\tindex} )$ is some strictly positive b-density. Thus, if $u$ and $v$ are supported in a sufficiently small neighborhood of $\dff$, then we have
\begin{equation} \label{L2 inner product three-cone near dff}
\langle  u, v \rangle_{L^{2}_{\mathrm{3co}}(\Xd)} = \int_{\mathbb{R}^{n_{\tindex}}} \langle u , v \rangle_{L^{2}_{\mathrm{b}}(X^{\tindex})} dz_{\tindex}
\end{equation}
where $L^{2}_{\mathrm{b}}( X^{\tindex} ) = L^{2}( X^{\tindex} , \nu_{\mathrm{b}} )$. Subsequently, the relationship between $A_{\psi_{\dff}}$ and $A_{\psi_{\dff}}^{\ast_{\mathrm{3co}}}$ can be understood via the usual calculation in the scattering setting{\ep}much like our consideration of composition above. \par

Likewise, $\nu_{\mathrm{3co}}$ is locally near $\cf$ equal to the density
\begin{equation*}
|dz_{\tindex}| \nu_{\mathrm{co}}
\end{equation*}
where $\nu_{\mathrm{co}} \in \mathcal{C}^{\infty} ( [ \hat{X}^{\tindex} ; \{ 0 \} ] ; {^{\mathrm{co}}\Omega} [ \hat{X}^{\tindex} ; \{ 0 \} ] )$ is some strictly positive $\mathrm{co}$-density. Thus, if $u$ and $v$ are supported in a sufficiently small neighborhood of $\cf$, then we have
\begin{equation*}
\langle u , v \rangle_{L^{2}_{\mathrm{3co}}(\Xd)} = \int_{\mathbb{R}^{n_{\tindex}}} \langle u , v \rangle_{L^{2}_{\mathrm{co}}( [ \hat{X}^{\tindex} ; \{ 0 \} ] )} dz_{\tindex},
\end{equation*}
where $L^{2}_{\mathrm{co}}( [ \hat{X}^{\tindex} ; \{ 0 \} ] ) = L^{2}( [ \hat{X}^{\tindex} ; \{ 0 \} ] , \nu_{\mathrm{co}} )$. \par

It follows that we have:
\begin{proposition}
\label{proposition symbol and indicial operators for adjoint}
Suppose that $\vom, \vor, \vov, \vos \in \mathcal{C}^{\infty}( \psf \Xd )$, $\vol, \vob \in \mathcal{C}^{\infty}( \mathcal{C}_{\tindex} \times \overline{\mathbb{R}^{n_{\tindex}}} )$ and let $A \in \Psf^{\vom, \vor, \vol, \vov, \vob, \vos}(\Xd)$. Then we have
\begin{equation*}
A^{\ast_{\mathrm{3co}}} \in \Psf^{\vom, \vor, \vol, \vov, \vob, \vos}(\Xd), \  A^{\ast} \in \Psf^{\vom, \vor, \vol, \vov, \vob, \vos}(\Xd),
\end{equation*}
where $A^{\ast}$ denotes the adjoint of $A$ taken with respect to $L^{2}$. Moreover, we have
\begin{equation}
\label{formula for adjoint symbol three-cone case}
\sigma(A^{\ast_{\mathrm{3co}}}) = \overline{ \sigma (A) }.
\end{equation}
Additionally, assume that $K > 0$. Then:
\begin{enumerate}
\item If $A$ is classical at $\dff$ modulo $\Psf^{\vom, \vor, \vol - K, \vov, \vob, \vos}(\Xd)$, then so is $A^{\ast_{\mathrm{3co}}}$, and we have
\begin{equation}
\label{indicial operator adjoint formula dff}
\hat{N}_{\dff}(A^{\ast_{\mathrm{3co}}}) = \hat{N}_{\dff}(A)^{\ast_{\mathrm{b}}},
\end{equation}
where $\ast_{\mathrm{b}}$ denotes adjoint taken with respect to $L^{2}_{\mathrm{b}}( X^{\tindex} )$.
\item If $A$ is classical at $\cf$ modulo $\Psf^{\vom, \vor, \vol, \vov, \vob - K, \vos}(\Xd)$, then so is $A^{\ast_{\mathrm{3co}}}$, and we have
\begin{equation}
\label{indicial operator adjoint formula cf}
\hat{N}_{\cf}(A^{\ast_{\mathrm{3co}}}) = \hat{N}_{\cf}(A)^{\ast_{\mathrm{co}}},
\end{equation}
where $\ast_{\mathrm{co}}$ denotes adjoint taken with respect to $L^{2}_{\mathrm{co}}( [ \hat{X}^{\tindex} ; \{ 0 \} ] )$.
\end{enumerate}

The same statements hold if $\vom = m$, $\vor = r$, $\vol = l$, $\vov = \nu$, $\vob = b$, $\vos = s$ are all constants, in which case we also set $\delta = 0$.
\end{proposition}
\begin{proof}
Note that the membership of $A^{\ast}$ follows from the fact that $\nu_{\mathrm{3co}}$ differs from the Euclidean density via the multiplication of a strictly positive $\mathcal{C}^{\infty}(\Xd)$ function. Moreover, the proof of (\ref{formula for adjoint symbol three-cone case}) is standard and will thus be omitted. \par

To prove (\ref{indicial operator adjoint formula dff}), notice that only the term $A_{\psi_{\dff}}$ in (\ref{adjoint decomposition of a generic operator}) matters. Thus, we will start with (\ref{L2 inner product three-cone near dff}) and then follow through the standard calculation, from which we easily see that 
\begin{equation} \label{expression of adjoint near dff}
A^{\ast_{\mathrm{3co}}}_{\psi_{\dff}} = \frac{1}{(2\pi)^{n_{\tindex}}} \int_{\mathbb{R}^{2n_{\tindex}}} e^{i ( z_{\tindex} - z_{\tindex}' ) \cdot \tscblz } \hat{A}^{\ast_{\mathrm{b}}}_{\psi_{\dff}}( z_{\tindex}', \tscblz )   d\zeta_{\tindex} |dz_{\tindex}'|,
\end{equation}
where $\hat{A}_{\psi_{\dff}}$ is the left operator-valued symbol of $A_{\psi_{\dff}}$. Then Lemma \ref{operator-valued symbols asymptotic definition}, part (1) shows that $A^{\ast_{\mathrm{3co}}}_{\psi_{\dff}}$ can be written as the left partial quantization of 
\begin{equation*}
\widehat{A_{\psi_{\dff}}^{\ast_{\mathrm{3co}}}} (z_{\tindex}, \tscblz ) \sim \sum_{\beta_{\tindex} 
\in \mathbb{N}_{0}^{n_{\tindex}} } \frac{1}{\beta_{\tindex}!} \partial_{\tscblz}^{\beta_{\tindex}}D_{z_{\tindex}}^{\beta_{\tindex}} \hat{A}_{\psi_{\dff}}^{\ast_{\mathrm{b}}} ( z_{\tindex}, \tscblz )
\end{equation*}
from which the classicality of $A_{\psi_{\dff}}^{\ast_{\mathrm{3co}}}$ is also clear. In particular, letting $x_{\tindex} \rightarrow 0$ results in the formula (\ref{indicial operator adjoint formula dff}). A similar calculation works for the case at $\cf$, proving formula (\ref{indicial operator adjoint formula cf}).
\end{proof}

We finally remark that the obvious analogues of the results in this subsection, including Propositions \ref{Composition proposition} and \ref{proposition symbol and indicial operators for adjoint}, can be proved in the context of (constant orders) conormal three-cone operators as well, with their proofs being largely the same. The precise statements will be omitted in order to not burden the length of this paper further,

\subsection{Operator wavefront set}
\label{Operator wavefront set subsection}

Having introduced the principal symbol and the indicial operators, we finally move onto the discussion of microlocalization. \par

Our discussions in this subsection and the next will mostly be focused on the setting of further-resolved, second microlocalized operators $\Psf^{\vom, \vor, \vol, \vob, \vov, \vos}(\Xd)$ (which also includes the constant orders cases by setting $\delta = 0$ as well).  \par

We will introduce (or rather recall) the symbolic wavefront and elliptic/characteristic sets as in the standard (e.g., scattering) settings. For the non-symbolic parts of the operators (i.e., the parts of the operators which can be written locally as partial quantizations, for which the indicial operators are relevant). We will need to define the wavefront and elliptic/characteristic sets differently.  \par

In this subsection, we will only define the operator wave front sets. Such wavefront sets are designed to capture microlocally (i.e., in phase spaces) where an operator $A$ is `trivial' in an appropriate sense. At the boundary faces of $\overline{^{\mathrm{d3sc,3co,res}}T^{\ast}}\Xd$ where the principal symbol map measures leading order decay, this notion of triviality can be made pointwise. Meanwhile, the indicial operators (which are non-local, i.e., they are global operators over each fiber of $\dff$ or $\cf$) are required to measure leading order decay at the other faces. Therefore, the corresponding notions of triviality must be non-local as well. 

\begin{definition}[Operator wavefront sets]
\label{definition of the operator wavefront set}
Suppose that $\vom, \vor, \vov, \vos \in \mathcal{C}^{\infty}( \psf \Xd )$, $\vol, \vob \in \mathcal{C}^{\infty}( \mathcal{C}_{\tindex} \times \overline{\mathbb{R}^{n_{\tindex}}} )$ and let $A \in \Psf^{\vom, \vor, \vol, \vov, \vob, \vos}(\Xd)$.  Then we will define
\begin{equation*}
\WFs(A)
\end{equation*}
as a subset of 
\begin{equation}
 \label{phase space of the symbols}
^{\mathrm{d3sc,3co,res}}S^{\ast} \Xd \cup  \overline{^{\mathrm{d3sc,3co,res}}T^{\ast}}_{\dmf} \Xd \cup \dtsccf  \cup \rf.
\end{equation}
A point $\beta$ belonging to (\ref{phase space of the symbols}) does not live in $\WFs(A)$ if there exists an open neighborhood $U$ of $\beta$ in (\ref{phase space of the symbols}), such that for all $\psi \in \mathcal{C}^{\infty}(\Xd)$ with $\supp \psi$ small, $\supp \psi \cap U \neq \emptyset$, if $\psi A \psi$ is written as some quantization of a symbol $b_{\psi}$ modulo an operator in $\Psf^{-\infty, -\infty, \vol, -\infty, \vob, -\infty}(\Xd)$ (as in the previous sections), then $b_{\psi}$ vanishes to infinite orders at $U$. \par

Moreover, we will define $\WFdff(A)$ and $\WFcf(A)$ as subsets $\mathcal{C}_{\tindex} \times \overline{\mathbb{R}^{n_{\tindex}}}$, such that:
\begin{enumerate}

\item A point $\beta_{\tindex} \in \mathcal{C}_{\tindex} \times \overline{\mathbb{R}^{n_{\tindex}}}$ does not live in $\WFdff(A)$ if there exists an open neighborhood $U_{\tindex}$ of $\beta_{\tindex}$ in $\mathcal{C}_{\tindex} \times \overline{\mathbb{R}^{n_{\tindex}}}$ such that if for a cut-off function $\psi_{\dff} \in \mathcal{C}^{\infty}( \Xd )$ at $\dff$, the operator-valued symbol $\hat{A}_{\psi_{\dff}}$ vanishes to infinite orders  at $U_{\tindex}$. 

\item A point $\beta_{\tindex} \in \mathcal{C}_{\tindex} \times \overline{\mathbb{R}^{n_{\tindex}}}$ does not live in $\WFcf(A)$ if there exists an open neighborhood $U_{\tindex}$ of $\beta_{\tindex}$ in $\mathcal{C}_{\tindex} \times \overline{\mathbb{R}^{n_{\tindex}}}$ such that if for a cut-off function $\psi_{\cf} \in \mathcal{C}^{\infty}( \Xd )$ at $\cf$, the operator-valued symbol $\hat{A}_{\psi_{\cf}}$ vanishes to infinite orders  at $U_{\tindex}$. 
\end{enumerate} \par 

The same statements hold if $\vom = m$, $\vor = r$, $\vol = l$, $\vov = \nu$, $\vob = b$, $\vos = s$ are all constants, in which case we also set $\delta = 0$.

If instead $A \in \Psi_{\mathrm{3coc}}^{m,r,l,b}(\Xd)$, then we can define $\mathrm{WF}_{\sigma}'(A)$ essentially verbatim, except that $\mathrm{WF}_{\sigma}'(A)$ must now be a subset of ${^{\mathrm{3co}}S^{\ast}}\Xd \cup \overline{^{\mathrm{3co}}T^{\ast}}_{\dmf}\Xd$.
\end{definition}

It is easy to see that the operator wavefront sets defined above satisfy the standard properties. In particular, assuming that $A_{j} \in \Psi_{\mathrm{d3sc,3co,res},\delta}^{\vom_{j} ,\vor_{j} ,\vol_{j} ,\vov_{j} ,\vob_{j} , \vos_{j}}(\Xd)$, $j=1,2$, and that the variable orders satisfy the assumptions of Proposition \ref{Composition proposition}. Then we have
\begin{equation*}
\begin{gathered}
\mathrm{WF}'_{\bullet}( A_{1} + A_{2} ) \subset \mathrm{WF}'_{\bullet}( A_1 ) \cup \mathrm{WF}'_{\bullet}( A_2 ), \quad \mathrm{WF}'_{\bullet}( A_1 A_2 ) \subset \mathrm{WF}'_{\bullet}(A_1) \cap \mathrm{WF}'_{\bullet}(A_2),
\end{gathered} 
\end{equation*}
where $\bullet = \sigma, \dff$ or $\cf$. These claims are standard if $\bullet = \sigma$. If instead $\bullet = \dff$ or $\bullet = \cf$, then they follow from writing $A_j$, $j=1,2$ as partial quantizations, and then applying the asymptotic summation formulae from Lemma \ref{composition subsection composition reduction lemma}. \par 

If $A \in \Psf^{\vom,\vor,\vol,\vov,\vob,\vos}(\Xd)$, then the point of considering the wavefront sets is that
\begin{equation*}
\begin{gathered}
\WFs(A) = \emptyset \Leftrightarrow A \in \Psi_{\mathrm{d3sc,3co,res},\delta}^{-\infty, -\infty, \mathsf{l}, - \infty, \mathsf{b}, -\infty}(\Xd), \\
\WFdff(A) = \emptyset \Leftrightarrow A \in \Psi_{\mathrm{d3sc,3co,res},\delta}^{\mathsf{m}, \mathsf{r}, -\infty, \mathsf{v}, \mathsf{b}, \mathsf{s}}(\Xd), \\
\WFcf(A) = \emptyset \Leftrightarrow A \in \Psi_{\mathrm{d3sc,3co,res},\delta}^{\mathsf{m}, \mathsf{r}, \mathsf{l} , \mathsf{v}, -\infty, \mathsf{s}}(\Xd).
\end{gathered}
\end{equation*} 
More generally, for the symbolic wavefront set $\mathrm{WF}'_{\sigma}(A)$, we also have
\begin{equation*}
\begin{gathered}
\mathrm{WF}_{\sigma}'(A) \cap {^{\mathrm{d3sc,3co,res}}S^{\ast}}\Xd = \emptyset  \Leftrightarrow A \in \Psi_{\mathrm{d3sc,3co,res},\delta}^{-\infty, \vor, \mathsf{l}, \vov, \mathsf{b}, \vos }(\Xd), \\
\mathrm{WF}_{\sigma}'(A) \cap \psf_{\dmf}\Xd = \emptyset  \Leftrightarrow A \in \Psi_{\mathrm{d3sc,3co,res},\delta}^{\vom, -\infty , \mathsf{l}, \vov, \mathsf{b}, \vos }(\Xd), \\
\mathrm{WF}_{\sigma}'(A) \cap \dtsccf = \emptyset  \Leftrightarrow A \in \Psi_{\mathrm{d3sc,3co,res},\delta}^{\vom , \vor, \mathsf{l}, -\infty , \mathsf{b}, \vos }(\Xd), \\
\mathrm{WF}_{\sigma}'(A) \cap \rf = \emptyset \Leftrightarrow A \in \Psf^{\vom, \vor, \vol, \vov, \vos, -\infty}(\Xd). 
\end{gathered}
\end{equation*} \par

We also have a uniform version of the symbolic wavefront set, which will be useful when we consider microlocal propagation estimates in \S\S \ref{principal type propagation section}--\ref{global radial point estimate section}. It is also required to state Lemma \ref{uniform elliptic regularity estimate lemma} below. Here, we shall follow the exposition of \cite[Chapter 9]{PeterNotes}.
\begin{definition}[Uniform operator wavefront set]
\label{definition uniform operator wavefront set}
Let $\{ A_{s} \} \in L^{\infty}( (0,1]_s ; \Psi_{\mathrm{d3sc,3co,res}}^{\vom, \vor, \vol, \vov, \vob, \vos} (\Xd) )$. Then we will define
\begin{equation*}
\WFsL( \{ A_s \}_s )
\end{equation*}
as a subset of (\ref{phase space of the symbols}). A point $\beta$ belonging to (\ref{phase space of the symbols}) does not live in $\WFsL( \{ A_s \} )$ if there some $B \in \Psi_{\mathrm{d3sc,3co,res},\delta}^{0,0,0,0,0,0}(\Xd)$ which is elliptic at $\beta$, such that
\begin{equation*}
\{ BA_{s} \} \subset L^{\infty}( (0,1]_s ; \Psi_{\mathrm{d3sc,3co,res}}^{-\infty, -\infty, \vol, -\infty, \vos, -\infty} ).
\end{equation*}
\end{definition}

\subsection{Elliptic and characteristic sets} 
\label{Microlocal elliptic parametrix subsection}
Next we introduce the elliptic/characteristic sets.

\begin{definition}[Elliptic/characteristic sets] 
\label{symbolic version of elliptic and characterstic sets}
Suppose that $\vom, \vor, \vov, \vos \in \mathcal{C}^{\infty}( \psf \Xd )$, $\vol, \vob \in \mathcal{C}^{\infty}( \mathcal{C}_{\tindex} \times \overline{\mathbb{R}^{n_{\tindex}}} )$ and let $A \in \Psf^{\vom, \vor, \vol, \vov, \vob, \vos}(\Xd)$.  Then we will define 
\begin{equation*}
\mathrm{Ell}_{\sigma, \delta}^{ \mathsf{m},\mathsf{r}, \mathsf{l} ,  \mathsf{v}, \mathsf{b} , \mathsf{s}}(A)
\end{equation*}
as a subset of (\ref{phase space of the symbols}). A point $\beta$ belonging to (\ref{phase space of the symbols}) lives in $\mathrm{Ell}_{\sigma, \delta}^{\vom, \vor, \vol, \vov, \vob, \vos}(A)$ if there exists a neighborhood $U \subset {\overline{ ^{\mathrm{d3sc,3co,res}}T^{\ast}} \Xd} $ of $\beta$ and $b \in S^{ - \mathsf{m}, - \mathsf{r}, -\mathsf{l}, - \mathsf{v}, - \mathsf{b}, - \mathsf{s} }_{\delta}( \overline{ ^{\mathrm{d3sc,3co,res}}T^{\ast}} \Xd )$ such that if $a \in S_{\delta}^{\vom, \vor, \vol, \vov, \vob, \vos}( \psf \Xd )$ is a representative of $\sigma(A)$, then we have
\begin{equation*}
ab - 1 \in S^{-1 + 2 \delta, -1 + 2 \delta, 0, - 1 + 2\delta, 0, - 1 + 2 \delta}_{\delta}( \overline{ ^{\mathrm{d3sc,3co,res}}T^{\ast}} \Xd ) \ \text{on} \ U
\end{equation*}
in the sense that $ab - 1$ satisfies the $ S^{-1 + 2 \delta, -1 + 2 \delta, 0, - 1 + 2\delta, 0, - 1 + 2 \delta}_{\delta}( \psf \Xd )$ estimates on $U \cap T^{\ast} \mathbb{R}^{n}$. We define $\mathrm{Char}_{\sigma, \delta}^{\mathsf{m},\mathsf{r}, \mathsf{l} ,  \mathsf{v}, \mathsf{b} , \mathsf{s}}(A)$ to be the complement of $\mathrm{Ell}_{\sigma, \delta}^{\mathsf{m},\mathsf{r}, \mathsf{l} ,  \mathsf{v}, \mathsf{b} , \mathsf{s}} (A)$ in (\ref{phase space of the symbols}). \par

Suppose that $K > 0$. Then: 
\begin{enumerate}
\item If $A$ is classical at $\dff$ modulo $\Psf^{\vom, \vor, \vol - K, \vov, \vob, \vos}(\Xd)$, then we will define $\mathrm{Ell}^{\vol}_{\dff}(A)$ by all $\beta_{\tindex} \in \mathcal{C}_{\tindex} \times \overline{\mathbb{R}^{n_{\tindex}}}$ for which there exists a  neighborhood $U_{\tindex} \subset \mathcal{C}_{\tindex} \times \overline{\mathbb{R}^{n_{\tindex}}}$ of $\beta_{\tindex}$ such that $\hat{N}_{\dff}(A)$ is invertible on $U_{\tindex}$ with inverse in 
\begin{equation}
\label{elliptic set at dff inverse membership}
\Psi_{\mathrm{sc,b,lp,res},\delta}^{-\vom, -\vov + 2 \vol, -\vob + 2 \vol, - \vos + \vol}( X^{\tindex} ; \mathcal{C}_{\tindex} \times \mathbb{R}^{n_{\tindex}} ),
\end{equation}
i.e., $\hat{N}_{\dff}(A)^{-1}$ can be identified as the restriction to $U_{\tindex}$ of some element in (\ref{elliptic set at dff inverse membership}). We define $\mathrm{Char}^{\vol}_{\dff}(A)$ to be the complement of $\mathrm{Ell}^{\vol}_{\dff}(A)$ in $\mathcal{C}_{\tindex} \times \overline{\mathbb{R}^{n_{\tindex}}}$.
\item If $A$ is classical at $\cf$ modulo $\Psf^{\vom, \vor, \vol, \vov, \vob - K, \vos}(\Xd)$, then we will define $\mathrm{Ell}^{\vob}_{\cf}(A)$ by all $\beta_{\tindex} \in \mathcal{C}_{\tindex} \times \overline{\mathbb{R}^{n_{\tindex}}}$ for which there exists a  neighborhood $U_{\tindex} \subset \mathcal{C}_{\tindex} \times \overline{\mathbb{R}^{n_{\tindex}}}$ of $\beta_{\tindex}$ such that $\hat{N}_{\cf}(A)$ is invertible on $U_{\tindex}$ with inverse in 
\begin{equation}
\label{elliptic set at cf inverse membership}
\Psi_{\mathrm{coc,lp,res}, \delta}^{ - \mathsf{m}, - \mathsf{r} + \frac{b}{2}, - \mathsf{l} + \frac{b}{2}, - \mathsf{v} + \vob }( [ \hat{X}^{\tindex} ; \{ 0 \} ] ; \mathcal{C}_{\tindex} \times \mathbb{R}^{n_{\tindex}} ), 
\end{equation}
i.e., $\hat{N}_{\cf}(A)^{-1}$ can be identified as the restriction to $U_{\tindex}$ of some element in (\ref{elliptic set at cf inverse membership}). We define $\mathrm{Char}^{\vob}_{\cf}(A)$ to be the complement of $\mathrm{Ell}^{\vob}_{\cf}(A)$ in $\mathcal{C}_{\tindex} \times \overline{\mathbb{R}^{n_{\tindex}}}$.
\end{enumerate}

The same statements hold if $\vom = m$, $\vor = r$, $\vol = l$, $\vov = \nu$, $\vob = b$, $\vos = s$ are all constants, in which case we also set $\delta = 0$.

If we instead assume that $A \in \Psi_{\mathrm{3occ}}^{m,r,l,b}(\Xd)$, then we will define $\mathrm{Ell}_{\sigma}^{m,r,l,b}(A)$ as a subset of ${^{\mathrm{3co}}S^{\ast}}\Xd \cup \overline{^{\mathrm{3co}}T^{\ast}}\Xd$. Moreover, a point $\beta$ belongs to $\mathrm{Ell}_{\sigma}^{m,r,l,b}(A)$ if there exists a neighborhood $U \subset \overline{^{\mathrm{3co}}T^{\ast}}\Xd$ of $\beta$ and $b \in S^{-m,-r,-l,-b}( \overline{^{\mathrm{3co}}T^{\ast}}\Xd )$ so that if $a \in S^{m,r,l,b}( \overline{^{\mathrm{3co}}T^{\ast}}\Xd )$ is a representative of $\sigma(A)$, then we have 
\begin{equation*}
ab - 1 \in S^{-1, -1, 0,0}( \overline{ ^{\mathrm{3co}}T^{\ast} }\Xd ) \ \text{on} \ U
\end{equation*}
in the sense that $ab - 1$ satisfies the $S^{-1,-1,0,0}( \overline{^{\mathrm{3co}}T^{\ast}} \Xd )$ estimates on $U \cap T^{\ast} \mathbb{R}^{n}$. Finally, we define $\mathrm{Char}_{\sigma}^{m,r,l,b}(A)$ to be the complement of $\mathrm{Ell}_{\sigma}^{m,r,l,b}(A)$ in ${^{\mathrm{3co}}S^{\ast}}\Xd \cup \overline{^{\mathrm{3co}}T^{\ast}}\Xd$.
\end{definition}

For brevity, we will henceforth omit writing the indices, subscripts as well as dependencies on the small parameter $\delta > 0 $ in the notations for elliptic and characteristic set above. More precisely, if $A \in \Psf^{\vom, \vor, \vol, \vov, \vob, \vos}(\Xd)$, then we will simplify notations to write $\mathrm{Ell}_{\sigma}(A) \coloneq \mathrm{Ell}_{\sigma, \delta}^{\vom, \vor, \vol, \vov, \vob, \vos}(A)$ (and also $\mathrm{Ell}_{\sigma}(A) \coloneq \mathrm{Ell}_{\sigma}^{m,r,l,b}(A)$ if $A \in \Psi_{\mathrm{3coc}}^{m,r,l,b}(\Xd)$), as well as
\begin{equation*}
\begin{gathered}
\text{$\mathrm{Ell}_{\dff}(A) \coloneq \mathrm{Ell}_{\dff}^{\vol}(A)$,  $\mathrm{Char}_{\dff}(A) \coloneq \mathrm{Char}_{\dff}^{\vol}(A)$, and} \\
\text{$\mathrm{Ell}_{\cf}(A) \coloneq \mathrm{Ell}_{\cf}^{\vob}(A)$,  $\mathrm{Char}_{\cf}(A) \coloneq \mathrm{Char}_{\cf}^{\vob}(A)$.}
\end{gathered}
\end{equation*}

\begin{remark}
Since invertibility is an open condition, an equivalent condition for $\beta_{\tindex} \in \mathcal{C}_{\tindex} \times \mathbb{R}^{n_{\tindex}}$, $\beta_{\tindex} \in \mathrm{Ell}_{\dff}(A)$ where $A \in \Psf^{\vom, \vor, \vol, \vov, \vob,\vos}(\Xd)$ is that $\hat{N}_{\dff}(A)$ be invertible at $\beta_{\tindex}$ with inverse in $\Psi_{\mathrm{sc,b}, \delta}^{ -\vos + \vol, -\vov + 2 \vol, -\vob + 2 \vol }( X^{\tindex} )$. However, if we instead assume that $\beta_{\tindex} \in \mathcal{C}_{\tindex} \times \partial \overline{\mathbb{R}^{n_{\tindex}}}$, then the formulation provided in Definition \ref{symbolic version of elliptic and characterstic sets} becomes necessary. The same statement is also true if instead $\beta_{\tindex} \in \mathcal{C}_{\tindex} \times \mathbb{R}^{n_{\tindex}}$ and $\beta_{\tindex} \in \mathrm{Ell}_{\cf}(A)$, except that we now require $\hat{N}_{\cf}(A)$ to be invertible at $\beta_{\tindex}$ with inverse in $\Psi_{\mathrm{coc},\delta}^{-\vov+\vob, -\vor + \vob/2, -\vol + \vob/2}( [ \hat{X}^{\tindex} ; \{ 0 \} ] )$.
\end{remark}

\begin{remark}
In fact, conditions (1) and (2) given in Definition \ref{symbolic version of elliptic and characterstic sets} are in general too strong. Consider for definiteness condition (1) only. Then it is possible for an operator in $\Psi_{\mathrm{sc,b}, \delta}^{ \vos - \vol, \vov - 2 \vol, \vob - 2\vol }( X^{\tindex} )$ to be invertible without an inverse in $\Psi_{\mathrm{sc,b},\delta}^{-\vos + \vol, - \vov + 2 \vol, -\vob + 2 \vol}(X^{\tindex})$. In order to account for these operators, we need to instead consider the `large calculus', which we have elected not to construct in this paper. Nevertheless, our construction does include a large class of non-trivial operators. For example, if $\hat{N}_{\dff}(A)$ is a scattering operator, then its inverse, if it exists, must lie in the scattering calculus as well. 
\end{remark}

The elliptic sets also satisfy the standard properties. Most notably, we always have
\begin{equation} \label{standard properties of elliptic set under composition}
\mathrm{Ell}_{\bullet} ( AB ) = \mathrm{Ell}_{\bullet}(A) \cap \mathrm{Ell}_{\bullet}(B)
\end{equation}
for $\bullet = \sigma, \dff$ or $\cf$. In every case, (\ref{standard properties of elliptic set under composition}) follows easily from the multiplicative properties of the principal symbol map and indicial operators.
\par

We now introduce suitable notions of microlocal elliptic parametrices.
\begin{proposition}[Microlocal elliptic parametrices] 
\label{Proposition microlocal ellipticity parametrix}
Let $\vom, \vor, \vov, \vos \in \mathcal{C}^{\infty}( \psf \Xd )$, $\vol, \vob \in \mathcal{C}^{\infty}( \mathcal{C}_{\tindex} \times \overline{\mathbb{R}^{n_{\tindex}}} )$ and $A \in \Psf^{\vom, \vor, \vol, \vov, \vob, \vos}(\Xd)$. Then:
\begin{enumerate}
\item If $ K \subset \Ellss(A)$ is compact, then there exists $B \in \Psi^{ - \mathsf{m}, - \mathsf{r},  - \mathsf{l}, -\mathsf{v}, -\mathsf{b} , - \mathsf{s} }_{\mathrm{d3sc,3co,res,\delta}}(\Xd)$ such that
\begin{equation} 
\label{standard elliptic parametrix}
\WFs(AB- I) \cap K = \emptyset, \quad \WFs(BA - I) \cap K = \emptyset.
\end{equation} 
\item Suppose that $A$ is classical at $\dff$ modulo $\Psf^{\vom, \vor, \vol - K', \vov, \vob, \vos}(\Xd)$, $K' > 0$. Then if $K_{\tindex} \subset \Elldff(A)$ is compact, there exists $B \in \Psi^{ - \mathsf{m}, - \mathsf{r},  - \mathsf{l}, - \mathsf{v}, -\mathsf{b} , - \mathsf{s} }_{\mathrm{d3sc,3co,res,\delta}}(\Xd)$ such that
\begin{equation*}
\WFdff(AB - I) \cap K_{\tindex} = \emptyset, \quad \WFdff(BA - I) \cap K_{\tindex} = \emptyset.
\end{equation*}
\item Suppose that $A$ is classical at $\cf$ modulo $\Psf^{\vom, \vor, \vol, \vov, \vob - K', \vos}(\Xd)$, $K' > 0$. Then if $K_{\tindex} \subset \Ellcf(A)$ is compact, there exists $B \in \Psi^{ - \mathsf{m}, - \mathsf{r},  - \mathsf{l}, - \mathsf{v}, -\mathsf{b} , - \mathsf{s} }_{\mathrm{d3sc,3co,res,\delta}}(\Xd)$ such that
\begin{equation*}
\WFcf(AB - I) \cap K_{\tindex} = \emptyset, \quad  \WFcf(BA - I) \cap K_{\tindex} = \emptyset.
\end{equation*}
\end{enumerate}
\end{proposition}
\begin{proof}
Part (1) is truly standard. Thus we will focus on the proof of part (2), though even here the proof is essentially standard still. \par

 Indeed, by definition there exists a neighborhood $U_{\tindex} \subset \mathcal{C}_{\tindex} \times \overline{\mathbb{R}^{n_{\tindex}}}$ of $K_{\tindex}$ such that $\hat{N}_{\dff}(A)$ is invertible on $U_{\tindex}$. Moreover, let $q_{\tindex} \in \mathcal{C}^{\infty}( \mathcal{C}_{\tindex} \times \overline{\mathbb{R}^{n_{\tindex}}} )$ be such that $q_{\tindex} = 1$ on $K_{\tindex}$ and is supported in $U_{\tindex}$. Then we have 
\begin{equation*}
q_{\tindex} \hat{N}_{\dff}(A)^{-1} \in \Psi_{\mathrm{sc,b,lp, res},\delta}^{ - \vom, - \vov + 2 \vol, - \vob + 2 \vol, -\vos + \vol }( X^{\tindex} ; \mathcal{C}_{\tindex} \times \overline{\mathbb{R}^{n_{\tindex}}} ).
\end{equation*} \par

Now, by Lemma \ref{lemma constructing operator from prescribed indicial operator}, we can find some $B_0 \in \Psi_{\mathrm{d3sc,3co,res},\delta}^{-\vom, -\vor, - \vol, - \vov, -\vob, -\vos}(\Xd)$ such that 
\begin{equation*}
\hat{N}_{\dff}( B_0 ) = q_{\tindex}  \hat{N}_{\dff}(A)^{-1}, \quad \text{$B_0$ is classical at $\dff$ modulo $\Psf^{-\vom, -\vor, -\infty, -\vov, -\vob, -\vos}(\Xd)$}. 
\end{equation*}
Let $q_{\tindex,0} \in \mathcal{C}^{\infty}( \mathcal{C}_{\tindex} \times \overline{\mathbb{R}^{n_{\tindex}}} )$ be identically $1$ on $K_{\tindex}$, $q_{\tindex,0} ( 1 - q_{\tindex} ) = 0$. Let $Q_{0} \in \Psi_{\mathrm{d3sc,3co,res}}^{0,0,0,0,0,0}(\Xd)$ be chosen such that $\hat{N}_{\dff}(Q_0) = q_{\tindex,0}$ and $Q_0$ is classical at $\dff$ modulo $\Psf^{0,0,-\infty,0,0,0}(\Xd)$. Then we will define
\begin{equation*}
E_{0,L} \coloneq Q_0 ( B_0 A - I ) \in \Psi_{\mathrm{d3sc,3co,res},\delta}^{0,0,0,0,0,0}(\Xd), \quad E_{0,R} \coloneq ( A B_0 - I ) Q_0 \in \Psi_{\mathrm{d3sc,3co,res},\delta}^{0,0,0,0,0,0}(\Xd)
\end{equation*}
with vanishing indicial operators at $\dff$. Thus, we in fact have
\begin{equation*}
E_{0,L} \in \Psi_{\mathrm{d3sc,3co,res},\delta}^{0,0, - 1+ 2\delta, 0,0,0}(\Xd), \quad E_{0,R} \in \Psi_{\mathrm{d3sc,3co,res},\delta}^{0,0,-1+2\delta, 0,0,0}(\Xd).
\end{equation*} \par

The rest of the argument follows from the standard asymptotic summation argument (cf. \cite[Proposition 5.16]{AndrasBook}), with the exception that the asymptotic summations in question must now be understood in the sense Lemma \ref{composition subsection symbol reduction formulae lemma}, part (1). The analogous calculation also works at $\cf$. This concludes the proof of the proposition.
\end{proof}
We also have the following special case which is not covered by Proposition \ref{Proposition microlocal ellipticity parametrix}. It will be useful when we discuss Sobolev spaces in \S \ref{subsection Sobolev spaces under second microlocalization}.

\begin{lemma}
\label{special lemma elliptic parametrix}
Suppose that $\vol, \vob \in \mathcal{C}^{\infty}( \mathcal{C}_{\tindex} \times \overline{\mathbb{R}^{n_{\tindex}}} )$. Let $x_{\dff}, x_{\cf} \in \mathcal{C}^{\infty}(\Xd)$ be global defining functions for $\dff$, $\cf$ respectively and $\Lambda \in \Psf^{0,0,\vol,\vol,\vob,\vob}(\Xd)$ be some quantization (depending on a partition) of $x_{\dff}^{-\vol} x_{\cf}^{-\vob}$. Then there exists $\Gamma \in \Psf^{0,0,-\vol, -\vol, -\vob, -\vob}(\Xd)$ such that  
\begin{equation*}
\mathrm{WF}_{\bullet}'(\Lambda \Gamma- I) = \emptyset, \quad \mathrm{WF}_{\bullet}'( \Gamma  \Lambda - I)  = \emptyset, \quad \bullet = \sigma, \dff, \cf.
\end{equation*}
\end{lemma}
\begin{proof}
Let $\Gamma_0 \in \Psf^{0,0,-\vol,-\vol, -\vob, -\vob}(\Xd)$ be defined by some quantization of $x_{\dff}^{\vol} x_{\cf}^{\vob}$. Then it is easy to see that
\begin{equation*}
E_{R,0} \coloneq \Lambda \Gamma _0- I, E_{L,0} \coloneq \Gamma_0 \Lambda - I \in \Psi_{\mathrm{d3sc,3co,res},\delta}^{-1 + 2 \delta, - 1 + 2\delta, - 1 + 2 \delta, -1 + 2\delta, -1 + 2\delta, - 1 + 2\delta}(\Xd).
\end{equation*}
Following the standard argument, the above membership can be improved to the existence of some $\Gamma_{R}, \Gamma_{L} \in \Psf^{0,0,-\vol, -\vol, -\vob, -\vob}(\Xd)$ such that
\begin{equation*}
\begin{gathered}
\Lambda \Gamma_{R}  - I, \Gamma_{L} \Lambda - I \in \Psi_{\mathrm{d3sc,3co,res},\delta}^{-\infty, -\infty, -\infty, -\infty, -\infty, -\infty}(\Xd), \quad \Gamma_{R} - \Gamma_{L} \in  \Psi_{\mathrm{d3sc,3co,res},\delta}^{-\infty, -\infty, -\infty, -\infty, -\infty, -\infty}(\Xd).
\end{gathered}
\end{equation*}
The construction of $\Gamma_{R}$, $\Gamma_{L}$ relies on an asymptotic summation procedure. To be precise, 
\begin{equation*}
\Gamma_{R} \coloneq \Gamma_{0}( I - E_{R}' ), \quad E_{R}' \sim \sum_{j=1}^{\infty} (-1)^{j} E_{R,0}^{j},
\end{equation*}
which is now computed modulo $\Psf^{-\infty,-\infty,-\infty,-\infty,-\infty,-\infty}(\Xd)$ in the sense that
\begin{equation*}
E_{R}' - \sum_{j=1}^{J-1} (-1)^{j} E_{R,0}^{j} \in \Psf^{- J ( 1  - 2\delta ), - J ( 1  - 2\delta ), - J ( 1  - 2\delta ), - J ( 1  - 2\delta ), - J ( 1  - 2\delta ), - J ( 1  - 2\delta )}(\Xd)
\end{equation*}
for every integer $J > 1$. \par

The existence of $E_R'$ is not covered by either parts of Lemma \ref{operator-valued symbols asymptotic definition} (the detailed proofs of which were also omitted to keep this paper under reasonable length), but is nevertheless standard, and can be proved using a patching argument and the local Frech\'et structures. \par 

A similar argument allows us to construct $\Gamma_{L}$ as well. 
\end{proof}

\subsection{Sobolev spaces}
\label{subsection Sobolev spaces under second microlocalization}
We now define Sobolev spaces in the three-cone as well the further-resolved, second microlocalized setting. In analogy to how Sobolev spaces were defined under second microlocalization in the two-body setting (see \S \ref{b-calculus subsection}), Sobolev spaces in the second microlocalized, three-body setting can be approached from either the three-cone perspective or the three-body perspective, where the main difference between the two perspectives lies in the choice of a scale of $L^2$ spaces. \par

Indeed, recall that from the three-cone perspective, we have already defined $L^{2}_{\mathrm{3co}}(\Xd)$, which are the $L^2$ functions with respect to a fixed strictly positive density $\nu_{\mathrm{3co}}$. However, from the three-body perspective, there is another, more natural notion of $L^2$ functions, namely
\begin{equation*}
L^{2}_{\mathrm{sc}} ( \overline{\mathbb{R}^{n}} ) = L^{2}( \overline{ \mathbb{R}^{n} } , \nu_{\mathrm{sc}}   ),
\end{equation*}
where $\nu_{\mathrm{sc}}$ is a canonical scattering density on $\overline{\mathbb{R}^{n}}$, which will indeed just be taken as the Euclidean density in this case. Thus we will henceforth omit the subscript and write
\begin{equation*}
L^{2} = L^{2}_{\mathrm{sc}}( \overline{\mathbb{R}^{n}} ).
\end{equation*} \par 

We now show that:

\begin{proposition}
\label{proposition boundedness of zero order operators}
If $A \in \Psi^{0,0,0,0,0,0}_{\mathrm{d3sc,3co,res}, \delta}(\Xd)$, then $A$ is a bounded $L^{2}_{\mathrm{3co}}(\Xd) \rightarrow L^{2}_{\mathrm{3co}}(\Xd)$. 
\end{proposition}
\begin{proof}
Let $\{ \tilde{\psi}_{j} \}_{j = 1}^{J}$ be a family of $\mathcal{C}^{\infty}(\Xd)$ such that $\sum_{j=1}^{J} \tilde{\psi}_j = 1$. For every $j = 1, ..., J$, let also $\psi_{j} \in \mathcal{C}^{\infty}(\Xd)$ be chosen with small support such that $\psi_{j} = 1$ on the support of $\psi_{j}$. Then by the usual observation, we can write $A = \sum_{j=1}^{J} B_{\psi_{j}} \tilde{\psi}_{j} + K$, where each $B_{\psi_j}$ is given by a suitable quantization and is supported on $\supp \psi_{j}$, while $K \in \Psi^{-\infty, -\infty, 0, -\infty, 0, -\infty}_{\mathrm{d3sc, 3co,res}, \delta}(\Xd)$. \par 

Now by Schur's lemma, it is easy to see that $K : L^{2}_{\mathrm{3co}} \rightarrow L^{2}_{\mathrm{3co}}$ is bounded. Meanwhile, on the support of each $\psi_{j}$, $B_{\psi_{j}}$ is written as a standard quantization in suitable (spatial) coordinates, concretely one of $(z_{\tindex}, z^{\tindex})$, $( z_{\tindex}, t^{\tindex}, y^{\tindex} )$, $(z_{\tindex}, \hat{t}_{\tindex}, y^{\tindex} )$, $( z_{\tindex}, \hat{z}^{\tindex} )$. \par

 In the same order, $\nu_{\mathrm{3co}}$ can be written in these coordinates as the multiple by some smooth, strictly positive $\mathcal{C}^{\infty}(\Xd)$ function of one of the densities $|dz_{\tindex} dz^{\tindex}|$, $|dz_{\tindex} d t^{\tindex} d y^{\tindex}|$, $|dz_{\tindex}  d \hat{t}_{\tindex} dy^{\tindex} |$ or $|dz_{\tindex} d\hat{z}^{\tindex}|$. It follows that the boundedness of $B_{\psi_{j}} \tilde{\psi_j}$ on $L^{2}_{\mathrm{3co}}(\Xd)$ can be inferred from H\"ormander's square root trick.
\end{proof}

It is striaghtforward to check that $\nu_{\mathrm{3co}}$ differs from $\nu_{\mathrm{sc}}$ by the multiplication of a strictly positive smooth function on $\Xd$. The above lemma thherefore implies the following corollary.

\begin{corollary}
\label{boundedness of order zero second microlocal operators on L^2}
If $A \in \Psi^{0,0,0,0,0,0}_{\mathrm{d3sc,3co,res}, \delta}(\Xd)$, then $A$ is a bounded $L^{2} \rightarrow L^{2}$. 
\end{corollary}

\begin{remark}
Since $\Psi_{\mathrm{3co}}^{0,0,0,0}(\Xd)$, $\Psi_{\mathrm{d3sc,3co,res}}^{0,0,0,0,0,0}(\Xd)$ are both contained in $\Psf^{0,0,0,0,0}(\Xd)$, the results of Proposition \ref{proposition boundedness of zero order operators} and Corollary \ref{boundedness of order zero second microlocal operators on L^2} are valid for these classes of operators as well.
\end{remark}

The Sobolev spaces we consider here will be defined with respect to $L^{2}$. This reflects the fact that we are really studying the \emph{scattering problem} under a second microlocalization, and the three-cone structure is but a convenient tool for us to understand the mechanism of this second microlocalziation.

Let first $m \geq 0$. Then we will define
\begin{equation} \label{Sobolev space with respect to positive m}
H^{m, 0 , 0, 0 }_{\mathrm{3co}}( \Xd ) \coloneq \{ u \in L^{2}  : \Lambda u \in L^{2} \},
\end{equation}
where $\Lambda \in \Psi^{m,0,0, 0}_{\mathrm{3coc}}(\Xd)$ is a fixed operator which is elliptic in the regularity sense (meaning that $\mathrm{Ell}_{\sigma}(\Lambda)$ at least contains ${^{\mathrm{3co}}S^{\ast}}\Xd$). Here (\ref{Sobolev space with respect to positive m}) is endowed with the norm
\begin{equation*}
\| u \|_{H^{m,0,0, 0}_{\mathrm{3co}}}^2 :  = \| u \|_{L^{2}}^2 + \| \Lambda u \|_{L^{2}}^2.
\end{equation*}
In the cases of $m \leq 0$, we define the Sobolev spaces by duality in $L^2$, i.e.,
\begin{equation*}
H^{m,0,0,0}_{\mathrm{3co}}( \Xd ) \coloneq \big( H^{-m,0,0,0}_{\mathrm{3co}} ( \Xd ) \big)^{\ast}.
\end{equation*} \par

More generally, for $r,l,b \in \mathbb{R}$, we will define the weighted spaces
\begin{equation} \label{definition of general Sobolev space 3co}
H^{m,r,l,b}_{\mathrm{3co}}( \Xd ) \coloneq x_{\dmf}^{r} x_{\dff}^{l} x_{\cf}^{b} H^{m,0,0,0}_{\mathrm{3co}}(\Xd),
\end{equation}
where $x_{\dmf}, x_{\dff}, x_{\cf} \in \mathcal{C}^{\infty}(\Xd)$ are defining function for $\dmf$, $\dff$ and $\cf$ respectively. The natural norm on (\ref{definition of general Sobolev space 3co}) is then given by
\begin{equation*}
\| u \|_{ H^{m,r,l,b}_{\mathrm{3co}} } \coloneq \| x_{\dmf}^{-r} x_{\dff}^{-l} x_{\cf}^{-b} u \|_{H^{m,0,0,0}_{\mathrm{3co}}}.
\end{equation*} \par

An obvious (and indeed standard) observation is that if
\begin{equation*}
m' \leq m, \ r' \leq r, \ l' \leq l, \ b' \leq b,
\end{equation*}
then we have
\begin{equation}
\label{three-cone embedding}
H_{\mathrm{3co}}^{m, r, l, b}(X) \subset H_{\mathrm{3co}}^{m',r',l',b'}(\Xd)
\end{equation}
with continuous inclusions.

Moreover, it is a standard consequence of Corollary \ref{boundedness of order zero second microlocal operators on L^2} and an elliptic regularity argument that we have $H^{0,0,0,0}_{\mathrm{3co}}( \Xd ) = L^{2}_{\mathrm{3co}}(\Xd)$. In fact, the choice of the operator $\Lambda$ in (\ref{Sobolev space with respect to positive m}) is not crucial so long as $\Lambda$ is elliptic in the regularity sense. In other words, the above definition would also work for any other $\Lambda' \in \Psi^{m,0,0,0}_{\mathrm{3coc}}(\Xd)$ that is elliptic in the regularity sense. The norms defined with respect to $\Lambda$ and $\Lambda'$ are then equivalent. It should be obvious that (\ref{definition of general Sobolev space 3co}) is independent of the choices of $x_{\dmf}$, $x_{\dff}$, $x_{\cf}$ as well. Lastly, $H_{\mathrm{3co}}^{m,r,l,b}(\Xd)$ is complete.
\par

The freedom of choosing $A$ also allows us characterize $H^{m,0,0,0}_{\mathrm{3co}}(\Xd)$ in terms of local charts: Let $\{ {\psi}_{j} \}_{j=1}^{J}$ be a smooth partition of unity on $X$. Then we have
\begin{equation}
\label{local characterization for 3co Sobolev spaces}
H^{m,0,0,0}_{\mathrm{3co}} = \{ u \in \mathcal{S}' : \psi_j u \in H^{m}( \mathbb{R}^{n}_{w_j} )  \}, \quad \| u \|_{H^{m,0,0,0}_{\mathrm{3co}}}^2 \simeq \sum_{j=1}^{J} \| {\psi}_{j} u \|_{H^{m}(\mathbb{R}^{n}_{w_j})}^2,
\end{equation}
where for each $j =1 , ..., J$, $H^{m}(\mathbb{R}^{n}_{w_j})$ is the standard Sobolev space on $\mathbb{R}^{n}_{w_j}$ with
\begin{equation*}
\| {\psi}_j u \|_{H^{m}(\mathbb{R}^{n}_{w_j})} \coloneq \| \langle D_{w_j} \rangle^{m} {\psi}_{j} u \|_{L^2(\mathbb{R}^{n}_{w_j})}.
\end{equation*}
Here the coordinates $w_j$ could take any one of the forms $z$, $(z_{\tindex}, t^{\tindex}, y^{\tindex})$, $(z_{\tindex}, \hat{t}_{\ff}, y^{\tindex})$, $(z_{\tindex}, \hat{z}^{\tindex})$, depending on where ${\psi}_j$ is supported.   \par

We next consider the effects of second microlocalization and the resolution introduced in \S \ref{a further resolution at fiber infinity}.  Assume first that $m,r,\nu,s \in \mathbb{R}$, and let $M \in \mathbb{R}$ be such that 
\begin{equation*}
M \leq \min \{ m, \nu, s \}.
\end{equation*}
Then we will define
\begin{equation*} 
\begin{gathered}
H^{m, 0, 0 , \nu, 0 ,s}_{\mathrm{d3sc,3co,res}}( \Xd ) \coloneq \{ u \in H^{ M , 0 , 0 , 0 }_{\mathrm{3co}}( \Xd ) : \Lambda u \in L^{2} \}, \\
\| u \|_{H^{m,0,0 ,\nu,0,s}_{\mathrm{d3sc,3co,res}}}^2 \coloneq \| u \|_{ H^{M,0,0,0}_{\mathrm{3co}} }^2 + \| \Lambda u \|_{L^2}^2
\end{gathered}
\end{equation*}
for some fixed $\Lambda \in \Psi^{m,0,0,\nu,0,s}_{\mathrm{d3sc,3co,res}}( \Xd )$ which is elliptic at $\dtsccf$, $\rf$ and the fiber infinity (again, in the sense that $\mathrm{Ell}_{\sigma}(\Lambda)$ at least contain these faces). \par

For general $l,b \in \mathbb{R}$, we will define the weighted spaces
\begin{equation} 
\label{constant order d3sc b res Sobolev spaces full}
H^{m,r,l,\nu,b,s}_{\mathrm{d3sc,3co,res}} ( \Xd ) = x_{\dmf}^{r} x_{\dff}^{l} x_{\cf}^{b} H_{\mathrm{d3sc,3co,res}}^{m, 0, 0, \nu - b, 0, s - l}(\Xd)
\end{equation}
with the norm
\begin{equation*}
\| u \|_{H_{\mathrm{d3sc,3co,res}}^{m,r,l,\nu,b,s}} \coloneq \| x
_{\dmf}^{-r} x_{\dff}^{-l} x_{\cf}^{-b} u \|_{H_{\mathrm{d3sc,3co,res}}^{m,0,0,\nu-b,0,s-l}}.
\end{equation*}
Here $x_{\dmf}$, $x_{\dff}$ and $x_{\cf}$ are defining functions as before. The weighting properties on the right hand side of (\ref{constant order d3sc b res Sobolev spaces full}) are so stated because we can arrange for
\begin{equation}
\label{Sobolev spaces subsection weights compare}
x_{\dff} \simeq \rho_{\dff} \rho_{\rf}, \ 
x_{\cf} \simeq \rho_{\mathrm{3cocf}_{\tindex}} \rho_{\mathrm{d3sccf}_{\tindex}}
\end{equation}
as smooth functions on $\psf \Xd$, where $\rho_{\dff}$, $\rho_{\rf}$, $\rho_{\mathrm{3cocf}_{\tindex}}$, $\rho_{\mathrm{d3sccf}_{\tindex}}$ are global defining functions for $\overline{ ^{\mathrm{d3sc,3co,res}}T^{\ast} }_{\dff} \Xd$, $\rf$, $ \mathrm{3cocf}_{\tindex}$ and $\dtsccf$ respectively. \par

Furthermore, in the presence of variable orders $\vom, \vor, \vov, \vos \in \mathcal{C}^{\infty}( \psf \Xd )$, we will first choose $M ,N,K, S \in \mathbb{R}$ such that 
\begin{equation*} 
\label{lower bounds for the admissible variable orders}
\mathsf{m} \geq M, \ \mathsf{r} \geq N, \ \mathsf{v} \geq K, \ \mathsf{s} \geq S.
\end{equation*}
Then for some fixed $\Lambda \in \Psi^{ \mathsf{m}, \mathsf{r}, 0, \mathsf{v}, 0 , \mathsf{s}}_{ \dtsccb, \delta}( \Xd )$ which is fully elliptic in the symbolic sense (i.e., $\mathrm{Ell}_{\sigma}(\Lambda)$ is the full set (\ref{phase space of the symbols})), we will define
\begin{equation*}
\begin{gathered}
H_{\mathrm{d3sc,3co,res}}^{\vom, \vor, 0, \vov, 0, \vos}(\Xd) \coloneq \{ u \in H_{\mathrm{d3sc,3co,res}}^{M,N,0,K,0,S}(\Xd) : \Lambda u \in L^2 \}, \\
\| u \|_{H_{\mathrm{d3sc,3co,res}}^{ \mathsf{m}, \mathsf{r}, 0 , \mathsf{v}, 0 , \mathsf{s} }}^2 \coloneq \| u \|_{H^{M,N, 0 ,K, 0 ,S}_{\mathrm{d3sc,3co,res}}}^2 + \| \Lambda u \|_{ L^{2}}^2.
\end{gathered}
\end{equation*} \par

Finally, in the presence of general $\vol, \vob \in \mathcal{C}^{\infty}( \mathcal{C}_{\tindex} \times \overline{\mathbb{R}^{n_{\tindex}}} )$, we will first let
\begin{equation*} \label{definition of the Sobolev space Lambda operator}
\tilde{\Lambda} \in \Psi^{0,0,  \vol, \vob, \vob , \vol }_{\dtsccb, \delta}(\Xd) 
\end{equation*}
be a fixed quantization of $x_{\dff}^{-\mathsf{l}} x_{\cf}^{-\mathsf{b}}$. Then for some $N,M,L,K,F,S \in \mathbb{R}$ such that
\begin{equation}
\label{conditions on the indices d3sc 3co res spaces}
M \geq \vom, \  N \geq r, \ \vol \geq L, \ \vov \geq K, \ \vob \geq F, \ \vos \geq S, 
\end{equation}
we will define
\begin{equation*}
\begin{gathered}
H^{\mathsf{m}, \mathsf{r}, \mathsf{l}, \mathsf{v}, \mathsf{b}, \mathsf{s}}_{\mathrm{d3sc,3co,res}}( \Xd ) \coloneq \{  u \in H_{\mathrm{d3sc,3co,res}}^{M,N,L,K, F,S}(\Xd) : \tilde{\Lambda} u \in H^{ \mathsf{m}, \mathsf{r}, 0, \mathsf{v} - \vob, 0, \mathsf{s} - \vol }_{\mathrm{d3sc,3co,res}}( \Xd ) \}, \\
\| u \|_{ H_{\mathrm{d3sc,3co,res}}^{ \vom, \vor, \vol, \vov, \vob, \vos } }^2 \coloneq \| u \|^2_{H_{\mathrm{d3sc,3co,res}}^{M,N,L,K,F,S}} + \| \tilde{\Lambda} u \|_{H_{\mathrm{d3sc,3co,res}}^{\vom, \vor, 0, \vov - \vob, 0, \vos - \vol} }^2.
\end{gathered}
\end{equation*}
We also record here that $H_{\mathrm{d3sc,3co,res}}^{\vom, \vor, \vol, \vov, \vob, \vos}(\Xd)$ is complete, the proof of which will be omitted.

We now consider boundedness of the further-resolved, second microlocalized operators on the various Sobolev spaces defined above, as well as some related consequences. \par

Due to the convoluted ways in which the above Sobolev spaces are defined, the proofs of boundedness on these spaces are considerably more involved than in the standard setting{\ep}and so are many other expected resulted related to the Sobolev spaces in this context. To keep the length of this paper under control, we will invest much efforts into the discussion of boundedness, but leave a few other properties to the readers. \par

We first consider the case with constant orders.

\begin{proposition}
\label{first proposition boundedness on Sobolev spaces}
Let $m,m',r,r',l,l' , b', b \in \mathbb{R}$. 
\begin{enumerate}
\item If $A \in \Psf^{m,r,l,m,b,m}(\Xd)$, then the map
\begin{equation}
\label{Sobolev boundedness 1}
A: H_{\mathrm{3co}}^{m',r',l',b'}(\Xd) \rightarrow H_{\mathrm{3co}}^{m' - m, r' - r, l' - l, s' - s}(\Xd).
\end{equation}
is bounded.
\item The specific choices of the elliptic operator $\Lambda \in \Psi_{\mathrm{d3sc,3co,res}}^{m,0,0,\nu,0,s}(\Xd)$ and $M \in \mathbb{R}$ used in the definition of $H_{\mathrm{d3sc,3co,res}}^{m,0,0,\nu,0,s}(\Xd)$ are not crucial{\ep}different choices give rises to the same space and equivalent norms.
\end{enumerate}
\end{proposition}
\begin{proof}
Part (1) follows from the local characterization (\ref{local characterization for 3co Sobolev spaces}) if $A \in \Psi_{\mathrm{3co}}^{m,0,0,0}(\Xd)$, but the proof can be made valid even if $A \in \Psf^{m,0,0,m,0,m}(\Xd)$. We omit the details here. \par

Part (2) follows from Corollary \ref{boundedness of order zero second microlocal operators on L^2} and the boundedness of (\ref{Sobolev boundedness 1}). To see this, assume that $\Lambda'$ and $M'$ are different choices. Let 
\begin{equation*}
\Gamma' \in \Psf^{-m,0,0,-\nu,0,-s}(\Xd) \subset \Psf^{-\min\{ m, \nu, s \}, 0 ,0, -\min\{ m, \nu, s \}, 0, -\min\{ m, \nu, s \}}(\Xd)
\end{equation*}
be a parametrix for $\Lambda'$ in the regularity sense that $I - \Gamma' \Lambda' \in \Psf^{-\infty, 0,0,-\infty,0,-\infty}(\Xd)$. Then
\begin{align*}
\| u \|_{H_{\mathrm{3co}}^{M,0,0,0}}^2 + \| \Lambda u \|_{L^2}^2 \leq {} & \| \Gamma' \Lambda' u \|_{H_{\mathrm{3co}}^{M,0,0,0}}^2 + \| \Lambda \Gamma' \Lambda' u \|_{L^{2}}^2 \\
& + \| ( I - \Gamma' \Lambda' ) u \|_{H_{\mathrm{3co}}^{M,0,0,0}}^2 + \| \Lambda ( I - \Gamma' \Lambda' ) u \|_{L^{2}}^2 \\
\leq {} & C ( \| \Lambda' u \|_{H_{\mathrm{3co}}^{M - \min \{ m, \nu, s \},0 ,0 ,0}}^2 + \| \Lambda' u \|_{L^{2}}^2 + \| u \|_{H_{\mathrm{3co}}^{M',0,0,0}}^2 )\\
\leq {} & C ( \| \Lambda' u \|_{L^2}^2 + \| u \|_{H_{\mathrm{3co}}^{M',0,0,0}}^2 ).
\end{align*}
This calculation also works in the opposite direction, which proves that the norms defined with respect to $\Lambda$, $M$ and $\Lambda'$, $M'$ are equivalent. 
\end{proof}

\begin{remark}
Though not yet phrased as such, the argument used in the proof of Proposition \ref{first proposition boundedness on Sobolev spaces}, part (2) is essentially elliptic regularity. However, we will delay the formal statements for microlocal elliptic regularity estimates until Proposition \ref{proposition microlocal elliptic regularity estimate chapter 1 thesis} below. 
\end{remark}

Next, we will show that

\begin{proposition}
\label{second proposition boundedness on Sobolev spaces}
Let $m,m',r,r',l,l', \nu, \nu' , b, b', s,s' \in \mathbb{R}$. 
\begin{enumerate}
\item If $A \in \Psf^{m,r,l,\nu,b,s}(\Xd)$, then the map
\begin{equation}
\label{Sobolev boundedness 1.5}
A : H_{\mathrm{d3sc,3co,res}}^{m',r',l',\nu',b',s'}(\Xd) \rightarrow H_{\mathrm{d3sc,3co,res}}^{m' - m, r' - r, l' - l, \nu' - \nu, b'-b, s'-s}(\Xd).
\end{equation} 
is bounded.
\item The specific choices of the elliptic operator $\Lambda \in \Psf^{\vom, \vor, 0, \vov, 0, \vos}(\Xd)$ and $M,N, K, S \in \mathbb{R}$ used in the definition of $H_{\mathrm{d3sc,3co,res}}^{\vom, \vor, 0 , \vov, 0, \vos}(\Xd)$ are not crucial{\ep}different choices give rises to the same space and equivalent norms.

\end{enumerate}
\end{proposition}

In order to prove part (2) of Proposition \ref{second proposition boundedness on Sobolev spaces}, we also need to show that:

\begin{lemma}
\label{lemma small inclusion Sobolev}
Let $\vom, \vor, \vov, \vos \in \mathcal{C}^{\infty}(\psf \Xd)$ and $\vom, \vor, \vov, \vos \leq 0$. Then we have 
\begin{equation*}
\text{$L^{2} \subset H_{\mathrm{d3sc,3co,res}}^{\vom, \vor, 0, \vov, 0, \vos }(\Xd)$ with continuous inclusion.}
\end{equation*}
\end{lemma}
\begin{proof}[Proof of the lemma]
Consider first the constant orders case. Thus assume that $m,r,\nu,s \leq 0$. Then by (\ref{three-cone embedding}) and Corollary \ref{boundedness of order zero second microlocal operators on L^2}, we can deduce that 
\begin{align*}
\| u \|_{H_{\mathrm{d3sc,3co,res}}^{m,r,0,\nu,0,s}}^2 & = \| x_{\dmf}^{-r} u \|_{H_{\mathrm{d3sc,3co,res}}^{m,0,0,\nu,0,s}}^2 \\
& \simeq \| u \|_{H_{\mathrm{3co}}^{M,r,0,0}}^2 + \| \Lambda_0 x_{\dmf}^{-r}  u  \|_{L^{2}}^2 \leq C \| u \|_{L^{2}}^2,
\end{align*}
where $M \leq \vom \leq 0$, $\Lambda_0 \in \Psf^{m, 0, 0, \nu  ,0 ,s }(\Xd)$, and thus
\begin{equation*}
\Lambda_0 x_{\dmf}^{-r}  \in \Psf^{m,r,0,\nu,0,s}(\Xd) \subset \Psf^{0,0,0,0,0,0}(\Xd)
\end{equation*} 
by assumption. \par

In the general, variable orders case, we have 
\begin{align*}
\| u \|_{H_{\mathrm{d3sc,3co,res}}^{\vom, \vor, 0, \vov, 0, \vos}}^2 & = \| u \|_{H_{\mathrm{d3sc,3co,res}}^{M,N,0,K,0,S}}^2 + \| \Lambda u \|_{L^2}^{2} \leq C \| u \|_{L^{2}}^2
\end{align*}
where $M, N, K,S \leq 0$, $\Lambda \in \Psf^{\vom, \vor, 0, \vov, 0, \vos}(X) \subset \Psf^{0,0,0,0,0,0}(\Xd)$ by assumptions. Thus, the above estimate follows from the constant orders case and Corollary \ref{boundedness of order zero second microlocal operators on L^2} once again.
\end{proof}

\begin{proof}[Proof of Proposition \ref{second proposition boundedness on Sobolev spaces}]
To prove part (1), we will first show that
\begin{equation}
\label{Sobolev boundedness constant order d3sc, 3co,res}
A_0 : H_{\mathrm{d3sc,3co,res}}^{m', 0, 0, \nu', 0, s'}(\Xd) \rightarrow H_{\mathrm{d3sc,3co,res}}^{m' - m, 0, 0, \nu' - \nu, 0, s' - s }(\Xd)
\end{equation}
is bounded for all $A_0 \in \Psi_{\mathrm{d3sc,3co,res},\delta}^{m,0,0,\nu,0,s}(\Xd)$. Assuming this is established, then the boundedness of (\ref{Sobolev boundedness 1.5}) follows directly from the observation that
\begin{equation*}
x_{\dmf}^{r'-r} x_{\dff}^{l'-l} x_{\cf}^{b'-b} A  x_{\dmf}^{r'} x_{\dff}^{l'} x_{\cf}^{b'} \in \Psi_{\mathrm{d3sc,3co,res},\delta}^{m,r, 0 , \nu - b, 0, s - l }(\Xd)
\end{equation*}
for all $A \in \Psf^{m,r,l,\mu,b,s}(\Xd)$. \par

To this end, note that by Proposition \ref{second proposition boundedness on Sobolev spaces}, part (2), we know that
\begin{equation*}
\| A_0 u \|_{H_{\mathrm{d3sc,3co,res}}^{m'-m, 0, 0, \nu' - \nu, 0, s'-s}}^2 \simeq \| A_0 u  \|_{H_{\mathrm{3co}}^{M,0,0,0}}^2 + \| \Lambda A_0 u \|_{L^{2}}^2,
\end{equation*}
where $\Lambda \in \Psi_{\mathrm{d3sc,3co,res}}^{m'-m, 0, 0, \nu' - \nu, 0, s'-s}(\Xd)$ is any operator that is elliptic at $\dtsccf$, $\rf$ and the fiber infinity, while $M$ is any constant which satisfies $M \leq \min \{ m' - m, \nu' - \nu, s' - s \}$.

Let $\mu \coloneq \max \{ m, \nu, s \}$. Then we have $A_0 \in \Psf^{\mu,0,0, \mu ,0, \mu }(\Xd)$. Moreover, for every $M' \leq \min \{ m', \nu', s' \}$, we can find some sufficiently negative $M$ such that $M + \mu \leq M'$. Thus by Proposition \ref{first proposition boundedness on Sobolev spaces}, part (1), we can estimate
\begin{equation}
\label{sobolev mapping estiamte 1}
\| A_0 u \|_{H_{\mathrm{3co}}^{M,0,0,0}}^2 \leq C \| u \|_{H_{\mathrm{3co}}^{M + \mu ,0,0,0}}^2 \leq C \| u \|_{H_{\mathrm{3co}}^{M',0,0,0}}^2 \leq C \| u \|_{H_{\mathrm{d3sc,3co,res}}^{m', 0,0,\nu',0,s'}}^2. 
\end{equation}
On the other hand, let $\Lambda' \in \Psf^{m', 0,0,\nu',0,s'}(\Xd)$ be elliptic at $\dtsccf$, $\rf$ and the fiber infinity, and let $\Gamma'$ be a parametrix of $\Lambda'$ such that
\begin{equation*}
\Lambda A_0 \Gamma' \in \Psf^{0,0,0,0,0,0}(\Xd), \quad \Lambda A_0 ( I - \Gamma' \Lambda' ) \in \Psf^{-\infty, 0, 0, -\infty,0 ,-\infty}(\Xd).
\end{equation*}
Then by Corollary \ref{boundedness of order zero second microlocal operators on L^2} and Proposition \ref{first proposition boundedness on Sobolev spaces}, part (1) once again, we have
\begin{align*}
\| \Lambda A_0 \|_{L^2}^2  & \leq \| \Lambda A_0 \Gamma' \Lambda' u \|_{L^2}^2 + \| \Lambda A_0 ( I - \Gamma' \Lambda' ) u \|_{L^2}^2  \\
& \leq C ( \| \Lambda' u \|_{L^2}^2  + \| u \|_{H^{M',0,0,0}_{\mathrm{3co}}}^2 ) \simeq \| u \|_{H^{m',0,0,\nu',0,s'}_{\mathrm{d3sc,3co,res}}}^2.
\end{align*}
Thus the boundedness of (\ref{Sobolev boundedness constant order d3sc, 3co,res}) is proved.  \par

For part (2), we can follow the argument used in the proof of Proposition \ref{first proposition boundedness on Sobolev spaces}, part (1). Suppose that $\Lambda \in \Psf^{\vom, \vor, 0, \vov, 0, \vos}(\Xd)$ and $M,N,K,S$ are used to define the norm of $H_{\mathrm{d3sc,3co,res}}^{\vom, \vor, 0, \vov, 0, \vos}(\Xd)$, and that $\Lambda'$ and $M',N',K',S' \in \mathbb{R}$ are different choices. Moreover, suppose that we write 
\begin{equation*}
m_0 \coloneq \min \vom, \ r_0 \coloneq \min \vor, \ \nu_0 \coloneq \min \vov, \ s_0 \coloneq \min \vos.
\end{equation*}
Then by an ellipticity argument (similar to the proof of Proposition \ref{first proposition boundedness on Sobolev spaces}, part (2)), we have
\begin{align*}
\| u \|_{H_{\mathrm{3co}}^{M,N,0,K,0,S}}^2 + \| \Lambda u \|_{L^2}^2 & \leq C ( \| \Lambda' u \|_{H_{\mathrm{d3sc,3co,res}}^{M - m_0, N - r_0, 0, K - \nu_0 0, S - s_0}}^2 + \| u \|_{H^{M',N',0,K',0,S'}_{\mathrm{d3sc,3co,res}}}^2 + \| \Lambda' u \|_{L^2}^2  ) \\
& \leq C ( \|  \Lambda' u \|_{L^{2}}^2 + \| u \|_{H_{\mathrm{d3sc,3co,res}}^{M',N',0,K',0,S'}}^2 ),
\end{align*}
where we have used that $M - \vom \leq 0$, $N - \vor \leq 0$, $K - \vov \leq 0$, $S - \vos \leq 0$ and Lemma \ref{lemma small inclusion Sobolev}. 
\end{proof}

Finally, we will consider boundedness in the presence of variable orders. We first consider the case when $\vol = \vob = 0$.

\begin{corollary}
\label{second corollary boundedness on Sobolev spaces}
If $\vom, \vom', \vor, \vor',  \vov, \vov', \vos, \vos' \in \mathcal{C}^{\infty}( \psf \Xd )$ and $A \in \Psf^{\vom, \vor, 0, \vov, 0, \vos}(\Xd)$, then
\begin{equation*}
A:  H_{\mathrm{d3sc,3co,res}}^{\vom', \vor', 0, \vov', 0, \vos'}(\Xd) \rightarrow H_{\mathrm{d3sc,3co,res}}^{\vom' - \vom, \vor' - \vor, 0, \vov' - \vov, 0, \vos' - \vos}(\Xd)
\end{equation*}
is a bounded linear map.
\end{corollary}
\begin{proof}
The proof of this corollary will be omitted since it is similar to the argument used in the proof of Proposition \ref{second proposition boundedness on Sobolev spaces}, part (1).
\end{proof}

By using Proposition \ref{second proposition boundedness on Sobolev spaces} and the above corollary, we can also make the following improvements to Lemma \ref{lemma small inclusion Sobolev}:

\begin{lemma}
\label{lemma big inclusion Sobolev}
Let $\vom, \vor,  \vov, \vos \in \mathcal{C}^{\infty}( \psf \Xd )$, $\vol, \vob \in \mathcal{C}^{\infty}(\mathcal{C}_{\tindex} \times \overline{\mathbb{R}^{n_{\tindex}}})$.
\begin{enumerate}
\item If $\vom, \vor, \vol, \vov, \vob, \vos \leq 0$, then we have
\begin{equation*}
\text{$L^2 \subset H_{\mathrm{d3sc,3co,res}}^{\vom, \vor, \vol, \vov, \vob, \vos}(\Xd)$ with continuous inclusion.}
\end{equation*}
\item If $\vob, \vol \leq 0$, then we have
\begin{equation*}
\text{$H_{\mathrm{d3sc,3co,res}}^{\vom, \vor, 0 , \vov, 0 , \vos} (\Xd) \subset H_{\mathrm{d3sc,3co,res}}^{\vom, \vor, \vol, \vov, \vob, \vos}(\Xd)$ with continuous inclusion.}
\end{equation*}
\end{enumerate}
\end{lemma}
\begin{proof}
We first prove part (1). The case of constant orders is an easy consequence of Lemma \ref{lemma small inclusion Sobolev}. In the variable orders case, note that by using Proposition \ref{second proposition boundedness on Sobolev spaces}, part (2), we have
\begin{align*}
\| u \|_{H_{\mathrm{d3sc,3co,res}}^{\vom, \vor, \vol, \vov, \vob ,\vos}}^2  & = \| u \|_{H_{\mathrm{d3sc,3co,res}}^{M,N,L,K,F,S}}^2 + \| \tilde{\Lambda} u \|_{H_{\mathrm{d3sc,3co,res}}^{\vom, \vor, 0, \vov - \vob, 0, \vos - \vol}}^2  \\
& \simeq \| u \|_{H_{\mathrm{d3sc,3co,res}}^{M,N,L,K,F,S}}^2 + \| \tilde{\Lambda} u \|_{H_{\mathrm{d3sc,3co,res}}^{\tilde{M}, \tilde{N},0, \tilde{K}, 0, \tilde{S} }} + \| \Lambda \tilde{\Lambda} u \|_{L^2}^2,
\end{align*}
where $M,N,K,K,F,S \leq 0$ must for now be fixed, but $\tilde{M}, \tilde{N}, \tilde{K}, \tilde{S} < 0$ can be taken arbitrarily negative. Moreover, we must have
\begin{equation*}
\tilde{\Lambda} \in \Psf^{0, 0, \vol, \vob, \vob , \vol}(\Xd) \subset \Psf^{0,0,0,0,0,0}(\Xd), \quad \Lambda \tilde{\Lambda} \in \Psf^{\vom, \vor, \vol, \vov, \vob, \vos}(\Xd) \subset \Psf^{0,0,0,0,0,0}(\Xd).
\end{equation*}
Thus the required claim follows from Corollary \ref{boundedness of order zero second microlocal operators on L^2} and the result in the constant orders case.  \par

For part (2), we first write
\begin{equation*}
\| u \|_{H_{\mathrm{d3sc,3co,res}}^{\vom, \vor, \vol, \vov, \vob, \vos}}^2 = \| x_{\dff}^{-L} x_{\cf}^{-K} u \|_{H_{\mathrm{d3sc,3co,res}}^{M,N,0,K - F, 0 ,S - L}}^2 + \| \tilde{\Lambda} u \|_{H_{\mathrm{d3sc,3co,res}}^{\vom, \vor, 0, \vov - \vob, 0, \vos - \vol}}^2,
\end{equation*}
where $M,N,L,K,F,S \leq 0$ are again fixed. Since $L \leq \vol \leq 0$, $F \leq \vob \leq 0$, we must have
\begin{equation*}
\tilde{\Lambda} \in \Psf^{0, 0, \vol, \vob, \vob, \vol}(\Xd) \subset \Psf^{0,0,0,\vob, 0, \vol}(\Xd),  \quad  x_{\dff}^{-L} x_{\cf}^{-K} \in \Psi_{\mathrm{d3sc,3co,res}}^{0,0,L,F,F,L}(\Xd) \subset \Psf^{0,0,0,F, 0, L}(\Xd).
\end{equation*}
Thus by Corollary \ref{second corollary boundedness on Sobolev spaces} and Proposition \ref{second proposition boundedness on Sobolev spaces}, part (2), we have
\begin{equation*}
\| u \|_{H_{\mathrm{d3sc,3co,res}}^{\vom, \vor, \vol, \vov, \vob, \vos}}^2 \leq C ( \| u \|_{H_{\mathrm{d3sc,3co,res}}^{M,N,0,K,0,S}}^2 + \| u \|_{H_{\mathrm{d3sc,3co,res}}^{\vom, \vor, 0, \vov, 0, \vos}}^2 ) \leq C \| u \|_{H_{\mathrm{d3sc,3co,res}}^{\vom, \vor, 0, \vov, 0, \vos}}^2
\end{equation*}
as required.
\end{proof}

A consequence of Lemma \ref{lemma big inclusion Sobolev} is the following improvement of Proposition \ref{second proposition boundedness on Sobolev spaces}, part (2):

\begin{corollary}
\label{big big independence of things sobolev corollary}
The specific choices of the quantization $\Lambda \in \Psf^{0,0, \vol , \vob, \vob,  \vol}(\Xd)$ of $x_{\dff}^{-\vol} x_{\cf}^{-\vob}$ and $M,N, L, K, F, S \in \mathbb{R}$ used in the definition of $H_{\mathrm{d3sc,3co,res}}^{\vom, \vor, \vol , \vov, \vob, \vos}(\Xd)$ are not crucial{\ep}different choices give rises to the same space and equivalent norms.
\end{corollary}

\begin{proof}
First, recall that by assumption, we have
\begin{equation*}
\| u \|_{H_{\mathrm{d3sc,3co,res}}^{\vom, \vor, \vol, \vov, \vob, \vos}}^2 = \| u \|_{H_{\mathrm{d3sc,3co,res}}^{M,N,L,K,F,S}}^2 + \| \tilde{\Lambda} u \|_{H_{\mathrm{d3sc,3co,res}}^{\vom, \vor, 0, \vov - \vob, 0, \vos - \vol}}^2.
\end{equation*} \par 

Next, let $l_0 \coloneq \min \vol$, $b_0 \coloneq \min \vob$. Suppose that $\tilde{\Lambda}' \in \Psf^{0,0,\vol,\vob,\vob,\vol}(\Xd)$ is another quantization of $x_{\dff}^{-\vol} x_{\cf}^{-\vob}$, and that $M',N',K',F',S' \in \mathbb{R}$ satisfy condition (\ref{conditions on the indices d3sc 3co res spaces}) as well. Let 
\begin{equation*}
\tilde{\Gamma}' \in \Psf^{0,0,-\vol, -\vob, -\vob, -\vol}(\Xd) \subset \Psf^{0,0, - l_0, -\vob, -b_0, -\vol}(\Xd)
\end{equation*}
be an elliptic parametrix of $\tilde{\Lambda}'$ in the sense of Lemma \ref{special lemma elliptic parametrix}. Then by Proposition \ref{second proposition boundedness on Sobolev spaces}, part (1), Lemma \ref{lemma big inclusion Sobolev}, part (2), and an elliptic regularity argument, we have
\begin{align*}
\| u \|_{H_{\mathrm{d3sc,3co,res}}^{M,N,L,K,F,S}}^2 \leq {} & C ( \| \tilde{\Lambda}' u \|_{H_{\mathrm{d3sc,3co,res}}^{M, N, L - l_0, K - \vob, F - b_0, S - \vol}}^2 + \| u \|_{H_{\mathrm{d3sc,3co,res}}^{M',N',L',K',F',S'}}^2 ) \\
\leq {} & C ( \| \tilde{\Lambda}' u \|_{H_{\mathrm{d3sc,3co,res}}^{M, N, 0, K - \vob, 0, S - \vol}}^2 + \| u \|_{H_{\mathrm{d3sc,3co,res}}^{M',N',L',K',F',S'}}^2 ).
\end{align*} \par

Let $\Lambda \in \Psf^{\vom, \vor, 0,  \vov - \vob, 0, \vos - \vol}(\Xd)$ be elliptic in the symbolic sense. Then by Lemma \ref{lemma small inclusion Sobolev}, we have
\begin{align*}
\| \tilde{\Lambda}' u \|_{H_{\mathrm{d3sc,3co,res}}^{M,N,0,K - \vob, 0, S - \vol}}^2 \leq {} & C ( \| \Lambda \tilde{\Lambda}' u \|_{H_{\mathrm{d3sc,3co,res}}^{M - \vom, N - \vor, 0, K - \vov, 0 , S - \vos }}^2  + \| \tilde{\Lambda}' u \|_{H_{\mathrm{d3sc,3co,res}}^{\vom, \vor, 0, \vov - \vob, 0, \vos - \vol}}^2 ) \\
\leq {} & C ( \| \Lambda \tilde{\Lambda}' u \|_{L^2}^2 + \| \tilde{\Lambda}' u \|_{H_{\mathrm{d3sc,3co,res}}^{\vom, \vor, 0, \vov - \vob, 0, \vos - \vol}} ) \leq C \| \tilde{\Lambda}' u \|_{H_{\mathrm{d3sc,3co,res}}^{\vom, \vor, 0, \vov - \vob, 0, \vos - \vol}}^2. 
\end{align*}
On the other hand, it is easy to see that
\begin{equation*}
\| \tilde{\Lambda} u  \|_{H_{\mathrm{d3sc,3co,res}}^{\vom, \vor, 0, \vov - \vob, 0, \vos - \vol}}^2 \leq C ( \| \tilde{\Lambda}' u \|_{H_{\mathrm{d3sc,3co,res}}^{\vom, \vor, 0, \vov - \vob, 0, \vos - \vol}}^2 + \| u \|_{H_{\mathrm{d3sc,3co,res}}^{M',N', L', K', F',S'}}^2 ).
\end{equation*}
Thus overall, we have proved that
\begin{equation*}
\| u \|_{H_{\mathrm{d3sc,3co,res}}^{M,N,L,K,F,S}}^2 + \| \tilde{\Lambda} u \|_{H_{\mathrm{d3sc,3co,res}}^{\vom, \vor, 0, \vov - \vob, 0, \vos - \vol}}^2 \leq C ( \| u \|_{H_{\mathrm{d3sc,3co,res}}^{M',N',L',K',F',S'}}^2 + \| \tilde{\Lambda}' u \|_{H_{\mathrm{d3sc,3co,res}}^{\vom, \vor, 0, \vov - \vob, 0, \vos - \vol}}^2 ).
\end{equation*} \par

Since the above argument also works in the opposite direction, we conclude that the norms defined by $\tilde{\Lambda}$, $M,N,L,K,F,S$ and $\tilde{\Lambda}'$, $M',N',L',K',F',S'$ must be equivalent.
\end{proof}

Finally, we will consider boundedness when $\vol, \vob \neq 0$.

\begin{proposition}[Boundedness on Sobolev spaces]
\label{big proposition sobolev spaces}
Let
\begin{equation*}
\vom, \vom', \vor, \vor',  \vov, \vov', \vos, \vos' \in \mathcal{C}^{\infty}( \psf \Xd ), \quad \vol, \vol', \vob, \vob' \in \mathcal{C}^{\infty}(\mathcal{C}_{\tindex} \times \overline{\mathbb{R}^{n_{\tindex}}}).
\end{equation*}
If $A \in \Psf^{\vom, \vor, \vol, \vov, \vob, \vos}(\Xd)$, then
\begin{equation} 
\label{possible Sobolev second microlocal bounding 1}
A :H^{ \mathsf{m}', \mathsf{r}', \mathsf{l}', \mathsf{v}', \mathsf{b}', \mathsf{s}' }_{\mathrm{d3sc,3co,res}}(\Xd) \rightarrow H^{ \mathsf{m}' - \mathsf{m}, \mathsf{r}' - \mathsf{r}, \mathsf{l}' - \mathsf{l}, \mathsf{v}' - \mathsf{v}, \mathsf{b}' - \mathsf{b}, \mathsf{s}' - \mathsf{s} }_{\mathrm{d3sc,3co,res}}(\Xd).
\end{equation}
is a bounded linear map.
\end{proposition}

\begin{proof}
Let $\tilde{\Lambda} \in \Psf^{0,0,\vol'-\vol, \vob' - \vob, \vob' - \vob, \vol' - \vol}(\Xd)$ be some quantization of $x_{\dff}^{-\vol'+\vol} x_{\cf}^{-\vob' + \vob}$. For any sufficiently negative $M,N,L,K,F,S \in \mathbb{R}$, we have proved that
\begin{equation*}
\| Au  \|_{H_{\mathrm{d3sc,3co,res}}^{\vom' - \vom, \vor' - \vor, \vol' - \vol, \vov' -\vov, \vob' - \vob, \vos' - \vos}}^2 \simeq \| A u  \|_{H_{\mathrm{d3sc,3co,res}}^{M,N, L, K,F,S}}^2 + \| \tilde{\Lambda} A u \|_{H_{\mathrm{d3sc,3co,res}}^{\vom' - \vom, \vor' - \vor, 0, (\vov' - \vob') - (\vov - \vob), 0, (\vos ' - \vol') - (\vos - \vol) }}^2.
\end{equation*}

By Corollary \ref{big big independence of things sobolev corollary}, and by using the same argument as in the proof of (\ref{sobolev mapping estiamte 1}),  it is easy to see that we can arrange the above constants so that
\begin{equation*}
\| Au \|_{H_{\mathrm{d3sc,3co,res}}^{M,N,L,K,F,S}}^2 \leq C \| u \|_{H_{\mathrm{d3sc,3co,res}}^{M',N',L',K',F',S'}}^2 \leq C \| u \|_{H_{\mathrm{d3sc,3co,res}}^{\vom',\vor',\vol',\vov',\vob',\vos'}}^2
\end{equation*}
for some arbitrarily negative $M', N', L' , K', F',S' \in \mathbb{R}$. \par

Meanwhile, let $\tilde{\Lambda}' \in \Psf^{0,0,\vol',\vob',\vob',\vol'}(\Xd)$ be some quantization of $x_{\dff}^{-\vol'} x_{\cf}^{-\vob'}$, and let $\tilde{\Gamma}'$ be a parametrix of $\tilde{\Lambda}'$ in the sense of Lemma \ref{special lemma elliptic parametrix}. Then we have
\begin{equation*}
\tilde{\Lambda} A \tilde{\Gamma}' \in \Psf^{\vom, \vor, 0, \vov - \vob, 0, \vos - \vol }(\Xd), \quad \tilde{\Lambda} A ( I -  \tilde{\Gamma}' \tilde{\Lambda}' ) \in \Psf^{-\infty, -\infty, -\infty, -\infty, -\infty, -\infty}(\Xd).
\end{equation*}
It follows that we can estimate
\begin{align*}
\| \tilde{\Lambda} A u \|_{H_{\mathrm{d3sc,3co,res}}^{\vom' - \vom, \vor' - \vor, 0, ( \vov' - \vob ) - ( \vov - \vob ), 0, ( \vos' - \vol' ) - ( \vos - \vol )}}^2 & \leq \|  \tilde{\Lambda} A \tilde{\Gamma}' \tilde{\Lambda}' u \|_{H_{\mathrm{d3sc,3co,res}}^{\vom' - \vom, \vor' - \vor, 0, ( \vov' - \vob ) - ( \vov - \vob ), 0, ( \vos' - \vol' ) - ( \vos - \vol )}}^2 \\
& \quad + \| \tilde{\Lambda} A ( I - \tilde{\Gamma}' \tilde{\Lambda}' ) u \|_{H_{\mathrm{d3sc,3co,res}}^{\vom' - \vom, \vor' - \vor, 0, ( \vov' - \vob ) - ( \vov - \vob ), 0, ( \vos' - \vol' ) - ( \vos - \vol )}}^2 \\
& \leq C ( \| \tilde{\Lambda}' u \|_{H_{\mathrm{d3sc,3co,res}}^{\vom', \vor', \vol', 0, \vov' - \vob', 0, \vos' - \vol'}}^2 + \| u \|_{H_{\mathrm{d3sc,3co,res}}^{M',N',L', K',F',S'}}^2 ) \\
& \leq C \| u \|_{H_{\mathrm{d3sc,3co,res}}^{\vom', \vor', \vol', \vov', \vob', \vos'}}^2,
\end{align*}
which proves the boundedness of (\ref{possible Sobolev second microlocal bounding 1}).
\end{proof}

\begin{corollary}[Inclusions of Sobolev spaces]
Let
\begin{equation*}
\vom, \vom', \vor, \vor',  \vov, \vov', \vos, \vos' \in \mathcal{C}^{\infty}( \psf \Xd ), \quad \vol, \vol', \vob, \vob' \in \mathcal{C}^{\infty}(\mathcal{C}_{\tindex} \times \overline{\mathbb{R}^{n_{\tindex}}}).
\end{equation*}
be such that
\begin{equation*}
\vom' \leq \vom, \ \vor' \leq \vor, \ \vol' \leq \vol, \ \vov' \leq \vov, \ \vob' \leq \vob, \ \vos' \leq \vos.
\end{equation*}
Then we have
\begin{equation*}
\text{$H_{\mathrm{d3sc,3co,res}}^{\vom, \vor, \vol, \vov, \vob, \vos}(X) \subset H_{\mathrm{d3sc,3co,res}}^{\vom', \vor', \vol', \vov', \vob', \vos'}(\Xd)$ with continuous inclusion. }
\end{equation*}
\end{corollary}
\begin{proof}
Let 
\begin{equation*}
\begin{gathered}
\Lambda \in \Psf^{\vom' - \vom, \vor' - \vor, 0 , (\vov' - \vob') - (\vov - \vob ), 0 , ( \vos' - \vol' ) - ( \vos - \vol ) }(\Xd) \subset \Psf^{0,0,0,- \vob' + \vob,0, - \vol' + \vol}(\Xd), \\
\tilde{\Lambda} \in \Psf^{0,0, \vol' - \vol, \vob' - \vob, \vob' - \vob, \vol' - \vol}(\Xd)  \subset \Psf^{0,0,0,0,0,0}(\Xd) \ \text{some quantization of $x_{\dff}^{-\vol'+\vol} x_{\cf}^{-\vob'+\vob}$}
\end{gathered}
\end{equation*}
be elliptic in the symbolic sense and in the sense of Lemma \ref{special lemma elliptic parametrix} respectively, and notice that
\begin{equation*}
\Lambda \tilde{\Lambda} \in \Psf^{0,0, \vol' - \vol, 0, \vob' - \vob, 0}(\Xd) \subset \Psf^{0,0, 0, 0, 0, 0}(\Xd) 
\end{equation*}
as well. Then by the standard elliptic regularity arguments and Proposition \ref{big proposition sobolev spaces}, we have
\begin{align*}
\| u \|_{H_{\mathrm{d3sc,3co,res}}^{\vom', \vor', \vol', \vov', \vob', \vos'}}^2 \leq {} & C ( \| \tilde{\Lambda} u \|_{H_{\mathrm{d3sc,3co,res}}^{\vom', \vor', \vol, \vov' - \vob' + \vob, \vob, \vos' - \vol' + \vol}}^2 + \| u \|_{H_{\mathrm{d3sc,3co,res}}^{\vom , \vor , \vol , \vov  , \vob , \vos }}^2  ) \\
\leq {} & C ( \| \Lambda \tilde{\Lambda} u \|_{H_{\mathrm{d3sc,3co,res}}^{\vom , \vor , \vol,  \vov , \vob, \vos }}^2 + \| \tilde{\Lambda} u \|_{H_{\mathrm{d3sc,3co,res}}^{\vom , \vor  , \vol,  \vov  , \vob, \vos }}^2 + \| u \|_{H_{\mathrm{d3sc,3co,res}}^{\vom , \vor , \vol, \vov, \vob, \vos}}^2 ) \leq C \| u \|_{H_{\mathrm{d3sc,3co,res}}^{\vom , \vor , \vol, \vov, \vob, \vos}}^2
\end{align*}
as required. 
\end{proof}

As mentioned, we will leave the following expected results for the readers to check:

\begin{proposition}[Dual spaces of Sobolev spaces]
\label{prop: dual spaces of the Sobolev spaces}
Let
\begin{equation*}
m,r,l, b \in \mathbb{R}, \quad \vom, \vor, \vov, \vos \in \mathcal{C}^{\infty}( \psf \Xd ), \quad \vol, \vob \in \mathcal{C}^{\infty}( \mathcal{C}_{\tindex} \times \overline{\mathbb{R}^{n_{\tindex}}} ).
\end{equation*}
Then we have
\begin{equation*}
\big( H_{\mathrm{3co}}^{m,r,l,b}(\Xd) \big)^{\ast} = H_{\mathrm{3co}}^{ - m, - r,  - l,   - b}(\Xd), \quad 
\big( H_{\mathrm{d3sc,3co,res}}^{\vom, \vor, \vol, \vov, \vob, \vos}(\Xd) \big)^{\ast} = H_{\mathrm{d3sc,3co,res}}^{ - \vom, - \vor, - \vol,  - \vov, - \vob, - \vos}(\Xd)
\end{equation*}
with equivalence of norms, where the dualities are taken with respect to $L^2$.
\end{proposition}

\begin{proposition}[Compact inclusions of Sobolev spaces]
Let
\begin{equation*}
\vom, \vom', \vor, \vor',  \vov, \vov', \vos, \vos' \in \mathcal{C}^{\infty}( \psf \Xd ), \quad \vol, \vol', \vob, \vob' \in \mathcal{C}^{\infty}(\mathcal{C}_{\tindex} \times \overline{\mathbb{R}^{n_{\tindex}}}).
\end{equation*}
be such that
\begin{equation*}
\vom' < \vom, \ \vor' < \vor, \ \vol' < \vol, \ \vov' < \vov, \ \vob'  <  \vob, \ \vos'  < \vos.
\end{equation*}
Then we have
\begin{equation*}
\text{$H_{\mathrm{d3sc,3co,res}}^{\vom, \vor, \vol, \vov, \vob, \vos}(X) \subset H_{\mathrm{d3sc,3co,res}}^{\vom', \vor', \vol', \vov', \vob', \vos'}(\Xd)$ with compact inclusion. }
\end{equation*}

\end{proposition}

\begin{remark}
Equivalently, for $m \in \mathbb{R}$, we can also write
\begin{equation}
\label{alternate characterization of 3co Soboelv spaces negative}
\begin{gathered}
H_{\mathrm{3co}}^{m,0,0,0}(\Xd) =\{ u \in H_{\mathrm{3co}}^{M,0,0,0} (\Xd) : \Lambda u \in L^{2}  \}, \quad \| u \|_{H_{\mathrm{3co}}^{m,0,0,0}}^2 \simeq \| u \|_{H_{\mathrm{3co}}^{M,0}}^2 + \| \Lambda u \|_{L^2}^2,
\end{gathered}
\end{equation}
where $M \leq 0$ and $\Lambda \in \Psi_{\mathrm{3co}}^{m,0,0,0}(\Xd)$ is elliptic in the regularity sense. Moreover, the specific choices of $M$ and $\Lambda$ are not crucial.
\end{remark}

\begin{remark}
We also remark that
\begin{equation*}
\label{equivalence between the three-cone Sobolev space and second microlocalized Soboelv space}
H_{\mathrm{3co}}^{m,r,l,b}(\Xd) = H_{\mathrm{d3sc,3co,res}}^{m,r,l, m + b, b, m + l} (\Xd)
\end{equation*}
with equivalence of norms. To see this, it suffices to assume that $r=l=b=0$. Then the equality follows easily from (\ref{alternate characterization of 3co Soboelv spaces negative}) since $\Psi_{\mathrm{3coc}}^{m,0,0,0}(\Xd) \subset \Psi_{\mathrm{d3sc,3co,res}}^{m,0,0,m,0,m}(\Xd)$. Moreover, if $\Lambda$ is elliptic in the regularity sense, then it must also be elliptic at $\dtsccf$, $\rf$ and the fiber infinity as an element of $\Psi_{\mathrm{d3sc,3co,res}}^{m,r,0,m,0,m}(\Xd)$. \par

\end{remark}

\subsection{Microlocal elliptic regularity estimates} 
\label{subsection;microlocal elliptic regularity estimates}
Finally, we will state several microlocal elliptic regularity estimates in the further-resolved, second microlocalized setting. 
\begin{proposition}[Microlocal elliptic estimates] 
\label{proposition microlocal elliptic regularity estimate chapter 1 thesis}
Suppose that
\begin{equation*}
\vom, \vom', \vor, \vor',  \vov, \vov', \vos, \vos' \in \mathcal{C}^{\infty}( \psf \Xd ), \quad \vol, \vol', \vob, \vob' \in \mathcal{C}^{\infty}(\mathcal{C}_{\tindex} \times \overline{\mathbb{R}^{n_{\tindex}}})
\end{equation*}
and $A \in \Psf^{\vom, \vor, \vol, \vov, \vob, \vos}(\Xd)$. Assume further that $Q, Q' \in \Psf^{0,0,0,0,0,0}(\Xd)$. Then:

\begin{enumerate}
\item If $\WFs(Q) \subset \Ellss(A) \cap \Ellss(Q')$, then for any choice of $M,N,K,S \in \mathbb{R}$, there exists a constant $C > 0$ such that
\begin{equation}
\label{elliptic estimate 1}
\| Qu \|_{H^{ \vom', \vor',  \vol', \vov' , \vob',  \vos' }_{\mathrm{d3sc,3co,res}}} \leq C ( \| Q' Au \|_{H^{  \mathsf{m}' - \vom,  \mathsf{r}' - \vor,  \mathsf{l}' - \vol , \mathsf{v}' - \vov ,  \mathsf{b}' - \vob, \mathsf{s}' - \vos}_{\mathrm{d3sc,3co,res}}} + \| u \|_{H^{M, N, \vol', K,  \vob', S}_{\mathrm{d3sc,3co,res}}}  )
\end{equation}
in the strong sense that if the right hand side of (\ref{elliptic estimate 1}) is finite, then so is the left hand side, and the estimate holds.
\item If $\WFdff(Q) \subset \Elldff(A) \cap \Elldff(Q')$, then for any choice of $L \in \mathbb{R}$, there exists a constant $C > 0$ such that
\begin{equation}
\label{elliptic estimate 2}
\| Qu \|_{H^{ \mathsf{m}', \mathsf{r}',  \mathsf{l}', \mathsf{v}',  \mathsf{b}',   \mathsf{s}'}_{\mathrm{d3sc,3co,res}}} \leq C ( \| Q' Au \|_{H^{  \mathsf{m}' - \vom,  \mathsf{r}' - \vor, \mathsf{l}' - \vol , \mathsf{v}' - \vov,  \mathsf{b}' - \vob, \mathsf{s}' - \vos}_{\mathrm{d3sc,3co,res}}} + \| u \|_{H^{ \mathsf{m}', \mathsf{r}',  L, \mathsf{v}', \mathsf{b}',   \mathsf{s}'}_{\mathrm{d3sc,3co,res}}}  )
\end{equation}
in the strong sense that if the right hand side of (\ref{elliptic estimate 2}) is finite, then so is the left hand side, and the estimate holds.
\item If $\WFcf(Q) \subset \Ellcf(A) \cap \Ellcf(Q')$, then for any choice of $F \in \mathbb{R}$, there exists a constant $C > 0$ such that 
\begin{equation}
\label{elliptic estimate 3}
\| Qu \|_{H^{ \mathsf{m}', \mathsf{r}', \mathsf{l}', \mathsf{v}', \mathsf{b}', \mathsf{s}'}_{\mathrm{d3sc,3co,res}}} \leq C ( \| Q' Au \|_{H^{  \mathsf{m}' - \vom,  \mathsf{r}' - \vor, \mathsf{l}' - \vol , \mathsf{v}' - \vov,  \mathsf{b}' - \vob, \mathsf{s}' - \vos}_{\mathrm{d3sc,3co,res}}} + \| u \|_{H^{ \mathsf{m}', \mathsf{r}',   \vol' , \mathsf{v}', F,   \mathsf{s}'}_{\mathrm{d3sc,3co,res}}}  )
\end{equation}
in the strong sense that if the right hand side of (\ref{elliptic estimate 3}) is finite, then so is the left hand side, and the estimate holds.
\end{enumerate}
\end{proposition}
\begin{proof}
In every case the proof makes use of Proposition \ref{Proposition microlocal ellipticity parametrix} and is standard (see for instance \cite[Corollary 5.5]{AndrasBook}). The detailed discussion will therefore be omitted.
\end{proof}

We will also state a uniform version of the microlocal elliptic regularity estimate in the symbolic sense. The proof is again omitted. See for instance \cite[Chapter 8]{PeterNotes} for the proof in the case on a smooth, compact manifold.
\begin{lemma}[Uniform microlocal elliptic estimate]
\label{uniform elliptic regularity estimate lemma}
Let $\{ A_s \} \in L^{\infty} ( (0,1]_{s}  ; \Psi_{\mathrm{d3sc,3co,res}}^{\vom, \vor, \vol, \vov, \vob, \vos} (\Xd) )$ and $u \in \mathcal{S}'$. Suppose that $B \in \Psi^{0,0,0,0,0,0}_{\mathrm{d3sc,3co,res},\delta}(\Xd)$ satisfies $\WFsL( \{ A_s \} ) \subset \Ellss(B)$. Then for any choice of $M,N,K,S \in \mathbb{R}$, there exists a constant $C > 0$, which is independent of $s \in (0,1]$, such that
\begin{equation} \label{uniform elliptic regularity estimate symbolic}
\| A_su \|_{H^{ \mathsf{m}' - \vom,   \mathsf{r}' - \vor,  \mathsf{l}' - \vol, \mathsf{v}' - \vov, \mathsf{b}' - \vob,  \mathsf{s}' - \vos}_{\mathrm{d3sc,3co,res}}} \leq C ( \|  B u \|_{H^{ \mathsf{m}', \mathsf{r}', \mathsf{l}' ,  \mathsf{v}',  \mathsf{b}',  \mathsf{s}'}_{\mathrm{d3sc,3co,res}}} + \| u \|_{H^{M, N,  \mathsf{l}' - \vol, K, \mathsf{b}' - \vob, S}_{\mathrm{d3sc,3co,res}}}  )
\end{equation}
in the strong sense that if the right hand side of (\ref{uniform elliptic regularity estimate symbolic}) is finite, then so is the left hand side, and the estimate holds.
\end{lemma}

Here, the statement of Lemma \ref{uniform elliptic regularity estimate lemma} requires the definition of uniform wavefront set, which we had defined in Definition \ref {definition uniform operator wavefront set}. We emphasize again that Lemma \ref{uniform elliptic regularity estimate lemma} will only be relevant for the proofs of the propagation estimates in \S\S \ref{principal type propagation section}--\ref{global radial point estimate section} below.

\section{The microlocal structure for the three-body Helmholtz operator} \label{section Hamiltonian structure of the Helmhotz operator}
    
    Starting from this section, we will begin analyzing microlocally the three-body Helmholtz operator $P$. Since $P \in \Psi^{2,-2,0}_{\mathrm{3scc}}( \tscX )$, we will first work in the three-body phase space $\overline{^{\mathrm{3sc}}T^{\ast}} \tscX$ (instead of directly the second-microlocal framework). The central objects to be looked at in this setting are thus the three-body principal symbol and indicial operators of $P$. The former object can be written in Euclidean coordinates $(z_{\tindex}, z^{\tindex}, \zeta_{\tindex}, \zeta^{\tindex})$ as
    \begin{equation} \label{principal symbol of P in Euclidean coordinates}
    {^{\mathrm{3sc}}\sigma(P)} = |\zeta_{\tindex}|^2 + |\zeta^{\tindex}|^{2} - \lambda^{2}.
    \end{equation}
However, while (\ref{principal symbol of P in Euclidean coordinates}) is a global expression, it is often useful to also consider it locally in a neighborhood of $^{\mathrm{3sc}}T^{\ast}_{\mf \cap  \ff} \tscX$, in which case one could also write
\begin{equation*} \label{the three-body principal symbol of P}
    {^{\mathrm{3sc}}\sigma}(P) = |\ltau|^2 + |\lmu|_{h_{\tindex}}^2 + |\utau|^2 + |\umu|_{h^{\tindex}}^2 - \lambda^2
\end{equation*}
in coordinates $(\bff, y_{\tindex}, y^{\tindex}, \ltau, \lmu, \utau, \umu) \in {^{\mathrm{3sc}}T^{\ast}_{\mf} \tscX}$. \par

Likewise, the three-body indicial operator can be written globally over $\mathcal{C}_{\tindex} \times \mathbb{R}^{n_{\tindex}}_{\zeta_{\tindex}}$ as
\begin{equation*} 
    {^{\mathrm{3sc}}\hat{N}}_{\ff}( P ) = \Delta_{z^{\tindex}} + V^{\tindex} - ( \lambda^{2} - |\zeta_{\tindex}|^2 ),
\end{equation*}
though it is often more useful to work with local coordinates $( y_{\tindex}, z^{\tindex}, \ltau, \lmu, \zeta^{\tindex} ) \in \mathcal{C}_{\tindex} \times \mathbb{R}^{n_{\tindex}}$, in which case we also write
\begin{equation} \label{the three-body indicial operator of P}
    {^{\mathrm{3sc}}\hat{N}}_{\ff}( P ) = \Delta_{z^{\tindex}} + V^{\tindex} - ( \lambda^2 -  |\ltau|^{2} - |\lmu|_{h_{\tindex}}^2 ).
\end{equation}

\subsection{Characteristic sets of $P$ in the three-body framework} 
\label{subsection characteristic variety of P in the three-body framework}
It is easy to see that the symbolic three-body characteristic variety of $P$, $\mathrm{Char}_{\mathrm{3sc}, \sigma}(P)$, is again contained away from fiber infinity. Moreover, it is the same as the two-body characteristic variety of $\Delta - \lambda^{2}$ away from ${^{\mathrm{3sc}}T^{\ast}_{\ff}} \tscX$, while in a small neighborhood of ${^{\mathrm{3sc}}T^{\ast}_{\ff} \tscX}$, it can be written in local coordinates as
\begin{equation*}
 \{ ( \bff, y_{\tindex}, y^{\tindex}, \ltau, \lmu, \utau, \umu ) \in {^{\mathrm{3sc}}T^{\ast}_{\mf} \tscX} :  |\ltau|^2 + |\lmu|_{h_{\tindex}}^2 + |\utau|^2 + |\umu|_{h^{\tindex}}^2 = \lambda^2   \}.
\end{equation*}  
For brevity, we will henceforth denote
\begin{equation*}
\Sigma_{\mf} \coloneq \mathrm{Char}_{\mathrm{3sc}, \sigma}( P )
\end{equation*}
where `mf' indicates that $\Sigma_{\mf}$ is compactly contained in ${ ^{\mathrm{3sc}}T^{\ast}_{\mf} } \tscX$. In fact, let 
\begin{equation*}
\beta_{\mathrm{3sc}} : \overline{^{\mathrm{3sc}}T^{\ast}} \tscX \rightarrow \overline{ ^{\mathrm{sc}}T^{\ast}} \Xo
\end{equation*}
be the natural blow-down map. Then we have
\begin{equation} \label{the main face characteristic variety as a lift}
\Sigma_{\mathrm{mf}} = \beta_{\mathrm{3sc}}^{\ast}( \Sigma_{\mathrm{sc}} ),
\end{equation}
where $\Sigma_{\mathrm{sc}} \subset \overline{ ^{\mathrm{sc}}T^{\ast} }_{\partial \Xo} \Xo$ is the characteristic set of $\Delta - \lambda^{2}$ as in the two-body case.
	 \par

    On the other hand,  we know that $\text{Ell}_{\mathrm{3sc},\ff}(P)$ is given by those elements of $\mathcal{C}_{\tindex} \times \mathbb{R}^{n_{\tindex}}$ such that (\ref{the three-body indicial operator of P}) is invertible with inverse in $\Psi_{\text{sc}}^{-2,0}( X^{\tindex} )$. It is well-known that $\Delta_{z^{\tindex}} + V^{\tindex} - \sigma, \sigma \in \mathbb{R}$ is invertible if and only if $\sigma < 0$ and $\sigma \notin \mathrm{Spec}_{p}( \Delta_{z^{\tindex}} + V^{\tindex} )$. Thus, if we assume for simplicity that there is no bound states, i.e., $\text{Spec}_{p}( {^{\mathrm{3sc}}\hat{N}_{\ff}} (P) ) = \emptyset$, then by standard result we know that $^{\mathrm{3sc}}\hat{N}_{\ff}(P)$ must be invertible if and only if $|\tau_{\alpha}|^{2} + |\mu_{\alpha}|_{h_{\alpha}}^{2} > \lambda^2$. Hence
    \begin{equation} \label{ff-characterstic variety of P}
    \Sigma_{\ff} \coloneq  \mathrm{Char}_{\mathrm{3sc},\ff}(P) = \{ ( y_{\tindex}, \ltau, \lmu ) \in \mathcal{C}_{\tindex} \times \mathbb{R}^{n_{\tindex}} :  |\tau_{\alpha}|^{2} + |\mu_{\alpha}|^{2} \leq \lambda^{2} \}
    \end{equation} 
    will be the three-body characteristic set of $P$ at $\ff$. 
    \par 
    Following \cite[Chapter 11]{AndrasThesis}, it will be useful to split (\ref{ff-characterstic variety of P}) into two components, namely a `normal' part
    \begin{equation*}
    \Sigma_{\mathrm{n}} \coloneq \{ ( y_{\tindex}, \ltau, \mu )  \in \mathcal{C}_{\tindex} \times \mathbb{R}^{n_{\tindex}} :  |\tau_{\alpha}|^{2} + |\mu_{\alpha}|_{h_{\alpha}}^{2} < \lambda^{2}  \}
    \end{equation*}
    and a `tangential' part
    \begin{equation} \label{tangential part of the characteristic variety}
    \Sigma_{\mathrm{t}} \coloneq \{ ( y_{\tindex}, \ltau, \lmu ) \in \mathcal{C}_{\tindex} \times \mathbb{R}^{n_{\tindex}} : |\tau_{\alpha}|^{2} + |\mu_{\alpha}|_{h_{\alpha}}^{2} = \lambda^{2}  \}.
    \end{equation}
    Clearly $\Sigma_{\text{ff}_{\alpha}} = \Sigma_{\mathrm{n}} \cup \Sigma_{\mathrm{t}}$ is a disjoint union. 
    \subsection{Rescaled Hamiltonian in the three-body framework}
    We expect microlocal regularity of solutions for $Pu = f$ to propagate along the three-body bicharacterstic flows of $P$, provided sufficient regularity of $f$ is assumed. In the cae of $V$ being a two-body, short-ranged potential, this is often done by studying the \emph{Legendrian} perspective to the flows of $P$, which can be understood as projections of the usual, interior flowlines to the boundary of $^{\mathrm{sc}}T^{\ast}\Xo$. See \cite{MelroseMaciej96} for further details.  \par

    In Euclidean coordinates $(z_{\alpha}, z^{\alpha}, \zeta_{\alpha}, \zeta^{\alpha})$, the Hamiltonian vector field of $p$ is given by
    \begin{equation} \label{Euclidean Hamiltonian vector field}
    H_{p} = 2 \zeta_{\alpha} \cdot \partial_{z_{\alpha}} + 2 \zeta^{\alpha} \cdot \partial_{z^{\alpha}}.
    \end{equation}
    If we switch coordinates to $( x_{\alpha}, y_{\alpha}, z^{\alpha}, \tau_{\alpha}, \mu_{\alpha}, \zeta^{\alpha} )$, then $H_{p}$ instead takes the form
    \begin{equation}  \label{Hamiltonian in interior front face coordinates}
    H_{p} = H_{ \intt } + 2 \zeta^{\alpha} \cdot \partial_{z^{\alpha}}.
    \end{equation}
    where
    \begin{equation} \label{Hamiltonian in interior front face coordinates 0.1}
    H_{ \intt } =  2 \tau_{\alpha} x_{\alpha} ( x_{\alpha} \partial_{x_{\alpha}} + R_{\mu_{\alpha}} ) - 2 x_{\alpha} |\mu_{\alpha}|_{h_{\alpha}}^2 \partial_{\tau_{\alpha}} + x_{\alpha} H_{ |\lmu|^{2}_{h_{\tindex}} }.
    \end{equation}
Here $R_{\lmu} = \lmu \cdot \partial_{\lmu}$ and $H_{|\lmu|_{h_{\tindex}}^2} =  \partial_{\lmu} |\lmu|_{h_{\tindex}}^2 \cdot \partial_{y_{\tindex}} - \partial_{y^{\tindex}} | \lmu |_{h_{\tindex}}^2 \cdot \partial_{\lmu}$. \par

Next, we want to write (\ref{Hamiltonian in interior front face coordinates}) in the coordinates $( \bff, y_{\alpha}, x^{\alpha}, y^{\alpha}, \tau_{\alpha}, \mu_{\alpha}, \tau^{\alpha}, \mu^{\alpha})$.  Upon doing so, the $\partial_{z^{\alpha}}$ terms will also be dependent on $\partial_{\bdf}$ after applying the transformation laws. Thus in these coordinates, and with some abuse of notation, we can write  
    \begin{equation} \label{Hamiltonian in interior front face coordinates 0.2}
    H_{p} = H_{ \intt } + H_{ \intn } - 2 \tau^{\alpha} x_{\alpha} \partial_{\bff}
    \end{equation}
where, as in (\ref{Hamiltonian in interior front face coordinates}), we are denoting
    \begin{equation} \label{Hamiltonian form near the boundary of front face}
    H_{ \intn } =  2 \tau^{\alpha} x^{\alpha} ( x^{\alpha} \partial_{x^{\alpha}} + R_{\mu^{\alpha}} ) - 2 x^{\alpha} |\mu^{\alpha}|_{h^{\alpha}}^2 \partial_{\tau^{\alpha}} + x^{\alpha} H_{|\umu|_{h^{\tindex}}^2}.
    \end{equation}
Here $R_{\umu} = \umu \cdot \partial_{\umu}$ and $H_{|\umu|_{h^{\tindex}}^2} = \partial_{\umu} |\umu|_{h^{\tindex}}^2 \cdot \partial_{y^{\tindex}} - \partial_{y^{\tindex}} |\umu|_{h^{\tindex}}^2 \cdot \partial_{\umu}$, and the abuse of notation is that the first term in (\ref{Hamiltonian in interior front face coordinates 0.2}) still denotes (\ref{Hamiltonian in interior front face coordinates 0.1}), except that we now understand
    \begin{equation*}
    2 \tau_{\tindex} x_{\tindex}^2 \partial_{x_{\tindex}} =  2\tau_{\tindex} \bdf^2 x^{\tindex} \partial_{\bff}
    \end{equation*}
    which is a result of the variable change. \par

    Following the two body literature, it is customary to rescale $H_{p}$ to first order by the global boundary defining function of $\overline{\mathbb{R}^{n}}$, i.e., the introduction of
    \begin{equation} \label{two-body hamiltonian}
    \mathsf{H}_{p} \coloneq x^{-1} H_{p}|_{x = 0}
    \end{equation}
    where $x = |z|^{-1}$. We will also do this in the free and interactive variables separately:
    \begin{equation*}
    \mathsf{H}_{ \intt } \coloneq  x_{\tindex}^{-1} H_{ \intt } |_{x_{\tindex} = 0}, \quad  \mathsf{H}_{ \intn } \coloneq (x^{\tindex})^{-1} H_{ \intn } |_{x^{\tindex} = 0}.
    \end{equation*} 
    However, since $x \simeq x_{\tindex}$ near $^{\mathrm{3sc}}T^{\ast}_{\ff} \tscX$, from (\ref{Hamiltonian in interior front face coordinates}) and (\ref{Hamiltonian form near the boundary of front face}), we can see that $\mathsf{H}_{p}$ will not be a smooth vector field over ${^{\mathrm{3sc}}T^{\ast}_{\mf}} \tscX$, as there will be a singularity of order one as $\bff \rightarrow 0$ (i.e., at ${^{\mathrm{3sc}}T}^{\ast}_{\ff} \tscX$). This suggests that we should instead consider the alternate rescaling  
    \begin{equation} \label{definotion of main face Hamiltonian}
    \mathsf{H}_{p,\text{mf}} \coloneq (x^{\tindex})^{-1} H_{p} |_{x^{\tindex} = 0}.
     \end{equation}
     Notice that 
    \begin{equation} \label{rescalled Hamiltonian flow at main face}
(x^{\tindex})^{-1} H_{p} = 2 \utau x^{\tindex} \partial_{x^{\tindex}} + \scuH + \bff \sclH + 2 ( \bff \ltau - \utau ) \bff \partial_{ \bff },
    \end{equation}
   so it follows that $\sH_{p,\mf}$ must restrict to a $\mathcal{V}_{\mathrm{b}}( ^{\mathrm{3sc}}T^{\ast}_{\mf} \tscX )$ vector field. In particular, the integral curves of $\mathsf{H}_{p,\text{mf}}$ over $^{\mathrm{3sc}}T^{\ast}_{\text{mf}} \tscX$ are defined for all time.
    \subsection{Radial points in the three-body framework}
    We next investigate the set of points in $\Sigma_{\mf}$ for which 
    \begin{equation} \label{stationary point for three-body hamiltonian flow}
    \mathsf{H}_{p,\text{mf}} = 0.
    \end{equation}
 These stationary points are useful, since we shall see that any integral curve of $\mathsf{H}_{p,\text{mf}}$ would always converge to such a point in either forward or backward time asymptotics. \par

    First, recall that by using arguments similar to those in the two-body settings, we have
    \begin{gather}
        \mathsf{H}_{ \intt } = 0   \Leftrightarrow \tau_{\tindex} = \pm ( \lambda^{2} - |\utau|^2 - |\umu|_{h^{\tindex}}^2  )^{1/2}, \ \mu_{\tindex} = 0,  \label{stationary condition for free Hamiltonian} \\
            \mathsf{H}_{ \intn } = 0 \Leftrightarrow \tau^{\alpha} = \pm ( \lambda^{2} - |\ltau|^2 - |\lmu|_{h_{\tindex}}^2 )^{1/2}, \ \mu^{\tindex} = 0.  \label{stationary condition for interactive Hamiltonian}
    \end{gather}
   Notice that if $\mathsf{H}_{p,\text{mf}} = 0$, then the right hand side of (\ref{stationary condition for interactive Hamiltonian}) must be true, since it is possible that $\bff \mathsf{H}_{ \intt } = 0$ if either the right hand side of (\ref{stationary condition for free Hamiltonian}) is satisfied or $\bff = 0$. \par

    If we insist that $\bff \neq 0$ (i.e., if we insist that we are away from $ ^{\mathrm{3sc}}T^{\ast}_{\mf} \tscX $), then even if both (\ref{stationary condition for free Hamiltonian}) and (\ref{stationary condition for interactive Hamiltonian}) are satisfied, we would still have the surviving term 
    \begin{equation} \label{surviving term in Hamiltonian vector field}
    2 \bff ( \tau_{\tindex} \bff - 2 \tau^{\tindex} ) \partial_{\bff}
    \end{equation}
    in (\ref{rescalled Hamiltonian flow at main face}). For (\ref{surviving term in Hamiltonian vector field}) to vanish, we would need to have
    \begin{equation} \label{stationary condition away from front face based on rho 1}
    \bff = \frac{\tau^{\alpha}}{\tau_{\alpha}} > 0,
    \end{equation}
    thus $\tau_{\alpha}$ and $\tau^{\alpha}$ must have the same signs. Hence, we see that there are stationary points for $\mathsf{H}_{p,\text{mf}}$ in $\Sigma_{\text{mf}}$ defined by $\mu_{\tindex} = 0$, $\mu^{\tindex} = 0$, where $\tau_{\tindex} \neq 0$ and $\tau^{\tindex} \neq 0$ have the same signs, and $\bff$ is determined by (\ref{stationary condition away from front face based on rho 1}). In fact, by using the constraint $|\tau_{\tindex}|^{2} + | \tau^{\tindex} |^{2} = \lambda^{2} $, one could also parameterize $\tau_{\tindex}$ and $\tau^{\tindex}$ by 
    \begin{equation} \label{large time condition for tau away from front face}
    \tau_{\tindex} = \pm \frac{1}{ ( 1 + \bff^2  )^{1/2} } \lambda, \quad \tau^{\tindex} = \pm \frac{ \bff }{ ( 1  + \bff^2  )^{1/2} } \lambda.
    \end{equation}
    Conversely, if we assume that $\lmu = 0$, $\umu = 0$, and also (\ref{stationary condition away from front face based on rho 1}), (\ref{large time condition for tau away from front face}) are satisfied, then $\sH_{p, \mf}$ must vanish as well. \par

    Let 
    \begin{equation*}
    \mathcal{R}_{\text{sc},\pm} \subset { ^{\mathrm{sc}}T^{\ast}_{\partial \Xo} \Xo }
    \end{equation*}
    be the usual radial points in the two-body scattering theory, which is defined by the set of all such points such that 
    \begin{equation*}
    \mathsf{H}_{p} = 0,
    \end{equation*}
    where $\sH_{p}$ is defined in (\ref{definotion of main face Hamiltonian}). For $\bff \neq 0$, these are precisely those points satisfying (\ref{stationary point for three-body hamiltonian flow}). Hence in the three-body setting, it would be natural to consider the lift
    \begin{equation*} \label{two-body radial set in three-body setting}
    \mathcal{R}_{\text{2sc,mf},\pm} \coloneq \beta_{\text{3sc}}^{\ast} ( \mathcal{R}_{\text{sc},\pm} ),
    \end{equation*}
    where $\beta_{\mathrm{3sc}} : { ^{\mathrm{3sc}}T^{\ast} \tscX } \rightarrow { ^{\mathrm{sc}}T^{\ast}\Xo }$ is the natural blow-down map. Then in a small coordinate neighborhood of $^{\mathrm{3sc}}T^{\ast}_{\ff} \tscX$, $\mathcal{R}_{\mathrm{2sc}, \mf, \pm}$ can be characterized by
    \begin{equation} \label{writing down the 3sc radial set in local coordinates}
    \mathcal{R}_{\text{2sc,mf},\pm} = \Big\{ \tau_{\alpha} = \pm \frac{1}{ ( 1 + \bff^2  )^{1/2} } \lambda, \, \mu_{\alpha} = 0, \,  \tau^{\alpha} = \pm \frac{ \bff }{ ( 1  + \bff^2  )^{1/2} } \lambda, \, \mu^{\alpha} = 0  \Big\}.
    \end{equation}
    \par

    Now, even if we let $\bff \rightarrow 0$, in which cases points in $\mathcal{R}_{\text{2sc,mf},\pm}$ must take the form
    \begin{equation} \label{main face radial points restricted to the front face}
    \tau_{\tindex} = \pm \lambda, \ \mu_{\tindex} = 0, \ \tau^{\tindex} = 0, \ \mu^{\tindex} = 0,
    \end{equation}
    we would still not have included all the stationary points of $\mathsf{H}_{p,\text{mf}}$. Indeed, if we instead impose the condition that $\bff = 0$ to begin with, then there are also stationary points for $\mathsf{H}_{p,\text{mf}}$ in $\Sigma_{\text{mf}}$ determined by the conditions
    \begin{equation*}
     \bff = 0, \ \tau^{\alpha} = \pm ( \lambda^{2} - |\ltau|^2 - |\lmu|_{h_{\tindex}}^2  )^{1/2}, \ \mu^{\alpha} = 0.
    \end{equation*}
    Thus, there is a new kind of radial points in this case, which we denote by
    \begin{equation} \label{local coordinates representation of Rnpm}
    \mathcal{R}_{\text{n},\pm} \coloneq \{  \bff = 0, \, \tau^{\alpha} = \pm ( \lambda^{2} - |\ltau|^2 - |\lmu|_{h_{\tindex}}^2  ), \, \mu^{\alpha} = 0 \},
    \end{equation} 
    where the subscript `$\mathrm{n}$' again stands for `normal' (which should really be transversal, though we follow the convention established in \cite[Chapter 11]{AndrasThesis} here and reserve the subscript `t' for `tangential') to reflect on the direction of integral curves when they meet $\mathcal{R}_{\text{n},\pm}$. 
    \par
  Finally, we will also introduce radial points in the tangential directions, i.e.,
    \begin{equation*}
    \mathcal{R}_{\text{t},\pm} \coloneq \{ \tau_{\tindex} = \pm \lambda,  \, \mu_{\tindex} = 0 \} \subset \mathcal{C}_{\tindex} \times \mathbb{R}^{n_{\tindex}}.
    \end{equation*}
    Notice that this is the direct analogy of the two-body radial sets if one views $\mathcal{C}_{\tindex} \times \overline{\mathbb{R}^{n_{\tindex}}}$ as the base infinity of the scattering cotangent bundle $\overline{ ^{\mathrm{sc}}T^{\ast} }X_{\tindex} \cong X_{\tindex} \times \overline{\mathbb{R}^{n_{\tindex}}}$.

    \subsection{Characterization for the flow of $\sH_{p,\mf}$}
    \label{subsection: characterization for the flow}
    Let us now compute the flow of $\mathsf{H}_{p,\text{mf}}$ on ${^{\mathrm{3sc}}T^{\ast}_{\text{mf}}} \tscX$. Since we only have the form (\ref{Hamiltonian form near the boundary of front face}) near ${^{\mathrm{3sc}}T^{\ast}_{\text{ff}_{\alpha}}} \tscX$, we will only consider the parts of the integral curves $\gamma$ which are contained in a small neighborhood of ${^{\mathrm{3sc}}T^{\ast}_{\text{ff}_{\tindex}} } \tscX$. In fact, we are most interested in those $\gamma$ whose trajectories will have non-empty intersection with ${^{\mathrm{3sc}}T^{\ast}_{\text{ff}_{\tindex}} } \tscX$ (the ones which do not intersect ${^{\mathrm{3sc}}T^{\ast}_{\text{ff}_{\tindex}} } \tscX$ can be identified as integral curves of $\sH_{p}$ as in the two-body case).  \par 
    \par 
    First, let us note that the flow of $\bff$ is determined by 
    \begin{equation} \label{flow equation for rho}
    \frac{d \bff}{dt} =  2 \bff ( \tau_{\tindex} \bff - \tau^{\tindex} ) = 2 \tau_{\tindex} \bff^2 - 2 \tau^{\tindex} \bff.
    \end{equation}
    Thus, we conclude that
    \begin{equation} \label{limiting behavour of integral curves}
    \begin{gathered}
    \text{the only possibilities for an integral curve $\gamma$ to intersect ${^{\mathrm{3sc}}T^{\ast}_{\partial \text{ff}_{\tindex}}} \tscX$ are either} \\ 
    \text{$\bdf(t) = 0$ identically or $\lim_{t \rightarrow \pm \infty} \bdf (t) = 0$,}
    \end{gathered}
    \end{equation}
    since otherwise we contradict uniqueness for the initial value problem. \par

    Generally, it is possible for an integral curve $\gamma(t)$ of $\mathsf{H}_{p,\text{mf}}$ to stay in a fixed small neighborhood of ${^{\mathrm{3sc}}T^{\ast}_{\text{ff}_{\alpha}} } \tscX$ for either
    \begin{equation*}
    t \in (a, \infty], \ a \in [-\infty, \infty) \ \text{or} \ t \in [-\infty, b), \ b \in (-\infty, \infty].
    \end{equation*}
    Hence, we will only consider the restrictions of $\gamma$ to such intervals. Often in our analysis, we are only interested in the cases of very large or small $t$, thus it will be convenient to assume that $a$ (resp. $b$) is very large (resp. small) as well. For definiteness, in the discussion below we will only look at the case $t \in (a, \infty]$ (i.e., large $t$) unless otherwise mentioned. 
    \par

    Now, let us observe that the evolution in the variables $(y^{\tindex}, \tau^{\tindex}, \mu^{\tindex})$ is independent of all the other variables. In particular, these are just the flow of $\mathsf{H}_{ \intn }$, so that the usual `two-body argument' applies. See \cite[\S 3]{MelroseMaciej96} for the discussion in the two-body setting. Here, we just recall that the flow equations for these variables when $\mu^{\tindex} \neq 0$ are
    \begin{equation} \label{flow equations subsystem} 
    \begin{gathered}
    \frac{dy^{\tindex}}{dt} = \frac{\partial |\hat{\mu}^{\tindex}|_{h^{\tindex}}^2}{\partial \mu^{\tindex}} | \mu^{\tindex} |_{h^{\tindex}}, \ \frac{d \hat{\mu}^{\tindex}}{dt} = - \frac{\partial |\hat{\mu}^{\tindex}|_{h^{\tindex}}^2}{\partial y^{\tindex}} |\mu^{\tindex}|_{h^{\tindex}}, \\
    \frac{d\tau^{\tindex}}{dt} = -  2| \mu^{\tindex} |_{h^{\tindex}}^2, \ \frac{d|\mu^{\tindex}|_{h^{\tindex}}}{dt} = 2 \tau^{\tindex}  |\mu^{\tindex}|_{h^{\tindex}}^2
    \end{gathered}
    \end{equation}
    where $\hat{\mu}^{\tindex} \coloneq \mu^{\tindex}/ |\mu^{\tindex}|_{h^{\tindex}}$. Thus, with the reparametrisation of time
    \begin{equation*}
    s^{\tindex}(t) \coloneq 2 \int_{0}^{t} |\mu^{\tindex}|_{h^{\tindex}} \hspace*{0.5mm} dt', \quad t \in (a, \infty)
    \end{equation*}
    we conclude that
    \begin{equation} \label{equations 1 for subsystem integral flows}
    \begin{gathered}
    \tau^{\tindex}  = ( \lambda^2 - |\ltau|^2 - |\lmu|_{h_{\tindex}}^2 )^{1/2} \cos( s^{\tindex} + s^{\tindex}_0 ), \\ 
    \mu^{\tindex}  = ( \lambda^{2} - |\ltau|^2 - |\lmu|_{h_{\tindex}}^2 )^{1/2} \sin(s^{\tindex} + s^{\tindex}_{0}) \hat{\mu}^{\tindex}
    \end{gathered}
    \end{equation}
    and
    \begin{equation} \label{equations 2 for subsystem integral flows}
    (y^{\tindex}, \hat{\mu}^{\tindex}) = \exp{ ( (s^{\tindex} + s^{\tindex}_{0}) H_{ |\umu|_{h^{\tindex}}^2 /2} ) }(y^{\tindex}_{0}, \hat{\mu}^{\tindex}_{0}) 
    \end{equation}
    are the solutions. Here the range of $s^{\tindex}$ is given for $s^{\tindex}_{a} \coloneq s^{\tindex}(a)$ by 
    \begin{equation*}
    s^{\tindex}_{0} \in [0, \pi] \ \text{and} \ s^{\tindex} \in [ s^{\tindex}_{a} , \pi - s^{\tindex}_{0}] \subseteq [ - s^{\tindex}_{0}, \pi - s^{\tindex}_{0} ]
    \end{equation*}
    while $( y^{\tindex}_{0}, \hat{\mu}^{\tindex}_{0})$ are constants. For $\mu^{\tindex} = 0$, there are also stationary conditions for the coordinates $(y^{\tindex}, \tau^{\tindex}, \mu^{\tindex})$ given by (\ref{stationary condition for free Hamiltonian}), which would define partially the coordinate representation of an integral curve of $\mathsf{H}_{p,\text{mf}}$. Notice here that $( \lambda^2 - |\ltau|^2 - |\lmu|_{h_{\tindex}}^2 )^{1/2}$ is conserved in time.  \par

    Evolution in the variables $( y_{\tindex}, \tau_{\tindex}, \mu_{\tindex} )$ is similar but slightly involved, requiring more careful considerations. The flow equations for $\mu_{\tindex} \neq 0$ are now
    \begin{equation} \label{flow equations tangential}
    \begin{gathered}
    \frac{d y_{\tindex}}{d t} = \bff \frac{ \partial |\hat{\mu}_{\tindex}|_{h_{\tindex}^2 }}{\partial \mu_{\tindex}} |\mu_{\tindex}|_{h_{\tindex}}, \ \frac{d\hat{\mu}_{\tindex}}{dt} = - \bff \frac{\partial |\hat{\mu}_{\tindex}|_{h_{\tindex}}^2}{\partial y_{\tindex}} |\mu_{\tindex}|_{h_{\tindex}}, \\
    \frac{d \tau_{\tindex}}{dt} = - 2 \bff  |\mu_{\tindex}|_{h_{\tindex}}^2, \ \frac{d|\mu_{\tindex}|_{h_{\tindex}}}{dt} = 2 \bff \tau_{\tindex}  | \mu_{\tindex} |_{h_{\tindex}}^2
    \end{gathered}
    \end{equation}
    where $\hat{\mu}_{\tindex} = \mu_{\tindex}/ |\mu_{\tindex}|_{h_{\tindex}}$. Assuming also $\bdf > 0$, then just as before, equations (\ref{flow equations tangential}) suggest that we should introduce the reparameterization
    \begin{equation} \label{variable change along C}
    s_{\alpha}(t) \coloneq 2 \int_{0}^{t} \bdf |\mu_{\tindex}|_{h_{\tindex}} \hspace*{0.5mm} dt', \quad t \in (a,  \infty ).
    \end{equation}
    Such a reparameterization will reduce (\ref{flow equations tangential}) to a system of equations which are similar to those in (\ref{flow equations subsystem}). Namely, we will have
    \begin{equation}  \label{Hamiltonian flow along C}
    \begin{gathered}
    \tau_{\tindex}  = ( \lambda^2 - |\ltau|^2 - |\umu|_{h^{\tindex}}^2 )^{1/2} \cos( s_{\tindex} + s_{\tindex,0}  ), \\
    \mu_{\tindex}  = ( \lambda^{2} - |\utau|^2 - |\umu|_{h^{\tindex}}^2 )^{1/2} \sin(s_{\tindex} + s_{\tindex,0} ) \hat{\mu}_{\tindex}
    \end{gathered}
    \end{equation}
    and 
    \begin{equation*}
    (y_{\tindex}, \hat{\mu}_{\tindex}) = \exp{ ( (s_{\tindex} + s_{\tindex,0} ) H_{ |\lmu|_{h_{\tindex}}^2 /2} ) }( y_{\tindex,0}, \hat{\mu}_{\tindex,0} ),
    \end{equation*}
    but they are only valid for even more restricted values of $s_{\tindex}$, i.e., with $s_{\tindex,a} \coloneq s_{\tindex}(a)$ that
    \begin{equation} \label{range characterisation for tangential flow}
    s_{\tindex,0} \in [0,\pi] \ \text{and} \ s_{\tindex} \in [s_{\tindex, a}, s_{\tindex,+} ] \subseteq [ - s_{\tindex,0}, \pi - s_{\tindex,0} ],
    \end{equation}
    while $(  y_{\tindex,0}, \hat{\mu}_{\tindex,0} )$ are again constants. Here $s_{\alpha,+} \coloneq \lim_{t \rightarrow \infty} s_{\tindex}(t)$ is generally undetermined. Likewise, if $\mu_{\tindex} = 0$, then the stationary conditions (\ref{stationary condition for interactive Hamiltonian}) would partially define the coordinate representation of an integral curve of $\mathsf{H}_{p,\text{mf}}$. \par 
    To determine the range of $s_{\alpha}$, it is useful to also introduce 
    \begin{equation}
    \tilde{t} \coloneq \int_{0}^{t} \bdf dt', \quad t \in[a , \infty)
    \end{equation}
    for the case $\gamma \not \subset {^{\mathrm{3sc}}T^{\ast}_{\partial \text{ff}_{\tindex}}} \tscX$. Notice from (\ref{definotion of main face Hamiltonian}) that we actually have
    \begin{equation} \label{reparametrised 2sc Hamiltonian vector field}
    \mathsf{H}_{p,\text{mf}} = \bdf \widetilde{\mathsf{H}}_{p}
    \end{equation}
    where
    \begin{equation*}
    \widetilde{\mathsf{H}}_{p} \coloneq  (1 + \bff^2)^{-1/2} \mathsf{H}_{p}.
    \end{equation*}
Since $( 1 + \bff^{2} )^{-1/2}$ is a smooth, non-vanishing function near $^{\mathrm{3sc}}T^{\ast}_{\ff} \tscX$, we can conclude that integral curves of $\widetilde{\mathsf{H}}_{p}$ are locally time-reparameterisation of the integral curves of $\mathsf{H}_{p}$, and in particular must have the same trajectories. Observe also that equation (\ref{reparametrised 2sc Hamiltonian vector field}) implies $\mathsf{H}_{p}$ is singular at ${^{\mathrm{3sc}}T^{\ast}_{\text{ff}_{\tindex}}} \tscX$. Hence the integral curves of $\mathsf{H}_{p}$ (and therefore those of $\widetilde{\mathsf{H}}_{p}$) will in general not be defined for all times.  \par

    The local flow for $\widetilde{\mathsf{H}}_{p}$ can be used to study the intersection of $\gamma$ with $^{\mathrm{3sc}}T^{\ast}_{ \partial \text{ff}_{\tindex}}[ \overline{\mathbb{R}^{n}} ; \mathcal{C} ]$ more effectively, since by parameterizing time via $\tilde{t}$, we can rewrite (\ref{flow equation for rho}) as 
    \begin{equation} \label{reparametrised time equation fo rho}
    \frac{d \bdf }{ d  \tilde{t} } =2 \tau_{\tindex} \bdf  -2 \tau^{\tindex}  ,
    \end{equation}
    which we can solve directly to obtain
    \begin{equation} \label{solution to first order equation for rho}
    \bdf = \ibdf e^{ 2 \int_{0}^{\tilde{t}} \tau_{\alpha} d \tilde{t}' } - 2 e^{2 \int_{0}^{ \tilde{t} } \tau_{\tindex} d \tilde{t}' } \int_{0}^{ \tilde{t} } e^{-2\int_{0}^{ \tilde{t}' }  \tau_{\tindex}  d \tilde{t}'' } \tau^{ \tindex }  d \tilde{t}',
    \end{equation}
    where $ \ibdf \coloneq \bdf(0) $ is the initial value. \par

    We can also relate $s_{\tindex}$ with $\tilde{t}$ via
    \begin{equation} \label{2s Hamiltonian flow time parameterisation}
    s_{\alpha} \coloneq  2 \int_{0}^{ \tilde{t} } |\mu_{\alpha}|_{h_{\alpha}} \hspace{0.5mm} d \tilde{t}'
    \end{equation}
    for $ \tilde{t} \in [0, T_{+}]$, where we have let
    \begin{equation} \label{definition of T_{+}}
    T_{+} \coloneq \lim_{t \rightarrow \infty} \tilde{t}.
    \end{equation}
    Thus the maximal value for which $s_{\alpha}$ could take will depend on whether or not $T_{+} = \infty$. In particular, it holds that $s_{\tindex,+} = \pi - s_{\tindex,0}$ in (\ref{range characterisation for tangential flow}) if and only if $T_{+} = \infty$. \par 
    For the fixed $a$ chosen above, we will also define 
    \begin{equation} \label{definition of T_{a}}
    T_{a} \coloneq \lim_{t \rightarrow a+} \tilde{t}.
    \end{equation}
    Hence the range of $\tilde{t}$ from $[a, \infty]$ is $[T_{a}, T_{+}]$.
    \par

    We now determine the precise conditions required for an integral curve $\gamma$ of $\sH_{p,\mf}$ in $\Sigma_{\mf}$ to have non-empty intersection with $^{\mathrm{3sc}}T^{\ast}_{ \partial \text{ff}_{\tindex}} \tscX$, provided that $\gamma \not\subset {^{\mathrm{3sc}}T^{\ast}_{ \partial \text{ff}_{\tindex}}} \tscX$.
    \begin{lemma} \label{lemma characterisation of flow intersecting front face}
    Let $\gamma$ be an integral curve of $\mathsf{H}_{p,\mathrm{mf}}$ in $\Sigma_{\mathrm{mf}}$ such that $\bff \leq \delta$ uniformly over $t \in (a,\infty)$ for some $\delta > 0$ small and possibly $a = - \infty$. Assume that
    \begin{equation} \label{assumption for intersection with the front faces}
    \lim_{t \rightarrow \infty} \gamma(t) \in {^{\mathrm{3sc}}T^{\ast}_{ \partial \mathrm{ff}_{\tindex}}} \tscX, \quad \gamma \not\subset {^{\mathrm{3sc}}T^{\ast}_{\partial \mathrm{ff}_{\tindex}}} \tscX.
    \end{equation}
    We will split our consideration into two cases: 
    \begin{itemize}
    \item Suppose that $\bff > 0$, and it holds that
       \begin{equation} \label{zeo energy condition 1 for integral curve to intersect ff}
    \text{$\tau^{\tindex} = 0$, $\mu^{\tindex} = 0$ for all $t \in (a ,\infty)$.} 
    \end{equation}
    Then we must have
    \begin{equation} \label{zerp energy condition 2 for integral curve to intersect ff}
    \begin{gathered}
    \text{$T_{+} = \infty$, $\tau_{\tindex} \neq \lambda$ and $\lim_{t \rightarrow \infty} \tau_{\tindex} = - \lambda$,} \\
    \text{with the possibility that $\ltau = - \lambda$ identically.} 
    \end{gathered}
    \end{equation} 
    Conversely, if (\ref{zeo energy condition 1 for integral curve to intersect ff}) and (\ref{zerp energy condition 2 for integral curve to intersect ff}) hold, then (\ref{assumption for intersection with the front faces}) must hold as well. \vspace{2mm}
    \item Suppose that $\bff > 0$, and (\ref{zeo energy condition 1 for integral curve to intersect ff}) does not hold. Then we must have
    \begin{equation} \label{condition for integral curve to intersect ff}
    \tau^{\tindex} = ( \lambda^{2} - |\ltau|^2 - |\lmu|_{h_{\tindex}}^2 )^{1/2}, \ \mu^{\tindex} = 0.
    \end{equation} 
    Moreover, suppose that $\utau \neq 0$. Then with $\rho_{\tindex,0} = \rho_{\tindex}(0)$, we have
    \begin{equation} \label{integral curve initial value condition for rho}
    \rho_{\tindex,0} = 2 \tau^{\tindex} \int_{0}^{T_{+}} e^{ -2 \int_{0}^{\tilde{t}'} \ltau d\tilde{t}'' } d\tilde{t}', \ T_{+} < \infty.
    \end{equation}
    Conversely, if (\ref{condition for integral curve to intersect ff}) and (\ref{integral curve initial value condition for rho}) hold with $\utau \neq 0$, then (\ref{assumption for intersection with the front faces}) must hold as well.
    \end{itemize}
    \end{lemma}
    \begin{proof}
    Throughout this proof, we will work mostly in time variable $\tilde{t} \in [ T_{a}, T_{+} ]$, where $T_{+}$, $T_{a}$ are respectively defined in (\ref{definition of T_{+}}) and (\ref{definition of T_{a}}). Furthermore, since we are only interested in the behaviour of large $\tilde{t}$, we will assume without loss of generality that $T_{a} \geq 0$. Notice that this corresponds to the choice of $a \geq 0$. \par

    We start by looking at the extreme case (\ref{zeo energy condition 1 for integral curve to intersect ff}). Thus in addition to (\ref{assumption for intersection with the front faces}), we will assume that $\tau^{\tindex} = 0$, $\mu^{\tindex} = 0$ for all $\tilde{t} \in [T_{a}, T_{+}]$. Then equation (\ref{reparametrised time equation fo rho}) is reduced to 
    \begin{equation} \label{reduced equation extreme case interactive variables zero energy}
    \bdf = \ibdf e^{2 \int_{0}^{\tilde{t}} \tau_{\tindex}  d\tilde{t}'}.  
    \end{equation} 
    If $\gamma$ intersects ${^{\mathrm{3sc}}T^{\ast}_{\partial \text{ff}_{\tindex}} } \tscX$ but is not entirely contained in ${^{\mathrm{3sc}}T^{\ast}_{\partial \text{ff}_{\tindex}} } \tscX$ (i.e., when (\ref{assumption for intersection with the front faces}) is satisfied), then by (\ref{limiting behavour of integral curves}) and the definition of $T_{+}$, we must have (\ref{reduced equation extreme case interactive variables zero energy}) going to $0$ as $\tilde{t} \rightarrow T_{+}$. In other words, we require that
    \begin{equation} \label{requirement for rho convergence to 0}
    \int_{0}^{T_{+}} \tau_{\tindex} d\tilde{t}' = -\infty.
    \end{equation}
    Now, by equations (\ref{Hamiltonian flow along C}) and the remarks which follows, we know that the only possibilities for $\tau_{\tindex}$ in this case are either $\tau_{\tindex} = \pm \lambda$ for all times, or $\lim_{\tilde{t} \rightarrow T_{+}} \tau_{\tindex} < \lambda$, with $\lim_{\tilde{t} \rightarrow T_{+}} \tau_{\tindex} = - \lambda$ if and only if $T_{+} = \infty$. From (\ref{requirement for rho convergence to 0}), we can conclude that $\tau_{\tindex} = \lambda$ is not a possibility. In fact, since $\tau_{\tindex}$ is bounded by $\lambda$, we must have
    \begin{equation*}
    - \lambda T_{+} \leq \int_{0}^{T_{+}} \tau_{\tindex} d\tilde{t}' \leq \lambda T_{+}.
    \end{equation*}
    Thus, the only possibility for (\ref{requirement for rho convergence to 0}) to be true would be if $T_{+} = \infty$, in which case $\lim_{\tilde{t} \rightarrow T_{+}} \tau_{\tindex} = - \lambda$, including the possibility of $\tau_{\tindex} = - \lambda$ identically. This proves one direction. Conversely, if (\ref{zeo energy condition 1 for integral curve to intersect ff}) and (\ref{zerp energy condition 2 for integral curve to intersect ff}) are true, then by (\ref{requirement for rho convergence to 0}) followed by (\ref{reduced equation extreme case interactive variables zero energy}), we easily conclude that $\lim_{\tilde{t} \rightarrow T_{+}} \bdf = 0$. Thus, (\ref{assumption for intersection with the front faces}) is satisfied as well. \par

    Next we consider the case where (\ref{zeo energy condition 1 for integral curve to intersect ff}) does not hold. We will show that (\ref{assumption for intersection with the front faces}) implies (\ref{condition for integral curve to intersect ff}). Arguing by contradiction, we will assume that $\lim_{\tilde{t} \rightarrow T_{+}} \bdf = 0$ but $\tau^{\tindex} \neq ( \lambda^{2} - |\ltau|^2 - |\lmu|_{h_{\tindex}}^2 )^{1/2}$. Then by (\ref{equations 1 for subsystem integral flows}), we must have $\lim_{\tilde{t} \rightarrow T_{+}} \tau^{\tindex} = - (  \lambda^{2} - |\ltau|^2 - |\lmu|_{h_{\tindex}}^2 )^{1/2} \neq 0$. Then by using (\ref{reparametrised time equation fo rho}), we can conclude that for any $\epsilon > 0$, there exists $\tilde{T}_{0} > 0$ such that for all $\tilde{t} \in [ \tilde{T}_{0}, T_{+} ]$, we have
    \begin{equation*}
    \frac{d \bdf }{ d \tilde{t} } \geq 2 ( \lambda^{2} - |\ltau|^2 - |\lmu|_{h_{\tindex}}^2 )^{1/2} - \epsilon > 0
    \end{equation*}
    for $\epsilon$ small enough. Hence $\bdf$ must be increasing for all $\tilde{t} \in [ \tilde{T}_{0}, T_{+} ]$ and we would not have $\lim_{\tilde{t} \rightarrow T_{+}} \bdf = 0$, a contradiction. \par

    To show (\ref{integral curve initial value condition for rho}), we just note that since $\utau$ is conserved, expression (\ref{solution to first order equation for rho}) gets reduced to 
     \begin{equation*}
    \bdf = e^{2 \int_{0}^{\tilde{t}} \tau_{\tindex} \hspace{0.5mm} d\tilde{t}' } \big( \ibdf - 2 \tau^{\tindex} \int_{0}^{\tilde{t}} e^{-2 \int_{0}^{\tilde{t}'} \tau_{\tindex} \hspace{0.5mm} d\tilde{t}'' } d\tilde{t}'   \big).
    \end{equation*}
    Since we know that $\bff$ can only reach $0$ at $\tilde{t} = T_{+}$, we conclude that this is only possible if 
    \begin{equation*}
    \rho_{\tindex,0} = 2 \tau^{\tindex} \int_{0}^{ T_{+} } e^{-2 \int_{0}^{\tilde{t}'} \tau_{\tindex} d\tilde{t}'' } \hspace{0.5mm} d\tilde{t}' 
    \end{equation*}
    as required. Notice that this indeed excludes the possibility that $T_{+} = \infty$, since otherwise
    \begin{equation*}
    2 \utau \int_{0}^{\tilde{t}} e^{ - 2 \int_{0}^{ \tilde{t}' } \ltau d \tilde{t}'' } d \tilde{t}'
    \end{equation*}
    must diverge to positive infinity, thereby making $\bff$ negative. The converse is obvious. \par 
    This concludes the proof of the lemma.
        \end{proof} 
        Finally, if $\gamma \subset {^{\mathrm{3sc}}T^{\ast}_{\partial \text{ff}_{\tindex}} } \tscX$, then the flow will stay in ${^{\mathrm{3sc}}T^{\ast}_{\partial \text{ff}_{\tindex}} } \tscX$ for all time, with $(y_{\tindex}, \tau_{\tindex}, \mu_{\tindex})$ being constants, while $(y^{\tindex}, \tau^{\tindex}, \mu^{\tindex})$ are determined by equations (\ref{equations 1 for subsystem integral flows}), (\ref{equations 2 for subsystem integral flows}), which in large ($+$) and small ($-$) asymptotics converge to $\pm ( \lambda^{2} - |\ltau|^2 - |\lmu|_{h_{\tindex}}^2 )^{1/2}$, $\mu^{\tindex} = 0$. \par

    We now have a summary of all the integral curves of $\mathsf{H}_{p, \text{mf}}$ near ${^{\mathrm{3sc}}T^{\ast}_{\partial \text{ff}_{\tindex}} \tscX}$.

    \begin{proposition} \label{precise description of integral flow near ff}
    Let $\gamma$ be an integral curve of $\mathsf{H}_{p, \mf }$ in $\Sigma_{\mf}$.

  (1)  Assume $\gamma$ is such that $\bff \leq \delta$ uniformly over $t \in (a, \infty) $ for some $\delta > 0$ small and possibly $a = - \infty$. Then the evolution of $(y^{\tindex}, \tau^{\tindex}, \mu^{\tindex})$ is determined completely by $\mathsf{H}_{ \intn }$, and we have
    \begin{equation} \label{how does the interactive variables flow in the dynamical picture away from tcocf}
    \begin{gathered}
    \text{either $\lim_{t \rightarrow \infty} \tau^{\tindex} = - ( \lambda^2 - |\ltau|^2 - |\lmu|_{h_{\tindex}}^2 )^{1/2}$, $\lim_{t \rightarrow \infty} \mu^{\tindex} = 0$, or} \\ 
    \text{$\tau^{\tindex} = \pm ( \lambda^{2} - |\ltau|^2 - |\lmu|_{h_{\tindex}}^2 )^{1/2}$, $\mu^{\tindex} = 0$ for all $t \in (a, \infty)$.}
    \end{gathered}
    \end{equation} \par

 Moreover, if $\lim_{t \rightarrow \infty} \gamma(t) \in {^{\mathrm{3sc}}T^{\ast}_{ \partial \ff } \tscX}$, $\gamma \not\subset {^{\mathrm{3sc}}T^{\ast}_{ \partial \ff } \tscX}$, then
    \begin{equation*}
    \text{$\tau^{\tindex} = ( \lambda^{2} - | \ltau |^2 - | \lmu |_{h_{\tindex}}^2 )^{1/2}$, $\mu^{\tindex} = 0$ for all $t \in (a,\infty)$,}
    \end{equation*}
 and the evolution of $ (y_{\alpha}, \tau_{\alpha}, \mu_{\alpha})$ is determined completely by $\sH_{ \intt }$. The latter dynamic is better understood in cases:
    \begin{itemize}
    \item[{(1.1)}]  If $| \ltau |^2 + |\lmu|_{h_{\tindex}}^2 = \lambda^2$ and $\tau_{\alpha} \neq \lambda$, then $\lim_{t \rightarrow \infty} \tau_{\tindex} =  - \lambda$ and $\lim_{t \rightarrow \infty} \mu_{\tindex} = 0$, with the possibility of $\ltau = -\lambda$ identically.
    \item[{(1.2)}]  If $ | \ltau |^2 + | \lmu |_{h_{\tindex}}^2 < \lambda^2$, then $\lim_{t \rightarrow \infty} \tau_{\tindex} = \tau_{\tindex,+}$ and $\lim_{t \rightarrow \infty} \mu_{\tindex} = \mu_{\tindex,+} $, where
    \begin{equation*}
    \text{$\tau_{\tindex,+} > - ( \lambda^{2} - |\utau|^2 - |\umu|_{h^{\tindex}}^2 )^{1/2}$, $\mu_{\tindex,+} \neq 0$.}
    \end{equation*}
    \end{itemize}
       These are the only possible cases for $\lim_{t \rightarrow \infty} \gamma(t) \in {^{\mathrm{3sc}}T^{\ast}_{\partial \ff}} \tscX$, $\gamma \not\subset {^{\mathrm{3sc}}T^{\ast}_{\partial \ff}} \tscX$. In particular, they occur if and only if $\lim_{t \rightarrow \infty} \gamma(t) \in \mathcal{R}_{\mathrm{n},-}$. \par

    If we instead assume that $\gamma \cap {^{\mathrm{3sc}}T^{\ast}_{\partial \ff} \tscX} = \emptyset$, then $\lim_{t \rightarrow \infty} \gamma(t) \in \mathcal{R}_{\mathrm{2sc},\mf,-}$.
 \par

   (2) The situation is analogous if $\gamma$ is chosen such that $\bdf \leq \delta$ for $t \in (-\infty, b)$, with the possibility that $b = \infty$. In these cases, the evolution of $(y^{\tindex}, \tau^{\tindex}, \mu^{\tindex})$ is still determined completely by $\mathsf{H}_{\intn}$, only now we have
    \begin{equation} \label{how does the interactive variables flow in the dynamical picture away from tcocf backward direction}
    \begin{gathered}
    \text{either $\lim_{t \rightarrow - \infty} \tau^{\tindex} = ( \lambda^2 - |\ltau|^2 - |\lmu|_{h_{\tindex}}^2 )^{1/2}$, $\lim_{t \rightarrow \infty} \mu^{\tindex} = 0$, or} \\ 
    \text{$\tau^{\tindex} = \pm ( \lambda^{2} - |\ltau|^2 - |\lmu|_{h_{\tindex}}^2 )^{1/2}$, $\mu^{\tindex} = 0$ for all $t \in (-\infty, b)$.}
    \end{gathered}
    \end{equation} \par

    Moreover, if $\lim_{t \rightarrow -\infty} \gamma(t) \in {^{\mathrm{3sc}}T^{\ast}_{ \partial \ff } \tscX}$, $\gamma \not\subset {^{\mathrm{3sc}}T^{\ast}_{ \partial \ff }\tscX}$, then
    \begin{equation*}
    \text{$\tau^{\tindex} = -( \lambda^{2} - |\ltau|^2 - |\lmu|_{h_{\tindex}}^2 )^{1/2}$, $\mu^{\tindex} = 0$ for all $t \in (-\infty,b)$,}
    \end{equation*} 
and the evolution of $(y_{\tindex}, \tau_{\tindex}, \mu_{\tindex})$ is determined completely by $\mathsf{H}_{ \intt }$. More precisely:
    \begin{itemize}
    \item[{(2.1)}]  If $\intt = \lambda^2$ and $\tau_{\alpha} \neq -\lambda$, then $\lim_{t \rightarrow \infty} \tau_{\tindex} =  \lambda$ and $\lim_{t \rightarrow \infty} \mu_{\tindex} = 0$, with the possibility of $\ltau = \lambda$ identically.
    \item[{(2.2)}]  If $\intt < \lambda^2$, then $\lim_{t \rightarrow \infty} \tau_{\tindex} = \tau_{\tindex,-}$ and $\lim_{t \rightarrow \infty} \mu_{\tindex} = \mu_{\tindex,-} $, where
    \begin{equation*}
    \text{$\tau_{\tindex,-} < ( \lambda^{2} - |\utau|^2 - |\umu|_
    {h^{\tindex}}^2 )^{1/2}$, $\mu_{\tindex,-} \neq 0$.}
    \end{equation*}
    \end{itemize} 
        These are the only possible cases for $\lim_{t \rightarrow -\infty} \gamma(t) \in {^{\mathrm{3sc}}T^{\ast}_{\partial \mathrm{ff}_{\tindex}}} \tscX$, $\gamma \not\subset {^{\mathrm{3sc}}T^{\ast}_{\partial \ff }} \tscX$. In particular, they occur if and only if $\lim_{t \rightarrow -\infty} \gamma(t) \in \mathcal{R}_{\mathrm{n},+}.$ \par

            If we instead assume that $\gamma \cap {^{\mathrm{3sc}}T^{\ast}_{\partial \ff} \tscX} = \emptyset$, then $\lim_{t \rightarrow -\infty} \gamma(t) \in \mathcal{R}_{\mathrm{2sc},\mf,+}$.    \par

            Otherwise, if $\gamma \subset {^{\mathrm{3sc}}T^{\ast}_{\partial \ff } } \tscX$, then $\bdf = 0$ and $(y_{\tindex}, \tau_{\tindex}, \mu_{\tindex})$ are constants along $\gamma$. In particular, we can always take $a = -\infty$, $b = \infty$. \par

    Finally, outside of any fixed neighborhood of ${^{\mathrm{3sc}}T^{\ast}_{\partial \ff } } \tscX$, the integral curves of $\mathsf{H}_{p,\mf }$ are exactly the same as those of $\mathsf{H}_{p}$.
    \end{proposition}

\subsection{Degeneracy for the flow of $\sH_{p,\mf}$}
Recall that we have a natural projection
    \begin{equation*}
    {^{\mathrm{3sc}}\pi_{\ff}^{\perp}} : {^{\mathrm{3sc}}T^{\ast}_{\ff} \tscX } \rightarrow \mathcal{C}_{\tindex} \times \mathbb{R}^{n_{\tindex}}.
    \end{equation*}
    We now note that
    \begin{equation} \label{pull-back of tangential elliptic set}
    \Sigma_{\text{mf}} \cap ({ ^{\mathrm{3sc}}\pi_{\ff}^{\perp} })^{-1} ( \Sigma_{\mathrm{t}} ) \subset \mathcal{R}_{\text{n},\pm},
    \end{equation}
    which can also be more concretely written in local coordinates as 
    \begin{equation*}
    \{ (   y_{\tindex}, y^{\tindex}, \ltau, \lmu, \utau, \umu ) \in {^{\mathrm{3sc}}T^{\ast}_{\partial \ff} \tscX} :  | \ltau |^2 + | \lmu |_{h_{\tindex}}^2 = \lambda^2, \utau = 0,  \umu = 0  \},
    \end{equation*}
is a set of points behaving in a degenerate way with respect to the flow of $\mathsf{H}_{p,\text{mf}}$, even in the radial sense. 
Indeed, since $\sH_{p,\mf}$ restricts to $\sH_{ \intn }$ at ${^{\mathrm{3sc}}T^{\ast}_{\partial \ff} }\tscX$, the degeneracy one encounters at (\ref{pull-back of tangential elliptic set}) is actually the same degeneracy one encounters in the two-body setting at zero energy. In particular, we have discussed (in \S \ref{motivation subsection}) that the usual combination of propagation of regularity and radial point estimates will become degenerate in this region (since there is no non-trivial flow one could use to propagate regularity). \par

The reason why we consider second microlocalization is precisely to overcome this challenge. Indeed, recall that we are blowing up $\overline{^{\mathrm{3sc}}T^{\ast}} \tscX$ at ${^{\mathrm{3sc}}\pi_{\ff}^{-1}}( o_{\mathcal{C}^{\tindex}} )$. Suppose we restrict first to $^{\mathrm{3sc}}T^{\ast}_{\ff} \tscX$. Then the blow-up in question, at least away from the fiber infinity, is equivalent to blowing up a set of points given in coordinates by $\{  x^{\tindex} = 0, \utau = 0, \umu = 0 \}$. In other words, we are exactly introducing polar coordinates at the zero section of ${ ^{\mathrm{sc}}T^{\ast}_{\partial X^{\tindex}} }X^{\tindex}$, which we view as a submanifold of ${ ^{\mathrm{3sc}}T^{\ast}_{\ff} } \tscX = \mathcal{C}_{\tindex} \times \mathbb{R}^{n_{\tindex}} \times {^{\mathrm{sc}}T^{\ast}X^{\tindex}}$. This will allow us to employ the methods of \cite[\S 5]{AndrasSM} to account for the degeneracy at (\ref{pull-back of tangential elliptic set}).

\subsection{Characteristic sets of $P$ in the second microlocal framework}
We now move onto the consideration of second microlocalization, and how the discussions of \S\S \ref{subsection characteristic variety of P in the three-body framework}--\ref{subsection: characterization for the flow} can be carried onto this setting. In particular, we will be interested in the dynamic of the Hamiltonian flow under finer rescaling. \par

We first return to the discussion on characteristic sets. First of all, it is easy to check via local calculations that 
\begin{equation*}
 P \in \Psi_{\mathrm{3coc}}^{2,0,0,0}(X), \  P \in \Psi_{\mathrm{d3sc,3co,res}}^{2,0,0,0,0,2}( \Xd ).
\end{equation*}
Let us write
\begin{equation*}
\bcv \coloneq \mathrm{Char}_{\sigma,\delta}( P )
\end{equation*}
for the characteristic set of $P$ as an operator in $\Psi_{\mathrm{d3sc,3co,res}}^{2,0,0,0,0,2}(\Xd)$ in the symbolic sense (see Definition \ref{symbolic version of elliptic and characterstic sets}). If $p$ denotes the principal symbol of $P$ in this sense, then $\bcv$ is defined at the `symbolic' boundary faces of $\psf \Xd$ by those points at which $p$ vanishes. Clearly, $\bcv$ is contained away from both the fiber infinity as well as $\rf$. Thus, it will be enough if we restrict our attention to $\psf_{\dmf}\Xd$ and $\dtsccf$. Moreover, if
\begin{equation*}
\beta_{\mathrm{d3sc,3co,res}}: \psf \Xd \rightarrow \overline{ ^{\mathrm{3sc}}T^{\ast}} \tscX
\end{equation*}
denotes the natural blow-down map, then we have
\begin{equation*}
\bcv \cap \psf_{\dmf}\Xd = \beta_{\mathrm{d3sc,3co,res}}^{\ast}( \Sigma_{\mf} ),
\end{equation*} 
much as in the case of (\ref{the main face characteristic variety as a lift}). \par
Now, in the second microlocalized setting, we will also need to consider the characteristic set of $P$ at $\cf$. But since $p = \intt + \intn - \lambda^2$ as a smooth function on ${^{\mathrm{3sc}}T^{\ast}} \tscX$ near $^{\mathrm{3sc}}T^{\ast}_{\ff} \tscX$, and since $\tcocf$ occurs by blowing up the set 
\begin{equation*}
\{\bff = x^{\tindex} = 0, \utau = 0, \umu = 0 \} \subset {^\mathrm{3sc}T^{\ast}_{\mf \cap  \ff } \tscX },
\end{equation*}
which is followed by another blow-up at the lift of $^{\mathrm{3sc}}T^{\ast}_{\mf \cap \ff} \tscX$. We must have
\begin{equation} \label{indicial operator of P at cf formula}
\hat{N}_{\cf}(P) = p|_{\tcocf} = \intt - \lambda^2.
\end{equation}
Thus, we can trivially conclude that
\begin{equation*}
\Sigma_{\cf} \coloneq \mathrm{Char}_{\cf}(P) = \Sigma_{\mathrm{t}},
\end{equation*}
where $\Sigma_{\mathrm{t}}$ is defined as in (\ref{tangential part of the characteristic variety}). \par

Finally, we consider the characteristic set of $P$ at $\dff$. Note that we can still write
\begin{equation*}
\hat{N}_{\dff}(P) = \hat{N}_{\ff}(P) = \Delta_{z^{\tindex}} + V^{\tindex} - ( \lambda^{2} - | \ltau |^2 - | \lmu |_{h_{\tindex}}^2  )
\end{equation*}
over the interior of $\dff$ (which is the same as $\ff$), except that we now view $\hat{N}_{\dff}(P)$ as a family of operators in $\Psi_{\mathrm{sc,b}}^{2, 0, 0}( X^{\tindex} )$ rather than $\Psi_{\mathrm{sc}}^{2,0}(X^{\tindex})$. As such, much of the analysis from \S \ref{subsection characteristic variety of P in the three-body framework} remains valid even in this case, and we can conclude that $\hat{N}_{\dff}(P)$ is not invertible if and only if $( y_{\tindex}, \ltau, \lmu ) \in \Sigma_{\ff}$. The only issue which remains is to determine whether or not these inverses belong to $\Psi_{\mathrm{sc,b}}^{-2, 0,0}( X^{\tindex} )$. However, this is automatic since we already know that they belong to $\Psi_{\mathrm{sc}}^{-2,0}(X^{\tindex})$. Thus, we will define
\begin{equation*}
\Sigma_{\dff} \coloneq \mathrm{Char}_{\dff}(P) = \Sigma_{\ff} = \Sigma_{\mathrm{t}} \cup \Sigma_{\mathrm{n}}.
\end{equation*}
In addition to the no bound state condition, we will now additionally assume that 
\begin{equation*}
\text{$\hat{N}_{\dff}(P)$ has no half-bound state}
\end{equation*}
in the sense that $\Delta_{z^{\tindex}} + V^{\tindex}$, which can be shown to be a mapping $H_{\mathrm{b}}^{\infty, b }(X^{\tindex}) \rightarrow H_{\mathrm{b}}^{\infty, b + 2}( X^{\tindex} )$ for all $|b+1| \leq (n^{\tindex}-2)/2$, has trivial kernel{\ep}an assumption that is independent of $b$ in this range. See \cite[Equation (1.4)]{AndrasSM}.

\subsection{Rescaled Hamiltonian in the second microlocal framework}
\label{subsection rescaled hamiltonian in the second microlocal framework}
Now, the Hamiltonian vector filed $H_{p}$ should be considered under the rescaling 
\begin{equation} \label{second microlocalized rescaled Hamiltonian vector field with general boundary defining functions}
\rho_{\dmf}^{-1} \rho_{\dtsccf}^{-1} \rho_{\tcocf}^{-2} H_{p},
\end{equation}
where $\rho_{\dmf}, \rho_{\dtsccf}, \rho_{\tcocf}$ are respectively defining functions of $\psf_{\dmf}\Xd, \dtsccf$ and $\tcocf$, so as to ensure that (\ref{second microlocalized rescaled Hamiltonian vector field with general boundary defining functions}) locally (near $\bcv$) restricts to a smooth vector field that is tangent to all boundary faces of $\psf \Xd$. This is a fact which requires to be proved. However, it can be easily checked if we write down $H_p$ in local coordinates. \par

We will first work in a small neighborhood of $\tcocf$, which will henceforth be assumed to always be the case until stated otherwise. Recall (from \S \ref{subsection definition of the three-cone bundle and the variable changes}) that in a neighborhood of $\psf_{\dmf} \Xd$ that is away from $\dtsccf$, we can use the coordinates
\begin{equation*}
\begin{gathered}
( \bff, \hat{x}^{\tindex}, y_{\tindex}, y^{\tindex}, \ltau, \lmu, \utaures, \umures ), \ \text{where} \\
\hat{x}^{\tindex} \simeq \rho_{\dmf},  \ ( \intnres )^{-1/2} \simeq \rho_{\dtsccf}, \ \bff \simeq \rho_{\tcocf}.
\end{gathered}
\end{equation*}
Thus $\rho_{\dmf}^{-1} \rho_{\dtsccf}^{-1} \rho_{\tcocf}^{-2} H_{p} \simeq (\hat{x}^{\tindex})^{-1} \bff^{-2} H_p$ (in the sense that their integral curves agree up to a reparametrization of time), where we can write 
\begin{align} \label{section overview microlocal Hamiltonian vector field 1}
\begin{split}
( \hat{x}^{\tindex} )^{-1} \bff^{-2} H_p & = 2 ( 2 \utaures - \ltau ) \hat{x}^{\tindex} \partial_{\hat{x}^{\tindex}} + 2  ( \ltau - \utaures ) \bff \partial_{\bff} \\
& \quad + 2  ( | \utaures |^2 - | \umures |^2_{h^{\tindex}} - \ltau \utaures ) \partial_{ \utaures } + 2  ( 2 \utaures - \ltau  ) R_{\umures} \\
& \quad +  H_{|\umures|_{h^{\tindex}}^2}  +  \sHto.
\end{split}
\end{align} 
Here $R_{\umures} = \umures \cdot \partial_{\umures}$ and $H_{|\umures|_{h^{\tindex}}^2} = \partial_{\umures} |\umures|_{h^{\tindex}}^2 \cdot \partial_{y^{\tindex}} - \partial_{y^{\tindex}} |\umures |_{h^{\tindex}}^2 \cdot \partial_{\umures} $. \par
\begin{remark}
In fact, although this will not be used directly in this paper, we will point out that the restriction of (\ref{section overview microlocal Hamiltonian vector field 1}) to $\psf_{\dmf} \Xd \cap \tcocf$ can be written as the perturbation of a familiar structure. Indeed, suppose we write
\begin{equation} \label{definition of the standardly rescaled Hamiltonian in the res variables}
\sH_{ \intnres } \coloneq 2 \utaures R_{\umures} - 2 | \umures |_{h^{\tindex}}^2 \partial_{\utaures} +  H_{| \umures |_{h^{\tindex}}^2 }
\end{equation}
in analogy with (\ref{Hamiltonian form near the boundary of front face}). Then upon setting $\hat{x}^{\tindex} = \bff = 0$,  we find that (\ref{section overview microlocal Hamiltonian vector field 1}) becomes
\begin{equation*}
\sH_{\intnres} + \sH_{\intt} + 2 ( \utaures - \ltau ) ( \utaures \partial_{\utaures} + R_{\umures} ).
\end{equation*}
It follows that the variables $(y_{\tindex}, \ltau, \lmu)$ evolve as they do in the free case. Meanwhile, the motions of $( y^{\tindex}, \utaures, \umures )$ are governed by a perturbation of the free Hamiltonian (in the interaction variables) by a smooth multiple (the structure of which is also interesting, see (\ref{writing down the two-body radial set in the second microlocal setting}) below) of the radial vector field in $(\utaures, \umures)$.
\par
\end{remark}
On the other hand, in a neighborhood of $\psf_{\dff}\Xd$ that is away from $\dtsccf$, we can use the coordinates
\begin{equation*}
\begin{gathered}
( x^{\tindex}, \hbff, y_{\tindex}, y^{\tindex}, \ltau, \lmu, \utaub, \umub ), \ \text{where} \\
\hbff \simeq \rho_{\dff},  \ ( \fint )^{-1/2} \simeq \rho_{\dtsccf}, \ x^{\tindex} \simeq \rho_{\tcocf}.
\end{gathered}
\end{equation*}
Thus $\rho_{\dmf}^{-1} \rho_{\dtsccf}^{-1} \rho_{\tcocf}^{-2} H_{p} \simeq (x^{\tindex})^{-2} H_p$, where
\begin{align*} \label{section overview microlocal Hamiltonian vector field 2}
\begin{split}
(x^{\tindex})^{-2} H_{p} & = 2 ( \hbff \ltau - 2 \utaub ) \hbff \partial_{\hbff} + 2  \utaub  x^{\tindex} \partial_{x^{\tindex}} \\
& \quad - 2  ( \fint ) \partial_{\utaub} + H_{ | \umub |_{h^{\tindex}}^2 } +  \hbff \sHto.
\end{split}
\end{align*}
Here $H_{|\umub|_{h^{\tindex}}^2 } = \partial_{\umub} |\umub|_{h^{\tindex}}^2 \cdot \partial_{y^{\tindex}} - \partial_{y^{\tindex}} |\umub|_{h^{\tindex}}^2 \cdot \partial_{\umub}$. \par

We next consider the structure of $H_p$ near $\dtsccf \cap \tcocf$. In such regions, it would be convenient to introduce polar coordinates, i.e.,
\begin{equation*}
\begin{gathered}
\hat{\tau}^{\tindex}_{\mathrm{sf}} = \frac{\utaures}{ ( \intnres )^{1/2} }, \ \hat{\mu}^{\tindex}_{\mathrm{sf}} = \frac{\umures}{ ( \intnres )^{1/2} }, \\
\hat{\tau}^{\tindex}_{\mathrm{b}} = \frac{\utaub}{ ( \fint )^{1/2} }, \ \hat{\mu}^{\tindex}_{\mathrm{b}} = \frac{\umub}{ ( \fint )^{1/2} },
\end{gathered}
\end{equation*}
as well as
\begin{equation*}
\rhoresf = \frac{1}{ ( \intnres )^{1/2} }, \ \rhobf = \frac{1}{ ( \intnb )^{1/2} }.
\end{equation*} \par
Suppose we are near $\psf_{\dmf}\Xd$, then with the standard abuse of notations, we can introduce coordinates
\begin{equation*}
\begin{gathered}
( \hat{x}^{\tindex},  (\intn)^{1/2} , y_{\tindex}, y^{\tindex}, \ltau, \lmu, \hutaures, \humures, \rhoresf ), \ \text{where} \\
\hat{x}^{\tindex} \simeq \rho_{\dmf},  \ ( \intn )^{1/2}  \simeq \rho_{\tcocf}, \ \rhoresf \simeq \rho_{\dtsccf}.
\end{gathered}
\end{equation*}
Thus $\rho_{\dmf}^{-1} \rho_{\dtsccf}^{-1} \rho_{\tcocf}^{-2} H_{p} \simeq (\hat{x}^{\tindex})^{-1} ( \intn )^{-1} \rhoresf^{-1} H_p$, where
\begin{align}  \label{second microlocal dynamic modification cal 3}
\begin{split}
& (\hat{x}^{\tindex})^{-1} ( \intn )^{-1} \rhoresf^{-1} H_{p} \\
& \qquad = 2  ( 2 \hutaures - \rhoresf \ltau ) \hat{x}^{\tindex} \partial_{\hat{x}^{\tindex}} + 2  (  \rhoresf \ltau -  \hutaures  ) \rhoresf \partial_{\rhoresf} \\
& \qquad \quad + \sH_{ \intnres, \infty} + \rhoresf \sHto.
\end{split}
\end{align} 
Here, we are writing 
\begin{equation*}
\sH_{\intnres, \infty} \coloneq  \rhoresf \sH_{ \intnres }|_{\rhoresf = 0}
\end{equation*}
to denote the restriction of $\sH_{\intnres}$ (which is defined in (\ref{definition of the standardly rescaled Hamiltonian in the res variables})) to $\rhoresf = 0$ in polar coordinates $( \hutaures, \humures, \rhoresf )$. In particular, if $\utaures$ is dominating, then we can write
\begin{equation} \label{res variable rescaled Hamiltonian at fiber infinity long time}
\sH_{\intnres, \infty}  = 2 \hutaures R_{\humures} + H_{ | \humures |_{h^{\tindex}}^2 },
\end{equation}
where $R_{\humures} = \humures \cdot \partial_{\humures}$ and $H_{| \humures |_{h^{\tindex}}^2} = \partial_{\humures} |\humures|_{h^{\tindex}}^2 \cdot \partial_{y^{\tindex}} - \partial_{y^{\tindex}} |\humures|_{h^{\tindex}}^2 \cdot \partial_{ \humures }$. For the brevity of exposition, we shall omit calculating the local expressions for $\sH_{\intnres,\infty}$ if instead one of $\humuresj$ for $j=1,...,n^{\tindex}-1$ is dominating. \par

Likewise, if we are near $\psf_{\dff} \Xd$, then we can use the coordinates (and slightly abusing notations as before)
\begin{equation*}
\begin{gathered}
(  (\intn)^{1/2}, \hbff ,  y_{\tindex}, y^{\tindex}, \ltau, \lmu, \hutaub, \humub, \rhobf ), \ \text{where} \\
\hbff \simeq \rho_{\dff}, \ ( \intn )^{1/2} \simeq \rho_{\tcocf}, \ \rhobf \simeq \rho_{\dtsccf}.
\end{gathered}
\end{equation*}
Thus $\rho_{\dmf}^{-1} \rho_{\dtsccf}^{-1} \rho_{\tcocf}^{-2} H_{p} \simeq (\intn)^{-1} \rhobf^{-1} H_p$ , where
\begin{align}  \label{second microlocal dynamic modification cal 5}
\begin{split}
( \intn )^{-1} \rhobf^{-1} H_{p} & = 2 ( \rhobf \hbff \ltau - 2 \hutaub ) \hbff \partial_{\hbff} + 2 \hutaub \rhobf \partial_{\rhobf} \\
& \quad  +  {^{\mathrm{b}}\sH}_{\intnb, \infty} + \rhobf \hbff \sHto.
\end{split}
\end{align} 
Here we are writing
\begin{equation*}
^{\mathrm{b}}\sH_{\intnb} \coloneq  - 2 ( \intnb ) \partial_{\utaub} + H_{| \umub |_{h^{\tindex}}^2 },
\end{equation*}
where the additional superscript `b' indicates that $ {^{\mathrm{b}}\sH_{\intnb}}$ should be viewed as the restriction to spatial infinity of a rescaled Hamiltonian vector field of $\intnb$ in the b-calculus. Then 
\begin{equation*}
{^{\mathrm{b}}\sH_{\intnb,\infty}} \coloneq \rhobf {^{\mathrm{b}}\sH_{\intnb}}|_{\rhobf = 0}
\end{equation*}
denotes the restriction of $\rhobf {^{\mathrm{b}}\sH_{\intnb}}$ to $\rhobf = 0$ in polar coordinates $(\hutaub, \humub, \rhobf)$. Moreover, if $\hutaub$ is dominating, then we can write
\begin{equation}  \label{second microlocal dynamic modification cal 5.1}
{^{\mathrm{b}}\sH_{\intnb,\infty}}  = 2 \hutaub R_{\humub} + H_{ |\humub|_{h^{\tindex}}^2 },  
\end{equation} 
where $R_{\humub} = \humub \cdot \partial_{\humub}$ and $H_{ | \humub |_{h^{\tindex}}^2 } = \partial_{ \humub } | \humub |_{h^{\tindex}}^2 \cdot \partial_{y^{\tindex}} - \partial_{y^{\tindex}} | \humub |_{h^{\tindex}}^2 \cdot \partial_{ \humub } $. If instead we assume that one of $\hat{\mu}^{\tindex}_{\mathrm{b},j}$, $j =1, ..., n^{\tindex}-1$ is dominating, then a similar calculation, which will again be omitted for the brevity of exposition, can be carried out as well. \par

Notice that the above formulae near $\dtsccf \cap \tcocf$ are valid for $\intn < \infty$. Thus in particular, they must also be valid in a neighborhood of $\Sigma_{\sigma} \cap \dtsccf$. Indeed, this is because $\intn = \lambda^{2} - \intt \leq \lambda^2$ when restricted to $\Sigma_{\sigma}$.

Finally, it will be instructive to present the expressions for $H_{p}$ in regions which contain $\bcv \cap \dtsccf$ but do not intersect $\tcocf$. Moreover, locally we will write down $H_{p}$ as a vector field on ${^{\mathrm{d3sc}}T^{\ast} \Xd}$ away from the lift of ${^\mathrm{3sc}\pi_{\ff}^{-1} (o_{\mathcal{C}^{\tindex}})}$. In other words, the blow-up which requires second microlocalization  needs not be considered. \par

To this end, assume that we are near ${^{\mathrm{d3sc}}T^{\ast}_{\dmf}\Xd}$. Then we can use the coordinates
\begin{equation*}
( \bff, \hat{x}^{\tindex}, y_{\tindex}, y^{\tindex}, \ltau, \lmu, \utau, \umu ), \ \text{where} \ \hat{x}^{\tindex} \simeq \rho_{\dmf}, \, \bdf \simeq \rho_{\dtsccf}.
\end{equation*}
Thus $\rho_{\dmf}^{-1} \rho_{\dtsccf}^{-1} \rho_{\tcocf}^{-2} H_{p} \simeq ( \hat{x}^{\tindex} )^{-1} \bff^{-1} H_p$. In fact, the rescaling in (\ref{rescalled Hamiltonian flow at main face}) remains to be valid in this case, and we easily find that
\begin{align} \label{section overview microlocal Hamiltonian vector field 3}
\begin{split}
(\hat{x}^{\tindex})^{-1} \bff^{-1} H_{p} & = 2 ( 2 \utau - \bff \ltau ) \hat{x}^{\tindex} \partial_{\hat{x}^{\tindex}} +  2 ( \bff \ltau - \utau ) \bff \partial_{ \bff } \\
& \quad + \sHno + \bff \sHto .
\end{split}
\end{align}
On the other hand, in a neighborhood of ${^{\mathrm{d3sc}}T^{\ast}_{\dff} \Xd}$, we can use the coordinates
\begin{equation*}
( x^{\tindex}, \hbff, y_{\tindex}, y^{\tindex}, \ltau, \lmu, \utau, \umu ), \ \text{where} \ \hbff \simeq \rho_{\dff}, \, x^{\tindex} \simeq \rho_{\dtsccf}.
\end{equation*}
Thus $\rho_{\dmf}^{-1} \rho_{\dtsccf}^{-1} \rho_{\tcocf}^{-2} H_{p} \simeq (x^{\tindex})^{-1} H_p$, where 
\begin{align} \label{section overview microlocal Hamiltonian vector field 4}
\begin{split}
(x^{\tindex})^{-1} H_p & = 2 \hbff ( x^{\tindex} \hbff \ltau - 2 \utau ) \partial_{\hbff} +  2 \utau x^{\tindex} \partial_{x^{\tindex}} \\
& \quad + \sHno + \hbff x^{\tindex} \sHto.
\end{split}
\end{align} 
\subsection{Radial points in the second microlocal framework}
\label{subsection radial points in the second microlocal framework}
We now proceed to compute the stationary points of $\rho_{\dmf}^{-1} \rho_{\dtsccf}^{-1} \rho_{\tcocf}^{-2} H_{p}$ within $\bcv$, which we will interpret as radial points in the second microlocalized setting.

First of all, calculations (\ref{second microlocal dynamic modification cal 3})--(\ref{section overview microlocal Hamiltonian vector field 4}) imply that the roles of the `normal' radial sets $\mathcal{R}_{\mathrm{n},\pm}$ should now be replaced by  a combined roles of two new sets. Namely
\begin{equation*}
\begin{gathered}
\mathcal{R}_{\mathrm{n}, \dmf, \pm} \subset \dtsccf \cap \psf_{\dmf} \Xd, \\
\mathcal{R}_{\mathrm{n}, \dff, \pm} \subset \dtsccf \cap \psf_{\dff} \Xd,
\end{gathered}
\end{equation*}
where by (\ref{second microlocal dynamic modification cal 3}) and (\ref{second microlocal dynamic modification cal 5}), we define
\begin{equation*}
\begin{gathered}
\mathcal{R}_{n, \dmf, \pm}   \coloneq \{ \hat{x}^{\tindex} =  \rhoresf = 0, \hutaures = \pm1, \humures = 0, \intn = \lambda^{2} - \intt \}, \\
\mathcal{R}_{n, \dff, \pm} \coloneq \{ \hbff = \rhobf = 0, \hutaub = \pm1, \humub = 0, \intn = \lambda^2 - \intt  \}.
\end{gathered}
\end{equation*}
Notice that by (\ref{section overview microlocal Hamiltonian vector field 3}) and (\ref{section overview microlocal Hamiltonian vector field 4}), we can also write
\begin{equation} \label{Local d3sc coordintes representation of Rndmfpm}
\begin{gathered}
\mathcal{R}_{\mathrm{n}, \dmf, \pm} \backslash \tcocf  = \{  \hat{x}^{\tindex} = \bff = 0, \utau = \pm ( \lambda^{2} - |\ltau|^2 - |\lmu|_{h_{\tindex}}^2 )^{1/2}, \umu = 0  \}, \\
\mathcal{R}_{\mathrm{n}, \dff, \pm} \backslash \tcocf = \{ \hbff =  x^{\tindex} = 0, \utau = \pm ( \lambda^2 - |\ltau|^2 - |\lmu|_{h_{\tindex}}^2 )^{1/2}, \umu = 0 \},
\end{gathered}
\end{equation}
which should be compared to the definitions of $\mathcal{R}_{\mathrm{n}, \pm}$ in (\ref{local coordinates representation of Rnpm}) above.

\par

Furthermore, as in the setting with no second microlocalization, we also have the lifts of the two-body, or `free' radial sets $\mathcal{R}_{\mathrm{sc},\pm}$ onto $\psf_{\dmf} \Xd$, which will be denoted by
\begin{equation*}
\brpm \subset \psf_{\dmf}\Xd.
\end{equation*}
Clearly, $\brpm$ can also be identified as the lifts of $\mathcal{R}_{\mathrm{2sc},\mf,\pm}$. Locally in a neighborhood of $\psf_{\dmf} \Xd \cap \tcocf$ , we can write
\begin{equation} \label{writing down the two-body radial set in the second microlocal setting}
\brpm = \Big\{ \ltau = \pm \frac{1}{ ( 1 + \bff^{2} )^{1/2} } \lambda, \,  \lmu = 0, \, \utaures = \ltau,  \, \umures = 0  \Big\},
\end{equation}
while $\brpm$ is still characterized by (\ref{writing down the 3sc radial set in local coordinates}) everywhere else.
\par

Finally, there exists a new kind of radial set, which can be defined in local coordinates by
\begin{equation*}
\mathcal{R}_{0,\pm} \coloneq \{ \bff = \hat{x}^{\tindex} = 0, \, \ltau = \pm \lambda,  \, \lmu = 0, \, \utaures = 0, \,  \umures = 0 \}.
\end{equation*}
Thus $\mathcal{R}_{0, \pm}$ is a subset of $\tcocf \cap \psf_{\dmf} \Xd $ which does not intersect $\dtsccf$.   \par

The nature (as equilibria) of these radial points can be easily determined via the standard dynamical methods. Here we shall omit these calculations, and just state:
\begin{proposition}
Let $\rho_{\dmf}, \rho_{\dtsccf}, \rho_{\tcocf}$ be respectively boundary defining functions for $\psf_{\dmf}\Xd, \dtsccf$ and $\tcocf$. Then the vector field
\begin{equation}
\label{second microlocally rescaled vector field in Proposition}
\rho_{\dmf}^{-1} \rho_{\dtsccf}^{-1} \rho_{\tcocf}^{-2} H_{p}
\end{equation}
is tangent to every boundary face of $\psf \Xd$ near $\bcv$. In particular, (\ref{second microlocally rescaled vector field in Proposition}) defines a complete flow over $\bcv$, with the following equilibria:
\begin{itemize}
\item $\mathcal{R}_{\mathrm{n,dmf},\pm} \subset \dtsccf \cap \psf_{\dmf} \Xd$, which are saddles;
\item $\mathcal{R}_{\mathrm{n},\dff,\pm} \subset \dtsccf \cap \psf_{\dff} \Xd$, which are saddles; 
\item $\mathcal{R}_{0,\pm} \subset \tcocf^{\circ} \cap \psf_{\dmf} \Xd$, which are saddles;
\item $\brpm \subset \psf_{\dmf} \Xd$, which are respectively source (corresponding to the $+$ sign) and sink (corresponding to the $-$ sign). 
\end{itemize} 
\end{proposition}
\subsection{Characterization for the flow of $\rho_{\dmf}^{-1} \rho_{\dtsccf}^{-1} \rho_{\tcocf}^{-2} H_{p}$}\label{subsection characterization for the flow of the second microlocal dynamic}
We proceed to characterize the flow of $\rho_{\dmf}^{-1} \rho_{\dtsccf}^{-1} \rho_{\tcocf}^{-2} H_{p}$ in $\bcv$ near those radial sets which are saddles. This will be important when we consider radial point estimates in the sections below. In particular, since we will state propagation estimates in purely dynamical terms, it will be necessary to understand how the flow of (\ref{second microlocally rescaled vector field in Proposition}) behaves near the saddle points.  \par

Notice that such an analysis will not be necessary near the sources and sink $\brpm$, since the local behaviors near $\brpm$ are automatically understood, i.e. all local integral curves of (\ref{second microlocally rescaled vector field in Proposition}) are attracted to $\brm$ repelled from $\brp$. \par

Before we state our results, we note that it will be convenient to represent the evolution in the interaction variables near $\dtsccf$ in a uniform way. Indeed, suppose we write
\begin{equation*}
\hat{\tau}^{\tindex} = \frac{\utau}{( \intn )^{1/2}}, \ \hat{\mu}^{\tindex} = \frac{ \umu }{ ( \intn )^{1/2} }.
\end{equation*}
Then by scaling we must have $\hat{\tau}^{\tindex} = \hutaures = \hutaub$,  $\hat{\mu}^{\tindex} = \humures = \humub$. Moreover, we can check that
\begin{equation} \label{uniform interactive variable flow Hamiltonian scaled}
 \sH_{\intnres, \infty} = {^{\mathrm{b}}\sH_{\intnb, \infty}}.
\end{equation}
depends only on derivatives in the directions of $( \hutau, \humu )$.

\begin{proposition} \label{proposition characterization of the integral curves which converge to dtsccf}
Let $\gamma$ be an integral curve of $\rho_{\dmf}^{-1} \rho_{\dtsccf}^{-1} \rho_{\tcocf}^{-2}  H_p$ in $\Sigma_{\sigma}$ such that $\gamma \subset \psf_{\dmf} \Xd$.

(1) Assume $\gamma$ is such that $\rho_{\dtsccf} \leq \delta$ uniformly over $t \in (a, \infty)$ for some $\delta > 0$ small and possibly $a = -\infty$. Then the evolution of $( y^{\tindex}, \hat{\tau}^{\tindex}, \hat{\mu}^{\tindex} )$ is determined by (\ref{uniform interactive variable flow Hamiltonian scaled}), and we have
    \begin{equation*}
    \begin{gathered}
    \text{either $\lim_{t \rightarrow \infty} \hat{\tau}^{\tindex} = - 1$, $\lim_{t \rightarrow \infty} \humu = 0$, or} \\ 
    \text{$\hat{\tau}^{\tindex} = \pm 1$, $\hat{\mu}^{\tindex} = 0$ for all $t \in (a, \infty)$,}
    \end{gathered}
    \end{equation*}
with $(\intn)^{1/2} \in [0, \lambda]$ being constant throughout the flow. In particular, it is impossible that $\lim_{t \rightarrow \infty} \gamma(t) \in \dtsccf \cap \tcocf$ if $\gamma \not\subset \tcocf$. For $(\intn)^{1/2} > 0$, the evolution can also be understood in terms of the variables $( y^{\tindex}, \tau^{\tindex}, \mu^{\tindex} )$ and (\ref{how does the interactive variables flow in the dynamical picture away from tcocf}). \par

Moreover, if $\lim_{t \rightarrow \infty} \gamma(t) \in \dtsccf$, $\gamma \not\subset \dtsccf$, then 
\begin{equation*}
\text{$\hat{\tau}^{\tindex} = 1$, $\hat{\mu}^{\tindex} = 0$ for all $t \in (a,\infty)$,}
\end{equation*}
and the evolution of $ (y_{\alpha}, \tau_{\alpha}, \mu_{\alpha})$ is determined completely by $\sH_{ \intt }$. The latter dynamic is better understood in cases:
    \begin{enumerate}
    \item[{(1.1)}] If $| \ltau |^2 + |\lmu|_{h_{\tindex}}^2 = \lambda^2$ and $\tau_{\alpha} \neq \lambda$, then $\lim_{t \rightarrow \infty} \tau_{\tindex} =  - \lambda$ and $\lim_{t \rightarrow \infty} \mu_{\tindex} = 0$, with the possibility of $\ltau = -\lambda$ identically.
    \item[{(1.2)}]  If $ | \ltau |^2 + | \lmu |_{h_{\tindex}}^2 < \lambda^2$, then $\lim_{t \rightarrow \infty} \tau_{\tindex} = \tau_{\tindex,+}$ and $\lim_{t \rightarrow \infty} \mu_{\tindex} = \mu_{\tindex,+} $, where
    \begin{equation*}
    \text{$\tau_{\tindex,+} > - ( \lambda^{2} - |\utau|^2 - |\umu|_{h^{\tindex}}^2 )^{1/2}$,   $\mu_{\tindex,+} \neq 0$.}
    \end{equation*}
    \end{enumerate} 
        These are the only possible cases for $\lim_{t \rightarrow \infty} \gamma(t) \in \dtsccf$ if $\gamma \not\subset \dtsccf$. In particular, they occur if and only if $\lim_{t \rightarrow \infty} \gamma(t) \in \mathcal{R}_{\mathrm{n},\dmf,+}$. \par

        If we instead assume that $\gamma \cap \dtsccf = \emptyset$, $ \gamma \not\subset \tcocf$, then $\lim_{t \rightarrow \infty} \gamma(t) \in \brm$. If $\gamma \subset \tcocf$, then $\gamma$ necessarily converges to $\mathcal{R}_{\mathrm{n},\dmf,+}$ as $t \rightarrow \infty$ if $\delta > 0$ is very small. \par

      (2) The situation is analogous if $\gamma$ is chosen such that $\rho_{\dtsccf} \leq \delta$ for $t \in (-\infty, b)$, with the possibility that $b = \infty$.  In this case, the evolution of $(y^{\tindex}, \hat{\tau}^{\tindex}, \hat{\mu}^{\tindex})$ is still determined completely by (\ref{uniform interactive variable flow Hamiltonian scaled}), only now we have
    \begin{equation*}
    \begin{gathered}
    \text{either $\lim_{t \rightarrow - \infty} \hat{\tau}^{\tindex} = 1 $, $\lim_{t \rightarrow \infty} \hat{\mu}^{\tindex} = 0$, or} \\ 
    \text{$\hutau = \pm  1 \ \humu= 0$ for all $t \in (-\infty, b)$,}
    \end{gathered}
    \end{equation*}
    with $(\intn)^{1/2} \in [0, \lambda]$ being constant throughout the flow. In particular, it is impossible that $\lim_{t \rightarrow -\infty} \gamma(t) \in \dtsccf \cap \tcocf$ if $\gamma \not\subset \tcocf$. For $(\intn)^{1/2} > 0$, the evolution can also be understood in terms of the variables $( y^{\tindex}, \tau^{\tindex}, \mu^{\tindex} )$ and (\ref{how does the interactive variables flow in the dynamical picture away from tcocf backward direction}). \par

    Moreover, if $\lim_{t \rightarrow -\infty} \gamma(t) \in \dtsccf$, $\gamma \not\subset \dtsccf$, then 
    \begin{equation*}
    \text{$\hutau = -1$, $\humu = 0$ for all $t \in (-\infty,b)$,}
    \end{equation*}
    and the evolution of $(y_{\tindex}, \tau_{\tindex}, \mu_{\tindex})$ is determined completely by $\mathsf{H}_{ \intt }$. More precisely:
        \begin{itemize}
    \item[{(2.1)}]  If $\intt = \lambda^2$ and $\tau_{\alpha} \neq -\lambda$, then $\lim_{t \rightarrow \infty} \tau_{\tindex} =  \lambda$ and $\lim_{t \rightarrow \infty} \mu_{\tindex} = 0$, with the possibility of $\ltau = \lambda$ identically.
    \item[{(2.2)}]  If $\intt < \lambda^2$, then $\lim_{t \rightarrow \infty} \tau_{\tindex} = \tau_{\tindex,-}$ and $\lim_{t \rightarrow \infty} \mu_{\tindex} = \mu_{\tindex,-} $, where
    \begin{equation*}
    \text{$\tau_{\tindex,-}  < ( \lambda^{2} - |\utau|^2 - |\umu|_
    {h^{\tindex}}^2 )^{1/2}$, $\mu_{\tindex,-} \neq 0$.}
    \end{equation*}
    \end{itemize} 
        These are the only possible cases for $\lim_{t \rightarrow -\infty} \gamma(t) \in \dtsccf$ if $\gamma \not\subset \dtsccf$. In particular, they occur if and only if $\lim_{t \rightarrow -\infty} \gamma(t) \in \mathcal{R}_{\mathrm{n},\mf,-}$.\par

    If we instead assume that $\gamma \cap \dtsccf = \emptyset$, $ \gamma \not\subset \tcocf$, then $\lim_{t \rightarrow -\infty} \gamma(t) \in \brp$. If $\gamma \subset \tcocf$, then $\gamma$ can only converge to $\mathcal{R}_{\mathrm{n},\dmf,-}$ as $t \rightarrow -\infty$ if $\delta > 0$ is very small.
    \end{proposition}
    \begin{proof}
    This proposition can be proved by essentially following the proof of Proposition \ref{precise description of integral flow near ff} verbatim. Indeed, we simply need to observe that the structure of (\ref{second microlocal dynamic modification cal 3}), when restricted to $\psf \Xd$, is almost identical to that of $\sH_{p,\mf}$, except that the evolution in the interaction variables must now be rescaled. 
    \end{proof}
    \begin{proposition} \label{proposition characterization of the integral curves which live on dtsccf}
    Let $\gamma$ be an integral curve of $\rho_{\dmf}^{-1} \rho_{\dtsccf}^{-1} \rho_{\tcocf}^{-2}  H_p$ in $\bcv$ such that $\gamma \subset \dtsccf$. Then the free variables $(y_{\tindex}, \ltau, \lmu)$ must be constants over $\gamma$. Moreover, if $\gamma$ is not a stationary point, then there are only the following possibilities for the trajectory of $\gamma$:
    \begin{enumerate}
\item $\gamma \subset \dtsccf \cap \psf_{\dmf} \Xd$, with $\lim_{t \rightarrow -\infty} \gamma(t) \in \mathcal{R}_{\mathrm{n},\dmf, +}$ and $\lim_{t \rightarrow \infty} \gamma(t) \in \mathcal{R}_{\mathrm{n}, \dmf,-}$, where the flows of the interaction variables are determined by (\ref{uniform interactive variable flow Hamiltonian scaled});
\item $\gamma \subset \dtsccf \cap \psf_{\dff} \Xd$, with $\lim_{t \rightarrow -\infty} \gamma(t) \in \mathcal{R}_{\mathrm{n},\dff, +}$ and $\lim_{t \rightarrow \infty}\gamma(t) \in \mathcal{R}_{\mathrm{n},\dff,-}$, where the flows of the interaction variables are determined by (\ref{uniform interactive variable flow Hamiltonian scaled}); 
\item $\lim_{t \rightarrow -\infty} \gamma(t) \in \mathcal{R}_{\mathrm{n}, \dmf, +}$ and $\lim_{ t \rightarrow \infty } \gamma(t) \in \mathcal{R}_{\mathrm{n},\dff,+}$, with $\hutau = 1, \humu = 0$ over $\gamma$;
\item $\lim_{t \rightarrow -\infty} \gamma(t) \in \mathcal{R}_{\mathrm{n},\dff, -}$ and $\lim_{t \rightarrow \infty} \gamma(t) \in \mathcal{R}_{\mathrm{n},\dmf,-}$, with $\hutau = -1, \humu = 0$ over $\gamma$.
\end{enumerate}
\end{proposition}
\begin{proof}
The proof of this result is also obvious, since by using (\ref{second microlocal dynamic modification cal 3}) and (\ref{second microlocal dynamic modification cal 5}), we can solve for the flow at $\dtsccf$ directly. Thus the proposed claim follows from enumerating over all the possible solutions.
\end{proof}

\begin{proposition} \label{proposition characterization of the integral curves which converge to R0}
Let $\gamma$ be an integral curve of $\rho_{\dmf}^{-1} \rho_{\dtsccf}^{-1} \rho_{\tcocf}^{-2}  H_p$ in $\bcv$. Assume that $\gamma$ is not a stationary point. Then:
\begin{enumerate} 
\item If $\lim_{t \rightarrow \infty} \gamma(t) \in \mathcal{R}_{0,-}$ and $\gamma \subset \psf_{\dmf}\Xd$, then $\gamma$ identifies with the lift of an integral curve $\tilde{\gamma}$ of $\sH_{p,\mf}$ in $\Sigma_{\mf}$ such that $\tilde{\gamma} \subset {^{\mathrm{3sc}}T^{\ast}_{\mf} \tscX}$, $\lim_{t \rightarrow \infty} \tilde{\gamma}(t) \in {^{\mathrm{3sc}}T^{\ast}_{\partial \ff} \tscX}$, and $\intt = \lambda^2$ identically along $\tilde{\gamma}$.  
\item If $\lim_{t \rightarrow -\infty} \gamma(t) \in \mathcal{R}_{0,+}$ and $\gamma \subset \psf_{\dmf}\Xd$, then $\gamma$ identifies with the lift of an integral curve $\tilde{\gamma}$ of $\sH_{p,\mf}$ in $\Sigma_{\mf}$ such that $\tilde{\gamma} \subset {^{\mathrm{3sc}}T^{\ast}_{\mf} \tscX}$, $\lim_{t \rightarrow -\infty} \tilde{\gamma}(t) \in {^{\mathrm{3sc}}T^{\ast}_{\partial \ff} \tscX}$, and $\intt = \lambda^2$ identically along $\tilde{\gamma}$. 
\item If either $\lim_{t \rightarrow \infty} \gamma(t) \in \mathcal{R}_{0,-}$ or $\lim_{t \rightarrow -\infty} \gamma(t) \in \mathcal{R}_{0,+}$ holds, and assume that $\gamma \subset \tcocf$, then we must have $\lim_{t \rightarrow -\infty} \gamma(t) \in \mathcal{R}_{0,+}$, $\lim_{t \rightarrow \infty} \gamma(t) \in \mathcal{R}_{0,-}$ and $\utaures = 0$, $\umures = 0$ identically over $\gamma$.
\item If $\lim_{t \rightarrow \infty} \gamma(t) \in \mathcal{R}_{0,+}$, then we must have $\gamma \subset \tcocf$, $\ltau = \lambda$, $\lmu = 0$ identically over $\gamma$, and $\intnres \rightarrow 0$ strictly monotonically as $t \rightarrow \infty$ along $\gamma$ in a small neighborhood of $\mathcal{R}_{0,+}$ . 
\item If $\lim_{t \rightarrow -\infty} \gamma(t) \in \mathcal{R}_{0,-}$, then we must have $\gamma \subset \tcocf$, $\ltau = -\lambda$,  $\lmu = 0$ identically over $\gamma$, and $\intnres \rightarrow 0$ strictly monotonically as $t \rightarrow -\infty$ along $\gamma$ in a small neighborhood of $\mathcal{R}_{0,-}$ . 
\end{enumerate}
In particular, in cases (1) and (2), we must have $\utaures = 0$, $\umures = 0$ identically over $\gamma$. \par

Moreover, these are the only possible cases for $\gamma$ to converge to $\mathcal{R}_{0,\pm}$ in forward or backwards time asymptotics.
\end{proposition}
\begin{proof}
First, it is clear that if $\lim_{t \rightarrow - \infty} \gamma(t) \in \mathcal{R}_{0,-}$ or $\lim_{t \rightarrow \infty} \gamma(t) \in \mathcal{R}_{0,+}$, then we must have $\ltau = -\lambda$, $\mu = 0$ resp. $\ltau = \lambda$, $\lmu = 0$ identically over $\gamma$, for otherwise there would be contradictions with the evolution of the free variables $(y_{\tindex}, \ltau, \lmu)$, which is entirely determined by $\sH_{\intt}$ (by the structure of (\ref{section overview microlocal Hamiltonian vector field 1}), which is the only relevant local expression  at $\bcv \cap \tcocf$). Moreover, since
\begin{align*}
( \hat{x}^{\tindex} )^{-1} \bff^{-2} H_p ( \intnres ) & = 4 ( \utaures - \ltau ) ( \intnres ), \\
( \hat{x}^{\tindex} )^{-1} \bff^{-2} H_p \bff & = 2 ( \ltau - \utaures ) \bff, 
\end{align*}
and since $\utaures \simeq 0$, $\ltau \simeq \pm \lambda$ in a small neighborhood of $\mathcal{R}_{0,\pm}$, we can conclude that in the forward direction of the flow, $\intnres$ must be strictly decreasing near $\mathcal{R}_{0,+}$ and strictly increasing near $\mathcal{R}_{0,-}$, while $\bff$, which is a local defining function for $\tcocf$, must be strictly increasing near $\mathcal{R}_{0, +}$ and strictly decreasing near $\mathcal{R}_{0,-}$. In particular, it is impossible that $\gamma \not\subset \tcocf$. Thus we have shown that (4) and (5) are the only possible trajectories of $\gamma$ if $\lim_{t \rightarrow \infty} \gamma(t) \in \mathcal{R}_{0,+}$ resp. $\lim_{t \rightarrow -\infty} \gamma(t) \in \mathcal{R}_{0,-}$. \par

Next, we will consider those $\gamma$ such that $\lim_{t \rightarrow \infty} \gamma(t) \in \mathcal{R}_{0,-}$. Assume that $\gamma \subset \tcocf$. Then by the above analysis, $\intnres$ can only increase in the forward direction of $\gamma$. Thus, it will be impossible for $\gamma$ to reach $\mathcal{R}_{0,-}$ unless $\utaures = 0$, $\umures = 0$ over $\gamma$. In this case, any non-trivial evolution can only occur in the free variables $(y_{\tindex}, \ltau, \lmu)$. Running their full courses, we then conclude it must be that $\lim_{t \rightarrow -\infty} \gamma(t) \in \mathcal{R}_{0,-}$ and $\lim_{t \rightarrow \infty} \gamma(t) \in \mathcal{R}_{0,+}$. Notice that the analogous argument allows us to obtain the same conclusion if we instead assume that $\lim_{t \rightarrow - \infty} \gamma(t) \in \mathcal{R}_{0,+}$ and $\gamma \subset \tcocf$. This proves (3). \par

Still assuming that $\lim_{t \rightarrow \infty} \gamma(t) \in \mathcal{R}_{0,+}$, except that now we have $\gamma \subset \psf_{\dmf} \Xd$. Then $\gamma$ can always be considered as the lift of $\tilde{\gamma} \subset {^{\mathrm{3sc}}T^{\ast}_{\mf}X}$, with $\tilde{\gamma}$ being some integral curve of $\sH_{p,\mf}$ in $\Sigma_{\mf}$. Thus, (1) follows immediately from Proposition \ref{precise description of integral flow near ff}. The same argument works to prove (2). 
\end{proof}

\section{Specification of the variable orders} \label{variable order construction section}
In this section, we will construct explicit examples of variable orders which will be used in the Fredholm analysis of $P$. We will also explain how such constructions will be sufficient in constructing the proposed Fredholm map. For simplicity, we will only consider the case when there is a single $\mathcal{C}_{\tindex}$, i.e., $\mathcal{C} = \mathcal{C}_{\tindex}$. It will be obvious our construction is done in such a way that it can be naturally extended to the cases where there are multiple $\mathcal{C}_{\tindex}$, $\tindex \in \ind$

\subsection{The dynamical conditions} 
\label{variable order dynamical conditions subsection}

In this subsection, we will outline the crucial dynamical conditions which will be required for our variable orders. Another set of smoothness conditions, which is in a broader sense optional, but will nevertheless be imposed for convenience, will be discussed separately in  \S \ref{variable order smooth square root subsection} below. \par

As in the standard requirements for variable order operators in the second microlocalized setting, we will require that $\vol_{+}, \vob_{+} \in \mathcal{C}^{\infty}( \psf \Xd )$ be chosen such that their restrictions to $\psf_{\dff} \Xd$ and $\tcocf$ can be identified as $\mathcal{C}^{\infty}( \mathcal{C}_{\tindex} \times \overline{\mathbb{R}^{n_{\tindex}}} )$ functions (i.e., the restrictions depend only on the free variables). Recall that this also implies that
\begin{equation*}
\begin{gathered}
{\vol_{+}}{|}_{ \psf_{\dff} \Xd } = {\vol_{+}}|_{\tcocf}, \\
{\vob_+}|_{ \psf_{\dff} \Xd } = {\vob_{+}}|_{\tcocf}
\end{gathered}
\end{equation*}
must be the same $\mathcal{C}^{\infty}( \mathcal{C}
_{\tindex} \times \overline{\mathbb{R}^{n_{\tindex}}} )$ functions. \par
Now, we will impose the dynamical condition that 
\begin{equation}  \label{var con 1.4}
\text{$\vol_{+}$ is non-increasing along the flow of $\rho_{\dmf}^{-1} \rho_{\dtsccf}^{-1} \rho_{\tcocf}^{-2} H_{p}$ in $\Sigma_{\sigma}$.}
\end{equation}
Likewise, when interpreted as a $\mathcal{C}^{\infty}( \mathcal{C}_{\tindex} \times \mathbb{R}^{n_{\tindex}} )$ function via restriction to $\psf_{\dff} \Xd$ (or $\tcocf$), we require that
\begin{equation} \label{var con 1.5}
\begin{gathered}
\text{$\vol_{+}$ is non-increasing along the flow of $\sH_{\intt}$ in $\Sigma_{\mathrm{t}}$}, \\
\text{$\vol_{+} > 0$ at $\mathcal{R}_{\mathrm{t},+}$ and $\vol_{+} < 0$ at $\mathcal{R}_{\mathrm{t},-}$.} 
\end{gathered}
\end{equation}
One then define $\vob_{+}$ to be dependent on $\vol_{+}$, i.e., we will assume that
\begin{equation} \label{var con 3}
\text{$\cob_{+} = \vob_{+} - 2 \vol_{+}$ is a constant, and $| \cob_{+} + 1 | < \frac{ n^{\tindex} - 2 }{2}$}. 
\end{equation} 
In particular, when interpreted as a $\mathcal{C}^{\infty}( \mathcal{C}_{\tindex} \times \mathbb{R}^{n_{\tindex}} )$ function as above, we must also have
\begin{equation} \label{var con 4}
\begin{gathered}
\text{$\vob_{+}$ is non-increasing along the flow of $\rho_{\dmf}^{-1} \rho_{\dtsccf}^{-1} \rho_{\tcocf}^{-2} H_{p}$ in $\Sigma_{\sigma}$, and} \\
\text{$\vob_{+}$ is non-increasing along the flow of $\sH_{\intt}$ in $\Sigma_{\mathrm{t}}$.}
\end{gathered}
\end{equation} 
We do not concern ourselves with the threshold properties of $\vob_{+}$.
\par

Moving onto the dynamical requirements for $\vor_{+} \in \mathcal{C}^{\infty}( \psf \Xd )$. First, notice that it would be enough to describe such conditions on $\bcv$, as one could argue via elliptic regularity otherwise. Moreover, since $\Sigma_{\sigma}$ intersects only $\psf_{\dmf} \Xd$ and $\dtsccf$, we only need to describe the conditions on these faces. \par

In fact, we will impose somewhat stronger conditions on $\dtsccf$, namely
\begin{equation} \label{var con 5}
\begin{gathered}
\text{$\vor_{+} - \vol_{+}$ is non-increasing along the flow of $\rho_{\dmf}^{-1} \rho_{\dtsccf}^{-1} \rho_{\tcocf}^{-2} H_{p}$ in $\bcv \cap \dtsccf$}. 
\end{gathered}
\end{equation}
We also have the corresponding threshold conditions:
\begin{equation} \label{var con 6}
\text{$\vor_{+} - \vol_{+} > -\frac{1}{2}$ at $\mathcal{R}_{\mathrm{n},\dmf,+} \cup \mathcal{R}_{\mathrm{n},\dff,+}$}, \quad \text{$\vor_{+} - \vol_{+} < - \frac{1}{2}$ at $\mathcal{R}_{\mathrm{n},\dmf,-} \cup \mathcal{R}_{\mathrm{n}, \dff,-}$.}
\end{equation} 
On the other hand, globally we just impose the standard requirement that
\begin{equation}  \label{var con 7}
\begin{gathered}
\text{$\vor_{+}$ is non-increasing along the flow of $\rho_{\dmf}^{-1} \rho_{\dtsccf}^{-1} \rho_{\tcocf}^{-2} H_{p}$ in $\bcv$},
\end{gathered}
\end{equation}
although this is slightly redundant in the sense that monotonicity for $\vor_{+}$ needs only to hold on $\Sigma_{\sigma} \cap \psf_{\dmf} \Xd$, as this property holds automatically on $\dtsccf$ by (\ref{var con 5}) and the monotonicity of $\vol_{+}$. Somewhat less standard are the threshold conditions at $\mathcal{R}_{0,\pm}$, where we require that
\begin{equation} \label{var con 8}
\text{$\vor_{+} - \vob_{+} < \frac{1}{2}$ at $\mathcal{R}_{0,+}$ and $\vor_{+} - \vob_{+} > \frac{1}{2}$ at $\mathcal{R}_{0,-}$}.
\end{equation}
Finally, on the lifts of the global, two-body radial sets $\mathcal{R}_{\mathrm{2sc},\dmf,\pm}$, we will require that
\begin{equation} \label{var con 9}
\text{$\vor_{+} > -\frac{1}{2}$ at $\mathcal{R}_{\mathrm{2sc},\dmf, +}$ and $\vor_{+} < - \frac{1}{2}$ at $\mathcal{R}_{\mathrm{2sc},\dmf, -}$}
\end{equation}
which is as expected.
 \par
In fact, from a minimalistic prospective, one could argue that the conditions at $\dtsccf$ are too strong away from $\tcocf$. Indeed, in such regions it is unnecessary to consider second microlocalization, and the analysis can be done completely within the three-body calculus, as we have shown dynamically in the previous section (however, one still needs to be careful with having a variable order at $\overline{^{\mathrm{3sc}}T^{\ast}}_{\ff}\tscX$, which is in general not allowed unless further modifications are made). Thus, if we additionally assume that
\begin{equation} \label{var con 10}
\begin{gathered}
\text{$\vol_{+}$ restricts to a $\mathcal{C}^{\infty}( ^{\mathrm{3sc}}T^{\ast} \tscX )$ function on $\Sigma_{\mf}$,} \\
\text{$\vor_{+}$ restricts to a $\mathcal{C}^{\infty}( ^{\mathrm{3sc}}T^{\ast} \tscX )$ function on $\Sigma_{\mf} \backslash \pi_{\ff}^{-1}(o_{\mathcal{C}^{\tindex}})$,}
\end{gathered}
\end{equation}
where we recall that $\pi_{\ff}$ is the natural projection ${^{\mathrm{3sc}}T^{\ast}_{\ff} \tscX} \rightarrow {^{\mathrm{sc}}T^{\ast} X^{\tindex}}$, $o_{\mathcal{C}^{\tindex}}$ is the zero section in $^{\mathrm{sc}}T^{\ast}_{\mathcal{C}^{\tindex}}X^{\tindex}$, and $\Sigma_{\mf}$ is the characteristic variety of $P$ in the three-body calculus. Then we can relax condition (\ref{var con 5}) in such regions, and instead require that
\begin{equation} \label{var con 11}
\text{$\vor_{+} - \vol_{+}$ is non-increasing along the flow of $\sH_{p,\mf}$ in $\Sigma_{\mf} \cap {^{\mathrm{3sc}}T^{\ast}_{\ff} \tscX} \backslash \pi_{\ff}^{-1}( o_{\mathcal{C}^{\tindex}} )$}.
\end{equation}
Similarly, we can also relax conditions (\ref{var con 6}) into
\begin{equation} \label{var con 12}
\text{$\vor_{+} - \vol_{+} > - \frac{1}{2}$ at $\mathcal{R}_{\mathrm{n},+}$ and $\vor_{+} - \vol_{+} < - \frac{1}{2}$ at $\mathcal{R}_{\mathrm{n},-}$}.
\end{equation}
\begin{remark}
Although conditions (\ref{var con 9})--(\ref{var con 12}) are indeed very natural, they are certainly not necessary for the analysis of this paper to work. Namely, although it is easy to construct variable orders $\vor_{+}$ for which these conditions are satisfied (as we shall see below shortly), it would be enough if they just satisfy conditions (\ref{var con 5})--(\ref{var con 6}). Notice that by smoothness, the latter conditions certainly imply (\ref{var con 9})--(\ref{var con 12}). We present here the alternatives only to illustrate the redundancy of second microlocalization away from ${\pi_{\ff}^{-1}(o_{\mathcal{C}^{\tindex}})}$.
\end{remark}
Now, (\ref{var con 1.4})--(\ref{var con 9}) covers the list of conditions for half of the variable orders. Recall that the other half is defined by
\begin{equation} \label{var con 7}
\vor_{-} = - \vor_{+} - 1, \ \vol_{-} = - \vol_{+}, \ \vob_{-} = - \vob_{+} - 2.
\end{equation}
The variable orders in (\ref{var con 7}) satisfy analogous conditions akin to (\ref{var con 1.4})--(\ref{var con 9}). Namely
\begin{equation*}
\text{$\vol_{-}$ is non-decreasing along flow of $\rho_{\dmf}^{-1} \rho_{\dtsccf}^{-1} \rho_{\tcocf}^{-2} H_{p}$ in $\bcv$,}
\end{equation*}
and, when interpreted as a $\mathcal{C}^{\infty}( \mathcal{C}_{\tindex} \times \overline{\mathbb{R}^{n_{\tindex}}} )$ function via restrictions,
\begin{equation*}
\begin{gathered}
\text{$\vol_{-}$ is non-decreasing along the flow of $\sH_{\intt}$ in $\Sigma_{\mathrm{t}}$,} \\
\text{$\vol_{-}  < 0$ at $\mathcal{R}_{\mathrm{t},+}$ and $\vol_{-} > 0$ at $\mathcal{R}_{\mathrm{t}.-}$.}
\end{gathered}
\end{equation*}
Moreover, by defining
\begin{equation*}
b_{-} = \vob_{-} - \vol_{-} = - \vob_{+} + \vol_{+} - 2,
\end{equation*}
we also have
\begin{equation*}
| b_{-} + 1 | = | - \vob_{+} + \vol_{+} - 1 | < \frac{ n^{\tindex} - 2 }{2}.
\end{equation*}
Likewise, when restricted to $\dtsccf$, we have
\begin{equation*}
\begin{gathered}
\text{$\vor_{-} - \vol_{-}$ is non-decreasing along the flow of $\rho_{\dmf}^{-1} \rho_{\dtsccf}^{-1} \rho_{\tcocf}^{-2} H_{p}$ in $\Sigma_{\sigma} \cap \dtsccf$, } \\
\text{$ \vor_{-} - \vol_{-} < - \frac{1}{2}$ at $\mathcal{R}_{\mathrm{n},\dmf, +} \cup \mathcal{R}_{\mathrm{n}, \dff, +}$}, \quad \text{$\vor_{-} - \vol_{-} > - \frac{1}{2}$ at $\mathcal{R}_{\mathrm{n},\dmf,-} \cup \mathcal{R}_{\mathrm{n}, \dff, -}$},
\end{gathered}
\end{equation*}
while globally, we have
\begin{equation*}
\text{$\vor_{-}$ is non-decreasing along the flow of $\rho_{\dmf}^{-1} \rho_{\dtsccf}^{-1} \rho_{\tcocf}^{-2} H_{p}$ in $\bcv$.}
\end{equation*}
Lastly, the remaining threshold conditions are given by
\begin{equation*}
\begin{gathered}
\text{$\vor_{-} - \vob_{-} > \frac{1}{2}$ at $\mathcal{R}_{0,+}$, $\vor_{-} - \vob_{-} < \frac{1}{2}$ at $\mathcal{R}_{0,+}$, and} \\
\text{$\vor_{-} < - \frac{1}{2}$ at $\mathcal{R}_{\mathrm{2sc},\dmf,+}$, $\vor_{-} > - \frac{1}{2}$ at $\mathcal{R}_{\mathrm{2sc},\dmf. -}$.}
\end{gathered}
\end{equation*} \par

We also have the analogies of (\ref{var con 9})--(\ref{var con 12}) in this setting. Namely 
\begin{equation*} 
\begin{gathered}
\text{$\vol_{-}$ restricts to a $\mathcal{C}^{\infty}( ^{\mathrm{3sc}}T^{\ast} \tscX )$ function on $\Sigma_{\mf}$,} \\
\text{$\vor_{-}$ restricts to a $\mathcal{C}^{\infty}( ^{\mathrm{3sc}}T^{\ast} \tscX )$ function on $\Sigma_{\mf} \backslash \pi_{\ff}^{-1}(o_{\mathcal{C}^{\tindex}})$,}
\end{gathered}
\end{equation*}
and
\begin{equation*} 
\begin{gathered}
\text{$\vor_{-} - \vol_{-}$ is non-decreasing along the flow of $\sH_{p,\mf}$ in $\Sigma_{\mf} \cap {^{\mathrm{3sc}}T^{\ast}_{\ff} \tscX} \backslash \pi_{\ff}^{-1}( o_{\mathcal{C}^{\tindex}} )$}, \\
\text{$\vor_{-} - \vol_{-} < - \frac{1}{2}$ at $\mathcal{R}_{\mathrm{n},+}$,} \quad \text{$\vor_{-} - \vol_{-} > - \frac{1}{2}$ at $\mathcal{R}_{\mathrm{n},-}$}.
\end{gathered}
\end{equation*}
\par
The point of these constructions is as follows: Let $m_{-} = - m_{+} + 2$, $s_{-} = - s_{+} + 2$ where $m_{+}, s_{+} \in \mathbb{R}$ are arbitrary. Then by Proposition \ref{prop: dual spaces of the Sobolev spaces}, we have
\begin{equation*}
\begin{gathered}
H_{\mathrm{d3sc,3co,res}}^{ m_{-}, \vor_{-}, \vol_{-}, \vor_{-} + \vol_{-}, \vob_{-}, s_{-} } = \left( H_{\mathrm{d3sc,3co,res}}^{ m_{+} - 2, \vor_{+} + 1, \vol_{+}, \vor_{+} + \vol_{+} + 1, \vob_{+} + 2, s_{+} - 2 } \right)^{\ast}, \\
H_{\mathrm{d3sc,3co,res}}^{ m_{-} - 2, \vor_{-} + 1, \vol_{-}, \vor_{-} + \vol_{-} + 1, \vob_{-} + 2, s_{-} - 2 } = \left( H_{\mathrm{d3sc,3co,res}}^{m_{+}, \vor_{+}, \vol_{+}, \vor_{+} + \vol_{+}, \vob_{+}, s_{+} } \right)^{\ast}.
\end{gathered}
\end{equation*}
Thus, the estimates
\begin{align} \label{variable order section semi-Fredholm estimates}
\begin{split}
\| u \|_{ H_{\mathrm{d3sc,3co,res}}^{ m_{ \pm }, \vor_{ \pm }, \vol_{ \pm }, \vor_{ \pm } + \vol_{ \pm }, \vob_{ \pm }, s_{ \pm }  } } & \leq C ( \| P u \|_{ H_{\mathrm{d3sc,3co,res}}^{ m_{ \pm } - 2, \vor_{ \pm } + 1, \vol_{ \pm }, \vor_{ \pm } + \vol_{\pm  } + 1, \vob_{ \pm } + 2, s_{ \pm } - 2 } } \\
& \quad + \| u \|_{H_{\mathrm{d3sc,3co,res}}^{M, \vor_{\pm} - \delta, \vol_{\pm} - \delta, \vor_{\pm} + \vol_{\pm} - \delta, \vob_{\pm} - \delta, S} } )
\end{split}
\end{align}
are exactly dual to each other. In particular, we would have shown that the maps (\ref{introduction; main maps}) is indeed Fredholm, provided both of the semi-Fredholm estimates in (\ref{variable order section semi-Fredholm estimates}) are satisfied.

\subsection{Explicit construction}
\label{variable order explicit construction subsection}
We now construct explicit examples of variable orders satisfying the aforementioned dynamical conditions. It is enough to do this corresponding to the variable orders with plus signs, i.e., we will construct $\vor, \vol$ such that in the above discussion, we can take
\begin{equation*}
\vol_{+} = \vol,  \ \vor_{+} = \vor.
\end{equation*} \par

Before we proceed, let us remark once again that the variable orders constructed in this subsection will not be the ones employed in the rest of this paper, for there is another set of smoothness conditions which will be imposed in \S \ref{variable order smooth square root subsection} below. We separate the discussion since the construction in this subsection is already very technical. Moreover, the modifications in \S \ref{variable order smooth square root subsection} will be built directly from the construction in this subsection. \par

Returning to the current discussion. First, we will construct a $\mathcal{C}^{\infty}(\psf \Xd)$ function $\vol$ satisfying conditions (\ref{var con 1.5}), which also restricts to a $\mathcal{C}^{\infty}( \mathcal{C}_{\tindex} \times \overline{\mathbb{R}^{n}} )$ function at $\tcocf$ and $\psf_{\dff} \Xd$. Let $\psi_1, \psi_{2} \in \mathcal{C}^{\infty}( \mathbb{R} )$ be chosen such that $\psi_{1}$ is a cut-off at $0$, $\psi_{2}$ is supported on $( -\infty , F]$, with $\psi_{2} = 1$ on $( - \infty, F/2]$ where $F > 0$ will be large. We can arrange $\psi_{2}$ to be identically $1$ on the support of $\psi_{1}$. We also write $\varphi_j = \psi_{j}(p)$ for $j=1,2$, where $p$ is the principal symbol of $P$, and
\begin{equation*}
\tilde{\varphi}_{j} = \psi_{j}( | \ltau |^{2} + |\lmu|_{h_{\tindex}}^{2} - \lambda^{2}), \quad j =1,2,
\end{equation*}
which is of course only locally valid. Furthermore, let $\varphi_{3} \in \mathcal{C}^{\infty}( \psf \Xd )$ be a cut-off function at $\psf_{\dff} \Xd$. Then for $\beta > 0$ we can define
\begin{equation} \label{example variable order l}
\vol = \beta \varphi_{1} \tau + \beta ( 1 - \varphi_{1} ) \tilde{\varphi}_{2} \varphi_{3} \ltau + \beta( 1 - \varphi_1 ) \varphi_{2} (1 - \varphi_{3} ) \tau,
\end{equation}
where $\tau$ is the function dual to $dx/x^2$ for $x = |z|^{-1}$. In local coordinates, it is easy to see
\begin{equation} \label{variable order section expression for tau near cf}
\tau = \frac{1}{( 1 + \bff^{2} )^{1/2}} \ltau + \frac{ \bff }{ ( 1 + \bff^{2} )^{1/2} } \utau
\end{equation}
as a function on $ {^{\mathrm{3sc}}T^{\ast}_{\mf}} \Xo$ near ${ ^{\mathrm{3sc}}T^{\ast}_{\ff} \Xo }$. Thus, the restriction of $\tau$ to $^{\mathrm{3sc}}T^{\ast}_{\ff} \Xo$, and subsequently to $\psf_{\cf} \Xd$ and $\psf_{\dff} \Xd$, is simply given by $\ltau$, and thus depend only on the free variables. \par

Notice that (\ref{example variable order l}) is indeed a $\mathcal{C}^{\infty}( \psf \Xd )$ function. This fact holds quite obviously for its first and third component. For the second component, one just need to observe that although $\tilde{\varphi}_{2} \ltau$ is independent of the interaction variables (it is bounded in the free variables, so there is no issue there), there is no singularity as the fiber variables get large in a small neighborhood of $\psf_{\dff} \Xd$ (which is the case on the support of $\varphi_3$) by nature of the blow-up creating the resolved face $\rf$.
\par

Now, $\vol$ restricts to a $\mathcal{C}^{\infty}_{c}( \mathcal{C}_{\tindex} \times \mathbb{R}^{n_{\tindex}} )$ function at $\tcocf$ and $\psf_{\dff} \Xd$. To see this, we simply note that since the restriction of $p$ to $\tcocf$ is $\intt - \lambda^{2}$ (i.e., by letting $\intn \rightarrow 0$), the restriction of $\vol$ to $\tcocf$ must be
\begin{align*}
\vol|_{\tcocf} & = \beta\tilde{\varphi}_{1} \tilde{\varphi}_2 \ltau + \beta ( 1 - \tilde{\varphi}_1 ) \tilde{\varphi}_2 \varphi_{3} \ltau + \beta ( 1 - \tilde{\varphi}_1 ) \tilde{\varphi_{2}} ( 1-  \varphi_{3} ) \ltau \\
& = \beta \tilde{\varphi}_1 \tilde{\varphi}_2 \ltau + \beta ( 1 - \tilde{\varphi}_1 ) \tilde{\varphi}_2 \ltau \\
& = \beta \tilde{\varphi}_2 \ltau.
\end{align*} 
On the other hand, suppose that the constant $F > 0$ in the definition of $\psi_{2}$ is chosen to be large enough. Then one even has $\tilde{\varphi}_{2}$ being identically $1$ on the support of $\varphi_{1}$, whenever the former expression is valid. This in particular holds if we restrict our attention to $\psf_{\dff} \Xd$. Thus we have
\begin{equation*}
\vol|_{\psf_{\dff} \Xd} = \beta \varphi_{1} \tilde{\varphi}_{2} \ltau + \beta  ( 1 - \varphi_{1} ) \tilde{\varphi_{2}} \ltau = \beta \tilde{\varphi}_2 \ltau,
\end{equation*}
thereby showing the required restriction properties. \par

Finally, we show that assumptions (\ref{var con 1.5}) are satisfied. We have already shown that the restriction is $\beta \tilde{\varphi}_2 \ltau$. By construction, $\tilde{\varphi}_2$ is identically $1$ on $\Sigma_{\mathrm{t}}$. Thus it suffices to compute
\begin{equation*}
\sH_{\intt} ( \beta \ltau )  = - 2 \beta | \lmu |_{h_{\tindex}}^2 \leq 0
\end{equation*}
which proves monotonicity. The threshold conditions are satisfied since $\ltau = \pm \lambda$ at $\mathcal{R}_{\mathrm{t},\pm}$.  \par

Next, we move onto the construction of a $\mathcal{C}^{\infty}( \psf \Xd )$ function $\vor$ satisfying the remaining conditions. This will be done with respect to the function $\vol$ constructed above. Note that the behavior of $\vor$ is unimportant at $\tcocf$ and $\psf_{\dff} \Xd$, so that by means of a cut-off, it suffices to construct $\vor$ at the `symbolic faces' only. In fact, since $\Sigma_{\sigma}$ is supported away from fiber infinity and $\rf$, it will be enough if we restrict our attention to just $\psf_{\dmf}\Xd$ and $\dtsccf$. \par

Let $\varphi_1$ be the cut-off near $\Sigma_{\sigma}$ constructed above. Let $\varphi_{4} \in \mathcal{C}^{\infty}( \overline{ ^{\mathrm{3sc}}T^{\ast} } \tscX )$ be a cut-off at $\overline{^{\mathrm{3sc}}T^{\ast}}_{\ff} \tscX$ such that $\varphi_{4} = 1$ in a neighborhood of $\overline{ ^{\mathrm{3sc}}T^{\ast}}_{\ff} \tscX$. Furthermore, let $\psi_{3} \in \mathcal{C}^{\infty}_{c}( [0, \infty) )$ be a cut-off at $0$ such that $\psi_{3} = 1$ in a neighborhood of $0$, $\psi_{3}' \leq 0$, and set
\begin{equation*}
 \varphi_{5} = \psi_{3}( \intnres ).
\end{equation*}
Then assuming support near $\tcocf$, we can explicitly compute that
\begin{equation*} \label{variable order construction 0.25}
( \hat{x}^{\tindex} )^{-1} \bff^{-2}   H_{p} \varphi_{5} = - ( \ltau - \utaures )  (\intnres)  \psi_3'( | \utaures |^2 + | \umures |_{h^{\tindex}}^2 )  .
\end{equation*}
Thus, the signs of $( \hat{x}^{\tindex} )^{-1} \bff^{-2}  H_{p} \varphi_5$ are determined near $\mathcal{R}_{0,\pm}$, i.e.,
\begin{equation} \label{variable order construction 0.3}
( \hat{x}^{\tindex} )^{-1} \bff^{-2} H_{p} \varphi_5 \geq 0 \ \text{near $\mathcal{R}_{0,+}$}, \quad (\hat{x}^{\tindex})^{-1} \bff^{-2} H_{p} \varphi_5 \leq 0 \ \text{near $\mathcal{R}_{0,-}$}.
\end{equation}
Finally, we will let $\varphi_{6}$ be a cut-off function at $\mathcal{R}_{\mathrm{2sc},\dmf, \pm}$ such that 
\begin{equation} \label{variable order construction 0.4}
\rho_{\dmf}^{-1}   \rho_{\tcocf}^{-2} H_{p} \varphi_{6} \leq 0 \ \text{near $\mathcal{R}_{\mathrm{2sc},\dmf, +}$}, \quad \rho_{\dmf}^{-1} \rho_{\tcocf}^{-2} H_{p} \varphi_{6} \geq 0 \ \text{near $\mathcal{R}_{\mathrm{2sc},\dmf, -}$}.
\end{equation}
Note that such a cut-off exists by standard dynamical constructions since $\mathcal{R}_{\mathrm{2sc},\dmf,\pm}$ are global sources and sinks (alternatively, see \S \ref{global radial point estimate section} for explicit construction of quadratic defining functions for $\mathcal{R}_{\mathrm{2sc},\dmf, \pm}$). Then we will require that
\begin{equation}  \label{variable order example construction of full r hat}
\vor = - \frac{1}{2} + \varphi_{1}  \tilde{\vor} + \vol,
\end{equation} 
where $\tilde{\vor} \in \mathcal{C}^{\infty}( \psf \Xd)$ is defined by
\begin{equation} \label{example variable order r0}
\tilde{\vor} =   \beta \tau  \varphi_6  - \tilde{\beta} \tau \varphi_4 \varphi_5 + \varphi_4 ( 1 - \varphi_5 ) ( 1 - \varphi_6 ) \vor^{\tindex}
\end{equation}
for some large $ \tilde{\beta} > 0$ such that $\beta > \tilde{\beta}$. Here, for $\beta^{\tindex} > 0$ we set
\begin{equation*}
\vor^{\tindex} =  \beta^{\tindex} \frac{\utau}{ ( \intn )^{1/2} }
\end{equation*}
away from $\tcocf$. This can be naturally extended into a neighborhood of $\tcocf$, except for at the points $\{ \utaures = 0, \umures = 0 \}$, by prescribing
\begin{equation*}
\begin{gathered}
\vor^{\tindex} = \beta^{\tindex} \frac{\utaures}{ ( \intnres )^{1/2} } \ \text{near $\psf_{\dmf}\Xd$, and} \\
\vor^{\tindex} = \beta^{\tindex} \frac{\utaub}{ ( \intnb )^{1/2} } \ \text{near $\psf_{\dff} \Xd$.}
\end{gathered}
\end{equation*}
We also remark that
\begin{equation*}
\varphi_4 ( 1 - \varphi_5 ) ( 1 - \varphi_6 ) \vor^{\tindex} = \varphi_4 ( 1 - \varphi_5 - \varphi_6 ) \vor^{\tindex}
\end{equation*}
if the support of $\psi_3$ is small enough (so indeed $\varphi_5$ and $\varphi_6$ have disjoint supports).
\par

We proceed to investigate the dynamical properties of $\vor$. Since we do this only on the characteristic set, we can assume $\varphi_{1} = 1$ in the calculations below. We first consider monotonicity. Note that
\begin{align*}
H_{p} \vor = &  - 2 \beta x | \mu |_{h}^2 - 2 \beta x | \mu |_{h}^{2} \varphi_6 + 2 \tilde{\beta} x | \mu |_{h}^2 \varphi_4 \varphi_5  + \varphi_4 ( 1 - \varphi_5 - \varphi_6 ) H_{p} \vor^{\tindex} \\
&  +  ( - \tilde{\beta} \tau - \vor^{\tindex} ) \varphi_4 H_{p} \varphi_5 + ( \beta  \tau - \vor^{\tindex} \varphi_4 ) H_{p} \varphi_6 + ( - \tilde{\beta} \tau \varphi_5 + ( 1 - \varphi_5 - \varphi_6 ) \vor^{\tindex} ) H_{p} \varphi_4.
\end{align*}
We will also split $-2 \beta x |\mu|_{h}^2$ into three parts in the below. \par

It is easy to see that
\begin{equation}  \label{variable order construction 0.5}
\rho_{\dmf}^{-1} \rho_{\dtsccf}^{-1} \rho_{\tcocf}^{-2} H_{p} \vor^{\tindex} \leq 0
\end{equation}
wherever it is defined. Indeed, away from $\tcocf$, $x^{\tindex} \simeq \rho_{\dmf} \rho_{\dtsccf}$, and we have
\begin{equation*}
(x^{\tindex})^{-1} H_{p} \vor^{\tindex} = - 2 \beta^{\tindex} \frac{ | \umu |_{h^{\tindex}}^2 }{ ( \intn )^{1/2} } \leq 0;
\end{equation*}
while in a neighborhood of $\tcocf$, we have 
\begin{equation*}
\begin{gathered}
(\hat{x}^{\tindex})^{-1} \rhoresf^{-1} \bff^{-2} H_{p} \vor^{\tindex} = - 2 \beta^{\tindex}  |\umures|_{h^{\tindex}}^2, \, \hat{x}^{\tindex} \simeq \rho_{\dmf}, \, \bff \simeq \rho_{\tcocf} \ \text{near $\psf_{\dmf}\Xd$} \\
\rhobf^{-1} (x^{\tindex})^{-2} H_{p} \vor^{\tindex} = - 2 \beta^{\tindex}  |\umub|_{h^{\tindex}}^2, \, x^{\tindex} \simeq \rho_{\dff}, \ \text{near $\psf_{\dff} \Xd$,}
\end{gathered}
\end{equation*}
where $\rhoresf = ( \intnres )^{-1/2}$, $\rhobf = ( \intnb )^{-1/2}$ can be taken as local defining functions for $\tcocf$ whenever they are valid. Hence (\ref{variable order construction 0.5}) must hold.  \par

By construction, we also have
\begin{align}  
\rho_{\dmf}^{-1} \rho_{\dtsccf}^{-1} \rho_{\tcocf}^{-2}  \big( ( - \tilde{\beta} \tau - \vor^{\tindex} )  \varphi_4  H_{p} \varphi_{5} - \frac{2}{3} \beta  x | \mu |_{h}^2 \big)  & \leq 0,  \label{variable order construction 1.1}  \\
  \rho_{\dmf}^{-1} \rho_{\dtsccf}^{-1} \rho_{\tcocf}^{-2}( \beta \tau - \vor^{\tindex} \varphi_4 )  H_{p} \varphi_6 & \leq 0, \label{variable order construction 1.2}
\end{align}
where (\ref{variable order construction 1.2}) follows directly from (\ref{variable order construction 0.4}) and the facts that $\tau > 0$ at $\mathcal{R}_{\mathrm{2sc},\dmf,+}$, $\tau < 0$ at $\mathcal{R}_{\mathrm{2sc},\dmf,-}$. Thus if $\beta$ is large enough, then these inequalities must also hold for $\beta \tau - \vor^{\tindex} \varphi_4$. \par

To show (\ref{variable order construction 1.1}), it is enough to compute near the support of $\rho_{\dmf}^{-1} \rho_{\dtsccf}^{-1} \rho_{\tcocf}^{-2} \varphi_4 H_{p}\varphi_5$, where $\rho_{\dmf} \rho_{\dtsccf} \rho_{\tcocf}^2 \simeq x \simeq \hat{x}^{\tindex} \bff^2$. Then (\ref{variable order construction 1.1}) is equivalent to
\begin{equation} \label{variable order construction 1.3}
( - \tilde{\beta} \tau - \vor^{\tindex} )  \varphi_4 (\hat{x}^{\tindex})^{-1} \bff^{-2}  H_{p} \varphi_{5} - \frac{2}{3} \beta | \mu |_{h}^2 \leq 0.
\end{equation}
Near the support of $\varphi_4 \varphi_5$, we can write
\begin{equation}  \label{variable order construction 1.4}
| \mu |_{h}^{2} =  \frac{\bff^2}{ ( 1 + \bff^2 )^{1/2} } ( \ltau - \utaures )^2 + | \lmu |_{h_{\tindex}}^2 +  \bff^{2} | \umures |_{h^{\tindex}}^2.
\end{equation}
Thus if $|\lmu|_{h_{\tindex}} > 0$, then $-2\beta |\mu|_{h}^2/3$ can be made dominant for large enough $\beta$, in which case (\ref{variable order construction 1.3}) is actually strictly negative. If $|\lmu|_{h_{\tindex}} = 0$, then on the characteristic variety we must be also near $\mathcal{R}_{0,\pm}$. Then by (\ref{variable order section expression for tau near cf}), (\ref{variable order construction 0.3}) and the facts $\ltau > 0$ near $\mathcal{R}_{0,+}$, $\ltau < 0$ near $\mathcal{R}_{0,-}$, we can conclude that $( - \tilde{\beta} \tau - \vor^{\tindex} )  \varphi_4 (\hat{x}^{\tindex})^{-1} \bff^{-2}  H_{p} \varphi_{5} \leq 0$ near $\mathcal{R}_{0,\pm}$, without the need for $- 2\beta |\mu|_{h}^2/3$.
\par

Likewise, we can show that
\begin{equation*}
2 \rho_{\dmf}^{-1} \rho_{\dtsccf}^{-1} \rho_{\tcocf}^{-2}  \big( - \frac{1}{3} \beta x | \mu |_{h}^2 + \tilde{\beta} x | \mu |_{h}^2 \varphi_4 \varphi_5 \big) \leq 0,
\end{equation*}
for near the support of $\varphi_4 \varphi_5$, we can use the middle term in (\ref{variable order construction 1.4}) to have $\beta |\mu|_{h}^2/3$ dominate $\tilde{\beta} | \mu|_{h}^2 \varphi_4 \varphi_5$, assuming that $\beta$ is large relative to $\tilde{\beta}$. \par

Finally, we will show that
\begin{equation*}  
\rho_{\dmf}^{-1} \rho_{\dtsccf}^{-1} \rho_{\tcocf}^{-2} \big( ( - \tilde{\beta} \tau \varphi_5 + ( 1 - \varphi_5 - \varphi_6 ) \vor^{\tindex} ) H_{p} \varphi_4 - \frac{2}{3} \beta x |\mu|_{h}^2 \big) \leq 0,
\end{equation*}
or rather, if we just look at a neighborhood of the support of $\rho_{\dmf}^{-1} \rho_{\dtsccf}^{-1} \rho_{\tcocf}^{-2} H_{p} \varphi_4$, then as before it would be enough to show that
\begin{equation} \label{variable order construction 1.5}
( - \tilde{\beta} \tau \varphi_5 + ( 1 - \varphi_5 - \varphi_6 ) \vor^{\tindex} ) (x^{\tindex})^{-1} H_{p} \varphi_4 - \frac{2}{3} \beta |\mu|_{h}^2 \leq 0.
\end{equation}
Recall that $\mathcal{R}_{\mathrm{2sc},\dmf,\pm}$ are the lifts of $\mathcal{R}_{\mathrm{2sc}}$ (i.e., the standard two-body radial sets in $^{\mathrm{sc}}T^{\ast}_{\partial \Xo} \Xo$) to $\psf \Xd$, so they are simply reduced to $\mathcal{R}_{\mathrm{sc},\pm}$ on the support of $H_{p} \varphi_4$ (where the three-body structure is irrelevant). We know that $ - \tilde{\beta} \tau \varphi_5 + ( 1 - \varphi_5 - \varphi_6 ) \vor^{\tindex}$ is supported away from $\mathcal{R}_{\mathrm{sc},\pm}$ where $\varphi_4$ vanishes. But since we are also on the characteristic variety, we must have $|\mu|_{h}^2 > 0$ in such regions. Thus the strict negativity of $-  2\beta |\mu|_{h}^2/3$ can be used to dominate (\ref{variable order construction 1.5}) if $\beta$ is large enough. This concludes the proof of monotonicity for $\vor$ in the second microlocal framework. \par

Next, we show that the threshold conditions for $\vor$ are also satisfied. This is obvious at $\dtsccf$, since we have 
\begin{equation*}
\begin{gathered}
\vor - \vol = -1/2 + \vor^{\tindex} \ \text{at} \ \dtsccf, \ \text{where} \ \vor^{\tindex}  = \pm \beta^{\tindex} \ \text{at} \ \mathcal{R}_{\mathrm{n},\dmf, \pm} \cup \mathcal{R}_{\mathrm{n},\dff, \pm}.
\end{gathered}
\end{equation*}
Thus (\ref{var con 6}) must be true. To show that the thresholds at $\mathcal{R}_{\mathrm{2sc},\dmf,\pm}$ and $\mathcal{R}_{0,\pm}$ are also satisfied, we just note that
\begin{equation*}
\begin{gathered}
\vor  = - \frac{1}{2}  + 2 \vol \ \text{at} \ \mathcal{R}_{\mathrm{2sc},\dmf, \pm}, \ \vor - \vob =  \frac{1}{2} - ( b_0 + 1 ) - \tilde{\beta} \tau - \vol \ \text{at} \ \mathcal{R}_{0,\pm}, \ \text{where} \\
\vol = \pm \beta \lambda \ \text{at} \ \mathcal{R}_{\mathrm{2sc},\dmf,\pm} \cup \mathcal{R}_{0,\pm}, \quad  \tau = \pm \lambda \ \text{at} \ \mathcal{R}_{0,\pm}.
\end{gathered}
\end{equation*}
Thus (\ref{var con 9}) must also be true. \par

To get (\ref{var con 8}), it suffices to take one of $\beta, \tilde{\beta}$ large enough so that $- \tilde{\beta} \tau - \vol$ dominates $-(b_{0} + 1)$, from which we can ensure that (\ref{var con 8}) must be true as well. This shows that the threshold conditions within the second microlocalized framework are all satisfied. \par

It is also easy to see that $\vol,\vor$ satisfy the additional smoothness assumptions (\ref{var con 10}). Moreover, the corresponding monotonicity and threshold conditions away from $\tcocf$, namely (\ref{var con 11}) and (\ref{var con 12}), are all satisfied as well. Therefore, this concludes the construction of a set of variable orders satisfying all of the required conditions.

\subsection{Modification for smooth square roots} 
\label{variable order smooth square root subsection}
Finally, we modify the constructions in \S \ref{variable order explicit construction subsection} so that, in addition to the dynamic requirements imposed in \S \ref{variable order dynamical conditions subsection}, we have
\begin{equation} \label{variable order smooth square root cal 1}
\begin{gathered}
\text{$(- \rho_{\dmf}^{-1} \rho_{\dtsccf}^{-1} \rho_{\tcocf}^{-2} H_{p} \vor )^{1/2}$, $(- \rho_{\dmf}^{-1} \rho_{\dtsccf}^{-1} \rho_{\tcocf}^{-2} H_{p} \vol )^{1/2}$} \\
\text{are $\mathcal{C}^{\infty}( \psf \Xd )$ functions in a neighborhood of $\bcv$.}
\end{gathered}
\end{equation}
Moreover, if $\vol$ is interpreted as a $\mathcal{C}^{\infty}( \mathcal{C}_{\tindex} \times \overline{\mathbb{R}^{n_{\tindex}}} )$ function via restriction, then we require
\begin{equation} \label{variable order smooth square root cal 2}
\text{$(-H_{\intt} \vol)^{1/2}$ is a $\mathcal{C}^{\infty}( \mathcal{C}_{\tindex} \times \overline{\mathbb{R}^{n_{\tindex}}} )$ function in a neighborhood of $\Sigma_{\mathrm{t}}$}.
\end{equation}
Lastly, we will require that
\begin{equation} \label{variable order smooth square root cal 3}
\begin{gathered}
\text{$(- \rho_{\dmf}^{-1} \rho_{\dtsccf}^{-1} \rho_{\tcocf}^{-2} H_{p} ( \vor - \vol ) )^{1/2}$ is a $\mathcal{C}^{\infty}( \psf \Xd )$ function} \\
\text{in a neighborhood of $\bcv \cap \dtsccf$},
\end{gathered}
\end{equation} 
and, if $\vor$, $\vol$ are interpreted as $\mathcal{C}^{\infty}( ^{\mathrm{3sc}}T^{\ast}[ \overline{\mathbb{R}^{n}} ; \mathcal{C}_{\tindex} ] )$ functions near $\Sigma_{\mf} \backslash \pi_{\ff}^{-1}( o_{\mathcal{C}^{\tindex}} )$, then
\begin{equation} \label{variable order smooth square root cal 4}
\begin{gathered}
\text{$( - \sH_{p, \mf}( \vor - \vol ) )^{1/2}$ is a $\mathcal{C}^{\infty}( { ^{\mathrm{3sc}}T^{\ast} [ \overline{\mathbb{R}^{n}} ; \mathcal{C}_{\tindex} ]  } )$ function } \\
\text{in a neighborhood of $\Sigma_{\mf} \cap {^{\mathrm{3sc}}T^{\ast}_{\ff} \tscX} \backslash \pi_{\ff}^{-1}( o_{\mathcal{C}^{\tindex}} )$ }.
\end{gathered}
\end{equation}
\par

Notice that conditions (\ref{variable order smooth square root cal 1})--(\ref{variable order smooth square root cal 4}) are indeed not already satisfied by the construction in \S \ref{variable order explicit construction subsection}, since for instance the function $|\mu|_{h}$ is not smooth at $\mu = 0$. Such an issue is also present in the two-body case, where the standard choice of variable order is $-1/2 + \tau$ (restricted to a neighborhood of the characteristic variety by means of a cut-off). In that case, an easy solution is to instead consider a function of the form $\phi(\tau)$, where $\phi \in \mathcal{C}^{\infty}(\mathbb{R})$ satisfies $\phi' \geq 0$ and $\phi(t) = \pm \epsilon_{\pm}$ constants in a small neighborhood of $t = \pm \lambda_{\pm}$ for some $\epsilon_{\pm} > 0$ small. \par

We will apply a similar strategy in the present context. Thus, let $\phi_{1}, \phi_{2} \in \mathcal{C}^{\infty}( \mathbb{R} )$ be such that $\phi_1(t) = \pm \epsilon_{1,{\pm}}$ in a small neighborhood of $t = \pm \lambda$, $\phi_{2}(t) = \pm \epsilon_{2, \pm}$ in a small neighborhood of $t = \pm 1$, and $\phi_{1}', \phi_{2}' \geq 0$. Moreover, let $\hat{\vol}$ be the function (\ref{example variable order l}). Then we modify $\vor$, $\vol$ by
\begin{equation*}
\vol = \beta \phi_{1} ( \beta^{-1} \hat{\vol} ) = \beta \phi_{1} ( \varphi_1 \tau + ( 1 - \varphi_1 ) \tilde{\varphi}_2 \varphi_3 \ltau + \beta ( 1 - \varphi_1 ) \varphi_2 ( 1 - \varphi_3 ) \tau  ),
\end{equation*}
and once again $\vor = - 1/2 + \varphi_1 \tilde{\vor} + \vol$, except that we now set
\begin{equation*}
\tilde{\vor} = \beta \phi_{1}( \tau ) \varphi_6 - \tilde{\beta} \phi_{1}( \tau ) \varphi_4 \varphi_5 + \varphi_4 ( 1 - \varphi_5 ) ( 1 - \varphi_6 ) \beta^{\tindex} \phi_2( (\beta^{\tindex})^{-1} \vor^{\tindex} ).
\end{equation*} \par

With these modifications, it is easy to check that the dynamical conditions required in \S \ref{variable order dynamical conditions subsection} are still satisfied. Indeed, the compositions of $\tau $ with $\phi_1$ and $\vor^{\tindex}$ with $\phi_2$ do not change monotonicity properties along the flows. Moreover, their values at the radial sets are also preserved by constructions. One can then follow through the calculations in \S \ref{variable order explicit construction subsection} quite easily. We leave to the readers to make these straightforward verifications. \par

Lastly, the obvious analogues to (\ref{variable order smooth square root cal 1})--(\ref{variable order smooth square root cal 3}) must also hold in the case where the variables orders are $\vor_{-}$, $\vol_{-}$, as determined by formulae (\ref{var con 7}) (recall that we were writing $\vor = \vor_{+}$, $\vol = \vol_{+}$ in \S \ref{variable order explicit construction subsection}). To be precise, we must also have
\begin{equation}  
\begin{gathered}
\text{$(\rho_{\dmf}^{-1} \rho_{\dtsccf}^{-1} \rho_{\tcocf}^{-2} H_{p} \vor_{-} )^{1/2}$, $(\rho_{\dmf}^{-1} \rho_{\dtsccf}^{-1} \rho_{\tcocf}^{-2} H_{p} \vol_{-} )^{1/2}$}. \\
\text{are $\mathcal{C}^{\infty}( \psf \Xd )$ functions in a neighborhood of $\bcv$.}
\end{gathered}
\end{equation}
If $\vol_{-}$ is interpreted as a $\mathcal{C}^{\infty}( \mathcal{C}_{\tindex} \times \overline{\mathbb{R}^{n_{\tindex}}} )$ function via restriction, then
\begin{equation}
\text{$(H_{\intt} \vol_{-} )^{1/2}$ is a  $\mathcal{C}^{\infty}( \mathcal{C}_{\tindex} \times \mathbb{R}^{n_{\tindex}} )$ function in a neighborhood of $\Sigma_{\mathrm{t}}$}. 
\end{equation}
Moreover, we require that
\begin{equation}
\begin{gathered}
\text{$( \rho_{\dmf}^{-1} \rho_{\dtsccf}^{-1} \rho_{\tcocf}^{-2} H_{p} ( \vor_{-} - \vol_{-} ) )^{1/2}$ is a $\mathcal{C}^{\infty}( \psf \Xd )$ function} \\ 
\text{in a neighborhood of $\bcv \cap \dtsccf$}.
\end{gathered}
\end{equation}
Finally, if $\vor_{-}$, $\vol_{-}$ are interpreted as $\mathcal{C}^{\infty}( ^{\mathrm{3sc}}T^{\ast} [ \overline{\mathbb{R}^{n}} ; \mathcal{C}_{\tindex} ]  )$ functions near $\Sigma_{\mf} \backslash \pi_{\ff}^{-1}( o_{\mathcal{C}^{\tindex}} )$, then
\begin{equation}
\begin{gathered}
\text{$( \sH_{p,\mf} ( \vor_{-} - \vol_{-} ) )^{1/2}$ is a $\mathcal{C}^{\infty}( { ^{\mathrm{3sc}}T^{\ast} } [ \overline{\mathbb{R}^{n}} ; \mathcal{C}_{\tindex} ] )$ function} \\
\text{ in a neighborhood of $\Sigma_{\mf} \cap {^{\mathrm{3sc}}T^{\ast}_{\ff} \tscX} \backslash \pi_{\ff}^{-1}( o_{\mathcal{C}^{\tindex}} )$ }.
\end{gathered}
\end{equation} \par

This concludes our discussion on the modifications.
\section{Scattering adjoint and commutator}
\subsection{Indicial operator of scattering adjoint}
In the consideration of indicial operators at $\cf$, we are morally working directly in the three-cone framework (i.e., we look at second microlocalization from the perspective of blowing up the $\mathrm{3co}$-calculus). It is therefore unsurprising that adjoints behave most naturally with respect to three-cone densities, as opposed to scattering densities. \par

Indeed, as we have already shown in Proposition \ref{proposition symbol and indicial operators for adjoint}, one has 
\begin{equation} \label{equation for adjoint second paper indicial operator at cf}
\sigma(A^{\ast_{\mathrm{3co}}}) = \overline{\sigma(A)}, \  \hat{N}_{\cf}( A^{\ast_{\mathrm{3co}}} ) = \hat{N}_{\cf}(A)^{\ast_{\mathrm{co}}}, \ \hat{N}_{\dff}( A^{\ast_{\mathrm{3co}}} ) = \hat{N}_{\dff}(A)^{\ast_{\mathrm{b}}}
\end{equation}
for any $A \in \Psf^{\vom, \vor, \vol, \vov, \vob, \vos}$ with sufficiently nice classicality properties. Here, the subscripts `3co', `co' and `b' abbreviate adjoints taken respectively with respect to some fixed positive 3co-density, and families of (with parameters in $\mathcal{C}_{\tindex}$, i.e., the base of the fibration defining $\cf$ resp. $\dff$) positive co- and b-densities (though this is less natural since there are no canonical global choices of such objects). \par

In fact, the relationship between these densities can be made concrete: Away from $\cf$, any 3co density $\nu_{\mathrm{3co}}$ simply reduces to a scattering density (since the three-cone structure, by design, must reduces to the three-body structure away from $\cf$), while in a neighborhood of $\cf$, we can write $\nu_{\mathrm{3co}}$ in local frames as 
\begin{equation} \label{positive density determine from 3co to co 1}
\nu_{\mathrm{3co}} = f_{\dff} \Big| \frac{dx_{\tindex}}{x_{\tindex}^{n_{\tindex}+1}}  dy_{\tindex} \frac{d\hbff}{\hbff} dy^{\tindex} \Big|, \quad  \nu_{\mathrm{3co}} = f_{\dmf} \Big| \frac{dx_{\tindex}}{x_{\tindex}^{n_{\tindex}+1}} dy_{\tindex} \frac{d \hat{x}^{\tindex}}{ (\hat{x}^{\tindex})^{n^{\tindex} + 1} } dy^{\tindex} \Big|
\end{equation}
respectively near $\dff$ and $\dmf$ for some positive functions $f_{\dff}, f_{\dmf} \in \mathcal{C}^{\infty}(\Xd)$. Then the corresponding family of co-densities required for (\ref{equation for adjoint second paper indicial operator at cf}) to hold is found by sending
\begin{equation} \label{positive density determine from 3co to co 2}
\begin{gathered}
 \Big| \frac{dx_{\tindex}}{x_{\tindex}^{n_{\tindex}+1}}  dy_{\tindex} \frac{d\hbff}{\hbff} dy^{\tindex} \Big| \mapsto \Big| \frac{d\hbff}{\hbff} dy^{\tindex} \Big|, \\
   \Big| \frac{dx_{\tindex}}{x_{\tindex}^{n_{\tindex} + 1}} dy_{\tindex} \frac{d \hat{x}^{\tindex}}{ (\hat{x}^{n^{\tindex}})^{n^{\tindex} + 1} } dy^{\tindex} \Big| \mapsto \Big| \frac{d \hat{x}^{\tindex} }{ ( \hat{x}^{\tindex} )^{n^{\tindex}+1} } dy^{\tindex} \Big|,
   \end{gathered}
\end{equation}
and also restricting $f_{\dff}, f_{\dmf}$ to $\cf = \mathcal{C}_{\tindex} \times [ \hat{X}^{\tindex} ; \{ 0 \} ]$. Likewise, by a variable change, we find that the first expression in (\ref{positive density determine from 3co to co 1}) can also be written as 
\begin{equation} \label{positive density determine from 3co to co 3}
\nu_{\mathrm{3co}} =  2 f_{\dff}  \Big| \frac{dx_{\tindex}}{x_{\tindex}^{n_{\tindex}+1}}  dy_{\tindex} \frac{d x^{\tindex}}{x^{\tindex}} dy^{\tindex} \Big|. 
\end{equation}
Then the b-densities in (\ref{equation for adjoint second paper indicial operator at cf}) is found by first sending
\begin{equation}  \label{positive density determine from 3co to co 4}
\Big| \frac{dx_{\tindex}}{x_{\tindex}^{n_{\tindex}+1}}  dy_{\tindex} \frac{d x^{\tindex}}{x^{\tindex}} dy^{\tindex} \Big| \mapsto \Big| \frac{dx^{\tindex}}{x^{\tindex}} dy^{\tindex} \Big|,
\end{equation}
and then restricting $2 f_{\dff}$ to $\dff = \mathcal{C}_{\tindex} \times X^{\tindex}$. This will give the structures of the densities near $\mathcal{C}_{\tindex} \times \mathcal{C}^{\tindex}$. To find their structures over the interior of $\mathcal{C}_{\tindex} \times (X^{\tindex})^{\circ}$, we can take compact regions near $\partial \Xd$ whose intersections with $\partial \Xd$ is only $\dff^{\circ}$ (which we remark can also be identified also with $\ff^{\circ}$), and write in local frame
\begin{equation} \label{positive density determine from 3co to co 5}
\nu_{\mathrm{3co}} = f_{0} \Big| \frac{d x_{\tindex}}{ x_{\tindex}^{n_{\tindex} + 1} } dy_{\tindex} dz^{\tindex} \Big|. 
\end{equation}
Then, sending
\begin{equation} \label{positive density determine from 3co to co 6}
\Big| \frac{d x_{\tindex}}{ x_{\tindex}^{n_{\tindex} + 1} } dy_{\tindex} dz^{\tindex}  \Big| \mapsto | dz^{\tindex} |,
\end{equation}
and also restricting $f_0$ to $\dff^{\circ} = \mathcal{C}_{\tindex} \times (X^{\tindex})^{\circ}$ gives the required structures in the interior.
\par

For the purpose of getting positive commutator estimates measured in $H_{\mathrm{d3sc,3co,res}}^{m,\vor,\vol,\vor+\vol,\vob,s}$, which is defined with respect to $L^{2}$, we will need to consider $A^{\ast}$, which denotes the adjoint of $A$ taken with respect to the canonical scattering density $\nu_{\mathrm{sc}}$ (which is simply the Euclidean metric in this case). The natural question which arises is then whether or not we have an analogy for (\ref{equation for adjoint second paper indicial operator at cf}). We will show that this requires some modifications.

\begin{lemma} \label{scattering adjoint indicial operator lemma}
Suppose that $A \in \Psf^{\vom, \vor, \vol, \vov, \vob, \vos}$, and let $A^{\ast}$ denote the adjoint of $A$ taken with respect to the canonical scattering density $\nu_{\mathrm{sc}}$. Then $A^{\ast}  \in \Psf^{\vom, \vor, \vol, \vov, \vob, \vos}$. Moreover, we have
\begin{align} 
\sigma(A^{\ast}) & = \overline{ \sigma(A) }, \label{Euclidean adjoint formula under indicial operator at cf -1}
\\ 
\hat{N}_{\cf}( A^{\ast} ) & = \hbff^{-n^{\tindex}/2} \hat{N}_{\cf}(A)^{\ast_{\mathrm{b}}} \hbff^{n^{\tindex}/2} \label{Euclidean adjoint formula under indicial operator at cf}, \\
\hat{N}_{\dff}( A^{\ast} ) & = \hat{N}_{\dff}( A )^{\ast}, \label{Euclidean adjoint formula under indicial operator at cf 1}
\end{align}
where $\hat{N}_{\cf}(A)^{\ast_{\mathrm{b}}}$, $\hat{N}_{\dff}(A)^{\ast}$ denote the adjoints of $\hat{N}_{\cf}(A)$ resp. $\hat{N}_{\dff}(A)$ taken with respect to the canonical b-density $\nu_{\mathrm{b}}([ \hat{X}^{\tindex} ; \{ 0 \} ])$ resp. scattering density $\nu_{\mathrm{sc}}(X^{\tindex})$. 
\end{lemma}
Note that (\ref{Euclidean adjoint formula under indicial operator at cf}) is a globally valid (i.e., even as $\hat{x}^{\tindex} = \hbff^{-1} \rightarrow 0$) expression on $[ \hat{X}^{\tindex} ; \{ 0 \} ]^{\circ}$.
\begin{proof}
We will prove (\ref{Euclidean adjoint formula under indicial operator at cf -1})--(\ref{Euclidean adjoint formula under indicial operator at cf 1}) by using (\ref{equation for adjoint second paper indicial operator at cf}). The idea is to identify a specific positive three-cone density $\nu_{\mathrm{3co}}$ that is naturally related to the canonical scattering density $\nu_{\mathrm{sc}}$. Moreover, since the three-cone structure differs from the three-body structure only in a neighborhood of the corner face, one only needs to modify $\nu_{\mathrm{sc}}$ near $\cf$. \par

Thus, recall first that we can write
\begin{equation*}
\nu_{\mathrm{sc}}   = \Big| \frac{dx_{\tindex}}{x_{\tindex}^{n_{\tindex} + 1}} dy_{\tindex} \frac{dx^{\tindex}}{ ( x^{\tindex} )^{n^{\tindex}+1} }  dy^{\tindex}\Big|
\end{equation*}
in a neighborhood of $\mathcal{C}_{\tindex}$. Then by a direct calculation, we have
\begin{equation}  \label{3co and co inner product scaling 1}
\begin{gathered}
\nu_{\mathrm{sc}} = 2^{-1} x_{\tindex}^{-n^{\tindex}/2} \hbff^{n^{\tindex}/2}  \Big| \frac{dx_{\tindex}}{x_{\tindex}^{n_{\tindex}+1}} dy_{\tindex} \frac{d\hbff}{\hbff} dy^{\tindex} \Big|, \\
 \nu_{\mathrm{sc}} =  2^{-1} x_{\tindex}^{-n^{\tindex}/2} (\hat{x}^{\tindex})^{n^{\tindex}/2} \Big| \frac{dx_{\tindex}}{x_{\tindex}^{n_{\tindex}+1}} dy_{\tindex} \frac{d \hat{x}^{\tindex} }{ ( \hat{x}^{\tindex} )^{n^{\tindex}+1} } dy^{\tindex} \Big|
\end{gathered}
\end{equation}
respectively near $\dff$ and $\dmf$ in a neighborhood of $\cf$. If we now define
\begin{equation} \label{relationship between inner product 3co}
\nu_{\mathrm{3co}} = ( \varphi_{\cf}  f + ( 1 - \varphi_{\cf} ) ) \nu_{\mathrm{sc}}, \quad f =  x_{\tindex}^{n^{\tindex}/2} ( \varphi_{\dff} \hbff^{-n^{\tindex}/2} + \varphi_{\dmf} (\hat{x}^{\tindex})^{-n^{\tindex}/2} ),
\end{equation}
where each $\varphi_{\bullet} \in \mathcal{C}^{\infty}(\Xd)$ is a cut-off at $\bullet$ for $\bullet = \dmf, \dff, \cf$ and also $\varphi_{\dff} + \varphi_{\dmf} = 1$, then by (\ref{3co and co inner product scaling 1}), we know that $\nu_{\mathrm{3co}}$ must be a well-defined three-cone density. \par

By (\ref{relationship between inner product 3co}), we have 
\begin{equation} \label{relationship between inner product 3co 1}
A^{\ast} = F A^{\ast_{\mathrm{3co}}} F^{-1}, \quad F = \varphi_{\cf} f + ( 1 - \varphi_{\cf} ).
\end{equation}
Notice that (\ref{relationship between inner product 3co 1}) already implies the required membership of $A^{\ast}$. Moreover, (\ref{Euclidean adjoint formula under indicial operator at cf -1}) follows easily by the first equality in (\ref{equation for adjoint second paper indicial operator at cf}) and the symbol calculus. \par

To prove (\ref{Euclidean adjoint formula under indicial operator at cf}), note that by the second equality in (\ref{equation for adjoint second paper indicial operator at cf}) and the multiplicative property of the indicial operator map, we have
\begin{equation} \label{3co and co inner product scaling -1}
\hat{N}_{\cf}( A^{\ast} ) =  (x_{\tindex}^{-n^{\tindex}/2} f) \hat{N}_{\cf}(A)^{\ast_{\mathrm{co}}} ( x_{\tindex}^{-n^{\tindex}/2} f )^{-1},
\end{equation}
where for the sake of brevity, we understand everything implicitly as being restricted to $\cf^{\circ}$. Thus what we really mean is
\begin{equation} \label{3co and co inner product scaling}
x_{\tindex}^{-n^{\tindex}/2}f     = \hbff^{-n^{\tindex}/2} ( \varphi_{\dff} + \varphi_{\dmf} ( \hat{x}^{\tindex} )^{-n^{\tindex}} ),
\end{equation}
where $\varphi_{\dff}, \varphi_{\dmf}$ are understood as being restricted to $\cf$. Moreover, $\hat{N}_{\cf}(A)^{\ast_{\mathrm{co}}}$ denote the adjoints of $\hat{N}_{\cf}(A)$ taken with respect to the positive co-densities $\nu_{\mathrm{co}}$ on $[ \hat{X}^{\tindex} ; \{ 0 \} ]$, as determined by $\nu_{\mathrm{3co}}$ and the rules (\ref{positive density determine from 3co to co 1}), (\ref{positive density determine from 3co to co 2}). By using (\ref{relationship between inner product 3co}), we therefore know that the aforementioned co-densities can be written explicitly as
\begin{equation*}
\nu_{\mathrm{co}} = 2^{-1 }\varphi_{\dff} \Big| \frac{d \hbff}{\hbff} d y^{\tindex} \Big| + 2^{-1} \varphi_{\dmf} \Big| \frac{d \hat{x}^{\tindex}}{ ( \hat{x}^{\tindex} )^{n^{\tindex} + 1} } dy^{\tindex} \Big|.
\end{equation*}
Thus we can find, with
\begin{equation*}
\nu_{\mathrm{b}}([\hat{X}^{\tindex} ; \{ 0 \}]) = \Big| \frac{d\hbff}{\hbff} dy^{\tindex} \Big| = \Big| \frac{ d \hat{x}^{\tindex} }{ \hat{x}^{\tindex} } dy^{\tindex} \Big|,
\end{equation*}
i.e., the canonical b-density on $[ \hat{X}^{\tindex}; \{ 0 \} ]$, we have 
\begin{equation*}
\nu_{\mathrm{co}} = 2^{-1} ( \varphi_{\dff} + \varphi_{\dmf} ( \hat{x}^{\tindex} )^{-n^{\tindex}} ) \nu_{\mathrm{b}}([ \hat{X}^{\tindex} ; \{ 0 \} ]).
\end{equation*}
This implies 
\begin{equation*}
\hat{N}_{\cf}(A)^{\ast_{\mathrm{co}}} = ( \varphi_{\dff} + \varphi_{\dmf} (\hat{x}^{\tindex})^{-n^{\tindex}} )^{-1} \hat{N}_{\cf}(A)^{\ast_{\mathrm{b}}} ( \varphi_{\dff} + \varphi_{\dmf} (\hat{x}^{\tindex})^{-n^{\tindex}} ),
\end{equation*}
so that by combining the above with (\ref{3co and co inner product scaling -1}), (\ref{3co and co inner product scaling}), we can conclude that (\ref{Euclidean adjoint formula under indicial operator at cf}) is achieved. \par

It remains to prove (\ref{Euclidean adjoint formula under indicial operator at cf 1}). Starting from (\ref{relationship between inner product 3co}), one could also write 
\begin{equation} \label{3co and co inner product scaling 3}
\varphi_{\cf} f =  \varphi_{\cf} ( \varphi_{\dff} (x^{\tindex})^{n^{\tindex}} + \varphi_{\dmf} x_{\tindex}^{n^{\tindex}/2} (\hat{x}^{\tindex})^{-n^{\tindex}/2} ).
\end{equation}
Thus by using (\ref{equation for adjoint second paper indicial operator at cf}) once more, this time the third equality, together with the multiplicative property of the indicial operator map and (\ref{relationship between inner product 3co 1}), we have
\begin{equation} \label{3co and co inner product scaling 4}
\hat{N}_{\dff}( A^{\ast} ) = ( \varphi_{\cf} (x^{\tindex})^{n^{\tindex}} + ( 1 - \varphi_{\cf} )  )\hat{N}_{\dff}(A)^{\ast_{\mathrm{b}}} ( \varphi_{\cf} (x^{\tindex})^{n^{\tindex}} + ( 1 - \varphi_{\cf} )  )^{-1} ,
\end{equation}
where everything is again understood as being restricted to $\dff^{\circ}$. Moreover, $\hat{N}_{\dff}(A)^{\ast_{\mathrm{b}}}$ denote the adjoints of $\hat{N}_{\dff}(A)$ taken with respect to the b-densities $\nu_{\mathrm{b}}(X^{\tindex})$ determined by $\nu_{\mathrm{3co}}$ and the rules (\ref{positive density determine from 3co to co 3})--(\ref{positive density determine from 3co to co 6}). Explicitly, by using (\ref{3co and co inner product scaling 1}), (\ref{3co and co inner product scaling 3}), we find that
\begin{equation*}
\nu_{\mathrm{b}}(X^{\tindex}) =  \varphi_{\cf} \big| \frac{dx^{\tindex}}{x^{\tindex}} dy^{\tindex} \big| + ( 1 - \varphi_{\cf} ) | dz^{\tindex} | = ( \varphi_{\cf} (x^{\tindex})^{n^{\tindex}} + ( 1 - \varphi_{\cf} ) ) |dz^{\tindex}|.
\end{equation*}
Hence we have
\begin{equation*}
\hat{N}_{\dff}(A)^{\ast_{\mathrm{b}}} = ( \varphi_{\cf} (x^{\tindex})^{n^{\tindex}} + ( 1 - \varphi_{\cf} )  )^{-1} \hat{N}_{\dff}(A)^{\ast} ( \varphi_{\cf} (x^{\tindex})^{n^{\tindex}} + ( 1 - \varphi_{\cf} )  ),
\end{equation*}
so that (\ref{Euclidean adjoint formula under indicial operator at cf 1}) follows by substituting the above into (\ref{3co and co inner product scaling 4}).
\end{proof}

\subsection{Commutator formulae} \label{Subsection commutator formulae}
We now state and prove the crucial commutator formulae which will be the foundation for what remains of this paper. \par

Recall from the symbol calculus (i.e., the multiplicative property (\ref{multiplicative property for the principal symbol d3sc,3co,res})) that if $A_j \in \Psf^{\vom_j, \vor_j, \vol_j, \vov_j, \vob_j, \vos_j}$ for $j =1,2$, then
\begin{equation} \label{second paper commutator class belong}
i[ A_1, A_2 ] \in \Psf^{ \vom_1 + \vom_1 - 1 +2 \delta , \vor_1 + \vor_2 - 1 + 2\delta, \vol_1 + \vol_2, \vov_1 + \vov_2 - 1 + 2\delta, \vob_1 + \vob_2, \vos_1 + \vos_2 - 1 +2 \delta }.
\end{equation}
Moreover, by a standard computation in coordinates, we have
\begin{equation} \label{second paper commutator formula}
\sigma( i [ A_1, A_2 ] ) - H_{\sigma(A_1)} \sigma(A_2) \in S^{\vom_1 + \vom_1 - 2 + 4  \delta , \vor_1 + \vor_2 - 2 + 4\delta, \vol_1 + \vol_2, \vov_1 + \vov_2 - 2 + 4\delta, \vob_1 + \vob_2, \vos_1 + \vos_2 - 2 + 4 \delta }_{\delta}.
\end{equation}
Here, the losses of $2\delta$ in (\ref{second paper commutator class belong}) can be relaxed if the orders at their corresponding faces are actually constants. If $A_1, A_2$ are sufficiently classical at $\cf$, and if we additionally require that
\begin{equation*}
\hat{N}_{\cf}( i [ A_1, A_2 ] ) = 0, 
\end{equation*}
then we can improve (\ref{second paper commutator class belong}) into
\begin{equation*}
i [ A_1, A_2 ] \in \Psf^{ \vom_1 + \vom_2 - 1 +2 \delta , \vor_1 + \vor_2 - 1 + 2\delta, \vol_1 + \vol_2, \vov_1 + \vov_2 - 1 + 2\delta, \vob_1 + \vob_2 - 2 + 4 \delta, \vos_1 + \vos_2 - 1 +2 \delta}. 
\end{equation*}
\par

However, as it turns out, we cannot rely on the symbol calculus alone for proving microlocal propagation estimates, even in the case where we only wish to propagate microlocal regularity in the symbolic sense. To clarify why this is the case, let us recall that in a simplified setting, the general idea of a positive commutator estimate (which is the standard method used to prove microlocal propagation) is to construct operators $A$, $B$, $B_1$ and $E$ such that
\begin{equation}  \label{bad commutator calculation symbolic sense example}
i [ P, A^{\ast} A ] = - B^{\ast} B - B_{1}^{\ast} B_1 + E + R, 
\end{equation} 
where $B_{1}^{\ast}B_1$ is a term that will be dropped in an estimate, $B$ is elliptic at a region where we wish to conclude microlocal regularity, while a priori control needs to be assumed on the wavefront set of $E_0$. Here, the notions of ellipticity, microlocal regularity and wavefront sets are all purposely made arbitrary, i.e., they can be measured either in the symbolic sense, or more globally in the senses at $\cf$, $\dff$. Typically in a microlocal propagation estimate, it is reasonable to focus on one sense of regularity at a time. The remainder term $R$ must then be of lower order in this chosen sense of regularity, and are of the same orders of regularity as $B^{\ast}B$ and $E$ in the other senses. \par

Suppose we focus on propagating microlocal regularity in the symbolic sense, which we recall is the simultaneous measurement of regularity in phase space at $\psf_{\dmf}\Xd$, $\dtsccf$, $\rf$ and fiber infinity. Then for the purpose of proving estimates with respect to $H_{\mathrm{d3sc,3co,res}}^{m, \vor, \vol, \vor+ \vol, \vob, s}$, in a positive commutator estimate we must choose $A \in \Psf^{-\infty, \vor + 1, \vol, \vor + \vol + 1/2 , \vob + 1, -\infty}$, where the orders at fiber infinity and $\rf$ are not essential since $P$ is elliptic at these faces. \par

It is most natural to use the symbol calculus for this calculation. However, since $\hat{N}_{\cf}(P)$ is scalar valued (see (\ref{indicial operator of P at cf formula})), we must always have $\hat{N}_{\cf} ( i [ P, A^{\ast}A ] ) = 0$. Thus by the classicality properties of $P$ (see the discussion below just before Lemma \ref{commutator formula lemma}), with the above choice of $A$, we in fact expect that
\begin{equation} \label{bad commutator membership symbolic sense example}
i[P,A^{\ast} A] \in \Psf^{-\infty, 2 \vor + 2 \delta, 2 \vol, 2 \vor + 2 \vol + 2 \delta, 2 \vob + 4 \delta, -\infty}, 
\end{equation}
where the presence of an additional $4\delta$ loss at $\tcocf$ cannot be removed and is due to the presence of the variable order $\vob$. Moreover, it turns out that with the standard construction of $A$, we often have 
\begin{equation*}
B \in \Psf^{-\infty, \vor, \vol, \vor + \vol, \vob, -\infty}, \ E \in \Psf^{-\infty, 2 \vor, 2 \vol, 2 \vor + 2 \vol, 2 \vob, -\infty}.
\end{equation*}
On the other hand, (\ref{bad commutator membership symbolic sense example}) forces us to only have
\begin{equation} \label{bad commutator calculation remainder term example}
R \in \Psf^{-\infty, 2 \vor - 1 + 4 \delta, 2 \vol, 2 \vor + 2 \vol - 1 + 4 \delta, 2 \vob + 4 \delta, -\infty},
\end{equation}
since $R$ merely arises as the remainder term in the symbol calculus and cannot be controlled precisely. In particular, with $B$, $E$ and $R$ chosen as above, identity (\ref{bad commutator calculation symbolic sense example}) would be insufficient for a microlocal propagation estimate. \par

Indeed, while the memberships of $B$ and $E$ suggest that  microlocal regularity in the symbolic sense should propagate at the level of $H_{\mathrm{d3sc,3co,res}}^{\ast, \vor, \vol, \vor + \vol, \vob, \ast}${\ep}which is what we expected; the membership of $R$ implies this is impossible, for we must assume a priori regularity of at least $H_{\mathrm{d3sc,3co,res}}^{\ast, \vor, \vol, \vor + \vol, \vob + 2 \delta}$. Thus, microlocal propagation cannot be proved in this way. \par

To overcome this problem, we must instead consider commutator formulae which account for decay both in the symbolic sense and in the indicial operator sense at $\cf$, to the extent that we can understand exactly which terms contribute to this loss of order $4\delta$ at $\tcocf$ in $R$. Note that we do not consider such a formula at $\dff$, since $\hat{N}_{\dff} ( i [ P, A^{\ast}A ] )$ will generally be non-vanishing for those $A$ that we are interested in.\par

For the estimates considered in \S \S \ref{tangential propagation section}, \ref{transversal propagation section} and \ref{decay at the corner face section}, we will only need to assume that $A$ is supported in a small neighborhood of $\cf$ (i.e., $A = \varphi_{\cf} A \varphi_{\cf}$ for some cut-off function $\varphi_{\cf} \in \mathcal{C}^{\infty}( \Xd )$ at $\cf$). Thus, only the behavior of $P$ near $\cf$ will be relevant. \par

In fact, suppose we impose over the interior of $[ \hat{X}^{\tindex} ; \{ 0 \} ]$ the natural coordinates $( \hbff, y^{\tindex} )$. Then it is not hard to see that in the coordinates $( x_{\tindex}, y_{\tindex}, \hbff, y^{\tindex} )$, we have
\begin{equation} \label{writing P locally near cf}
P = (x_{\tindex} D_{x_{\tindex}} )^2 + x_{\tindex}^2 \Delta_{h_{\tindex}} + i(n_{\tindex} - 1) x_{\tindex} ( x_{\tindex}^2 D_{x_{\tindex}} ) +  P_{1} + P_{2} + V - \lambda^2.
\end{equation}
Here, we are writing
\begin{equation*}
P_{1} \coloneq x_{\tindex} \hat{N}_{\cf}( \Delta_{z^{\tindex}} ) + 2 x_{\tindex} ( x_{\tindex}^2 D_{x_{\tindex}} ) ( \hbff D_{\hbff} ), \
P_{2} \coloneq - i x_{\tindex}^2 ( \hbff D_{\hbff} ) + x_{\tindex}^2 ( \hbff D_{\hbff} )^2,
\end{equation*}
where
\begin{equation} \label{the cf indicial operator of the interactive laplacian}
\hat{N}_{\cf}( \Delta_{z^{\tindex}} ) = 4 ( \hbff^{-1/2} \hbff D_{\hbff} )^{2} - 2 i ( n^{\tindex} - 1 ) \hbff^{-1} ( \hbff D_{\hbff} ) + \hbff^{-1} \Delta_{h^{\tindex}},
\end{equation}
while one also has
\begin{equation*}
\Delta_{z_{\tindex}} = (x_{\tindex} D_{x_{\tindex}} )^2 + x_{\tindex}^2 \Delta_{h_{\tindex}} + i(n_{\tindex} - 1) x_{\tindex} ( x_{\tindex}^2 D_{x_{\tindex}} ).
\end{equation*}
Clearly, a similar formula holds if we instead use $( x_{\tindex}, y_{\tindex},  \hat{x}^{\tindex}, y^{\tindex} )$ as coordinates. \par

Notice that (as operators locally restricted near $\cf$), we have
\begin{equation*}
\Delta_{z_{\tindex}} - \lambda^{2} \in \Psi_{\mathrm{3coc}}^{2,0,0,0}, \ P_{1}  \in \Psi_{\mathrm{3coc}}^{2, -1, 0, -2}, \ P_{2} \in \Psi_{\mathrm{3coc}}^{2,-2,-2,-4}, \ V \in \Psi_{\mathrm{3coc}}^{0, -2 - \delta, 0 , -2 - \delta }.
\end{equation*}
Moreover, $\Delta_{z_{\tindex}} - \lambda^2$ is homogenous of degree zero in $z_{\tindex}$. Thus we can conclude that $P$ is classical at $\cf$ modulo $\Psi^{2, 0, 0, -2}_{\mathrm{3coc}}$ as an operator in $\Psi_{\mathrm{3coc}}^{2,0,0,0}$. It is also easy to check that (still restricted near $\cf$) we have
\begin{equation} 
\label{the cf indicial operator of the interactive laplacian 1.1}
\begin{gathered}
\Delta_{z_{\tindex}} - \lambda^2 \in \Psi_{\mathrm{d3sc,3co,res}}^{2, 0, 0, 0, 0 ,0}, \ P_1 \in \Psi_{\mathrm{d3sc,3co,res}}^{2, -1, 0, 0, -2, 2 }, \\
P_2 \in \Psi_{\mathrm{d3sc,3co,res}}^{2, -2, -2, -,2 , -4, 0}, \ V \in \Psi_{\mathrm{d3sc,3co,res}}^{0, -2 - \delta, 0, - 2 - \delta, - 2 - \delta, 0}.
\end{gathered}
\end{equation}
Thus $P$ is classical at $\cf$ modulo $\Psi_{\mathrm{d3sc,3co,res}}^{2, 0, 0, 0, -2 , 0 }$ as an operator in $\Psi_{\mathrm{d3sc,3co,res}}^{2, 0, 0, 0, 0, 2}$.

However, in \S \ref{global radial point estimate section}, we will consider radial point estimates at $\brpm$. Since these are the lifts of the two-body radial sets, they will no longer be contained in a small neighborhood of the corner faces. Thus, the support condition on $A$ must in general be relaxed.

Before we proceed with the main results of this subsection, let us also use this opportunity to present an improvement of Proposition \ref{proposition vanishing of indicial operator implies lower order membership} in a special case. This will be required in \S \S \ref{principal type propagation section}--\ref{global radial point estimate section} below. First, we introduce a definition. We say that a symbol $a \in S^{\vom, \vor, \vol \vov, \vob, \vos}$ is classical if $a$ belongs to
\begin{equation*}
S_{\mathrm{cl}}^{ \vom, \vor, \vol, \vov, \vob, \vos }( \psf \Xd ) \coloneq  \rho_{\infty}^{- \vom } \rho_{\mathrm{dmf}}^{- \vor } \rho_{\dff}^{- \vol } \rho_{\dtsccf}^{- \vov } \rho_{\tcocf}^{-\vob} \rho_{\mathrm{rf}_{\tindex}}^{-\vos} \mathcal{C}^{\infty}( \psf \Xd ).
\end{equation*}
We would like to consider elements of $\Psf^{\vom, \vor, \vol, \vov,  \vob, \vos}$ which are defined by some quantization of symbols in $S^{\vom, \vor, \vol, \vov, \vob, \vos}_{\mathrm{cl}}$. \par

However, note that this is not an invariant definition, in the sense that it depends on a choice of quantization. In general, upon changing the quantization, we would need to consider symbols with a polyhomogenous expansion as well (by the Kuranishi trick combined with a left symbol reduction procedure in various coordinates).  \par

Note also that if $A \in \Psf^{\vom, \vor, \vol, \vov, \vob, \vos}$ is given by the quantization of a classical symbol, then upon Taylor expanding the symbol of $A$ at $\tcocf$, we can see that $A$ must be classical at $\cf$ modulo $\Psf^{\vom, \vor, \vol, \vov, \vob - 1, \vos}$ as well.

\begin{lemma} 
\label{Lemma: improvement the vanishing of principal symbol and indicial operator implies something}
Suppose that $A \in \Psf^{\vom, \vor, \vol, \vov, \vob, \vos}$ is defined by some quantization of a classical symbol. If $\sigma(A) = 0$, $\hat{N}_{\cf}(A) = 0$, then $
A \in \Psf^{\vom - 1 + 2 \delta, \vor - 1 + 2 \delta, \vol, \vov - 1 + 2 \delta, \vob - 1, \vos - 1 + 2\delta}.
$
\end{lemma}
\begin{proof}
By assumption, if we localize $A$ and write it as a specific, suitable left quantization of $a$, then $a = a_0 + a_1$, where $a_0 \in S_{\mathrm{cl}}^{\vom, \vor, \vol, \vov, \vob, \vos}$, and $a_1$ is given by an asymptotic summation of some lower orders, polyhomogenous terms (see the comments above). In particular, we can check directly that $a_1 \in S_{\delta}^{\vom - 1 + 2\delta, \vor - 1 + 2\delta, \vol, \vov - 1 + 2\delta, \vob, \vos - 1 + 2\delta} $, and moreover $x_{\tindex}^{\vob/2} a_1$ restricts to a polyhomogenous symbol on $\tcocf$. \par

Now, since $\sigma(A) = 0$, we know that $a \in S_{\delta}^{\vom - 1 + 2\delta, \vor  - 1 + 2\delta, \vol, \vov - 1 + 2\delta, \vob, \vos - 1 + 2\delta}$. On the other hand, $\hat{N}_{\cf}(A)$ is given by some left quantization of the restriction of $x_{\tindex}^{\vob/2}a$ to $\tcocf^{\circ}$. Since by assumption we must have $\hat{N}_{\cf}(A) = 0$, it follows that $x_{\tindex}^{\vob/2}a = 0$ identically on $\tcocf^{\circ}$. Consequentially, we know that $a \in S^{\vom - 1 + 2\delta, \vor - 1 + 2\delta, \vol, \vov - 1 + 2\delta, \vob, \vos}$. Since this works locally every near the diagonal, the claim follows.
\end{proof}

\begin{lemma} 
\label{commutator formula lemma}
Let $\vor, \vol, \vob$ be chosen as in \S \ref{variable order construction section}. Suppose that $A \in \Psf^{-\infty, \vor + 1/2, \vol, \vor + \vol + 1/2, \vob + 1 , -\infty}$ is classical at $\cf$ modulo $\Psf^{-\infty, \vor + 1/2, \vol, \vor + \vol + 1/2, \vob + 1 - K, -\infty}$, $K > 0$. Then there exists operators
 \begin{equation}  \label{tangential discussion 0.8}
 B \in \Psf^{-\infty, \vor, \vol - 1/2, \vor + \vol - 1/2, \vob, -\infty}, \ E \in \Psf^{-\infty, 2 \vor + 2 \delta, 2 \vol, 2 \vor + 2 \vol + 2 \delta, 2 \vob, -\infty},
 \end{equation}
 such that $B$ is classical at $\cf$ modulo $\Psf^{-\infty, \vor, \vol -1/2, \vor + \vol - 1/2, \vob - K, -\infty}$, $E$ is classical at $\cf$ modulo $\Psf^{-\infty, 2\vor + 2 \delta, 2 \vol, 2 \vor + 2 \vol + 2 \delta, 2 \vob - K, -\infty}$, and for any boundary defining function $x \in \mathcal{C}^{\infty}(\Xo)$, we have
\begin{equation}\label{tangential discussion 0.85}
 i [ P,  A ^{\ast} A] =  B^{\ast} ( \log x ) B + E + R,
\end{equation}
where $R \in  \Psf^{-\infty, 2 \vor - 1 + 4 \delta, 2 \vol , 2 \vor + 2 \vol - 1 + 4 \delta, 2 \vob - 2\delta, - \infty}$. 
Moreover, we can compute that
\begin{equation} \label{tangential discussion 0.9}
\hat{N}_{\cf}( B )  = ( - \mathsf{H}_{ \inttcf } \vob )^{1/2} \hat{N}_{\cf}( A ), \ \sigma( B ) = (-H_{p} \vob)^{1/2} \sigma(A)
\end{equation} 
as well as
\begin{align} 
\begin{split}
 \hat{N}_{\cf}(E) & =  i[  \hat{N}_{\cf}( \Delta_{z^{\tindex}} ) + 2 \ltau \hbff D_{\hbff}, \hat{N}_{\cf}( A^{\ast} A ) ] \label{tangential discussion 1} \\
& \quad + \mathsf{H}_{\inttcf} \hat{N}_{\cf}( A^{\ast}A ) - ( 2\vob + 2 ) \ltaucf \hat{N}_{\cf}( A^{\ast}A ), 
\end{split} \\
\sigma(E) & = H_{p} \sigma(A^{\ast} A) + (\log x ) ( H_{p} \vob ) \sigma( A^{\ast}A ).  \label{tangential discussion 1.001}
\end{align} \par

In the case where $A$ is supported near $\cf$, we can take $B,G$ to be supported near $\cf$ as well. Moreover,   the above formulae hold even if $x$ is valid only in a neighborhood of $\mathcal{C}_{\tindex}$.

Moreover, if $A$ is given by some quantization of a classical symbol, then so is $B$.
\end{lemma}

\begin{remark}
Lemma \ref{commutator formula lemma} is only a special case. One could certainly show that more general formulae hold as well, where the operators in the commutator are chosen more freely. However, even the statements of such results appear unnecessarily complicated for our purpose. Thus, we have elected to present only what is necessary below.
\end{remark}

Since $x \simeq \rho_{\dmf} \rho_{\dff} \rho_{\dtsccf}^{2} \rho_{\tcocf}^{2}  \rho_{\rf}$, we immediately see that $\log x \in \Psf^{0,2 \delta, 2 \delta, 4 \delta, 4 \delta, 2 \delta}$. This is clearly too strong in the sense that we only needed to remove the loss of $4\delta$ at $\tcocf$ in the remainder term (\ref{bad commutator calculation remainder term example}) appearing in (\ref{bad commutator calculation symbolic sense example}). Though this seems harmless for now, in the propagation estimates below, it will become obvious that one nevertheless has to compensate for this by imposing stricter conditions on the variable orders than those required in \S \ref{variable order construction section}. \par

We will avoid this by showing the following improvement:

\begin{corollary} \label{commutator corollary}
Under the assumptions of Lemma \ref{commutator formula lemma}, we can also choose
 \begin{equation*}
 \tilde{E} \in \Psf^{-\infty, 2 \vor + 2 \delta, 2 \vol, 2 \vor + 2 \vol + 2 \delta, 2 \vob, -\infty}
 \end{equation*}
 such that if $x_{\cf} \in \mathcal{C}^{\infty}( \Xd )$ is a defining function for $\cf$, then we have
\begin{equation*} 
 i [ P,  A ^{\ast} A] =  2B^{\ast} ( \log x_{\cf} ) B + \tilde{E} + R,
\end{equation*}
where $B,R$ are chosen as in (\ref{tangential discussion 0.85}). Moreover, we can compute that
\begin{equation*}
\sigma(\tilde{E}) = H_{p} \sigma( A^{\ast}A ) + 2 ( \log x_{\cf} ) ( H_{p} \vob ) \sigma(A^{\ast}A), 
\end{equation*}
while conditions (\ref{tangential discussion 0.9}) remain to be satisfied. 
\end{corollary}

Thus, through the above corollary, we have removed only the logarithmic loss with respect to $x_{\cf} \simeq \rho_{\dtsccf} \rho_{\tcocf}$. Unfortunately, it is not so clear at this stage how one could further separate the decay between $\dtsccf$ and $\tcocf$.
\begin{proof}[Proof of Corollary \ref{commutator corollary}]
Starting from formula (\ref{tangential discussion 0.85}), it suffices to choose the defining functions such that $x = x_{\dff} x_{\dff} x_{\cf}^2$, where $x_{\cf}$ is as in the statement of the corollary, while $x_{\dmf}, x_{\dff} \in \mathcal{C}^{\infty}( \Xd )$ are respectively defining functions for $\dmf$ and $\dff$. Then we have
\begin{equation*}
i [ P, A^{\ast}A ] =2 B^{\ast} ( \log x_{\cf} ) B + B^{\ast} ( \log x_{\dmf} + \log x_{\dff} ) B + E + \tilde{R}.
\end{equation*}
Thus, all we need to do now is to define $\tilde{E} = B^{\ast} ( \log x_{\dmf} + \log x_{\dff} ) B + E$.
\end{proof}

In proving Lemma \ref{commutator formula lemma}, the following will also be useful. 
\begin{lemma} 
\label{Lemma choice operator that has symbol and indicial operator at cf}
Suppose that $\vor \in \mathcal{C}^{\infty}( \psf \Xd )$, $\vol, \vob \in \mathcal{C}^{\infty}( \mathcal{C}_{\tindex} \times \overline{\mathbb{R}^{n_{\tindex}}} )$. Let
\begin{equation*}
\hat{A}_{\cf} \in \mathcal{C}^{\infty}_{c} \big( \mathcal{C}_{\tindex} \times \mathbb{R}^{n_{\tindex}} ; \Psi_{\mathrm{coc}, \delta}^{\vor + \vol - \vob, \vor - \vob/2, \vol - \vob/2}( [ \hat{X}^{\tindex} ; \{ 0 \} ] ) \big).
\end{equation*}
Moreover, let 
\begin{equation*}
a \in S^{\vom, \vor, \vol, \vov, \vob, \vos}( \psf \Xd ), \ a_{\cf} \in \mathcal{C}^{\infty}_{c} \big( \mathcal{C}_{\tindex} \times \mathbb{R}^{n_{\tindex}} ; S_{\delta}^{\vor + \vol - \vob, \vor - \vob/2, \vol - \vob/2} ( \overline{^{\mathrm{co}}T^{\ast}} [ \hat{X}^{\tindex} ; \{ 0 \} ] ) \big)
\end{equation*}
be chosen such that
\begin{equation*}
q_{\cf} ( x_{\tindex}^{\vob/2} a - a_{\cf} ) \in S^{-\infty, \vor, \vol, \vor + \vol, \vob - K, - \infty}_{\delta} ( \psf \Xd ), \quad K > 0, 
\end{equation*}
where $q_{\cf} \in \mathcal{C}^{\infty}( \psf \Xd )$ is some cut-off function at $\tcocf$. Assume additionally that
\begin{equation*}
\text{$\hat{A}_{\cf}$ is given by some quantization of $a_{\cf}$ near the diagonal.}
\end{equation*}
Then we can find $A \in \Psf^{-\infty, \vor, \vol, \vor + \vol, \vob, -\infty }$ which is classical at $\cf$ modulo $\Psf^{-\infty, \vor, \vol, \vor + \vol, \vob - K, -\infty}$ such that $\hat{N}_{\cf}(A) = \hat{A}_{\cf}$, $\sigma(A) = a$. In fact, if $\hat{A}_{\cf}$ is defined by some quantization of $a_{\cf}$, then $A$ is defined by some quantization of $a$. 
\end{lemma}
\begin{proof}
Let us write $\hat{A}_{\cf} = \hat{B}_{\cf} + \hat{R}_{\cf}$ such that $\hat{B}_{\cf}$ is given by some quantization of $a_{\cf}$ (i.e., it is a near diagonal part of $\hat{A}_{\cf}$) and $\hat{R}_{\cf}$ is a family of operators in $\Psi_{\mathrm{coc}}^{-\infty, -\infty, \vol - \vob/2}$ (i.e., it is an off-diagonal part of $\hat{A}_{\cf}$). Then we can choose $B$ carefully so that it is some quantization of $a$ such that $\hat{N}_{\cf}(B) = \hat{B}_{\cf}$. In particular, we can write $\hat{A}_{\cf} = \hat{N}_{\cf}(B) + \hat{R}_{\cf}$. Next, let $R$ be defined by partially quantizing some dilation of $\hat{R}_{\cf}$ so that $\hat{N}_{\cf}(R) = \hat{R}_{\cf}$. Then we can finally define $A = B + R$ which has the required properties. If $\hat{A}_{\cf}$ is defined just by some quantization of $a_{\cf}$, then we can assume that $\hat{R}_{\cf} = 0$. Hence we can choose $A = B$, which is given just by some quantization of $a$.
\end{proof}
\begin{proof}[Proof of Lemma \ref{commutator formula lemma}]
We first prove the claims in the case where $A$ is supported near $\cf$, i.e., $A = \varphi_{\cf} A \varphi_{\cf}$, where $\varphi_{\cf} \in \mathcal{C}^{\infty}(\Xd)$ is a cut-off function at $\cf$. Then by using (\ref{writing P locally near cf}), we know that the commutator in question can be written as
\begin{equation*}
i [ P, A^{\ast}A ] = i[ \Delta_{z_{\tindex}}, A^{\ast}A  ] + i [  P_{1},  A^{\ast}A ] + i [ P_{2}, A^{\ast}A ] + i[ V, A^{\ast} A  ].
\end{equation*}
It is easy to see from (\ref{the cf indicial operator of the interactive laplacian 1.1}) that when restricted near $\cf$, we have
\begin{equation}  \label{commutator new cal 0.5}
\begin{gathered}
i [ P_1, A^{\ast}A  ]  \in \Psf^{-\infty, 2 \vor - 1 + 2 \delta, 2 \vol, 2 \vor + 2 \vol + 2 \delta, 2 \vob, - \infty}, \\
 i [ P_2, A^{\ast}A ] \in \Psf^{-\infty, 2 \vor - 1 + 2 \delta, 2 \vol - 2, 2 \vor + 2 \vol - 2 + 2 \delta, 2\vob - 2, -\infty}, \\
 i [ V, A^{\ast} A ] \in \Psf^{-\infty, 2 \vor - 2, 2 \vol, 2 \vor + 2 \vol - 2, 2 \vob - 2\delta, - \infty}.
\end{gathered}
\end{equation}
Thus in particular, the indicial operators of $i [ P_2, A^{\ast}A ]$ of order $2\vob$ at $\cf$ simply vanishes. Meanwhile, since $\hat{N}_{\cf}( P_1) = \hat{N}_{\cf}( \Delta_{z^{\tindex}} ) + 2 \ltau ( \hbff D_{\hbff} )$, it is clear that
\begin{equation} \label{commutator new cal 1}
\hat{N}_{\cf}( i [ P_1, A^{\ast}A ] )  = i [ \hat{N}_{\cf}( \Delta_{z^{\tindex}} ) + 2 \ltau ( \hbff D_{\hbff} ), \hat{N}_{\cf}( A^{\ast}A ) ].
\end{equation} \par

We next apply Lemma \ref{composition subsection composition reduction lemma} to the study of $i[ \Delta_{z_{\tindex}}, A^{\ast}A ]$. In this exceptionally simple case, the `operator-valued symbol' of $\Delta_{z_{\tindex}}$ at $\cf$ is simply the scalar valued function $|\zeta_{\tindex}^{\mathrm{3co}}|^2$, i.e., $\Delta_{z_{\tindex}}$ can be written as the partial (and thus full) quantization of $|\zeta_{\tindex}^{\mathrm{3co}}|^2$. For brevity, let us henceforth write 
\begin{equation*}
G \coloneq A^{\ast}A.
\end{equation*}
Then $G$ is again supported near $\cf$. Let $\psi_{\cf}, \phi_{\cf} \in \mathcal{C}^{\infty}(\Xd)$ be cut-off functions at $\cf$ such that $\psi_{\cf} = 1$ on $\supp \phi_{\cf}$, $\phi_{\cf} = 1$ on $\supp G$ (thus $\psi_{\cf} = 1$ on $\supp G$). Then we have
\begin{equation*}
[ \Delta_{z_{\tindex}}, G ] = \psi_{\cf} [ \Delta_{z_{\tindex}}, G ] \psi_{\cf} + ( 1 - \psi_{\cf} ) \Delta_{z_{\tindex}} \phi_{\cf} G + G \phi_{\cf} \Delta_{z_{\tindex}} ( 1 - \psi_{\cf} ),
\end{equation*}
where we note that
\begin{equation*}
( 1 - \psi_{\cf} ) \Delta_{z_{\tindex}} \phi_{\cf} G + G \phi_{\cf} \Delta_{z_{\tindex}} ( 1 - \psi_{\cf} ) \in \Psf^{-\infty, -\infty, -\infty, -\infty, -\infty, -\infty}.
\end{equation*}
Now, let $\varphi_{\cf} \in \mathcal{C}^{\infty}(\Xd)$ be chosen such that $\varphi_{\cf} = 1$ on $\supp \psi_{\cf}$ (thus $\varphi_{\cf} = 1$ on $\supp G$). Moreover, for brevity, let $\hat{G} = \hat{G}_{\psi_{\cf}}$ denote the operator-valued symbol of $G = \varphi_{\cf} G \varphi_{\cf}$ (in the notation of \S \S \ref{the conormal three-cone operators section}--\ref{Section: Microlocal properties of the second microlocalized algebra}). Then by Lemma \ref{composition subsection composition reduction lemma}, we can conclude that
\begin{equation*}
i[ \Delta_{z_{\tindex}}, G ] = C + \tilde{R}, 
\end{equation*}
where $C \in \Psf^{- \infty, 2 \vor - 1 + 2 \delta, 2 \vol - 1 + 2 \delta, 2 \vor + 2\vol  - 1 + 4 \delta, 2\vob + 4 \delta, -\infty}$ is the partial quantization of $H_{|\zeta_{\tindex}^{\mathrm{3co}}|^2} \hat{G}$, where $H_{|\zeta_{\tindex}^{\mathrm{3co}}|^2}$ is the Hamiltonian vector field of $|\zeta_{\tindex}^{\mathrm{3co}}|^2$ with $(z_{\tindex}, \zeta_{\tindex}^{\mathrm{3co}})$ being viewed as the cannonical coordinates (since $G$ is supported near $\cf$, it is also unnecessary to specify a cut-off at $\cf$ in the quantization of $H_{|\zeta_{\tindex}^{\mathrm{3co}}|^2} \hat{G}$); while the remainder term
\begin{equation*}
\tilde{R} \in \Psf^{-\infty, 2 \vor - 1 + 4 \delta, 2 \vol - 2 + 4 \delta, 2 \vor + 2 \vol -3 + 8 \delta, 2 \vob - 2 + 8 \delta, - \infty}
\end{equation*} 
decays better in every possible sense.
\par
In fact, suppose we change the free coordinates from $(z_{\tindex}, \zeta_{\tindex}^{\mathrm{3co}} )$ to $( x_{\tindex}, y_{\tindex}, \ltau^{\mathrm{3co}}, \lmu^{\mathrm{3co}} )$ such that $\ltau^{\mathrm{3co}}$ is dual to $x_{\tindex}^{-2} dx_{\tindex}$, $\mu_{\tindex,j}^{\mathrm{3co}}$ is dual to $x_{\tindex}^{-1} d{y_{\tindex,j}}$ for $j = 1,..., n_{\tindex} - 1$. Then we can write $|\zeta_{\tindex}^{\mathrm{3co}}|^{2} = |\ltau^{\mathrm{3co}}|^{2} + | \lmu^{\mathrm{3co}} |_{h_{\tindex}}^2$. Moreover, suppose that $\vob_{\tindex} \in \mathcal{C}^{\infty}_{c}( \mathcal{C}_{\tindex} \times \mathbb{R}^{n_{\tindex}} )$ denotes the restriction of $\vob$ to $\tcocf$ (notice that $\vob_{\tindex}$ can be chosen to have compact support since we only needed to specify it in a neighborhood of $\Sigma_{\cf}$) and write $\hat{G}= x_{\tindex}^{-\vob_{\tindex} - 1} \hat{G}_0$. Then we have
\begin{align}  \label{commutator new cal 2}
\begin{split}
H_{|\ltau^{\mathrm{3co}}|^{2} + | \lmu^{\mathrm{3co}} |_{h_{\tindex}}^2} \hat{G} = & - ( \log x_{\tindex} ) ( H_{|\ltau^{\mathrm{3co}}|^{2} + | \lmu^{\mathrm{3co}} |_{h_{\tindex}}^2} \vob_{\tindex} ) \hat{G} \\
&  - (  2 \vob_{\tindex} + 2 ) \ltau^{\mathrm{3co}} x_{\tindex}^{-\vob_{\tindex}} \hat{G}_{0} + x_{\tindex}^{-\vob_{\tindex}-1} H_{|\ltau^{\mathrm{3co}}|^{2}  + | \lmu^{\mathrm{3co}} |_{h_{\tindex}}^2} \hat{G}_{0},
\end{split}
\end{align}
where we recall that $x_{\tindex}$ decays to second order at $\cf$. Correspondingly, we can write
\begin{equation*}
C = ( \log x_{\tindex} ) \tilde{C}_{1} + C_0,
\end{equation*}
where we respectively let
\begin{equation*}
\tilde{C}_1 \in \Psf^{-\infty, 2 \vor - 1 , 2 \vol - 1, 2 \vor + 2 \vol - 1, 2 \vob, -\infty}, \ C_0 \in \Psf^{-\infty, 2 \vor + 2 \delta, 2 \vol - 1 + 2 \delta, 2 \vor + 2 \vol - 1 + 2 \delta, 2 \vob, -\infty}
\end{equation*}
be the partial quantizations of $ (- H_{|\ltau^{\mathrm{3co}}|^{2} + | \lmu^{\mathrm{3co}} |_{h_{\tindex}}^2} \vob_{\tindex} ) \hat{G}$ and the second line of (\ref{commutator new cal 2}). \par

It is clear that the membership of $- ( 2 \vob_{\tindex} + 2) \tau_{\tindex}^{\mathrm{3co}} x_{\tindex}^{-\vob_{\tindex}} \hat{G}_0$ holds as suggested. To obtain the membership of $x_{\tindex}^{-\vob_{\tindex} - 1} H_{|\ltau^{\mathrm{3co}}|^{2}  + | \lmu^{\mathrm{3co}} |_{h_{\tindex}}^2} \hat{G}_0$ (for which only the decay at $\cf$ needs to be checked), let us write
\begin{align*}
\hat{G}_0 = x_{\tindex}^{\vob_{\tindex}+1} \hat{G} & = x_{\tindex}^{\vob_{\tindex} + 1} \hat{G}_{\cf, q_{\cf}} + x_{\tindex}^{\vob_{\tindex} + 1} ( \hat{G} - \hat{G}_{\cf, q_{\cf}} ).
\end{align*}
Here, $\hat{G}_{\cf, q_{\cf}}$ is the operator-valued symbol of $G_{\cf, q_{\cf}} \in \Psf^{-\infty, 2 \vor + 1, 2 \vol, 2 \vor + 2 \vol + 1, 2 \vob + 2 , -\infty }$, where $G_{\cf, q_{\cf}}$ is constructed as in \S \ref{subsection partial classicality and indicial operators}  depending on a cut-off function $q_{\cf} \in \mathcal{C}^{\infty}( \psf \Xd )$ at $\tcocf$. \par

In fact, without loss of generality (i.e., by choosing $\vob_{\tindex}$ as the extension of $\vob$ in the construction of $G_{\cf, q_{\cf}}$), we can assume that $x_{\tindex}^{\vob_{\tindex} + 1} \hat{G}_{\cf, q_{\cf}}$ is smooth up to $\{ x_{\tindex} = 0 \}$. Thus, $x_{\tindex}^{-\vob_{\tindex} - 1} H_{|\ltau^{\mathrm{3co}}|^{2}  + | \lmu^{\mathrm{3co}} |_{h_{\tindex}}^2} ( x_{\tindex}^{\vob_{\tindex} + 1} \hat{G}_{\cf, q_{\cf}} )$ partially quantizes to an element of $\Psf^{-\infty, 2 \vor, 2 \vol, 2 \vor + 2 \vol, 2 \vob, -\infty}$. \par

Meanwhile, by assumption, we know that $x_{\tindex}^{\vob_{\tindex} + 1} ( \hat{G} - \hat{G}_{\cf, q_{\cf}} )$ decays to higher order at $\cf$, and thus $x_{\tindex}^{-\vob_{\tindex} - 1} H_{|\ltau^{\mathrm{3co}}|^{2}  + | \lmu^{\mathrm{3co}} |_{h_{\tindex}}^2} (  x_{\tindex}^{\vob_{\tindex} + 1} ( \hat{G} - \hat{G}_{\cf, q_{\cf}} ) )$ must partially quantize to an element of the correct operator class as well.

   \par

Now, suppose we relabel $( \ltau^{\mathrm{3co}}, \lmu^{\mathrm{3co}} ) = ( \ltau, \lmu )$ at $x_{\tindex} = 0$. Then we have
\begin{equation*}
\hat{N}_{\cf}( \tilde{C}_1 ) =   ( - \sH_{\intt} \vob_{\tindex} ) \hat{N}_{\cf}( G ).
\end{equation*}
Moreover, observe that
\begin{equation} \label{special commutator estimate section restriction of Hpb}
(x_{\tindex}^{-1} H_{p} \vob )|_{\tcocf} = \sH_{\intt} \vob_{\tindex}.
\end{equation}
Thus by Lemma \ref{Lemma choice operator that has symbol and indicial operator at cf}, we can choose $C_{1} \in \Psf^{-\infty, 2 \vor - 1, 2 \vol - 1, 2 \vor + 2 \vol - 1, 2 \vob, -\infty}$ such that
\begin{equation*}
\hat{N}_{\cf} ( C_1 ) = \hat{N}_{\cf} ( \tilde{C}_1 ), \ \sigma( C_1 ) = ( - H_{p} \vob ) \sigma(G),
\end{equation*}
and that $C_1$ is classical at $\cf$ modulo $\Psf^{-\infty, 2 \vor - 1, 2 \vol - 1 , 2 \vor + 2 \vol - 1, 2 \vob - K, -\infty}$. On the other hand, we clearly have
\begin{equation} 
\label{commutator new cal 3}
\hat{N}_{\cf}( C_0 ) = - (  2 \vob + 2 ) \ltau \hat{N}_{\cf}( G ) + \sH_{|\ltau |^{2}  + | \lmu  |_{h_{\tindex}}^2} \hat{N}_{\cf}( G ).
\end{equation}
Thus, suppose we put 
\begin{equation*}
E \coloneq i [ P_1, G ] + i[ P_2, G ] +  C_0  + ( \log x_{\tindex} ) ( \tilde{C}_1 - C_1 ).
\end{equation*}
Then by construction we must have $ \tilde{C}_1 -  C_1 \in \Psf^{-\infty, 2 \vor - 1 , 2 \vol - 1 , 2 \vor + 2 \vol - 1 , 2 \vob - K, -\infty}$. In particular, we have shown that $E$ indeed belongs to the correct class of operators. By using (\ref{commutator new cal 1}) and (\ref{commutator new cal 3}), we can conclude that (\ref{tangential discussion 1}) holds as well.  \par

Next we prove (\ref{tangential discussion 1.001}). Since $i[ P_j, G ]$, $j =1,2$ and $C_0$ all belong to classes of operators contained in $\Psf^{-\infty, 2 \vor + 2 \delta, 2 \vol, 2 \vor + 2 \vol + 2 \delta, 2 \vob, -\infty}$, we can use their respective symbol calculi for this calculation. Thus, let $p_j$ be the principal symbol of $P_j$ and $g$ be the principal symbol of $G$. Write also $g = x_{\tindex}^{-\vob_{\tindex}} g_0$. Then we have
\begin{align} \label{commutator new cal 4}
\begin{split}
& \sigma( i [ P_1, G ] + i [ P_2, G ] + i [ V, G ]+ C_0 ) = H_{p_1} g + H_{p_2} g \\
& \qquad -( 2 \vob_{\tindex} + 2 ) \ltau^{\mathrm{3co}} x_{\tindex}^{-\vob_{\tindex}} g_{0} + x_{\tindex}^{-\vob_{\tindex} - 1} H_{ |\ltau^{\mathrm{3co}}|^2 + | \lmu^{\mathrm{3co}} |_{h_{\tindex}}^{2} } g_0.
\end{split}
\end{align}  \par

Computing the principal symbol of $( \log x_{\tindex} )  ( \tilde{C}_1 - C_1 )$ requires more care as $\tilde{C}_1$, $C_1$ belong to $\Psf^{-\infty, 2 \vor, 2 \vol - 1, 2 \vor + 2 \vol - 1, 2 \vob, - \infty}$ only. However, notice that we only need to be concerned with the near-diagonal regions, where $\tilde{C}_1$ can be written as some quantization of 
\begin{equation*}
( - H_{|\ltau^{\mathrm{3co}}|^{2} + | \lmu^{\mathrm{3co}} |_{h_{\tindex}}^2} \vob_{\tindex} )\tilde{g}, \quad  \tilde{g} - g \in S^{-\infty, 2 \vor + 2\delta, 2 \vol, 2 \vor + 2 \vol + 2 \delta, 2 \vob + 2 , -\infty}
\end{equation*}
by construction. We can then choose $C_1$ to be the same quantization of $( - H_{p}\vob) \tilde{g}$, such that the principal symbol of $\tilde{C}_1 - C_1$ becomes
\begin{align*}
 -(H_{|\ltau^{\mathrm{3co}}|^{2} + | \lmu^{\mathrm{3co}} |_{h_{\tindex}}^2} \vob_{\tindex} ) \tilde{g} + (H_{p} \vob) \tilde{g} = &  -(H_{|\ltau^{\mathrm{3co}}|^{2} + | \lmu^{\mathrm{3co}} |_{h_{\tindex}}^2} \vob_{\tindex} ) g + (H_{p} \vob) g  \\
 &  -(H_{|\ltau^{\mathrm{3co}}|^{2} + | \lmu^{\mathrm{3co}} |_{h_{\tindex}}^2} \vob_{\tindex} - H_p \vob ) (  \tilde{g} - g ),
\end{align*}
where the last term in the above can be modded out by (\ref{special commutator estimate section restriction of Hpb}) and the fact that $( \ltau^{\mathrm{3co}}, \lmu^{\mathrm{3co}} ) = (\ltau, \lmu)$ at $\tcocf$. By combining this with (\ref{commutator new cal 4}), we have thus shown that
\begin{align*}
\sigma(E) - ( \log x_{\tindex} ) (H_{p} \vob) g & = H_{p_1} g + H_{p_2} g - (\log x_{\tindex})( H_{|\ltau^{\mathrm{3co}}|^{2} + | \lmu^{\mathrm{3co}} |_{h_{\tindex}}^2} \vob_{\tindex} ) g \\
& \quad - ( 2 \vob_{\tindex} + 2 ) \ltau^{\mathrm{3co}} x_{\tindex}^{-\vob_{\tindex}} g_0 + x_{\tindex}^{-\vob_{\tindex} - 1}  H_{|\ltau^{\mathrm{3co}}|^{2} + | \lmu^{\mathrm{3co}} |_{h_{\tindex}}^2}  g_{0} \\
& = H_{p_1} g + H_{p_2} g + H_{|\ltau^{\mathrm{3co}}|^{2} + | \lmu^{\mathrm{3co}} |_{h_{\tindex}}^2}  ( x_{\tindex}^{-\vob_{\tindex} - 1} g_0 )
\end{align*}
which is equal to $H_{p} g$ as required. \par

Thus far, we have shown that (\ref{tangential discussion 0.85}) holds with $(\log x_{\tindex}) C_1, \tilde{R}$ in places of $B^{\ast} ( \log x )B$ and $R$ respectively. We now show that $( \log x_{\tindex} ) C_{1}$ can be replaced by $ B^{\ast} ( \log x_{\tindex} )B$, where $B$ is supported near $\cf$ and satisfy conditions (\ref{tangential discussion 0.8}) and (\ref{tangential discussion 0.9}), and $B$ is classical at $\cf$ modulo $\Psf^{-\infty, \vor, \vol - 1/2, \vor + \vol - 1/2, \vob - K, -\infty}$. These conditions imply $\hat{N}_{\cf}(B^{\ast}B) = \hat{N}_{\cf}(C_1)$, $\sigma(B^{\ast}B) = \sigma(C_1)$. Such a $B$ can be chosen by Lemma \ref{Lemma choice operator that has symbol and indicial operator at cf}. Moreover, if $A$ is chosen to be just some quantization, then so must be the case of $\hat{N}_{\cf}(A)$. Thus we can choose $B$ to be a quantization as well. Now, we can already replace $(\log x_{\tindex}) C_{1}$ by $ ( \log x_{\tindex} )B^{\ast}B$, with an error that can be absorbed into $\tilde{R}$. We will also absorb $i [ V, A^{\ast} A ]$ into the error term.   \par

To show that $( \log x_{\tindex} ) B^{\ast} B$ can be further replaced by $B^{\ast} ( \log x_{\tindex} ) B$, it suffices to note that
\begin{equation*}
( \log x_{\tindex} ) B^{\ast} - B^{\ast} ( \log x_{\tindex} ) \in \Psf^{-\infty, \vor - 3/2 + 2\delta, \vol - 3/2 + 2\delta, \vor + \vol -5/2 + 2 \delta, \vob - 2 + 2 \delta, -\infty}.
\end{equation*}
Indeed, this is again because $\log x_{\tindex}$ depends only on the free variables. Thus, the operator-valued symbols of $B^{\ast}$ and $\log x_{\tindex}$ commute at the principal level. Denoting the new error term by $R$ now proves (\ref{tangential discussion 0.85}) in the case where $A$ is supported near $\cf$ and $x = x_{\tindex}$.
\par

In the case of a general $A$, we can write $A = A_1 + A_2$, where $A_{1}$ is supported near $\cf$ and $A_2 \in \Psf^{-\infty, 2 \vor + 1, 2 \vol, -\infty, -\infty, -\infty}$. By applying the above to $A_{1}^{\ast}A_1$, we then conclude that
\begin{equation*}
i [ P, A_{1}^{\ast} A_1 ] = B^{\ast}_{1} (\log x_{\tindex}) B_{1} + E_{1,1} + R_{1,1},
\end{equation*}
where $B_1, E_{1,1}$ and $R_{1,1}$ are constructed as above. We will then write, say
\begin{equation*}
B_{1}^{\ast} ( \log x_{\tindex} ) B_{1} =  B_{1}^{\ast} ( \log x ) \big( \frac{\log x_{\tindex}}{\log x} \big) B_1,
\end{equation*}
where we note that $\log x_{\tindex} / \log x$ is a strictly positive $\mathcal{C}^{\infty}(\Xo)$ function near $\mathcal{C}_{\tindex}$ which converges to $1$ at $\partial \Xo$. Thus we can replace it by $1$ provided the error term is modified appropriately. The same logic can be applied in the principal symbol calculation. So noting that $\hat{N}_{\cf}(A^{\ast}A) = \hat{N}_{\cf}(A_{1}^{\ast} A_1)$, we have
\begin{align*} 
\begin{split}
 \hat{N}_{\cf}( E_{1,1} ) & =  i[  \hat{N}_{\cf}( \Delta_{z^{\tindex}} ) + 2 \ltau \hbff D_{\hbff}, \hat{N}_{\cf}( A^{\ast} A ) ] \\
& \quad + \mathsf{H}_{\inttcf} \hat{N}_{\cf}( A^{\ast}A ) - ( 2\vob + 2 ) \ltaucf \hat{N}_{\cf}( A^{\ast} A ),  \\
\sigma( E_{1,1} ) & = H_{p} \sigma(A_1^{\ast} A_1) + (\log x ) ( H_{p} \vob ) \sigma( A_1^{\ast} A_1 ),
\end{split}
\end{align*} 
where we used that $\hat{N}_{\cf}(A) = \hat{N}_{\cf}(A_1)$.  We also have
\begin{equation*}
\hat{N}_{\cf}( B_1 ) = ( - \sH_{\intt} \vob )^{1/2} \hat{N}_{\cf}( A ), \ \sigma(B_1)  = ( -H_p \vob )^{1/2} \sigma(A_1)
\end{equation*} 
by construction.
\par
It remains to define $B_{2} \in \Psf^{-\infty, \vor, \vol - 1/2, -\infty , - \infty , -\infty}$ such that $\sigma(B_2) = ( - H_{p} \vob )^{1/2} \sigma(A_2)$. For all other $i,j =1,2$ not all equal to $1$, we also define $E_{i,j} \in \Psf^{-\infty, 2 \vor + 2 \delta, 2 \vol, -\infty,  -\infty, -\infty}$ such that $\sigma(E_{i,j}) = H_{p} \sigma( A_{i}^{\ast} A_{j}) + ( \log x ) ( H_{p} \vob ) \sigma(A_{i}^{\ast} A_{j})$. Then by the symbol calculus, we plainly have
\begin{equation*}
i [ P, A_{i} A_{j} ] = B_{i} ( \log x ) B_{j} + E_{i,j} + R_{i,j},
\end{equation*}
where $R_{i,j} \in \Psf^{-\infty, 2 \vor - 1 + 4 \delta, 2 \vol, -\infty , -\infty, -\infty }$. It follows that if we set
\begin{equation*}
B = B_{1} + B_{2}, \ E = E_{1,1} + E_{1,2} + E_{2,1} + E_{2,2}, \ R = R_{1,1} + R_{1,2} + R_{2,1} + R_{2,2},
\end{equation*}
then the required conditions must all hold.
\end{proof}

\section{Real principal type propagation of regularity} 
\label{principal type propagation section}
Having established the necessary tools and background, in this section, we finally begin our discussion on microlocal propagation. We first look at principal type propagation of regularity estimate, which can be proved with respect to any variable orders $\vor_{\pm}, \vol_{\pm}$ and $\vob_{\pm}$ satisfying the conditions specified in \S \ref{variable order construction section}. However, for definiteness, we will only be focusing on the estimate for $\vor_{+}, \vol_{+}$ and $\vob_{+}$. Modifications for the case of $\vor_{-}, \vol_{-}$ and $\vob_{-}$ are then straightforward. See the beginning of \S \ref{almost semi-Fredholm estimates with symbolic decay section} for details. \par

\begin{proposition} \label{principal type propagation second microlocal version local}
Let $\vor = \vor_{+}$, $\vol = \vol_{+}$, $\vob = \vob_{+}$ be variable orders satisfying the conditions specified in \S \ref{variable order construction section}. Let $\gamma$ be an integral curve segment of $\rho_{\dmf}^{-1} \rho_{\dtsccf}^{-1} \rho_{\tcocf}^{-2} H_{p}$ such that $\gamma(T) = \beta$ for some $T>0$. Moreover, suppose that $E \in \Psf^{-\infty, 0,0,0,0,-\infty}$ satisfy $\gamma(0) \in \Ellss(E)$. Then we can find an arbitrarily small neighborhood $U$ of $\beta$ for which the following implication holds: If we choose $B \in \Psf^{-\infty, 0, -\infty , 0 ,0 ,-\infty}$ with $\WFs(B) \subset U$, then for every $u \in \mathcal{S}'$ and any choice of $N, M, L, K, S \in \mathbb{R}$, there exists a constant $C > 0$ such that
\begin{equation} \label{principal type propagation of regularity semi-global estimate}
\| B u \|_{ H_{\mathrm{d3sc,3co,res}}^{\ast, \vor, \vol, \vor + \vol, \vob, \ast} }  \leq C ( \| Pu \|_{H_{\mathrm{d3sc,3co,res}}^{\ast, \vor + 1, \vol , \vor + \vol + 1, \vob + 2, \ast}}   + \| E u \|_{H_{\mathrm{d3sc,3co,res}}^{\ast, \vor, \vol , \vor + \vol, \vob, \ast}} + \| u \|_{ H_{\mathrm{d3sc,3co,res}}^{ N, M, \vol , K, \vob, S } } )
\end{equation}
in the strong sense that if the right hand of (\ref{principal type propagation of regularity semi-global estimate}) is finite, then so is the left hand side, and the estimate holds.
\end{proposition} 
\begin{proof}
Although we will provide full argument, we note that the proof is standard, with the situation away from the corners being identical to \cite[Proposition 5.24]{AndrasBook}. The point of our repetition here is to demonstrate the usage of Corollary \ref{commutator corollary} when one is near $\tcocf$,  \par

Typically, one starts by proving a local version of the proposition, focusing on a segment $\gamma ([ 0, t_0 ])$, where $\gamma$ is an integral curve of $\rho_{\dmf}^{-1} \rho_{\dtsccf}^{-1} \rho_{\tcocf}^{-2} H_{p}$, and $t_0 > 0$ is small. By a rescaling, we can also assume that $t_0 = 1$. Notice that the proposed statement is null if $\gamma$ is stationary. Thus assuming otherwise, then we must have $\rho_{\dmf}^{-1} \rho_{\dtsccf}^{-1} \rho_{\tcocf}^{-2} H_{p} \neq 0$ at $\gamma(0)$. It follows from the standard ODE theory that we can find local coordinates $q = ( q_{1}, q' )$ centered at $\gamma(0)$ such that $
\rho_{\dmf}^{-1} \rho_{\dtsccf}^{-1} \rho_{\tcocf}^{-2} H_{p}  = \partial_{q_{1}}$. \par

Next, we let $\psi_{1} \in \mathcal{C}^{\infty}(\mathbb{R})$ be such that $\psi_{1} = e^{-F/t}$ for $t \geq 0$ and $\psi_{1} = 0$ for $t \leq 0$. Let also $\psi_{2} \in \mathcal{C}^{\infty}( \mathbb{R} )$ be such that $\psi_{2} = 0$ on $(-\infty, - \epsilon )$, $\psi_{2}' \geq 0$, and $\psi_{2} = 1$ identically on $[ \epsilon  , \infty )$ where $\epsilon > 0$ is small. Finally, let $\psi_{3} \in \mathcal{C}^{\infty}_{c}( [ 0 ,\infty) ) $ be a cut-off at $0$. Then we will define
\begin{equation*}
\varphi_{1} = \psi_{1}( 1 + \epsilon - q_1 ), \ \varphi_{2} = \psi_{2}( q_{1} ), \ \varphi_{3} = \psi_{3} ( |q'|^2  )
\end{equation*}
and set
\begin{equation*}  
a_0 = \rho_{\dmf}^{-\vor - 1/2} \rho_{\dff}^{-\vol} \rho_{\dtsccf}^{-\vor - \vol - 1/2} \rho_{\tcocf}^{-\vob - 1} \varphi_{1} \varphi_{2} \varphi_{3},
\end{equation*}
which is supported on a small neighborhood of $\gamma( [0,1] )$, and the support of $a_0$ contains a slightly smaller neighborhood of $\gamma([0,1])$. \par

We remark that in general, $\gamma$ can intersect at most three boundary faces of $\psf \Xd$. Thus at least one of the defining functions in the definition of $a_0$ can be dropped. We keep all four here in order to have a globally valid proof, but one should always keep in mind of this localization. To ensure that the logarithmic terms below can be controlled, we also force the defining functions to be bounded by $1$. \par

We now compute 
\begin{equation*}
H_{p} a_0^2 = -\tilde{b}_0^2 + 2 ( \log x_{\cf} ) b_{1,0}^2 - f_0 + e_0.
\end{equation*}
for some symbols $b_0$, $b_{1,0}$, $f_0$ and $e_0$, and $x_{\cf} \in \mathcal{C}^{\infty}(\Xd)$ is a boundary defining function for $\cf$. Before explaining what the definition of these terms are, let us first note that in the computation of $H_p a_0^2$, one must encounter a sum of logarithmic terms
\begin{align*}
- 2 \big( ( H_p \vor )( \log \rho_{\dmf} ) + ( H_p \vol ) ( \log \rho_{\dff} ) + (H_{p} ( \vor + \vol ) ) ( \log \rho_{\dtsccf} ) + ( H_p \vob ) ( \log \rho_{\tcocf} ) \big) a_0^2.
\end{align*}
Such a sum can be written as the negative of a sum of squares. By the symbol calculus, one could thus quantize it into the negative of a non-negative operator, which can be dropped in a positive commutator estimate. However, notice that the presence of $\log \rho_{\tcocf}$ in the sum forces a loss of arbitrarily small order (say $4\delta$) at $\tcocf$. The discussion at the beginning of \S \ref{Subsection commutator formulae} then suggests that simply using the symbol calculus will not be sufficient, at least when $\gamma$ lives on the boundary of $\tcocf$. \par

To get around this issue, we will use Corollary \ref{commutator corollary}. For this, we will need to arrange the logarithmic terms slightly differently. Suppose that $x_{\cf} = \rho_{\dtsccf} \rho_{\tcocf}$ (note that this can at least be arranged locally, which is really all we need in an estimate), then noting that $H_{p} \vob = 2 H_{p} \vol$ for $\vob - 2 \vol$ is a constant, we will set
\begin{equation*}
b_{1,0} = ( - H_p \vob )^{1/2} a_0, \, f_0 = 2 \big( (H_p \vor) (\log \rho_{\dmf}) + (H_p \vol) ( \log \rho_{\dff} ) + ( H_p (\vor - \vol) ) ( \log \rho_{\dtsccf} ) \big) a_0^2.
\end{equation*}
Here the idea is that $b_{1,0}$ should quantize into the operator $B$ in Lemma \ref{commutator formula lemma}. Though this is not made explicitly in the above, we also emphasize that $\rho_{\dmf}^{-1} \rho_{\dtsccf}^{-1} \rho_{\cf}^{-2} H_{p}$ vanishes to order one at both $\psf_{\dff} \Xd$ and $\dtsccf$, so that
\begin{equation*}
b_{1,0} \in S^{-\infty, \vor, \vol - 1/2, \vor + \vol - 1/2, \vob, -\infty}, \, f_0 \in S^{-\infty, 2 \vor + 2 \delta, 2 \vol - 1 + 2 \delta, 2 \vor + 2 \vol - 1 + 4 \delta, 2 \vob, -\infty}.
\end{equation*}
Furthermore, as usual we define
\begin{equation*}
\begin{gathered}
\tilde{b}_0^2 =  \rho_{\dmf}^{-2 \vor} \rho_{\dff}^{-2 \vol} \rho_{\dtsccf}^{-2 \vor - 2 \vol} \rho_{\tcocf}^{-2 \vob} \Big( \frac{2F}{ ( 1 + \epsilon + q_1 )^2 } + f' \Big) \varphi_1^2 \varphi_2^2 \varphi_3^2, \\
e_0 = 2 \rho_{\dmf}^{-2 \vor} \rho_{\dff}^{-2 \vol} \rho_{\dtsccf}^{-2 \vor - 2 \vol} \rho_{\tcocf}^{-2 \vob} \psi_2'(q_1) \varphi_2 \varphi_1^2 \varphi_3^2,
\end{gathered}
\end{equation*}
where $\tilde{b}_{0}^2$ can indeed be written as a square of symbol if $F$ is large enough, while
\begin{align*}
f' & = \rho_{\dmf}^{-1} \rho_{\dtsccf}^{-1} \rho_{\tcocf}^{-2} \big( (2 \vor + 1) (\rho_{\dmf}^{-1} H_p \rho_{\dmf} ) + 2 \vol ( \rho_{\dff}^{-1} H_p \rho_{\dff} ) \\
& \quad + ( 2 \vor + 2 \vol + 1 ) ( \rho_{\dtsccf} H_{p} \rho_{\dtsccf} ) + ( 2 \vob + 2 ) ( \rho_{\tcocf}^{-1} H_{p} \rho_{\tcocf} ) \big)
\end{align*}
comes from differentiating the weights, and we note that each $ \rho_{\dmf}^{-1} \rho_{\dtsccf}^{-1} \rho_{\tcocf}^{-2}  (\rho_{\bullet} H_{p} \rho_{\bullet}) $ is smooth since $\rho_{\dmf}^{-1} \rho_{\dtsccf}^{-1} \rho_{\tcocf}^{-2} H_{p}$ is a b-vector field. Note that the support of $e_0$ is contained in a small neighborhood of $\gamma(0)$. \par

It remains to consider regularization. Concretely, we can let $x$ be the canonical boundary defining function for $\Xo$, and set
\begin{equation*}
\phi_{s} = 1 + \frac{s}{x}, \, h_{s} =  \frac{s/x}{1+s/x}, \quad s \in [0,1].
\end{equation*}
For large $K \geq 0$, we can then compute that 
\begin{equation*}
H_{p} \phi_{s}^{-2K} = 4 K x h_{s} \tau \phi_{s}^{-2K}.
\end{equation*}
Let $a_{s} = \phi_{s}^{-K} a_s$,  $b_{1,s} = \phi_{s}^{-K}b_{1,0}$,  $f_{s} = \phi_{s}^{-2K} f_0$, $e_s = \phi_{s}^{-2K} e_0$, as well as 
\begin{align*}
b_s^2  & =  \rho_{\dmf}^{-2 \vor} \rho_{\dff}^{-2 \vol} \rho_{\dtsccf}^{-2 \vor - 2 \vol} \rho_{\tcocf}^{-2 \vob} \phi_s^{-2K} \\
& \quad \times \Big( \frac{2F}{ ( 1 + \epsilon + q_1 )^2 } + f' - \frac{ 4Kx h_s }{ \rho_{\dmf} \rho_{\dtsccf} \rho_{\tcocf}^2 }  - \delta_0 \phi_{s}^{-2K} \varphi_1^2 \varphi_2^2 \varphi_3^2  \Big) \varphi_1^2 \varphi_2^2 \varphi_3^2
\end{align*}
for any $\delta_0 > 0$. Here we remind the readers that $x \simeq \rho_{\dmf} \rho_{\dff} \rho_{\dtsccf}^2 \rho_{\tcocf}^2 \rho_{\rf}$, thus the terms in the large bracket above are indeed smooth. Moreover, the supports of $\{ a_s \}, \{ b_{s} \}$ uniformly contain a small neighborhood of $\gamma([0,1])$, with them being uniformly supported on a slightly larger neighborhood of $\gamma([0,1])$; the family $\{ e_{s} \}$ is also uniformly supported on a small neighborhood of $\gamma(0)$. Then we have
\begin{equation} \label{principal type propagation symbolic identity full}
H_p a_s^2 + \delta_0 \rho_{\dmf}^{2 \vor + 2} \rho_{\dff}^{2\vol} \rho_{\dtsccf}^{2 \vor + 2 \vol + 2} \rho_{\tcocf}^{2 \vob + 4} a_{s}^4 = - b_s^2 + 2 ( \log x_{\cf} ) b_{1,s}^2 - f_{s} + e_s.
\end{equation} \par
Let 
\begin{equation*}
\begin{gathered}
A_s \in L^{\infty}( (0,1]_{s} ; \Psf^{-\infty, 2 \vor + 1, 2 \vol, 2 \vor + 2 \vol + 1, 2 \vob + 2, -\infty} ), \, B_{s} \in L^{\infty}( (0,1]_{s} ; \Psf^{-\infty, \vor, \vol, \vor + \vol, \vob, -\infty} ), \\
F_{s} \in L^{\infty}( (0,1]_{s} ; \Psf^{-\infty, 2 \vor + 2 \delta, 2 \vol - 1 + 2 \delta, 2 \vor + 2 \vol - 1  + 4\delta, 2 \vob, -\infty } ), \, E_{s} \in L^{\infty}( (0,1]_s ; \Psf^{-\infty, 2 \vor, 2 \vol, 2 \vor + 2 \vol, 2 \vob, -\infty} ).
\end{gathered}
\end{equation*}
be some quantizations of $a_s$, $b_s$, $f_s$ and $e_s$ respectively, such that each $F_{s}$ is arranged to be non-negative with respect to $L^2$. Let also
\begin{equation*}
B_{1,s} \in L^{\infty}( (0,1]_s ; \Psf^{-\infty, \vor, \vol - 1/2, \vor + \vol - 1/2, \vob, -\infty} )
\end{equation*}
be such that
\begin{equation} \label{always these conditions for B1s}
\sigma(B_{1,s}) = ( - H_p \vob )^{1/2} a_s, \ \hat{N}_{\cf}( B_{1,s} ) = ( - \sH_{\intt} \vob )^{1/2} \hat{N}_{\cf}(A_s).
\end{equation}
Then by Corollary \ref{commutator corollary}, we have
\begin{equation*}
i[ P, A_{s}^{\ast}A_s ] - 2 B_{1,s}^{\ast}( \log x_{\cf} ) B_{1,s} \in L^{\infty}( [0,1)_{s} ; \Psf^{-\infty, 2 \vor + 2 \delta, 2 \vol, 2 \vor + 2 \vol + 2 \delta, 2 \vob, -\infty} ),
\end{equation*}
with principal symbol $H_{p} a_{s}^2 + 2 ( \log x_{\cf} )( H_{p} \vob )a_{s}^2$. Thus by (\ref{principal type propagation symbolic identity full}) and the definition of $b_{1,s}^2$, we can use the symbol calculus to find that
\begin{equation} \label{principal type propagation operator identity full}
i[ P, A_{s}^{\ast} A_{s} ] + \delta_{0} ( \Lambda A_{s}^{\ast} A_{s} )^{\ast} ( \Lambda A_{s}^{\ast} A_{s} ) = - B_{s}^{\ast} B_{s} + 2 B_{1,s}^{\ast} (\log x_{\cf} ) B_{1,s} - F_{s} + E_{s} + R_{s}
\end{equation}
where $\Lambda \in \Psf^{-\infty, - \vor - 1, - \vol, - \vor - \vol - 1, - \vob - 2, -\infty}$ is some quantization of $ \rho_{\dmf}^{ \vor + 1} \rho_{\dff}^{\vol} \rho_{\dtsccf}^{\vor + \vol + 1} \rho_{\tcocf}^{\vob + 2}$, and
\begin{equation*}
R_{s} \in L^{\infty}( (0,1]_{s} ; \Psf^{-\infty, 2 \vor - 1 + 4 \delta, 2 \vol, 2 \vor + 2 \vol - 1 + 4 \delta, 2 \vob, -\infty} )
\end{equation*}
has no unsavory $\delta$ losses at $\tcocf$. \par

The rest of the argument is truly standard: We shall pair (\ref{principal type propagation operator identity full}) with $u \in \mathcal{S}'$ in $L^2$. Then by dropping some non-negative terms, we get that
\begin{equation} \label{principal type cal 1}
\| B_{s} u \|_{L^{2}}^2 + \delta_0 \| \Lambda A_{s}^{\ast} A_{s} \|_{L^2}^2 \leq - 2 \mathrm{Im} \langle Pu, A_{s}^{\ast} A_s u \rangle_{L^2} + \langle E_{s}u, u \rangle_{L^2} + \langle R_{s} u , u \rangle_{L^2}.
\end{equation}
Let us choose $G,\tilde{E} \in \Psf^{-\infty, 0,0,0,0,-\infty}$ such that 
\begin{equation} \label{principal type cal 1.5}
\begin{gathered}
\WFsL( \{ A_{s} \} ) \subset \Ellss( G ), \, \, \WFsL(\{ E_s \}) \subset \Ellss( \tilde{E} ), \\
 \WFs( I - G ) \cap \WFsL( \{ A_s \} ) = \WFs( I - \tilde{E} ) \cap \WFs( \{ E_{s} \} ) = \emptyset, \\
\end{gathered}
\end{equation}
and $\WFs( G )$ is contained in a small neighborhood of $\gamma([0,1])$, $\WFs(\tilde{E})$ is contained in a small neighborhood of $\gamma(0)$. Noting also that $\Lambda$ must be elliptic (in the symbolic sense) by construction. Then by elliptic regularity and Young's inequality, we have
\begin{equation} \label{principal type cal 2}
2  | \mathrm{Im} \langle Pu, A_s^{\ast} A_s \rangle_{L^{2}} | \leq C \delta_0^{-1} \| P u \|_{ \so^{\ast, \vor + 1, \vol , \vor + \vol +1, \vob + 2, \ast } }^2  + \delta_0 \| \Lambda A_s^{\ast} A_s u \|_{L^{2}}^2 + C \| u \|_{ \so^{N, M, \vol , K, \vob, S } }^2.
\end{equation}
as well as
\begin{equation} \label{principal type cal 3}
\begin{gathered}
| \langle E_{s}u , u  \rangle_{L^2} | \leq C ( \| \tilde{E} u \|_{H_{\mathrm{d3sc,3co,res}}^{\ast, \vor, \vol, \vor + \vol, \vob, \ast}}^2 + \| u \|_{H_{\mathrm{d3sc,3co,res}}^{N,M,\vol,K,\vob,S}}^2 ), \\
| \langle R_s u ,u  \rangle_{L^2}  | \leq C ( \| G u \|_{H_{\mathrm{d3sc,3co,res}}^{\ast, \vor - 1/2 + 2 \delta, \vol, \vor + \vol - 1/2 + 2 \delta, \vob, \ast }}^2 + \| u \|_{H_{\mathrm{d3sc,3co,res}}^{N,M, \vol, K, \vob, S}}^2  ).
\end{gathered}
\end{equation}
By substituting (\ref{principal type cal 2}) and (\ref{principal type cal 3}) back into (\ref{principal type cal 1}), we now find
\begin{align}  \label{principal type cal 4}
\begin{split}
\| B_{s} u \|_{L^2} & \leq  C ( \| P u \|_{H_{\mathrm{d3sc,3co,res}}^{\ast, \vor + 1 , \vol, \vor + \vol + 1, \vob +2 , \ast}} \\
& \quad + \| \tilde{E} u \|_{H_{\mathrm{d3sc,3co,res}}^{\ast, \vor, \vol, \vor + \vol, \vob, \ast}} + \| G u \|_{H_{\mathrm{d3sc,3co,res}}^{\ast, \vor - 1/2 + 2\delta, \vol, \vor + \vol - 1/2 + 2 \delta, \vob, \ast}} + \| u \|_{H_{\mathrm{d3sc,3co,res}}^{N,M,\vol,K,\vob,S}} )
\end{split}
\end{align}
with the constant $C > 0$ being independent of $s$. Since the unit ball in $L^2$ is compact, $B_{s}u$ has a weakly convergent subsequence with limit $v \in L^2$. On the other hand, $B_{s} u \rightarrow B_0 u$ weakly. Hence $v = B_{0} u$. Taking this limit, and let $U$ (as in the statement of the proposition) be defined such that $U \subset \WFs(B_0)$. Then by making another application of elliptic regularity, namely
\begin{equation*}
\| B u \|_{H_{\mathrm{d3sc,3co,res}}^{\ast, \vor, \vol, \vor + \vol, \vob, \ast}} \leq C( \| B_0 u \|_{L^2} + \| u \|_{H_{\mathrm{d3sc,3co,res}}^{N,M,\vol, K, \vob, S}} ),
\end{equation*}
we can conclude from (\ref{principal type cal 4}) that
\begin{align} \label{principal type cal 5}
\begin{split}
\| B u \|_{L^2} & \leq  C ( \| P u \|_{H_{\mathrm{d3sc,3co,res}}^{\ast, \vor + 1 , \vol, \vor + \vol + 1, \vob +2 , \ast}} \\
& \quad + \| \tilde{E} u \|_{H_{\mathrm{d3sc,3co,res}}^{\ast, \vor, \vol, \vor + \vol, \vob, \ast}} + \| G u \|_{H_{\mathrm{d3sc,3co,res}}^{\ast, \vor - 1/2 + 2\delta, \vol, \vor + \vol - 1/2 + 2 \delta, \vob, \ast}} + \| u \|_{H_{\mathrm{d3sc,3co,res}}^{N,M,\vol,K,\vob,S}} ).
\end{split}
\end{align} \par

Finally, it is clear that the above argument can be iterated for as many times as we like. Notice that by construction, $\WFsL( \{ A_{s} \} )$ contains a small neighborhood of $\gamma([0,1])$, while $\WFsL(\{ \tilde{E}_{s} \})$ contains a small neighborhood of $\gamma(0)$. Thus after finitely many iterations, we can get
\begin{equation*} 
\| B u \|_{ H_{\mathrm{d3sc,3co,res}}^{\ast, \vor, \vol, \vor + \vol, \vob, \ast} }  \leq C ( \| Pu \|_{H_{\mathrm{d3sc,3co,res}}^{\ast, \vor + 1, \vol , \vor + \vol + 1, \vob + 2, \ast}}   + \| \tilde{E}' u \|_{H_{\mathrm{d3sc,3co,res}}^{\ast, \vor, \vol , \vor + \vol, \vob, \ast}} + \| u \|_{ H_{\mathrm{d3sc,3co,res}}^{ N, M, \vol , K, \vob, S } } )
\end{equation*}
where $\tilde{E}' \in \Psf^{-\infty, 0,0,0,0, -\infty}$ is such that $\Ellss(\tilde{E}')$ contains a slightly larger neighborhood of $\gamma(0)$ and $\WFs( I - \tilde{E}' )$ is disjoint from a slightly smaller neighborhood of $\gamma(0)$. The required claim now follows from elliptic regularity estimate. \par

Thus far, we have proved a local version of the required estimate (\ref{principal type propagation of regularity semi-global estimate}), focusing on a single $\gamma$. The global version, where the length of $\gamma$ is not restricted, can then be patched together by compactness. This concludes the proof of the proposition.
\end{proof}
As usual (see for instance \cite[Chapter 9, \S 6]{PeterNotes}), it is also convenient to state a semi-global version of Proposition \ref{principal type propagation second microlocal version local}:
\begin{proposition}  \label{principal type propagation second microlocal version}
Let $\vor = \vor_{+}$, $\vol = \vol_{+}$, $\vob = \vob_{+}$ be variable orders satisfying the conditions specified in \S \ref{variable order construction section}. Suppose that $B, E \in \Psf^{-\infty, 0,0,0,0,-\infty}$ satisfy $\WFs(B) \subset \Ellss(G)$. Moreover, assume that 
\begin{equation*}
\begin{gathered}
\text{all backward integral curves of $\rho_{\dmf}^{-1} \rho_{\dtsccf}^{-1} \rho_{\tcocf}^{-2} H_{p}$} \\
\text{ starting from $\WFs(B) \cap \bcv$ enter $\Ellss(E)$ in finite time}. 
\end{gathered}
\end{equation*}
Then for every $u \in \mathcal{S}'$ and any choice of $N, M, L, K, S \in \mathbb{R}$, there exists a constant $C > 0$ such that
\begin{equation} \label{principal type propagation of regularity semi-global estimate}
\| B u \|_{ H_{\mathrm{d3sc,3co,res}}^{\ast, \vor, \vol, \vor + \vol, \vob, \ast} }  \leq C ( \| Pu \|_{H_{\mathrm{d3sc,3co,res}}^{\ast, \vor + 1, \vol , \vor + \vol + 1, \vob + 2, \ast}}   + \| E u \|_{H_{\mathrm{d3sc,3co,res}}^{\ast, \vor, \vol , \vor + \vol, \vob, \ast}} + \| u \|_{ H_{\mathrm{d3sc,3co,res}}^{ N, M, \vol , K, \vob, S } } )
\end{equation}
in the strong sense that if the right hand of (\ref{principal type propagation of regularity semi-global estimate}) is finite, then so is the left hand side, and the estimate holds.
\end{proposition}

We note that Proposition \ref{principal type propagation second microlocal version} follows easily from Proposition \ref{principal type propagation second microlocal version local}, a compactness argument, and elliptic regularity estimate. The proof is therefore omitted.

\begin{remark} \label{remark unnecessary to cover full details}
Once a formula such as (\ref{principal type propagation symbolic identity full}) has been arranged, and assuming  appropriate conditions such as (\ref{principal type cal 1.5}) are satisfied, then one can typically arrive at an estimate of the form (\ref{principal type cal 4}) by going through calculations which are exactly the same as the ones above (though one might not always be able to iterate the argument). This is in particular true for many of the radial point estimates we will consider below. In order to not burden the length of this paper further, henceforth we shall refer to the aforementioned calculations simply as ``the standard procedure'', and omit much of the details whenever we can.
\end{remark}

\begin{remark}
The presence of $\delta_0 > 0$ in (\ref{principal type propagation symbolic identity full}) is clearly redundant, since one has the freedom of choosing $F > 0$ large. We insert such a quantity here since this will be necessary in the radial point estimates below. Thus in doing so, our argument here can get carried over without any modification.
\end{remark}

As discussed in \S \ref{variable order construction section}, the second microlocal structure can be somewhat relaxed when one is away from $\tcocf$. In the case of principal type propagation, this is reflected in that regularity can be shown to propagates along the flow of $\sH_{p,\mf}$, as opposed to $\rho_{\dmf}^{-1} \rho_{\dtsccf}^{-1} \rho_{\tcocf}^{-2} H_{p}$, when one is away from $\tcocf$. Moreover, one could localize microlocally by using the three-body calculus. Then we have the follow estimate, the proof of which will be omitted. 
\begin{proposition} \label{principal type propagation three-body version}
Let $\vor = \vor_{+}$, $\vol = \vol_{+}$, $\vob = \vob_{+}$ be the variable orders constructed in \S \ref{variable order construction section}. Suppose that $B,G,E \in \Psi^{-\infty, 0 ,0}_{\mathrm{3sc}}$ satisfy $\mathrm{WF}'_{\mathrm{3sc},\sigma}(B) \subset \mathrm{Ell}_{\mathrm{3sc},\sigma}(E)$ and
\begin{equation*}
 \mathrm{WF}_{\mathrm{3sc},\sigma}'(B) \cap  \pi_{\ff}^{-1}(o_{\mathcal{C}^{\tindex}}) = \mathrm{WF}_{\mathrm{3sc},\sigma}'(G) \cap \pi_{\ff}^{-1}(o_{\mathcal{C}^{\tindex}}) = \mathrm{WF}'_{\mathrm{3sc},\sigma}(E) \cap \pi_{\ff}^{-1}(o_{\mathcal{C}^{\tindex}}) = \emptyset.
\end{equation*}
 Assume that
 \begin{equation*}
 \begin{gathered}
\text{all backward integral curves of $\sH_{p,\mf}$} \\
\text{starting from $\mathrm{WF}_{\mathrm{3sc},\sigma}'(B) \cap \Sigma_{\mf}$ enter $\mathrm{Ell}_{\mathrm{3sc},\sigma}(E)$ in finite time.} 
\end{gathered}
\end{equation*} 
Then for every $u \in \mathcal{S}'$ and any choice of $N, M, K, F, S \in \mathbb{R}$, there exists a constant $C > 0$ such that
\begin{equation*}
\| B u \|_{ H_{\mathrm{d3sc,3co,res}}^{\ast, \vor, \vol, \vor + \vol, \ast, \ast} }  \leq C ( \| Pu \|_{H_{\mathrm{d3sc,3co,res}}^{\ast, \vor + 1, \vol , \vor + \vol + 1, \ast , \ast}}   + \| E u \|_{H_{\mathrm{d3sc,3co,res}}^{\ast, \vor, \vol , \vor + \vol, \ast , \ast}} + \| u \|_{ H_{\mathrm{d3sc,3co,res}}^{ N, M, \vol , K, F, S } } )
\end{equation*}
\end{proposition}

\section{Radial point estimates in the tangential directions} 
\label{section: radial point estimates in the tangential directions}
We will now begin our discussion on radial point estimates. In this section, we will consider radial point estimates in the tangential directions, namely those radial points whose dynamical significances are contained in $\tcocf$. Here, the terminology of `tangential propagation' is slightly abusive{\ep}it is due to the condition that $(y_{\tindex}, \ltau, \lmu) \in \Sigma_{\mathrm{t}}$ when restricted to $\Sigma_{\sigma} \cap \tcocf$. In particular, one should not confused it with the notion of tangential propagation discussed in \cite[Chapter 15]{AndrasThesis}. Notice that we have excluded the discussions on $\mathcal{R}_{\mathrm{2sc},\dmf,\pm}$ in this section, which are of a more global dynamical nature, even though they have non-empty intersections with $\tcocf$. \par

As pointed out in Remark \ref{remark unnecessary to cover full details}, we will omit much of the unnecessary details in this section, and focus primarily on constructing the commutants, as well as the dynamical properties implied by these constructions. 
\label{tangential propagation section}

\subsection{Radial point estimates at $\mathcal{R}_{\mathrm{n},\dmf,\pm} \cap \tcocf$} \label{subsection first radial point estimate at dmf}
Let us first consider radial point estimates at $\mathcal{R}_{\mathrm{n}, \dmf, \pm} \cap \tcocf$. To this end, notice that $\mathcal{R}_{\mathrm{n},\dmf, \pm} \cap \tcocf$ fibers over $\Sigma_{\mathrm{t}}$. For each $\beta_{\tindex} \in \Sigma_{\mathrm{t}}$ such that
\begin{equation} \label{local coordinates for free beta}
\beta_{\tindex} = ( y_{\tindex,0}, \tau_{\tindex,0}, \mu_{\tindex,0}  ), \quad |\tau_{\tindex,0}|^{2} + | \mu_{\tindex,0} |^{2}_{h_{\tindex}}|_{y_{\tindex} = \, y_{\tindex,0}} = \lambda^{2},
\end{equation}
we define
\begin{equation}
\label{eq; definition of R n dmf at point}
\mathcal{R}_{\mathrm{n},\dmf, \pm} ( \beta_{\tindex} ) = \mathcal{R}_{\mathrm{n},\dmf,\pm} \cap \{ y_{\tindex} =  y_{\tindex,0}, \ltau = \tau_{\tindex,0}, \lmu = \mu_{\tindex,0} \}.
\end{equation}
Notice that $\mathcal{R}_{\mathrm{n},\dmf, \pm} ( \beta_{\tindex} ) \subset \mathcal{R}_{\mathrm{n},\dmf, \pm}  \cap \tcocf$ for all $\beta_{\tindex} \in \Sigma_{\mathrm{t}}$. \par

We will prove an estimate for every such $\mathcal{R}_{\mathrm{n},\dmf,\pm}(\beta_{\tindex})$.

\begin{proposition}[Below threshold estimate at $\mathcal{R}_{\mathrm{n,dmf},-}  \cap \tcocf$] \label{Proposition Rndm-}
Let $\vor = \vor_{+}$, $\vol = \vol_{+}$ and $\vob = \vob_{+}$ be variable orders satisfying the conditions specified in \S \ref{variable order construction section}. Suppose that $\beta_{\tindex} \in \Sigma_{\mathrm{t}}$. Moreover, let $E \in \Psf^{-\infty, 0, -\infty , 0 ,0 ,-\infty}$ be chosen such that 
\begin{equation*}
\begin{gathered}
\text{all backward integral curves of $\rho_{\dmf}^{-1} \rho_{\dtsccf}^{-1} \rho_{\tcocf}^{-2} H_{p}$ starting from $\mathcal{R}_{\mathrm{n},\dmf, -}(\beta_{\tindex})$} \\
\text{are either stationary or enter $\Ellss(E)$ in finite time. }
\end{gathered}
\end{equation*}
Then we can find an arbitrarily small neighborhood $U$ of $\mathcal{R}_{\mathrm{n}, \dmf,-}(\beta_{\tindex}) $ for which the following holds: If we choose $B \in \Psf^{-\infty, 0, -\infty , 0 ,0 ,-\infty}$ with $\WFs(B) \subset U$, then for every $u \in \mathcal{S}'$ and any choice of $N,M,L,K,S \in \mathbb{R}$, there exists $C > 0$ such that
\begin{equation} \label{semi-fredhom estimate at Rndmf-}
\| B u \|_{ H_{\mathrm{d3sc,3co,res}}^{\ast, \vor, \ast, \vor + \vol, \vob, \ast} }  \leq C ( \| Pu \|_{H_{\mathrm{d3sc,3co,res}}^{\ast, \vor + 1, \ast, \vor + \vol + 1, \vob + 2, \ast}}   + \| E u \|_{H_{\mathrm{d3sc,3co,res}}^{\ast, \vor, \ast, \vor + \vol, \vob, \ast}} + \| u \|_{ H_{\mathrm{d3sc,3co,res}}^{ N, M, L, K, \vob, S } } )
\end{equation}
in the strong sense that if the right hand of (\ref{semi-fredhom estimate at Rndmf-}) is finite, then so is the left hand side, and the estimate holds.
\end{proposition}
\begin{proof}
Suppose we make the change of variables from $( \hat{x}^{\tindex}, \bff, y_{\tindex}, y^{\tindex}, \ltau, \lmu, \utaures, \umures )$ to
\begin{equation} \label{Rndmf- cal0.1}
\hat{x}^{\tindex}, \,  ( \intn )^{1/2}, \,  y_{\tindex}, \,  y^{\tindex}, \, \ltau,  \, \lmu, \, \rhoresf, \,  \humures, 
\end{equation}
where we are writing $\rhoresf = ( \intnres )^{-1/2}$ and
\begin{equation} \label{Rndmf- cal0.2}
\hutaures = \frac{\utaures}{ ( \intnres )^{1/2} }, \  \humures = \frac{\umures}{ ( \intnres )^{1/2} }.
\end{equation}
We remind the readers here that in these coordinates, $\hat{x}^{\tindex}$ defines $\psf_{\dmf}\Xd$, $(\intn)^{1/2}$ defines $\tcocf$, $\rhoresf$ defines $\dtsccf$, and 
\begin{align*}
& \mathcal{R}_{\mathrm{n},\dmf,-}(\beta_{\tindex}) \\
& \qquad  = \{ \hat{x}^{\tindex} =  ( \intn )^{1/2} = \rhoresf = 0, \hutaures = -1,  \humures = 0, \intt = \lambda^2 \} .
\end{align*} 
Moreover, recall from (\ref{second microlocal dynamic modification cal 3}) and (\ref{definition of the standardly rescaled Hamiltonian in the res variables}) that we have
\begin{align*}
(\hat{x}^{\tindex})^{-1} ( \intn )^{-1} \rhoresf^{-1} H_{p} & = 2  ( 2 \hutaures - \rhoresf \ltau ) \hat{x}^{\tindex} \partial_{\hat{x}^{\tindex}} - 2  ( \hutaures -  \rhoresf \ltau ) \rhoresf \partial_{\rhoresf} \\
&  \quad + 2 \hutaures R_{\humures} + H_{ | \humures |_{h^{\tindex}}^2 } + \rhoresf \sHto,
\end{align*}
where $R_{\humures} = \humures \cdot \partial_{\humures}$ and $H_{| \humures |_{h^{\tindex}}^2} = \partial_{\humures} |\humures|_{h^{\tindex}}^2 \cdot \partial_{y^{\tindex}} - \partial_{y^{\tindex}} |\humures|_{h^{\tindex}}^2 \cdot \partial_{ \humures }$. \par
Now, consider the vector field
\begin{equation*}
 2 ( 1 + \rhoresf \ltau )  \partial_{\rhoresf} + \sHto \in \mathcal{V}( [0,\infty)_{\rhoresf} \times \mathcal{C}_{\tindex} \times \mathbb{R}^{n_{\tindex}} ),
\end{equation*}
which is inward pointing to the boundary of $[0 ,\infty)_{\rhoresf} \times \mathcal{C}_{\tindex} \times \mathbb{R}^{n_{\tindex}}$. It follows that we can find local coordinates $q_{\mathrm{n}} = ( q_{\mathrm{n},1}, q_{\mathrm{n}}' )$ centered at $\{ 0 \} \times \{ \beta_{\tindex} \}$ such that $q_{\mathrm{n},1}$ is a boundary defining function for $[0,\infty)_{\rhoresf} \times \mathcal{C}_{\tindex} \times \mathbb{R}^{n_{\tindex}}$, and that
\begin{equation*}
2 ( 1 + \rhoresf \ltau ) \partial_{\rhoresf} + \sHto = \partial_{q_{\mathrm{n},1}}.
\end{equation*}
In particular, we now have
\begin{align} \label{Rndmf- cal1}
\begin{split}
(\hat{x}^{\tindex})^{-1} ( \intn )^{-1} \rhoresf^{-1} H_{p} &  = 2 ( 2 \hutaures -  \rhoresf \ltau ) \hat{x}^{\tindex} \partial_{\hat{x}^{\tindex}} - 2  ( \hutaures  + 1 ) \rhoresf\partial_{\rhoresf}  \\
& \quad + 2 \hutaures R_{\humures} + H_{ | \humures |_{h^{\tindex}}^2 } +  \rhoresf \partial_{q_{\mathrm{n},1}}
\end{split}
\end{align}
in a small neighborhood of $\mathcal{R}_{\mathrm{n},\dmf,-}(\beta_{\tindex})$.
 \par

Let $\psi \in \mathcal{C}^{\infty}_{c}( [0, \infty) )$ be chosen such that
\begin{equation} \label{tangential propagation psi condition}
\begin{gathered}
\text{$\psi$ a cut-off at $0$, $\psi = 1$ in a small neighborhood of $0$,} \\
\text{$\psi \leq 0$, and ${(\psi' \psi)}^{1/2} \in \mathcal{C}^{\infty}_{c}( [0,\infty) )$.}
\end{gathered}
\end{equation}
Then we will set
\begin{equation*}
\begin{gathered}
\varphi_{1} = \psi( \hat{x}^{\tindex}  ), \, \varphi_{2} = \psi( \intn ), \, \varphi_{3} = \psi( \epsilon^{-1} | q_{\mathrm{n}}' |^2 + q_{\mathrm{n},1}  ), \, \varphi_{4} = \psi( ( \hutaures + 1 )^2 + | \humures |_{h^{\tindex}}^2 ),
\end{gathered}
\end{equation*}
where $\epsilon > 0$ is arbitrarily small, but will be henceforth fixed. Consider the symbol
\begin{equation*} \label{Rndmf- cal1.1}
a_0 = ( \hat{x}^{\tindex} )^{-\vor - 1/2} ( \intn )^{ -(\vob - 1)/2 }  \rhoresf^{- \vor - \vol - 1/2} \varphi_1 \varphi_2 \varphi_3 \varphi_4,
\end{equation*}
and notice that it is supported in an arbitrarily small neighborhood of $\mathcal{R}_{\mathrm{n,dmf},-}( \beta_{\tindex} )$ if in addition to the smallness of $\epsilon$, we take the support of $\psi$ to be arbitrarily small as well.
 \par

We now display some crucial calculations. Notice that we have
\begin{align}  \label{Rndmf- cal1.1.1}
\begin{split}
H_{p}( \hat{x}^{\tindex} )^{-2\vor - 1} = & - 2 ( H_{p} \vor ) ( \log \hat{x}^{\tindex} ) ( \hat{x}^{\tindex} )^{-2\vor - 1} \\
&  - ( 4 \vor + 2 ) ( 2 \hutaures - \rhoresf \ltau ) ( \intn ) \rhoresf  ( \hat{x}^{\tindex} )^{-2\vor},
\end{split}
\end{align}
and
\begin{equation}  \label{Rndmf- cal1.1.2}
H_{p}( \intn )^{ -\vob - 1 } =  - 2 ( H_{p} \vob ) ( \log ( \intn )^{1/2} ) ( \intn )^{ -\vob - 1 },
\end{equation}
as well as 
\begin{align}  \label{Rndmf- cal1.1.3}
\begin{split}
 H_{p} \rhoresf^{-2 \vor - 2 \vol - 1}  = &  - 2 ( H_{p} ( \vor + \vol ) ) ( \log  \rhoresf ) \rhoresf^{-2 \vor - 2\vol - 1} \\
&  + ( 4 \vor + 4 \vol + 2 )( \hutaures - \rhoresf \ltau ) \hat{x}^{\tindex} ( \intn ) \rhoresf^{-2\vor - 2\vol}.
\end{split}
\end{align}
Moreover, since $( q_{\mathrm{n},1}, q_{\mathrm{n}}' )$ depends only on $( \rhoresf, y_{\tindex}, \ltau, \lmu )$, by (\ref{Rndmf- cal1}) we also have
\begin{equation*}  \label{Rndmf- cal2}
(\hat{x}^{\tindex})^{-1} ( \intn )^{-1} \rhoresf^{-1} H_{p} ( |q_{\mathrm{n}}'|^2 + q_{\mathrm{n},1} )  = \rhoresf  ( 1 + O ( \hutaures + 1 ) ).
\end{equation*} 
Lastly, we have
\begin{equation*} \label{Rndmf- cal2.0.1}
H_{p}( ( \hutaures + 1 )^2 + | \humures |_{h^{\tindex}}^2 ) = -4 \hat{x}^{\tindex} ( \intn ) \rhoresf | \humures |_{h^{\tindex}}^2.
\end{equation*}
\par

Putting the above together, we can write
\begin{equation*}
H_{p} a_0^2 = - \tilde{b}_0^2 + 2 ( \log \bff ) b_{1,0}^2 - b_{2,0}^2  - f_0   + e_{1,0} + e_{2,0},
\end{equation*}
for some symbols $\tilde{b}_0$, $b_{1,0}$, $b_{2,0}$, $f_0$, $e_{1,0}$ and $e_{2,0}$, which we now proceed to define. \par

We first set
\begin{equation} \label{Rndmf- cal2.1} 
\tilde{b}_0^2 =  ( \hat{x}^{\tindex} )^{-2 \vor}  ( \intn )^{-\vob} \rhoresf^{-2 \vor - 2\vol} \big( ( 4 \vor - 4 \vol + 2 ) \hutaures + 4 \vol \rhoresf \ltau  \big) \varphi_1^2 \varphi_2^2 \varphi_3^2 \varphi_4^2,
\end{equation}
where the square root can be taken due to the threshold condition 
\begin{equation} \label{below threshold ndmf+}
\vor - \vol < - \frac{1}{2} \ \text{at} \ \mathcal{R}_{\mathrm{n,dmf},-}
\end{equation}
and the fact that
\begin{equation*} 
( 4 \vor - 4 \vol + 2 ) \hutaures - 4 \vol \rhoresf \ltau \simeq - 4 \vor + 4 \vol - 2 
\end{equation*}
on the support of $\tilde{b}_0^2$. \par

Meanwhile, if we gather the terms in $H_{p}a^{2}_0$ which contains logarithmic factors into a single sum, then we get
\begin{equation} \label{Rndmf- cal2.2} 
 - 2 \big(  (H_{p} \vor ) ( \log \hat{x}^{\tindex} ) + (   H_{p} \vob ) ( \log ( \intn )^{1/2} ) + ( H_{p} ( \vor + \vol ) ) ( \log \rhoresf ) \big) a^2_0.
\end{equation}
Noting that $x_{\tindex} = \hat{x}^{\tindex} ( \intn ) \rhoresf$, we can also consider the term given by
\begin{equation}  \label{Rndmf- cal2.2.0.1} 
 2 ( H_{p} \vob ) ( \log \bff ) a^2_0 =  2\big( ( H_{p} \vob ) ( \log (\intn)^{1/2} ) + (H_{p} \vob ) ( \log \rhoresf ) \big) a_0^2.
\end{equation}
Keeping in mind that $\vob - 2 \vol$ is a constant, and thus $H_{p} \vob = 2 H_{p} \vol$, it is therefore clear that (\ref{Rndmf- cal2.2}) can be written as $2 (\log \bff ) b_{1,0}^2 - f_0$, where 
\begin{equation} \label{Rndmf- cal2.2.0.0.1} 
b_{1,0}^2 = -  (H_{p} \vob) a^2_0, \, f_0 = 2 \big( ( H_{p} \vor ) ( \log \hat{x}^{\tindex} )  + ( H_{p} ( \vor - \vol ) ) ( \log \rhoresf ) \big) a^2_0.
\end{equation}
Here we remark that $b_{1,0}^2$, $f_0$ can be written as a square of symbol resp. a sum of squares of symbols due to the construction of the the variable orders. \par

We also have the terms coming from differentiating $\varphi_{3}^2$, from which we get
\begin{equation*} 
b_{2,0}^2  = - 2 ( \hat{x}^{\tindex} )^{-2\vor} ( \intn )^{-\vob} \rhoresf^{-2\vor - 2\vol + 1} \varphi_1^2 \varphi_2^2 \varphi_4^2  \varphi_3 \psi'(\epsilon^{-1} | q_{\mathrm{n}}' |^2 + q_{\mathrm{n},1}) ( 1 + O( \hutaures + 1 ) ).
\end{equation*}
This can again be written as a square of symbol due to (\ref{tangential propagation psi condition}) and the fact that $O( \hutaures + 1 )$ can be made arbitrarily small on the support of $b_{2,0}^2$. Lastly, differentiating $\varphi_1$ and $\varphi_4$ will produce terms for which we need a priori control of. These are respectively given by
\begin{equation} \label{Rndmf- cal2.2.1}
\begin{gathered}
e_{1,0} = - 4 ( \hat{x}^{\tindex} )^{-2 \vor + 1} ( \intn )^{-\vob} \rhoresf^{-2\vor - 2\vol} ( \rhoresf \ltau -  2 \hutaures  ) \psi'( \hat{x}^{\tindex} ) \varphi_1 \varphi_2^2 \varphi_3^2 \varphi_4^2, \\
e_{2,0}  = -8 ( \hat{x}^{\tindex} )^{-2 \vor} ( \intn )^{-\vob} \rhoresf^{-2\vor - 2\vol} |\humures|_{h^{\tindex}}^2  \psi'(( \hutaures + 1 )^2 + | \humures |_{h^{\tindex}}^2 ) \varphi_4 \varphi_1^2 \varphi_2^2 \varphi_3^2.
\end{gathered}
\end{equation} \par

Next we consider regularization. Let us set
\begin{equation} \label{regularization Rndmf-}
\phi_{s} = 1 + \frac{s}{x_{\tindex}}, \, h_s = \frac{ s / x_{\tindex} }{ 1 + s / x_{\tindex} }, \quad s \in [0,1].
\end{equation}
Thus $\phi_{s}$ is a regularization is every possible sense. For large $K > 0$, we can then compute
\begin{equation}  \label{Rndmf- cal2.2.2} 
H_p \phi_{s}^{-2K} = 4 K \hat{x}^{\tindex} ( \intn ) \rhoresf^2 h_s \ltau \phi_{s}^{-2K}. 
\end{equation}
Let $a_{s} = \phi_{s}^{K} a_{0}$,  $b_{j,s} = \phi_{s}^{-K} b_{0,s}$, $j = 1, 2$, $f_{s} = \phi_{s}^{-2K} f_{0}$,  $e_{k,s} = \phi_{s}^{-2K} e_{k,0}$, $k = 1,2$, as well as
\begin{align} \label{Rndmf- cal2.3} 
\begin{split}
b_{s}^{2} & = ( \hat{x}^{\tindex} )^{-2\vor} ( \intn )^{-\vob} \rhoresf^{-2\vor - 2 \vol} \phi_{s}^{-2K} \\
& \quad \times \big( ( 4 \vor - 4 \vol + 2 ) \hutaures + \rhoresf ( 4 \vol - 4K h_s ) \ltau - \delta_0 \phi_{s}^{-2K} \varphi_1^2 \varphi_2^2 \varphi_3^2 \varphi_4^2 \big) \varphi_1^2 \varphi_2^2 \varphi_3^2 \varphi_4^2
\end{split}
\end{align}
for some sufficiently small $\delta_0 > 0$. Then we in fact have
\begin{align*}
 H_{p} a_{s}^2 + \delta_0 ( \hat{x}^{\tindex} )^{2 \vor + 2} (\intn)^{ \vob + 2 } \rhoresf^{2 \vor + 2 \vol + 2} a_{s}^4  = & - b_{s}^2 + 2 ( \log \bff ) b_{1,s}^2 \\
 &  - b_{2,s}^2 - f_s + e_{1,s} + e_{2,s}.
\end{align*} \par

Here, we shall remark that $b_{s}^2$ is still non-negative, and its square root is a symbol, since by taking $\delta_0 > 0$ to be sufficiently small, the last term in the wide parenthesis from the second line of (\ref{Rndmf- cal2.3}) becomes negligible. For every fixed $K$, one could also take the support of $\psi$ to be small enough such that $\rhoresf$ is bounded by another arbitrarily small constant, thereby making the middle term there small as well. However, this construction will depend on $K$, which is the amount of a priori regularity assumption allowed. Nevertheless, this will not matter since it will be enough for our consideration here to be local.
 \par

Let $A_{s}$, $B_{s}$, $B_{2,s}$, $F_{s}$, $E_{k,s}$, $k=1,2$ be some quantizations of $a_{s}$, $b_{s}$, $b_{2,s}$, $f_{s}$ and $e_{k,s}$, $k=1,2$ respectively, such that each $F_{s}$ is arranged to be non-negative with respect to $L^2$. Let also $B_{1,s}  \in L^{\infty}( (0, 1]_{s} ; \Psf^{- \infty, \vor , - \infty, \vor + \vol, \vob, -\infty} )$ be chosen such that conditions (\ref{always these conditions for B1s}) are satisfied, and $\Lambda$ be some quantization of an elliptic extension of $(\hat{x}^{\tindex})^{\vor + 1} ( \intn )^{(\vob + 2)/2} \rhoresf^{\vor + \vol + 1}$ from the support of $a_{0}$. Then by Corollary \ref{commutator corollary}, we have
\begin{align} \label{Rndmf- cal2.2.9.1}
\begin{split}
i[ P,A_{s}^{\ast} A_{s} ] + \delta_{0} ( \Lambda A_{s}^{\ast} A_{s} )^{\ast} ( \Lambda A_{s}^{\ast} A_{s} ) = & - B_{s}^{\ast} B_{s} + 2 B_{1,s}^{\ast} (\log \bff) B_{1,s}  \\
&- B_{2,s}^{\ast} B_{2,s} - F_{s}  + E_{1,s} + E_{2,s} + R_{s}
\end{split}
\end{align}
for some $R_{s} \in L^{\infty}( (0, 1]_{s} ; \Psf^{-\infty, 2 \vor -1 + 2 \delta, - \infty, 2 \vor + 2\vol - 1 + 2 \delta, 2\vob, -\infty} )$. \par

The rest of the argument is standard. Let $ \tB, \tilde{E} \in \Psfo^{-\infty,0,-\infty,0,0,-\infty}$ be chosen such that 
\begin{equation*}  \label{Rndmf- cal2.3.1}  
\begin{gathered}
 \WFsL( \{  A_{s} \} ) \subset \Ellss( \tB ), \, \WFsL( \{ E_{k, s} \} ) \subset \Ellss( \tilde{E} ), \ k =1,2, \\
\WFs( I - \tB ) \cap \WFsL( \{ A_s \} ) = \WFs( I - \tilde{E} ) \cap \WFsL( \{ E_{s} \} )  = \emptyset.
\end{gathered}
\end{equation*}
Let $U$ be chosen such that $U \subset \Ellss(B_0)$. Then the standard procedure gives us
\begin{align} \label{local estimate at Rndmf-}
\begin{split}
\| B u \|_{ H_{\mathrm{d3sc,3co,res}}^{\ast, \vor, \ast, \vor + \vol, \vob, \ast} } & \leq  C ( \|  P u \|_{ H_{\mathrm{d3sc,3co,res}}^{\ast, \vor + 1, \ast, \vor + \vol  + 1, \vob + 2, \ast } } \\
& \quad + \| \tilde{E} u \|_{ H_{\mathrm{d3sc,3co,res}}^{\ast, \vor, \ast, \vor + \vol, \vob, \ast} } + \| \tB u \|_{H_{\mathrm{d3sc,3co,res}}^{\ast, \vor - 1 + 2\delta, \ast, \vor + \vol - 1 + 2 \delta, \vob, \ast}} + \| u \|_{ H_{\mathrm{d3sc,3co,res}}^{N, M, L, K, \vob, S } } ).
\end{split}
\end{align} 
In fact, if $\WFs( \tB )$ is taken small enough, then this argument can be iterated finitely many times, from which we can drop the $\tB u$ term in the above. \par

It remains to show that the dynamical condition required in the statement of the proposition is satisfied with $\tilde{E}$ in place of $E$. To this end, suppose that $\gamma$ is an integral curve of $\rho_{\dmf}^{-1} \rho_{\dtsccf}^{-1} \rho_{\tcocf}^{-2} H_{p}$ such that $\lim_{t \rightarrow \infty} \gamma(t) \in \mathcal{R}_{\mathrm{n,dmf},-}(\beta_{\tindex})$. Then it follows from Proposition \ref{proposition characterization of the integral curves which converge to dtsccf}, case~(1) that we cannot have $\lim_{t \rightarrow \infty} \gamma(t) \in \dtsccf$, $\gamma \not\subset \dtsccf$.  \par

Thus we must have $\gamma \subset \dtsccf$, in which case only cases~(1) and (4) of Proposition \ref{proposition characterization of the integral curves which live on dtsccf} may occur. In case~(1), we must have 
\begin{equation*}
\gamma \subset \psf_{\dmf} \Xd \cap \dtsccf \cap \tcocf.
\end{equation*}
Moreover, the trajectory of $\gamma$ is determined completely by (\ref{uniform interactive variable flow Hamiltonian scaled}), which is just the rescaled two-body flow in the interaction variables, with the free variables held constant at $\beta_{\tindex}$. In particular, $\gamma$ must enter $\WFs(E_{1,0})$ in finite backward time. \par

In case~(4), $\gamma$ travels from $\mathcal{R}_{\mathrm{n},\dff,+}$ to $\mathcal{R}_{\mathrm{n}, -}$ with $\hutaures = -1$, $\humures = 0$ for large time. In this case, it is obvious that $\gamma$ enter $\WFs(E_{2,0})$ in finite backward time. Since $\Ellss(\tilde{E})$ contains both $\WFs(E_{1,0})$ and $\WFs(E_{2,0})$, we have obtained the desired dynamical condition in a slightly local case. One can then bridge this gap to the case of a general $E$ by means of elliptic regularity and principal type propagation of regularity estimates. 
\end{proof}

\begin{proposition}[Above threshold estimate at $\mathcal{R}_{\mathrm{n,dmf},+}  \cap  \tcocf$] \label{Proposition Rndm+}
Let $\vor = \vor_{+}$, $\vol = \vol_{+}$ and $\vob = \vob_{+}$ be variable orders satisfying the conditions specified in \S \ref{variable order construction section}. Suppose that $\beta_{\tindex} \in \Sigma_{\mathrm{t}}$. Moreover, let $E \in \Psf^{-\infty, 0, -\infty , 0 ,0 ,-\infty}$ be chosen such that 
\begin{equation*}
\begin{gathered}
\text{all backward integral curves of $\rho_{\dmf}^{-1} \rho_{\dtsccf}^{-1} \rho_{\tcocf}^{-2} H_{p}$ starting from $\mathcal{R}_{\mathrm{n},\dmf, +}(\beta_{\tindex})$} \\
\text{are either stationary or enter $\Ellss(E)$ in finite time. }
\end{gathered}
\end{equation*}
Then we can find an arbitrarily small neighborhood $U$ of $\mathcal{R}_{\mathrm{n}, \dmf,+}(\beta_{\tindex})$ for which the following holds: If we choose $B \in \Psf^{-\infty, 0, -\infty , 0 ,0 ,-\infty}$ with $\WFs(B) \subset U$, then for every $u \in \mathcal{S}'$ and any choice of $N,M,L,K,S \in \mathbb{R}$, there exists $C > 0$ such that
\begin{equation} \label{semi-fredhom estimate at Rndmf+}
\| B u \|_{ H_{\mathrm{d3sc,3co,res}}^{\ast, \vor, \ast, \vor + \vol, \vob, \ast} }  \leq C ( \| Pu \|_{H_{\mathrm{d3sc,3co,res}}^{\ast, \vor + 1, \ast, \vor + \vol + 1, \vob + 2, \ast}}   + \| E u \|_{H_{\mathrm{d3sc,3co,res}}^{\ast, \vor, \ast, \vor + \vol, \vob, \ast}}  + \| u \|_{ H_{\mathrm{d3sc,3co,res}}^{ N, M, L, K, \vob, S } } )
\end{equation}
in the strong sense that if the right hand of (\ref{semi-fredhom estimate at Rndmf+}) is finite, then so is the left hand side, and the estimate holds.
\end{proposition}
\begin{proof}
As in the proof of Proposition \ref{Proposition Rndm-}, we will again change variables to (\ref{Rndmf- cal0.1}), with conventions (\ref{Rndmf- cal0.2}) being used, though this time we have
\begin{align*}
& \mathcal{R}_{\mathrm{n},\dmf,+}(\beta_{\tindex}) \\
& \qquad = \{ \hat{x}^{\tindex} =  ( \intn )^{1/2} = \rhoresf = 0, \hutaures = 1, \humures = 0, \intt = \lambda^2 \}.
\end{align*}
Now, consider the vector field
\begin{equation*}
 - 2 ( 1 - \rhoresf \ltau )  \partial_{\rhoresf} + \sHto \in \mathcal{V}( [0,\infty)_{\rhoresf} \times \mathcal{C}_{\tindex} \times \mathbb{R}^{n_{\tindex}} ),
\end{equation*}
which is outward pointing to the boundary of $[0 ,\infty)_{\rhoresf} \times \mathcal{C}_{\tindex} \times \mathbb{R}^{n_{\tindex}}$. Then we can find local coordinates $q_{\mathrm{n}} = ( q_{\mathrm{n},1}, q_{\mathrm{n}}' )$ centered at $\{ 0 \} \times \{ \beta_{\tindex} \}$ such that $q_{\mathrm{n},1}$ is a boundary defining function for $[0,\infty)_{\rhoresf} \times \mathcal{C}_{\tindex} \times \mathbb{R}^{n_{\tindex}}$, and that
\begin{equation*}
- 2 (1 - \rhoresf \ltau ) \partial_{\rhoresf} + \sHto = -\partial_{q_{\mathrm{n},1}}.
\end{equation*}
In particular, we now have
\begin{align} \label{Rndmf+ cal1}
\begin{split}
(\hat{x}^{\tindex})^{-1} ( \intn )^{-1} \rhoresf^{-1} H_{p} & = 2 ( 2 \hutaures -  \rhoresf \ltau ) \hat{x}^{\tindex} \partial_{\hat{x}^{\tindex}} - 2  ( \hutaures  - 1 ) \rhoresf\partial_{\rhoresf}  \\
&   \quad + 2 \hutaures R_{\humures} + H_{ | \humures |_{h^{\tindex}}^2 }   -  \rhoresf \partial_{q_{\mathrm{n},1}}
\end{split}
\end{align} 
in a small neighborhood of $\mathcal{R}_{\mathrm{n},\dmf,+}(\beta_{\tindex})$. \par

Let $\psi \in \mathcal{C}^{\infty}_{c}( [0, \infty) )$ be defined such that (\ref{tangential propagation psi condition}) are satisfied, though we will now also extend $\psi$ into a $\mathcal{C}^{\infty}( \mathbb{R} )$ function via a constant extension (i.e., letting it be identically $1$) to $(-\infty,0)$ (this will only be relevant in the definition of $\varphi_4$ below). Then we will set
\begin{equation*}
\begin{gathered}
\varphi_{1} = \psi( \hat{x}^{\tindex}  ), \, \varphi_{2} = \psi( \intn ), \, \varphi_3 = \psi ( q_{\mathrm{n},1} ), \\
 \varphi_{4} = \psi( \epsilon^{-1} | q_{\mathrm{n}}' |^2 - q_{\mathrm{n},1}  ), \, \varphi_{5} = \psi( ( \hutaures + 1 )^2 + | \humures |_{h^{\tindex}}^2 ),
\end{gathered}
\end{equation*}
where $\epsilon > 0$ is arbitrarily small, but will be henceforth fixed. Consider the symbol
\begin{equation*}
a_0 = ( \hat{x}^{\tindex} )^{- \vor - 1/1} ( \intn )^{ - \vob/2 - 1/2 }  \rhoresf^{- \vor -  \vol - 1/2} \varphi_1 \varphi_2 \varphi_3 \varphi_4 \varphi_5,
\end{equation*}
and notice that it can be made to be supported in an arbitrarily small neighborhood of $\mathcal{R}_{\mathrm{n},\dmf,+}(\beta_{\tindex})$ as in the below threshold case.
 \par

We proceed to introduce some important calculations as before. Firstly, formulae (\ref{Rndmf- cal1.1.1})--(\ref{Rndmf- cal1.1.3}) can continue to be made uses of. However, we now also have 
\begin{equation*}
\begin{gathered}
(\hat{x}^{\tindex})^{-1} ( \intn )^{-1} \rhoresf^{-1} H_{p} q_{\mathrm{n,1}} = - \rhoresf ( 1 + O( \hutaures - 1 ) ) , \\ 
(\hat{x}^{\tindex})^{-1} ( \intn )^{-1} \rhoresf^{-1} H_{p} ( |q_{\mathrm{n}}'|^2 - q_{\mathrm{n},1} )  = \rhoresf  ( 1 + O ( \hutaures - 1 ) ),
\end{gathered}
\end{equation*} 
and
\begin{equation*}
H_{p}( ( \hutaures - 1 )^2 + | \humures |_{h^{\tindex}}^2 ) =  4 \hat{x}^{\tindex} ( \intn ) \rhoresf | \humures |_{h^{\tindex}}^2.
\end{equation*} \par

Putting the above together, we can write
\begin{equation*}
H_{p} a_0^2 = - \tilde{b}_{0}^2 + 2 ( \log \bff ) b_{1,0}^2 - b_{2,0}^2 - b_{3,0}^2 - b_{4,0}^2 - f_0 + e_{0},
\end{equation*}
for some symbols $\tilde{b}_{0}$, $b_{1,0}$, $b_{2,0}$, $b_{3,0}$, $b_{4,0}$, $f_{0}$ and $e_0$, which we now proceed to define. \par

We first set
\begin{equation}  \label{Rndmf+ cal1} 
\tilde{b}_0^2 =  ( \hat{x}^{\tindex} )^{-2 \vor}  ( \intn )^{-\vob} \rhoresf^{-2 \vor - 2\vol} \big( ( 4 \vor - 4 \vol + 2 ) \hutaures - 4 \vol \rhoresf \ltau  \big) \varphi_1^2 \varphi_2^2 \varphi_3^2 \varphi_4^2 \varphi_5^2.
\end{equation}
Notice that we can arrange the support of $\psi$ to be small enough such that
\begin{equation*}
( 4 \vor - 4 \vol + 2 ) \hutaures - 4 \vol \rhoresf \ltau \simeq  4 \vor - 4 \vol + 2
\end{equation*}
on the support of $\tilde{b}_0^2$. Thus, by the threshold condition
\begin{equation}  \label{Rndmf+ cal1.1} 
\vor - \vol > - \frac{1}{2} \ \text{at} \ \mathcal{R}_{\mathrm{n,dmf},+},
\end{equation}
we know that $\tilde{b}_{0}^2$ can indeed be written as a square of symbol. Likewise, the terms involving logarithmic factors can be written as $2 ( \log \bff ) b_{1,0}^2 - f_0$, where $b_{1,0}^2 =  - (H_{p}\vob)  a_0^{2}$ and $f_0$ is defined exactly as in (\ref{Rndmf- cal2.2.0.0.1}). \par

Now, unlike the proof of Proposition \ref{Proposition Rndm-}, we instead need to set
\begin{equation*}  \label{Rndmf+ cal3} 
\begin{gathered}
b_{2,0}^2 = - 4 ( \hat{x}^{\tindex} )^{-2 \vor + 1} ( \intn )^{-\vob} \rhoresf^{-2\vor - 2\vol} ( \rhoresf \ltau -  2 \hutaures  ) \psi'( \hat{x}^{\tindex} ) \varphi_1 \varphi_2^2 \varphi_3^2 \varphi_4^2 \varphi_5^2, \\
b_{3,0}^2  = - 8 ( \hat{x}^{\tindex} )^{-2 \vor + 1} ( \intn )^{-\vob} \rhoresf^{-2\vor - 2\vol}  |\humures|_
{h^{\tindex}}^2   \varphi_1^2 \varphi_2^2 \varphi_3^2 \varphi_4^2 \varphi_5 \psi'( ( \hutaures + 1 )^2 + | \humures |_{h^{\tindex}}^2 ),
\end{gathered}
\end{equation*} 
and also 
\begin{equation*}
b_{4,0}^2 = - 2 ( \hat{x}^{\tindex} )^{-2\vor} ( \intn )^{-\vob} \rhoresf^{-2\vor -2\vol + 1} \varphi_1^2 \varphi_2^2 \varphi_3^2 \varphi_5^2 \varphi_4 \psi'( \epsilon^{-1} |q_{n}'|^2 - q_{\mathrm{n},1} ) ( 1 + O( \hutaures  - 1) ).
\end{equation*}
It is clear that $b_{3,0}^2$ can be written as a square of symbol. To show the same for $b_{2,0}^2$, we need only to observe that $\rhoresf \ltau - 2 \hutaures  \simeq -2$ on the support of $b_{3,0}^2$. Similarly, to show that $b_{4,0}^2$ can be written as a square of symbol, we can assume that $O( \hutaures - 1 )$ is arbitrarily small on the support of $b_{4,0}^2$. Finally, the term which requires a priori control is given by
\begin{equation*}
e_{0} = - 2 ( \hat{x}^{\tindex} )^{-2\vor} ( \intn )^{-\vob} \rhoresf^{-2\vor - 2\vol + 1} \varphi_1^2 \varphi_2^2 \varphi_4^2 \varphi_5^2 \varphi_3 \psi'( q_{\mathrm{n},1} ) ( 1 + O( \hutaures - 1 ) ).
\end{equation*} \par

Let $\phi_{s}$, $h_s$ for $s \in (0,1]$ again be defined as in (\ref{regularization Rndmf-}) and (\ref{Rndmf- cal2.2.2}). For large $K > 0$, we also set $a_{s} = \phi_{s}^{-K} a_{0}$, $b_{j,s} = \phi_{s}^{-K}$,  $j = 1,2,3,4$, $f_{s} = \phi_{s}^{-2K} f_0$, $e_{s} = \phi_{s}^{-2K} e_{0}$, as well as
\begin{align*}
\begin{split}
b_{s}^{2} & = ( \hat{x}^{\tindex} )^{-2\vor} ( \intn )^{-\vob} \rhoresf^{-2\vor - 2 \vol} \phi_{s}^{-2K} \\
& \quad \times \big( ( 4 \vor - 4 \vol + 2 ) \hutaures - \rhoresf ( 4 + 4K h_s ) \ltau - \delta_0 \phi_{s}^{-2K} \varphi_1^2 \varphi_2^2 \varphi_3^2 \varphi_4^2 \varphi_5^2 \big) \varphi_1^2 \varphi_2^2 \varphi_3^2 \varphi_4^2 \varphi_5^2
\end{split}
\end{align*}
for some sufficiently small $\delta_0 > 0$. Then we have
\begin{align*}
 H_{p} a_{s}^2 + \delta_0 ( \hat{x}^{\tindex} )^{2 \vor + 2} (\intn)^{ \vob + 2 } \rhoresf^{2 \vor + 2 \vol + 2} a_{s}^4  = & - b_{s}^2 + 2 ( \log \bff ) b_{1,s}^2 \\
 &  - b_{2,s}^2 - b_{3,s}^2 - b_{4,s}^2 - f_{s} + e_{s}.
\end{align*} 
Let $A_{s}$, $B_{s}$, $E_{s}$ be some quantizations of $a_{s}$, $b_{s}$ and $e_{s}$ respectively. Furthermore, let $\tB,\tilde{E} \in \Psf^{-\infty, 0, -\infty, 0, 0, -\infty}$ be chosen such that
\begin{equation*} 
\begin{gathered}
 \WFsL \{  A_{s} \} ) \subset \Ellss( \tB ), \, \WFsL( \{ E_{s} \}  ) \subset \Ellss( \tilde{E} ), \\
\WFs( I - \tB ) \cap \WFsL( \{ A_s \}  )  = \WFs( I - \tilde{E} ) \cap \WFsL( \{ E_{s} \} ) = \emptyset.
\end{gathered}
\end{equation*}
Let $U$ be chosen such that $U \subset \Ellss(B_0)$. Then the standard procedure gives us 
\begin{align} \label{semi-fredhom estimate at Rndmf+ need it}
\begin{split}
 \| B u \|_{ H_{\mathrm{d3sc,3co,res}}^{\ast, \vor, \ast, \vor + \vol, \vob, \ast} } & \leq C ( \| Pu \|_{H_{\mathrm{d3sc,3co,res}}^{\ast, \vor + 1, \ast, \vor + \vol + 1, \vob + 2, \ast}} \\
& \quad + \| \tilde{E} u \|_{H_{\mathrm{d3sc,3co,res}}^{\ast, \vor, \ast, \vor + \vol, \vob, \ast}} + \| \tB u \|_{H_{\mathrm{d3sc,3co,res}}^{\ast, \vor - 1/2 + \delta, \ast, \vor + \vol - 1/2 + \delta, \vob, \ast}} + \| u \|_{ H_{\mathrm{d3sc,3co,res}}^{ N, M, L, K, \vob, S } } ). 
\end{split}
\end{align} \par

Now, since the order at $\dff$ in (\ref{semi-fredhom estimate at Rndmf+ need it}) can be chosen arbitrarily, we may as well choose it to be (or something smaller than) $\vol-1/2 + \delta$. Then the threshold condition (\ref{Rndmf+ cal1.1}) again holds. In particular, the argument can be iterated for as many times as required, from which we conclude that the $\tB u$ term above can be dropped. \par

It remains to show that the dynamical condition required in the statement of the proposition is satisfied with $\tilde{E}$ in place of $E$. To this end, suppose that $\gamma$ is an integral curve of $\rho_{\dmf}^{-1} \rho_{\dtsccf}^{-1} \rho_{\tcocf}^{-2} H_{p}$ such that $\lim_{t \rightarrow \infty} \gamma(t) \in \mathcal{R}_{\mathrm{n},\dmf,+}$. Then it follows from Proposition \ref{proposition characterization of the integral curves which live on dtsccf} that we cannot have $\gamma \subset \dtsccf$. \par

Thus we must have $\lim_{t \rightarrow \infty} \gamma(t) \in \dtsccf \cap \tcocf$, $\gamma \not \subset \dtsccf$. By Proposition \ref{proposition characterization of the integral curves which converge to dtsccf}, we know that $\gamma \subset \psf_{\dmf} \Xd \cap \tcocf$ with $\hutaures = 1$, $\humures = 0$ being constants over $\gamma$. Thus it follows from (\ref{Rndmf+ cal1}) that the trajectory of $\gamma$ is locally determined by the flow of $\partial_{q_{\mathrm{n},1}}$. In particular, $\gamma$ must enter $\WFs(E_0)$ and subsequently $\Ellss( \tilde{E} )$ in finite backward time. The case of a general $E$ can then by achieved by means of elliptic regularity and principal type propagation of regularity estimates.
\end{proof}

\subsection{Radial point estimates at $\mathcal{R}_{\mathrm{n},\dff,\pm} \cap \tcocf $} \label{Radial point estimate at dff subsection}
We now move onto the discussion of radial point estimates at $\mathcal{R}_{\mathrm{n},\dff, \pm} \cap \tcocf$. As before, we can again do this individually for every choice of $\beta_{\tindex} \in \Sigma_{\mathrm{t}}$, i.e., assuming (\ref{local coordinates for free beta}) is true, we let
\begin{equation}
\label{eq; definition of R n dff at point}
\mathcal{R}_{\mathrm{n}, \dff, \pm}( \beta_{\tindex} ) = \mathcal{R}_{ \mathrm{n}, \dff, \pm } \cap \{ y_{\tindex} =  y_{\tindex,0}, \ltau = \tau_{\tindex,0}, \lmu = \mu_{\tindex,0} \}.
\end{equation}
Notice that $\mathcal{R}_{\mathrm{n}, \dff, \pm}( \beta_{\tindex} ) \subset \mathcal{R}_{\mathrm{n}, \dff, \pm} \cap \tcocf$ for all $\beta_{\tindex} \in \Sigma_{\mathrm{t}}$. \par

There exists an estimate for every such $\mathcal{R}_{\mathrm{n}, \dff, \pm}(\beta_{\tindex})$.

\begin{proposition}[Below threshold estimate at $\mathcal{R}_{\mathrm{n},\dff, -} \cap \tcocf$] \label{below threshold Rndff- proposition}
Let $\vor = \vor_{+}$, $\vol = \vol_{+}$ and $\vob = \vob_{+}$ be variable orders satisfying the conditions specified in \S \ref{variable order construction section}. Suppose that $\beta_{\tindex} \in \Sigma_{\mathrm{t}}$. Moreover, let $E \in \Psf^{-\infty, -\infty , 0 , 0 ,0 ,-\infty}$ be chosen such that 
\begin{equation*}
\begin{gathered}
\text{all backward integral curves of $\rho_{\dmf}^{-1} \rho_{\dtsccf}^{-1} \rho_{\tcocf}^{-2} H_{p}$ starting from $\mathcal{R}_{\mathrm{n},\dff, -}(\beta_{\tindex})$} \\
\text{are either stationary or enter $\Ellss(E)$ in finite time. }
\end{gathered}
\end{equation*}
Then we can find an arbitrarily small neighborhood $U$ of $\mathcal{R}_{\mathrm{n}, \dff,-}(\beta_{\tindex})$ for which the following holds: If we choose $B \in \Psf^{-\infty, -\infty , 0 , 0 ,0 ,-\infty}$ with $\WFs(B) \subset U$, then for every $u \in \mathcal{S}'$ and any choice of $N,M,L,K,S \in \mathbb{R}$, there exists $C > 0$ such that
\begin{equation} \label{semi-fredhom estimate at Rndff-}
\| B u \|_{ H_{\mathrm{d3sc,3co,res}}^{\ast, \ast, \vol , \vor + \vol, \vob, \ast} }  \leq C ( \| Pu \|_{H_{\mathrm{d3sc,3co,res}}^{\ast, \ast , \vol , \vor + \vol + 1, \vob + 2, \ast}} + \| E u \|_{H_{\mathrm{d3sc,3co,res}}^{\ast, \ast, \vol , \vor + \vol, \vob, \ast}} + \| u \|_{ H_{\mathrm{d3sc,3co,res}}^{ N, M, \vol, K, \vob, S } } )
\end{equation}
in the strong sense that if the right hand of (\ref{semi-fredhom estimate at Rndff-}) is finite, then so is the left hand side, and the estimate holds.
\end{proposition}
\begin{proof}
Let us change variables from $ ( \hbff, x^{\tindex}, y_{\tindex}, y^{\tindex}, \ltau, \lmu, \utaub, \umub )$ to
\begin{equation} \label{Rndff- cal1}
\hbff, (\intn)^{1/2}, \, y_{\tindex}, \, y^{\tindex}, \, \ltau, \, \lmu, \, \rhobf, \, \humub,
\end{equation}
where we are writing $\rhobf = ( \intnb )^{-1/2}$ and
\begin{equation} \label{Rndff- cal2}
\hutaub = \frac{ \utaub }{ ( \intnb )^{1/2} }, \ \humub = \frac{ \umub }{ ( \intnb )^{1/2} }.
\end{equation}
We remind the readers here that in these coordinates, $\hbff$ defines $\psf_{\dff}\Xd$, $(\intn)^{1/2}$ defines $\tcocf$, $\rhobf$ defines $\dtsccf$, and 
\begin{align*}
&  \mathcal{R}_{\mathrm{n},\dff,-}(\beta_{\tindex}) \\
& \qquad = \{ \hbff =  ( \intn )^{1/2} = \rhobf = 0, \hutaub = -1,  \humub = 0, \intt = \lambda^2 \}.
\end{align*} 
Moreover, recall from (\ref{second microlocal dynamic modification cal 5}) and (\ref{second microlocal dynamic modification cal 5.1}) that we have
\begin{align*}
( \intn )^{-1} \rhobf^{-1} H_{p} & = 2 ( \rhobf \hbff \ltau - 2 \hutaub ) \hbff \partial_{\hbff} + 2 \hutaub \rhobf \partial_{\rhobf} \\
&  \quad + 2 \hutaub R_{\humub} + H_{ |\humub|_{h^{\tindex}}^2 }   + \rhobf \hbff \sHto,
\end{align*}
where $R_{\humub} = \humub \cdot \partial_{\humub}$ and $H_{ | \humub |_{h^{\tindex}}^2 } = \partial_{ \humub } | \humub |_{h^{\tindex}}^2 \cdot \partial_{y^{\tindex}} - \partial_{y^{\tindex}} | \humub |_{h^{\tindex}}^2 \cdot \partial_{ \humub } $. \par

Let $\psi \in \mathcal{C}^{\infty}_{c} ( [0,\infty) )$ be defined such that (\ref{tangential propagation psi condition}) are satisfied. Recall that we are assuming $\beta_{\tindex} = ( y_{\tindex,0}, \tau_{\tindex,0}, \mu_{\tindex,0} )$ in local coordinates. Then we will define
\begin{equation*}
\begin{gathered}
\varphi_{1} = \psi( \hbff ), \, \varphi_{2} = \psi( \intn ),  \\
\varphi_3 = \psi( | y_{\tindex} - y_{\tindex,0} |^2 + | \ltau - \tau_{\tindex,0} |^2 + | \lmu - \mu_{\tindex,0} |_{h_{\tindex}}^2  ), \,  \varphi_4 = \psi(  (\hutaub + 1 )^2 + | \humub |_{h^{\tindex}}^2 ),
\end{gathered}
\end{equation*}
as well as
\begin{equation*}
a_0 = \hbff^{- \vol} ( \intn )^{ - \vob/2 - 1/2 } \rhobf^{ - \vor - \vol - 1/2} \varphi_1 \varphi_2 \varphi_3 \varphi_4.
\end{equation*} \par

Let us simplify the calculation by implicitly assuming that $a_0$ is supported near $\dtsccf$. Note that this is acceptable since $\dtsccf$ is the only `symbolic face' which $\mathcal{R}_{\mathrm{n},\dff,-}$ intersects. Our assumption therefore amounts to the omission of a cut-off $\psi(\rhobf)$, which become symbolically residual on the support of $a_0$ under the derivative of $H_p$. In particular, we can assume that $a_0$ is supported in an arbitrarily small neighborhood of $\mathcal{R}_{\mathrm{n},\dff,-}(\beta_{\tindex})$. Note that we implicitly pay the price for this by accounting for a term which decays to lower orders at $\dtsccf$ in the symbol calculus procedure below.
\par

Next, we compute that
\begin{equation*}
H_{p} \hbff^{-2\vol} = - 2 ( H_{p} \vol ) ( \log \hbff ) \hbff^{-2 \vol }  - 4 \vol  ( \rhobf \hbff \ltau - 2 \hutaub ) ( \intn ) \rhobf \hbff^{-2\vol},
\end{equation*}
and again, as in (\ref{Rndmf- cal1.1.2}), we have
\begin{equation*} 
H_{p}( \intn )^{ -\vob - 1 } =  - 2 ( H_{p} \vob ) ( \log ( \intn )^{1/2} ) ( \intn )^{ -\vob - 1 }
\end{equation*}
as well as
\begin{align*}
H_{p} \rhobf^{ -2\vor - 2\vol - 1 } & = -2 ( H_{p} ( \vor + \vol )  ) ( \log \rhobf ) \rhobf^{ -2 \vor - 2\vol - 1 }  - ( 4 \vor + 4 \vol + 2 ) \hutaub ( \intn ) \rhobf^{ -2\vor - 2\vol }.
\end{align*}
Moreover, we also have
\begin{equation*}
H_{p}  ( ( \hutaub + 1 )^2 + | \humub |_{h^{\tindex}}^2 ) = - 4 ( \intn ) \rhobf | \humub |_{h^{\tindex}}^2.
\end{equation*} \par

Now, we can write
\begin{equation*}
H_{p} a_0^2 = - \tilde{b}_0^2  +  2( \log x^{\tindex}) b_{1,0}^2 - b_{2,0}^2 - f_0+ e_0 + \tilde{r}_0
\end{equation*}
for some symbols $\tilde{b}_{0}$, $b_{1,0}$, $b_{2,0}$, $b_{3,0}$, $e_0$ and $\tilde{r}_0$, which we now proceed to define. \par

We first set
\begin{equation}  \label{Rndff- cal4}
\tilde{b}_0^2 =  \hbff^{-2\vol} ( \intn )^{-\vob} \rhobf^{ -2 \vor - 2\vol } \big( ( 4 \vor - 4 \vol + 2 ) \hutaub + 4 \vol \rhobf \hbff \ltau \big) \varphi_1^2 \varphi_2^2 \varphi_3^2 \varphi_4^2.
\end{equation}
Notice that $\tilde{b}_0^2$ can be written as a square of symbol since we can arrange the support of $\psi$ to be small enough such that 
\begin{equation*}
( 4 \vor - 4 \vol + 2 ) \hutaub + 4 \vol \rhobf \hbff \ltau \simeq - 4 \vor + 4 \vol - 2
\end{equation*}
on the support of $\tilde{b}_{0}^2$, which is positive by the threshold condition
\begin{equation} \label{threshold condition Rndff-}
\vor - \vol < - \frac{1}{2} \ \text{at} \ \mathcal{R}_{\mathrm{n}, \dff ,-}.
\end{equation} \par

On the other hand, the sum of terms containing logarithmic factors is
\begin{equation} \label{Rndff- cal5}
- 2 \big( (H_{p} \vol) ( \log \hbff ) + ( H_{p} \vob ) ( \log (\intt)^{1/2} ) + ( H_{p}( \vor + \vol ) ) ( \log \rhobf) \big) a^2_0.
\end{equation}
We can then write (\ref{Rndff- cal5}) as $2( \log x^{\tindex} ) b_{1,0}^2 - f_0$, where
\begin{equation}  \label{Rndff- cal5.1}
b_{1,0}^2 = - ( H_{p} \vob ) a_0^2, \, f_0 = 2 \big( (H_{p} \vol) ( \log \hbff ) + (H_{p}(\vor - \vol) ) ( \log \rhobf )   \big) a_0^2.
\end{equation}
Here we have again used that $H_{p} \vob = 2 H_{p} \vol$. Note also that $b_{1,0}^2$, $f_0$ can be written as a square of symbol resp. a sum of squares of symbols by construction of the variable orders. \par

Additionally, differentiating $\varphi_1$ produces the term
\begin{equation*}
b_{2,0}^2 = - 4 ( \rhobf \hbff \ltau - 2 \hutaub ) \hbff^{-2\vol + 1} ( \intn )^{-\vob}  \rhobf^{ -2\vor - 2 \vol } \varphi_2^2 \varphi_3^2 \varphi_4^2 \varphi_1  \psi'( \hbff ),
\end{equation*}
and this can be written as a square of symbol since $\rhobf \hbff \ltau - 2 \hutaub \simeq 2$ on the support of $b_{2,0}^2$. Meanwhile, we will require a priori control on the support of 
\begin{equation*}
e_0 = - 8 \hbff^{-2\vol} ( \intn )^{-\vob} \rhobf^{-2\vor - 2\vol} \varphi_1^2 \varphi_2^2 \varphi_3^2 \varphi_4
 \psi'( ( \hutaub + 1 )^2 + | \humub |_{h^{\tindex}}^2 ) | \humub |_{h^{\tindex}}^2.
 \end{equation*}
which comes from differentiating the $\varphi_4$ term. Lastly, differentiating the $\varphi_3$ term, we get
\begin{align} \label{Rndff- cal6}
\begin{split}
\tilde{r}_{0} & = \hbff^{-2\vol} ( \intn )^{-\vob - 1} \rhobf^{-2\vor - 2\vol - 1} \varphi_1^2 \varphi_2^2 \varphi_4^2 H_{p} \varphi_3^2 \\
& = 2 \hbff^{-2\vol + 1} ( \intn )^{-\vob } \rhobf^{-2\vor - 2\vol + 1} \varphi_1^2 \varphi_2^2 \varphi_4^2 \varphi_3 \sHto \varphi_3,
\end{split}
\end{align}
which, as we will see below, will be irrelevant in a symbolic sense.
\par

Let $\phi_{s}$, $h_{s}$ for $s \in (0,1]$ again be defined as in (\ref{regularization Rndmf-}) and (\ref{Rndmf- cal2.2.2}), though we now have
\begin{equation*}
H_p \phi_{s}^{-2K} = 4 K \hbff ( \intn ) \rhobf^2 h_{s} \ltau \phi_{s}^{-2K}.
\end{equation*}
Then for large $K > 0$, we will set 
\begin{equation*}
a_{s} = \phi_{s}^{-K} a_{0}, \, b_{j,s} = \phi_{s}^{-K} b_{j,0}, \, j =1,2, \, f_{s} = \phi_{s}^{-2K} f_0, \, e_{s} = \phi_{s}^{-2K} e_{0}, \, \tilde{r}_{s} = \phi_{s}^{-2K} \tilde{r}_0,
\end{equation*}
as well as
\begin{align*}
b_{s}^2 & = \hbff^{-2\vol} ( \intn )^{-\vob} \rhobf^{-2\vor - 2\vol} \\
& \quad \times  \big( ( 4 \vor - 4 \vol + 2 ) \hutaub + ( 4 \vol - 4K h_{s} ) \rhobf \hbff \ltau - \delta_0 \phi_{s}^{-2K} \varphi_1^2 \varphi_2^2 \varphi_3^2 \varphi_4^2  \big) \varphi_1^2 \varphi_2^2 \varphi_3^2 \varphi_4^2
\end{align*}
for some sufficiently small $\delta_0 > 0$. Now we have
\begin{equation*}
H_{p}a_s^2 + \delta_0 \hbff^{2 \vol} ( \intn )^{\vob + 2} \rhobf^{2\vor + 2\vol +2} a_{s}^4 = - b_{s}^2 + 2 ( \log x^{\tindex} ) b_{1,s}^2 - b_{2,s}^2 - f_{s} + e_{s} + \tilde{r}_{s}. 
\end{equation*} 
Notice again that $b_{s}^{2}$ can be written as a square of symbol for fixed $K$ by positivity. The support of $\psi$ will therefore depend on the amount of a priori regularization. However, this will only be a local issue, and in particular will be harmless.
\par

Let $A_{s}$, $B_{s}$, $B_{2,s}$, $F_{s}$, $E_{s}$, $\tilde{R}_{s}$ be some quantizations of $a_{s}$, $b_{s}$, $b_{2,s}$, $f_{s}$, $e_{s}$ and $\tilde{r}_{s}$ respectively such that each $F_{s}$ is non-negative with respect to $L^{2}$. Notice that
\begin{equation*}
\tilde{R}_{s} \in L^{\infty}( (0,1]_{s} ;  \Psf^{-\infty, -\infty, 2 \vol -1 , 2 \vor + 2 \vol - 1, 2 \vob, -\infty} ).
\end{equation*}
Thus $\tilde{R}_{s}$ belongs to a class of operators which vanishes in the symbol calculus below, which can be absorbed into the family $R_{s}$ below. \par

Furthermore, let $B_{1,s} \in L^{\infty}( (0,1]_s ; \Psf^{-\infty, -\infty, \vol - 1/2, \vor + \vol - 1/2, \vob, -\infty} )$ be chosen such that conditions (\ref{always these conditions for B1s}) are satisfied, and let $\Lambda $ be some quantization of an elliptic extension of $\hbff^{2\vol} ( \intn )^{\vob + 2} \rhobf^{ 2 \vor + 2 \vol + 2 }a_{s}^4$ from the support of $a_{0}$. Then by Corollary \ref{commutator corollary}, we have
\begin{equation*}
i [ P, A_{s} ] + \delta_0 ( \Lambda A_{s}^{\ast} A_{s} )^{\ast} ( \Lambda A_{s}^{\ast} A_{s} ) = - B^{\ast}_{s} B_{s} + 2 B_{1,s}^{\ast} ( \log x^{\tindex} )B_{1,s} - B_{2,s}^{\ast} B_{2,s} - F_s + E_{s} + R_{s}
\end{equation*}
for some $R_{s}  \in L^{\infty} (  (0,1]_{s} ;  \Psf^{-\infty, -\infty, 2 \vol, 2 \vor + 2 \vol - 1 + 2 \delta, 2 \vob, - \infty} )$.  \par

The rest of the argument is once again standard. Let $\tB, \tilde{E} \in \Psfo^{-\infty, -\infty , 0 ,0,0,-\infty}$ be chosen such that 
\begin{equation}  \label{Rndff- cal7}
\begin{gathered}
 \WFsL( \{  A_{s} \} ) \subset \Ellss( \tB ), \, \WFsL( \{ E_{s} \} ) \subset \Ellss( \tilde{E} ), \\
\WFs( I - \tB ) \cap \WFsL( \{ A_s \} ) = \WFs( I - \tilde{E} ) \cap \WFsL( \{ E_s \}  )  = \emptyset,
\end{gathered}
\end{equation}
Let $U$ be chosen such that $U \subset \Ellss(B_0)$. Then the standard procedure gives us 
\begin{align} \label{Rndff- cal7.1}
\begin{split}
\| B u \|_{ H_{\mathrm{d3sc,3co,res}}^{\ast, \ast, \vol , \vor + \vol, \vob, \ast} } & \leq C ( \| Pu \|_{H_{\mathrm{d3sc,3co,res}}^{\ast, \ast , \vol , \vor + \vol + 1, \vob + 2, \ast}} \\
& \quad + \| \tilde{E} u \|_{H_{\mathrm{d3sc,3co,res}}^{\ast, \ast, \vol , \vor + \vol, \vob, \ast}} + \| \tB u \|_{H_{\mathrm{d3sc,3co,res}}^{\ast,\ast,\vol, \vor + \vol - 1/2 + \delta, \ast}} + \| u \|_{ H_{\mathrm{d3sc,3co,res}}^{ N, M, \vol, K, \vob, S } } ).
\end{split}
\end{align} 
If $\WFs( \tB )$ is taken small enough, then this argument can be iterated finitely many times, from which we can drop the $\tB u$ term in the above. \par

It remains to show that the above estimate is compatible with the required dynamical condition as in the statement of the proposition. To this end, let $\gamma$ be an integral curve of $\rho_{\dmf}^{-1} \rho_{\dtsccf}^{-1} \rho_{\tcocf}^{-2} H_{p}$ such that $\lim_{t \rightarrow \infty} \gamma(t) \in \mathcal{R}_{\mathrm{n}, \dff, -}$. Then by Proposition \ref{proposition characterization of the integral curves which live on dtsccf}, case~(2), we can conclude that 
\begin{equation*}
\gamma \subset \psf_{\dff} \Xd \cap \dtsccf \cap \tcocf.
\end{equation*}
Moreover, the trajectory of $\gamma$ is determined completely by (\ref{uniform interactive variable flow Hamiltonian scaled}), which is just the rescaled two-body flow in the interaction variables, with the free variables held constant at $\beta_{\tindex}$. In particular, $\gamma$ must enter $\WFs(E_0)$, and subsequently $\Ellss(\tilde{E})$, in finite backward time. The case of a general $E$ can then be achieved by means of elliptic regularity and principal type propagation of regularity estimates.  
\end{proof}

\begin{proposition}[Above threshold estimate at $\mathcal{R}_{\mathrm{n},\dff, +} \cap \tcocf$] \label{Propositive estimate at Rndff+}
Let $\vor = \vor_{+}$, $\vol = \vol_{+}$ and $\vob = \vob_{+}$ be variable orders satisfying the conditions specified in \S \ref{variable order construction section}. Suppose that $\beta_{\tindex} \in \Sigma_{\mathrm{t}}$. Moreover, let $G, E \in \Psf^{-\infty, -\infty , 0 , 0 ,0 ,-\infty}$ be chosen such that
\begin{equation*}
\mathcal{R}_{\mathrm{n},\dff,+}(\beta_{\tindex}) \subset \Ellss(G), 
\end{equation*}
with $\WFs(G)$ contained in a given small neighborhood of $\mathcal{R}_{\mathrm{n},\dff,+}(\beta_{\tindex})$, and
\begin{equation*}
\begin{gathered}
\text{all backward integral curves of $\rho_{\dmf}^{-1} \rho_{\dtsccf}^{-1} \rho_{\tcocf}^{-2} H_{p}$ starting from $\mathcal{R}_{\mathrm{n},\dff, +}(\beta_{\tindex})$} \\
\text{are either stationary or enter $\Ellss(E)$ in finite time.}
\end{gathered}
\end{equation*}
Then we can find an arbitrarily small neighborhood $U$ of $\mathcal{R}_{\mathrm{n}, \dff, +}(\beta_{\tindex})$ for which the following holds: If we choose $B \in \Psf^{-\infty, -\infty , 0 , 0 ,0 ,-\infty}$ with $\WFs(B) \subset U$, then for every $u \in \mathcal{S}'$ and any choice of $N,M,L,K,S \in \mathbb{R}$, there exists $C > 0$ such that
\begin{align}  \label{semi-fredhom estimate at Rndff+}
\begin{split}
 \| B u \|_{ H_{\mathrm{d3sc,3co,res}}^{\ast, \ast, \vol, \vor + \vol, \vob, \ast} } & \leq  C ( \| P u \|_{ H_{\mathrm{d3sc,3co,res}}^{\ast, \ast, \vol, \vor + \vol  + 1, \vob + 2, \ast } } \\
& \quad + \| E u \|_{ H_{\mathrm{d3sc,3co,res}}^{\ast, \ast, \vol, \vor + \vol, \vob, \ast} } + \| G u \|_{ H_{\mathrm{d3sc,3co,res}}^{\ast, \ast, \vol, \vor + \vol - 1/2 + \delta, \vob, \ast } } + \| u \|_{ H_{\mathrm{d3sc,3co,res}}^{N, M, \vol, K, \vob, S } } ).
\end{split}
\end{align}
in the strong sense that if the right hand of (\ref{semi-fredhom estimate at Rndff+}) is finite, then so is the left hand side, and the estimate holds.
\end{proposition}
\begin{proof}
Consider again coordinates (\ref{Rndff- cal1}), with conventions (\ref{Rndff- cal2}) being used, and assume that $\beta = ( y_{\tindex,0}, \tau_{\tindex,0}, \mu_{\tindex,0} )$. Let $\psi \in \mathcal{C}^{\infty}_{c}( [0,\infty) )$ be such that (\ref{tangential propagation psi condition}) are satisfied. Then we set
\begin{equation*}
\begin{gathered}
\varphi_{1} = \psi( \hbff ), \,  \varphi_{2} = \psi( \intn ),  \\
\varphi_3 = \psi( | y_{\tindex} - y_{\tindex,0} |^2 + | \ltau - \tau_{\tindex,0} |^2 + | \lmu - \mu_{\tindex,0} |_{h_{\tindex}}^2  ), \, \varphi_4 = \psi(  (\hutaub - 1 )^2 + | \humub |_{h^{\tindex}}^2 )
\end{gathered}
\end{equation*}
as well as
\begin{equation*}
a_{0} = \hbff^{-\vol} ( \intn )^{-\vob/2 - 1/2} \rhobf^{-\vor - \vol - 1/2} \varphi_1 \varphi_2 \varphi_3 \varphi_4,
\end{equation*} 
where we implicitly assume that $a_0$ is supported in a neighborhood of $\dtsccf$.
\par

We will compute
\begin{equation} \label{Rndff+ cal1}
H_{p} a_0^2 = - \tilde{b}_{0}^2 + 2 ( \log x^{\tindex} ) b_{1,0}^2 - b_{2,0}^2 - f_0 + e_{0} + \tilde{r}_0,
\end{equation}
where $\tilde{b}_0$, $\tilde{r}_0$ are symbols defined in exactly the same ways as in (\ref{Rndff- cal4}) and (\ref{Rndff- cal6}) respectively, except in showing the positivity of $\tilde{b}_0^2$, we now use
\begin{equation*}
( 4 \vor - 4 \vol + 2 ) \hutaub + 4 \vol \rhobf \hbff \ltau \simeq  4 \vor - 4 \vol + 2
\end{equation*}
on the support of $\tilde{b}_{0}^2$, which is again positive by the threshold condition
\begin{equation} \label{threshold condition Rndff+1}
\vor - \vol > - \frac{1}{2} \ \text{at} \ \mathcal{R}_{\mathrm{n}, \dff ,+}.
\end{equation}
Moreover, we still let $b_{1,0}^2$, $f_0$ be defined as in (\ref{Rndff- cal5.1}). However, one has to now switch the roles for the rest of the terms in (\ref{Rndff+ cal1}). Namely, since
\begin{equation*}
H_{p}  ( ( \hutaub - 1 )^2 + | \humub |_{h^{\tindex}}^2 ) = 4 ( \intn ) \rhobf | \humub |_{h^{\tindex}}^2,
\end{equation*}
we shall instead set
\begin{equation*}
b_{2,0}^2 =  - 8 \hbff^{-2\vol} ( \intn )^{-\vob} \rhobf^{-2\vor - 2\vol} \varphi_1^2 \varphi_2^2 \varphi_3^2 \varphi_4
 \psi'( ( \hutaub - 1 )^2 + | \humub |_{h^{\tindex}}^2 ) | \humub |_{h^{\tindex}}^2.
 \end{equation*}
On the other hand, the term which require a priori control for becomes
\begin{equation*}
e_0 = 4 ( \rhobf \hbff \ltau - 2 \hutaub ) \hbff^{-2\vol + 1} ( \intn )^{-\vob}  \rhobf^{ -2\vor - 2 \vol } \varphi_2^2 \varphi_3^2 \varphi_4^2 \varphi_1  \psi'( \hbff ).
\end{equation*} \par

From here, we can again go through the standard procedure, involving regularization, which is followed by quantizations, and where Corollary \ref{commutator corollary} must also be applied. This will produce operators $B_{s}$, $A_{s}$ and $E_{s}$ for $s \in (0,1]$, which are bounded families of operators. \par

We can then choose $U$ such that $U \subset \Ellss(B_0).$ Moreover, let $\ttB, \tilde{E} \in \Psfo^{-\infty, -\infty, 0, 0, 0 -\infty}$ be chosen such that conditions (\ref{Rndff- cal7}) are satisfied with $\ttB$ in places of $\tB$. Then we can show that an estimate of the form (\ref{Rndff- cal7.1}) holds, where we use elliptic regularity to obtain the $Gu$ term on the right hand side. However, notice that we cannot in general iterate this estimate freely, since the above threshold requirement (\ref{threshold condition Rndff+1}) might no longer be satisfied. \par

It remains to show that the above estimate is compatible with the required dynamical condition as in the statement of the proposition. To this end, let $\gamma$ be an integral curve of $\rho_{\dmf}^{-1} \rho_{\dtsccf}^{-1} \rho_{\tcocf}^{-2} H_{p}$ such that $\lim_{t \rightarrow \infty} \gamma(t) \in \mathcal{R}_{\mathrm{n},\dff, +}(\beta_{\tindex})$. Then by using Proposition \ref{proposition characterization of the integral curves which live on dtsccf}, case~(3), we know that the only possibly is for $\gamma$ to travel from $\mathcal{R}_{\mathrm{n}, \dmf, +}$ to $\mathcal{R}_{\mathrm{n},\dff, +}$, with $\hutaub = 1$, $\humub = 1$ for large time. In particular, $\gamma$ must enter $\WFs(E_0)$ and subsequently $\Ellss(\tilde{E})$ in finite backward time. The case of a general $E$ now follows from elliptic regularity and principal type propagation of regularity estimates.
\end{proof}

\subsection{Radial point estimates at $\mathcal{R}_{0,\pm}$} 
We next consider radial point estimates at $\mathcal{R}_{0,\pm}$.
\begin{proposition}[Below threshold estimate at $\mathcal{R}_{0,+}$] \label{radial point estimate at R0p}
Let $\vor = \vor_{+}$, $\vob = \vob_{+}$ be variable orders satisfying the conditions specified in \S \ref{variable order construction section}. Let $E \in \Psfo^{-\infty, 0 , -\infty, -\infty, 0, -\infty}$ be chosen such that
\begin{equation*}
\begin{gathered}
\text{all backward integral curves of $\rho_{\dmf}^{-1} \rho_{\dtsccf}^{-1} \rho_{\tcocf}^{-2} H_{p}$ starting from $\mathcal{R}_{0,+}$} \\
\text{are either stationary or enter $\Ellss(E)$ in finite time.}
\end{gathered}
\end{equation*}
Then we can find some arbitrarily small neighborhood $U$ of $\mathcal{R}_{0,+}$ for which the following holds: If we choose $B \in \Psf^{-\infty, -\infty , 0 , 0 ,0 ,-\infty}$ with $\WFs(B) \subset U$, then for every $u \in \mathcal{S}'$ and any choice of $N,M,L,K,S \in \mathbb{R}$, there exists $C > 0$ such that
\begin{equation} \label{semi-fredhom estimate at R0p}
\| B u \|_{ H_{\mathrm{d3sc,3co,res}}^{\ast, \vor, \ast, \ast, \vob, \ast} } \leq C ( \| Pu \|_{H_{\mathrm{d3sc,3co,res}}^{\ast, \vor + 1, \ast, \ast, \vob + 2, \ast}} + \| E u \|_{H_{\mathrm{d3sc,3co,res}}^{\ast, \vor, \ast, \ast, \vob, \ast}} + \| u \|_{ H_{\mathrm{d3sc,3co,res}}^{ N,M,L,K, \vob, S } } )
\end{equation}
in the strong sense that if the right hand of (\ref{semi-fredhom estimate at R0p}) is finite, then so is the left hand side, and the estimate holds.
\end{proposition}
\begin{proof}
We will work in coordinates $( \hat{x}^{\tindex}, \bff, y_{\tindex}, y^{\tindex}, \ltau, \lmu, \utaures, \umures )$, where we recall that $\hat{x}^{\tindex}$ defines $\psf_{\dmf}\Xd$ and $\bff$ defines $\tcocf$. Let $\psi \in \mathcal{C}^{\infty}_{c}( [0 , \infty) )$ be such that conditions (\ref{tangential propagation psi condition}) are satisfied. We now set
\begin{equation*}
\varphi_{1} = \psi( \bff ), \ \varphi_{2} = \psi ( |\utaures|^2 + | \umures |_{h^{\tindex}}^2 ), \ \varphi_3 = \psi( ( \ltau - \lambda )^2 + | \lmu |_{h_{\tindex}}^2 ).
\end{equation*}
and consider the symbol defined by
\begin{equation*}
a_{0} = ( \hat{x}^{\tindex} )^{- \vor - 1/2} \bff^{- \vob - 1} \varphi_1 \varphi_2 \varphi_3.
\end{equation*} 
Much as in the proofs of Propositions \ref{below threshold Rndff- proposition} and \ref{Propositive estimate at Rndff+},  we will assume implicitly that $a_0$ is supported in a small neighborhood of $\psf_{\dmf} \Xd$, which amounts to the omission of multiplying by $a_0$ another cut-off $\psi(\hat{x}^{\tindex})$. However, this omission will be harmless in terms of the symbol calculus, since the only `symbolic face' which $\mathcal{R}_{0,+}$ intersects is $\psf_{\dmf}\Xd$. In particular, we can view $a_0$ as being supported in a small neighborhood of $\mathcal{R}_{0,+}$. \par
As usual, we proceed to make the computation
\begin{equation*}
H_{p}a_0^2 = - \tilde{b}_{0}^2 + 2 (\log \bff ) b_{1,0}^2 - b_{2,0}^2 - b_{3,0}^2 - b_{4,0}^2 + e_{0},
\end{equation*}
where
\begin{equation} \label{R0p plus radial point cal 1}
\tilde{b}_{0}^2  = ( \hat{x}^{\tindex} )^{- 2 \vor } \bff^{-2 \vob} \ ( ( 8 \vor - 4 \vob ) \utaures - ( 4\vor - 4 \vob - 2 ) \ltau  )  \varphi_{1}^2 \varphi_2^2 \varphi_3^2
\end{equation}
as well as
\begin{equation*}
\begin{gathered}
b_{3,0}^2 = -  8 \lambda (\hat{x}^{\tindex})^{-2\vor} \bff^{-2\vob} |\lmu|_{h_{\tindex}}^2 \psi'( (\ltau - \lambda)^2 + |\lmu|_{h_{\tindex}}^2 ) \varphi_3 \varphi_1^2 \varphi_2^2, \\
b_{4,0}^2 = - 4 ( \hat{x}^{\tindex} )^{-2 \vor} \bff^{-2\vob + 1 } ( \ltau - \utaures ) \psi'(\bff) \varphi_1 \varphi_{2}^2 \varphi_3^2, \\
e_{0} =  - 8 (\hat{x}^{\tindex})^{-2 \vor} \bff^{-2 \vob} ( \ltau - \utaures ) ( \intnres ) \psi'( \intnres ) \varphi_2 \varphi_1^2 \varphi_3^2.
\end{gathered}
\end{equation*}
Moreover, the sum of terms containing logarithmic factors can be written as
\begin{equation} \label{R0p plus radial point cal 1.5}
-2 \big( ( H_{p} \vor ) ( \log \hat{x}^{\tindex} ) + ( H_{p} \vob ) ( \log \bff ) \big) a_0^2 = 2 ( \log \bff ) b_{1,0}^2 - b_{2,0}^2
\end{equation}
where $b_{1,0}^2 = - ( H_{p} \vob ) a_0^2$ and $b_{2,0}^2 = - 2 ( H_{p} \vor )( \log \hat{x}^{\tindex} ) a_0^2$ can be written as squares of symbols by construction of the variable orders. Likewise, $\tilde{b}_{0}^2$ can be written as a square of symbol since we can arrange the support of $\psi$ to be small enough such that $( 8 \vor - 4 \vob ) \utaures \simeq 0$, $\ltau \simeq \lambda$ on the support of $\tilde{b}_0^2$. Then we have
\begin{equation*}
( 8 \vor - 4 \vob ) \utaures - ( 4 \vor - 4 \vob - 2 ) \ltau \simeq - ( 4 \vor - 4 \vob - 2 ) \lambda,
\end{equation*}
which is positive by the threshold condition
\begin{equation*} \label{threshold condition Rndff+}
\vor - \vob <  \frac{1}{2} \ \text{at} \ \mathcal{R}_{0,+}.
\end{equation*}
Lastly, $b_{3,0}^2$ and $b_{4,0}^2$ can be written as squares of symbols by conditions (\ref{tangential propagation psi condition}), and because $\ltau - \utaures \simeq  \lambda$ on their respective supports. \par
Next we consider regularization. Let now
\begin{equation} \label{R0p plus radial point cal 2}
\phi_{s} = 1 + \frac{s}{x^{\tindex}}, \, h_{s} = \frac{s/x^{\tindex}}{1 + s/x^{\tindex}},  \quad s \in [0,1],
\end{equation}
which regularizes at both $\psf_{\dmf} \Xd$ and $\tcocf$. For large $K > 0$, we then have
\begin{equation}  \label{R0p plus radial point cal 3}
H_{p} \phi_{s}^{-2K}= 4 K \hat{x}^{\tindex} \bff^2 h_s \utaures \phi_{s}^{-2K}.
\end{equation}
Thus, let $a_{s} = \phi_{s}^{-2K} a_{0}$, $e_{s} = \phi_{s}^{-2K}e_0$,  $b_{j,s} = \phi_{s}^{-K} b_{0,s}$, $j = 1,2,3,4$,
and also
\begin{align} \label{R0p plus radial point cal 4}
\begin{split}
b_{s}^2 & = ( \hat{x}^{\tindex} )^{- 2 \vor } \bff^{-2 \vob} \phi_{s}^{-2K}  \varphi_{1}^2 \varphi_2^2 \varphi_3^2 \\
& \quad \times ( ( 8 \vor - 4 \vob - 4K h_s ) \utaures - ( 4\vor - 4 \vob - 2 ) \ltau  - \delta_0  \phi_{s}^{-2K} \varphi_1^2 \varphi_2^2 \varphi_3^2 )
\end{split}
\end{align}
for some sufficiently small $\delta_0 > 0$, then we have
\begin{equation*}
H_{p} a_s^2 + \delta_0 ( \hat{x}^{\tindex} )^{2 \vor + 2} \bff^{2 \vob + 4} a_{s}^4 = - b_{s}^2 + 2 (\log \bff ) b_{1,s}^2 - b_{2,s}^2 - b_{3,s}^2 - b_{4,s}^2 + e_{s}.
\end{equation*} \par

Notice that if $\delta_0$ is small enough, then (\ref{R0p plus radial point cal 4}) can be made non-negative provided $(8 \vor - 4 \vob - 4 K h_{s}) \utaures$ is small enough. This is again ensured if the support of $\psi$ is small. Thus in this instance, the support of $a_0$ will be dependent on $K$, though this will not matter since it suffices for our consideration here to be local. \par

Let $A_{s}$, $B_{s}$, $E_{s}$ be some quantizations of $a_{s}$, $b_{s}$ and $e_{s}$ respectively, and let also $\tB, \tilde{E} \in \Psfo^{-\infty, 0, -\infty, -\infty, 0, -\infty }$ be chosen such that
\begin{equation} \label{R0p plus radial point cal 6}
\begin{gathered}
 \WFsL( \{  A_{s} \}  ) \subset \Ellss( \tB ), \, \WFsL( \{ E_s \}  ) \subset \Ellss ( \tilde{E} ), \\
\WFs( I - \tB ) \cap \WFsL( \{ A_s \}  ) = \WFs( I - \tilde{E} ) \cap \WFsL( \{ E_s \} ) = \emptyset.
\end{gathered}
\end{equation}
Let $U$ be chosen such that $U \subset \Ellss(B_0)$. Then the standard procedure, where Corollary \ref{commutator corollary} is involved, will give us
\begin{align*}  
\begin{split}
\| B u \|_{ H_{\mathrm{d3sc,3co,res}}^{\ast, \vor, \ast, \ast, \vob, \ast} } & \leq C ( \|  Pu \|_{H_{\mathrm{d3sc,3co,res}}^{\ast, \vor + 1, \ast, \ast, \vob + 2, \ast}}  \\
& \quad + \| E u \|_{H_{\mathrm{d3sc,3co,res}}^{\ast, \vor, \ast, \ast, \vob, \ast}} + \| \tB u \|_{H_{\mathrm{d3sc,3co,res}}^{\ast, \vor -1/2 + \delta, \ast, \ast, \vob, \ast}} + \| u \|_{ H_{\mathrm{d3sc,3co,res}}^{ \ast, N, \ast, \ast, \vob, \ast } } ).
\end{split}
\end{align*}
If $\WFs( \tB )$ is taken small enough, then this argument can be iterated finitely many times, from which we can drop the $\tB u$ term in the above. \par

It remains to show that the above estimate is compatible with the required dynamical condition as in the statement of the proposition. To this end, let $\gamma$ be an integral curve of $\rho_{\dmf}^{-1} \rho_{\dtsccf}^{-1} \rho_{\tcocf}^{-2} H_{p}$ such that $\lim_{t \rightarrow \infty} \gamma(t) \in \mathcal{R}_{0,+}$. Then by Proposition \ref{proposition characterization of the integral curves which converge to R0}, case~(4), we know that $\gamma \subset \tcocf$, $\ltau = \lambda$, $\lmu = 0$ held constant, and $\intnres \rightarrow 0$ as $ t \rightarrow \infty$ along $\gamma$. Thus, it is obvious that $\gamma$ must enter $\WFs( E_0 )$ and subsequently $\Ellss( E_0 )$ in finite backward time. The case of a general $E$ now follows from elliptic regularity and principal type propagation of regularity estimates. 
\end{proof}

\begin{proposition}[Above threshold estimate at $\mathcal{R}_{0,-}$] \label{radial point estimate at R0m}
Let $\vor = \vor_{+}$, $\vob = \vob_{+}$ be variable orders satisfying the conditions specified in \S \ref{variable order construction section}. Let $G, E \in \Psfo^{-\infty, 0 , -\infty, -\infty, 0, -\infty}$ be chosen such that
\begin{equation*}
\mathcal{R}_{0,+} \subset \Ellss(G),
\end{equation*}
with $\WFs(G)$ contained in a given small neighborhood of $\mathcal{R}_{0,+}$, and
\begin{equation*}
\begin{gathered}
\text{all backward integral curves of $\rho_{\dmf}^{-1} \rho_{\dtsccf}^{-1} \rho_{\tcocf}^{-2} H_{p}$ starting from $\mathcal{R}_{0,-}$} \\
\text{are either stationary or enter $\Ellss(E)$ in finite time.}
\end{gathered}
\end{equation*}
Then we can find some arbitrarily small neighborhood $U$ of $\mathcal{R}_{0,-}$ for which the following holds: If we choose $B \in \Psf^{-\infty, -\infty , 0 , 0 ,0 ,-\infty}$ with $\WFs(B) \subset U$, then for every $u \in \mathcal{S}'$ and any choice of $N,M,L,K,S \in \mathbb{R}$, there exists $C > 0$ such that
\begin{align} \label{semi-fredholm estimate at R0m weak}
\begin{split}
\| B u \|_{ H_{\mathrm{d3sc,3co,res}}^{\ast, \vor, \ast, \ast , \vob, \ast} } & \leq  C ( \|  P u \|_{ H_{\mathrm{d3sc,3co,res}}^{\ast, \vor + 1, \ast, \ast, \vob + 2, \ast } } \\
& \quad + \| E u \|_{ H_{\mathrm{d3sc,3co,res}}^{\ast, \vor , \ast, \ast, \vob, \ast} } + \| \tB u \|_{ H_{\mathrm{d3sc,3co,res}}^{\ast, \vor - 1/2 + \delta, \ast, \ast, \vob, \ast } } + \| u \|_{ H_{\mathrm{d3sc,3co,res}}^{N, M, L, K, \vob, S } } ).
\end{split}
\end{align}
in the strong sense that if the right hand of (\ref{semi-fredholm estimate at R0m weak}) is finite, then so is the left hand side, and the estimate holds. \par
However, in this case, we could in fact improve upon (\ref{semi-fredholm estimate at R0m weak}), and get an estimate
\begin{equation} \label{semi-fredhom estimate at R0m}
\| B u \|_{ H_{\mathrm{d3sc,3co,res}}^{\ast, \vor, \ast, \ast, \vob, \ast} } \leq C ( \| Pu \|_{H_{\mathrm{d3sc,3co,res}}^{\ast, \vor + 1, \ast, \ast, \vob + 2, \ast}} + \| E u \|_{H_{\mathrm{d3sc,3co,res}}^{\ast, \vor, \ast, \ast, \vob, \ast}} + \| u \|_{ H_{\mathrm{d3sc,3co,res}}^{ N,M,L,K, \vob, S } } )
\end{equation}
that is again in the strong sense. This improvement is not done via an iteration argument as in the below threshold case. Instead, an interpolation argument is required.
\end{proposition}
\begin{proof}
We will consider essentially the same commutant as in the proof of Proposition \ref{radial point estimate at R0p}, except that the cut-off functions are made to support $a_0$ in a neighborhood of $\mathcal{R}_{0,-}$. Thus, let $\psi \in \mathcal{C}^{\infty}_{c}( [0 , \infty) )$ be such that conditions (\ref{tangential propagation psi condition}) are satisfied.  We now set
\begin{equation*}
\varphi_{1} = \psi( \bff ), \ \varphi_{2} = \psi ( |\utaures|^2 + | \umures |_{h^{\tindex}}^2 ), \ \varphi_3 = \psi( ( \ltau - \lambda )^2 + | \lmu |_{h_{\tindex}}^2 ),
\end{equation*}
i.e. only the definition of $\varphi_{3}$ has been changed. We consider the symbol defined by
\begin{equation*}
a_{0} = ( \hat{x}^{\tindex} )^{-2 \vor - 1} \bff^{-2 \vob - 2} \varphi_1 \varphi_2 \varphi_3,
\end{equation*}
with the implicit understanding that it is supported near $\psf_{\dmf} \Xd$. We then compute
\begin{equation*}
H_{p}a_0^2 = - \tilde{b}_{0}^2 + 2( \log \bff ) b_{1,0}^2 - b_{2,0}^2 - b_{3,0}^2 + e_{1,0} + e_{2,0},
\end{equation*}
where $\tilde{b}_{0}^2$, $b_{1,0}^2$ and $b_{2,0}^2$ are defined exactly as in (\ref{R0p plus radial point cal 1}) and (\ref{R0p plus radial point cal 1.5}), where we now arrange the support of $\psi$ to be small enough, such that
\begin{equation*}
( 8 \vor - 4 \vob ) \utaures - ( 4 \vor - 4 \vob - 2 ) \ltau \simeq ( 4 \vor - 4\vob - 2 ) \lambda
\end{equation*}
on the support of $\tilde{b}_0^2$. In particular, the positivity of $\tilde{b}_0^2$ is guaranteed by the threshold 
\begin{equation*}
\vor - \vob >  \frac{1}{2} \ \text{at} \ \mathcal{R}_{0,+},
\end{equation*}
thus $\tilde{b}_0^2$ can be written as a square of symbol. Additionally, we also set
\begin{equation*}
\begin{gathered}
b_{3,0}^2 = 8 (\hat{x}^{\tindex})^{-2 \vor} \bff^{-2 \vob} ( \intnres ) ( \ltau - \utaures ) \psi'( \intnres ) \varphi_2 \varphi_1^2 \varphi_3^2, \\
e_{1,0} = 4 ( \hat{x}^{\tindex} )^{-2 \vor} \bff^{-2\vob + 1 } ( \ltau - \utaures ) \psi'(\bff) \varphi_1 \varphi_{2}^2 \varphi_3^2,  \\
e_{2,0} = - 8 \lambda (\hat{x}^{\tindex})^{-2\vor} \bff^{-2\vob} |\lmu|_{h_{\tindex}}^2 \psi'( (\ltau + \lambda)^2 + |\lmu|_{h_{\tindex}}^2 ) \varphi_3 \varphi_1^2 \varphi_2^2,
\end{gathered}
\end{equation*}
where we used again $ \ltau - \utaures \simeq - \lambda$ on the support of $b_{3,0}^2$ to see that it can be written as a square of symbol as well. \par

Consider again regularization. Let $\phi_{s}$, $h_{s}$, $s \in (0,1]$ be as in (\ref{R0p plus radial point cal 2}), then for large $K > 0$, we shall set $a_{s} = \phi_{s}^{-K} a_{0}$,  $b_{j,s} = \phi_{s}^{-K} b_{j,0}$,  $j = 1,2,3$,  $e_{k,s} = \phi_{s}^{-2K}e_{j,0}$,  $k = 1,2$, and let $b_{s}^2$ be defined as in (\ref{R0p plus radial point cal 4}). In particular, if the support of $\psi$ is taken to be sufficiently small, then $b_s^2$ is non-negative for the same reason that (\ref{R0p plus radial point cal 4}) is non-negative. Moreover, we have
\begin{equation} \label{R0m cal 0.5}
H_{p} a_s^2 + \delta_0 ( \hat{x}^{\tindex} )^{2 \vor + 2} \bff^{2 \vob + 4} a_{s}^4 = - b_{s}^2 + 2 (\log \bff ) b_{1,s}^2 - b_{2,s}^2 - b_{3,s}^2 + e_{1,s} + e_{2,s}.
\end{equation}  \par

At this stage, we can follows through the standard procedure once again. To be precise, let $A_{s}$, $B_{s}$, $E_{k,s}$, $k =1,2$ be some quantizations of $a_{s}$, $b_{s}$, $e_{k,s}$, $k = 1,2$ respectively, followed by choosing $\ttB, \tilde{E}  \in \Psfo^{-\infty,0,-\infty,-\infty,0,-\infty}$ such that
\begin{equation} \label{R0m cal 1}
\begin{gathered}
 \WFs( \{  A_{s} \} ) \subset \Ellss( \ttB ), \WFsL( \{ E_{k,s} \}  ) \subset \Ellss( \tilde{E} ), \ k =1,2, \\
\WFs( I - \ttB ) \cap \WFsL( \{ A_s \} ) = \WFs( I - \tilde{E} ) \cap \WFs( \{ E \} )  = \emptyset.
\end{gathered}
\end{equation}
Let $U$ be chosen such that $U \subset \Ellss(B_0)$. Then the standard procedure gives us 
\begin{align} \label{R0m cal 2}
\begin{split}
\| B u \|_{ H_{\mathrm{d3sc,3co,res}}^{\ast, \vor, \ast, \ast , \vob, \ast} } & \leq  C ( \| P u \|_{ H_{\mathrm{d3sc,3co,res}}^{\ast, \vor + 1, \ast, \ast, \vob + 2, \ast } } \\
& \quad + \| \ttB u \|_{ H_{\mathrm{d3sc,3co,res}}^{\ast, \vor , \ast, \ast, \vob, \ast} } + \| \tilde{B} u \|_{ H_{\mathrm{d3sc,3co,res}}^{\ast, \ast - 1/2 + \delta, \ast, \ast, \vob, \ast } } + \| u \|_{ H_{\mathrm{d3sc,3co,res}}^{N, M, L, K, \vob, S } } ),
\end{split}
\end{align}
which is estimate (\ref{semi-fredholm estimate at R0m weak}) in a local sense. One can then replace $\ttB u$ by $\tB u$ in the above via elliptic regularity estimate provided $\WFs(\ttB)$ is small (by choice). \par

It remains to show that the above estimate is compatible with the required dynamical condition as in the statement of the proposition. To this end, let $\gamma$ be an integral curve of $\rho_{\dmf}^{-1} \rho_{\dtsccf}^{-1} \rho_{\tcocf}^{-2} H_{p}$ such that $\lim_{t \rightarrow \infty} \gamma(t) \in \mathcal{R}_{0,-}$. Then the only possible trajectories which $\gamma$ could follow come from Proposition \ref{proposition characterization of the integral curves which converge to R0}, cases~(1) and (3). \par

In case (1), we must have $\gamma \subset \psf_{\dmf} \Xd$, where $\utaures = 0$, $\umures = 0$ and $\intt = \lambda^2$ are held constant over $\gamma$. Moreover, by Proposition \ref{precise description of integral flow near ff}, case~(1), we know that $\lim_{t \rightarrow \infty} \ltau = -\lambda$ along $\gamma$. Thus, it is obvious that $\gamma$ enters $\WFs(E_{1,0})$ in finite backward time. In fact, the same conditions almost characterizes case (2) as well, except that we now have $\gamma \subset \tcocf$. Therefore, in this case it is also obvious that $\gamma$ must enter $\WFs(E_{2,0})$ in finite backward time. \par
Since $\Ellss(\tilde{E})$ contains both $\WFs(E_{1,0})$ and $\WFs(E_{2,0})$, it follows that in either cases, $\gamma$ must enter $\Ellss(\tilde{E})$ in finite backward time. The case of a general $E$ now follows from elliptic regularity and principal type propagation of regularity estimates.
  \par

Now, in the below threshold case, we were able to improve from (\ref{semi-fredholm estimate at R0m weak}) to (\ref{semi-fredhom estimate at R0m}) via iterations. However, as in the proof of Proposition \ref{Propositive estimate at Rndff+}, in general this cannot be done in the above threshold case, since even starting from (\ref{semi-fredholm estimate at R0m weak}), one would require the above threshold condition 
\begin{equation*}
\vor - \vob - \frac{1}{2} + \delta < \frac{1}{2} \ \text{at} \ \mathcal{R}_{0,-},
\end{equation*}
to be satisfied, which, after finitely many iterations, will necessarily cease to hold. However, we note also that, unlike in the typical radial point estimate, the amount the regularization that we are allowed to have here is unlimited. Thus, we can nevertheless prove (\ref{semi-fredhom estimate at R0m}) by using the following twist of the standard procedure. Unfortunately, this forces us to invoke again full details of the standard procedure.  \par

Let $A_{s}$, $B_{s}$, $B_{j,s}$, $j=2,3$, $E_{k,s}$, $k=1,2$ be some quantizations of $a_{s}$, $b_{s}$, $b_{j,s}$, $j=2,3$, $e_{k,s}$, $k=1,2$ respectively, and let $B_{1,s} \in L^{\infty}( ( 0, 1 ]_{s} ; \Psf^{-\infty, \vor, -\infty, -\infty, \vob, -\infty} )$ be chosen such that conditions (\ref{always these conditions for B1s}) are satisfied. Let also $\Lambda$ be some quantization of an elliptic extension of $(\hat{x}^{\tindex})^{\vor + 1} \bff^{\vob + 2}$ from the support of $a_0$. Then by using (\ref{R0m cal 0.5}), we have
\begin{align} \label{R0m cal 0}
\begin{split}
i [ P, A_s ] + \delta_0 ( \Lambda A_{s}^{\ast} A_{s} )^{\ast} ( \Lambda A_{s}^{\ast} A_{s} ) = & - B_{s}^{\ast} B_{s} +  2 B_{1,s}^{\ast} (\log \bff) B_{1,s} \\
& - B_{2,s}^{\ast} B_{2,s} - B_{3,s}^{\ast} B_{3,s} + E_{1,s} + E_{2,s} + R_{s}
\end{split}
\end{align}
for some $ R_{s}  \in L^{\infty}( (0,1]_{s} ; \Psf^{-\infty, \vor - 1/2 + \delta, -\infty, -\infty, \vob , -\infty } )$ by Corollary \ref{commutator corollary}. By pairing (\ref{R0m cal 0}) with $u \in \mathcal{S}'$ and dropping some of the non-negative terms, it follows that
\begin{align} \label{R0m plus radial point cal 4}
\begin{split}
& \| B_{s} u \|_{L^{2}}^2 + \delta_0 \| \Lambda A_{s}^{\ast} A_s u \|_{L^2}^2 \\
& \qquad \leq - 2 \mathrm{Im} \langle Pu, A_s^{\ast} A_{s} u \rangle_{L^2}   +  \langle E_{1,s}u ,u \rangle_{L^2} + \langle  E_{2,s}u , u \rangle_{L^2} + \langle R_{s}u ,u \rangle_{L^2} .
\end{split}
\end{align} \par

Following the standard procedure, we again have
\begin{equation*}
 2 | \mathrm{Im}  \langle Pu, A_s^{\ast} A_s \rangle_{L^{2}} | \leq C \delta_0^{-1} \| P u \|_{ \so^{\ast, \vor + 1, \ast , \ast, \vob + 2, \ast } }^2  + \delta_0 \| \Lambda A_s^{\ast} A_s u \|_{L^{2}}^2 + C \| u \|_{ \so^{N, M, L , K, \vob, S } }^2.
\end{equation*}
On the other hand, Let  $\ttB, \tilde{E} \in \Psfo^{-\infty, 0,-\infty,-\infty,0,-\infty}$ be chosen such that conditions (\ref{R0m cal 1}) are still satisfied, then we also have
\begin{equation*}
| \langle E_{j,s} u , u \rangle_{L^2} | \leq C ( \| \tilde{E} u \|_{H_{\mathrm{d3sc,3co,res}}^{\ast, \vor, \ast, \ast, \vob, \ast }}^2 + \| u \|_{H_{\mathrm{d3sc,3co,res}}^{N, M, L, K, \vob, S}}^2 ), \quad j =1,2,
\end{equation*}
so by substituting the above estimates back into (\ref{R0m plus radial point cal 4}), we now have
\begin{equation*}
 \| B_{s}u \|_{L^2}^{2}  \leq C ( \| P u \|_{H_{\mathrm{d3sc,3co,res}}^{\ast, \vor + 1, \ast, \ast, \vob + 2, \ast}} + \| \tilde{E} u \|_{H_{\mathrm{d3sc,3co,res}}^{\ast, \vor, \ast, \ast, \vob, \ast}}^2 + | \langle R_{s} u , u \rangle_{L^2} | + \| u \|_{H_{\mathrm{d3sc,3co,res}}^{N,M,L,K, \vob, S}} ).
\end{equation*} \par

It remains to estimate the $ \langle R_{s}u , u \rangle_{L^2} $ term, which is where the twist is required. We will essentially still do this through the usual method, except we now compute
\begin{equation*}
| \langle R_{s}u , u  \rangle_{L^2} | = | \langle \phi_{s}^{K} R_{s} \phi_{s}^{K} \phi_{s}^{-K}u , \phi_{s}^{-K} u \rangle_{L^2} |  \leq C ( \| \ttB \phi_{s}^{-K} u \|_{H_{\mathrm{d3sc,3co,res}}^{\ast, \vor -1/2 + \delta, \ast, \ast, \vob, \ast}}^2 + \| u \|_{H_{\mathrm{d3sc,3co,res}}^{ N,M,L,K, \vob, S}}^2 )
\end{equation*}
Indeed, this can be done since $\phi_{s}^{K} R_{s} \phi_{s}^{K}$ still belongs to $L^{\infty}( (0,1) ; \Psf^{-\infty, \vor - 1/2 + \delta, -\infty, -\infty, \vob , -\infty }  )$, and that multiplication by $\phi_{s}^{K}$ does not change the uniform wavefront set of $R_{s}$. In fact, we could have chosen $B_{s}$ to take the form $B_{0} \phi_{s}^{-K}$. Thus by elliptic estimate, we have
\begin{align} \label{R0m plus radial point cal 4.5}
\begin{split}
 \| B_0 \phi_{s}^{-K} u \|_{H_{\mathrm{d3sc,3co,res}}^{\ast, \vor, \ast, \ast, \vob, \ast}} & \leq C ( \| Pu \|_{H_{\mathrm{d3sc,3co,res}}^{\ast, \vor+1, \ast, \ast, \vob + 2, \ast}} \\
& \quad + \| \tilde{E} u \|_{H_{\mathrm{d3sc,3co,res}}^{\ast, \vor, \ast, \ast, \vob, \ast}} + \|  \ttB \phi_{s}^{-K} u \|_{H_{\mathrm{d3sc,3co,res}}^{\ast, \vor - 1/2 + \delta, \ast, \ast, \vob, \ast}} + \| u \|_{H_{\mathrm{d3sc,3co,res}}^{N,M,L,K,\vob, S}} ). 
\end{split}
\end{align} \par
Notice that so far we only needed to assume that $u \in H_{\mathrm{d3sc,3co,res}}^{N, M, L, L, \vob, S}$ and $\tilde{E} u \in H_{\mathrm{d3sc,3co,res}}^{\ast, \vor, \ast, \ast, \vob, \ast}$, so it remains to remove the $\ttB \phi_{s}^{-K}$ term with these regularity assumptions. For this, we will write $\ttB$ as a sum of two terms $\tilde{G}_{0}, \tilde{G}_{1} \in \Psfo^{-\infty, 0, -\infty, -\infty, 0, -\infty}$ modulo an error that is symbolically residual, so that
\begin{align} \label{R0m plus radial point cal 5}
\begin{split}
 \| \tilde{G} \phi_{s}^{-K} u \|_{H_{\mathrm{d3sc,3co,res}}^{\ast, \vor - 1/2 + \delta, \ast ,\ast , \vob, \ast}} & \leq  \| \tilde{G}_0 \phi_{s}^{-K} u \|_{H_{\mathrm{d3sc,3co,res}}^{\ast, \vor - 1/2 + \delta, \ast ,\ast , \vob, \ast}}\\
& \quad + \| G_1 \phi_{s}^{-K} u \|_{H_{\mathrm{d3sc,3co,res}}^{\ast, \vor - 1/2 + \delta, \ast ,\ast , \vob, \ast}} + \| u \|_{H_{\mathrm{d3sc,3co,res}}^{N,M,L,K, \vob, S}}.
\end{split}
\end{align}
Moreover, we can choose $\tilde{G}_0$, $\tilde{G}_1$ to be such that $\mathcal{R}_{0, - } \subset \Ellss( \tilde{G}_0 ) \subset \WFs( \tilde{G}_0 ) \subset \Ellss( B )$ and $\WFs( \tilde{G}_1 ) \cap \mathcal{R}_{0,-} = \emptyset$. Notice that the latter condition can only be satisfied since $\mathcal{R}_{0,-}$ is isolated from the other radial sets (hence this argument does not apply at a fiber of $\mathcal{R}_{\mathrm{n},\dff}$). \par 
One can now estimate the $\tilde{G}_0$ term by interpolation and elliptic regularity. This yields
\begin{equation} \label{R0m plus radial point cal 6}
\| \tilde{G}_0 \phi_{s}^{-K} u \|_{H_{\mathrm{d3sc,3co,res}}^{\ast, \vor - 1/2 + \delta, \ast, \ast, \vob, \ast}} \leq C \epsilon \| B_0 \phi_{s}^{-K} \|_{H_{\mathrm{d3sc,3co,res}}^{\ast, \vor, \ast, \ast, \vob, \ast}} + C_{\epsilon} \| u \|_{H_{\mathrm{d3sc,3co,res}}^{-N,-M,-L,-K, \vob, -S}}. 
\end{equation}
On the other hand, we can use principal type propagation of regularity to estimate the $\tilde{G}_1$ term. Indeed, for every point in $\WFs( \tilde{G}_1 )$ and some integral curve segment of $\rho_{\dmf}^{-1} \rho_{\dtsccf}^{-1} \rho_{\tcocf}^{-2} H_{p}$ ending at this point, microlocal regularity can be propagated backwards either into $\Ellss(B_0)$, or alternatively into an arbitrarily small (by making $\WFs(\tilde{G}_0)$ sufficiently small) neighborhood of $\Ellss(\tilde{E})$. \par
Thus, suppose we choose $\tilde{E}' \in \Psf^{-\infty, -\infty, 0,0,0,-\infty}$ such that $\WFs( \tilde{E} ) \subset \Ellss( \tilde{E}' )$, and that every integral curve described above enters $\Ellss( \tilde{E}' )$ in finite backward time, then principal type propagation of regularity estimates ensure that 
\begin{equation} \label{R0m plus radial point cal 7}
\| G_1 \phi_{s}^{-K} u \|_{H_{\mathrm{d3sc,3co,res}}^{\ast, \vor - 1/2 + \delta, \ast, \ast, \vob, \ast}} \leq C ( \| Pu \|_{H_{\mathrm{d3sc,3co,res}}^{\ast, \vor + 1/2 + \delta, \ast, \ast, \vob + 2, \ast}} + \| \tilde{E}' u \|_{H_{\mathrm{d3sc,3co,res}}^{\ast, \vor - 1/2 + \delta, \ast, \ast, \vob, \ast}} ).
\end{equation}
We can now back substitute (\ref{R0m plus radial point cal 6}), (\ref{R0m plus radial point cal 7}) into (\ref{R0m plus radial point cal 5}) and then (\ref{R0m plus radial point cal 4.5}), which, for $\epsilon > 0$ small enough, finally gives us
\begin{align*}
\frac{1}{2}\| B_0 \phi_{s}^{-K} u \|_{H_{\mathrm{d3sc,3co,res}}^{\ast, \vor, \ast, \ast, \vob, \ast}} & \leq ( 1 - C \epsilon) \| B_0 \phi_{s}^{-K} u \|_{H_{\mathrm{d3sc,3co,res}}^{\ast, \vor, \ast, \ast, \vob, \ast}} \\
& \quad + C_{\epsilon} (  \| P u \|_{H_{\mathrm{d3sc,3co,res}}^{\ast, \vor + 1, \ast, \ast, \vob + 2, \ast }} + \| \tilde{E}' u \|_{H_{\mathrm{d3sc,3co,res}}^{\ast, \vor, \ast ,\ast, \vob, \ast}} + \| u \|_{H_{\mathrm{d3sc,3co,res}}^{N,M,L,K, \vob, S}} ).
\end{align*} \par
The rest of the argument is now again standard. 
\end{proof}

\section{Radial point estimates in the transversal directions} \label{transversal propagation section}
In \S \S \ref{subsection first radial point estimate at dmf},  \ref{Radial point estimate at dff subsection}, we discussed radial point estimates for $\mathcal{R}_{\mathrm{n},\dmf, \pm}$, $\mathcal{R}_{\mathrm{n},\dff,\pm}$ when they are restricted to $\tcocf$. In this section, we will state the analogous estimates when $\mathcal{R}_{\mathrm{n},\dmf, \pm}$, $\mathcal{R}_{\mathrm{n},\dff,\pm}$ are restricted away from $\tcocf$. \par

It should be obvious (to those readers who have examined \S\S \ref{subsection first radial point estimate at dmf}, \ref{Radial point estimate at dff subsection}) that the results of \S\S \ref{subsection; radial estimates at dmf away from 3cocf}, \ref{subsection; radial estimates at dff away from 3cocf} can be proved using essentially the same procedures as those presented in the proofs of Propositions \ref{Proposition Rndm-}--\ref{Propositive estimate at Rndff+}. Thus, most of the proofs in this section will be omitted.

In fact, the contents of Propositions \ref{subsection; radial estimates at dmf away from 3cocf}--\ref{transversal above threshold Rndff+} can be easily absorbed into the statements of Propositions \ref{Proposition Rndm-}--\ref{Propositive estimate at Rndff+} directly. Nevertheless, we have decided to state the propositions of this section separately to emphasize on the their dyamical natures, i.e., they are indeed directly analogous to the propagation estimates in the `normal directions' from \cite[Chapter 14]{AndrasThesis}. In particular, such a discussion leads naturally to the brief presentation of \S \ref{radial point estimate 3sc no second microlocal}, which is in some sense the more natural way to view transversal propagation in this context, where the introduction of second microlocalization is actually unnecessary.

\subsection{Radial point estimates at $\mathcal{R}_{\mathrm{n},\dmf,\pm} \backslash \tcocf$}
\label{subsection; radial estimates at dmf away from 3cocf}
We first state the radial point estimates at $\mathcal{R}_{\mathrm{n,dmf},\pm} \backslash \tcocf$. For every $\beta_{\tindex} \in \Sigma_{\mathrm{n}}$, we will define $\mathcal{R}_{\mathrm{n,dmf}, \pm}( \beta_{\tindex} ) \subset \mathcal{R}_{\mathrm{n}, \dmf, \pm} \backslash \tcocf$ again by (\ref{eq; definition of R n dmf at point}). Then we have a radial point estimate for every such $\mathcal{R}_{\mathrm{n},\dmf,\pm}(\beta_{\tindex})$.

\begin{proposition}[Below threshold estimate at $\mathcal{R}_{\mathrm{n,dmf},-}  \backslash \tcocf$] \label{transversal below threshold at ndmf-}
Let $\vor = \vor_{+}$, $\vol = \vol_{+}$ be variable orders satisfying the conditions specified in \S \ref{variable order construction section}. Suppose that $\beta_{\tindex} \in \Sigma_{\mathrm{n}}$. Moreover, let $ E \in \Psfo^{-\infty, 0, -\infty, 0, -\infty, -\infty}$ be chosen such that
\begin{equation*}
\begin{gathered}
\text{all backward integral curves of $\rho_{\dmf}^{-1} \rho_{\dtsccf}^{-1} \rho_{\tcocf}^{-2} H_{p}$ starting from $\mathcal{R}_{\mathrm{n},\dmf, -}(\beta_{\tindex})$} \\
\text{are either stationary or enter $\Ellss(E)$ in finite time. }
\end{gathered}
\end{equation*}
Then we can find an arbitrarily small neighborhood $U$ of $\mathcal{R}_{\mathrm{n}, \dmf,-}(\beta_{\tindex})$ for which the following holds: If we choose $B \in \Psf^{-\infty, 0, -\infty , 0 ,-\infty ,-\infty}$ with $\WFs(B) \subset U$, then for every $u \in \mathcal{S}'$ and any choice of $N,M,L,K, F, S \in \mathbb{R}$, there exists $C > 0$ such that
\begin{equation} \label{RPS away from tcocf ndmf- actual estimate}
\| B u \|_{ H_{\mathrm{d3sc,3co,res}}^{\ast, \vor, \ast, \vor + \vol, \ast, \ast} }  \leq C ( \| Pu \|_{H_{\mathrm{d3sc,3co,res}}^{\ast, \vor + 1, \ast, \vor + \vol + 1, \ast, \ast}}   + \| E u \|_{H_{\mathrm{d3sc,3co,res}}^{\ast, \vor, \ast, \vor + \vol, \ast, \ast}} + \| u \|_{ H_{\mathrm{d3sc,3co,res}}^{ N, M, L, K, F, S } } )
\end{equation}
in the strong sense that if the right hand of (\ref{RPS away from tcocf ndmf- actual estimate}) is finite, then so is the left hand side, and the estimate holds.
\end{proposition}

\begin{proposition}[Above threshold estimate at $\mathcal{R}_{\mathrm{n,dmf},+}  \backslash \tcocf$] \label{transversal below threshold at ndmf+}
Let $\vor = \vor_{+}$, $\vol = \vol_{+}$ be variable orders satisfying the conditions specified in \S \ref{variable order construction section}. Suppose that $\beta_{\tindex} \in \Sigma_{\mathrm{n}}$. Moreover, let $ E \in \Psfo^{-\infty, 0, -\infty, 0, -\infty, -\infty}$ be chosen such that
\begin{equation*}
\begin{gathered}
\text{all backward integral curves of $\rho_{\dmf}^{-1} \rho_{\dtsccf}^{-1} \rho_{\tcocf}^{-2} H_{p}$ starting from $\mathcal{R}_{\mathrm{n},\dmf, +}(\beta_{\tindex})$} \\
\text{are either stationary or enter $\Ellss(E)$ in finite time. }
\end{gathered}
\end{equation*}
Then we can find an arbitrarily small neighborhood $U$ of $\mathcal{R}_{\mathrm{n}, \dmf, + }(\beta_{\tindex})$ for which the following holds: If we choose $B \in \Psf^{-\infty, 0, -\infty , 0 ,-\infty ,-\infty}$ with $\WFs(B) \subset U$, then for every $u \in \mathcal{S}'$ and any choice of $N,M,L,K, F, S \in \mathbb{R}$, there exists $C > 0$ such that
\begin{equation} \label{RPS away from tcocf ndmf- actual estimate}
\| B u \|_{ H_{\mathrm{d3sc,3co,res}}^{\ast, \vor, \ast, \vor + \vol, \ast, \ast} }  \leq C ( \| Pu \|_{H_{\mathrm{d3sc,3co,res}}^{\ast, \vor + 1, \ast, \vor + \vol + 1, \ast, \ast}}   + \| E u \|_{H_{\mathrm{d3sc,3co,res}}^{\ast, \vor, \ast, \vor + \vol, \ast, \ast}} + \| u \|_{ H_{\mathrm{d3sc,3co,res}}^{ N, M, L, K, F, S } } )
\end{equation}
in the strong sense that if the right hand of (\ref{RPS away from tcocf ndmf- actual estimate}) is finite, then so is the left hand side, and the estimate holds.
\end{proposition}

\begin{proof}[Proofs of Propositions \ref{transversal below threshold at ndmf-} and \ref{transversal below threshold at ndmf+}]
The proofs of these propositions follow directly from mildly modifying the proofs of Propositions \ref{Proposition Rndm-} and \ref{Proposition Rndm+} respectively, which follows from the fact that (\ref{section overview microlocal Hamiltonian vector field 1}) holds even at $\mathcal{R}_{\mathrm{n}, \dmf, \pm} \backslash \tcocf$. 
\end{proof}

\subsection{Radial point estimates at $\mathcal{R}_{\mathrm{n},\dff,\pm} \backslash \tcocf$}  
\label{subsection; radial estimates at dff away from 3cocf}
Next we consider radial point estimates at $\mathcal{R}_{\mathrm{n}, \dff, \pm} \backslash \tcocf$. Let $\beta_{\tindex} \in \Sigma_{\mathrm{n}}$, and define $\mathcal{R}_{\mathrm{n}, \dff, \pm}(\beta_{\tindex}) \subset \mathcal{R}_{\mathrm{n}, \dff, \pm} \backslash \tcocf$  by (\ref{eq; definition of R n dff at point}) as before. Then we have a radial point estimate for each $\mathcal{R}_{\mathrm{n}, \dff, \pm}(\beta_{\tindex})$.
\begin{proposition}[Below threshold estimate at $\mathcal{R}_{\mathrm{n},\dff, -} \backslash \tcocf$] 
\label{transversal radial Rndff-}
Let $\vor = \vor_{+}$, $\vol = \vol_{+}$ be variable orders satisfying the conditions specified in \S \ref{variable order construction section}. Suppose that $\beta_{\tindex} \in \Sigma_{\mathrm{n}}$. Moreover, let $ E \in \Psfo^{-\infty, -\infty, 0, 0, -\infty, -\infty}$ be chosen such that
\begin{equation*}
\begin{gathered}
\text{all backward integral curves of $\rho_{\dmf}^{-1} \rho_{\dtsccf}^{-1} \rho_{\tcocf}^{-2} H_{p}$ starting from $\mathcal{R}_{\mathrm{n},\dff, -}(\beta_{\tindex})$} \\
\text{are either stationary or enter $\Ellss(E)$ in finite time. }
\end{gathered}
\end{equation*}
Then we can find an arbitrarily small neighborhood $U$ of $\mathcal{R}_{\mathrm{n}, \dff,-}(\beta_{\tindex})$ for which the following holds: If we choose $B \in \Psf^{-\infty, -\infty , 0, 0 ,-\infty ,-\infty}$ with $\WFs(B) \subset U$, then for every $u \in \mathcal{S}'$ and any choice of $N,M,K, F, S \in \mathbb{R}$, there exists $C > 0$ such that
\begin{equation} \label{RPS away from tcocf ndmf- actual estimate}
\| B u \|_{ H_{\mathrm{d3sc,3co,res}}^{\ast, \ast, \vol, \vor + \vol, \ast, \ast} }  \leq C ( \| Pu \|_{H_{\mathrm{d3sc,3co,res}}^{\ast, \ast, \vol, \vor + \vol + 1, \ast, \ast}}   + \| E u \|_{H_{\mathrm{d3sc,3co,res}}^{\ast, \ast, \vol, \vor + \vol, \ast, \ast}} + \| u \|_{ H_{\mathrm{d3sc,3co,res}}^{ N, M, \vol, K, F, S } } )
\end{equation}
in the strong sense that if the right hand of (\ref{RPS away from tcocf ndmf- actual estimate}) is finite, then so is the left hand side, and the estimate holds.
\end{proposition}

\begin{proposition}[Above threshold estimate at $\mathcal{R}_{\mathrm{n},\dff, +} \backslash \tcocf$] \label{transversal above threshold Rndff+}
Let $\vor = \vor_{+}$, $\vol = \vol_{+}$ be variable orders satisfying the conditions specified in \S \ref{variable order construction section}. Suppose that $\beta_{\tindex} \in \Sigma_{\mathrm{n}}$. Moreover, let $G, E \in \Psf^{-\infty, -\infty , 0 , 0 ,0 ,-\infty}$ be chosen such that
\begin{equation*}
\mathcal{R}_{\mathrm{n},\dff,+}(\beta_{\tindex}) \backslash \tcocf \subset \Ellss(G), 
\end{equation*}
with $\WFs(G)$ contained in a given small neighborhood of $\mathcal{R}_{\mathrm{n},\dff,+}(\beta_{\tindex})$, and
\begin{equation*}
\begin{gathered}
\text{all backward integral curves of $\rho_{\dmf}^{-1} \rho_{\dtsccf}^{-1} \rho_{\tcocf}^{-2} H_{p}$ starting from $\mathcal{R}_{\mathrm{n},\dff, +}(\beta_{\tindex})$} \\
\text{are either stationary or enter $\Ellss(E)$ in finite time.}
\end{gathered}
\end{equation*}
Then we can find an arbitrarily small neighborhood $U$ of $\mathcal{R}_{\mathrm{n}, \dff, +}(\beta_{\tindex})$ for which the following holds: If we choose $B \in \Psf^{-\infty, -\infty , 0 , 0 ,-\infty ,-\infty}$ with $\WFs(B) \subset U$, then for every $u \in \mathcal{S}'$ and any choice of $N,M,K,S \in \mathbb{R}$, there exists $C > 0$ such that
\begin{align}  \label{semi-fredhom estimate at Rndff+}
\begin{split}
 \| B u \|_{ H_{\mathrm{d3sc,3co,res}}^{\ast, \ast, \vol, \vor + \vol, \ast, \ast} } & \leq  C ( \| P u \|_{ H_{\mathrm{d3sc,3co,res}}^{\ast, \ast, \vol, \vor + \vol  + 1, \ast, \ast } } \\
& \quad + \| E u \|_{ H_{\mathrm{d3sc,3co,res}}^{\ast, \ast, \vol, \vor + \vol, \ast, \ast} } + \| G u \|_{ H_{\mathrm{d3sc,3co,res}}^{\ast, \ast, \vol, \vor + \vol - 1/2 + \delta, \ast, \ast } } + \| u \|_{ H_{\mathrm{d3sc,3co,res}}^{N, M, \vol, K, F , S } } ).
\end{split}
\end{align}
in the strong sense that if the right hand of (\ref{semi-fredhom estimate at Rndff+}) is finite, then so is the left hand side, and the estimate holds.
\end{proposition}

\begin{proof}[Proofs of Propositions \ref{transversal radial Rndff-} and \ref{transversal above threshold Rndff+}]
The proofs of these propositions follow again from modifying the proofs of Propositions \ref{below threshold Rndff- proposition} and \ref{Propositive estimate at Rndff+} respectively, which can be done since (\ref{section overview microlocal Hamiltonian vector field 1}) remains valid at $\mathcal{R}_{\mathrm{n}, \dmf, \pm} \backslash \tcocf$. 
\end{proof}

\subsection{Radial point estimates at $\mathcal{R}_{\mathrm{n}, \pm} \backslash \pi_{\ff}^{-1}(o_{\mathcal{C}^{\tindex}})$}
\label{radial point estimate 3sc no second microlocal}
Much as in the discussion at the end of \S \ref{principal type propagation section}, we can also relax the second microlocal structure in the transversal radial point estimates away from $\tcocf$. These estimates can thus be stated with respect to $\mathcal{R}_{\mathrm{n},\pm}$. In particular, only the three-body structure is relevant. We recall that the dynamical roles of these sets are replaced by the combined effects of $\mathcal{R}_{\mathrm{n}, \dmf, \pm}$, $\mathcal{R}_{\mathrm{n}, \dff, \pm}$ in the second microlocal framework. \par

For $\beta_{\tindex} = ( y_{\tindex,0}, \tau_{\tindex,0}, \mu_{\tindex,0} ) \in \Sigma_{\mathrm{n}}$, we will again write
\begin{equation*}
\mathcal{R}_{\mathrm{n},\pm}(\beta_{\tindex}) = \mathcal{R}_{\mathrm{n}, \pm} \cap \{ y_{\tindex} = y_{\tindex,0}, \tau_{\tindex} = \tau_{\tindex,0}, \mu_{\tindex} = \mu_{\tindex,0} \} \subset \mathcal{R}_{\mathrm{n},\pm} \backslash { \pi_{\ff}^{-1}(o_{\mathcal{C}^{\tindex}}) }.
\end{equation*}
Thus we have the following results:

\begin{proposition}[Below threshold estimate at $\mathcal{R}_{\mathrm{n}, -} \backslash \pi_{\ff}^{-1}(o_{\mathcal{C}^{\tindex}})$] 
\label{below threshold Rnm}
Let $\vor = \vor_{+}$, $\vol = \vol_{+}$ be variable orders satisfying the conditions specified in \S \ref{variable order construction section}. Suppose that $\beta_{\tindex} \in \Sigma_{\mathrm{n}}$. Moreover, let $ E \in \Psi^{-\infty, 0, 0}_{\mathrm{3sc}}$ be chosen such that $\mathrm{WF}_{\mathrm{3sc},\sigma}'(E) \cap \pi_{\ff}^{-1}(o_{\mathcal{C}^{\tindex}}) = \emptyset$, and
\begin{equation*}
\begin{gathered}
\text{all backward integral curves of $\rho_{\dmf}^{-1} \rho_{\dtsccf}^{-1} \rho_{\tcocf}^{-2} H_{p}$ starting from $\mathcal{R}_{\mathrm{n},-}(\beta_{\tindex})$} \\
\text{are either stationary or enter $\Ellss(E)$ in finite time. }
\end{gathered}
\end{equation*}
Then we can find an arbitrarily small neighborhood $U$ of $\mathcal{R}_{\mathrm{n},-}(\beta_{\tindex})$ for which the following holds: If we choose $B \in \Psi^{-\infty, 0,0}_{\mathrm{3co}}$ with $\mathrm{WF}_{\mathrm{3sc},\sigma}'(B) \subset U$, then for every $u \in \mathcal{S}'$ and any choice of $N,M,L,K, F, S \in \mathbb{R}$, there exists $C > 0$ such that
\begin{equation} \label{RPS away from tcocf ndmf- actual estimate}
\| B u \|_{ H_{\mathrm{d3sc,3co,res}}^{\ast, \vor, \ast, \vor + \vol, \ast, \ast} }  \leq C ( \| Pu \|_{H_{\mathrm{d3sc,3co,res}}^{\ast, \vor + 1, \ast, \vor + \vol + 1, \ast, \ast}}   + \| E u \|_{H_{\mathrm{d3sc,3co,res}}^{\ast, \vor, \ast, \vor + \vol, \ast, \ast}} + \| u \|_{ H_{\mathrm{d3sc,3co,res}}^{ N, M, L, K, F, S } } )
\end{equation}
in the strong sense that if the right hand of (\ref{RPS away from tcocf ndmf- actual estimate}) is finite, then so is the left hand side, and the estimate holds.
\end{proposition}

\begin{proposition}[Above threshold estimate at $\mathcal{R}_{\mathrm{n}, +} \backslash \pi_{\ff}^{-1}(o_{\mathcal{C}^{\tindex}})$]  
\label{above threshold Rnp}
Let $\vor = \vor_{+}$, $\vol = \vol_{+}$ be variable orders satisfying the conditions specified in \S \ref{variable order construction section}. Suppose that $\beta_{\tindex} \in \Sigma_{\mathrm{n}}$. Moreover, let $G,E \in \Psi^{-\infty, 0,0}_{\mathrm{3sc}}$ be chosen such that
\begin{equation*}
\mathrm{WF}_{\mathrm{3sc},\sigma}'(G) \cap \pi_{\ff}^{-1}( o_{\mathcal{C}^{\tindex}} ) = \mathrm{WF}_{\mathrm{3sc},\sigma}'(E) \cap \pi_{\ff}^{-1}( o_{\mathcal{C}^{\tindex}} ) = \emptyset, \quad \mathcal{R}_{\mathrm{n}, +}(\beta_{\tindex}) \backslash \pi_{\ff}^{-1}(o_{\mathcal{C}^{\tindex}}) \subset \mathrm{Ell}_{\mathrm{3sc},\sigma}(G), 
\end{equation*}
with $\mathrm{WF}_{\mathrm{3sc},\sigma}'(G)$ contained in a given small neighborhood of $\mathcal{R}_{\mathrm{n},+}(\beta_{\tindex})$, and
\begin{equation*}
\begin{gathered}
\text{all backward integral curves of $\rho_{\dmf}^{-1} \rho_{\dtsccf}^{-1} \rho_{\tcocf}^{-2} H_{p}$ starting from $\mathcal{R}_{\mathrm{n}, +}(\beta_{\tindex}) $} \\
\text{are either stationary or enter $\Ellss(E)$ in finite time.}
\end{gathered}
\end{equation*}
Then we can find an arbitrarily small neighborhood $U$ of $\mathcal{R}_{\mathrm{n}, +}(\beta_{\tindex}) $ for which the following holds: If we choose $B \in \Psi^{-\infty,0,0}_{\mathrm{3sc}}$ with $\mathrm{WF}_{\mathrm{3sc},\sigma}'(B) \subset U$, then for every $u \in \mathcal{S}'$ and any choice of $N,M,L,K,F, S \in \mathbb{R}$, there exists $C > 0$ such that
\begin{align} \label{RPS away from tcocf ndmf+ actual estimate}
\begin{split}
 \| B u \|_{ H_{\mathrm{d3sc,3co,res}}^{\ast, \ast, \vol, \vor + \vol, \ast, \ast} } & \leq  C ( \| P u \|_{ H_{\mathrm{d3sc,3co,res}}^{\ast, \ast, \vol, \vor + \vol  + 1, \ast, \ast } } \\
& \quad + \| E u \|_{ H_{\mathrm{d3sc,3co,res}}^{\ast, \ast, \vol, \vor + \vol, \ast, \ast} } + \| G u \|_{ H_{\mathrm{d3sc,3co,res}}^{\ast, \ast, \vol, \vor + \vol - 1/2 + \delta, \ast, \ast } } + \| u \|_{ H_{\mathrm{d3sc,3co,res}}^{N, M, \vol, K, F, S } } ).
\end{split}
\end{align}
in the strong sense that if the right hand of (\ref{RPS away from tcocf ndmf+ actual estimate}) is finite, then so is the left hand side, and the estimate holds.
\end{proposition}
Although the proofs of Propositions \ref{below threshold Rnm}, \ref{above threshold Rnp} will again be omitted{\ep}much like what was done for Proposition \ref{principal type propagation three-body version}, we remark that their proofs, if written, will be analogous to that of Propositions \ref{transversal below threshold at ndmf-}, \ref{transversal below threshold at ndmf+}. \par

Indeed, recall from (\ref{Local d3sc coordintes representation of Rndmfpm}) that $\mathcal{R}_{\mathrm{n,dmf}, \pm} \backslash \tcocf$ can also be realized as a subset of ${^{\mathrm{d3sc}}T^{\ast}\Xd}$, so that the proofs of Propositions \ref{transversal below threshold at ndmf-}, \ref{transversal below threshold at ndmf+} can also be phrased through the coordinates in ${ ^{\mathrm{d3sc}}T^{\ast}\Xd }$. By comparing (\ref{Local d3sc coordintes representation of Rndmfpm}) with the coordinates representations of $\mathcal{R}_{\mathrm{n}, \pm} \backslash \pi_{\ff}^{-1}(o_{\mathcal{C}^{\tindex}})$ in (\ref{local coordinates representation of Rnpm}), it should thus be convincing that the proofs of \ref{below threshold Rnm}, \ref{above threshold Rnp} also follow respectively from that of Propositions \ref{transversal below threshold at ndmf-}, \ref{transversal below threshold at ndmf+}. In fact, the dynamical roles of these radial sets in their respective flows are extremely similar.

\section{Radial point estimates at $\brpm$} 
\label{global radial point estimate section}
Finally, we shall consider radial point estimates at $\brpm$. Recall that $\brm$ is a sink while $\brp$ is a source. Thus, one could say that these are the only `global' radial points, in the sense that they are not saddles, i.e., microlocal regularity \emph{everywhere} eventually flow into $\brm$, while they must all came from $\brp$. \par

Since $\brpm$ simply reduce to the standard two-body radial sets away from $\tcocf$, proving radial point estimate at $\brpm$ amounts to, in a sense the usual argument away from $\tcocf$. One then needs to simply incorporate the necessary modification near $\tcocf$ so that the argument becomes globally valid, which will be our strategy below.

\begin{proposition}[Below threshold estimate at $\brm$] \label{global below threshold}
Let $\vor = \vor_{+}$, $\vob = \vob_{+}$ be variable orders satisfying the conditions specified in \S \ref{variable order construction section}. Let $E \in \Psfo^{-\infty, 0 , -\infty, -\infty, 0, -\infty}$ be chosen such that
\begin{equation*}
\begin{gathered}
\text{all backward integral curves of $\rho_{\dmf}^{-1} \rho_{\dtsccf}^{-1} \rho_{\tcocf}^{-2} H_{p}$ starting from $\brm$} \\
\text{are either stationary or enter $\Ellss(E)$ in finite time. }
\end{gathered}
\end{equation*}
Then we can find an arbitrarily small neighborhood $U$ of $\brm$ for which the following holds: If we choose $B \in \Psf^{-\infty, 0, -\infty , 0 ,-\infty ,-\infty}$ with $\WFs(B) \subset U$, then for every $u \in \mathcal{S}'$ and any choice of $N,M,L,K, S \in \mathbb{R}$, there exists $C > 0$ such that
\begin{equation} \label{semi-fredhom estimate at brm}
\| B u \|_{ H_{\mathrm{d3sc,3co,res}}^{\ast, \vor, \ast, \ast, \vob, \ast} } \leq C ( \| Pu \|_{H_{\mathrm{d3sc,3co,res}}^{\ast, \vor + 1, \ast, \ast, \vob + 2, \ast}} + \| E u \|_{H_{\mathrm{d3sc,3co,res}}^{\ast, \vor, \ast, \ast, \vob, \ast}} + \| u \|_{ H_{\mathrm{d3sc,3co,res}}^{ N,M,L,K, \vob, S } } )
\end{equation}
in the strong sense that if the right hand of (\ref{semi-fredhom estimate at brm}) is finite, then so is the left hand side, and the estimate holds.

\end{proposition}
\begin{proof}
Let $\psi \in \mathcal{C}^{\infty}_{c}( [0,\infty) )$ be such that conditions (\ref{tangential propagation psi condition}) are satisfied. Then we can consider a $\mathcal{C}^{\infty} ( \tscX )$ extension of $\bff$ defined by
\begin{equation}  \label{global radial m 0.5}
\tbff = (1 - \psi) ( \bff ) + \psi ( \bff ) \bff
\end{equation}
so that $\tbff$ is a global defining function for $\ff$. Suppose that $x = |z|^{-1}$. Then the condition $ x = \tbff \rho_{\mf}$ also determines a global defining function $\rho_{\mf} \in \mathcal{C}(   \tscX )$ for $\mf$. It follows that $\hat{\rho}_{\mf} = \rho_{\mf} / \tbff$ can be identified as a $\mathcal{C}^{\infty} (\Xd)$ function in a small neighborhood of $\dmf$, for which it is a defining function for $\dmf$.  \par

Let $\tilde{\psi} \in \mathcal{C}^{\infty}_{c}( [0,\infty) )$ be another cut-off at $0$ such that $\tilde{\psi} = 1$ on the support of $\psi$. Then we will define
\begin{equation*}
\rho_{ \brm } = \tilde{\psi} ( \hat{\rho}_{\mf} ) \big( \psi( \bff ) \big( ( \ltau - \utaures )^2 + | \umures |_{h^{\tindex}}^2 \big) + ( \tau + \lambda )^2 \big).
\end{equation*}
Here, $\tau$ is dual to $dx/x^2$, while $|\mu|_{h}$ is the rescaled metric norm on $\partial \Xo$, where each $\mu_{j}$ is dual to $dy_{j}/x$, $j=1,..., n-1$. Notice that, as the notation suggests, $\rho_{\brm}$ defines $\brm$ quadratically on the intersection between $\psf_{\dmf} \Xd$ and $\bcv$. In particular
\begin{equation} \label{global radial m 1}
\brm = \{ \rho_{\brm} = 0 \} \cap \psf_{\dmf} \Xd \cap \bcv.
\end{equation}
Indeed, recall that we can write, at least in a neighborhood of $\brm \cap \tcocf$, that
\begin{equation*}
\begin{gathered}
\tau = \frac{1}{ ( 1 + \bff^2 )^{1/2} } \ltau + \frac{ \bff^2 }{ ( 1 + \bff^2 ) } \utaures, \\
| \mu |_{h}^2 = \frac{\bff}{ 1 + \bff^2 } ( \ltau - \utaures )^2 + | \lmu |_{h_{\tindex}}^2 + \bff^2 | \umures |_{h^{\tindex}}^2.
\end{gathered}
\end{equation*}
Thus for $\bff > 0$, the conditions $\tau = -\lambda$ in $\bcv$ implies precisely that
\begin{equation*}
\ltau = \utaures, \, \ltau = -\lambda, \, |\mu_{\tindex}| = 0, \, | \umures |_{h^{\tindex}} = 0,
\end{equation*}
though for $\bff = 0$, it implies only that of 
\begin{equation*}
\ltau = -\lambda, \, | \lmu |_{h_{\tindex}} = 0,
\end{equation*}
the latter being separately compensated by the vanishing of $( \ltau - \utaures )^2 + | \umures |_{h^{\tindex}}^2$. \par

Next, suppose we set
\begin{equation*}
\, \varphi_{1} = \psi( \rho_{\brm} ), \,  \varphi_{2} = \psi(p).
\end{equation*}
Then we can consider the symbol defined by
\begin{equation*}
a_0 = \hat{\rho}_{\mf}^{-\vor - 1/2} \tbff^{-\vob - 1} \varphi_1 \varphi_2.
\end{equation*}
Here, we are again assuming that $a_0$ is supported in a neighborhood of $\psf_{\dmf}\Xd$, which is allowed since it is the only `symbolic face' which $\brm$ intersects. Explicitly, this amounts to the omission of an additional cut-off $\psi(\hat{\rho}_{\mf})$, the derivative of which is symbolically residual on the support of $a_0$. In particular, we can assume that $a_0$ is supported in an arbitrarily small neighborhood of $\brm$ by (\ref{global radial m 1}). \par

It is straightforward to verify that
\begin{equation*}
H_{p} \tbff = \hat{\rho}_{\mf} \tbff^{3} f_0 f_1 \tilde{F}, \, H_{p} \hat{\rho}_{\mf} = 2 \hat{\rho}_{\mf}^2 \tbff^2 ( \tau - F ),
\end{equation*}
where $f_0$, $f_1$ and $\tilde{F}$ are determined by
\begin{equation*}
x_{\tindex} = f_0 x, \, \bff = f_{1} \tbff , \, \tilde{F} =  2\big( \psi'(\bff) ( \bff - 1 ) + \psi(\bff) \big) ( \ltau - \utaures ).
\end{equation*}
Here, $f_0$, $f_{1}$ are strictly positive smooth functions which are only locally defined, while $\tilde{F}$ can be made non-negative if $\bff$ is uniformly bounded by $1$. These conditions can be satisfied if the support of $\psi$ is taken to be sufficiently small. Thus, if we set $F = f_{0} f_{1} \tilde{F}$, then we have
\begin{equation} \label{global radial m 2.05}
\begin{gathered}
H_{p} \tbff^{-2 \vob - 2} = - 2 ( H_p \vob ) ( \log \tbff ) \tbff^{-2 \vob - 2} - (2 \vob + 2) F \hat{\rho}_{\mf} \tbff^{-2\vob}, \\
H_{p} \hat{\rho}_{\mf}^{-2\vor - 1} = - 2 ( H_p \vor ) ( \log \hat{\rho}_{\mf} ) \hat{\rho}_{\mf}^{-2\vor - 1} - ( 4 \vor + 2 ) ( \tau - F ) \hat{\rho}_{\mf}^{-2\vor} \tbff^{2}. 
\end{gathered}
\end{equation}
 \par

We can compute
\begin{equation}  \label{global radial m 2.1}
H_{p} a_0^2 = - \tilde{b}_0^2 + (\log) b_{1,0}^2 - \tilde{b}_{2,0}^2  + e_0
\end{equation} 
for some symbols $\tilde{b}_0$, ${b}_{1,0}$, $b_{2,0}$ and $e_0$, which we now proceed to define. We will set
\begin{equation} \label{global radial m 2.2}
\tilde{b}_0^2 = \hat{\rho}_{\mf}^{-2 \vor} \tbff^{-2\vob} \big( ( 4 \vor + 2 ) \tau + ( 2 \vob - 4 \vor ) F \big) \varphi_1^2 \varphi_2^2 \varphi_3^2. 
\end{equation} 
Note that $\tilde{b}_0^2$ can indeed be written as a square of symbol, since we can arrange the support of $\psi$ to be small enough such that $\tau \simeq - \lambda$ and $F \simeq 0$ on the support of $\tilde{b}_0^2$. Then we have
\begin{equation*}
( 4 \vor + 2 ) \tau + ( 2 \vob - 4 \vor )F \simeq -( 4 \vor + 2 ) \lambda,
\end{equation*}
which is positive by the threshold condition
\begin{equation*} \label{threshold condition brm}
\vor <  - \frac{1}{2} \ \text{at} \ \brm.
\end{equation*}
Next, we note that the sum of the logarithmic terms can be written as
\begin{equation}  \label{global radial m 2.3}
-  2 \big( (H_{p} \vor) ( \log \hat{\rho}_{\mf} ) + (H_{p} \vob) ( \log \tbff ) \big) a_0^2 = 2 (\log \tbff ) b_{1,0}^2 - b_{2,0}^2,
\end{equation}
where we have defined $b_{1,0}^2 = - ( H_{p} \vob ) a^2_0$ and $b_{2,0}^2 = 2 ( H_{p} \vor ) ( \log \hat{\rho}_{\mf} ) a_0^2$. We remark that these can be written as squares of symbols by construction of the variable orders as well as the fact that $\tbff \leq 1$. We also set 
\begin{align} \label{global radial m 2.3.1}
\begin{split}
e_0 & = 2 \lambda \hat{\rho}_{\mf}^{-2\vor - 1} \tbff^{-2 \vob - 2} \varphi_2^2 \varphi_{3}^2 \varphi_1 \psi'( \rho_{\brm} ) H_{p} \rho_{\brm} \\
& = 2 \hat{\rho}_{\mf}^{-2 \vor} \tbff^{-2\vob} \varphi_2^2 \varphi_{3}^2 \varphi_1 \psi'( \rho_{\brm} ) \sH_{p}\rho_{\brm}.
\end{split}
\end{align}
Note that we have not computed $e_0$ entirely. Indeed, this is not necessary, as all we need is $e_0$ be supported in a punctured neighborhood of $\brm$, so that we may assume a priori control there in the propagation estimate.  \par

We now consider regularization. We will set
\begin{equation}  \label{global radial m 2.1}
\phi_{s} = 1 + \frac{s}{x}, \, f_{s} =  \frac{s/x}{1+s/x}, \quad s \in [0,1].
\end{equation}
For large $K \geq 0$, we can then compute
\begin{equation*}
H_{p} \phi_{s}^{-2K} = 4 K \hat{\rho}_{\mf} \tbff^2 f_{s} \tau \phi_{s}^{-2K}.
\end{equation*}
Now, let $a_{s} = \phi_{s}^{-K} a_0$, $b_{2,s} = \phi_{s}^{-K} b_{2,0}$, $e_{s} = \phi_{s}^{-2K} e_{0}$, as well as
\begin{equation}  \label{global radial m 2.2}
b_{s}^{2} = \hat{\rho}_{\mf}^{-2 \vor} \tbff^{-2\vob} \phi_{s}^{-2K} \big(  ( 4 \vor + 2 - 4 K f_{s} ) \tau + ( 2 \vob - 4 \vor ) F  - \delta_0 \phi_{s}^{-2K} \varphi_1^2 \varphi_2^2 \varphi_3^2 \big) \varphi_1^2 \varphi_2^2 \varphi_3^2
\end{equation}
for some sufficiently small $\delta_0 > 0$. Then we have
\begin{equation} \label{global radial m 2.5}
H_{p} a_{s}^2 + \delta_0 \hat{\rho}_{\mf}^{2 \vor + 2} \tbff^{2 \vob + 2} a_{s}^4 = - b_{s}^2 + 2 (\log \tbff) b_{1,s}^2 - b_{2,s}^2 + e_{s}.
\end{equation}
Here, note that we have
\begin{equation*}
\vor < - \frac{1}{2} + K \ \text{at $\brm$}
\end{equation*}
for any choice of $K \geq 0$. Thus we may conclude for $\delta_0 > 0$ small enough that $b_{s}^2$ can indeed be written as a square of symbol as well. 
\par

Let $A_{s}$,  $B_{s}$, $B_{2,s}$, $E_{s}$ be some quantizations of $a_s$, $b_{s}$, $b_{2,s}$ and $e_{s}$ respectively. Let also $B_{1,s} \in L^{\infty}( (0,1] ; \Psf^{-\infty, \vor, -\infty , -\infty  , \vob, -\infty} )$ be chosen such that conditions (\ref{always these conditions for B1s}) are satisfied. Furthermore, let $\Lambda$ be some quantization of an elliptic extension of $\hat{\rho}_{\mf}^{\vor + 1} \tbff^{\vob+2}$ from the support of $a_0$. Then by (\ref{global radial m 2.5}) we have
\begin{equation*}
 i [ P,  A_{s}^{\ast} A_{s}   ] + \delta_{0}  ( \Lambda A_{s}^{\ast} A_{s}  ) ( \Lambda  A_{s}^{\ast} A_s )  = -  B_{s}^{\ast} B_{s} + 2 B_{1,s}^{\ast} ( \log \tbff ) B_{1,s}   - B_{2,s}^{\ast} B_{2,s}  + E_{s} + R_{s},
\end{equation*}
where the membership $R_{s} \in L^{\infty}( (0,1]_{s} ; \Psf^{-\infty, 2\vor - 1 + 2 \delta  , -\infty, -\infty, \vob , -\infty } )$ is ensured by Corollary \ref{commutator corollary}. \par

The rest of the argument is again standard. Thus, suppose that $ G,\tilde{E} \in \Psfo^{-\infty, 0, -\infty, -\infty, 0, -\infty}$ are chosen such that
\begin{equation}    \label{global radial m 2.6}
\begin{gathered}
\WFsL( \{  A_{s} \} ) \subset \Ellss( G), \, \WFsL( \{ E_{s} \} ) \subset \Ellss( \tilde{E} ) \\
\WFs( I - G ) \cap \WFsL( \{ A_s \} ) = \WFs( I - \tilde{E} ) \cap \WFsL( \{ E_s \} )  = \emptyset,
\end{gathered}
\end{equation}
then by the standard procedure, we have
\begin{align*}
\| B u \|_{ H_{\mathrm{d3sc,3co,res}}^{\ast, \vor, \ast, \ast, \vob, \ast} } & \leq C ( \| Pu \|_{H_{\mathrm{d3sc,3co,res}}^{\ast, \vor + 1, \ast, \ast, \vob + 2, \ast}} \\
& \quad + \| \tilde{E} u \|_{H_{\mathrm{d3sc,3co,res}}^{\ast, \vor, \ast, \ast, \vob, \ast}} + \| G u \|_{H_{\mathrm{d3sc,3co,res}}^{\ast, \vor, \ast, \ast, \vob, \ast} }^2 + \| u \|_{ H_{\mathrm{d3sc,3co,res}}^{ N,M,L,K, \vob, S } } ),
\end{align*}
and after finitely many iterations the $Gu$ term can be dropped. Thus, by elliptic regularity estimate, we have
\begin{equation*}  
\| B u \|_{ H_{\mathrm{d3sc,3co,res}}^{\ast, \vor, \ast, \ast, \vob, \ast} } \leq C ( \| Pu \|_{H_{\mathrm{d3sc,3co,res}}^{\ast, \vor + 1, \ast, \ast, \vob + 2, \ast}}  + \| E u \|_{H_{\mathrm{d3sc,3co,res}}^{\ast, \vor, \ast, \ast, \vob, \ast}}+ \| u \|_{ H_{\mathrm{d3sc,3co,res}}^{ N,M,L,K, \vob, S } } ).
\end{equation*} \par

Notice that since $\brm$ is a sink,  we have proved a local estimate with the following dynamical property: Assume that $\gamma$ is an integral curve of $\rho_{\dmf}^{-1} \rho_{\dtsccf}^{-1} \rho_{\tcocf}^{-2} H_{p}$ such that $\lim_{t \rightarrow \infty} \gamma(t) \in \brm$, then since $\WFs(E_0)$ is contained in a punctured neighborhood of $\brm$, $\gamma$ must enter $\WFs(E_0)$ and subsequently $\WFs(\tilde{E})$ in finite backward time. The statement for a general $E$ thus follows from elliptic regularity and principal type propagation of regularity estimates.
\end{proof}

\begin{proposition}[Above threshold estimate at $\brp$] \label{Global radial set estimate +}
Let $\vor = \vor_{+}$, $\vob = \vob_{+}$ be variable orders satisfying the conditions specified in \S \ref{variable order construction section}. Let $G \in \Psf^{-\infty, 0 , -\infty , -\infty ,0 ,-\infty}$ be chosen such that $\brm \subset \Ellss(G)$, with $\WFs(G)$ contained in a given small neighborhood of $\brp$. Then we can find an arbitrarily small neighborhood $U$ of $\brp$ for which the following holds: If we choose $B \in \Psf^{-\infty, 0 , -\infty , -\infty ,0 ,-\infty}$ with $\WFs(B) \subset U$, then for every $u \in \mathcal{S}'$ and any choice of $N,M,L,K,S \in \mathbb{R}$, there exists $C > 0$ such that
\begin{equation} \label{semi-fredhom estimate at brp}
\| B u \|_{ H_{\mathrm{d3sc,3co,res}}^{\ast, \vor, \ast, \ast, \vob, \ast} } \leq C ( \| Pu \|_{H_{\mathrm{d3sc,3co,res}}^{\ast, \vor + 1, \ast, \ast, \vob + 2, \ast}} + \| G u \|_{H_{\mathrm{d3sc,3co,res}}^{\ast, r_0, \ast, \ast, \vob, \ast}} + \| u \|_{ H_{\mathrm{d3sc,3co,res}}^{ N,M,L,K, \vob, S } } )
\end{equation}
in the strong sense that if the right hand of (\ref{semi-fredhom estimate at brp}) is finite, then so is the left hand side, and the estimate holds.
\end{proposition}

\begin{proof}
Let $\psi \in \mathcal{C}^{\infty}_{c}( [0,\infty) )$ be such that conditions (\ref{tangential propagation psi condition}) are satisfied. Let $\tilde{\psi} \in \mathcal{C}^{\infty}_{c}( [ 0, \infty) )$ be another cut-off at $0$ such that $\tilde{\psi} = 1$ on the support of $\psi$. Moreover, let $\tbff \in \mathcal{C}^{\infty}( \tscX )$ be the extension of $\bff$ defined in (\ref{global radial m 0.5}), and $\hat{\rho}_{\mf} \in \mathcal{C}^{\infty}( \Xd )$ be the defining function for $\psf_{\dmf} \Xd$ given by $\hat{\rho}_{\mf} = \rho_{\mf} / \tbff$, where $\rho_{\mf}$ is determined by $x = \rho_{\mf} \tbff$. Here $x$ is the canonical boundary defining function for $\Xo$. Then we will write
\begin{equation*}
\rho_{ \brp } =  \tilde{\psi}( \hat{\rho}_{\mf} )  \psi( \bff ) \big( ( \ltau - \utaures )^2 + | \umures |_{h^{\tindex}}^2 \big) + ( \tau - \lambda )^2 ,
\end{equation*}
where we note that $\rho_{\brp}$ defines $\brp$ quadratically on the intersection between $\bcv$ and $\psf_{\dmf} \Xd$, which can be seen in the same way as in the below threshold case. \par

Unlike in the below threshold case, here we will need to make explicit computation of $H_{p} \rho_{ \brp }$. For this, we first note that
\begin{align*}
& H_{p}  \big( ( \ltau - \utaures )^2 + |\umures|_{h^{\tindex}}^2 \big)  = 4 \hat{\rho}_{\mf} \tbff^{2} \big( \utaures \big( ( \ltau - \utaures )^2 + | \umures |_{h^{\tindex}}^2 \big) - ( \ltau - \utaures ) | \lmu |_{h_{\tindex}}^2   \big)  \\
& \qquad = 4 \lambda \hat{\rho}_{\mf} \tbff^2 \big( ( \ltau - \utaures )^2 + | \umures |_{h^{\tindex}}^2  \\
& \qquad \quad + ( \utaures - \lambda ) \big( ( \ltau - \utaures )^2 + | \umures |_{h^{\tindex}}^2 \big) - 4 ( \ltau - \utaures ) | \lmu |_{h_{\tindex}}^2 \big), 
\end{align*}
as well as 
\begin{align*}
& H_{p} ( \tau - \lambda )^2 = - 4 \hat{\rho}_{\mf} \tbff^2 ( \tau - \lambda ) | \mu |_{h}^2 \\
& \qquad = 4 \hat{\rho}_{\mf} \tbff^2 \big(  \lambda ( \tau - \lambda )^2 +  ( \tau - \lambda )^3 +  ( \tau - \lambda ) ( | \tau |^2 + | \mu |_{h}^2 - \lambda ) \big). 
\end{align*}
Moreover, we have
\begin{equation*}
H_{p} \psi( \bff ) = ( \ltau - \utaures ) \hat{\rho}_{\mf} \tbff^2 \bff \psi'(\bff). 
\end{equation*}
Thus, overall
\begin{equation} \label{global radial p 1}
H_{p} \rho_{ \brp } = 4 \lambda \hat{\rho}_{\mf} \tbff^2 ( \rho_{\brp} +  ( \tau - \lambda )^2 + F_{0} )
\end{equation}
in a small neighborhood of $\psf_{\dmf} \Xd$, where $F_0$ vanishes cubically at $\brp$. \par

We now set
\begin{equation*}
 \varphi_{1} = \psi( \rho_{\brm} ), \, \varphi_{2} = \psi(p)
\end{equation*}
and consider the symbol defined by
\begin{equation*}
a_0 = \hat{\rho}_{\mf}^{-\vor - 1/2} \tbff^{-\vob - 1} \varphi_1 \varphi_2 .
\end{equation*}
Then by using (\ref{global radial m 2.05}) and (\ref{global radial p 1}), we can compute
\begin{equation*}
H_{p} a_{0}^2 = - \tilde{b}_{0}^2  + 2 (\log \tbff) b_{1,0}^2 - b_{2,0}^2 - b_{3,0}^2,
\end{equation*}
for some symbols $b_{0}$, $b_{1,0}$, $b_{2,0}$ and $b_{3,0}$. Here, $\tilde{b}_{0}^2$ is defined by the same formula as (\ref{global radial m 2.2}), except that non-negativity is now ensured due to the threshold condition 
\begin{equation*}
\vor > - \frac{1}{2}   \ \text{at $\brp$}.
\end{equation*}
Thus in particular, $\tilde{b}_0^2$ can indeed be written as the square of a symbol. Likewise, we will define $b_{1,0}$, $b_{2,0}$ as in (\ref{global radial m 2.3}). We also define $b_{3,0}$ to be the negative of $e_0$ as in (\ref{global radial m 2.3.1}), i.e.,
\begin{equation*}
b_{3,0}^{2} = - 8 \lambda \rho_{\dmf}^{-2 \vor} \tbff^{-2 \vob} \varphi_2^2 \varphi_3^2 \varphi_1 \psi'( \rho_{\brp} )  ( \rho_{\brp} +  ( \tau - \lambda )^2 + F_{0} ). 
\end{equation*}
Now, on the support of $\psi'( \rho_{\brp} )$, we know that $\rho_{\brp}$ is bounded below by some positive constant. Meanwhile, since $F_0$ vanishes cubically at $\brp$, we know that $|F_0| \leq C \rho_{\brp}^{3/2}$ for some $C > 0$. Finally, $(\tau - \lambda)^2$ is non-negative. Then since $\psi' \leq 0$, we can conclude that $b_{3,0}^2$ is indeed non-negative, and can be written as the square of a symbol due to (\ref{tangential propagation psi condition}). \par

Let $\phi_{s}$, $f_{s}$, $s \in [0,1]$ be defined as in (\ref{global radial m 2.1}) and set $b_{j,s} = \phi_{s}^{-K} b_{j,s}$, $j =,1,2,3$. Furthermore, let $b_{s}^2$ be defined as in (\ref{global radial m 2.2}), then the positivity of $b_{s}^2$ forces the amount of regularization $K$ to be restricted, i.e., we need to choose $K$ such that
\begin{equation} \label{global radial p 2}
\vor > - \frac{1}{2} + K \ \text{at $\brp$}.
\end{equation}
Assuming henceforth that (\ref{global radial p 2}) is satisfied. Then we have
\begin{equation*}
H_{p} a_{s}^2 + \delta_0 \hat{\rho}_{\mf}^{2 \vor + 2} \tbff^{2 \vob + 2} a_{s}^4 = - b_{s}^2 + 2 (\log \tbff) b_{1,s}^2 - b_{2,s}^2 - b_{3,s}^2.
\end{equation*}
Let $A_{s}$, $B_{s}$, $B_{j,s}$ be some quantizations of $a_{s}$, $b_{s}$, $b_{j,s}$ respectively for $j=2,3$, and moreover let $B_{1,s} \in L^{\infty}( (0,1] ; \Psf^{-\infty, \vor, -\infty , -\infty  , \vob, -\infty} )$ be chosen such that conditions (\ref{always these conditions for B1s}) are satisfied. Let $\Lambda$ be some quantization of an elliptic extension of $\hat{\rho}_{\mf}^{\vor + 1} \tbff^{\vob+2}$ from the support of $a_0$. Then we have
\begin{equation*}
 i [ P,  A_{s}^{\ast} A_{s}   ] + \delta_{0}  ( \Lambda A_{s}^{\ast} A_{s}  ) ( \Lambda  A_{s}^{\ast} A_s )  = -  B_{s}^{\ast} B_{s} + 2 B_{1,s}^{\ast} ( \log \tbff ) B_{1,s}   - B_{2,s}^{\ast} B_{2,s} - B_{3,s}^{\ast} B_{3,s} + R_{s},
\end{equation*}
where $R_{s} \in L^{\infty}( [0,1)_{s} ; \Psf^{-\infty, 2\vor - 1 + 2 \delta  , -\infty, -\infty, \vob , -\infty } )$ by Corollary \ref{commutator corollary}. \par

One mild issue here is that since $K$ cannot be chosen arbitrarily large, in order to carry out the usual procedure, we then need to introduce an additional regularization argument to show that
\begin{equation*}
i \langle [ P, A_{s}^{\ast} A_{s} ]u ,u \rangle_{L^2} = 2 \mathrm{Im} \langle Pu , A_{s}^{\ast} A_s u \rangle_{L^2}
\end{equation*}
with the point being that integration by part can be carried out. Such a regularization will be independent of the existing ones. See for examples \cite[Chapter 11, \S 2]{PeterNotes} or \cite[\S 5.4.7]{AndrasBook} for details. \par

However, assuming this is done, and suppose that $\tilde{G}, \tilde{E} \in \Psfo^{-\infty, 0, -\infty, -\infty, 0, -\infty}$ are chosen such that conditions (\ref{global radial m 2.6}) are satisfied with $\tilde{G}$ in places of $G$. Then by the standard procedure, we must have
\begin{equation*}
\| B u \|_{ H_{\mathrm{d3sc,3co,res}}^{\ast, \vor, \ast, \ast, \vob, \ast} } \leq C ( \| Pu \|_{H_{\mathrm{d3sc,3co,res}}^{\ast, \vor + 1, \ast, \ast, \vob + 2, \ast}} + \| \tilde{G} u \|_{ H_{\mathrm{d3sc,3co,res}}^{ \ast, \vor - 1/2 + 2\delta, \ast, \ast, \vob, \ast } } + \| u \|_{ H_{\mathrm{d3sc,3co,res}}^{ N,M,L,K, \vob, S } } ).
\end{equation*}
This estimate cannot be iterated arbitrarily many times due to the regularization issue mentioned above. But if $\vor-1/2 + 2\delta \leq r_0$, then we can get (\ref{semi-fredhom estimate at brp}) via the trivial inclusion estimate and elliptic regularity. Otherwise, we can iterate further and obtain (\ref{semi-fredhom estimate at brp}) nonetheless. This proves the required estimate.
\end{proof}

\section{Global estimate with symbolic decay} \label{almost semi-Fredholm estimates with symbolic decay section}
In \S \S \ref{principal type propagation section}--\ref{global radial point estimate section}, we have provided a complete description of microlocal propagation for solutions to $Pu = f${\ep}albeit only along the symbolic characteristic set $\Sigma_{\sigma}$. Moreover, these propagation estimates were stated only in the forward direction. It is easy to see that one can also obtain the analogous propagation results in the backward direction. However, in that case, the roles of $\vor_{+}$, $\vol_{+}$, $\vob_{+}$ and $\vor_{-}$, $\vol_{-}$, $\vob_{-}$ must be switched. \par

More concretely, the principal type propagation estimate in the backward direction can be proved by applying Proposition \ref{principal type propagation second microlocal version} to $-P$ instead of $P$, in which case every `backward' in the statements of the proposition should be replaced by `forward'. In the cases of the the radial point estimates (i.e., Propositions \ref{Proposition Rndm-}--\ref{Global radial set estimate +}), one observes that the natures of the radial points, when viewed as equilibria for the flow of $\rho_{\dmf}^{-1} \rho_{\dtsccf}^{-1} \rho_{\tcocf}^{-2}H_{p}$, are switched whenever one replaces every `$+$' by `$-$'. We refer the readers to \S \ref{subsection characterization for the flow of the second microlocal dynamic} for details.  \par

The first step in obtaining the proposed Fredholm estimates would be to patch the aforementioned propagation estimates together. More precisely, we will use principal type propagation estimates to connect the various radial point estimates in $\bcv$, keeping careful track the behaviors of the flow of $\rho_{\dmf}^{-1} \rho_{\dtsccf}^{-1} \rho_{\tcocf}^{-2} H_p$ in the meanwhile. On the other hand, standard microlocal elliptic regularity estimate can be applied away from $\Sigma_{\sigma}$. \par

The end result is the following:
\begin{proposition} \label{fredholm estimate with symbolic decay}
Let $m_{\pm}$,  $\vor_{\pm}$, $\vol_{\pm}$, $\vob_{\pm}$, $s_{\pm}$ be variable orders satisfying the conditions specified in \S \ref{variable order construction section}. Then for every $M, S \in \mathbb{R}$ and sufficiently small $\delta > 0$, we have
\begin{align} \label{almost semi-Fredholm estimates with symbolic decay}
\begin{split}
 \| u \|_{ H_{\mathrm{d3sc,3co,res}}^{ m_{ \pm }, \vor_{ \pm }, \vol_{ \pm }, \vor_{ \pm } + \vol_{ \pm }, \vob_{ \pm }, s_{ \pm }  } } & \leq  C ( \| P u \|_{ H_{\mathrm{d3sc,3co,res}}^{ m_{ \pm } - 2, \vor_{ \pm } + 1, \vol_{ \pm }, \vor_{ \pm } + \vol_{\pm  } + 1, \vob_{ \pm } + 2, s_{ \pm } - 2 } } \\
 & \quad + \| u \|_{H_{\mathrm{d3sc,3co,res}}^{M, \vor_{\pm} - \delta, \vol_{\pm} , \vor_{\pm} + \vol_{\pm} - \delta, \vob_{\pm} ,S} } )
\end{split}
\end{align}
in the strong sense that if the right hand sides of (\ref{almost semi-Fredholm estimates with symbolic decay}) is finite, then so are the left hand sides, and the estimate holds.
\end{proposition}
We remark that the presence of only a $\delta$-gain at $\psf_{\dmf}\Xd$ and $\dtsccf$ in the last term of (\ref{almost semi-Fredholm estimates with symbolic decay}) is due to the constraints at the above threshold radial estimates.
\begin{proof}
Fix $u \in \mathcal{S}'$ for which the right hand sides of (\ref{almost semi-Fredholm estimates with symbolic decay}) are finite. For definiteness, we will focus on the case where the orders are $m_{+}$, $\vor_{+}$, $\vol_{+}$, $\vob_{+}$, $s_{+}$. We also drop the plus signs from now on. Moreover, since (\ref{almost semi-Fredholm estimates with symbolic decay}) is weaker than the microlocal elliptic regularity estimates, which can be applied to $P$ away from $\Sigma_{\sigma}$, we will henceforth focus only on the situation in a neighborhood of $\bcv$. Additionally, we will assume that there is only one $\mathcal{C}_{\tindex}$, as the general case follows easily from propagation along the free flow and this local discussion. \par

We now proceed to patch the propagation estimates together. For brevity, we will only verbally explain why the flow of $\rho_{\dmf}^{-1} \rho_{\dtsccf}^{-1} \rho_{\tcocf}^{-2} H_{p}$ in $\bcv$ indeed allows this, for writing down the actual estimates would be extremely cumbersome task. We refer to \cite[\S 3.2]{AndrewNSC} for how this is done carefully in the two-body setting. Furthermore, our discussion below will be from the viewpoint of the forward flow of $\rho_{\dmf}^{-1} \rho_{\dtsccf}^{-1} \rho_{\tcocf}^{-2} H_{p}$, for the backward flow will only be relevant in the case where the variable orders have the minus signs. We also refer to Figure \ref{figure 6} for a graphical illustration.  \par

\begin{figure} 
\centering
\includegraphics{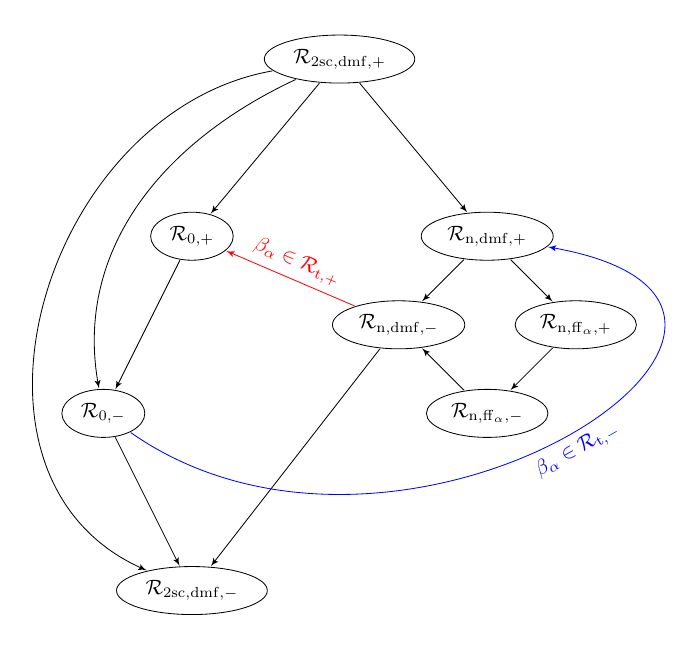}
\caption{A schematic flowchart showing trajectories of all the possible forward integral curves $\gamma$ of $\rho_{\dmf}^{-1} \rho_{\dtsccf}^{-1} \rho_{\tcocf}^{-2} H_{p}$ in $\Sigma_{\sigma}$. The microlocal regularity of $u$ -- assuming $Pu$ is sufficiently regular, can be propagated along such curves over $\Sigma_{\sigma}$ from $\mathcal{R}_{\mathrm{2sc,dmf}, +}$ to $\mathcal{R}_{\mathrm{2sc,dmf}, -}$, and we can assume a priori microlocal regularity at $\mathcal{R}_{\mathrm{2sc,dmf},+}$. Note that the red arrow represents those $\gamma$ along which $\beta_{\tindex} \in \mathcal{R}_{\mathrm{t},+}$, while the blue arrow represents those $\gamma$ along which $\beta_{\tindex} \in \mathcal{R}_{\mathrm{t},-}$. Thus, by monotonicity of the $\rho_{\dmf}^{-1} \rho_{\dtsccf}^{-1} \rho_{\tcocf}^{-2} H_{p}$-flow in the free variables, once the microlocal regularity of $u$ has been propagated along the blue arrow, it can no longer be propagated along the red arrow. In particular, microlocal propagation of regularity following the above diagram does not give rise to a circular argument.} 
\label{figure 6}
\end{figure}

We will start propagating from $\brp$. Recall that this is a global source for the flow of $\rho_{\dmf}^{-1} \rho_{\dtsccf}^{-1} \rho_{\tcocf}^{-2}H_p$. Moreover, by Proposition \ref{Global radial set estimate +}, we can always conclude microlocal regularity for $u$ near $\brp$, in the sense that $u$ microlocally belongs to $H_{\mathrm{d3sc,3co,res}}^{m,\vor,\vol,\vor+\vol,\vob,s}$ in a neighborhood of $\brp$. Our goal now is to propagate this amount of microlocal regularity for $u$ to the whole of $\bcv$, though by principal type propagation, it would be enough if we can conclude microlocal regularity at the radial sets. \par

To this end, let $\gamma$ denote an arbitrary forward integral curve of $\rho_{\dmf}^{-1} \rho_{\dtsccf}^{-1} \rho_{\tcocf}^{-2} H_{p}$. For brevity, henceforth we will write 
\begin{equation*}
\gamma_{-} = \lim_{t \rightarrow -\infty} \gamma(t), \ \gamma_{+} = \lim_{t \rightarrow + \infty} \gamma(t).
\end{equation*}
Note that if $\gamma$ begins at $\brp$ and $\gamma$ does not intersect neither $\dtsccf$ nor $\tcocf$, then $\gamma$ can be viewed as a usual forward integral curve on $^{\mathrm{sc}}T^{\ast}_{\partial \Xo} \Xo$ as in the two-body case. Thus $\gamma$ must end at $\brm$. Then principal type propagation estimate (i.e., Proposition \ref{principal type propagation second microlocal version}) allows us to propagate regularity along $\gamma$ into a punctured neighborhood of $\brp$.  \par

Still assuming that $\gamma$ begins at $\brp$, but now $\gamma$ has non-empty intersection with either $\dtsccf$ or $\tcocf$. Then there are a few cases to be looked at, which depend on whether $\gamma$ is contained in one of the boundary faces or not. Note that a complete description of these cases have been covered throughly in \S \ref{subsection characterization for the flow of the second microlocal dynamic}, and we expect the readers to be familiar with the contents therein for the following discussions. \par

We first examine the cases where $\gamma$ has non-empty intersection with $\tcocf$. Suppose that $\gamma$ is contained in $\tcocf$ but not $\dtsccf$. Thus $\gamma$ is contained in $\tcocf \cap \psf_{\dmf} \Xd$. Then the only possibilities for the complete trajectory of $\gamma$ are:
\begin{enumerate}
\item $\gamma_{-} \in \brp$, $\gamma_{+} \in \brm$.
\item $\gamma_{-} \in \brp$, $\gamma_{+} \in \mathcal{R}_{\mathrm{n},\dmf, +}(\beta_{\tindex})$, $\beta_{\tindex} \in \Sigma_{\mathrm{t}}$. 
\item $\gamma_{-} \in  \brp$, $\gamma_{+} \in \mathcal{R}_{0, +}$. 
\end{enumerate} \par

Observe that if $\beta_{\tindex} \in \Sigma_{\mathrm{t}} \backslash \mathcal{R}_{\mathrm{t},-}$, then case (2) exhausts all the possible $\gamma$ for which $\gamma_{+}  \in \mathcal{R}_{\mathrm{n},\dmf,+}(\beta_{\tindex})$. Indeed, by definition, we are considering all $\gamma \subset \tcocf$ with this property. On the other hand, by Proposition \ref{proposition characterization of the integral curves which live on dtsccf} we know that no $\gamma \subset \dtsccf$ can satisfy $\gamma_{+} \in \brm$. Thus, it follows from Proposition \ref{Proposition Rndm-} that we can conclude microlocal regularity for $u$ at $\bigcup_{\beta_{\tindex} \in \Sigma_{\mathrm{t}} \backslash \mathcal{R}_{\mathrm{t},-} }\mathcal{R}_{\mathrm{n},\dmf,+}(\beta_{\tindex})$. \par

Now, suppose we start from case (2) above with $\beta_{\tindex} \in \Sigma_{\mathrm{t}} \backslash \mathcal{R}_{\mathrm{t},-}$. Then we can continue to propagate microlocal regularity for $u$ out of each $\mathcal{R}_{\mathrm{n},\dmf,+}(\beta_{\tindex})$. Note that the free variables must be constants along the flow at $\dtsccf$. Then we have the following possibilities:
\begin{enumerate}[label={}]
\item[(4)] $\gamma_{-} \in \mathcal{R}_{\mathrm{n},\dmf,+}(\beta_{\tindex})$, $\gamma_{+} \in \mathcal{R}_{\mathrm{n}, \mathrm{dmf}, -}(\beta_{\tindex})$, $\beta_{\tindex} \in \Sigma_{\mathrm{t}} \backslash \mathcal{R}_{\mathrm{t},-}$ by Proposition \ref{proposition characterization of the integral curves which live on dtsccf}(1).
\item[(5)] $ \gamma_{-} \in \mathcal{R}_{\mathrm{n},\dmf,+}(\beta_{\tindex})$, $ \gamma_{+} \in \mathcal{R}_{\mathrm{n}, \dff, +}(\beta_{\tindex})$, $\beta_{\tindex} \in \Sigma_{\mathrm{t}} \backslash \mathcal{R}_{\mathrm{t},-}$ by Proposition \ref{proposition characterization of the integral curves which live on dtsccf}(3). 
\end{enumerate}
In fact, Proposition \ref{proposition characterization of the integral curves which live on dtsccf} implies that case (5) exhausts all the possible $\gamma$ for which $\gamma_{+} \in \mathcal{R}_{\mathrm{n},\dff,+}(\beta_{\tindex})$, $\beta_{\tindex} \in \Sigma_{\mathrm{t}} \backslash \mathcal{R}_{\mathrm{t},-}$. Thus by Proposition \ref{Propositive estimate at Rndff+}, we can conclude microlocal regularity for $u$ at every such $\mathcal{R}_{\mathrm{n},\dff,+}(\beta_{\tindex})$.
\par

We next take the conclusion of case (5) to be our starting point, and then propagate microlocal regularity for $u$ out of $\mathcal{R}_{\mathrm{n},\dff,+}(\beta_{\tindex})$ for every $\beta_{\tindex} \in \Sigma_{\mathrm{t}} \backslash \mathcal{R}_{\mathrm{t},-}$. But the only possibility in this case is that
\begin{enumerate}[label={}]
\item[(6)] $\gamma_{-} \in \mathcal{R}_{\mathrm{n}, \dff, +}(\beta_{\tindex})$, $\gamma_{+} \in \mathcal{R}_{\mathrm{n},\dff,-}(\beta_{\tindex})$, $\beta_{\tindex} \in  \Sigma_{\mathrm{t}} \backslash \mathcal{R}_{\mathrm{t},-}$ by Proposition \ref{proposition characterization of the integral curves which live on dtsccf}(2).
\end{enumerate} 
In fact, Proposition \ref{proposition characterization of the integral curves which live on dtsccf} implies that case (6) exhausts all the possible $\gamma$ for which $\gamma_{+} \in \mathcal{R}_{\mathrm{n},\dff,-}(\beta_{\tindex})$, $\beta_{\tindex} \in \Sigma_{\mathrm{t}} \backslash \mathcal{R}_{\mathrm{t},-}$. Thus by Proposition \ref{below threshold Rndff- proposition}, we can conclude microlocal regularity for $u$ at every such $\mathcal{R}_{\mathrm{n},\dff,-}(\beta_{\tindex})$. \par

From the conclusion of case (6), we can further propagate microlocal regularity for $u$ out of $\mathcal{R}_{\mathrm{n},\dff,-}(\beta_{\tindex})$, $\beta_{\tindex} \in \Sigma_{\mathrm{t}} \backslash \mathcal{R}_{\mathrm{t},-}$. But then there is again only one possible behavior for the trajectory of $\gamma$, namely
\begin{enumerate}[label={}]
\item[(7)] $\gamma_{-} \in \mathcal{R}_{\mathrm{n},\dff,-}(\beta_{\tindex})$, $\gamma_{+} \in \mathcal{R}_{\mathrm{n},\dmf,-}(\beta_{\tindex})$, $\beta_{\tindex} \in \Sigma_{\mathrm{t}} \backslash \mathcal{R}_{\mathrm{t},-}$ by Proposition \ref{proposition characterization of the integral curves which live on dtsccf}(4).
\end{enumerate}
Moreover, Proposition \ref{proposition characterization of the integral curves which live on dtsccf} implies that the combination of cases (4) and (7) above exhaust all the possible $\gamma$ for which $\gamma_{+} \in \mathcal{R}_{\mathrm{n},\dmf,-}(\beta_{\tindex})$, $\beta_{\tindex} \in \Sigma_{\mathrm{t}} \backslash \mathcal{R}_{\mathrm{t},-}$. Then by Proposition \ref{Proposition Rndm-}, we can conclude microlocal regularity for $u$ at every such $\mathcal{R}_{\mathrm{n},\dmf,-}(\beta_{\tindex})$ as well. \par

The above discussion allows us to further propagate microlocal regularity for $u$ out of $\mathcal{R}_{\mathrm{n},\dmf,-}(\beta_{\tindex})$, $\beta_{\tindex} \in \Sigma_{\mathrm{t}} \backslash \mathcal{R}_{\mathrm{t},-}$. Suppose that $\gamma$ begins at such a $\mathcal{R}_{\mathrm{n},\dmf,-}(\beta_{\tindex})$. Then the only possibilities for the complete trajectory of $\gamma$ are now: 
\begin{enumerate}[label={}]
\item[(8)] $\gamma_{-} \in \mathcal{R}_{\mathrm{n},\dmf,-}(\beta_{\tindex})$, $\gamma_{+} \in \brm$, $\beta_{\tindex} \in \Sigma_{\mathrm{t}} \backslash \mathcal{R}_{\mathrm{t},-}$. 
\item[(9)] $\gamma_{-} \in \mathcal{R}_{\mathrm{n},\dmf,-}(\beta_{\tindex})$, $\gamma_{+} \in \mathcal{R}_{0,+}$, $\beta_{\tindex} \in \mathcal{R}_{\mathrm{t},+}$ by Proposition \ref{proposition characterization of the integral curves which converge to R0}(4).
\end{enumerate}
Indeed, one simply need to rule out the cases where $\gamma_{-} \in \mathcal{R}_{\mathrm{n}, \dmf, -}(\beta_{\tindex})$, $\gamma_{+} \in \mathcal{R}_{0,-}$. This is impossible if $\beta_{\tindex} \in \Sigma_{\mathrm{t}} \backslash \mathcal{R}_{\mathrm{t},-}$, which is what we are assuming here. Since otherwise we would have contradiction with the evolution in the free variables. \par

In fact, we are particularly interested in case (9), since when combined with case (3) above, Proposition \ref{proposition characterization of the integral curves which converge to R0}(4) shows that we have exhausted all the possible $\gamma$ for which $\gamma_{+} \in \mathcal{R}_{0,+}$. Thus by Proposition \ref{radial point estimate at R0p}, we can conclude microlocal regularity for $u$ at $\mathcal{R}_{0,+}$. \par

Continuing to propagate microlocal regularity for $u$ from $\mathcal{R}_{0,+}$, the only possibility is
\begin{enumerate}[label={}]
\item[(10)] $\gamma_{-} \in \mathcal{R}_{0,+}$, $\gamma_{+} \in \mathcal{R}_{0,-}$ by Proposition \ref{proposition characterization of the integral curves which converge to R0}(3).
\end{enumerate}
We remark that Proposition \ref{proposition characterization of the integral curves which converge to R0}(3) in particular rules out the possibility where $\gamma_{-} \in \mathcal{R}_{0,+}$ and $\gamma_{+} \in \brm$. \par
At this point, it is also of interest to consider the case where $\gamma$ begins at $\brp$, has non-empty intersection with $\tcocf$, does not intersect $\dtsccf$, and is not contained in $\tcocf$. Then by Proposition \ref{proposition characterization of the integral curves which converge to R0}(1), the only possibility is that
\begin{enumerate}[label={}]
\item[(11)] $\gamma_{-} \in \brp$, $\gamma_{+} \in \mathcal{R}_{0,-}$.
\end{enumerate}
In fact, by Proposition \ref{precise description of integral flow near ff}(1)(a), this translates to the three-body setting as those $\gamma$ given by the lift of some integral curve of $\sH_{p,\mf}$, say $\tilde{\gamma}$, such that $\tilde{\gamma}$ intersects ${^{\mathrm{3sc}}T^{\ast}}_{\partial \ff} X$ after traveling a full distance of $\pi$. In particular, again by Proposition \ref{proposition characterization of the integral curves which converge to R0}(1), case (11) exhausts all the possible $\gamma$ for which $\gamma_{+} \in \mathcal{R}_{0,-}$. Thus by Proposition \ref{radial point estimate at R0m}, we can conclude microlocal regularity for $u$ at $\mathcal{R}_{0,-}$. \par

We now focus on microlocal propagation from $\mathcal{R}_{0,-}$. Suppose that $\gamma$ begins at $\mathcal{R}_{0,-}$, then the possibilities for the trajectory of $\gamma$ are:
\begin{enumerate}[label={}]
\item[(12)] $\gamma_{-} \in \mathcal{R}_{0,-}$, $\gamma_{+} \in \brm$. 
\item[(13)] $\gamma_{-} \in \mathcal{R}_{0,-}$, $\gamma_{+} \in \mathcal{R}_{\mathrm{n},\dmf,+}(\beta_{\tindex})$, $\beta_{\tindex} \in \mathcal{R}_{\mathrm{t},-}$.
\end{enumerate}
By Proposition \ref{proposition characterization of the integral curves which converge to dtsccf}(1) and Proposition \ref{proposition characterization of the integral curves which converge to R0}(5), it should be clear that the combination of case (2) above with case (13) exhausts all the possible $\gamma$ for which $\gamma_{+} \in \mathcal{R}_{\mathrm{n},\dmf,+}(\beta_{\tindex})$ where $\beta_{\tindex} \in \mathcal{R}_{\mathrm{t},-}$. Thus by Proposition \ref{Proposition Rndm+} once again, we can conclude microlocal regularity for $u$ at all such $\mathcal{R}_{\mathrm{n},\dmf,+}(\beta_{\tindex})$ as well. \par

It follows that we can propagate microlocal regularity for $u$ out of $\mathcal{R}_{\mathrm{n},\dmf,+}(\beta_{\tindex})$ following steps (4)--(7) verbatim, except we now have $\beta_{\tindex} \in \mathcal{R}_{\mathrm{t},-}$ instead of $\beta_{\tindex} \in \Sigma_{\mathrm{t}} \backslash \mathcal{R}_{\mathrm{t},-}$. Thus, we can conclude microlocal regularity for $u$ at all $\mathcal{R}_{\mathrm{n},\dmf,\pm}(\beta_{\tindex})$ and $\mathcal{R}_{\mathrm{n}, \dff, \pm}(\beta_{\tindex})$ where $\beta_{\tindex} \in \mathcal{R}_{\mathrm{t},-}$. Moreover, if $\gamma$ begins at $\mathcal{R}_{\mathrm{n},\dmf,-}(\beta_{\tindex})$, $\beta_{\tindex} \in \mathcal{R}_{\mathrm{t},-}$, then the only possibility is that
\begin{enumerate}[label={}]
\item[(14)] $\gamma_{-} \in \mathcal{R}_{\mathrm{n},\dmf,-}(\beta_{\tindex})$, $\gamma_{+} \in \brm$, $\beta_{\tindex} \in \mathcal{R}_{\mathrm{t},-}$.
\end{enumerate} 
Indeed, this must be the case, for otherwise there would again be contradiction with evolution in the free variables.
\par

Finally, we consider the cases where $\gamma$ begins at $\brp$, does not intersect $\tcocf$, is not contained in $\dtsccf$, and has non-empty intersection with $\dtsccf$. Then the only possibility for the trajectory of $\gamma$ is that
\begin{enumerate}[label={}]
\item[(15)] $\gamma_{-} \in \brp $, $\gamma_{+} \in \mathcal{R}_{\mathrm{n},\dmf,+}(\beta_{\tindex})$, $\beta_{\tindex} \in \Sigma_{\mathrm{n}}$.
\end{enumerate} 
We can then go through the strategy outlined in steps (4)--(7), except that we now assume $\beta_{\tindex} \in \Sigma_{\mathrm{n}}$. Moreover, we will use Propositions \ref{transversal below threshold at ndmf-}--\ref{transversal above threshold Rndff+} in the arguments. Then the conclusion is that microlocal regularity for $u$ can now be deduced at $\mathcal{R}_{\mathrm{n},\dmf,\pm}(\beta_{\tindex})$, $\mathcal{R}_{\mathrm{n},\dff, \pm}(\beta_{\tindex})$ where $\beta_{\tindex} \in \Sigma_{\mathrm{n}}$. Moreover, if $\gamma$ begins at such a $\mathcal{R}_{\mathrm{n},\dmf, -}(\beta_{\tindex})$, then by viewing $\gamma$ as the lift of some integral of $\sH_{p,\mf}$, it should be clear that the only possibility is
\begin{enumerate}[label={}]
\item[(16)] $\gamma_{-} \in \mathcal{R}_{\mathrm{n},\dmf,-}(\beta_{\tindex})$, $\gamma_{+} \in \brm$, $\beta_{\tindex} \in \Sigma_{\mathrm{n}}$. 
\end{enumerate} \par
It remains to conclude microlocal regularity for $u$ at $\brm$. For this, it suffices to note that cases (1), (8), (10), (12), (14) and (16) together exhaust all the possible $\gamma$ which can end at $\brm$. Thus the conclusion follows from Proposition \ref{global below threshold}.
\end{proof}

\section{Decay estimate at the corner face} \label{decay at the corner face section}
By Proposition \ref{fredholm estimate with symbolic decay}, thus far we have managed to prove the required Fredholm estimates with an error term with better decay in every symbolic sense. In this section, we will further improve the remainder term of (\ref{almost semi-Fredholm estimates with symbolic decay}) at the corner face $\cf$. \par

To be precise, we will show that
\begin{proposition} \label{decay at the corner face proposition}
Let $m_{\pm}$, $\vor_{\pm}$, $\vol_{\pm}$, $\vob_{\pm}$, $s_{\pm}$ be the orders constructed in \S \ref{variable order construction section}. Then for sufficiently negative $M, S  < 0$ and arbitrarily small $\delta > 0$, we have
\begin{align} \label{corner face decay equation 1}
\begin{split}
 \| u \|_{ H_{\mathrm{d3sc,3co,res}}^{ m_{ \pm }, \vor_{ \pm }, \vol_{ \pm }, \vor_{ \pm } + \vol_{ \pm }, \vob_{ \pm }, s_{ \pm }  } } & \leq  C ( \| P u \|_{ H_{\mathrm{d3sc,3co,res}}^{ m_{ \pm } - 2, \vor_{ \pm } + 1, \vol_{ \pm }, \vor_{ \pm } + \vol_{\pm  } + 1, \vob_{ \pm } + 2, s_{ \pm } - 2 } } \\
 & \quad + \| u \|_{H_{\mathrm{d3sc,3co,res}}^{M, \vor_{\pm} - \delta, \vol_{\pm} , \vor_{\pm} + \vol_{\pm} - \delta , \vob_{\pm} - \delta ,S} } )
\end{split}
\end{align}
for all $u \in H_{\mathrm{d3sc,3co,res}}^{m_{\pm}, \vor_{\pm}, \vol_{\pm}, \vor_{\pm} + \vol_{\pm}, \vob_{\pm}, s_{\pm}}$ such that $Pu \in H_{\mathrm{d3sc,3co,res}}^{ m_{ \pm } - 2, \vor_{ \pm } + 1, \vol_{ \pm }, \vor_{ \pm } + \vol_{\pm  } + 1, \vob_{ \pm } + 2, s_{ \pm } - 2 }$.
\end{proposition}

For definiteness,  we will again focus on the case when the orders are $m_{+}$, $\vor_{+}$, $\vol_{+}$, $\vob_{+}$, $s_{+}$, and also drop the plus signs from now on. Moreover, we will continue to assume that $\mathcal{C} = \mathcal{C}_{\tindex}$ has just one element. It will be obvious to the readers that this calculation can be extended naturally to the general case. Indeed, the main ingredient of the proof will be a commutator calculation for the indicial operators at $\cf$. By disjointedness, this calculation can be made at each $\cf$ simultaneously, and the end result will be the same. See \cite{AndrasThesis} for how this kind of calculations work in the case where $\mathrm{ff}$ is more general in the three-body setting (i.e., not just given by a disjoint union of $\ff$, where each $\ff$ is the lift of some $\mathcal{C}_{\tindex}$).

To prove (\ref{corner face decay equation 1}), we will let $Q \in \Psfo^{-\infty, 0, 0 ,0 ,0, -\infty}$ be such that $\hat{N}_{\cf}(Q) = q_{\tindex} \in \mathcal{C}^{\infty}_{c}( \mathcal{C}_{\tindex} \times \mathbb{R}^{n_{\tindex}} )$ is supported in an arbitrarily small neighborhood of $\Sigma_{\mathrm{t}}$, and is in particular identically $1$ near $\Sigma_{\mathrm{t}}$. Then starting from the remainder term in (\ref{almost semi-Fredholm estimates with symbolic decay}), we can write
\begin{equation} 
\label{corner face decay equation 2}
 \| u \|_{H_{\mathrm{d3sc,3co,res}}^{ M, \vor - \delta , \vol, \vor + \vol - \delta, \vob, S }} \leq C (  \| Q u \|_{H_{\mathrm{d3sc,3co,res}}^{ M, \vor - \delta , \vol , \vor + \vol - \delta , \vob, S }}  + \| ( I - Q ) u \|_{H_{\mathrm{d3sc,3co,res}}^{ M, \vor - \delta , \vol, \vor + \vol - \delta, \vob, S }} ).
\end{equation}
However, recall that $P \in \Psfo^{2,0,0, 0 ,0,2}$ is elliptic on the support of $1- q_{\tindex}$ at $\cf$, so that by elliptic regularity (in particular Proposition \ref{proposition microlocal elliptic regularity estimate chapter 1 thesis}, part (3)), we have
\begin{equation*}
\| ( I - Q ) u \|_{H_{\mathrm{d3sc,3co,res}}^{M, \vor - \delta , \vol , \vor + \vol - \delta , \vob, S}} \leq C ( \| P u \|_{H_{\mathrm{d3sc,3co,res}}^{M-2, \vor - \delta, \vol  , \vor + \vol - \delta , \vob, S - 2}} + \| u \|_{H_{\mathrm{d3sc,3co,res}}^{M, \vor - \delta , \vol , \vor + \vol - \delta , F, S}} ),
\end{equation*}
where $F < 0$ is arbitrarily negative, so that (\ref{corner face decay equation 2}) can in fact be improved to
\begin{align} \label{corner face decay equation 2.4}
\begin{split}
\| u \|_{H_{\mathrm{d3sc,3co,res}}^{ M, \vor - \delta , \vol, \vor + \vol - \delta, \vob, S }}  & \leq C( \| Pu \|_{H_{\mathrm{d3sc,3co,res}}^{M-2, \vor - \delta , \vol, \vor + \vol - \delta, \vob  , S - 2 }} \\
& \quad + \| Q u \|_{H_{\mathrm{d3sc,3co,res}}^{ M, \vor - \delta, \vol  , \vor + \vol -  \delta, \vob, S }}  +  \| u \|_{H_{\mathrm{d3sc,3co,res}}^{M, \vor - \delta, \vol   , \vor + \vol - \delta, F, S }} ). 
 \end{split}
\end{align} \par

It remains the estimate the norm of $Qu$ appearing in (\ref{corner face decay equation 2.4}). For this, we will consider a further spatial decomposition. Namely, let $\varphi_{\dmf} \in \mathcal{C}^{\infty}( \Xd )$ be a cut-off function at $\dmf$ such that $\varphi_{\dmf}$ is identically $1$ near $\dmf$, and let $\varphi_{\dff} = 1 - \varphi_{\dmf}$. We will compute
\begin{equation}   \label{corner face decay equation 2.5}
\| Qu \|_{H_{\mathrm{d3sc,3co,res}}^{M, \vor - \delta , \vol, \vor + \vol - \delta, \vob, S}} \leq \| \varphi_{\dmf} Q u \|_{H_{\mathrm{d3sc,3co,res}}^{M,  \vor - \delta , \vol, \vor + \vol - \delta, \vob, S}} + \| \varphi_{\dff} Q \|_{H_{\mathrm{d3sc,3co,res}}^{M, \vor - \delta , \vol , \vor + \vol - \delta, \vob ,  S}} 
\end{equation}
and then estimate the norms of $\varphi_{\dmf} Q u$ and $\varphi_{\dff} Q u$ using different strategies.

\subsection{Decay estimate near the main face}
We first consider the situation near $\dmf$. 
\begin{lemma} \label{decay at corner face decay near main face lemma}
Let $M \in \mathbb{R}$, $\vor, \vov \in \mathcal{C}^{\infty}( \psf \Xd )$, $\vob \in \mathcal{C}^{\infty}( \mathcal{C}_{\tindex} \times \overline{\mathbb{R}^{n_{\tindex}}} )$. Let also $\varphi_{\dmf} \in \mathcal{C}^{\infty}( \Xd )$ be a cut-off function at $\dmf$ whose support is contained in $\{ x_{\dmf} \leq \epsilon \}$, where $\epsilon > 0$ is arbitrarily small and $x_{\dmf} \in \mathcal{C}^{\infty}( \Xd )$ is a global boundary defining function for $\dmf$. Then for every $\delta > 0$, there exists $C > 0$ such that we have
\begin{equation} 
\label{corner face decay equation 3}
\| x_{\dmf}^{\delta} \varphi_{\dmf} u \|_{H_{\mathrm{d3sc,3co,res}}^{M, \vor, \ast , \vov, \vob , \ast }} \leq \epsilon^{\delta} C \| \varphi_{\dmf} u \|_{H_{\mathrm{d3sc,3co,res}}^{M, \vor, \ast , \vov, \vob, \ast}}.
\end{equation}
\end{lemma}
\begin{proof}
For convenience, in the discussion below we will take $\ast = 0$, for it would be irrelevant on the support of $\varphi_{\dmf}$ even if it was non-zero. \par

We first recall from \S \ref{subsection Sobolev spaces under second microlocalization} that
\begin{equation} \label{corner face decay equation 4.9.9.1}
\| x_{\dmf}^{\delta} \varphi_{\dmf} u \|_{H_{\mathrm{d3sc,3co,res}}^{M, \vor, 0 , \vov, \vob, 0}}^2 \simeq \| x_{\dmf}^{\delta} \varphi_{\dmf} u \|_{H_{\mathrm{d3sc,3co,res}}^{M,N,0,K,F,0}}^2 + \| \tilde{\Lambda} x_{\dmf}^{\delta} \varphi_{\dmf} u  \|_{H_{\mathrm{d3sc,3co,res}}^{M,\vor, 0 ,\vov - \vob, 0 ,0}}^2,
\end{equation}
where $\tilde{\Lambda} \in \Psf^{0,0,0,\vob, \vob, 0}$ is some quantization of $x_{\cf}^{-\vob}$, and $N, K, F < 0$ are sufficiently negative constants. In particular, $\tilde{\Lambda}$ must be supported in a neighborhood of the diagonal. Moreover, we can write
\begin{equation}
\label{corner face decay equation 4.9.1}
\| x_{\dmf}^{\delta} \varphi_{\dmf} u \|_{H_{\mathrm{d3sc,3co,res}}^{M,N,0,K,F,0}}^2 \simeq \| x_{\dmf}^{\delta} \varphi_{\dmf} u \|_{H_{\mathrm{3co}}^{M', N, 0, F}}^2 + \| \Lambda' x_{\dmf}^{\delta} \varphi_{\dmf} u \|_{L^2}^2,
\end{equation}
as well as 
\begin{align}
\label{corner face decay equation 4.9.2}
\begin{split}
\| \tilde{\Lambda} x_{\dmf}^{\delta} \varphi_{\dmf} u \|_{H_{\mathrm{d3sc,3co,res}}^{M, \vor, 0, \vov - \vob, 0 ,0}}^2 \simeq {} & \| \tilde{\Lambda} x_{\dmf}^{\delta} \varphi_{\dmf} u \|_{H_{\mathrm{d3sc,3co,res}}^{M', N, 0, K, 0 ,0}}^2 + \| \Lambda \tilde{\Lambda} x_{\dmf}^{\delta} \varphi_{\dmf} u \|_{L^2}^2 \\
\simeq {} & \|  \tilde{\Lambda} x_{\dmf}^{\delta} \varphi_{\dmf} u \|_{H_{\mathrm{3co}}^{M'', N, 0,0}}^2 \\
&  + \| \Lambda'' \tilde{\Lambda} x_{\dmf}^{\delta} \varphi_{\dmf} u \|_{L^2} + \| \Lambda \tilde{\Lambda} x_{\dmf}^{\delta} \varphi_{\dmf} u \|_{L^2}^2, 
\end{split}
\end{align}
where $\Lambda \in \Psf^{M, \vor, 0, \vov - \vob, 0,0}$, $\Lambda' \in \Psi_{\mathrm{d3sc,3co,res}}^{M,N,0,K,F,0}$ and $\Lambda'' \in \Psi_{\mathrm{d3sc,3co,res}}^{M', N, 0, K , 0,0} $ are some elliptic operators in the symbolic sense which are also supported in a neighborhood of the diagonal, while $M', M'' < 0$ are sufficiently negative constants.

We will consider the above terms individually. To start, it is easy to see that we have
\begin{align*}
\| x_{\dmf}^{\delta} \varphi_{\dmf} u \|_{H_{\mathrm{3co}}^{M', N , 0 , F}}^2 & \simeq \| x_{\cf}^{-F} x_{\dmf}^{\delta} \varphi_{\dmf} u \|_{L^2}^2 + \|  \Lambda_0 x_{\dmf}^{\delta} \varphi_{\dmf} u \|_{L^2}^2 \\
& \leq C \epsilon^{2\delta} ( \| x_{\cf}^{-F} \varphi_{\dmf} u \|_{L^2}^2 + \| x_{\dmf}^{-\delta} \Lambda_0 x_{\dmf}^{\delta} \varphi_{\dmf} u \|_{L^2}^2 ),
\end{align*}
where $\Lambda_0 \in \Psi_{\mathrm{3coc}}^{M', N ,0,F}$ is elliptic in the symbolic sense, and moreover supported near the diagonal. But since $x_{\dmf}^{-\delta} \Lambda_0 x_{\dmf}^{\delta}$ belongs to $\Psi_{\mathrm{3coc}}^{M',N,0,F}$, by boundedness, we can conclude
\begin{equation}
\label{corner face decay equation 5}
\| x_{\dmf}^{\delta} \varphi_{\dmf} u \|_{H_{\mathrm{3co}}^{M', N , 0 ,F}}  \leq C \epsilon^{\delta} \| \varphi_{\dmf} u \|_{H_{\mathrm{3co}}^{M', N ,0 , F}}.
\end{equation}
Likewise, we have 
\begin{align*}
\| \tilde{\Lambda} x_{\dmf}^{\delta} \varphi_{\dmf} u \|_{H_{\mathrm{3co}}^{M'', N,0,0}}^{2} & \simeq \| \tilde{\Lambda} x_{\dmf}^{\delta} \varphi_{\dmf} u \|_{L^2}^{2} + \| \Lambda_0' \tilde{\Lambda} x_{\dmf}^{\delta} \varphi_{\dmf} u \|_{L^2}^2 \\
& \leq C  \epsilon^{2 \delta} (  \| x_{\dmf}^{-\delta} \tilde{\Lambda} x_{\dmf}^{\delta} \varphi_{\dmf} u  \|_{L^2}^2 + \| x_{\dmf}^{-\delta} \Lambda_0' x_{\dmf}^{\delta} x_{\dmf}^{-\delta} \tilde{\Lambda} x_{\dmf}^{\delta} \varphi_{\dmf} u \|_{L^2}^2  ),
\end{align*}
where $\Lambda_0' \in \Psi_{\mathrm{3coc}}^{M'', N, 0,0}$ is elliptic in the symbolic sense and is supported near the diagonal. Thus, since $x_{\dmf}^{-\delta} \Lambda_0' x_{\dmf}^{\delta} \in \Psi_{\mathrm{3coc}}^{M'', N, 0,0}$, we can conclude that
\begin{equation}
\label{corner face decay equation 5.0.1}
\| \tilde{\Lambda} x_{\dmf}^{\delta} \varphi_{\dmf} u \|_{H_{\mathrm{3co}}^{M'', N, 0,0}} \leq C \epsilon^{\delta} \| x_{\dmf}^{-\delta} \tilde{\Lambda} x_{\dmf}^{\delta} \varphi_{\dmf} u \|_{H_{\mathrm{3co}}^{M'', N, 0,0}}
\end{equation}
holds as well.

In fact, the same strategy can also be used to estimate the $L^2$-norms in (\ref{corner face decay equation 4.9.1}) and (\ref{corner face decay equation 4.9.2}). This leads us to the estimates
\begin{equation}
\label{corner face decay equation 5.0.2}
\| \Lambda' x_{\dmf}^{\delta} \varphi_{\dmf} u \|_{L^2} \leq C \epsilon^{\delta} \| x_{\dmf}^{-\delta} \Lambda' x_{\dmf}^{\delta} \varphi_{\dmf} u \|_{L^2}, 
\end{equation}
and
\begin{align}
\label{corner face decay equation 5.0.3}
\begin{split}
\| \Lambda'' \tilde{\Lambda} x_{\dmf}^{\delta} \varphi_{\dmf} u \|_{L^2} & \leq C \epsilon^{\delta} \| x_{\dmf}^{-\delta} \Lambda'' x_{\dmf}^{\delta} x_{\dmf}^{-\delta} \tilde{\Lambda} x_{\dmf}^{\delta} \varphi_{\dmf} u \|_{L^2}, \\
\| \Lambda \tilde{\Lambda} x_{\dmf}^{\delta} \varphi_{\dmf} u \|_{L^2} & \leq C \epsilon^{\delta} \| x_{\dmf}^{-\delta} \Lambda x_{\dmf}^{\delta} x_{\dmf}^{-\delta} \tilde{\Lambda} x_{\dmf}^{\delta} \varphi_{\dmf} u \|_{L^2}. 
\end{split}
\end{align}

By substituting (\ref{corner face decay equation 5}), (\ref{corner face decay equation 5.0.2}) into (\ref{corner face decay equation 4.9.1}), we now have
\begin{align}
\label{corner face decay equation 5.0.4}
\begin{split}
\| x_{\dmf}^{\delta} \varphi_{\dmf} u \|_{H_{\mathrm{d3sc,3co,res}}^{M,N,0,K,F,0}}^2 & \leq C \epsilon^{2 \delta} ( \| \varphi_{\dmf} u \|_{H_{\mathrm{3co}}^{M', N, 0, F}}^2 + \| x_{\dmf}^{-\delta} \Lambda' x_{\dmf}^{\delta} \varphi_{\dmf} u \|_{L^2}^2 ) \\
& \leq C \epsilon^{2 \delta} \| \varphi_{\dmf} u \|_{H_{\mathrm{d3sc,3co,res}}^{M,N,0,K,F,0}}^2.
\end{split}
\end{align}
Similarly, by substituting (\ref{corner face decay equation 5.0.1}), (\ref{corner face decay equation 5.0.3}) into (\ref{corner face decay equation 4.9.2}), we have
\begin{align*}
\begin{split}
& \| \tilde{\Lambda} x_{\dmf}^{\delta} \varphi_{\dmf} u \|_{H_{\mathrm{d3sc,3co,res}}^{M, \vor, 0, \vov - \vob, 0,0}}^2 \leq C \epsilon^{2 \delta} ( \|  x_{\dmf}^{-\delta} \tilde{\Lambda} x_{\dmf}^{\delta} \varphi_{\dmf} u \|_{H_{\mathrm{3co}}^{M'', N, 0,0}}^2  \\
& \qquad + \| x_{\dmf}^{-\delta} \Lambda'' x_{\dmf}^{\delta} x_{\dmf}^{-\delta} \tilde{\Lambda} x_{\dmf}^{\delta} \varphi_{\dmf} u \|_{L^2}^2 + \| x_{\dmf}^{-\delta} \Lambda x_{\dmf}^{\delta} x_{\dmf}^{-\delta} \tilde{\Lambda} x_{\dmf}^{\delta} \varphi_{\dmf} u \|_{L^2}^2 ),
\end{split}
\end{align*}
 which implies that 
 \begin{equation}
 \label{corner face decay equation 5.0.5}
 \| \tilde{\Lambda} x_{\dmf}^{\delta} \varphi_{\dmf} u \|_{H_{\mathrm{d3sc,3co,res}}^{M, \vor, 0, \vov - \vob, 0,0}} \leq C \epsilon^{\delta} \| x_{\dmf}^{-\delta} \tilde{\Lambda} x_{\dmf}^{\delta} \varphi_{\dmf} u  \|_{H_{\mathrm{d3sc,3co,res}}^{M, \vor, 0, \vov - \vob, 0 ,0}}.
 \end{equation}
Finally, by substituting (\ref{corner face decay equation 5.0.4}) and (\ref{corner face decay equation 5.0.5}) back into (\ref{corner face decay equation 4.9.9.1}), we arrive at
\begin{equation*}
\| x_{\dmf}^{\delta} \varphi_{\dmf} u \|_{H_{\mathrm{d3sc,3co,res}}^{M, \vor, 0, \vov, \vob, 0}}^2 \leq C \epsilon^{2\delta} ( \| \varphi_{\dmf} u \|_{H_{\mathrm{d3sc,3co,res}}^{M,N,0,K,F,0}}^2 + \| x_{\dmf}^{-\delta} \tilde{\Lambda} x_{\dmf}^{\delta} \varphi_{\dmf} u \|_{H_{\mathrm{d3sc,3co,res}}^{M, \vor, 0, \vov - \vob, 0,0}}^2  ),
\end{equation*}
from which (\ref{corner face decay equation 3}) follows immediately since $x_{\dmf}^{-\delta} \tilde{\Lambda} x_{\dmf}^{\delta} \in \Psf^{0,0,0,\vob,\vob,0}$.
\end{proof}

Now, suppose we apply (\ref{corner face decay equation 3}) with $Qu$ in places of $u$. Then for any arbitrarily small $\epsilon > 0$, we can take the support of $\varphi_{\dmf}$ to be sufficiently small, such that
\begin{align*}
\| \varphi_{\dmf} Q u \|_{H_{\mathrm{d3sc,3co,res}}^{M, \vor - \delta , \vol , \vor + \vol - \delta, \vob, S}} &  = \| x_{\dmf}^{\delta} \varphi_{\dmf} x_{\dmf}^{-\delta} Q u \|_{H_{\mathrm{d3sc,3co,res}}^{M, \vor - \delta, \vol , \vor + \vol - \delta, \vob, S}}  \\
&  \leq C \epsilon^{\delta} \| \varphi_{\dmf}x_{\dmf}^{-\delta} Qu \|_{H_{\mathrm{d3sc,3co,res}}^{ M, \vor - \delta, \vol , \vor + \vol - \delta, \vob, S }} \leq C \epsilon \| u \|_{H_{\mathrm{d3sc,3co,res}}^{m, \vor, \vol, \vor + \vol, \vob, s}}.
\end{align*}
By substituting this equation back into (\ref{corner face decay equation 2.5}), which is followed by (\ref{corner face decay equation 2.4}) and then (\ref{almost semi-Fredholm estimates with symbolic decay}), we conclude that
\begin{align*}
\| u \|_{H_{\mathrm{d3sc,3co,res}}^{m, \vor, \vol, \vor + \vol, \vob, s } } & \leq C(  \| P u \|_{H_{\mathrm{d3sc,3co,res}}^{ M - 2, \vor - \delta, \vol, \vor + \vol - \delta , \vob , S - 2 } } \\
& \quad + \epsilon \| u \|_{H_{\mathrm{d3sc,3co,res}}^{m, \vor, \vol, \vor + \vol, \vob, s }} + \| \varphi_{\dff}  Q u \|_{H_{\mathrm{d3sc,3co,res}}^{M, \vor - \delta , \vol  , \vor + \vol - \delta, \vob, S}} + \| u \|_{H_{\mathrm{d3sc,3co,res}}^{M, \vor - \delta, \vol  , \vor + \vol - \delta, F, S}} ),
\end{align*}
and upon taking $\epsilon$ to be small enough, we would have shown that
\begin{align} \label{big decay estimate near cf and dmf}
\begin{split}
\| u \|_{H_{\mathrm{d3sc,3co,res}}^{m, \vor, \vol, \vor + \vol, \vob, s } } & \leq C(  \| P u \|_{H_{\mathrm{d3sc,3co,res}}^{ M - 2, \vor - \delta, \vol, \vor + \vol - \delta, \vob , S - 2 } } \\
& \quad + \| \varphi_{\dff} Q u \|_{H_{\mathrm{d3sc,3co,res}}^{M, \vor - \delta, \vol  , \vor + \vol - \delta, \vob, S}} + \| u \|_{H_{\mathrm{d3sc,3co,res}}^{M, \vor - \delta , \vol  , \vor + \vol - \delta, F, S}} ).
\end{split}
\end{align}

\subsection{Adaptation of Vasy's argument} \label{adaption of vasy argument subsection}
We now make the positive commutator calculations which will allow us to gain decay at $\cf$ away from $\dmf$. The strategy here will be to reduce our construction to one which resembles the construction made by Vasy in \cite[\S 4]{AndrasSM}. \par

We begin by noting a suitable reduction to the variable order $\vor$ at $\dtsccf$. This will be necessary for Vasy's argument to apply (since we use it as a black box here). Recall from \S \ref{variable order construction section}, in particular \S \ref{variable order smooth square root subsection}, that when restricted to $\dtsccf$, the variable order $\vor$ takes the form
\begin{equation*}
\vor = - \frac{1}{2} + \varphi \beta^{\tindex} \phi( (\beta^{\tindex})^{-1} \vor^{\tindex} ) + \vol, \quad \beta^{\tindex} > 0.
\end{equation*}
Here, $\varphi \in \mathcal{C}^{\infty}( \psf \Xd )$ is a cut-off at $\bcv$, $\phi$ is any $\mathcal{C}^{\infty}( \mathbb{R} )$ function such that $\phi ' \geq 0$ with $\phi(t) = \pm \epsilon_{\pm}$ near $ t = \pm 1$ for some $\epsilon_{\pm} > 0$, and 
\begin{equation*}
\vor^{\tindex} =  \beta^{\tindex} \frac{ \utau }{ ( \intn )^{1/2} }.
\end{equation*}
Henceforth, we shall further require that, say
\begin{equation*}
t + \frac{1}{2} \leq \phi (t) \leq t + 1.
\end{equation*}
Such a $\phi$ can certainly be arranged. For instance, it is possible if $\epsilon_{-} = -1/4$, $\epsilon_{+} = 7/4$. If additionally we require $\varphi_{1} \leq 1$, then it is easy to see that
\begin{equation} \label{our variable order compared with Vasy's one}
\vor - \beta^{\tindex} \leq \hvor \leq \vor - \frac{\beta^{\tindex}}{2}, \quad \hvor = -\frac{1}{2} + \varphi \vor^{\tindex} + \vol.
\end{equation} \par

We now state the following crucial lemma:
\begin{lemma} \label{decay at corner face away from main face proposition}
Let $Q \in \Psf^{-\infty,0,0,0,0,-\infty}$, $\varphi_{\dff} \in \mathcal{C}^{\infty}(\Xd)$ be specified as above. Then for every sufficiently negative $M,N,S<0$ and small $\delta  > 0$, we can find $C > 0$ such that
\begin{align} \label{decay at cf main lemma inequality}
\begin{split}
\| \varphi_{\dff} Q u \|_{H_{\mathrm{d3sc,3co,res}}^{M, \ast , \vol, \hvor + \vol, \vob, S}} & \leq C( \epsilon^{-1} \| P u \|_{H_{\mathrm{d3sc,3co,res}}^{ M , N , \vol , \vor + \vol + 1 , \vob + 2 , S}} \\
&\quad +  \epsilon \| u \|_{H_{\mathrm{d3sc,3co,res}}^{M, N, \vol, \hvor + \vol, \vob, S}} + \| u \|_{H_{\mathrm{d3sc,3co,res}}^{M, N , \vol, \hvor + \vol + \delta, \vob -  \delta, S}} )
\end{split}
\end{align}
whenever the right hand side of (\ref{decay at cf main lemma inequality}) is finite.
\end{lemma}
\begin{proof}[Proof of Proposition \ref{decay at the corner face proposition} assuming Lemma \ref{decay at corner face away from main face proposition}]
Starting from (\ref{big decay estimate near cf and dmf}), it remains to estimate the norm of $\varphi_{\dff} Qu$. First note that by taking $\beta^{\tindex} > 0$ sufficiently small, we can take the case of $\delta = \beta^{\tindex}$ (recall that we indeed have the freedom in choosing $\delta$, so long as it is small enough) in (\ref{almost semi-Fredholm estimates with symbolic decay}). Then by using (\ref{our variable order compared with Vasy's one}), we can estimate
\begin{equation} \label{tangential away from o cal 0.05}
\| \varphi_{\dff} Q u \|_{ H_{\mathrm{d3sc,3co,res}}^{M, \ast , \vol, \vor + \vol - \beta^{\tindex}, \vob, S} } \leq C \| \varphi_{\dff} Q u  \|_{H_{\mathrm{d3sc,3co,res}}^{M, \ast , \vol, \hvor + \vol, \vob, S}}.
\end{equation}
On the other hand, applying Lemma \ref{decay at corner face away from main face proposition} with some $\delta' > 0$ in places of $\delta$, we note that the last two terms on the right hand side of (\ref{decay at cf main lemma inequality}) satisfy
\begin{align} \label{tangential away from o cal 0.06}
\begin{split}
\| u \|_{H_{\mathrm{d3sc,3co,res}}^{M, N, \vol, \hvor + \vol, \vob, S}} & \leq C \| u \|_{H_{\mathrm{d3sc,3co,res}}^{M,N,\vol, \vor + \vol - \beta^{\tindex}/2,\vob, S}}, \\
 \| u \|_{H_{\mathrm{d3sc,3co,res}}^{M, N , \vol, \hvor + \vol + \delta', \vob -  \delta', S}} & \leq C \| u \|_{H_{\mathrm{d3sc,3co,res}}^{M,N, \vol, \vor + \vol - \beta^{\tindex}/2 + \delta', \vob -  \delta', S}}.
\end{split}
\end{align}
Thus by substituting (\ref{tangential away from o cal 0.06}) into (\ref{decay at cf main lemma inequality}) and the resulting estimate into (\ref{tangential away from o cal 0.05}), we have
\begin{align*}
\begin{split}
\| \varphi_{\dff} Q u \|_{H_{\mathrm{d3sc,3co,res}}^{M, \ast , \vol, \vor + \vol - \beta^{\tindex}, \vob, S}} & \leq C( \epsilon^{-1} \| P u \|_{H_{\mathrm{d3sc,3co,res}}^{m-2, \vor  + 1 , \vol , \vor + \vol + 1 , \vob + 2 , s - 2}} \\
&\quad +  \epsilon \| u \|_{H_{\mathrm{d3sc,3co,res}}^{M, N, \vol, \vor + \vol, \vob, S}} + \| u \|_{H_{\mathrm{d3sc,3co,res}}^{M, N , \vol, \vor + \vol - \beta^{\tindex}/2 + \delta' , \vob - \delta', S}} ).
\end{split}
\end{align*}
Now, by choosing $\delta'$ to be so small such that $\beta^{\tindex} \geq 4 \delta'$, we can further substitute the above estimate back into (\ref{big decay estimate near cf and dmf}), and we finally get
\begin{align*}
\begin{split}
\| u \|_{H_{\mathrm{d3sc,3co,res}}^{m, \vor, \vol, \vor + \vol, \vob, s } } & \leq C( \epsilon^{-1} \| P u \|_{H_{\mathrm{d3sc,3co,res}}^{ m - 2, \vor + 1, \vol, \vor + \vol + 1 , \vob + 2 , s - 2 } } \\
& \quad + \epsilon \| u \|_{H_{\mathrm{d3sc,3co,res}}^{m, \vor, \vol, \vor + \vol, \vob, s }}  + \| u \|_{H_{\mathrm{d3sc,3co,res}}^{M, \vor - \beta^{\tindex} , \vol, \vor + \vol - \delta', \vob - \delta', S}} )
\end{split}
\end{align*}
Then the required claim follows from taking $\epsilon$ small enough, which is followed by choosing $\delta$ such that $\min \{ \beta^{\tindex}, \delta' \} \geq \delta$.
\end{proof}
The rest of this subsection and the next will be devoted to the proof of Lemma \ref{decay at corner face away from main face proposition}. To begin, we shall arrange for a $A \in \Psf^{-\infty, -\infty , \vol, \hvor + \vol  + 1/2, \vob + 1, -\infty}$ such that
\begin{equation} \label{tangential away from o cal 0.1}
\hat{N}_{\cf}(A) = q_{\tindex} \hbff^{-n^{\tindex}/4} A^{\tindex} \hbff^{n^{\tindex}/4} \hbff^{-\vol + \vob /2 + 1/2}  \varphi_{\Co}.
\end{equation}
Here we remark that $-\vol + \vob/2$ is a constant. We also make $q_{\tindex}$ explicit by setting
\begin{equation*}
q_{\tindex} = \psi ( \intt - \lambda^{2} ) e^{K \ltau},
\end{equation*}
where $\psi \in \mathcal{C}^{\infty}_{c}( \mathbb{R} )$ is a cut-off function near $0$, while $\varphi_{\Co} \in \mathcal{C}^{\infty}( [ \hat{X}^{\tindex} ; \{ 0 \} ] )$ is a cut-off at $\Co$ such that $\varphi_{\mathcal{C}^{\tindex}_{0}}$ is identically $1$ near $\Co$, and $K >0$ is large to be determined later. \par

The difficulty therefore lies in the construction of $A^{\tindex}$, which we will actually construct to live in $\Psi_{\mathrm{bc}, \delta}^{ \hvor + \vol - \vob - 1/2, 0, 0 } ( [ \hat{X}^{\tindex} ; \{ 0 \} ] )$. Moreover, we will choose $A^{\tindex}$ to be symmetric with respect to $L^{2}_{\mathrm{b}}([ \hat{X}^{\tindex} ; \{ 0 \} ]) = L^{2}([ \hat{X}^{\tindex} ; \{ 0 \} ], \nu_{\mathrm{b}})$, where $\nu_{\mathrm{b}}$ is the natural b-density on $[ \hat{X}^{\tindex} ; \{ 0 \} ]$. Lastly, $A^{\tindex}$ must be dilation invariant with respect to $\hbff$. The advantage of this setup is that
\begin{equation} \label{tangential away from o cal 1}
 \hat{N}_{\cf}( A^{\ast} A ) = q_{\tindex}^2 \varphi_{\Co}  \hbff^{ - \vol + \vob / 2 + 1/2 -n^{\tindex}/4} (A^{\tindex})^2 \hbff^{ - \vol + \vob / 2 + 1/2 + n^{\tindex}/4} \varphi_{\Co}.
\end{equation} \par

Let us remark that (\ref{tangential away from o cal 0.1}) indeed lives in the cone calculus, and in particular is a member of $\Psi_{\mathrm{coc}, \delta}^{\hvor + \vol - \vob - 1/2, -\infty, \vol - \vob/2  - 1/2 }( [ \hat{X}^{\tindex} ; \{ 0 \} ] )$ due to the presence of the cut-off $\varphi_{\Co}$. In fact, it is the case that $\tilde{A}^{\tindex} \varphi_{\Co}$ is a conormal co-operator for any conormal b-operator $\tilde{A}^{\tindex}$. One can see this rather easily by splitting the kernel of $\tilde{A}^{\tindex} \varphi_{\Co}$ into pieces that are supported on the various regions of $[ \hat{X}^{\tindex} ; \{ 0 \} ]$, and then observing that they belong to the correct classes. \par

We now return to the commutator formula proved in \S \ref{Subsection commutator formulae}. In particular, we will apply Lemma \ref{commutator formula lemma} in the current context. It will be convenient to relax the requirement for the principal symbol calculations in (\ref{tangential discussion 0.9}) and (\ref{tangential discussion 1.001}). Then Lemma \ref{commutator formula lemma} implies that for any $\tilde{B} \in \Psf^{-\infty, -\infty , \vol - 1/2, \hvor + \vol - 1/2, \vob, -\infty}$, $\tilde{E} \in \Psf^{-\infty, -\infty , 2 \vol, 2 \hvor + 2 \vol + 2 \delta, 2 \vob, -\infty}$ which are sufficiently classical at $\cf$, such that 
 \begin{align}
 \hat{N}_{\cf}( \tilde{B} ) & = ( - \mathsf{H}_{ \inttcf } \vob )^{1/2} \hat{N}_{\cf}( A ),  \label{corner face commutator formula 1} \\
 \begin{split}
 \hat{N}_{\cf}( \tilde{E} ) & =  i[  \hat{N}_{\cf}( \Delta_{z^{\tindex}} ) + 2 \ltau \hbff D_{\hbff}, \hat{N}_{\cf}( A^{\ast} A ) ] \label{corner face commutator formula 2} \\
& \quad + \mathsf{H}_{\inttcf} \hat{N}_{\cf}( A^{\ast}A ) - ( 2\vob + 2 ) \ltaucf \hat{N}_{\cf}( A^{\ast}A ),
\end{split}
 \end{align}
 we can find some $\tilde{R} \in  \Psf^{-\infty, -\infty, 2 \vol , 2 \hvor + 2 \vol + 2 \delta, 2 \vob - 2\delta, - \infty}$ for which
 \begin{equation} \label{corner face commutator formula 3}
 i [ P,  A ^{\ast} A] =  \tilde{B}^{\ast} ( \log x ) \tilde{B} + \tilde{E} + \tilde{R}.
\end{equation}   \par

We will study the extent to which (\ref{corner face commutator formula 2}) is negative with our construction of $A$. To this end, we will consider first negativity the commutator 
\begin{equation}  \label{tangential away from o cal 2.9}
i[ \hat{N}_{\cf}( \Delta_{z^{\tindex}} ) + 2 \ltau \hbff D_{\hbff}, \hat{N}_{\cf}( A^{\ast} A ) ].  
\end{equation}
We will need to factor out the cut-offs. Thus we write (\ref{tangential away from o cal 2.9}) as
\begin{align*} \label{tangential away from o cal 3}
q_{\tindex}^2 \varphi_{\Co} i [ \hat{N}_{\cf}( \Delta_{z^{\tindex}} ) + 2 \ltau \hbff D_{\hbff},  \hbff^{-\vol + \vob/2 + 1/2 - n^{\tindex}/4} ( A^{\tindex})^2 \hbff^{ -\vol + \vob/2 + 1/2 + n^{\tindex}/4 } ] \varphi_{\Co} & \\
+ \, q_{\tindex}^2 ( E_{L}^{\tindex} + E_{R}^{\tindex} ) & ,
\end{align*} 
where we have defined
\begin{equation*} \label{tangential away from o cal 3.5}
\begin{gathered}
E_{L}^{\tindex} =  i [ \hat{N}_{\cf}( \Delta_{z^{\tindex}} ) + 2 \ltau \hbff D_{\hbff} ,  \varphi_{\Co}  ] \hbff^{-\vol + \vob/2 + 1/2 - n^{\tindex}/4} ( A^{\tindex})^2 \hbff^{ -\vol + \vob/2 + 1/2 + n^{\tindex}/4 }  \varphi_{\Co}, \\
E_{R}^{\tindex} = i \hbff^{n^{\tindex}/2} \varphi_{\Co} \hbff^{-\vol + \vob/2 + 1/2 - n^{\tindex}/4} ( A^{\tindex})^2 \hbff^{ -\vol + \vob/2 + 1/2 + n^{\tindex}/4 }   [ \hat{N}_{\cf}( \Delta_{z^{\tindex}} ) + 2 \ltau \hbff D_{\hbff}, \varphi_{\Co} ].
 \end{gathered}
\end{equation*}
Note that $E_{L}^{\tindex}, E_{R}^{\tindex} \in \Psi_{\mathrm{coc}, \delta}^{2 \hvor + 2 \vol - 2 \vob  , -\infty, -\infty}( [ \hat{X}^{\tindex} ; \{ 0 \} ] )$ due to the commutation with $\varphi_{\Co}$. Indeed, the presence of $\varphi_{\Co}$ alone guarantees the decay at $\overline{^{\mathrm{co}}T^{\ast}}_{\Co}[ \hat{X}^{\tindex} ; \{ 0 \} ]$, while the fact that it's identically $1$ near $\Cinfty$ ensures decay at $\overline{^{\mathrm{co}}T^{\ast}}_{\Cinfty}[ \hat{X}^{\tindex} ; \{ 0 \} ]$.  \par

We now focus on the term
\begin{equation} \label{tangential away from o cal 4}
i [ \hat{N}_{\cf}( \Delta_{z^{\tindex}} ),  \hbff^{-\vol + \vob/2 + 1/2 - n^{\tindex}/4} ( A^{\tindex})^2 \hbff^{ -\vol + \vob/2 + 1/2 + n^{\tindex}/4 } ].
\end{equation}
Since this is at first an operator in $\Psi_{\mathrm{bc}, \delta}^{2 \hvor + 2 \vol - 2 \vob + 2 \delta, -2 \vol + \vob, 2 \vol - \vob}( [ \hat{X}^{\tindex} ; \{ 0 \} ] )$, we will multiply both sides of it by $\hbff^{\vol - \vob / 2}$. This doesn't affect the sign of (\ref{tangential away from o cal 4}), but may lead to an operator which is dilation invariant in $\hbff$. Indeed, we can compute that
\begin{align} \label{tangential main commutator calculation}
\begin{split}
& i \hbff^{\vol - \vob / 2} [ \hat{N}_{\cf} ( \Delta_{z^{\tindex}} ), \hbff^{-\vol + \vob / 2 + 1/2 - n^{\tindex}/4} (A^{\tindex})^{2} \hbff^{-\vol + \vob / 2 + 1/2 + n^{\tindex}/4} ] \hbff^{\vol - \vob/2} \\ 
& \qquad = i \hbff^{ \vol - \vob / 2 } (  \hat{N}_{\cf}( \Delta_{z^{\tindex}} ) \hbff^{-\vol + \vob / 2 + 1/2 - n^{\tindex}/4} (A^{\tindex})^{2} \hbff^{-\vol + \vob / 2 + 1/2 + n^{\tindex}/4}  ) \hbff^{\vol - \vob / 2}  \\
& \qquad \quad- i \hbff^{ \vol - \vob / 2} ( \hbff^{-\vol + \vob / 2 + 1/2 - n^{\tindex}/4} (A^{\tindex})^{2} \hbff^{-\vol + \vob / 2 + 1/2 + n^{\tindex}/4} \hat{N}_{\cf}( \Delta_{z^{\tindex}} ) ) \hbff^{\vol - \vob / 2}  \\
& \qquad = i \hbff^{-n^{\tindex}/4} ( \hbff^{ \vol - \vob/2 - 1/2 }  \efN \hbff^{ - \vol + \vob/2 + 1/2 } )  ( \hbff^{-1/2} (A^{\tindex})^{2} \hbff^{1/2} ) \hbff^{n^{\tindex}/4}  \\
& \qquad \quad - i \hbff^{ -n^{\tindex}/4 } ( \hbff^{1/2} (A^{\tindex})^{2} \hbff^{-1/2} ) ( \hbff^{ - \vol + \vob/2 + 1/2 } \efN \hbff^{ \vol - \vob/2 - 1/2 } ) \hbff^{n^{\tindex}/4}, 
\end{split}
\end{align}
where we have also defined
\begin{align*}
\efN & = \hbff^{  ( n^{\tindex} + 2 )/4 } \hat{N}_{\cf} ( \Delta_{z^{\tindex}} ) \hbff^{ - ( n^{\tindex} - 2 )/4 } \\
& = 4 ( \hbff D_{\hbff} )^{2} + \Delta_{h^{\tindex}} + \left( \frac{n^{\tindex} - 2}{2} \right)^{2},
\end{align*}
which is again translation invariant in $\hbff$. \par

Now, recall from the properties of the Mellin transform in $\hbff$ that since both $(A^{\tindex})^2$ and $\efN$ are dilation invariant in $\hbff$, the effect of conjugating these operators by $\hbff^{k}$ (i.e., multiplication by this from the right, and by its inverse from the left) is replacing $\hbff D_{\hbff}$ by $\hbff D_{\hbff} - ik$, or on the Mellin transformed side replacing $\utaucob$ by $\utaucob - ik$. Writing such a change by affixing $( \cdot - ik )$ to the operator, followed by conjugating (\ref{tangential main commutator calculation}) by $\hbff^{-n^{\tindex}/4}$, we can continue the calculation from (\ref{tangential main commutator calculation}) and conclude that it is equal to
\begin{equation} \label{tangential interactive commutator calculation 2}
i \efN( \cdot - i ( b  + 1 )/2 ) A^{\tindex}( \cdot - i/2 )^2  - i A^{\tindex}( \cdot + i/2 )^2 \efN( \cdot + i( b + 1 )/2 ),
\end{equation}
where we are writing $b = \vob - 2 \vol$. \par

Notice that the terms in (\ref{tangential interactive commutator calculation 2}) are all multiplication operators on the Mellin transform side (i.e., no $\hbff$ dependency). So {if we choose $A^{\tindex}$ to depend on $y^{\tindex}$ and its b-dual variables only through $\Delta_{h^{\tindex}}$, and still symmetric}, then (\ref{tangential interactive commutator calculation 2}) can be reduced to a commutative calculation on the Mellin transformed side, and is thus equal to
\begin{equation} \label{tangential interactive commutator calculation imaginary}
- 2 \mathrm{Im}  A^{\tindex} ( \cdot - i / 2 )^2 \efN( \cdot - i ( b  + 1 )/2 ) ,
\end{equation}
in the sense that this is the imaginary part of (\ref{tangential interactive commutator calculation 2}) on the Mellin transformed side. It follows that in order to control the negativity of (\ref{tangential interactive commutator calculation 2}), and subsequently that of (\ref{tangential main commutator calculation}), it suffices to do so for the term (\ref{tangential interactive commutator calculation imaginary}).  \par

In fact, up to a sign and a simple scaling, the control of negativity of (\ref{tangential interactive commutator calculation imaginary}) has already been considered in details by Vasy in \cite[\S 4.3]{AndrasSM}. Thus, here we shall only discuss how to directly apply these results, and refer the readers to \cite[\S 4.3]{AndrasSM} for more details. \par

To this end, we will construct $A^{\tindex}$ so that it `almost' takes the form 
\begin{align} \label{important construction of tilde A tangential propagation}
\begin{split}
&   \exp \Big(  \frac{\tilde{\beta}^{\tindex}}{2} \frac{ 2 \utaucob }{ ( 4 ( \utaucob )^2 + \Delta_{h^{\tindex}} + \tilde{F}^{2} )^{1/2} }  \\
& \qquad + \big( \frac{\beta^{\tindex}}{2}  \frac{ 2 \utaucob }{  ( 4 ( \utaucob )^2 + \Delta_{h^{\tindex}} + F^2 )^{1/2} }  \big) \log ( 4 ( \utaucob )^{2} + \Delta_{h^{\tindex}} + F^2 ) \\
& \qquad - \frac{ b   + 1 }{2} \log ( 4 ( \utaucob )^2 + \Delta_{h^{\tindex}} + \check{F}^2 ) \Big),
\end{split}
\end{align}
where $\tilde{F} \geq F \geq \check{F} > 1$, $\tilde{\beta} \geq 0$ are parameters, and the square root and the logarithm are defined with branch cuts along the negative real axis, and are real for positive arguments. Notice that through a simple variable change (which we remark is the natural one as canonical transformations of the dual coordinates) $\utaucob = - \utaub/2$, we can put (\ref{important construction of tilde A tangential propagation}) into the form
\begin{align} \label{important construction of tilde A tangential propagation Vasy form}
\begin{split}
&  \exp \Big(  \frac{\tilde{\beta}^{\tindex}}{2} \frac{ \utaub }{ (  ( \utaub )^2 + \Delta_{h^{\tindex}} + \tilde{F}^{2} )^{1/2} }  \\
& \qquad + \big( \frac{\beta^{\tindex}}{2}  \frac{ \utaucob }{  ( ( \utaub )^2 + \Delta_{h^{\tindex}} + F^2 )^{1/2} }  \big) \log ( ( \utaub )^{2} + \Delta_{h^{\tindex}} + F^2 ) \\
& \qquad - \frac{ b   + 1 }{2} \log ( ( \utaub )^2 + \Delta_{h^{\tindex}} + \check{F}^2 ) \Big) ,  
\end{split}
\end{align}
which is exactly the expression that was considered in \cite[Equation (4.25)]{AndrasSM}, corresponding to the variable order
\begin{equation*} \label{special variable order form where Andras used it}
\hvor + \vol - \vob = \frac{1}{2} - ( b + 1 ) + \beta^{\tindex} \frac{\utaub}{ ( \intnb )^{1/2} },
\end{equation*}
where $|b  + 1 | < (n^{\tindex} - 2)/2$, which is also the variable order we presently use. In particular, this fixes a sign issue appearing in (\ref{tangential interactive commutator calculation imaginary}), in the sense that, with $\utaub$ instead of $\utaucob$ being the Mellin transformed variable, (\ref{tangential interactive commutator calculation imaginary}) also takes the same form as its correspondence in \cite[Equation (4.22)]{AndrasSM}. \par

Finally, one convolves (\ref{important construction of tilde A tangential propagation}) with the Gaussian
\begin{equation} \label{the Gaussian}
\frac{ 1 }{ \sqrt{ \pi s } }\exp \Big( - \frac{ 4 ( \utaucob )^2  }{2s} \Big) = \frac{ 1 }{ \sqrt{ \pi s } }\exp \Big( - \frac{ ( \utaub )^2  }{2s} \Big) 
\end{equation}
for $s > 0$ small, due to the mild problem (the `almost' part) that (\ref{important construction of tilde A tangential propagation}) is not an entire (operator valued) function. The rest of the arguments in \cite[\S 4.3]{AndrasSM} applies verbatim.  \par

In summary, we conclude that 
\begin{lemma} \label{indicial operator at cf tangential principal lemma}
Let $A^{\tindex}$ be given by the Mellin conjugate of the the convolution between (\ref{important construction of tilde A tangential propagation}) and (\ref{the Gaussian}). Then (\ref{tangential interactive commutator calculation imaginary}) is a negative definite for $\hat{F}$ sufficiently large. In particular, there exists $B^{\tindex}\in \Psi_{\mathrm{bc}, \delta}^{ \hvor + \vol - \vob , 0, 0}( [ \hat{X}^{\tindex} ; \{ 0 \} ] )$ which is elliptic, symmetric and positive definite with respect to $L^{2}_{\mathrm{b}}( [ \hat{X}^{\tindex} ; \{ 0 \} ] )$, such that
\begin{equation} \label{indicial operator at cf tangential principal lemma equation}
- 2 \mathrm{Im} A^{\tindex}( \cdot - i / 2 )^2 \efN( \cdot - i ( b  + 1 )/2 )  \leq - (B^{\tindex})^{2}.
\end{equation}
\end{lemma}

A simple corollary of this lemma, by using (\ref{indicial operator at cf tangential principal lemma equation}) as well as (\ref{tangential main commutator calculation}), is that we can find a positive indefinite operator $L^{\tindex}$ belonging to $\Psi^{2 \hvor + 2 \vol - 2 \vob + 2 \delta, 0, 0}_{\mathrm{bc}, \delta}( [ \hat{X}^{\tindex} ; \{ 0 \} ] )$ for any choice of small $\delta > 0$, such that we can write (\ref{tangential away from o cal 4}) as
\begin{equation*}
- \hbff^{-\vol + \vob/2  - n^{\tindex}/4}(  (B^{\tindex})^{2} + L^{\tindex} ) \hbff^{ -\vol + \vob/2  + n^{\tindex}/4 }.
\end{equation*}
We remark that the additional term $L^{\tindex}$ must contain the logarithmic term which usually occurs in a commutator estimate involving variable orders. \par

Next, we shall consider the term
\begin{equation} \label{tangential direction second commutator}
2i \ltau [ \hbff D_{\hbff}, \hbff^{-\vol + \vob/2 + 1/2 - n^{\tindex}/4} ( A^{\tindex})^2 \hbff^{ -\vol + \vob/2 + 1/2 + n^{\tindex}/4 } ]
\end{equation}
which can actually be understood using the same strategy as above. Indeed, by multiplying both sides of (\ref{tangential direction second commutator}) now by $\hbff^{\vol - \vob/2 - 1/2}$, we can write
\begin{align} \label{second term tangential calculation 1}
\begin{split}
& i \hbff^{\vol - \vob / 2 - 1/2} [ \hbff D_{\hbff},  \hbff^{-\vol + \vob/2 + 1/2 - n^{\tindex}/4} ( A^{\tindex})^2 \hbff^{ -\vol + \vob/2 + 1/2 + n^{\tindex}/4 } ] \hbff^{\vol - \vob / 2 - 1/2} \\
& \qquad = i ( \hbff^{\vol - \vob/2 - 1/2}( \hbff D_{\hbff} ) \hbff^{-\vol + \vob/2 + 1/2} ) ( \hbff^{-n^{\tindex}/4} (A^{\tindex})^2 \hbff^{n^{\tindex}/4} ) \\
& \qquad  \quad - i ( \hbff^{-n^{\tindex}/4}(A^{\tindex})^2  \hbff^{n^{\tindex}/4} ) ( \hbff^{-\vol+\vob/2 +1/2} ( \hbff D_{\hbff} ) \hbff^{\vol-\vob/2 -1/2} )
\end{split}
\end{align}
Now, as $(A^{\tindex})^2$ commutes with $(\hbff D_{\hbff})( \cdot + i ( \vol - \vob/2 - 1/2 ) )$ on the Mellin transformed side, so must be the case for $(A^{\tindex})^2 ( \cdot - in^{\tindex}/4 )$. Thus we can conclude that (\ref{second term tangential calculation 1}) is also equal to
\begin{equation*}
( - 2 \vol + \vob  + 1 ) \hbff^{-n^{\tindex}/4} (A^{\tindex})^2 \hbff^{n^{\tindex}/4}.
\end{equation*}
Subsequently, we can write (\ref{tangential direction second commutator}) as
\begin{align} \label{second term tangential calculation 2}
\begin{split}
& ( - 4 \vol + 2\vob  + 2 ) \ltau \hbff^{- \vol + \vob/2  + 1/2 -n^{\tindex}/4} (A^{\tindex})^2  \hbff^{ -\vol + \vob/2  + 1/2 + n^{\tindex}/4}.
\end{split} 
\end{align}

\subsection{Proof of Lemma \ref{decay at corner face away from main face proposition}} 
Recalling (\ref{tangential away from o cal 1}), thus far we have worked out that
\begin{align*}
i [ \hat{N}_{\cf}( \Delta_{z^{\tindex}} ) + 2 \ltau \hbff D_{\hbff}, \hat{N}_{\cf}(A^{\ast}A) ] = & - q_{\tindex}^2 \varphi_{\Co}  \hbff^{-\vol + \vob/2 - n^{\tindex}/4} ( (B^{\tindex})^{2} + L^{\tindex} ) \hbff^{ -\vol + \vob/2 + n^{\tindex}/4 } \varphi_{\Co}  \\
& - ( 4 \vol - 2 \vob - 2 ) \ltau \hat{N}_{\cf}( A^{\ast}A ) + q_{\tindex}^{2} ( E^{\tindex}_{L} + E_{R}^{\tindex}).
\end{align*}
Recall that our objective is to study $\hat{N}_{\cf}(\tilde{E})$. From (\ref{corner face commutator formula 2}), it remains to consider
\begin{equation} \label{tangential away from o cal 5.5}
\sH_{\intt} \hat{N}_{\cf}( A^{\ast}A ) - ( 2\vob + 2 ) \ltaucf \hat{N}_{\cf}( A^{\ast}A ),
\end{equation}
but by a straightforward computation, we find that
\begin{equation*}
\sH_{\intt} \hat{N}_{\cf}( A^{\ast}A ) = - 4 K | \lmu |_{h_{\tindex}}^2 \hat{N}_{\cf}( A^{\ast}A ),  
\end{equation*}
so we can conclude that
\begin{align*}  
 \hat{N}_{\cf}( \tilde{E} ) = & - q_{\tindex}^2 \varphi_{\Co}  \hbff^{-\vol + \vob/2  - n^{\tindex}/4} ( (B^{\tindex})^{2} + L^{\tindex} ) \hbff^{ -\vol + \vob/2 + n^{\tindex}/4 }  \varphi_{\Co} \\
& - 4 ( K  |\mu_{\tindex}|_{h_{\tindex}}^2 +  \vol \ltau ) \hat{N}_{\cf}( A^{\ast}A ) + q_{\tindex}^2( E_{L}^{\tindex} + E_{R}^{\tindex} ).
\end{align*} \par

Now, suppose we choose $B \in  \Psf^{-\infty, -\infty, \vol, \hvor + \vol, \vob, -\infty} $ such that 
\begin{equation*}
\hat{N}_{\cf}(B) = q_{\tindex} \hbff^{-n^{\tindex}/4} B^{\tindex} \hbff^{n^{\tindex}/4} \hbff^{-\vol + \vob/2} \varphi_{\Co},
\end{equation*}
then we have
\begin{equation} \label{tangential away from o cal 6}
\hat{N}_{\cf}( B^{\ast} B ) = q_{\tindex}^2 \varphi_{\Co}  \hbff^{-\vol + \vob/2 - n^{\tindex}/4} (B^{\tindex})^{2} \hbff^{ -\vol + \vob/2 + n^{\tindex}/4 }  \varphi_{\Co} .
\end{equation} \par

On the other hand, by the threshold conditions we know that $\vol \ltau > 0$ near $\mathcal{R}_{\mathrm{t},\pm}$, while away from $\mathcal{R}_{\mathrm{t},\pm}$ but still near $\Sigma_{\mathrm{t}}$, we can take $K$ large enough so that $K | \mu_{\tindex} |^{2}_{h_{\tindex}} + \vol \ltau$ remains to be positive. In particular, positivity always holds on the support of $q_{\tindex}$. Thus, we can choose $B_{1} \in \Psf^{-\infty, -\infty, \hvor + \vol, \vob, -\infty}$ such that 
\begin{equation*}
\hat{N}_{\cf}(B_1) = 2 ( K |\mu|_{h_{\tindex}}^2 + \vol \ltau )^{\frac{1}{2}} \hat{N}_{\cf}(A),
\end{equation*}
in which case we also have
\begin{equation}  \label{tangential away from o cal 7}
\hat{N}_{\cf}( B_{1}^{\ast} B_{1} ) = 4 ( K |\mu_{\tindex}|_{h_{\tindex}}^2 + \vol \ltau ) \hat{N}_{\cf}( A^{\ast} A ).
\end{equation} \par

Next we come back to the terms $E_{L}^{\tindex}$ and $E_{R}^{\tindex}$. By taking $\varphi_{\Co}$ to be identically $1$ in an arbitrarily large neighborhood of $\Co$, it is clear that we have 
\begin{center}
$\varphi_{\mathcal{C}^{\tindex}_{\infty}} E_{L}^{\tindex} = E_{L}^{\tindex}$, $E_{R}^{\tindex} \varphi_{\mathcal{C}^{\tindex}_{\infty}} = E_{R}^{\tindex}$,
\end{center}
where $\varphi_{\mathcal{C}^{\tindex}_{\infty}} \in \mathcal{C}^{\infty}( [ \hat{X}^{\tindex} ; \{ 0 \} ] )$ is some bump function at $\mathcal{C}^{\tindex}_{\infty}$ that is identically $1$ where $\varphi_{\Co}$ is not identically $1$. In particular, the support of $\varphi_{\Cinfty}$ can be assumed to be arbitrarily small (by varying $\varphi_{\Co}$). In fact, we can assume that $\varphi_{\mathcal{C}^{\tindex}_{\infty}}$ is the restriction to $\cf$ of some cut-off at $\dmf$, say $\varphi_{\dmf} \in \mathcal{C}^{\infty}( \Xd )$. Thus, if we choose $E_{L}, E_{R} \in \Psf^{-\infty, -\infty , -\infty , 2\hvor + 2\vol, 2\vob, -\infty}$ such that $\hat{N}_{\cf}( E_{L} ) = q_{\tindex}^2 E_{L}^{\tindex}$, $\hat{N}_{\cf}( E_{R} ) = q_{\tindex}^2 E_{R}^{\tindex}$, then we also have
\begin{equation}  \label{tangential away from o cal 8}
\begin{gathered}
\hat{N}_{\cf}( \varphi_{\dmf} E_{L} ) = \varphi_{\mathcal{C}^{\tindex}_{\infty}} q_{\tindex}^2 E_{L}^{\tindex} = q_{\tindex}^2 E_{L}^{\tindex}, \\
 \hat{N}_{\cf}( E_{R} \varphi_{\dmf} ) = q_{\tindex}^2 E_{R}^{\tindex} \varphi_{\mathcal{C}^{\tindex}_{\infty} } = q_{\tindex}^2 E_{R}^{\tindex}.
\end{gathered}
\end{equation} \par

Finally, we would like to choose some $L \in \Psf^{-\infty, -\infty, 2 \vol, 2 \vor + 2\vol + 2 \delta, 2 \vob, -\infty }$ such that
\begin{equation} \label{tangential away from o cal 8.25}
\hat{N}_{\cf}( L ) = q_{\tindex}^2 \varphi_{\Co} \hbff^{-\vol + \vob/2 - n^{\tindex}/4} L^{\tindex} \hbff^{ -\vol +  \vob/2 + n^{\tindex}/4 } \varphi_{\Co},
\end{equation} 
with $L$ still being positive indefinite with respect to the $L^{2}$ inner product. By the positivity of $L^{\tindex}$, one could take its square root, and write $L^{\tindex} = (B^{\tindex}_{2})^{2}$. However, it is easy to see that $L^{\tindex}$ cannot be elliptic  (again, in view of $L^{\tindex}$ being the logarithmic term in a standard symbolic calculation), thus one in general cannot show that $B^{\tindex}$ belongs to the small b-calculus. To get around this problem, we instead construct $L$ explicitly. \par

Let $\Lambda \in \Psf^{0, \vob/2, \vob/2, \vob, \vob, \vob/2}$ be some quantization of $\phi(x_{\tindex}) x_{\tindex}^{-\vob/2}$. Then we will explicitly define
\begin{equation}  \label{tangential away from o cal 8.5}
L = x_{\tindex}^{n^{\tindex}/4} \Lambda^{\ast} Q^{\ast} \varphi_{\Co} \hbff^{-\vol + \vob/2 - n^{\tindex}/4} L^{\tindex} \hbff^{-\vol + \vob/2 + n^{\tindex}/4} \varphi_{\Co} Q \Lambda x_{\tindex}^{-n^{\tindex}/4}.
\end{equation}
Notice that with the above definition, $L$ is indeed positive indefinite with respect to $L^2$. To show this, it would be enough to recall (\ref{3co and co inner product scaling 1}), and then make a straightforward computation using the positivity of $L^{\tindex}$ with respect to $L^{2}_{\mathrm{b}}( [ \hat{X}^{\tindex} ; \{ 0 \} ] )$, from which we have
\begin{equation*}
\langle Lu , u \rangle_{L^2} = \int_{[0,\infty) \times \mathcal{C}_{\tindex}} \langle L^{\tindex} v, v\rangle_{L^{2}_{\mathrm{b}}( [ \hat{X}^{\tindex} ; \{ 0 \} ] )} \frac{dx_{\tindex}}{x_{\tindex}^{n_{\tindex}+1}} dy_{\tindex} \geq 0,
\end{equation*}
where we are writing $v = \hbff^{-\vol + \vob/2 + n^{\tindex}/4} \varphi_{\Co} Q \Lambda x_{\tindex}^{-n^{\tindex}/4}u$. \par

It is not immediately obvious that (\ref{tangential away from o cal 8.5}) belongs to the second microlocal calculus. We will clarify this via the following lemma. In fact, it is not anymore difficult to consider slightly more general cases.
\begin{lemma}   \label{lemma construction of L decay at cf}
Suppose that 
\begin{equation*}
\begin{gathered}
l^{\tindex} \in \mathbb{R}, \quad \vom^{\tindex}, \vor^{\tindex} \in \mathcal{C}^{\infty}( \overline{ ^{\mathrm{co}}T^{\ast}} [ \hat{X}^{\tindex} ; \{ 0 \} ] ), \quad A^{\tindex} \in \Psi_{\mathrm{coc}, \delta}^{\vom^{\tindex} , \vor^{\tindex}, l^{\tindex} }( [ \hat{X}^{\tindex} ; \{ 0 \} ] ), \\
\vor, \vov \in \mathcal{C}^{\infty}( \psf \Xd ), \quad \vol, \vob \in \mathcal{C}^{\infty}( \mathcal{C}_{\tindex} \times \overline{\mathbb{R}^{n_{\tindex}}} ), \quad Q \in \Psf^{-\infty, \vor , \vol, \vov, \vob, -\infty} 
\end{gathered}
\end{equation*}
Assume additionally that $Q$ is supported near $\cf$, and moreover $\WFs(Q)$ is contained in a neighborhood of $\tcocf$ where the free variables are bounded (in particular, $\WFs(Q)$ is compactly contained away from fiber infinity). Then we have
\begin{equation} \label{tangential away from o cal 8.6}
\varphi_{\cf} A^{\tindex} Q \in \Psf^{-\infty, \vor + \vor^{\tindex} , \vol + l^{\tindex} , \vov + \vom^{\tindex} , \vob, -\infty }, \quad \hat{N}_{\cf}( \varphi_{\cf} A^{\tindex} Q ) = A^{\tindex} \hat{N}_{\cf}(Q),
\end{equation}
where $\varphi_{\cf} \in \mathcal{C}^{\infty}( \Xd )$ is some cut-off function at $\cf$. Here the variables orders $\vom^{\tindex}, \vor^{\tindex}$ make senses in (\ref{tangential away from o cal 8.6}) by being understood as arbitrary extensions $\mathcal{C}^{\infty}( \psf \Xd )$ extensions outside of the $\WFs(Q)$.
\end{lemma}
\begin{proof}[Proof of Lemma \ref{lemma construction of L decay at cf}]
Writing $Q$ as a partial quantization near $\cf$, it is clear that
\begin{equation}  \label{tangential away from o cal 8.6.1}
A^{\tindex} Q = \frac{1}{( 2 \pi )^{n_{\tindex}}} \int_{\mathbb{R}^{n_{\tindex}}}  e^{ i ( z_{\tindex} - z_{\tindex}' ) \cdot \zeta^{\mathrm{3co}}_{\tindex} } A^{\tindex}  \hat{Q} ( z_{\tindex}, \zeta^{\mathrm{3co}}_{\tindex} ) d\zeta_{\tindex}^{\mathrm{3co}},
\end{equation}
where $\hat{Q}$ is the operator-valued symbol of $Q$. 
To show the membership of (\ref{tangential away from o cal 8.6}), one needs to show that the kernel of $A^{\tindex} \hat{Q}$ belongs to the correct classes when restricted to various regions. In fact, since we are already restricted to a neighborhood of $\cf$, it is enough to consider all operators of the form $\phi^{\tindex} A^{\tindex} \hat{Q} \phi^{\tindex}, \phi^{\tindex} A^{\tindex} \hat{Q}  \psi^{\tindex}$, where $\phi^{\tindex}, \psi^{\tindex} \in \mathcal{C}^{\infty}( [ \hat{X}^{\tindex} ; \{ 0 \} ])$ are cut-off functions with disjoint supports. \par

It is straightforward (albeit a very mundane task) to show that the kernels of $\phi^{\tindex} A^{\tindex} \hat{Q} \psi^{\tindex}$ behave correctly. This amounts to using the smoothness of either one of the operators away from the diagonal, and the fact that the usual decay conditions near $\mathcal{C}^{\infty}_{0}$ and $\mathcal{C}^{\infty}_{\infty}$ get carried over to the composition. Moreover, as $A^{\tindex}$ does not depend on the free variables, the symbolic properties of $\hat{Q}$ in $( z_{\tindex}, \zeta_{\tindex}^{\mathrm{3co}} )$ are also carried over to the composition. \par

The only issue is if $A^{\tindex}$ is supported away from the diagonal while $\hat{Q}$ is supported near the diagonal (in the interaction variables). In this case, one in general does not have the required infinite orders of decay in $\langle \zeta_{\tindex}^{\mathrm{3co}} \rangle$. However, this is irrelevant since by assumption our $Q$ decays to infinite orders at fiber infinity.  \par
Thus, it remains to consider operators of the form $\phi^{\tindex} A^{\tindex} \hat{Q} \psi^{\tindex}$, with the problem being the composition of two quantizations, where singularities could occur. In general, one needs to check the calculations near both $\mathcal{C}^{\tindex}_{0}$ and $\mathcal{C}^{\infty}_{\infty}$. But since these calculations are analogous, we will focus only on the case near $\mathcal{C}^{\tindex}_{0}$ for definiteness (in particular, the case near $\mathcal{C}^{\tindex}_{0}$ would be enough for our purpose in this section, due to the presence of $\varphi_{\mathcal{C}^{\tindex}}$ in (\ref{tangential away from o cal 8.5})). \par
To this end, locally over $\mathcal{C}^{\tindex}_{0}$, let $a^{\tindex}$ denote the left symbol of $A^{\tindex}$ and $q$ the right symbol (in the interaction variables only) of $\hat{Q}$. Then we can write $A^{\tindex}\hat{Q}$ as
\begin{align*}
& \frac{1}{( 2 \pi )^{n^{\tindex}}} \int_{\mathbb{R}^{n^{\tindex}}}   e^{ - i ( \hat{t}_{\ff} - \hat{t}_{\ff}' ) \utaucob + i ( y^{\tindex} - (y^{\tindex})' ) \cdot \umucob } \\
& \qquad \times a^{\tindex} ( \hat{t}_{\ff}, y^{\tindex}, \utaucob, \umucob ) q ( z_{\tindex}, \zeta_{\tindex}^{\mathrm{3co}}, \hat{t}_{\ff}', (y^{\tindex})',  \utaucob, \umucob )   d\utaucob d\umucob
\end{align*}
where $\hbff = e^{-\hat{t}_{\ff}}$. Thus, the correctness of the first claim in (\ref{tangential away from o cal 8.6}) comes down to whether or not left symbol reduction can be applied to the above expression, such that the resulting symbol belongs to the correct class. \par
In fact, upon going through the usual procedure, which involves Taylor expanding about the diagonal $\{ \hat{t}_{\ff} = \hat{t}_{\ff}', y^{\tindex} = (y^{\tindex})'\}$, followed by taking an asymptotic summation, one can see that each term in the expansion belongs to the correct symbol classes. This is because every such term is the product with a symbol that is some derivatives of $q$, and that $a^{\tindex}$ can be viewed as a symbol conormal to $\psf\Xd$ where $q$ is microlocally non-trivial. Indeed, the only obstacle preventing $a^{\tindex}$ from automatically being conormal to $\psf \Xd$ near $\cf$ (in the spatial sense) is the lack of conormality at fiber infinity.   \par

Therefore, it remains to consider the remainder term $\hat{R}$ (i.e., the difference between the near-diagonal part of $A^{\tindex} \hat{Q}$ and the quantization of the aforementioned asymptotic summation). The standard method then allows us to write $\hat{R}$ as some b-quantization of a symbol $r$ in the interaction variables, where $r$ is a symbol residual in every symbolic sense. \par

Finally, the indicial operator property in (\ref{tangential away from o cal 8.6}) follows easily from (\ref{tangential away from o cal 8.6.1}). This concludes the proof of Lemma \ref{lemma construction of L decay at cf}.
\end{proof}
We now choose $Q \in \Psf^{-\infty,0,0,0,0,-\infty}$ to be some quantization of $q_{\tindex} \psi(\rho_{\tcocf})$, where $\psi \in \mathcal{C}^{\infty}( [0,\infty) )$ is a cut-off at $0$ and $\rho_{\tcocf}$ a boundary defining function for $\tcocf$. Then the required properties for $L$ follows easily from Lemma \ref{lemma construction of L decay at cf}. \par

By putting (\ref{tangential away from o cal 6}), (\ref{tangential away from o cal 7}) and (\ref{tangential away from o cal 8}) together with the above, we can now choose
\begin{equation*}
\tilde{E} = - B^{\ast} B - B_{2}^{\ast} B_2 - L  + \varphi_{\dmf} E_{L} + E_{R} \varphi_{\dmf}.
\end{equation*}
Then it follows from (\ref{corner face commutator formula 3}) that we have
\begin{equation*}
i [ P, A^{\ast} A] = -B^{\ast}B - B^{\ast}_1 B_1 + \tilde{B}^{\ast} ( \log x ) \tilde{B} - L  + \varphi_{\dmf} E_{L} + E_{R} \varphi_{\dmf} + \tilde{R}
\end{equation*}
where we recall that $\tilde{R} \in  \Psf^{-\infty, -\infty, 2 \vol , 2 \hvor + 2 \vol + 2 \delta, 2 \vob - 2 \delta, - \infty}$.
Thus, by paring the above in $L^{2}$ and then throwing away the negative indefinite terms, we get that
\begin{equation} \label{tangential away from o cal 9}
\| B u \|_{L^{2}}^2 \leq | \langle AP u , A u  \rangle_{L^{2}}  | + | \langle \varphi_{\dmf} E_{L} u, u \rangle_{L^2} | + | \langle E_{R} \varphi_{\dmf} u,  u \rangle_{L^2} | + | \langle \tilde{R} u, u \rangle_{L^2}  |.
\end{equation}
\par

We first connect the left hand side of (\ref{tangential away from o cal 9}) with (\ref{decay at cf main lemma inequality}). Note that since $B^{\tindex}$ is elliptic and positive, it must be invertible with inverse $(B^{\tindex})^{-1} \in \Psi_{\mathrm{bc}, \delta}^{-\hvor - \vol, 0, 0}( [ \hat{X}^{\tindex} ; \{ 0 \} ] )$. If we choose $\tilde{q}_{\tindex} \in \mathcal{C}^{\infty}_{c}( \mathcal{C}_{\tindex} \times \mathbb{R}^{n_{\tindex}} ), \tilde{\varphi}_{\Co} \in \mathcal{C}^{\infty}( [ \hat{X}^{\tindex} ; \{ 0 \} ] )$ such that they are identically $1$ on the supports of $q_{\tindex}$ and $\varphi_{\Co}$ respectively, then we can define an operator $G \in \Psf^{-\infty, -\infty, -\vol, -\hvor - \vol, -\vob, -\infty}$ with
\begin{equation*}
\hat{N}_{\cf}( G ) = \tilde{q}_{\tindex} \tilde{\varphi}_{\Co} \hbff^{\vol -\vob/2} \hbff^{-n^{\tindex}/4} ( B^{\tindex} )^{-1} \hbff^{n^{\tindex}/4}.
\end{equation*}
It is then obvious that
\begin{equation*}
\hat{N}_{\cf}( GB ) = \tilde{q}_{\tindex} q_{\tindex} \tilde{\varphi}_{\Co} \hbff^{\vol - \vob/2} \hbff^{-n^{\tindex}/4} ( B^{\tindex} )^{-1} B^{\tindex} \hbff^{n^{\tindex}/4} \hbff^{-\vol+\vob/2  } \varphi_{\Co} = q_{\tindex} \varphi_{\Co}.
\end{equation*}
Thus, since $Q \in \Psf^{-\infty, 0, 0, 0, 0, -\infty}$ is chosen such that $\hat{N}_{\cf}(Q) = q_{\tindex}$, if we also choose $\varphi_{\dff}$ such that $\varphi_{\mathcal{C}^{\tindex}_0}$ is exactly the restriction of $\varphi_{\dff}$ to $\cf$, then we have
\begin{equation*}
GB - \varphi_{\dff} Q \in \Psf^{-\infty, -\infty ,0 , 0, - \delta, -\infty}.
\end{equation*}
It follows that we can estimate
\begin{align} \label{tangential away from o cal 10}
\begin{split}
\| \varphi_{\dff} Q u \|_{ H_{\mathrm{d3sc,3co,res}}^{M, \ast , \vol, \hvor + \vol, \vob, S} } & \leq C ( \|  GB u \|_{ H_{\mathrm{d3sc,3co,res}}^{ M , N , \vol, \hvor + \vol, \vob, S} } + \|  u \|_{H_{\mathrm{d3sc,3co,res}}^{ M, N, \vol, \hvor + \vol, \vob - \delta, S }} ) \\
& \leq C ( \| B u \|_{L^2} + \| u \|_{H_{\mathrm{d3sc,3co,res}}^{ M, N, \vol, \hvor + \vol, \vob -  \delta, S }} ).
\end{split}
\end{align}
\par
We next consider the right hand side of (\ref{tangential away from o cal 9}). By using Young's inequality, we have
\begin{align*}
 | \langle APu, Au \rangle_{L^2} | & \leq \epsilon^{-2} \| APu \|_{ H_{\mathrm{d3sc,3co,res}}^{0, 0 , 0, 1/2, 1,0} }^2 + \epsilon^2 \| A u \|_{H_{\mathrm{d3sc,3co,res}}^{0, 0, 0, -1/2, -1, 0}}^2 \\
& \leq   \epsilon^{-2} C \| P u \|_{H_{\mathrm{d3sc,3co,res}}^{M, N , \vol, \hvor + \vol + 1 , \vob + 2 , S}}^2 + \epsilon^2 C \| u \|_{ H_{\mathrm{d3sc,3co,res}}^{M, N , \vol, \hvor + \vol, \vob, S} }^2.
\end{align*}
On the other hand, let $\tilde{\varphi}_{\dmf} \in \mathcal{C}^{\infty}(\Xd)$ be a cut-off at $\dmf$ such that $\tilde{\varphi}_{\dmf} = 1$ on the support of $\varphi_{\dmf}$. Note that the support of $\tilde{\varphi}_{\dmf}$ can be made arbitrarily small (since this is the case for $\varphi_{\dmf}$). Then by Lemma \ref{decay at corner face decay near main face lemma} as well as the decay properties of $E_{L}$ and $E_{R}$, we can estimate
\begin{align*}
| \langle \varphi_{\dmf} E_{L} u, u \rangle_{L^2} | & \leq \| \varphi_{\dmf} E_{L} u \|_{H_{\mathrm{d3sc,3co,res}}^{-M, -N + 1, - \vol, - \hvor - \vol, -\vob, - S}} \|  \tilde{\varphi}_{\dmf} u \|_{H_{\mathrm{d3sc,3co,res}}^{M, N - 1, \vol, \hvor + \vol, \vob, S} } \\
& \leq \epsilon^2 C \| u \|_{H_{\mathrm{d3sc,3co,res}}^{M, N, \vol, \hvor + \vol, \vob, S}}^2, 
\end{align*}
and likewise
\begin{equation*}
| \langle E_{R} \varphi_{\dmf} u, u \rangle_{L^2} | \leq \epsilon^2 C \| u \|_{H_{\mathrm{d3sc,3co,res}}^{M, N, \vol, \hvor + \vol , \vob, S}}^2.
\end{equation*}
Lastly, for the $\langle \tilde{R} u, u \rangle_{L^2}$ term, we estimate easily that
\begin{align*}
| \langle \tilde{R} u , u \rangle_{L^2} | & \leq \| \tilde{R} u \|_{H_{\mathrm{d3sc,3co,res}}^{- M, -N, - \vol, - \hvor - \vol - \delta, -\vob + \delta , - S}} \| u \|_{H_{\mathrm{d3sc,3co,res}}^{ M, N , \vol, \hvor + \vol + \delta, \vob - \delta, S}} \\
& \leq C \| u \|_{H_{\mathrm{d3sc,3co,res}}^{ M,  N , \vol, \hvor + \vol + \delta, \vob - \delta, S }}^2.
\end{align*} \par

Putting everything together, we now conclude that 
\begin{align*}
\| Bu \|_{L^2} & \leq C( \epsilon^{-1} \| P u \|_{H_{\mathrm{d3sc,3co,res}}^{M, N , \vol , \vor + \vol + 1 , \vob + 2 , S}} \\
&\quad +  \epsilon \| u \|_{H_{\mathrm{d3sc,3co,res}}^{M, N, \vol, \hvor + \vol, \vob, S}} + \| u \|_{H_{\mathrm{d3sc,3co,res}}^{M, N , \vol, \hvor + \vol + \delta, \vob - \delta, S}} ).
\end{align*}
This inequality can be further substituted into (\ref{tangential away from o cal 10}), so now we have
\begin{align*}
\| \varphi_{\dff} Q u \|_{H_{\mathrm{d3sc,3co,res}}^{M, \ast, \vol, \hvor + \vol, \vob, S}} & \leq C( \epsilon^{-1} \| P u \|_{H_{\mathrm{d3sc,3co,res}}^{M, N , \vol , \vor + \vol + 1 , \vob + 2 , S}} \\
&\quad +  \epsilon \| u \|_{H_{\mathrm{d3sc,3co,res}}^{M, N, \vol, \hvor + \vol, \vob, S}} + \| u \|_{H_{\mathrm{d3sc,3co,res}}^{M, N , \vol, \hvor + \vol + \delta, \vob - \delta, S}} ),
\end{align*}
which is exactly the required estimate (\ref{decay at cf main lemma inequality}).

\section{Decay estimate at the front face} \label{decay at the front face section}
In this section, we will start from Proposition \ref{decay at the corner face proposition} and show that the error term in (\ref{corner face decay equation 1}) can be further improved at $\dff$. This will be the final ingredient for the proof of Theorem \ref{intro; main theorem 2}, which will be presented in \S \ref{section 10 proof of main theorem subsection} below. \par

At least initially, in this section, we will still work with the assumption that $\mathcal{C}$ has just one element. However, unlike in the proofs of Propositions \ref{fredholm estimate with symbolic decay} and \ref{decay at the corner face proposition}, the calculations in this section do carry to the general case straightforwardly. This is because our analysis below will involve directly the \emph{normal operator} of $P$ at $\dff$, which we now define as
\begin{equation} \label{section 10 normal operator of P}
N_{\dff} ( P ) = \Delta + V^{\tindex} - \lambda^{2},
\end{equation}
i.e., just what the operator $P$ would look like if $\mathcal{C} = \mathcal{C}_{\tindex}$. We will prove an estimate for every $\tindex \in \ind$ with $P$ replaced by $N_{\dff}(P)$, and then at the end patch them together. Notice that we clearly have
\begin{equation*}
\hat{N}_{\dff}( N_{\dff} ( P ) ) = \hat{N}_{\dff}(P)
\end{equation*}
for every $\tindex \in \ind$.

\subsection{Invertibility of the indicial operator}
First of all, recall that we have
\begin{equation*}
\hat{N}_{\dff}(P) = \Delta_{z^{\tindex}} + V^{\tindex} - ( \lambda^{2} - | \tau_{\tindex} |^{2} - | \mu_{\tindex} |_{h_{\tindex}}^2 ) = \Delta_{z^{\tindex}} + V^{\tindex} - ( \lambda^{2} - |\zeta_\alpha|^2) ).
\end{equation*}
As mentioned already, one can view $\hat{N}_{\dff}(P)$ as a family of stationary `two-body' operators. Such operators have been extensively studied in the literature. In particular, they are shown to be invertible between suitable anisotropic Sobolev spaces in a uniform sense. The precise statements can be found in \cite[Theorem 4.5, but see also Theorem 1.1, Remark 1.2 and Remark 1.4]{AndrasSM}, which will be the cornerstones for the rest of this section. \par

In stating a result which best fits our current context, we will rephrase \cite[Theorem 4.5]{AndrasSM} by the following:
\begin{theorem} \label{Andras second microlocal main theorem}      
Suppose that $s \in \mathbb{R}$, and let $l \in \mathbb{R}$ be such that $|l+1| < (n^{\tindex} - 2)/2$. Let also $\vor^\tindex_{\pm} \in \mathcal{C}^{\infty}( {^{\mathrm{b}}S^{\ast} X^{\tindex}} )$ be given by 
\begin{equation} \label{section 10 usual variable order}
\vor^{\tindex}_{\pm} = -\frac1{2} \pm \beta^{\tindex} \frac{\utau}{( \intn )^{1/2}}, \quad \beta^{\tindex} > 0. 
\end{equation}
Then the maps
\begin{equation} \label{Andras second microlocal main map}
\hat{N}_{\dff}(P)(\zeta_\alpha) : H_{\mathrm{sc,b}}^{s, \vor^\tindex_{\pm}, l}(X^{\tindex}) \rightarrow \big\{ u \in H_{\mathrm{sc,b}}^{s, \vor^\tindex_{\pm}, l}(X^{\tindex}) : \hat{N}_{\dff}(P)(\zeta_\alpha) u \in  H_{\mathrm{sc,b}}^{s - 2, \vor^\tindex_{\pm} + 1 , l + 2}(X^{\tindex})  \big\}
\end{equation}
are uniformly invertible for all $\zeta_{\tindex} \in \mathbb{R}^{n_{\tindex}}$. Here, uniformity means that
\begin{equation} \label{Vasy uniform Fredholm estimates two-body}
\| u \|_{H_{\mathrm{sc,b}}^{s, \vor^\tindex_{\pm} , l}} \leq C \| \hat{N}_{\dff}(P) (\zeta_{\tindex}) u \|_{H_{\mathrm{sc,b}}^{s-2, \vor^\tindex_{\pm} + 1, l + 2}}
\end{equation}
for a constant $C>0$ that is independent of $u \in H_{\mathrm{sc,b}}^{s, \vor^\tindex_{\pm}, l}$ and $\zeta_\alpha$.
\end{theorem}
\begin{proof} 
To conserve notations, we will use $\vor^{\tindex}$ for either choices of $\vor^{\tindex}_{+}$ or $\vor^{\tindex}_{-}$ in this proof. 

  The result for bounded $|\zeta_\tindex|$ is \cite[Theorem 5.7]{AndrasSM}. We prove the estimate, including uniformity, for all $\zeta_\tindex \in \mathbb{R}^{n_{\tindex}}$. For sufficiently large $|\zeta_\tindex|$, observe that  $\hat{N}_{\dff}(P)(\zeta_\alpha)$ is invertible in the scattering calculus, and therefore also has inverse in the scattering calculus. We claim that this inverse is uniformly bounded as $|\zeta_\tindex| \to \infty$.

  To prove this, we first set
  \begin{equation} \label{section 10 Andras theorem cal 1}
 A(\sigma) = \frac{1}{(2\pi)^{n^{\tindex}}} \int_{\mathbb{R}^{n^{\tindex}}} e^{ i ( z^{\tindex} - (z^{\tindex})' ) \cdot \zeta^{\tindex} } ( |\zeta^{\tindex}| + V + \sigma )^{-1} d\zeta^{\tindex} , \quad \sigma = |\zeta_\alpha|^2 - \lambda^2. 
  \end{equation}
It is not difficult to see that, for all sufficiently large $\sigma$ so that the symbol is greater than $1$, this is a  uniformly bounded (as $\sigma$ tends to infinity) family of scattering symbols of order $(-2,0)$. That it is a symbol of order $(-2, 0)$ is clear, so only the uniformity needs addressing. \par

This uniformity follows from the following observation: When we take any derivative of the symbol, we obtain a quotient where the numerator consists of a product of derivatives of $|\zeta^\alpha|^2 + V$, and thus is independent of $\sigma$, while the denominator is a positive integral power of $|\zeta^\alpha|^2 + V + \sigma$. From this, it is easy to check that the symbol estimates hold uniformly as $\sigma \to \infty$. 
  It is also easy to check that as a symbol of order $(0,0)$, resp. $(-1, 0)$, it is $O(\sigma^{-1})$, resp. $O(\sigma^{-1/2})$, as $\sigma \to \infty$.

We now consider 
  $$
  (\Delta_{z^\alpha} + V + \sigma) A(\sigma) = I + E(\sigma),
  $$
  where this identity defines $E(\sigma)$. By standard elliptic theory, $E(\sigma)$ is a scattering pseudodifferential operator of order $(-1, -1)$. However, as a symbol of order $(0,0)$, all its seminorms decay as $O(\sigma^{-1})$. Since the operator norm is controlled by a finite number of seminorms in $S^{0,0}$, it follows that the operator norm of $E(\sigma)$ is $O(\sigma^{-1})$, and is thus bounded by $1/2$ for sufficiently large $\sigma$. Therefore, $I + E(\sigma)$ is invertible for large $\sigma$, and it follows that 
  \begin{equation*}
      (\Delta_{z^\alpha} + V + \sigma)^{-1} = A(\sigma) ( I + E(\sigma))^{-1} \in \Psi_{\mathrm{sc}}^{-2,0}(X^{\tindex})
  \end{equation*}
has uniformly bounded seminorms in this space. This inverse, which is $(\hat{N}_{\dff}(P)(\zeta_\tindex))^{-1}$, is uniformly bounded as a map between scattering Sobolev spaces $H_{\mathrm{sc}}^{s, \vor} \to H_{\mathrm{sc}}^{s+2, \vor}$ (including variable order spaces).

The function $\vor^{\tindex}$ is not an admissible variable order at zero frequency in the scattering calculus. However, it is an admissible variable order in the $\mathrm{b}$-calculus. Thus, to prove the uniform boundedness, we will require a second ingredient with respect to b-Sobolev spaces. It follows from Vasy \cite[Theorem 4.5, Remark 1.4]{AndrasSM} that if we assume $u \in H_\mathrm{b}^{\vor^\tindex  -l, l}$ and $\hat{N}_{\dff}(P)(\zeta_\tindex) u \in H_{\mathrm{b}}^{\vor^\tindex -l-1, l+2}$, then there is a uniform bound 
\begin{equation*}
 \| u \|_{H_\mathrm{b}^{\vor^\tindex  -l, l}}   \leq C \| \hat{N}_{\dff}(P)(\zeta_\tindex) u \|_{H_\mathrm{b}^{\vor^\tindex  -l-1, l+2}}.
\end{equation*}
The constant $C$ here is independent of $\zeta_\tindex$. Indeed, this estimate is proved via a positive commutator estimate, where the commutator is independent of $\zeta_\tindex$, leading immediately to the uniformity.

Now, to combine these two estimates, we choose a microlocal cutoff $Q \in \Psi_{\mathrm{sc,b}}^{0,0,-\infty}$ so that $Q$ is microlocally equal to zero in a neighborhood of the b-face, and microlocally equal to the identity in a neighborhood of the fiber infinity. It follows that $[\hat{N}_{\dff}(P)(\zeta_\tindex), Q] \in \Psi_{\mathrm{sc,b}}^{-\infty, -1, -\infty}$. Thus, the commutator lies in \emph{both} the scattering calculus, as an operator of order $(-\infty, -1)$, and in the b-calculus, as an operator of order $(-1, -\infty)$. Now we take $u \in H_{\mathrm{sc,b}}^{s, \vor^\tindex, l}$ and then decompose it as $u = Qu + ( I - Q)u$. We also let $\tilde \vor^\tindex$ be a variable order in the scattering calculus which is equal to  $\vor^\tindex$ on the wavefront set of $Q$, with $\tvor^{\tindex} \leq \vor^{\tindex} + 1$ everywhere, and moreover $\tvor^{\tindex}$ is constant, with $\tvor^{\tindex} \leq l$ in a neighborhood of the zero section over the boundary. Then we have the estimates 
\begin{equation*}
\begin{gathered}
 \| Qu \|_{H_{\mathrm{sc,b}}^{s, \vor^\tindex, l}} \leq C  \| Qu \|_{H_{\mathrm{sc}}^{s, \tilde \vor^\tindex}} + \| u \|_{H_{\mathrm{sc}}^{N, N}} , \\ 
\| ( I - Q) u \|_{H_{\mathrm{sc,b}}^{s, \vor^\tindex, l}} \leq C \| ( I - Q)u \|_{H_{\mathrm{b}}^{\vor^\tindex, l}} + \| u \|_{H_{\mathrm{sc}}^{N, N}}, 
\end{gathered}
\end{equation*}
where $N$ is sufficiently negative. On each piece, we apply our uniform estimate on sc-, resp. b-Sobolev spaces, from which we obtain
\begin{equation*}
\begin{gathered}
\| Qu \|_{H_{\mathrm{sc}}^{s, \tilde \vor^\tindex}} \leq C \| \hat{N}_{\dff}(P)(\zeta_\tindex) Qu \|_{H_{\mathrm{sc}}^{s-2, \tilde \vor^\tindex+1}}, \\
\| ( I  - Q)u \|_{H_{\mathrm{b}}^{\vor^\tindex-l, l}} \leq C \| \hat{N}_{\dff}(P)(\zeta_\tindex) ( I - Q)u \|_{H_{\mathrm{b}}^{\vor^\tindex-l-1, l+2}}. 
\end{gathered}
\end{equation*}
 \par

We now commute the normal operator past the microlocal cutoffs. This yields
\begin{align*}\begin{split} 
 \| \hat{N}_{\dff}(P)(\zeta_\tindex) Qu \|_{H_{\mathrm{sc}}^{s-2, \tilde \vor^\tindex+1}} & \leq \| Q\hat{N}_{\dff}(P)(\zeta_\tindex) u \|_{H_{\mathrm{sc}}^{s-2, \tilde \vor^\tindex+1}} \\ 
 & \quad + \| [\hat{N}_{\dff}(P)(\zeta_\tindex), Q] u \|_{H_{\mathrm{sc}}^{s-2, \tilde \vor^\tindex+1}}, 
 \end{split}
 \end{align*}
 as well as
 \begin{align}
 \begin{split} \label{section 10 Andras theorem commute 2}
 \| \hat{N}_{\dff}(P)(\zeta_\tindex) ( I - Q)u \|_{H_{\mathrm{b}}^{\vor^\tindex-l-1, l+2}} & \leq \| ( I - Q) \hat{N}_{\dff}(P)(\zeta_\tindex) u \|_{H_{\mathrm{b}}^{\vor^\tindex-l-1, l+2}} \\ 
 & \quad + \| [\hat{N}_{\dff}(P)(\zeta_\tindex), ( I - Q)]u \|_{H_{\mathrm{b}}^{\vor^\tindex-l-1, l+2}}. 
\end{split}
\end{align} 
By the membership of $Q$, it is clear that
\begin{equation*}
\begin{gathered}
\|  Q \hat{N}_{\dff} (P) ( \zeta_{\tindex} ) u \|_{H_{\mathrm{sc}}^{ s - 2, \tvor^{\tindex} + 1 }} \leq C  \| \hat{N}_{\dff} (P) ( \zeta_{\tindex} ) u \|_{H_{\mathrm{sc,b}}^{s-2, \vor^{\tindex} + 1, \ast }}, \\
\| ( I - Q ) \hat{N}_{\dff}(P) ( \zeta_{\tindex} ) u \|_{H_{\mathrm{b}}^{\vor^{\tindex} - l - 1, l + 2}} \leq C \| \hat{N}_{\dff}(P) ( \zeta_{\tindex} ) u  \|_{H_{\mathrm{sc,b}}^{\ast, \vor^{\tindex} + 1, l + 2}}.
\end{gathered}
\end{equation*}
Likewise, we can easily estimate
\begin{equation*}
\| [\hat{N}_{\dff}(P)(\zeta_\tindex), Q] u \|_{H_{\mathrm{sc}}^{s-2, \tvor^{\tindex} + 1, } } \leq C \| u \|_{H_{\mathrm{sc}}^{\ast, \tvor^{\tindex}} }.
\end{equation*}
To estimate the second commutator term (i.e., the one in the second line of (\ref{section 10 Andras theorem commute 2})), we recall that $[\hat{N}_{\dff}(P)(\zeta_\tindex), ( I - Q)] = - [ \hat{N}_{\dff}(P)( \zeta_{\tindex} ), Q ]$ is microlocally supported away from the b-face, so it can be seen as a scattering operator. Hence the b-Sobolev space can be swapped for a sc-Sobolev space, up to an error that is completely residual. It follows that we have
\begin{equation*}
\| [ \hat{N}_{\dff}(P) (\zeta_{\tindex}), ( I - Q ) ] u \|_{H_{\mathrm{b}}^{\vor^{\tindex} - l -1, l +2} } \leq C \| u \|_{H_{\mathrm{sc}}^{N, \tvor^{\tindex}}} 
\end{equation*}
as well.  \par

Finally, to estimate the residual terms, it suffices to apply our uniform estimate on scattering Sobolev spaces, from which we get that
\begin{equation}
    \| u \|_{H_{\mathrm{sc}}^{N, \tilde \vor^\tindex}} \leq \| \hat{N}_{\dff}(P)(\zeta_\tindex) u \|_{H_{\mathrm{sc}}^{N-2, \tilde \vor^\tindex}},
\end{equation}
  and this suffices to complete the estimate by the construction of $\tilde{\mathsf{r}}^{\tindex}$.
\end{proof}

Theorem~\ref{Andras second microlocal main theorem} can be generalized to variable orders other than the specific $\vor^\tindex_{\pm}$ stated in the theorem. For example, let $\tvor^{\tindex}_{\pm}$ denote variable orders satisfying the usual monotonicity and threshold conditions, i.e., 
 \begin{equation} \label{Andras second microlocal variable order requirement -1}
 \begin{gathered}
 \text{$\pm \tvor^{\tindex}_{\pm}$ is non-increasing along the flow of $ (\intn)^{-1/2} (x^{\tindex})^{-1} H_{\intn}$} \\
 \text{when restricted to ${ ^{\mathrm{b}}S^{\ast} X^{\tindex} }$,}
 \end{gathered}
 \end{equation}
 as well as
 \begin{equation} \label{Andras second microlocal variable order requirement}
\begin{gathered}
\text{$ \pm \tvor^{\tindex}_{\pm} < - \frac{1}{2}$ at $ \Big\{  \frac{\utau}{( \intn )^{1/2}}$ = $\frac{\utaub}{( \intnb )^{1/2}} = -1$ \Big\}}, \\
\text{$ \pm \tvor^{\tindex}_{\pm} > - \frac{1}{2}$ at $ \Big\{  \frac{\utau}{( \intn )^{1/2}}$ = $\frac{\utaub}{( \intnb )^{1/2}} = 1$ \Big\}},
\end{gathered}
\end{equation}
 and also $\tvor^{\tindex}_{\pm} \geq \vor^\tindex_{\pm}$. Then we claim that the result also holds with $\tvor_{\pm}^{\tindex}$ replacing $\vor^\tindex_{\pm}$.  \par

To prove this, we only need to address the b-Sobolev estimate above, since the scattering result already holds with any scattering variable order. In terms of the b-Sobolev estimate, we first prove a uniform estimate with error term:
 \begin{equation}
\| u \|_{H_{\mathrm{b}}^{\tvor^{\tindex}_{\pm} -l, l}}   \leq C  \big( \| \hat{N}_{\dff}(P)(\zeta_\tindex) u \|_{H_{\mathrm{b}}^{\tvor^{\tindex}_{\pm} -l-1, l+2}} + \| u \|_{H_{\mathrm{b}}^{N, l}} \big),
 \label{eq:b Sobolev est general r}\end{equation}
where $N$ is sufficiently negative. This can be done with the standard method, so we only sketch the argument. We note that we have to estimate $u$ in three microlocal regions: 
\begin{itemize}
\item near the boundary and localized to finite b-momentum, in which case this is trivially estimated by the term $\| u \|_{H_{\mathrm{b}}^{N, l}}$ with a fixed constant; 
\item away from the boundary, in which case we get a uniform bound using the same argument as was used in the scattering calculus argument below (\ref{section 10 Andras theorem cal 1}); 
\item near the boundary and near infinite b-momentum, which is done using positive commutator estimates, with constants independent of $|\zeta_\tindex|$ because the commutators are independent of $\zeta_\tindex$. 
\end{itemize}
Once \eqref{eq:b Sobolev est general r} has been established, we simply estimate the error term on the right hand side using Theorem~\ref{Andras second microlocal main theorem}, and this suffices due to the assumption that $\vor^\tindex_{\pm} \leq \tvor^{\tindex}_{\pm}$. 

 Next, we also prove the result for a variable order $\tvor^{\tindex}_{\pm}$ satisfying the usual monotonicity and threshold conditions, and satisfying the opposite inequality, $\tvor^{\tindex}_{\pm} \leq \vor^\tindex_{\pm}$. This follows directly from duality and from the result just noted when the variable order is \emph{greater} than $\vor^\tindex_{\pm}$, recalling that the operator norm of an operator is equal to the operator norm of its adjoint. 

 We will record our findings above as a corollary:

\begin{corollary} \label{Andras second microlocal main corollary}
Suppose that $s \in \mathbb{R}$, and let $l \in \mathbb{R}$ be such that $|l + 1| < ( n^{\tindex} - 2 )/2$. Let also $\tvor^{\tindex}_{\pm} \in \mathcal{C}^{\infty}( {^{\mathrm{b}}S^{\ast}} X^{\tindex}  )$ be such that either $\tvor_{\pm}^{\tindex} \geq \vor^{\tindex}_{\pm}$ or $\tvor_{\pm}^{\tindex} \leq \vor^{\tindex}_{\pm}$, where $\vor^{\tindex}_{\pm}$ are given by (\ref{section 10 usual variable order}). Then the same statements (\ref{Andras second microlocal main map}) and (\ref{Vasy uniform Fredholm estimates two-body}) hold with $\vor^{\tindex}_{\pm}$ replaced by $\tvor^{\tindex}_{\pm}$.
\end{corollary}

\begin{remark}
In fact, we expect the results of Theorem \ref{Andras second microlocal main theorem} to work for all choices of $\tvor^{\tindex} \in \mathcal{C}^{\infty}( ^{\mathrm{b}}S^{\ast}X^{\tindex} )$ which satisfy either one of the monotonicity/threshold conditions in (\ref{Andras second microlocal variable order requirement -1}) and (\ref{Andras second microlocal variable order requirement}). The proofs should follow straightforwardly, possibly with minor modifications, from the proof of \cite[Theorem 5.7]{AndrasSM}. 
\end{remark}

In order to directly apply Corollary \ref{Andras second microlocal main corollary}, one needs a variable order which can at least locally (in phase space) be identified as an element of $\mathcal{C}^{\infty}( ^{\mathrm{b}}S^{\ast} X^{\tindex} )$. As we will see shortly, such a role is taken by $\vor - \vol$. Recall again (either from \S \S \ref{variable order construction section} or \ref{adaption of vasy argument subsection}) that when restricted to $\dtsccf$, the variable orders $\vor = \vor_{+}$, $\vol = \vol_{+}$ satisfy
\begin{equation*}
\vor = - \frac{1}{2} + \varphi \beta^{\tindex} \phi ( ( \beta^{\tindex} )^{-1} \vor^{\tindex} ) + \vol, \quad \vor^{\tindex} = \vor^{\tindex}_{+} = \beta^{\tindex} \frac{\utau}{( \intn )^{1/2}}, \quad \beta^{\tindex} > 0,
\end{equation*}
where $\varphi \in \mathcal{C}^{\infty}( \psf \Xd )$ is a cut-off at $\bcv$ and $\phi_1 \in \mathcal{C}^{\infty}(\mathbb{R})$. Specifically, $\varphi = \psi (p)$, where $\psi \in \mathcal{C}^{\infty}(\mathbb{R})$ is a cut-off at $0$ and $p$ is the principal symbol of $P$, while $\phi' \geq 0$ with $\phi(t) < 0$ constant near $-1$, $\phi (t) > 0$ constant near $1$, and 
\begin{equation} \label{front face decay cal 0.68}
t + \frac{1}{2} \leq \phi (t) \leq t + 1.
\end{equation}
\par

It would therefore be convenient if we can remove the dependency of $\vor - \vol$ on the free variables in a manageable way, and which only needs to hold on some small neighborhood of $\psf_{\dff} \Xd$, say $U$, where the free variables are also bounded. One simple way to do this is to consider
\begin{equation*}
\tvor^{\tindex} = - \frac{1}{2} + \beta^{\tindex} \phi ( (\beta^{\tindex})^{-1}  \vor^{\tindex} ).
\end{equation*}
Notice that $\vor^{\tindex}$ is indeed smooth on $U$, thus $\tvor^{\tindex}$ must be smooth on $U$ as well. Moreover, $\tvor^{\tindex}$ restricts to a $\mathcal{C}^{\infty}( ^{\mathrm{b}}S^{\ast} X^{\tindex} )$ function on $\dtsccf$, and by the construction of $\phi$, both the monotonicity condition (\ref{Andras second microlocal variable order requirement -1}) and the threshold conditions (\ref{Andras second microlocal variable order requirement}) are satisfied. \par

In fact, it is easy to see that
\begin{equation*}
\tvor^{\tindex} - \beta^{\tindex} \leq \vor^{\tindex} \leq \tvor^{\tindex} - \frac{\beta^{\tindex}}{2}.
\end{equation*}
Thus one can apply Corollary \ref{Andras second microlocal main corollary} for the variable order  $\tvor^{\tindex} - 2\delta$ if $\delta > 0$ is chosen such that
\begin{equation} \label{front face decay cal 0.69}
 \beta^{\tindex} \leq 2 \delta,
\end{equation}
since in this case we have $\tvor^{\tindex} - 2 \delta \leq \tvor - \beta^{\tindex}$. Recall that we can always take $\beta^{\tindex}$ to be arbitrarily small. Another simple comparison that we can make is
\begin{equation}  \label{front face decay cal 0.7}
\vor - 2 \beta^{\tindex} \leq \tvor^{\tindex} + \vol \leq \vor + 2 \beta^{\tindex}
\end{equation}  
Here, the factor $2$ comes from the right hand side of (\ref{front face decay cal 0.68}).  \par
A similar calculation works for the case with $\vor_{-}$ and $\vol_{-}$.
\par

In practice, it will be convenient to change the free variables from $(y_{\tindex}, \ltau, \lmu)$ to $( z_{\tindex}, \zeta_{\tindex} )$. This is because since $\intt = |\zeta_{\tindex}|^2$ (a crucial property of the translation invariance of the Euclidean metric), one sees that $\hat{N}_{\dff}(P)$ is dependent only on $|\zeta_{\tindex}|^2$, and can thus be viewed as a Fourier multiplier in the free variables. Namely, let $\mathcal{F}_{z_{\tindex} \rightarrow \zeta_{\tindex}}$ denote the Fourier transform which sends $z_{\tindex}$ to its dual variable $\zeta_{\tindex}$. Then we have
\begin{equation}  \label{front face decay cal 0.8}
N_{\dff}(P) = \mathcal{F}_{ z_{\tindex} \rightarrow \zeta_{\tindex} }^{-1} \hat{N}_{\dff}(P) \mathcal{F}_{ z_{\tindex} \rightarrow \zeta_{\tindex} },
\end{equation}
where we recall that
\begin{equation*}
N_{\dff}(P) = \Delta + V^{\tindex} - \lambda^2.
\end{equation*}
This property is essential for the analysis in the subsections below.

\subsection{An interpolation inequality} \label{section 10 interpolation inequality subsection}
It is very convenient to use interpolation. Thus in this subsection, we will state a crude interpolation result which can be proved just using $L^2$ methods, and which suffices for our needs. We will interpolate in all exponents except for the exponent at $\dff$, which will be fixed throughout.

\begin{lemma} \label{interpolation lemma section 10}
Let $u \in H^{m, \vor, \vol, \vor + \vol, 2\vol + b, s}_{\mathrm{d3sc, 3co, res}}$. Assume that $b \in \mathbb{R}$ satisfies $|b+1| < (n^\tindex -2)/2$, and that the order at $\dtsccf$ is larger than the order at $\tcocf$, i.e., $\vor  \geq  \vol + b$. 
Then for any $\delta > 0$, $N > \delta$ and $\epsilon > 0$, there exists $C_{\epsilon}$ such that 
\begin{equation} \label{section 10 interpolation 1}
\| u \|_{H^{m - \delta, \vor - \delta, \sfl , \sfr + \sfl - \delta, 2\sfl + b - \delta, s - \delta}_{\mathrm{d3sc, 3co, res}}}
\leq \epsilon \| u \|_{H^{m, \sfr, \sfl, \sfr + \sfl , 2\sfl + b , s}_{\mathrm{d3sc, 3co, res}}} + 
C_{\epsilon} \| u \|_{H^{m - N, \sfr - N, \sfl, \sfr + \sfl - N, 2\sfl + b - N, s - N}_{\mathrm{d3sc, 3co, res}}}. 
\end{equation} 
\end{lemma}

Before we proceed to the proof of Lemma \ref{interpolation lemma section 10}, let us first remark on its relevancy to us. Let $\vor_{\pm}$, $\vol_{\pm}$, $\vob_{\pm}$ be the variable orders constructed in \S \ref{variable order construction section}, and let also $b_{\pm} = \vob_{\pm} - 2 \vol_{\pm}$. Then Lemma \ref{interpolation lemma section 10} can only be applied provided that
\begin{equation} \label{section 10 interpolation subsection cal 1}
\vor_{\pm} \geq \vol_{\pm} + b_{\pm}.
\end{equation}
We consider first the case where the signs of the variable orders are $+$. By writing out the precise definitions of these functions, we find that condition (\ref{section 10 interpolation subsection cal 1}) is equivalent to 
\begin{equation*} \label{section 10 interpolation subsection cal 1.1}
\frac{1}{2} + \varphi \beta^{\tindex} \phi ( ( \beta^{\tindex} )^{-1} \vor^{\tindex} ) \geq b_{+} + 1,
\end{equation*}
where the functions $\varphi$, $\phi$, etc. have been recalled above. Thus, if we take the parameter $\beta^{\tindex}$ to be sufficiently small, then a non-trivial interval in the $b_{+}$ parameter, centered at $b_{+} = -1$, satisfies the condition $\vor_{+} \geq \vol_{+} + b_{+}$ (as well as $|b_{+} + 1| < ( n^{\tindex} - 2 )/2$). \par

Similarly, in the cases where the variable orders have $-$ signs, recall that we have
\begin{equation*}
\vor_{-} = - \vor_{+} - 1, \, \vol_{-} = - \vol_{+} , \, \vob_{-} = - \vob_{+} - 2, \, b_{-} = - b_{+} - 2.
\end{equation*}
Thus (\ref{section 10 interpolation subsection cal 1}) is equivalent to
\begin{equation*} \label{section 10 interpolation subsection cal 1.2}
- \frac{1}{2} + \varphi \beta^{\tindex} \phi ( (\beta^{\tindex})^{-1} \vor^{\tindex} ) \leq b_{+} + 1,
\end{equation*}
and the discussion above again applies in this case as well. 
\begin{proof}[Proof of Lemma \ref{interpolation lemma section 10}]
It suffices to prove the result when $N = 2^k \delta$ for some positive integer $k$. We can interpolate first in the orders at the `symbolic' faces, followed by an interpolation at the two corner faces $\dtsccf$ and $\tcocf$ (where measurement at the latter face should be interpreted more globally at $\cf$) simultaneously.

Inductively, it suffices to show that, at any particular face or combination of faces{\ep}let us take the fiber infinity for definiteness{\ep}we have
\begin{equation}
\| u \|_{H^{m - \delta, \sfr , \sfl, \sfr + \sfl , 2\sfl + b, s }_{\mathrm{d3sc, 3co, res}}}
\leq  C ( \| u \|^{1/2}_{H^{m, \sfr, \sfl, \sfr + \sfl , 2\sfl + b , s}_{\mathrm{d3sc, 3co, res}}}   
 \| u \|^{1/2}_{H^{m - 2\delta, \sfr, \sfl, \sfr + \sfl, 2\sfl + b, s}_{\mathrm{d3sc, 3co, res}}} + \| u \|_{H^{m - N, \sfr , \sfl, \sfr + \sfl , 2\sfl + b, s }_{\mathrm{d3sc, 3co, res}}} ).
\label{eq: interp first step}\end{equation} 
By iterating this $k$ times, we arrive at the inequality
\begin{equation}
\| u \|_{H^{m - \delta, \sfr , \sfl, \sfr + \sfl , 2\sfl + b, s }_{\mathrm{d3sc, 3co, res}}}
\leq  C ( \| u \|^{1- 2^{-k}}_{H^{m, \sfr, \sfl, \sfr + \sfl , 2\sfl + b , s}_{\mathrm{d3sc, 3co, res}}}  
 \| u \|^{2^{-k}}_{H^{m - 2^k \delta, \sfr, \sfl, \sfr + \sfl, 2\sfl + b, s}_{\mathrm{d3sc, 3co, res}}} ).
\label{eq: interp 2nd step} \end{equation} 
 Then the inequality follows from Young's inequality together with the `Peter-Paul' trick.

We continue to work at fiber infinity, although the argument is identical at any symbolic face. To this end, choose an operator $A \in \Psf^{m - \delta, \sfr , \sfl, \sfr + \sfl , 2\sfl + b, s}$ that is elliptic at fiber infinity. Then we can find a parametrix $B$ such that $R = BA - I \in \Psf^{-N, 0, 0, 0, 0, 0}$ using the usual elliptic construction at fiber infinity. It follows that  
\begin{equation}
\| u \|_{H^{m - \delta, \sfr , \sfl, \sfr + \sfl , 2\sfl + b, s }_{\mathrm{d3sc,3co,res}}}
\leq C (  \| Au \|_{L^2} + \| u \|_{H^{m - \delta - N, \sfr , \sfl, \sfr + \sfl , 2\sfl + b, s }_{\mathrm{d3sc,3co,res}}} ). 
\label{eq: interp first step 2}
\end{equation} 
Thus 
\begin{equation}
\| u \|^2_{H^{m - \delta, \sfr , \sfl, \sfr + \sfl , 2\sfl + b, s }_{\mathrm{d3sc, 3co, res}}}
\leq C (  \ang{ Au, Au}_{L^2} + \| u \|^2_{H^{m - \delta - N, \sfr , \sfl, \sfr + \sfl , 2\sfl + b, s }_{\mathrm{d3sc,3co,res}}} ).
\label{eq: interp first step 3}\end{equation} 
We similarly choose an operator $A_\delta \in \Psf^{\delta, 0, 0, 0, 0, 0}$ that is elliptic at fiber infinity. Let $B_\delta$ be a parametrix for $A_\delta$ similar to the above, with $B_\delta A_\delta - I = R_\delta \in \Psf^{-2N, 0, 0, 0, 0, 0}$. Then 
\begin{align}
\begin{split}
\label{eq:L2 interp}
\ang{ Au, Au}_{L^2} & = \ang{(B_\delta A_\delta - R_\delta) Au, Au}_{L^2}   = \ang{A_\delta Au, B_\delta ^* Au}_{L^2}  - \ang{R_\delta Au, Au}_{L^2} \\
& \leq C (  \| u \|_{H^{m, \sfr, \sfl, \sfr + \sfl , 2\sfl + b , s}_{\mathrm{d3sc,3co,res}}}   
 \| u \|_{H^{m - 2\delta, \sfr, \sfl, \sfr + \sfl, 2\sfl + b, s}_{\mathrm{d3sc,3co,res}}} + \| u \|^2_{H^{m - N, \sfr , \sfl, \sfr + \sfl , 2\sfl + b, s }_{\mathrm{d3sc,3co,res}}} ),
\end{split}
\end{align}
where the second line follows by Cauchy-Schwarz and by using boundedness on the second microlocalized, d3sc,3co,res-Sobolev spaces. Together, this and \eqref{eq: interp first step 3}  imply \eqref{eq: interp first step}.

After interpolating at the symbolic faces, we can reduce to the case of $s=m+ \vol$, $\vor = r+\vol$, where $r$ is constant, and that the order at $\dtsccf$ is the same as the order at $\tcocf$, namely, $2\vol + b$. 
The last condition is needed as we will choose below an operator with invertible indicial operator at $\tcocf$, valued in multiples of the identity. This imposes the condition that the orders at $\dtsccf$ and $\tcocf$ agree.

Consider the operator $x_{\dmf}^{-r} x_{\dff}^{-b} \Lambda_{\vol}$, where $x_{\dmf}, x_{\dff} \in \mathcal{C}^{\infty}(\Xd)$ are respectively boundary defining functions for $\dmf$ and $\dff$, and where $\Lambda_{\vol}$ is a quantization of the symbol $|z|^{\sfl}$. Then $\Lambda_{\vol} \in \Psf^{0, \vol, \vol, 2\vol, 2\vol, \vol}$, and its indicial operators at both $\cf$ and $\dff$ are equal to the identity. \par 
By a direct calculation, we find that 
\begin{equation}
\Lambda_{\vol}  \Lambda_{-\vol} - I  \in \Psf^{0, 0, -1 + \delta, -1+\delta, -2+\delta, -1+\delta}.
\end{equation}
We can make a finite Neumann series to produce a better approximate inverse,  and reduce the order of the error at $\cf$ as much as we like. Assume that $u \in H_{\mathrm{d3sc,3co,res}}^{m, \sfl + r, \sfl, 2\sfl + b, 2\sfl + b, m + \vol}$. We now set  
\begin{align*}
A &= ( I + \Delta)^{m/2} x_{\dmf}^{-r +(b-\delta)/2}  x_{\dff}^{(b-\delta)/2} \Lambda_{\vol + b/2 - \delta/2} \in \Psf^{m, \vol + r, \vol, 2\vol + b - \delta, 2\vol + b - \delta, m+\vol }, \\
A_\delta &=  x_{\dmf}^{\delta/2} x_{\dff}^{\delta/2} \Lambda_{\delta/2} \in \Psf^{0, 0, 0, \delta, \delta, 0}, 
\end{align*} 
where $\Lambda_{\vol}$ is as above. The operator $A$ has an inverse modulo $\Psf^{0, 0, 0, -N, -N, 0}$, since it is the product of three invertible operators together with $\Lambda_{\vol}$ which has this property. Thus the operators $A$ and $A_\delta$ symbolically at $\dtsccf$ and globally at $\cf$ have the analogous algebraic properties as we had above at fiber infinity. The calculation in \eqref{eq:L2 interp} can now be repeated to obtain interpolation at these two faces. 

This completes the proof of the lemma. 
\end{proof}

Thus, by using the above lemma, we can now start with the remainder terms found in Proposition \ref{decay at the corner face proposition}, and estimate as in (\ref{section 10 interpolation subsection cal 1}), assuming that the variable orders now satisfy more restrictive conditions as discussed after the statement of Lemma \ref{interpolation lemma section 10}. By absorbing the $\epsilon$-term, it follows that we would have reduced (\ref{corner face decay equation 1}) to a pair of estimates
\begin{equation} \label{section 10 renewed corner face estimate}
 \| u \|_{ H_{\mathrm{d3sc,3co,res}}^{ m_{ \pm }, \vor_{ \pm }, \vol_{ \pm }, \vor_{ \pm } + \vol_{ \pm }, \vob_{ \pm }, s_{ \pm }  } } \leq  C ( \| P u \|_{ H_{\mathrm{d3sc,3co,res}}^{ m_{ \pm } - 2, \vor_{ \pm } + 1, \vol_{ \pm }, \vor_{ \pm } + \vol_{\pm  } + 1, \vob_{ \pm } + 2, s_{ \pm } - 2 } } + \| u \|_{H^{M, N, \vol_{\pm}, K , F, S}_{\mathrm{d3sc, 3co, res}}} ),
\end{equation}
assuming still that $u \in H_{\mathrm{d3sc,3co,res}}^{m_{\pm}, \vor_{\pm}, \vol_{\pm}, \vor_{\pm} + \vol_{\pm}, \vob_{\pm}, s_{\pm}}$, $Pu \in H_{\mathrm{d3sc,3co,res}}^{ m_{ \pm } - 2, \vor_{ \pm } + 1, \vol_{ \pm }, \vor_{ \pm } + \vol_{\pm  } + 1, \vob_{ \pm } + 2, s_{ \pm } - 2 }$

\subsection{Estimating the remainder} In what remains of this section as well as the next, we will estimate the remainder terms in (\ref{section 10 renewed corner face estimate}). \par

As mentioned at the beginning of this section, we first prove an estimate with respect to $N_{\dff}(P)$, $\tindex \in \ind$. In fact, even now we shall proceed slightly generally, and our result below will be proved with respect to a suitable microlocalization at $\dff$.

\begin{proposition} \label{section 10 main proposition}
Let $\vor_{\pm}$, $\vol_{\pm}$, $\vob_{\pm}$ be the orders constructed in \S \ref{variable order construction section} and $U_{\tindex} \subset \mathcal{C}_{\tindex} \times \mathbb{R}^{n_{\tindex}}$ be a small neighborhood of $\Sigma_{\dff}$. Suppose that $B \in \Psf^{-\infty, -\infty, 0, 0, 0, -\infty}$ satisfies
\begin{equation} \label{wave front set elliptic set condition decay at dff}
\Sigma_{\dff} \subset \mathrm{Ell}_{\dff}(B) \subset \mathrm{WF}_{\dff}'(B) \subset U_{\tindex}.
\end{equation}
Then we can find a small $\delta > 0$, such that all sufficiently negative $M, N, K, F, S < 0$ and small $\delta >0$, there exists $C > 0$ such that
\begin{align} \label{decay at dff main proposition estimate}
\begin{split}
\| B u \|_{H_{\mathrm{d3sc,3co,res}}^{M, \ast , \vol_{\pm}, K , F, S }} & \leq C ( \| N_{\dff}(P) u \|_{H^{ M , N , \vol_{\pm} , \vor_{\pm} + \vol_{\pm} + 1 - \delta,  \vob_{\pm} + 2 - \delta  , S  }_{\mathrm{d3sc, 3co, res}}} \\
& \quad + \| u \|_{ H^{ M, N , \vol_{\pm} - \delta, \vor_{\pm} + \vol_{\pm} - \delta, \vob_{\pm} - \delta, S   }_{\mathrm{d3sc,3co,res}} } ).
\end{split}
\end{align}
\end{proposition}
For definiteness, we will focus on the case with orders $\vor_{+}$, $\vol_{+}$, $\vob_{+}$, and drop the $+$ signs.
\begin{proof}
First, let $\rho_{\dff} \in \mathcal{C}^{\infty}( \psf \Xd )$ be a defining function for $\psf_{\dff} \Xd$, and let $q_{\tindex} \in \mathcal{C}_{c}^{\infty}( \mathcal{C}_{\tindex} \times \mathbb{R}^{n_{\tindex}})$ be a cut-off at $U_{\tindex}$ such that $q_{\tindex} = 1$ on $U_{\tindex}$. We will define 
\begin{equation*}
Q_0 \in \Psf^{-\infty, -\infty, 0, 0, 0, 0}, \ \Lambda \in \Psf^{0,\vol, \vol, 2\vol, 2 \vol, \vol}
\end{equation*}
such that $Q_0$ is some quantization of $q_{\tindex} \psi(\rho_{\dff})$, where $\psi \in \mathcal{C}^{\infty}_{c}( [0,1) )$ is a cut-off function at $0$, and $\Lambda$ is some quantization of $\psi( x_{\tindex} ) x_{\tindex}^{-\vol}$. Let also
\begin{equation*}
B^{\tindex} \in \Psi_{\mathrm{sc}}^{-N,-N}(X^{\tindex})
\end{equation*}
be invertible with inverse in $\Psi_{\mathrm{sc}}^{N,N}(X^{\tindex})$, where we will take $N>0$ to be arbitrarily large. Note that we also have $B^{\tindex} \in \Psi_{\mathrm{sc,b}}^{-N,-N,-N}(X^{\tindex})$. Then we will define
\begin{equation*}
B_0 = \varphi_{\dff} B^{\tindex} Q_0 \Lambda,
\end{equation*}
where $\varphi_{\dff} \in \mathcal{C}^{\infty}( \Xd )$ is a cut-off at $\dff$. \par

A priori, it is not clear that $B_0$ belongs to the second microlocal calculus. This is clarified through the following lemma, which is a direct analogy to Lemma \ref{lemma construction of L decay at cf}, except our analysis now takes place near $\dff$.

\begin{lemma} \label{membership of weird operators near dff}
Suppose that 
\begin{equation*}
\begin{gathered}
l^{\tindex} \in \mathbb{R}, \quad \vom^{\tindex}, \vor^{\tindex} \in \mathcal{C}^{\infty}( \overline{ ^{\mathrm{sc,b}}T^{\ast}} X^{\tindex} ), \quad A^{\tindex} \in \Psi_{\mathrm{sc,b}, \delta }^{\vom^{\tindex} , \vor^{\tindex}, l^{\tindex} }( X^{\tindex} ), \\
\vov, \vos \in \mathcal{C}^{\infty}( \psf \Xd ), \quad \vol, \vob \in \mathcal{C}^{\infty}( \mathcal{C}_{\tindex} \times \overline{\mathbb{R}^{n_{\tindex}}} ), \quad Q \in \Psf^{-\infty, -\infty , \vol, \vov, \vob, \vos}. 
\end{gathered}
\end{equation*}
Assume additionally that $Q$ is supported near $\dff$, and moreover $\WFs(Q)$ is contained in a neighborhood of $\psf_{\dff} \Xd$ where the free variables are bounded (in particular, $\WFs(Q)$ is compactly contained away from fiber infinity). Then we have
\begin{equation} \label{decay at dff special operator in lemma eq}
\varphi_{\dff} A^{\tindex} Q \in \Psf^{-\infty, -\infty, \vol, \vov + \vor^{\tindex}, \vob + l^{\tindex}, \vos + \vom^{\tindex} }, \quad \hat{N}_{\dff}( \varphi_{\dff} A^{\tindex} Q ) = A^{\tindex} \hat{N}_{\dff}(Q),
\end{equation}
where $\varphi_{\dff} \in \mathcal{C}^{\infty}( \Xd )$ is some cut-off function at $\cf$. Here the variables orders $\vom^{\tindex}, \vor^{\tindex}$ make sense in (\ref{decay at dff special operator in lemma eq}) by being understood as arbitrary extensions $\mathcal{C}^{\infty}( \psf \Xd )$ extensions outside of the $\WFs(Q)$.\end{lemma}

The proof of this lemma proceeds similarly to that of Lemma \ref{lemma construction of L decay at cf}, and is thus omitted.

By the above lemma, we now know that $B_0  \in \Psf^{-\infty, -\infty, \vol, 2\vol - N, 2\vol - N, \vol -N}$. Our strategy next is to estimate the $L^2$ norm of $B_0$, and have it be bounded by the right hand side of (\ref{decay at dff main proposition estimate}). Assuming this is done, then Proposition \ref{section 10 main proposition} can be proved easily. \par

Indeed, by applying elliptic regularity at $\dff$, i.e., Proposition \ref{proposition microlocal elliptic regularity estimate chapter 1 thesis}, part (1), and using condition (\ref{wave front set elliptic set condition decay at dff}), as well as the fact that $U_{\tindex} \subset \mathrm{Ell}_{\dff}(B_0)$ by construction. Assuming that $F, S, L < 0$ are sufficiently negative. Then we have
\begin{align} \label{decay at dff main reduction step}
\begin{split}
\| B u \|_{H_{\mathrm{d3sc,3co,res}}^{M, \ast , \vol, K , F , S}}  &\leq C \| B u \|_{H_{\mathrm{d3sc,3co,res}}^{M,\ast, \vol,  2\vol - N , 2 \vol - N, \vol - N}} \\
& \leq C ( \| B_0 u \|_{L^2} + \| u \|_{ H_{\mathrm{d3sc,3co,res}}^{M, \ast , L , 2\vol - N, 2\vol - N, \vol -N } }  )
\end{split}
\end{align}
if $N > 0$ is relatively small.  \par

For the rest of this proof, we will focus on estimating the $L^2$ norm of $B_0 u$. The main idea is to use the Plancherel theorem in the free variables, as well as the fact that $B^{\tindex}$ depends only on the interaction variables. Then we have 
\begin{align}   \label{front face decay cal 1.9}
\begin{split}
 \| B_0 u \|_{L^{2}}^{2} & \leq \int_{\mathbb{R}^{n_{\tindex}}} \| B^{\tindex} Q_0 \Lambda  u  \|_{L^{2}(X^{\tindex})}^2 d z_{\tindex} \\
&  = \int_{\mathbb{R}^{n_{\tindex}}} \| B^{\tindex} \mathcal{F}_{z_{\tindex} \rightarrow \zeta_{\tindex}} (Q_0 \Lambda   u ) \|_{ L^2 (X^{\tindex}) }^2 d\zeta_{\tindex} \\
& \leq C \int_{\mathbb{R}^{n_{\tindex}}} \| \mathcal{F}_{z_{\tindex} \rightarrow \zeta_{\tindex}} (Q_0 \Lambda  u ) \|_{H_{\mathrm{sc,b}}^{ - N, -N, - N }  (X^{\tindex}) }^2 d \zeta_{\tindex} \\
& \leq C \int_{\mathbb{R}^{n_{\tindex}}} \| \mathcal{F}_{z_{\tindex} \rightarrow \zeta_{\tindex} } ( Q_0 \Lambda u ) \|_{H_{\mathrm{sc,b}}^{-N, \tvor^{\tindex} - 2 \delta, b - 2\delta }(X^{\tindex})}^2 d\zeta_{\tindex}, 
\end{split}
\end{align}
where we recall that $b = \vob - 2 \vol$ is a constant such that $|b + 1| < ( n^{\tindex} - 2 )/2$. Here $\delta > 0$ is chosen to be sufficiently small such that $\tvor^{\tindex} - 2 \delta$, $b - 2 \delta$ still satisfy the threshold conditions specified in Corollary \ref{Andras second microlocal main corollary}, which has also been discussed after the statement of the corollary. It follows that we can apply Corollary \ref{Andras second microlocal main corollary} to further estimate
\begin{align}
\begin{split}
& \int_{\mathbb{R}^{n_{\tindex}}} \| \mathcal{F}_{z_{\tindex} \rightarrow \zeta_{\tindex}} (Q_0 \Lambda u ) \|_{H_{\mathrm{sc,b}}^{-N, \tvor^{\tindex} - 2 \delta, b - 2\delta } (X^{\tindex}) }^2 d \zeta_{\tindex} \\
& \qquad \leq C \int_{\mathbb{R}^{n_{\tindex}}} \| \hat{N}_{\dff}(P) \mathcal{F}_{z_{\tindex} \rightarrow \zeta_{\tindex}} ( Q_0 \Lambda u ) \|_{ H_{\mathrm{sc,b}}^{ - N -2, \tvor^{\tindex} + 1 - 2 \delta , b + 2 - 2 \delta } ( X^{\tindex} ) }^2 d \zeta_{\tindex}.
\end{split}
\end{align} \par

Now, recall that we have
\begin{align}  \label{front face decay cal 2}
\begin{split}
& \| \hat{N}_{\dff}(P) \mathcal{F}_{z_{\tindex} \rightarrow \zeta_{\tindex}} ( Q_0 \Lambda  u ) \|_{H_{\mathrm{sc,b}}^{ - N - 2, \tvor^{\tindex} + 1 - 2 \delta,  b + 2 - 2 \delta }(X^{\tindex})}^2 \\
& \qquad \simeq \|  \hat{N}_{\dff}(P) \mathcal{F}_{z_{\tindex} \rightarrow \zeta_{\tindex}} (Q_0 \Lambda   u) \|_{H_{\mathrm{b}}^{-K,b+ 2 - 2 \delta  }(X^{\tindex}) }^2  + \|  A^{\tindex} \hat{N}_{\dff}(P) \mathcal{F}_{z_{\tindex} \rightarrow \zeta_{\tindex}} ( Q_0 \Lambda  u ) \|_{L^{2}(X^{\tindex})}^2,
\end{split}
\end{align}
where $K > 0$ is taken arbitrarily large, and
\begin{equation*}
A^{\tindex} \in \Psi_{\mathrm{sc,b}, \delta }^{ - N - 2, \tvor^{\tindex} + 1 - 2 \delta, b + 2 - 2 \delta  }(X^{\tindex})
\end{equation*}
is elliptic in the symbolic sense, and in particular can be taken supported near the diagonal. \par

Integrating the last term in (\ref{front face decay cal 2}) over $\zeta_{\tindex} \in \mathbb{R}^{n_{\tindex}}$, followed by using Plancherel in the free variables again and then identity (\ref{front face decay cal 0.8}), we see that
\begin{align}  \label{front face decay cal 2.1}
\begin{split}
& \int_{\mathbb{R}^{n_{\tindex}}}  \|  A^{\tindex} \hat{N}_{\dff}(P) \mathcal{F}_{z_{\tindex} \rightarrow \zeta_{\tindex}} ( Q_0 \Lambda  u ) \|_{L^{2}(X^{\tindex})}^2 d\zeta_{\tindex} \\
& \qquad = \int_{\mathbb{R}^{n_{\tindex}}} \| A^{\tindex} N_{\dff}(P) Q_0 \Lambda   u \|_{L^{2}(X^{\tindex})}^2 d \zeta_{\tindex} = \| A^{\tindex} N_{\dff}(P) Q_0 \Lambda   u \|_{L^{2}}^2.
\end{split}
\end{align} 
We can then commute out
\begin{equation}  \label{front face decay cal 2.1.1}
\| A^{\tindex} N_{\dff}(P) Q_0 \Lambda u \|_{L^{2}} \leq \| A^{\tindex} Q_0 \Lambda N_{\dff}(P) u \|_{L^{2}} + \| A^{\tindex} [ N_{\dff}(P), Q_0 \Lambda ] u \|_{L^2}.
\end{equation} 
By choosing $Q_0$ to be supported near the diagonal, we can assume that $A^{\tindex} Q_0$ is supported near $\dff$, i.e., $A^{\tindex} Q_0 = \varphi_{\dff} A^{\tindex} Q_0$ where $\varphi_{\dff}$ is a cut-off at $\dff$. Here, we have also used that $Q_0$ is supported near $\dff$, since it is given by a quantization of some symbol that is supported near $\psf_{\dff}\Xd$. Then Lemma \ref{membership of weird operators near dff} tells us that
\begin{equation} \label{front face decay cal 2.1.2}
A^{\tindex} Q_0 \Lambda \in \Psf^{-\infty, -\infty, \vol, 2 \vol + \tvor^{\tindex} + 1 - 2 \delta, 2 \vol + b + 2 - 2 \delta , \vol - N - 2  }.
\end{equation}
On the other hand, since $\hat{N}_{\dff}( Q_0 \Lambda )$ is scalar valued (and thus $\hat{N}_{\dff} ( [ N_{\dff}(P), Q_0 \Lambda ] ) = 0$), we can apply a calculation similar to what was done in \S \ref{Subsection commutator formulae} to conclude that
\begin{equation*}
[ N_{\dff}(P), Q_0 \Lambda ] \in \Psf^{-\infty,-\infty, \vol - 1 + \delta , 2 \vol - 1 + \delta , 2 \vol - 2 + \delta , \vol - 1 + \delta }.
\end{equation*}
Moreover, $[ N_{\dff}(P), Q_0 \Lambda ]$ is also an operator with support near $\dff$ (since $\Lambda$ is supported near the diagonal), so we can again conclude from Lemma \ref{membership of weird operators near dff} that
\begin{equation} \label{front face decay cal 2.1.3}
A^{\tindex} [ N_{\dff}(P) , Q_0 \Lambda ] \in \Psf^{-\infty, -\infty, \vol - 1 + \delta, 2 \vol + \tvor^{\tindex} - \delta, 2 \vol + b - \delta, \vol - 3  - N + \delta}.
\end{equation} 
\par

Now, one might expect to be able to estimate (\ref{front face decay cal 2.1.1}) directly. However, there is still the mild problem where the variable order $\tvor^{\tindex}$ only makes sense near $\WFs(Q_0 )$, and needs to be interpreted via extensions in (\ref{front face decay cal 2.1.2}) and (\ref{front face decay cal 2.1.3}). \par

For this, we simply take $\tilde{Q} \in \Psf^{-\infty,-\infty,0,0,0,0}$  such that $\WFs( Q_0 ) \cap \WFs( I - \tilde{Q} ) = \emptyset$, and moreover $\WFs( \tilde{Q} )$ is contained in a small neighborhood of $\psf_{\dff} \Xd$ where the free variables are bounded. Then we plainly have
\begin{align*}
\| A^{\tindex} Q_0 \Lambda N_{\dff}(P) u \|_{L^{2}} & \leq \| A^{\tindex} Q_0 \Lambda \tilde{Q} N_{\dff}(P) u\|_{L^2} \\
& \quad + \| A^{\tindex} Q_0 \Lambda ( I - \tilde{Q} ) N_{\dff}(P) u \|_{L^2}, 
\end{align*}
and 
\begin{align*}
\| A^{\tindex}[ N_{\dff}(P), Q_0 \Lambda ] u \|_{L^2} & \leq \| A^{\tindex} [ N_{\dff}(P), Q_0 \Lambda ] \tilde{Q} u \|_{L^2} \\
& \quad + \|  A^{\tindex}[ N_{\dff}(P), Q_0 \Lambda ] ( I - \tilde{Q} ) u \|_{L^2}.
\end{align*} 
By (\ref{front face decay cal 2.1.2}) and (\ref{front face decay cal 0.7}), and using the fact that $2 \vol + b = \vob$, we can estimate 
\begin{align*}
  \| A^{\tindex} Q_0 \Lambda \tilde{Q} N_{\dff}(P) u \|_{L^2} & \leq C \| \tilde{Q} N_{\dff}(P) u \|_{H_{\mathrm{d3sc,3co,res}}^{\ast, \ast, \vol, 2 \vol + \tvor^{\tindex} + 1 - 2 \delta, 2 \vol + b + 2 - 2 \delta, \vol - N - 2  } } \\
 & \leq C  \| \tilde{Q} N_{\dff}(P) u \|_{H_{\mathrm{d3sc,3co,res}}^{\ast, \ast, \vol, \vor + \vol + 1 + 2 \beta^{\tindex} - 2 \delta , \vob + 2 - 2 \delta, \vol - N - 2  } }   \\
& \leq C \| N_{\dff}(P) u \|_{H_{\mathrm{d3sc,3co,res}}^{ \ast , \ast , \vol, \vor + \vol + 1, \vob + 2, S}} ,
\end{align*}
where in the last step we will take $2 \beta^{\tindex} < \delta$. Notice that this relation between $\delta$ and $\beta_{\tindex}$ is consistent with (\ref{front face decay cal 0.69}). We can therefore write $2\beta^{\tindex} - \delta = - \delta'$ for some $\delta' > 0$. Moreover, from (\ref{front face decay cal 2.1.3}), the same calculation yields
\begin{equation*}
\| A^{\tindex} [ N_{\dff}(P) , Q_0 \Lambda ] \tilde{Q} u \|_{L^2} \leq C  \| u \|_{H_{\mathrm{d3sc,3co,res}}^{\ast,\ast, \vol - 1 + \delta, \vor + \vol - \delta', \vob - \delta, S}},
\end{equation*}
where we remind that $\vob = 2 \vol + b$. Meanwhile, by Lemma \ref{membership of weird operators near dff}, we also have
\begin{equation*}
\begin{gathered}
A^{\tindex} Q_0 \Lambda ( I - \tilde{Q} ) \in \Psf^{-\infty, -\infty, \vol, -\infty,  2 \vol + b + 2 - 2 \delta, -\infty}, \\
A^{\tindex}[ N_{\dff}(P) , Q_0 \Lambda ] ( I - \tilde{Q} ) \in \Psf^{-\infty, -\infty, \vol - 1 + \delta, -\infty, 2 \vol + b - \delta, -\infty},
\end{gathered}
\end{equation*}
and thus
\begin{equation*}
\begin{gathered}
\| A^{\tindex} Q_0 \Lambda ( I - \tilde{Q} ) N_{\dff}(P) \|_{L^2} \leq C \| N_{\dff}(P) u \|_{ H_{\mathrm{d3sc,3co,res}}^{\ast,\ast, \vol, \ast, \vob + 2, \ast  } }, \\
\|  A^{\tindex}[ N_{\dff}(P) , Q_0 \Lambda ] ( I - \tilde{Q} ) u  \|_{L^2} \leq C \| u \|_{H_{\mathrm{d3sc,3co,res}}^{\ast, \ast, \vol - 1 + \delta, \ast, \vob - \delta, \ast } }.
\end{gathered}
\end{equation*}
It follows from putting the above back into (\ref{front face decay cal 2.1.1}), we finally get that
\begin{align} \label{front face decay cal 2.1.4}
\begin{split}
\| A^{\tindex} N_{\dff}(P) Q \Lambda u \|_{L^{2}} & \leq C ( \| N_{\dff}(P) u \|_{H_{\mathrm{d3sc,3co,res}}^{ \ast , \ast , \vol, \vor + \vol + 1, \vob + 2, S}}  + \| u \|_{H_{\mathrm{d3sc,3co,res}}^{\ast,\ast, \vol - 1 + \delta, \vor + \vol - \delta', \vob - \delta, S}} ),
\end{split}
\end{align}
which is an estimate for the first term in the second line of (\ref{front face decay cal 2}).
\par

In order to estimate the first term in the second line of (\ref{front face decay cal 2}), we will need the following lemma:

\begin{lemma} \label{useful lemma 3co decay at dff}
 Let $k$ be an even positive integer. Then for all $v$ supported in a fixed small neighborhood of $\dff$, say $\{ x_{\dff} \leq 1 \}$, where $x_{\dff} \in \mathcal{C}^{\infty}(\Xd)$ is a boundary defining function of $\Xd$, we have
 \begin{equation}
\int_{\mathbb{R}^{n_{\tindex}}} \ang{\zeta_\tindex}^{-2k} \| \mathcal{F}_{ z_{\tindex} \rightarrow \zeta_{\tindex} } v \|_{H^{-k, b}_{\mathrm{b}} (X^{\tindex})}^2 \, d\zeta_{\tindex} \leq C \| v \|^2_{H^{-k, b, b, 0}_{\mathrm{3co}}},
\label{eq:v norm} \end{equation}
and consequently, 
 \begin{equation}
\int_{\mathbb{R}^{n_{\tindex}}}  \|  \mathcal{F}_{z_{\tindex} \rightarrow \zeta_{\tindex}} v \|_{H^{-k, b}_{\mathrm{b}}(X^{\tindex})}^2 \, d\zeta_{\tindex} \leq C \| (1 + \Delta_{z_{\tindex}} )^{k/2} v \|^2_{H^{-k, b, b, 0}_{\mathrm{3co}}} . 
\label{eq:v norm 2} \end{equation}
\end{lemma}
\begin{proof}
The second identity \eqref{eq:v norm 2} follows immediately from \eqref{eq:v norm} by replacing $v$ with $(1 + \Delta_{z_{\tindex}} )^{k/2} v$. So it suffices to prove \eqref{eq:v norm}. 
 
 We define $\mathcal{H}^{k, -b}$ to be the Hilbert space 
 \begin{equation}\begin{gathered}
 \mathcal{H}^{k, -b} = \Big\{ w \in \mathcal{S}'(\mathbb{R}^n) : \int_{\mathbb{R}^{n_{\tindex}}} \ang{\zeta_{\tindex}}^{2k} \| \mathcal{F}_{z_{\tindex} \rightarrow \zeta_{\tindex}} w(\zeta_a) \|_{H^{k, -b}_\mathrm{b} ( X^\tindex )}^2 \, d\zeta_\tindex < \infty \Big\}, \\
 \ang{w_1, w_2}_{\mathcal{H}^{k, -b}} \coloneq \int_{\mathbb{R}^{n_{\tindex}}} \ang{\zeta_\tindex }^{2k} \ang{ \mathcal{F}_{z_{\tindex} \rightarrow \zeta_{\tindex} } w_1(\zeta_{\tindex} ), \mathcal{F}_{z_{\tindex} \rightarrow \zeta_{\tindex}} w_2(\zeta_\tindex ) }_{H^{k, -b}_\mathrm{b} (X^\tindex )} \, d\zeta_\tindex .
 \end{gathered}\end{equation}
 Now let $v \in H^{-k, b, b, 0}_{\mathrm{3co}}$ have support in $\{ x_{\dff} \leq 1 \}$. Then $v$ can be represented as 
 \begin{equation*}
 v = x_{\partial X^{\tindex}}^{b} \sum_j V^{(j)} v_j, \quad V^{(j)} \in \mathcal{V}_{\mathrm{3co}}^k, \quad v_j \in L^2,
 \end{equation*}
 with the squared norm of $v$ comparable to the infimum of $\sum_j \| v_j \|_2^2$ over all such representations. Here the $V^{(j)}$ should generate $\mathcal{V}_{\mathrm{3co}}^k$ as a module over $C^\infty(\Xd)$. We can take such a generating set, and split each  $V^{(j)} = \chi V^{(j)} + (1 - \chi) V^{(j)}$, where $\chi = 1$ on $\{  x_{\dff} \leq 1 \}$, vanishing outside $\{ x_{\dff} \geq 2 \}$. Then the $(1 - \chi) V^{(j)}$ elements can be discarded due to the support condition on $v$, and we may hence assume that the 3co-vector fields can be chosen from a generating set of partial derivatives in $z_{\tindex}$, and b-derivatives in the interaction variables. Then consider the $L^2$ pairing $ \ang{v, w}_{L^2}$, where $w \in \mathcal{H}^{k, -b}$. By using the representation as above for $v$, integrating by parts to move the vector fields to $w$, and expressing in terms of the partial Fourier transform applied to both functions, we see that 
 \begin{equation}
  \label{eq:v norm 3} 
 | \ang{v, w}_{L^2}| \leq C \| v \|_{H^{-k, b, b, 0}_{\mathrm{3co}}} \| w \|_{\mathcal{H}^{k, -b}}. 
\end{equation}
 Thus, $\{ v \in H^{-k, b, b, 0}_{\mathrm{3co}} \mid \text{supp } v \subset \{ x_{\dff} \leq 1 \} \}$ is embedded in the dual space of $\mathcal{H}^{k, -b}$, which is easily seen to be the space $\mathcal{H}^{-k, b}$. Moreover, \eqref{eq:v norm 3} shows that the $\mathcal{H}^{-k, b}$ norm of $v$ is bounded by the $H^{-k, b, b, 0}_{\mathrm{3co}}$ norm of $v$. This implies \eqref{eq:v norm}. 
\end{proof}

By using (\ref{eq:v norm 2}), and taking $K$ to be a large even integer, we can now estimate
\begin{align} \label{front face decay cal 2.1.4.1}
\begin{split}
& \int_{\mathbb{R}^{n_{\tindex}}} \|  \hat{N}_{\dff}(P) \mathcal{F}_{z_{\tindex} \rightarrow \zeta_{\tindex}} (Q_0 \Lambda   u) \|_{H_{\mathrm{b}}^{-K,b+ 2 - 2 \delta  }(X^{\tindex}) }^2 d\zeta_{\tindex} \\
& \qquad \leq C  \| ( I + \Delta_{z_{\tindex}} )^{K/2} N_{\dff}(P) ( Q_0 \Lambda u ) \|_{H_{\mathrm{3co}}^{-K, b + 2 - 2 \delta, b + 2 - 2 \delta , 0}}^2 \\
& \qquad \leq C \| N_{\dff}(P) ( I + \Delta_{z_{\tindex}} )^{K/2}  Q_0 \Lambda u \|_{H_{\mathrm{d3sc,3co,res}}^{-K, b + 2 - 2 \delta, 0, b + 2 - K - 2 \delta , b + 2 - 2\delta, -K }}^2.
\end{split}
\end{align}
Although we just have $( I + \Delta_{z_{\tindex}} )^{K/2} \in \Psf^{K,0,0,0,0,K}$, we can actually show that, upon being composed with $Q_0$, it holds that
\begin{equation} \label{front face decay cal 2.1.4.2}
( I + \Delta_{z_{\tindex}} )^{K/2} Q_0 \in \Psf^{-\infty, -\infty, 0,0,0,0}.
\end{equation}
Indeed, since $Q_0$ is only microlocally supported in a neighborhood of $\psf_{\dff}\Xd$ where the free variables are bounded, the symbol of $( I + \Delta_{z_{\tindex}} )^{K/2}$, namely just $\langle \zeta_{\tindex} \rangle^{K}$, will also be bounded where $Q_0$ is microlocally supported.  \par

Thus starting from the last line of (\ref{front face decay cal 2.1.4.1}), we can commute 
\begin{align*}
& \| N_{\dff}(P) ( I + \Delta_{z_{\tindex}} )^{K/2}  Q_0 \Lambda u \|_{H_{\mathrm{d3sc,3co,res}}^{-K, b + 2 - 2 \delta, 0, b + 2 - K - 2 \delta , b + 2 - 2\delta, -K }} \\
& \qquad \leq \| ( I + \Delta_{z_{\tindex}} )^{K/2} Q_0 \Lambda N_{\dff}(P) u \|_{H_{\mathrm{d3sc,3co,res}}^{-K, b + 2 - 2 \delta, 0, b + 2 - K - 2 \delta , b + 2 - 2\delta, -K }} \\
& \qquad \quad + \| [ N_{\dff}(P), (I + \Delta_{z_{\tindex}} )^{K/2} Q_0 \Lambda ] u  \|_{H_{\mathrm{d3sc,3co,res}}^{-K, b + 2 - 2 \delta, 0, b + 2 - K - 2 \delta , b + 2 - 2\delta, -K }}.
\end{align*}
By using (\ref{front face decay cal 2.1.4.2}) and again $2 \vol + b = \vob$, we can estimate
\begin{align*}
& \| ( I + \Delta_{z_{\tindex}} )^{K/2} Q_0 \Lambda N_{\dff}(P) u \|_{H_{\mathrm{d3sc,3co,res}}^{-K, b + 2 - 2 \delta, 0, b + 2 - K - 2 \delta , b + 2 - 2\delta, -K }} \\
& \qquad \leq C \| N_{\dff}(P) u  \|_{H_{\mathrm{d3sc,3co,res}}^{\ast,\ast, \vol, \vob + 2 - K - 2 \delta,\vob + 2 - 2 \delta, \vol - K }}.
\end{align*}
On the other hand, it is clear that $\hat{N}_{\dff}( ( I + \Delta_{z_{\tindex}} )^{K/2} ) = \langle \zeta_{\tindex} \rangle^{K}$, thus $\hat{N}_{\dff}( [ N_{\dff}(P) , ( I + \Delta_{z_{\tindex}} )^{K/2} Q_0 \Lambda ] ) = 0$,
and we can again show that
\begin{equation*}
[ N_{\dff}(P), ( I + \Delta_{z_{\tindex}} )^{K/2} Q_0 \Lambda ] \in \Psf^{-\infty, -\infty, \vol - 1 + \delta, 2 \vol - 1 + \delta, 2 \vol - 2 + \delta, \vol + 1 + \delta }.
\end{equation*}
It follows that
\begin{align*}
& \| [ N_{\dff}(P) , (I + \Delta_{z_{\tindex}} )^{K/2} Q_0 \Lambda ] u  \|_{H_{\mathrm{d3sc,3co,res}}^{-K, b + 2 - 2 \delta, 0, b + 2 - K - 2 \delta , b + 2 - 2\delta, -K }} \\
& \qquad \leq C \| u \|_{H_{\mathrm{d3sc,3co,res}}^{\ast,\ast, \vol - 1 + \delta, \vob + 1 - K - \delta, \vob - \delta, \vol + 1 + \delta - K }},
\end{align*}
where we have again achieved an error estimate with improved order (i.e., strictly less than $\vol$) at $\dff$. Thus if $K$ is large enough, then we have
\begin{align*}
\begin{split}
& \int_{\mathbb{R}^{n_{\tindex}}} \|  \hat{N}_{\dff}(P) \mathcal{F}_{z_{\tindex} \rightarrow \zeta_{\tindex}} (Q_0 \Lambda   u) \|_{H_{\mathrm{b}}^{-K,b+ 2 - 2 \delta  }(X^{\tindex}) }^2 d\zeta_{\tindex} \\
& \qquad \leq C ( \| N_{\dff}(P) u \|_{H_{\mathrm{d3sc,3co,res}}^{\ast, \ast, \vol, F , \vob + 2 , S}}^2  + \| u \|_{H_{\mathrm{d3sc,3co,res}}^{ \ast, \ast, \vol - 1 + \delta, F, \vob - \delta, S } }^2 ),
\end{split}
\end{align*}
where $F,S < 0$ can be made arbitrarily negative. \par

Finally, by combining the above with (\ref{front face decay cal 2.1.4}), as well as the calculations which precede it, we conclude that
\begin{equation*}
\begin{split}
\| B_0 u \|_{L^2} \leq C ( \| N_{\dff}(P) u \|_{H_{\mathrm{d3sc,3co,res}}^{ \ast , \ast , \vol, \vor + \vol + 1, \vob + 2, S}}  + \| u \|_{H_{\mathrm{d3sc,3co,res}}^{\ast,\ast, \vol - 1 + \delta, \vor + \vol - \delta', \vob - \delta, S}} ).
\end{split}
\end{equation*}
Substituting this estimate back into (\ref{decay at dff main reduction step}) proves the Proposition.
\end{proof}

\section{Proof of Theorem \ref{intro; main theorem 2}} \label{section 10 proof of main theorem subsection}
Finally, in this section, we will finish the proof of Theorem \ref{intro; main theorem 2}. In fact, we can now clarify the assumptions on the orders $m_{\pm}$, $\vor_{\pm}$, $\vol_{\pm}$, $\vob_{\pm}$, $s_{\pm}$: We will take them in accordance with the construction of \S \ref{variable order construction section}, and moreover assume that the additional constraints imposed in Lemma \ref{interpolation lemma section 10} are satisfied. See also the discussions following the statement of Lemma \ref{interpolation lemma section 10}.   \par

Notice that, in most of the discussions above, we have assumed $\mathcal{C} = \mathcal{C}_{\tindex} $. However, as we have mentioned numerous time throughout this paper, it should be obvious that much of the results generalize easily to the case of $\mathcal{C} = \{ \mathcal{C}_{\tindex} : \tindex \in \ind \}$ as defined in (\ref{three-body blowup}).  \par

In particular, the construction of the second microlocalized operators $\Psf^{\vom, \vor, \vol, \vov, \vob, \vos}(\Xd)$ and the  phase space $\psf \Xd$ can be easily generalized to the setting where
\begin{equation*}
X = [ [ \overline{\mathbb{R}^{n}} ; \mathcal{C} ] ; \mathrm{ff} \cap \mf ].\end{equation*}
Here we recall that $\mathrm{ff} = \{ \ff : \tindex \in \ind \}$, where each $\ff$ is the lift of $\mathcal{C}_{\tindex}$ to $[ \overline{\mathbb{R}^{n}} ; \mathcal{C} ]$. As for the variable orders, using the same notations to denote the boundary faces of $\psf \Xd$, we will again require that
\begin{equation*}
\mathsf{m}  \in \mathcal{C}^{\infty}( { ^{\mathrm{d3sc,3co,res}}S^{\ast}} \Xd ),  \quad \mathsf{r} \in \mathcal{C}^{\infty}( \overline{ ^{\mathrm{d3sc,3co,res}}T^{\ast} }_{\mathrm{dmf}} \Xd ).
\end{equation*}
However, we must now assume $\vov, \vos, \vol, \vob \in \mathcal{C}^{\infty}(\psf \Xd)$ are such that
\begin{equation*}
\begin{gathered}
 \mathsf{v}|_{\dtsccf} \in \mathcal{C}^{\infty}( \dtsccf ), \quad \mathsf{s}|_{\rf}   \in \mathcal{C}^{\infty}( \mathrm{rf}_{\tindex} ), \\
  \vol |_{\psf_{\ff}\Xd} = \vol |_{\tcocf}, \, \vob |_{\psf_{\ff}\Xd} = \vob |_{\tcocf} \in \mathcal{C}^{\infty}( \mathcal{C}_{\tindex} \times \overline{ \mathbb{R}^{n_{\tindex}} } )
 \end{gathered}
\end{equation*}
for \emph{every} $\tindex \in \ind$. The Sobolev space $H_{\mathrm{d3sc,3co,res}}^{\vom, \vor, \vol, \vov, \vob, \vos}$ can be defined as well. Then by disjointness, Propositions \ref{almost semi-Fredholm estimates with symbolic decay}, \ref{decay at the corner face proposition} can be seen to easily generalize.

Now, starting from the statement of Proposition \ref{decay at the corner face proposition}, it remains to show that there exists some $\delta > 0$ such that
\begin{align} \label{proof of main theorem cal 1}
\begin{split}
\| u \|_{H_{\mathrm{d3sc,3co,res}}^{M, \vor_{\pm} - \delta , \vol_{\pm}, \vor_{\pm} + \vol_{\pm} - \delta, \vob_{\pm} - \delta, S } } & \leq C ( \| P u \|_{H^{ m_{\pm} - 2 , \vor_{\pm} + 1 , \vol_{\pm} , \vor_{\pm} + \vol_{\pm} + 1, \vob_{\pm} + 2  , s_{\pm} - 2  }_{\mathrm{d3sc, 3co, res}}} \\
& \quad + \| u \|_{ H^{ M, \vor_{\pm} - \delta , \vol_{\pm} - \delta, \vor_{\pm} + \vol_{\pm} - \delta, \vob_{\pm} - \delta, S  }_{\mathrm{d3sc,3co,res}} } ),
\end{split}
\end{align}
where $M,S \in \mathbb{R}$ are arbitrarily negative. But through Lemma \ref{interpolation lemma section 10}, we have already shown that the main result of Proposition \ref{decay at the corner face proposition} can be further improved to (\ref{section 10 renewed corner face estimate}), namely
\begin{equation} \label{final section estimate 1}
 \| u \|_{ H_{\mathrm{d3sc,3co,res}}^{ m_{ \pm }, \vor_{ \pm }, \vol_{ \pm }, \vor_{ \pm } + \vol_{ \pm }, \vob_{ \pm }, s_{ \pm }  } } \leq  C ( \| P u \|_{ H_{\mathrm{d3sc,3co,res}}^{ m_{ \pm } - 2, \vor_{ \pm } + 1, \vol_{ \pm }, \vor_{ \pm } + \vol_{\pm  } + 1, \vob_{ \pm } + 2, s_{ \pm } - 2 } } + \| u \|_{H^{M, N, \vol_{\pm}, K , F, S}_{\mathrm{d3sc, 3co, res}}} ),
\end{equation}
where $N,K,F \in \mathbb{R}$ are arbitrarily negative. Thus, it is enough to show that (\ref{proof of main theorem cal 1}) follows from (\ref{final section estimate 1}). \par

Let $q_{\tindex} \in \mathcal{C}^{\infty}_{c}( \mathcal{C}_{\tindex} \times \mathbb{R}^{n_{\tindex}} )$, $\tindex \in \ind$ be a family such that each $q_{\tindex}$ cuts off at $\Sigma_{\dff}$. Moreover, assume that $Q \in \Psf^{-\infty, -\infty, 0,0,0,0}$ is such that $\hat{N}_{\dff}(Q) = q_{\tindex}$ for every $\tindex \in \ind$. Then we have
\begin{equation*}
\| u \|_{H_{\mathrm{d3sc,3co,res}}^{M, N , \vol_{\pm}, K, F, S } } \leq \| Q u \|_{H_{\mathrm{d3sc,3co,res}}^{M, N , \vol_{\pm}, K, F, S } }  + \| ( I - Q )  u \|_{H_{\mathrm{d3sc,3co,res}}^{M, N , \vol_{\pm}, K, F, S } } .
\end{equation*} \par

Since $\mathrm{WF}_{\dff}'( I - Q ) \subset \mathrm{Ell}_{\dff}( P )$, by elliptic regularity at $\dff$, we already have
\begin{equation*}
\| ( I - Q )  u \|_{H_{\mathrm{d3sc,3co,res}}^{M, N , \vol_{\pm}, K, F, S } }  \leq C ( \| P u \|_{H_{\mathrm{d3sc,3co,res}}^{M - 2, N , \vol_{\pm}, K, F, S - 2 } } +  \| u \|_{H_{\mathrm{d3sc,3co,res}}^{M, N , L, K, F, S } }   ),
\end{equation*} 
where $L \in \mathbb{R}$ is arbitrarily negative. \par

On the other hand, let $\varphi_{\mathcal{C}_{\tindex}}$, $\tindex \in \ind$ be a family of $\mathcal{C}^{\infty}( \overline{\mathbb{R}^{n}} )$ functions such that $\sum_{\tindex \in \ind} \varphi_{\mathcal{C}_{\tindex}} = 1$. Moreover, we assume that the support of each $\varphi_{\mathcal{C}_{\tindex}}$ contains $\mathcal{C}_{\tindex}$, but not any other $\mathcal{C}_{\tindex'}$ such that $\tindex' \neq \tindex$. Then we can write 
\begin{equation*}
\| Q u \|_{H_{\mathrm{d3sc,3co,res}}^{M,N, \vol_{\pm}, K, F, S}} \leq \sum_{\tindex \in \ind} \| Q \varphi_{\mathcal{C}_{\tindex}} u \|_{H_{\mathrm{d3sc,3co,res}}^{M,N, \vol_{\pm}, K, F, S}}.
\end{equation*}
By Proposition \ref{section 10 main proposition}, for every $\tindex \in \ind$, we have
\begin{align*}
\begin{split}
\| Q \varphi_{\mathcal{C}_{\tindex}} u \|_{H_{\mathrm{d3sc,3co,res}}^{M, N , \vol_{\pm}, K , F, S }} & \leq C ( \| N_{\dff}(P) \varphi_{\mathcal{C}_{\tindex}} u \|_{H^{ M , N , \vol_{\pm} , \vor_{\pm} + \vol_{\pm} + 1 , \vob_{\pm} + 2  , S  }_{\mathrm{d3sc, 3co, res}}} \\
& \quad + \| u \|_{ H^{ M, N , \vol_{\pm} - \delta, \vor_{\pm} + \vol_{\pm} - \delta, \vob_{\pm} - \delta, S   }_{\mathrm{d3sc,3co,res}} } )
\end{split}
\end{align*}
for some $\delta > 0$. We can also estimate
\begin{align*}
\| N_{\dff}(P) \varphi_{\mathcal{C}_{\tindex}} u \|_{H^{ M , N , \vol_{\pm} , \vor_{\pm} + \vol_{\pm} + 1, \vob_{\pm} + 2  , S  }_{\mathrm{d3sc, 3co, res}}} & \leq \| ( N_{\dff}(P) - P ) \varphi_{\mathcal{C}_{\tindex}} u  \|_{H^{ M , N , \vol_{\pm} , \vor_{\pm} + \vol_{\pm} + 1,  \vob_{\pm} + 2  , S  }_{\mathrm{d3sc, 3co, res}}} \\
& \quad + \| P  \varphi_{\mathcal{C}_{\tindex}} u \|_{H^{ M , N , \vol_{\pm} , \vor_{\pm} + \vol_{\pm} + 1, \vob_{\pm} + 2  , S  }_{\mathrm{d3sc, 3co, res}}}.
\end{align*} \par

We now note that
\begin{equation*}
 N_{\dff}(P) - P = \sum_{\tindex' \neq \tindex} V^{\tindex'},
\end{equation*}
where $V^{\tindex'} \in S^{-2 - \delta}( X^{\tindex'} )$, $\tindex' \neq \tindex$. Thus, upon lifting $V^{\tindex'}$ to $\psf \Xd$, we find that:
\begin{equation*}
\begin{gathered}
\text{$V^{\tindex'}$ decays to order $-2 - \delta$ at $\psf_{\dmf} \Xd$, $\psf_{\dff} \Xd$ and $\rf$}; \\
\text{$V^{\tindex'}$ decays to order $-4 - 2 \delta$ at $\dtsccf$ and $\tcocf$.}
\end{gathered}
\end{equation*}
Meanwhile, the spatial orders at $\mathrm{dff}_{\tindex'}$, $\mathrm{cf}_{\tindex'}$ where $\tindex' \neq \tindex$ do not matter on the support of $\varphi_{\mathcal{C}_{\tindex}}$ (i.e., $\varphi_{\mathcal{C}_{\tindex}}$ vanishes to infinite orders at these faces). Thus we have
\begin{equation*}
( N_{\dff}(P) - P ) \varphi_{\mathcal{C}_{\tindex}} \in \Psi_{\mathrm{d3sc,3co,res}}^{0, - 2 - \delta, -2 - \delta, -4 - 2 \delta, -4 - 2 \delta, - 2 - \delta}.
\end{equation*}
It follows that we can estimate
\begin{equation*}
\| ( N_{\dff}(P) - P ) \varphi_{\mathcal{C}_{\tindex}} u  \|_{H^{ M , N , \vol_{\pm} , \vor_{\pm} + \vol_{\pm} + 1 , \vob_{\pm} + 2  , S  }_{\mathrm{d3sc, 3co, res}}} \leq C \| u \|_{H_{\mathrm{d3sc,3co,res}}^{M, N -2 - \delta, \vol_{\pm} - 2 - \delta, \vor_{\pm} + \vol_{\pm} - 3 - 2 \delta,  \vob_{\pm} - 2 - 2 \delta, S - 2 - \delta }}.
\end{equation*}
Meanwhile, we clearly have
\begin{align*}
& \| P  \varphi_{\mathcal{C}_{\tindex}} u \|_{H^{ M , N , \vol_{\pm} , \vor_{\pm} + \vol_{\pm} + 1, \vob_{\pm} + 2  , S  }_{\mathrm{d3sc, 3co, res}}} \\
& \qquad \leq \| \varphi_{\mathcal{C}_{\tindex}} P  u \|_{H^{ M , N , \vol_{\pm} , \vor_{\pm} + \vol_{\pm} + 1, \vob_{\pm} + 2  , S  }_{\mathrm{d3sc, 3co, res}}} + \| [ P , \varphi_{\mathcal{C}_{\tindex}} ] u \|_{H^{ M , N , \vol_{\pm} , \vor_{\pm} + \vol_{\pm} + 1, \vob_{\pm} + 2  , S  }_{\mathrm{d3sc, 3co, res}}} \\
& \qquad \leq C ( \| P u \|_{H_{\mathrm{d3sc,3co,res}}^{M,N, \vol_{\pm}, \vor_{\pm} + \vol_{\pm} + 1, \vob_{\pm} + 2, S}} + \| u \|_{H_{\mathrm{d3sc,3co,res}}^{M, N, \ast, \ast, \ast, \ast}} ),
\end{align*}
since $[ P , \varphi_{\mathcal{C}_{\tindex}} ] = [ \Delta, \varphi_{\mathcal{C}_{\tindex}} ]$ spatially vanishes to infinity orders everywhere except for at $\dmf$. \par

This concludes the proof of Theorem \ref{intro; main theorem 2}.

\end{document}